\documentclass[reqno]{amsart}%hidelinks to remove the red rectangle in the references

\usepackage{a4,amsthm,amsfonts,amssymb,euscript}
\usepackage{latexsym, multicol, fancybox}
\usepackage{graphicx}
\usepackage{color}
\usepackage{amsmath, amsthm, amssymb, bm}
\usepackage{epstopdf}
\usepackage{caption}
\usepackage{psfrag}

\usepackage{mathrsfs}
\usepackage{enumitem}
\usepackage{hyperref}

\usepackage{slashed,ulem,cancel}

 \usepackage{tikz}

\usepackage{a4,amsmath,amssymb,amsthm,slashed}
\usepackage[top=1in, bottom=1in, left=1.35in, right=1.35in]{geometry}%set the margins;
\newtheorem{theorem}{Theorem}[section]
\newtheorem{lemma}[theorem]{Lemma}
\newtheorem{proposition}[theorem]{Proposition}
\newtheorem{corollary}[theorem]{Corollary}
\newtheorem{definition}[theorem]{Definition}
\newtheorem{remark}[theorem]{Remark}
\newtheorem{conjecture}[theorem]{Conjecture}
\newtheorem{assumption}[theorem]{Assumption}

\numberwithin{equation}{section}

\newcommand{\bea}{\begin{eqnarray}}
\newcommand{\eea}{\end{eqnarray}}
\def\beaa{\begin{eqnarray*}}
\def\eeaa{\end{eqnarray*}}
\def\ba{\begin{array}}
\def\ea{\end{array}}
\def\be#1{\begin{equation} \label{#1}}
\def \eeq{\end{equation}}
\newcommand{\bsub}{\begin{subequations}}
\newcommand{\esub}{\end{subequations}}

\def\lab{\label}
\newcommand{\nn}{\nonumber}
\def\les{\lesssim}
\def\c{\cdot}
\def\f12{{\frac 1 2}}

\def\err{\mbox{Err}}
\def\ov{\overline}

\def\a{{\alpha}}
\def\b{{\beta}}
\def\ga{\gamma}
\def\Ga{\Gamma}
\def\de{\delta}
\def\De{\Delta}
\def\ep{\epsilon}
\def\ka{\kappa}
\def\la{\lambda}
\def\La{\Lambda}

\def\Si{\Sigma}
\def\om{\omega}

\def\Th{\Theta}

\def\vphi{\varphi}

\def\th{\theta}

\def\nab{\nabla}
\def\pr{{\partial}}
\def\div{{\mbox div\,}}

\def\f{\frac}

\def\th{\theta}
\def\rd{\partial}

\def\Carter{\mathfrak{C}}

\def\AA{{\mathcal A}}
\def\BB{{\mathcal B}}

\def\FF{{\mathcal F}}
\def\GG{{\mathcal G}}
\def\HH{{\mathcal H}}
\def\II{{\mathcal I}}
\def\JJ{{\mathcal J}}

\def\LL{{\mathcal L}}

\def\MM{{\mathcal M}}
\def\NN{{\mathcal N}}

\def\QQ{{\mathcal Q}}

\def\A{{\bf A}}

\def\D{{\bf D}}
\def\E{{\bf E}}
\def\F{{\bf F}}

\def\g{{\bf g}}

\def\M{{\bf M}}

\def\S{{\bf S}}

\def\R{{r^2 +a^2}}

\DeclareFontFamily{U}{mathx}{\hyphenchar\font45}
\DeclareFontShape{U}{mathx}{m}{n}{
      <5> <6> <7> <8> <9> <10>
      <10.95> <12> <14.4> <17.28> <20.74> <24.88>
      mathx10
      }{}
\DeclareSymbolFont{mathx}{U}{mathx}{m}{n}
\DeclareFontSubstitution{U}{mathx}{m}{n}
\DeclareMathAccent{\widecheck}{0}{mathx}{"71}

\def\Gac{\widecheck{\mathbf{\Ga}}}

\def\dk{\mathfrak{d}}

\def\reg{s}
\def\Reals{\mathbb{R}}

\providecommand{\abs}[1]{\lvert#1\rvert}

\def\trap{\text{trap}}
\def\nontrap{\cancel{\trap}}

\def\BLS{\S^{\textbf{BL}}}
\providecommand{\bf}[1]{\textbf{#1}}

\def\Opw{\mathbf{Op}_w}

\def\df{\delta_{\FF}}

\def\xit{\xi_{\tt}}
\def\xiphi{\xi_{\tphi}}

\def\xitheta{\xi_{\th}}
\def\xirstar{\tilde{\xi}_{r^*}}
\def\tt{\tau}
\def\tphi{\tilde{\phi}}
\def\tmod{t_{\text{mod}}}
\def\phimod{\phi_{\text{mod}}}
\def\dbl{\delta_{\textbf{BL}}}

\def\dred{\delta_{\text{red}}}
\def\dhor{\delta_{\HH}}
\def\dec{\de_{\text{dec}}}

\def\gcheck{\widecheck{\g}}

\def\gam{\g_{a,m}}

\def\Vref{V_{\text{ref}}}

\def\qs{|q|^2}

\def\EM{{\bf EM}}
\def\EMF{{\bf EMF}}
\def\EF{{\bf EF}}

\newcounter{mnotecount}[section]

\def\Gtz{\GG_{\Xi}}

          \def\Mtrap{\,\MM_{\trap}}
\def\Mntrap{{\MM_{\nontrap}}}

\def\tauu{\underline{\tau}}
\def\tauut{\widetilde{\tauu}}

\def\NNtilde{\widetilde{\NN}}

\def\MF{{\bf MF}}

\def\bsplit{\begin{split}}

\def\prtan{\pr_{\text{tan}}}

\begin{document}
\allowdisplaybreaks

\title{Energy-Morawetz estimates for the wave equation in perturbations of Kerr}
\author{Siyuan Ma and J\'{e}r\'{e}mie Szeftel}

\begin{abstract}
In this paper, we prove energy and Morawetz estimates for solutions to the  scalar wave equation in spacetimes with metrics that are perturbations, compatible with nonlinear applications, of Kerr metrics in the full subextremal range. Central to our approach is the proof of a global in time energy-Morawetz estimate conditional on a low frequency control of the solution using microlocal multipliers adapted to the $r$-foliation of the spacetime.  This result constitutes a first step towards extending the current proof of Kerr stability in \cite{GCM1} \cite{GCM2} \cite{KS:Kerr} \cite{GKS} \cite{Shen}, valid in the slowly rotating case, to a complete resolution of the black hole stability conjecture, i.e., the statement that the Kerr family of spacetimes is nonlinearly stable for all subextremal angular momenta.
\end{abstract}

\maketitle

\tableofcontents

%%%%%%%%%%%%%

\section{Introduction}

%%%%%%%%%%%%%

%%%%%%%%%%%%%%%%%%%%%%

\subsection{Kerr stability conjecture}

%%%%%%%%%%%%%%%%%%%%%%

We begin with introducing the Einstein vacuum equations, the Kerr solution and the Kerr stability conjecture.

%%%%%%%%%%%%%%%%%%%%%%%%%%%

\subsubsection{Einstein vacuum equations}

%%%%%%%%%%%%%%%%%%%%%%%%%%%

The Einstein vacuum equations (EVE) in a {Lorentzian manifold} $(\MM, \g)$ take the form
\beaa
\mathbf{R}_{\a\b}=0,
\eeaa
where $\mathbf{R}_{\a\b}$ {denotes} the Ricci curvature tensor of the metric $\g$. Foundational contributions by Choquet-Bruhat \cite{CB52}, and by  
Choquet-Bruhat and Geroch \cite{CBG69}, formulate EVE as an evolution problem of hyperbolic type and associate to any suitable sufficiently regular initial data set a unique (up to diffeomorphisms) maximal Cauchy development.

%%%%%%%%%%%%%%%%%%%%%%%%%%%

\subsubsection{Kerr solution}

%%%%%%%%%%%%%%%%%%%%%%%%%%%

The EVE admit a family of explicit solutions, found by Kerr \cite{Kerr63} in 1963, which describe asymptotically flat, stationary, axially symmetric black hole spacetimes. The metrics of Kerr spacetimes are parameterized by an angular momentum per unit mass $a$ and a mass $m$,  satisfying $|a|\leq m$, and take the following form in the Boyer--Lindquist  \cite{BL67} coordinates $(t,r,\th, \phi)$
\bea\lab{eq:expressionofKerrmetricinBLcoordinates:intro}
\gam=-\frac{\Delta \qs}{\Sigma^2} dt^2 + \frac{\sin^2\th\Sigma^2 }{\qs}\bigg(d\phi - \frac{2amr}{\Sigma^2} dt\bigg)^2 +\frac{\qs}{\Delta} dr^2 + \qs d\th^2,
\eea
where
\bea
\Delta = r^2 - 2mr +a^2, \quad \qs=r^2+a^2\cos^2\th, \quad \Sigma^2=(\R)^2 - a^2\sin^2\th \Delta.
\eea
Note that the particular case $a=0$ with $m>0$ corresponds to the family of Schwarzschild spacetimes, introduced by Schwarzschild \cite{Sch16} in 1916.

We consider in this work the family of \textit{subextremal} Kerr spacetimes, in which the two parameters $(a,m)$ satisfy the strict inequality $|a|<m$. Such a subextremal Kerr spacetime contains a black hole $\{r<r_+\}$ with a nondegenerate event horizon located at $\{r=r_+\}$ where $r_+:=m+\sqrt{m^2-a^2}$ is the larger root of $\Delta=\De(r)$, see Figure \ref{fig:penrosediagramofKerr} for the corresponding Penrose diagram.

\begin{figure}[htbp]
  \begin{center}
\begin{tikzpicture}[scale=1.1]
\tikzstyle{every node}=[font=\Small]
      \draw[dashed, color=red, thick] (0.05,3.95) -- (3.2,0.8);
  \fill[yellow!50] (-0.05,3.95)--(-3.2,0.8) arc(225:315: 4.51 and 4.51) -- (0.05,3.95);
   \draw[dashed, color=red, thick] (-0.05,3.95)--(-3.2,0.8) ;
     \node at (-2.8,3) {Black hole region};
     \node at (0.1,4.3) {$i_+$};
       \node[rotate=315]  at (2.0,2.4) {$\II_+$};
       \node[rotate=45] at (-1.6,2.1) {Event horizon $r=r_+$};
         \draw[] (0,4) circle (0.05);
         \node at (0,1) {Domain of outer communication};  
\end{tikzpicture}
\end{center}
\caption{\footnotesize{Penrose diagram of subextremal Kerr spacetimes.}}
\lab{fig:penrosediagramofKerr}
\end{figure}
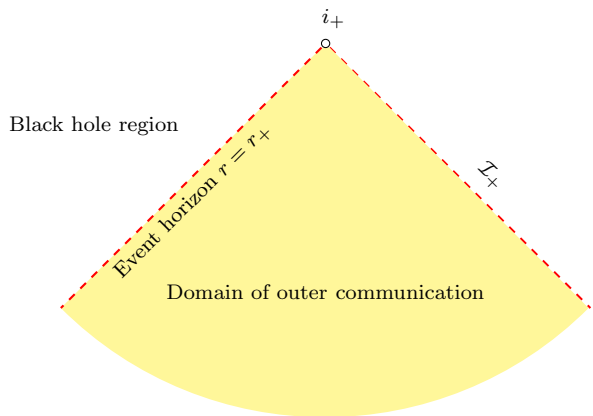

%%%%%%%%%%%%%%%%%%%%%%%%%%%

\subsubsection{Kerr stability conjecture}

%%%%%%%%%%%%%%%%%%%%%%%%%%%

The \textit{black hole stability conjecture} is one of the central open problems in general relativity. We provide a rough statement below.

\begin{conjecture}[Kerr stability conjecture]
The maximal Cauchy development of any initial data set for EVE, that is sufficiently close to a subextremal Kerr initial data in a suitable sense, has a complete future null infinity and a domain of outer communication\footnote{The domain of outer communication is the complement of the black hole region.} which is asymptotic to a nearby member of the subextremal Kerr family. 
\end{conjecture}

For an in-depth introduction to the Kerr stability conjecture, see for example \cite{KS:Brief}. The state of the art on this conjecture is its resolution in the slowly rotating case, i.e., $|a|/m\ll 1$, in the series of works \cite{GCM1} \cite{GCM2} \cite{KS:Kerr} \cite{GKS} \cite{Shen}. The  restriction on the size of $a$ is connected to the proof of energy-Morawetz estimates in \cite{GKS}, and a complete resolution of the Kerr stability conjecture will require in particular to prove energy-Morawetz estimates for the scalar wave equation and for Teukolsky equations on perturbations of any subextremal Kerr background.

Our focus in the present work is to initiate the first step in this program, i.e., to derive energy-Morawetz estimates for the scalar wave equation on perturbations, compatible with the ones in \cite{KS:Kerr}, of any subextremal Kerr background. Before stating our main result, we first start with a review of energy-Morawetz estimates for the scalar wave equation on asymptotically flat black hole backgrounds.

 %%%%%%%%%%%%%%%%%%%%%%%%%%%%%%%%%%%%%%%%%%

\subsection{State of the art on energy-Morawetz {estimates} for scalar wave equations}

%%%%%%%%%%%%%%%%%%%%%%%%%%%%%%%%%%%%%%%%%%

In this section, we review the literature concerning the derivation of energy-Morawetz estimates for the scalar wave equation on Kerr and perturbations of Kerr. We start by discussing the influence of the geometry of Kerr spacetimes on the decay of scalar waves.

%%%%%%%%%%%%%%%%%%%%%%%%%%%%%%%%%%%%%%%%

\subsubsection{Influence of the geometry of Kerr spacetimes on the decay of scalar waves}
\lab{sec:influencegeomKerrdecaywaves}

%%%%%%%%%%%%%%%%%%%%%%%%%%%%%%%%%%%%%%%%

The simplest toy model for the Kerr stability conjecture consists in deriving decay estimates in the domain of outer communication, i.e., in $r>r_+$, for solutions $\psi$ to the scalar wave equation on Kerr
\bea\lab{eq:scalarwaveeqonKerrwithRHS:F:intro}
\square_{\gam}\psi = 0.
\eea
Now, solutions $\psi$ to \eqref{eq:scalarwaveeqonKerrwithRHS:F:intro} decay in $r>r_+$ if scalar waves leave the domain of outer communication, i.e., if they travel towards null infinity, or enter the black hole region $r<r_+$. In particular, this behavior takes place at least in the following two favorable regions of the Kerr spacetime: 
\begin{itemize}
\item \textit{The asymptotically flat region.} This is the region $r\geq R$, for $R\gg m$ large enough, where the Kerr metric is close to the Minkowski metric. 

\item \textit{The redshift region.} This is the region $|r-r_+|\leq r_+\dred$ where $0<\dred\les \sqrt{m^2-a^2}$ is small enough, i.e., the redshift region is close to the event horizon.
\end{itemize}

On the other hand, the complement  of the redshift and{asymptotically flat {regions} contains two unfavorable regions from the point of view of the decay of  solutions to \eqref{eq:scalarwaveeqonKerrwithRHS:F:intro}:
\begin{itemize}
\item \textit{Trapping region.} This is a finite {co-dimension} hypersurface of the cotangent bundle of the spacetime which is spanned by trapped null geodesics, i.e., null geodesics neither going inside the black hole nor towards null infinity. As high-frequency waves tend to travel along null geodesics, such trapped null geodesics are potential obstructions to decay.  

\item \textit{Ergoregion and superradiant frequencies.} In view of \eqref{eq:expressionofKerrmetricinBLcoordinates:intro}, $\pr_t$ is a Killing vector field, which thus generates a conserved energy by Noether's theorem. Since $(\gam)_{tt}=-\frac{\De-a^2\sin^2\th}{|q|^2}$, $\pr_t$ is timelike only outside of the ergoregion $r\leq m+\sqrt{m^2-a^2\cos^2\th}$ so that the corresponding conserved energy is not coercive for a set of frequencies of the co-tangent bundle called superradiant, leaving the possibility of the growth of a coercive energy.
\end{itemize}

Finally, two properties of trapped and superradiant frequencies allow to mitigate their potential obstructions to decay of waves:
\begin{enumerate}
\item The trapping in Kerr is normally hyperbolic, i.e., unstable; null geodesics off from the trapped hypersurface escape fast towards null infinity or inside the black hole region. 
\item There is no overlap between superradiant and trapped frequencies. This can simply be observed  in physical space\footnote{For example, in Schwarzschild, i.e., for $a=0$, trapped frequencies are located at $r=3m$ while there are no superradiant frequencies.} in the range $|a|<\frac{m}{\sqrt{2}}$, but requires to go to frequency space in the range $\frac{m}{\sqrt{2}}\leq |a|<m$\, see Figures \ref{fig:penrosediagramofKerr1} and \ref{fig:penrosediagramofKerr2} for the corresponding Penrose diagrams.
\end{enumerate}

\begin{figure}[htbp]
  \begin{center}
\begin{tikzpicture}[scale=1.1]
\tikzstyle{every node}=[font=\Small]
  \draw[color=black] (-3.2,1.84) arc(300:324.65:10 and 7.5);
  %theleftboundary$\AA$ofthespacetime
      \draw[dashed, color=red] (0.05,3.95) -- (3.2,0.8);
      %nullinfinity
         \draw[] (0,4) circle (0.05);
         %circleof$i_+$
          \draw[color=red!100] (-0.02, 3.96) arc(145:167:10 and 9);
          %r=2m
%               \fill[red!50] (-0.01, 3.96) arc(145:165.7:10 and 9) arc(234:186.8:3.53 and 1.35) arc(301.8:324.65:10 and 7.5);
%    \fill[red!50] (-0.01, 3.96) arc(137:160:10 and 9) arc(234:186.8:2.73 and 1.35) arc(301.8:325.8:10.48 and 7.5);
 \fill[red!50] (-0.03,3.97)--(-2.6,1.4)  arc(211:186.9:2.73 and 1.35) arc(303:325.65:10.392 and 7.5);
   \fill[orange!50] (-0.03,3.97)  arc(137:158.5:10 and 9) arc(265.65:254.4:3 and 8);
     %leftpart of ergoregion
   \fill[yellow!100]  (-0.025, 3.96) arc(145:166:10 and 9) arc(265.65:256.3:3 and 6) arc(158.5:137:10 and 9);
   %rightpart of ergoregion
                  \draw[dashed, color=red] (-0.03,3.97)--(-3.2,0.8) ;
        %eventhorizon
        \draw[color=blue,dashed] (-0.03, 3.97) arc(137:160:10 and 9);
        %rightboundaryofredshiftregion
        \draw[color=blue,dashed] (0.025, 3.95) arc(45:14.8:5 and 7);
        %leftboundaryofAFregion
        \draw[color=blue] (0.025, 3.95) arc(40:14.5:5 and 8);
                %rightboundaryoftrapping
        \draw[color=blue] (0, 3.95) arc(149.84:167.7:10 and 11.15);
                %leftboundaryoftrapping                              
             \fill[blue!70] (0.02, 3.95) arc(40:15.4:5 and 8) arc(290:255:3.5 and 3) arc(167:149.84:10 and 11.15) -- cycle;
             %trappingregion
           \fill[purple!40] (0.03, 3.95) arc(45:16.7:5 and 7) arc(292.5:314.1:3.53 and 2.9)--(0.04,3.96)--cycle;
       %asymptoticallyflatregion    
        \draw[->,color=red] (-2.455,1.115)--(-2.51,1.5); 
        %arrowpointtoHH_+ 
            \draw[->,color=red] (-1.45,0.73)--(-1.51,1); 
        %arrowpointtor=2m
        %        \draw[->,color=blue] (-2.5,0.95)--(-2,1.15); 
%arrowtorightboundaryofredshiftregion
%        \draw[->] (1.8,3)--(0.70,2.75);     
%arrowpointtonullinfinity
        \node[color=red, font=\Tiny] at (-2.4,1) {$\HH_+$};
         \node[rotate=32, font=\tiny] at (-2.25,2.48) {$\AA$};
     \node[font=\tiny] at (0.1,4.2) {$i_+$};
       \node[rotate=20, font=\tiny] at (-1.85,2.22) [align=center]{redshift\\region};
        \node[rotate=20, font=\tiny] at (-1.85,1.55) [align=center]{ergoregion};
       \node[rotate=315, font=\tiny]  at (1.95,2.3) {$\II_+$};
        \node[color=red, font=\Tiny] at (-1.4,0.67) {$r=2m$};
          \node[font=\tiny] at (0,1.6) [align=center]{trapping\\region};
        \node[rotate=315, font=\tiny] at (1.45,2.15) [align=center]{Asymptotically flat\\ region};  
\end{tikzpicture}
\end{center}
\caption{\footnotesize{Penrose diagram of Kerr for $|a|<m/\sqrt{2}$: the redshift region consists of the red and orange parts, the ergoregion consists of the orange and yellow parts, the trapping region consists of the blue part, and the asymptotically flat region consists of the purple region. }}
\lab{fig:penrosediagramofKerr1}
\end{figure}
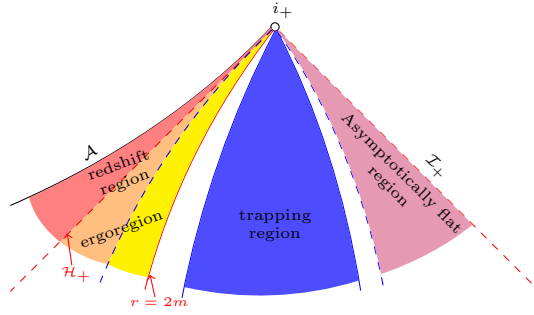

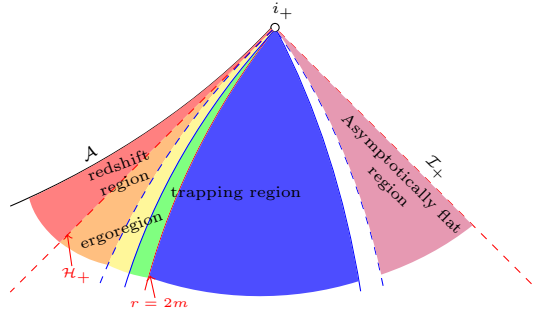
\begin{figure}[htbp]
  \begin{center}
\begin{tikzpicture}[scale=1.1]
\tikzstyle{every node}=[font=\Small]
  \draw[color=black] (-3.2,1.84) arc(300:324.65:10 and 7.5);
  %theleftboundary$\AA$ofthespacetime
      \draw[dashed, color=red] (0.05,3.95) -- (3.2,0.8);
      %nullinfinity
         \draw[] (0,4) circle (0.05);
         %circleof$i_+$
          \draw[color=red!100] (-0.02, 3.96) arc(145:167:10 and 9);
 \fill[red!50] (-0.03,3.97)--(-2.6,1.4)  arc(211:186.9:2.73 and 1.35) arc(303:325.65:10.392 and 7.5);
  \fill[orange!50] (-0.03,3.97)  arc(137:158.5:10 and 9) arc(265.65:254.4:3 and 8);
     %leftpart of ergoregion
   \fill[green!50]  (-0.025, 3.96) arc(145:166:10 and 9) arc(265.65:261.2:3 and 6) arc(164.05:142:10 and 8.5);
   %rightpart of ergoregion or middle part of trapping
      \fill[yellow!50]  (-0.025, 3.96) arc(142:164.28:10 and 8.5) arc(261.2:256.3:3 and 6) arc(158.5:137:10 and 9);
   %middle part of ergoregion
                  \draw[dashed, color=red] (-0.03,3.97)--(-3.2,0.8) ;
        %eventhorizon
        \draw[color=blue,dashed] (-0.03, 3.97) arc(137:160:10 and 9);
        %rightboundaryofredshiftregion
        \draw[color=blue,dashed] (0.025, 3.95) arc(45:14.8:5 and 7);
        %leftboundaryofAFregion
        \draw[color=blue] (0.025, 3.95) arc(40:14.5:5 and 8);
                %rightboundaryoftrapping
        \draw[color=blue] (-0.02, 3.95) arc(142:165.4:10 and 8.5);
                %leftboundaryoftrapping                              
             \fill[blue!70] (0.02, 3.95) arc(40:15.4:5 and 8) arc(290:247.55:3.5 and 3) arc(165.95:145:10 and 9) -- cycle;
             %rightpartoftrappingregion
           \fill[purple!40] (0.03, 3.95) arc(45:16.7:5 and 7) arc(292.5:314.1:3.53 and 2.9)--(0.04,3.96)--cycle;
       %asymptoticallyflatregion    
        \draw[->,color=red] (-2.455,1.115)--(-2.51,1.5); 
        %arrowpointtoHH_+ 
            \draw[->,color=red] (-1.45,0.73)--(-1.51,1); 
        \node[color=red, font=\Tiny] at (-2.4,1) {$\HH_+$};
         \node[rotate=32, font=\tiny] at (-2.25,2.48) {$\AA$};
     \node[font=\tiny] at (0.1,4.2) {$i_+$};
       \node[rotate=20, font=\tiny] at (-1.85,2.22) [align=center]{redshift\\region};
        \node[rotate=20, font=\tiny] at (-1.85,1.55) [align=center]{ergoregion};
       \node[rotate=315, font=\tiny]  at (1.95,2.3) {$\II_+$};
        \node[color=red, font=\Tiny] at (-1.4,0.67) {$r=2m$};
          \node[font=\tiny] at (-0.47,2) [align=center]{trapping region};
        \node[rotate=315, font=\tiny] at (1.45,2.15) [align=center]{Asymptotically flat\\ region};  
\end{tikzpicture}
\end{center}
\caption{\footnotesize{Penrose diagram of Kerr for {$m/\sqrt{2}<|a|<m$}: the redshift region consists of the red and orange parts, the ergoregion consists of the orange, yellow and green  parts, the trapping region consists of the green and blue parts, and the asymptotically flat region consists of the purple region. In particular, the ergoregion and trapping region intersect with each other in physical space.}}
\lab{fig:penrosediagramofKerr2}
\end{figure}

In the next section, we present a robust strategy to derive quantitative decay estimates for solutions $\psi$ to \eqref{eq:scalarwaveeqonKerrwithRHS:F:intro} which takes into account the above geometric features of subextremal Kerr spacetimes.

%%%%%%%%%%%%%%%%%%%%%%%%%%%%%%%%%%%%%%%%

\subsubsection{Derivation of quantitative decay estimates for scalar waves}
\lab{sec:strategytogetdecayfromenergyMorawetzredshiftandrp:intro}

%%%%%%%%%%%%%%%%%%%%%%%%%%%%%%%%%%%%%%%%

The first uniform boundedness result for the scalar wave equation on a  Schwarzschild background has been obtained by Kay-Wald  \cite{KW87}. In Kerr, the first result is due to Whiting \cite{Whiting}, who proved the absence of exponentially growing modes in the full subextremal range. To go beyond these initial results and derive quantitative decay for solutions to the wave equation, the following robust strategy has emerged: 
\begin{enumerate}
\item \textit{Energy-Morawetz estimates.} This step, which has been introduced initially in the breakthrough work of Blue-Soffer \cite{BS03}, must necessarily come first, as it is the only step that is implemented on a causal region. The resulting estimate yields weak decay which is {square} integrable in time, with weights in $r$ degenerating both at spatial infinity and at the event horizon. Moreover, the weights  in $r$ in front of all first order derivatives except $\pr_r$  degenerate also in the trapping region.  

\item \textit{Redshift estimates.} The degeneracy at the event horizon in the first step is removed thanks to an estimate, introduced by Dafermos-Rodnianski  \cite{DR09}, in the redshift region. 

\item \textit{$r^p$-weighted method.} This step, introduced by Dafermos-Rodnianski in \cite{DR10}, relies on estimates in the region $r\geq R$ with $R$ large enough. First, estimates yielding decay in $r$ are derived, and then, using the mean value theorem, this decay in $r$ is traded for decay in time hence concluding the derivation of quantitative decay estimates. 
\end{enumerate}

The last two steps are performed in the simpler redshift and faraway regions and can easily be shown to apply to the full subextremal range and for perturbations of Kerr as well. On the other hand, the first step, on the derivation of  energy-Morawetz estimates, is by far the most demanding as it must deal with the entire domain of outer communication, and in particular with both trapped and superradiant frequencies. From now on, we review the state of the art on the derivation of energy-Morawetz estimates in Kerr and perturbations of Kerr, starting with the particular case of Schwarzschild.

\begin{remark}
Let us mention an alternative approach for deriving decay estimates on black hole spacetimes, based on a blend of spectral and microlocal methods. Originally designed  to prove decay estimates for wave equations on Kerr-de Sitter black holes, it led to   the seminal proof of the nonlinear stability of slowly rotating Kerr-de Sitter black holes by Hintz-Vasy \cite{HVas}. It has since been extended to the asymptotically flat setting, first in slowly rotating Kerr by H\"afner-Hintz-Vasy \cite{HHV} and then in {weakly charged slowly rotating} Kerr-Newman by He \cite{He}. We refer the reader to \cite{Hafner:review} for a review of this approach.
\end{remark}

%%%%%%%%%%%%%%%%%%%%%%%%%%%%%%%%%%%%%%%%

\subsubsection{Decay for scalar wave equations on Schwarzschild}

%%%%%%%%%%%%%%%%%%%%%%%%%%%%%%%%%%%%%%%%

A Morawetz estimate, together with a degenerate conserved energy estimate based on the causal Killing vectorfield $\pr_t$, is first derived by Blue-Soffer in \cite{BS03} using a radial vectorfield linearly degenerating at the trapping radius $r=3m$. It is then used in \cite{BS03} and later in \cite{BS07,BS06} to study the decay of semilinear wave equations in Schwarzschild spacetime. These energy-Morawetz estimates are subsequently reproduced by Dafermos-Rodnianski \cite{DR09}, in which the degeneracy near the event horizon is removed by the so-called redshift estimates via exploiting the positivity of the surface gravity on the event horizon, and by Marzuola-Metcalfe-Tataru-Tohaneanu \cite{MMTT} where the authors additionally prove Strichartz estimates.

%%%%%%%%%%%%%%%%%%%%%%%%%%%%%%%%%%%%%%%%%%

\subsubsection{Energy-Morawetz {estimates} for scalar wave equations on Kerr}
\lab{sec:reviewofenergyMorawetzresultsforscalarwaveon/Kerr:intro}

%%%%%%%%%%%%%%%%%%%%%%%%%%%%%%%%%%%%%%%%%%

Recall from Section \ref{sec:influencegeomKerrdecaywaves} that, in contrast to the case of Schwarzschild, trapping is not localized in physical space and superradiance is present in Kerr for $0<|a|<m$. Also, superradiant and trapped frequencies overlap in physical space in the range $\frac{m}{\sqrt{2}}\leq |a|<m$, but are well separated in frequency space. 

In the slowly rotating case, i.e., $|a|\ll m$, Dafermos-Rodnianski \cite{DR11} obtained a uniform boundedness result using a frequency decomposition into modes. Soon after, the first energy-Morawetz estimates, as well as weak decay estimates, for scalar fields in slowly rotating Kerr were proved by Tataru-Tohaneanu \cite{Ta-Toh} using microlocal multipliers. Andersson-Blue \cite{AB15} then obtained a new proof using a purely physical space approach based on Carter's Killing $2$-tensor. 

The proof of energy-Morawetz estimates in the full subextremal range has been obtained by Dafermos-Rodnianski-Shlapentokh-Rothman \cite{DRSR} using mode decomposition based on the full separability of the wave equation in Kerr.

%%%%%%%%%%%%%%%%%%%%%%%%%%%%%%%%%%%%%%%%%%

\subsubsection{Energy-Morawetz {estimates for scalar waves} on perturbations of Kerr  with $|a|\ll m$}

%%%%%%%%%%%%%%%%%%%%%%%%%%%%%%%%%%%%%%%%%%

In view of nonlinear applications, it is important to extend the results in Kerr reviewed in Section \ref{sec:reviewofenergyMorawetzresultsforscalarwaveon/Kerr:intro} to perturbations of Kerr. There are few recent works going in that direction. Lindblad-Tohaneanu proved in \cite{LT18,LT20} global existence for small data solutions to a quasilinear wave equation in small perturbations of, first Schwarzschild and then Kerr with $|a|/m\ll 1$, using a microlocal approach adapted to the constant time level sets as in \cite{Ta-Toh}. Dafermos-Holzegel-Rodnianski-Taylor \cite{DHRT22} later  generalized those results to allow the presence of quadratic semilinear terms satisfying the null condition. On the other hand, the analysis for wave equations in perturbations of Kerr beyond the slowly rotating case is an open problem and the focus of the present paper.

\begin{remark}
We have reviewed above the state of the art concerning energy-Morawetz estimates for wave equations in Kerr or in perturbations of Kerr. A closely related topic concerns the derivation of energy-Morawetz estimates for wave equations in Kerr-Newman or in perturbations of Kerr-Newman, where Kerr-Newman is a 3-parameter family of asymptotically flat charged black holes that are stationary solutions to Einstein-Maxwell equations. For a review of this problem, we refer to the introduction of \cite{Giorgi}.
\end{remark}

%%%%%%%%%%%%%%%%%%%%%%%

\subsection{First version of the main result}

%%%%%%%%%%%%%%%%%%%%%%%

The goal of this paper is the derivation of energy-Morawetz estimates for the inhomogeneous scalar wave equation on perturbations of Kerr with $|a|<m$.

Given constants $(a, m)$ with $|a|<m$ and $0<\dhor\ll m-|a|$,  let the spacetime $(\MM, \g)$, whose Penrose diagram is depicted in Figure \ref{fig:penrosediagramofM}, be such that:
\begin{itemize}
\item $\MM=\{(\tau, r, \om)\,/\,\tau\in\Reals,  r_+(1-\dhor)\leq r<+\infty,  \om \in\mathbb{S}^2\}$ is a four dimensional manifold, where $(\tau, r)$ are two coordinates on $\MM$ and $r_+:=m+\sqrt{m^2-a^2}$,

\item $\g$ is a Lorentzian metric on $\MM$ sufficiently close to the Kerr metric\footnote{The closeness of $\g$ to $\g_{a,m}$ will be made precise in Section \ref{subsubsect:assumps:perturbedmetric}.} $\g_{a,m}$,

\item the level sets of $\tau$ are spacelike and asymptotically null as $r\to +\infty$,

\item the boundary $\AA:=\{r=r_+(1-\dhor)\}$ of $\MM$ is spacelike.
\end{itemize}

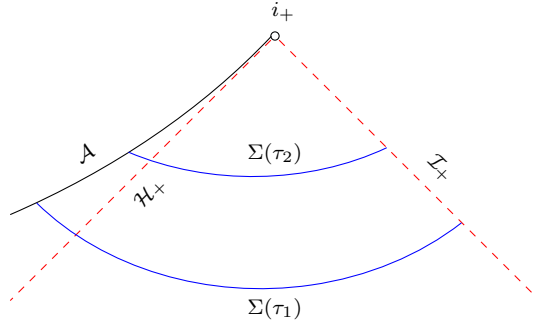
\begin{figure}[htbp]
  \begin{center}
\begin{tikzpicture}[scale=1.1]
\tikzstyle{every node}=[font=\Small]
  \draw[color=black] (-3.2,1.84) arc(300:324.65:10 and 7.5);
    \draw[dashed, color=red] (-0.05,3.95)--(-3.2,0.8) ;
      \draw[dashed, color=red] (0.05,3.95) -- (3.2,0.8);
     \node[rotate=32] at (-2.3,2.58) {$\AA$};
     \node at (0.1,4.3) {$i_+$};
       \node[rotate=315]  at (2.0,2.4) {$\II_+$};
       \node[rotate=45] at (-1.5,2.1) {$\HH_+$};
         \draw[] (0,4) circle (0.05);
  \draw[blue] (-2.88,1.98) arc(220:313.5:3.53 and 2.9);
         \node at (0,0.7) {$\Sigma(\tau_1)$};
    \draw[blue] (-1.77,2.595) arc(240:302.7:3 and 2.2);
      \node at (0,2.58) {$\Sigma(\tau_2)$};

\end{tikzpicture}
\end{center}
\caption{\footnotesize{Penrose diagram of $(\MM, \g)$. $\Si(\tau_1)$ and $\Si(\tau_2)$ are two spacelike and asymptotically null level hypersurfaces of a function $\tau$, and $\AA=\{r=r_+(1-\dhor)\}$ is spacelike.}}\lab{fig:penrosediagramofM}
\end{figure}

Then, we consider the Cauchy problem for the inhomogeneous scalar wave equation on $(\MM, \g)$
\begin{align}
\label{intro:eq:scalarwave}
\left\{
             \begin{array}{rl}
            \Box_{\g} \psi &= F, \qquad \text{on }\,\,\MM,\\
(\psi, N_{\Sigma(\tau)}\psi)\vert_{\Sigma(\tau_1)}&=(\psi_0, \psi_1),
             \end{array}
           \right.
\end{align}
where $\tau_1\geq 1$ and $N_{\Sigma(\tau)}$ is the future-directed unit normal vector to $\Sigma(\tau)$.

Our main result is the derivation of energy-Morawetz estimates for solutions to \eqref{intro:eq:scalarwave}, where $\g$ is a perturbation of a Kerr metric $\gam$ with $|a|<m$. We provide below a rough version of our main theorem, see Theorem \ref{thm:main} for the precise version.

\begin{theorem}[Main theorem, rough version] 
\label{thm:main:0}
Let $\g$ be a perturbation of a Kerr metric $\gam$ with $|a|<m$ in the sense of Section \ref{subsubsect:assumps:perturbedmetric}. Then, we have for solutions to the wave equation \eqref{intro:eq:scalarwave} the following energy-Morawetz-flux estimates, for any $1\leq\tau_1<\tau_2 <+\infty$ and any given $0<\de\leq 1$,
\bea
\label{intro:MainEnerMora:psi}
\sup_{\tau\in [\tau_1, \tau_2]}\E^{(1)}[{\psi}](\tau) + {\M^{(1)}_\de}[{\psi}](\tau_1,\tau_2) +{\F}^{(1)}[{\psi}](\tau_1, \tau_2) \lesssim \E^{(1)}[\psi](\tau_1)+{\mathcal{N}^{(1)}_\de}[ \psi, F](\tau_1, \tau_2).
\eea
Here, the terms $\E^{(1)}[{\cdot}]( \tau)$, ${\F}^{(1)}[{\cdot}](\tau_1, \tau_2)$ and ${\M^{(1)}_\de}[\cdot](\tau_1,\tau_2)$ are {first-order energy on a constant-$\tau$ hypersurface $\Sigma(\tau)$,  fluxes on both $\AA(\tau_1,\tau_2)$ and $\II_+(\tau_1,\tau_2)$ and Morawetz terms over $\MM(\tau_1,\tau_2)$} that are defined as in Section \ref{subsect:norms}, the term ${\mathcal{N}^{(1)}_\de}[ \psi, F](\tau_1, \tau_2)$ is also defined in Section \ref{subsect:norms}, and the implicit constants in $\lesssim$ are independent of  $\tau_1$ and $\tau_2$, and depend only on the black hole parameters $a$ and $m$, as well as on the constants $\dhor$ and $\de$\footnote{The dependence on $a$, $m$, $\dhor$ and $\de$ will be always compressed in $\lesssim$. Also, note that the constants in $\les$ depend in addition on the constant $\dec$ tied to the assumptions on the closeness of $\g$ to $\gam$ made in Section \ref{subsubsect:assumps:perturbedmetric}.}. 
\end{theorem}

\begin{remark}[Relevance to nonlinear stability of subextremal Kerr spacetimes]
\lab{rem:thm1:metricassumptions}
Concerning the relevance of Theorem \ref{thm:main:0} to the nonlinear stability of subextremal Kerr spacetimes, notice that:
\begin{itemize}
\item A large part of the proof of Kerr stability for $|a|\ll m$ in \cite{GCM1,GCM2,KS:Kerr,GKS,Shen} does not require  the smallness of $|a|/m$. In fact, the results in \cite{GCM1,GCM2,KS:Kerr,Shen} are valid in the full subextremal range, while the assumption $|a|\ll m$ is only needed in \cite{GKS} for the derivation of the main energy-Morawetz estimates for the scalar wave equation, Teukolsky equations, and Bianchi identities in perturbations of Kerr. 

\item The metric assumptions in Theorem \ref{thm:main:0}, i.e., the assumptions for the metric perturbations in Section \ref{subsubsect:assumps:perturbedmetric}, are consistent with the decay estimates for the metric in the proof of the nonlinear stability of Kerr for small angular momentum in \cite{KS:Kerr}.
\end{itemize} 
We thus expect that Theorem \ref{thm:main:0} will play a key role in extending the proof of Kerr stability for $|a|\ll m$ in \cite{GCM1,GCM2,KS:Kerr,GKS,Shen} to the full subextremal range.
\end{remark}

\begin{remark}[High order energy-Morawetz-flux estimates]
\lab{rem:thm1:highorderEMF}
 The extension to higher order derivatives of energy-Morawetz-flux estimates \eqref{intro:MainEnerMora:psi} for equation \eqref{intro:eq:scalarwave} can be derived in the same manner as in our proof and would require metric assumptions for more derivatives than in Section \ref{subsubsect:assumps:perturbedmetric}. In the present paper, we prefer to close our estimates at the minimal regularity level, i.e., at the level of two derivatives for the metric perturbations.
\end{remark}

\begin{remark}[Further decay estimates]\lab{rmk:variousextensionsmaintheorem}
Theorem \ref{thm:main:0} is the main building block for proving the following further decay estimates for solutions to the wave equation \eqref{intro:eq:scalarwave}:
\begin{enumerate}
\item $r^p$-weighted estimates \cite{DR10}  for deriving energy decay estimates in $(\tau, r)$,
\item pointwise decay estimates in $(\tau, r)$,
\end{enumerate}
see the discussion in Section \ref{sec:strategytogetdecayfromenergyMorawetzredshiftandrp:intro}. We do not derive these estimates here as they are fairly standard for perturbations of Kerr in the range $|a|<m$ and focus instead on the proof of the key estimate \eqref{intro:MainEnerMora:psi}.
\end{remark}

%%%%%%%%%%%%%%%%%%%%%%%%%%%

\subsection{Strategy of the proof}
\lab{sec:strategyoftheproofofthemaintheorem:intro}

%%%%%%%%%%%%%%%%%%%%%%%%%%%

In this section, we provide an outline of the strategy of the proof of Theorem \ref{thm:main:0}.

%%%%%%%%%%%%%%%%%%%%%%%%%%%%%

\subsubsection{Extension to a global-in-time problem}

%%%%%%%%%%%%%%%%%%%%%%%%%%%%%

The proof of Theorem \ref{thm:main:0} ultimately relies on the use of microlocal Morawetz and energy multipliers adapted to the $r$-foliation of the spacetime $\MM$, see Section \ref{subsubsect:sect1:conditionalMora}. Now, given that these microlocal multipliers are non local in time, they can only be applied to global in time solutions to the scalar wave equation. In particular, this requires first to extend the local-in-time Cauchy problem \eqref{intro:eq:scalarwave}, i.e., 
\begin{align*}
\left\{
             \begin{array}{rl}
            \Box_{\g} \psi &= F, \qquad \text{on }\,\,\MM(\tau_1, \tau_2),\\
(\psi, N_{\Sigma(\tau)}\psi)\vert_{\Sigma(\tau_1)} &=(\psi_0, \psi_1),
             \end{array}
           \right.
\end{align*}
to a global-in-time problem
\bea
\lab{eq:sect1:globalintimeequation}
\square_{\tilde{\g}}\widetilde{\psi}= \widetilde{F}\quad\textrm{on}\quad\MM,
\eea
which is such that the following properties are satisfied:
\begin{enumerate}
\item $\tilde{\g}$ equals $\g$ in $\MM(\tau_1+1, \tau_2-1)$, coincides with the Kerr metric $\gam$ in $\MM\setminus{\MM}(\tau_1, \tau_2)$ and satisfies the metric perturbation assumptions of Section \ref{subsubsect:assumps:perturbedmetric} as well, 

\item $\widetilde{\psi}$ is smoothly extended by $0$ for $\tau\leq \tau_1-1$, and $\widetilde{\psi}=\psi$ on $\MM(\tau_1+1, \tau_2-1)$,

\item $ \widetilde{F}$ is supported in ${\MM}(\tau_1-1,\tau_2)$. 
\end{enumerate}
Based on this extension, the proof of Theorem \ref{thm:main:0} is then reduced, in Section \ref{subsect:proofofThm4.1}, to the global in time (i.e., for $\tau\in\mathbb{R}$) first-order energy-Morawetz estimates of Theorem \ref{th:main:intermediary} for $\widetilde{\psi}$.

From now on, and until the end of Section \ref{sec:strategyoftheproofofthemaintheorem:intro}, we will only consider $(\tilde{\g}, \widetilde{\psi}, \widetilde{F})$. Thus, without any confusion, we may denote $(\tilde{\g}, \widetilde{\psi}, \widetilde{F})$ simply by $(\g, \psi, F)$.

%%%%%%%%%%%%%%%%%%%%%%%%%%%%%

\subsubsection{Second reduction}

%%%%%%%%%%%%%%%%%%%%%%%%%%%%%

Having reduced the proof of Theorem \ref{thm:main:0} to the global in time energy-Morawetz estimates of Theorem \ref{th:main:intermediary}, we now reduce the proof of Theorem \ref{th:main:intermediary} to the following two energy-Morawetz estimates:
\begin{enumerate}
\item global microlocal energy-Morawetz estimates that are conditional on lower order derivatives, see Theorem \ref{th:mainenergymorawetzmicrolocal},

\item global lower order energy-Morawetz estimates that lose one derivative,  see Proposition \ref{prop:energymorawetzmicrolocalwithblackbox}.
\end{enumerate}
The proof of the global in time energy-Morawetz estimates of Theorem \ref{th:main:intermediary} then easily follows from these above two energy-Morawetz estimates, see Section \ref{sec:proofofTheoremth:main:intermediary}.

\begin{remark}
The above strategy for the proof of Theorem \ref{th:main:intermediary} consists in reducing it to two estimates, the first one being  conditional on lower derivatives, and the second one controlling lower order derivatives at the expense of a loss of one derivative. This is reminiscent of the combination of 
\begin{itemize}
\item high frequency estimates, 

\item decay estimates for lower order derivatives by throwing the quasilinear term on the RHS,
\end{itemize} 
in the proof of the nonlinear stability of slowly rotating Kerr-de Sitter black holes by Hintz-Vasy \cite{HVas}. It is in fact  even closer to the way bootstrap assumptions on decay and energy are recovered in the new proof of that result by Fang in \cite{Fang1}.
\end{remark}

%%%%%%%%%%%%%%%%%%%%%%%%%%%%%%%%%%%%%%%%%

\subsubsection{Derivation of global conditional microlocal energy-Morawetz estimates}
\label{subsubsect:sect1:conditionalMora}

%%%%%%%%%%%%%%%%%%%%%%%%%%%%%%%%%%%%%%%%%

To complete the proof of Theorem \ref{th:main:intermediary}, and hence of Theorem \ref{thm:main:0}, it remains to prove Theorem \ref{th:mainenergymorawetzmicrolocal} and Proposition \ref{prop:energymorawetzmicrolocalwithblackbox}. We first outline the proof of Theorem \ref{th:mainenergymorawetzmicrolocal}, i.e., the proof of global microlocal energy-Morawetz estimates conditional on lower order derivatives.

First, in order to define our microlocal Morawetz and energy multipliers, we introduce a microlocal calculus adapted to the $r$-foliation of the spacetime $\MM$ in Section \ref{sect:microlocalcalculus}. The operators are differential in $r$ and pseudo-differential w.r.t. the tangential directions to the level sets $H_r$ of $r$. Also,  we rely on the Weyl quantization as the good properties of the corresponding symbolic calculus for commutators, composition and adjoint are convenient to decompose our microlocal energy-Morawetz estimates in:
\begin{itemize} 
\item main terms which enjoy suitable coercivity properties,

\item lower order derivatives terms which can be thrown on the RHS in view of the fact that our desired microlocal energy-Morawetz estimates are conditional on lower order derivatives.
\end{itemize}

Next, in order to prove our global microlocal energy-Morawetz estimates conditional on lower order derivatives, we decompose $\MM$ into three regions, see Figure \ref{fig:separatedregion1}:
\begin{enumerate}
\item the region $r\leq r_+(1+2\dred)$, with $\dred\ll 1-|a|/m$, on which we derive (physical space) redshift estimates, 
\item the region $r_+(1+\dhor')\leq r\leq R$, with $0<\dhor'\ll \dred$ and $R\gg 20 m$,  where we derive microlocal energy-Morawetz estimates conditional on lower order terms,

\item the region $r\geq R$, where we derive physical space energy-Morawetz estimates.
\end{enumerate}

\begin{figure}[htbp]
  \begin{center}
\begin{tikzpicture}[scale=1.1]
\tikzstyle{every node}=[font=\Small]
  \draw[color=black] (-3.2,1.84) arc(300:324.65:10 and 7.5);
      \draw[dashed, color=red] (0.05,3.95) -- (3.2,0.8);
     \node[rotate=32] at (-2.3,2.58) {$\AA$};
     \node at (0.1,4.3) {$i_+$};
       \node[rotate=315]  at (2.0,2.4) {$\II_+$};
         \draw[] (0,4) circle (0.05);
  
          \draw[color=red!100] (-0.01, 3.946) arc(145:167:10 and 9);
               \fill[red!50] (-0.01, 3.96) arc(145:165.7:10 and 9) arc(234:186.8:3.53 and 1.35) arc(301.8:324.65:10 and 7.5);
        \draw[color=blue,dashed] (-0.02, 3.96) arc(137:160:10 and 9);
        \node[rotate=32] at (-1.80,2.15) [align=center]{redshift\\ estimates};
%        \draw[color=red] (-0.02, 3.96) arc(137:160:10 and 9);
       
        \node at (-0.28,2.4) [align=center]{microlocal \\ Morawetz\\estimates};
       
        \fill[purple!40] (0.03, 3.95) arc(45:17.5:5 and 7) arc(292.5:313.22:3.53 and 2.9)--(0.04,3.96)--cycle;
                \draw[black] (-2.88,1.98) arc(220:313.5:3.53 and 2.9);
         \node at (-0.2,1.2) {$\Sigma(\tau_1-1)$};
        \node[rotate=315] at (1.9,1.5) [align=center]{EMF estimates\\ near infinity};  
         \draw[color=blue,dashed] (0.02, 3.95) arc(45:16:5 and 7);
        \draw[->, color=blue] (1.06,0.8)--(1.26,1.06);       
        \node[color=blue, font=\tiny] at (1.1,0.7) {$r=R$};
        \draw[->,color=red] (-0.95,0.6)--(-1.51,1); 
        \node[color=red, font=\tiny] at (-0.4,0.47) {$r=r_+(1+2\dred)$};
        \draw[->,color=blue] (-2.5,0.95)--(-2,1.15); 
        \node[color=blue, font=\tiny,rotate=332] at (-2.6,0.8) {$r=r_+(1+\dhor')$};
\end{tikzpicture}
\end{center}
\caption{\footnotesize{Global conditional microlocal Morawetz-flux estimates in $\MM$ are a consequence of a combination of the redshift estimates in the red-colored region where $r\leq r_+(1+2\dred)$, the microlocal Morawetz estimates in the region $\MM_{r_+(1+\dhor'), R}$ bounded by two dashed blue-colored curves, and the EMF estimates in the {purple} region where $r\geq R$. }}\lab{fig:separatedregion1}
\end{figure}
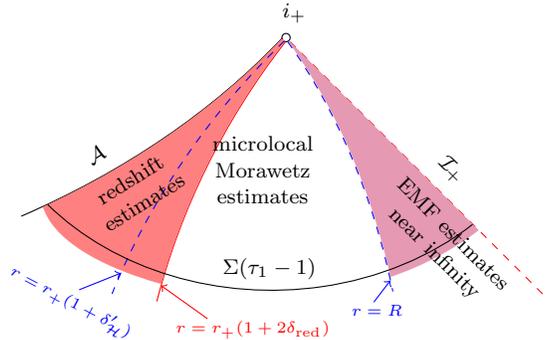

The main part of the proof is the derivation of microlocal energy-Morawetz estimates conditional on lower order terms in the region $r_+(1+\dhor')\leq r\leq R$. It  relies on a microlocal approach:
\begin{itemize}
\item inspired by the one in \cite{Ta-Toh}, see also \cite{LT18,LT20}, where we replace the mixed symbols differential in $\tau$ and microlocal on $\Si(\tau)$ of \cite{Ta-Toh} by the mixed symbols differential in $r$ and microlocal on $H_r$ of Section \ref{sect:PDO:rfoliation},

\item closely following the way of handling high frequencies in phase space in \cite{DRSR} for the inhomogeneous scalar wave equation in a subextremal Kerr spacetime. 
\end{itemize}

%%%%%%%%%%%%%%%%%%%%%%%%%%%%%%%%%%%%%%%%%

\subsubsection{Derivation of global lower order energy-Morawetz estimates with a derivative loss}

%%%%%%%%%%%%%%%%%%%%%%%%%%%%%%%%%%%%%%%%%

To conclude the proof of Theorem \ref{thm:main:0}, we need to prove the global lower order energy-Morawetz estimates with a loss of one derivative of Proposition \ref{prop:energymorawetzmicrolocalwithblackbox}. This relies, in the spirit of \cite{HVas},  on throwing the quasilinear terms to the right-hand side to obtain a wave equation in $(\MM,\gam)$
\beaa
\Box_{\gam}\psi = F - \big(\Box_{\g}\psi-\Box_{\gam}\psi \big),
\eeaa
and then applying the  EMF estimates for linear inhomogeneous wave equations in a subextremal Kerr in \cite{DRSR} as a black box estimate.

%%%%%%%%%%%%%%%%%%%%%%%%

\subsection{Structure of the rest of the paper}

%%%%%%%%%%%%%%%%%%%%%%%%

We introduce in Section \ref{sect:assumpsandintermediary} some preliminaries on the geometry of perturbations of Kerr spacetimes, and provide the definition of energy, Morawetz and flux norms. Some useful basic estimates for the inhomogeneous scalar wave equation in perturbations of Kerr are collected  in Section \ref{sect:basicestimatesforwaveequations}, including the control of error terms arising from the derivation of energy-Morawetz estimates, local energy estimates, redshift estimates, Morawetz estimates near infinity, and conditional high-order energy-Morawetz estimates. A precise version of the main theorem is presented in Section \ref{sect:maintheoremandproof}, along with its proof  under an assumed global in time first-order energy-Morawetz estimate stated in Theorem \ref{th:main:intermediary}. The remaining sections \ref{sect:microlocalcalculus}--\ref{sec:proofofprop:energymorawetzmicrolocalwithblackbox} are then devoted to the proof of Theorem \ref{th:main:intermediary}. To this end, we first introduce a microlocal calculus adapted to the $r$-foliation of the spacetime $(\MM,\g)$ in Section \ref{sect:microlocalcalculus}. Then, we prove Theorem \ref{th:main:intermediary} in Section \ref{sec:proofofth:main:intermediary}, by assuming two global in time energy-Morawetz estimates for the wave equation in perturbations of Kerr: Theorem \ref{th:mainenergymorawetzmicrolocal} and Proposition \ref{prop:energymorawetzmicrolocalwithblackbox}. Theorem \ref{th:mainenergymorawetzmicrolocal}, a global in time microlocal energy-Morawetz estimate conditional on the control of lower order terms, is proved in Section \ref{sect:CondEMF:Dynamic} by making use of microlocal multipliers adapted to the r-foliation. Finally, Proposition \ref{prop:energymorawetzmicrolocalwithblackbox}, which controls the lower order terms at the price of losing one derivative, is shown in Section \ref{sec:proofofprop:energymorawetzmicrolocalwithblackbox}.

%%%%%%%%%%%%%%%%%%%%%%%%%%%%%%%%

\subsection{Acknowledgments}

%%%%%%%%%%%%%%%%%%%%%%%%%%%%%%%%%

The first author is supported by the National Natural Science Foundation of China under Grant No. 12326602 and 12288201, and is much indebted to Fei Wang for her constant  support,  encouragement and grooviest songs. The second author is supported by the ERC grant ERC-2023 AdG 101141855 BlaHSt.

%%%%%%%%%%%%%%%%%%%%%%%%%%%%%%%%%%%%%

\section{Preliminaries}
\label{sect:assumpsandintermediary}

%%%%%%%%%%%%%%%%%%%%%%%%%%%%%%%%%%%%%%

{We discuss the geometry of Kerr spacetimes} in Section \ref{subsect:normalizedcoords} and of the {perturbed Kerr spacetime} in sections \ref{sect:subregionsandhypersurfaces},
\ref{subsect:assumps:perturbedmetric} and
\ref{subsect:nullinf}. Choices of constants that are involved in the {statement of our main result} are made in Section \ref{sec:smallnesconstants}. Energy, Morawetz and flux norms are defined in Section \ref{subsect:norms}, followed by some useful functional inequalities on hypersurfaces in {section}  \ref{subsect:functionalineqSigmatau}.

%%%%%%%%%%%%%%%%%%%%%%%%%%%%%

\subsection{Normalized coordinates in Kerr spacetimes}
\label{subsect:normalizedcoords}

%%%%%%%%%%%%%%%%%%%%%%%%%%%%%

The Kerr metric in Boyer--Lindquist coordinates $(t,r,\th,\phi)$ is given by
\begin{align}
\gam={}& \g_{tt}dt^2 +\g_{rr}dr^2+(\g_{t\phi}+\g_{\phi t})dtd\phi +\g_{\phi\phi}d\phi^2 +\g_{\th\th}d\th^2,
\end{align}
where
\begin{equation}
\begin{split}
\g_{tt}={}&-\frac{\Delta-a^2\sin^2\theta}{|q|^2}, \quad \g_{t\phi}={}\g_{\phi t}=-\frac{2amr\sin\theta}{|q|^2}, \quad \g_{rr}={}\frac{|q|^2}{\Delta},\\
\g_{\phi\phi}={}&\frac{(\R)^2-a^2\sin^2\theta\Delta}{|q|^2}\sin^2\theta, \quad \g_{\th\th}={}|q|^2,
\end{split}
\end{equation}
with 
\bea
\Delta:=r^2-2mr+a^2,\qquad |q|^2:=r^2+a^2\cos^2\th.
\eea
{In particular, $\partial_{t}$ and $\partial_{\phi}$ are Killing vectorfields and $|\det(\gam)|=|q|^4\sin^2\theta$.} The larger root 
\begin{align}
r_+:=m + \sqrt{m^2 -a^2}
\end{align}
of $\Delta=\De(r)$ corresponds to the location of the event horizon. For convenience, we define 
\bea
\mu :=\frac{\Delta}{\R}.
\eea 

The nontrivial components of the inverse metric are
\begin{equation}
\begin{split}
\label{Kerrmetric:inverse:BL}
\g^{tt}={}&-\frac{(\R)^2-a^2\sin^2\theta\Delta}{|q|^2\Delta},  \quad \g^{rr}={}\frac{\Delta}{|q|^2},\\
\g^{\phi\phi}={}&\frac{\Delta-a^2\sin^2\theta}{|q|^2\Delta\sin^2\th}, \quad \g^{\th\th}={}\frac{1}{|q|^2}, \quad \g^{t\phi}={}\g^{\phi t}=-\frac{2amr}{|q|^2\Delta}.
\end{split}
\end{equation}

We define as well a tortoise coordinate $r^*$ by 
$$dr^*= \mu^{-1}dr, \qquad r^*(3m)=0.$$
Without confusion, we call $(t,r^*,\th,\phi)$ the tortoise coordinates and we denote $\pr_{r^*}$ as the coordinate derivative in this tortoise coordinate system.

It is well-known that the metric is singular on the event horizon in both the Boyer--Lindquist and the tortoise coordinates. To extend the Kerr metric beyond the future event horizon, we define the ingoing Eddington--Finkelstein coordinates $(v_+, r,\th,\phi_+)$ by 
\bea
dv_+=dt+\mu^{-1}dr, \quad d\phi_+=d\phi+\frac{a}{\Delta}dr \,\,\, \text{mod } 2\pi.
\eea
The Kerr metric in this coordinate system is 
\begin{align}\lab{eq:KerrmetriciningoingEddigtonFinkelstein}
\gam={}&-\bigg(1-\frac{2mr}{|q|^2}\bigg)dv_+^2 +2dr dv_+ -\frac{4amr\sin^2\th}{|q|^2}dv_+d\phi_+ -2a\sin^2\th dr d\phi_+\nn\\
&+|q|^2 d\th^2+\frac{(\R)^2-a^2\sin^2\theta\Delta}{|q|^2}\sin^2\th d\phi_+^2.
\end{align}

In the following lemma, we introduce coordinates systems, referred to as normalized coordinates, and used in particular in the statement of the main result of this paper. 

\begin{lemma}[Normalized coordinates]
\label{lem:specificchoice:normalizedcoord}
We fix constants $\dhor$ and $\dbl$ such that 
$$0<\dhor\ll \dbl\ll 1-\frac{|a|}{m}.$$ 
There exists a choice of smooth functions $\tmod=\tmod(r)$ and $\phimod=\phimod(r)$ such that the coordinate systems $(\tt, r, x^1_0, x^2_0)$ and $(\tt, r, x^1_p, x^2_p)$, defined respectively on $\th\neq 0, \pi$ and $\th\neq\frac{\pi}{2}$, with 
\bea
\tau=v_+-\tmod, \quad \tphi=\phi_+ -  \phimod, \quad x^1_0=\th, \quad x^2_0=\tphi, \quad x^1_p=\sin\th\cos\tphi, \quad x^2_p=\sin\th\sin\tphi,
\eea
satisfy the following properties:
\begin{enumerate}
\item defining the causal spacetime region $\MM$ and corresponding spacelike boundary $\AA$ by 
\beaa
\bsplit
{}\qquad\MM&:=\big(\{(\tt, r, x^1_0, x^2_0),\,\, \th\neq 0, \pi\}\cup\{(\tt, r, x^1_p, x^2_p),\,\, \th\neq \pi/2\}\big)\cap\{r\geq r_+(1-\dhor)\}, \\ 
{}\qquad\AA&:=\pr\MM=\big(\{(\tt, r, x^1_0, x^2_0),\,\, \th\neq 0, \pi\}\cup\{(\tt, r, x^1_p, x^2_p),\,\, \th\neq \pi/2\}\big)\cap\{r=r_+(1-\dhor)\},
\end{split}
\eeaa
$\MM$ is covered by $(\tt, r, x^1_0, x^2_0)$ and $(\tt, r, x^1_p, x^2_p)$ with the metric components and  inverse metric components being smooth on their respective coordinate patch, 

\item $(\tt, r, x^1_0, x^2_0)$ coincides with Boyer--Lindquist coordinates\footnote{In particular, we have
\beaa
\tmod'(r)=\mu^{-1}, \qquad \phimod'(r)=\frac{a}{\De}\quad\textrm{on}\quad r\in [r_+(1+2\dbl), 12m].
\eeaa}
 in $r\in [r_+(1+2\dbl),  12m]$, 

\item for $r\notin (r_+(1+\dbl), 13m)$, we choose 
\beaa
\begin{split}
\tmod'(r) &=\frac{m^2}{r^2}, \qquad \phimod'(r)=0\quad\textrm{on}\quad r\leq r_+(1+\dbl),\\
\tmod'(r)&=2\mu^{-1}-\frac{m^2}{r^2}, \qquad \phimod'(r)=\frac{2a}{\De} \quad\textrm{on}\quad r\geq 13m,
\end{split}
\eeaa

\item the level sets of $\tt$ in $\MM$ are globally spacelike, transverse to the future event horizon $\HH_+$ and the spacelike boundary $\AA$, and asymptotically null to future null infinity $\II_+$.
\end{enumerate}

Furthermore, the nontrivial inverse metric components in the coordinate system $(\tt, r, \th, \tphi)$ are
\begin{align}
\label{eq:inverse:hypercoord}
\gam^{\tt \tt}={}&\frac{a^2\sin^2\th}{|q|^2} -\frac{2(\R)}{|q|^2}\tmod'+\frac{\Delta}{|q|^2}(\tmod')^2, \,\, \gam^{rr}=\frac{\Delta}{|q|^2}, 
\,\, \gam^{\tt r}=\gam^{r \tt}=\frac{\R}{|q|^2}(1-\mu \tmod'), \nn\\
\gam^{r\tphi} ={}&\gam^{\tphi r} =\frac{a}{|q|^2}-\frac{\Delta}{|q|^2}\phimod', \quad \gam^{\tt\tphi}= \gam^{\tphi \tt}= \frac{a}{|q|^2} (1-\tmod')-\phimod'\frac{\R}{|q|^2} (1-\mu \tmod') ,\nn\\
\gam^{\th\th}={}&\frac{1}{|q|^2}, \quad \gam^{\tphi\tphi}=\frac{1}{|q|^2\sin^2\th}-\frac{2a}{|q|^2}\phimod' +\frac{\Delta}{|q|^2}(\phimod')^2,
\end{align}
and the volume form verifies, with $(x^1, x^2)$ denoting either $(x^1_0, x^2_0)$ or $(x^1_p, x^2_p)$,
\bea\lab{eq:assymptiticpropmetricKerrintaurxacoord:volumeform}
\sqrt{|\det(\gam)|}d\tau dr dx^1dx^2 = |q|^2\sqrt{\det(\mathring{\ga})}d\tau dr dx^1dx^2,
\eea
where $\mathring{\ga}$ denotes the metric on the standard unit 2-sphere.

Finally, for $r\geq 13m$, the inverse metric and metric components  satisfy the following asymptotics on their respective coordinate patch, with $(x^1, x^2)$ denoting either $(x^1_0, x^2_0)$ or $(x^1_p, x^2_p)$,
\bea\lab{eq:assymptiticpropmetricKerrintaurxacoord:1}
\bsplit
\gam^{rr}=&1+O(mr^{-1}), \qquad \gam^{r \tt}=-1+O(m^2r^{-2}),\qquad \gam^{rx^a}=O(mr^{-2}), \\
\gam^{\tt \tt}=&O(m^2r^{-2}), \qquad \gam^{\tau x^a}=O(mr^{-2}),\qquad \gam^{x^ax^b}=\frac{1}{r^2}\mathring{\ga}^{x^ax^b}+O(m^2r^{-4})
\end{split}
\eea
and
\bea\lab{eq:assymptiticpropmetricKerrintaurxacoord:2}
\bsplit
(\gam)_{rr}=&O(m^2r^{-2}), \qquad (\gam)_{r \tt}=-1+O(m^2r^{-2}),\qquad (\gam)_{rx^a}=O(m),\\
(\gam)_{\tt \tt} =& -1+O(mr^{-1}),\qquad (\gam)_{\tau x^a}=O(m),\qquad  (\gam)_{x^ax^b}=r^2\mathring{\ga}_{x^ax^b}+O(m^2).
 \end{split}
\eea
\end{lemma}

\begin{remark}\lab{rmk:phimodprimeisproportionaltoa!!}
Additionally, we may choose $\phimod$ such that 
\beaa
\phimod'(r)=a\phi_{\textrm{mod},0}'(r), \qquad \phi_{\textrm{mod},0}'(r)\geq 0\quad \forall r\in(r_+(1-\dhor), +\infty),
\eeaa
so that $\phimod'(r)$ has the same sign as $a$. From now on, we will assume that our choice of $\phimod$ satisfies this property. In view of \eqref{eq:inverse:hypercoord}, it implies that the inverse metric coefficients $\gam^{\a\b}$ in the normalized coordinates system $(\tau, r, x^1, x^2)$ are invariant under the change $(a, \tphi)\to (-a, -\tphi)$. 
\end{remark}

\begin{proof}
Since $d\tt=dv_+-\tmod' dr$ and $d\tphi=d\phi_+-\phimod'dr$ where $\tmod'=\partial_r \tmod$ and $\phimod'=\partial_r \phimod$, in view of \eqref{eq:KerrmetriciningoingEddigtonFinkelstein}, the Kerr metric in the $(\tau, r, \th, \tphi)$ coordinate system takes the form:
\begin{align}\lab{eq:metricinnormalizedcoordinatesgeneralformula}
\nn\gam={}&-\bigg(1-\frac{2mr}{|q|^2}\bigg)\big(d\tt+\tmod' dr\big)^2 +2dr\big(d\tt+\tmod' dr\big)\\
& -\frac{4amr\sin^2\th}{|q|^2}\big(d\tt+\tmod' dr\big)\big(d\tphi+\phimod'dr\big) -2a\sin^2\th dr(d\tphi+\phimod'dr)\nn\\
&+|q|^2 d\th^2+\frac{(\R)^2-a^2\sin^2\theta\Delta}{|q|^2}\sin^2\th \big(d\tphi+\phimod'dr\big)^2.
\end{align}
We next compute
the components of the inverse metric in terms of the inverse metric components in the Boyer-Lindquist coordinates using $d\tt=dt+(\mu^{-1}-\tmod')dr$ and $d\tphi=d\phi+(\frac{a}{\Delta}-\phimod')dr$:
\beaa
\begin{split}
\g^{\tt \tt}={}&\g^{tt}+ (\mu^{-1}-\tmod')^2 \g^{rr}, \quad \g^{rr}= \g^{rr}, \quad
\g^{\tt r}=\g^{r \tt}= (\mu^{-1}-\tmod') \g^{rr}, \\
  \quad \g^{r\tphi} ={}&\g^{\tphi r}=\left(\frac{a}{\Delta}-\phimod'\right) \g^{rr},\quad  \g^{\tt\tphi}= \g^{\tphi\tt}= \g^{t\phi}+\left(\frac{a}{\Delta} -\phimod'\right)(\mu^{-1}-\tmod')\g^{rr},\\
   \g^{\th\th}={}&\g^{\th\th}, \quad \g^{\tphi\tphi}=\g^{\phi\phi}+\left(\frac{a}{\Delta}-\phimod'\right)^2\g^{rr}.
\end{split}
\eeaa
In view of \eqref{Kerrmetric:inverse:BL} on the values of the inverse metric components in the Boyer--Lindquist coordinates, the expression \eqref{eq:inverse:hypercoord} of the values of these inverse metric components are then verified. From \eqref{eq:inverse:hypercoord} and \eqref{eq:metricinnormalizedcoordinatesgeneralformula}, and by the requirement that $\tmod$ and $\phimod$ are smooth functions of $r$, it is manifest that $\MM$ is covered by $(\tt, r, x^1_0, x^2_0)$ and $(\tt, r, x^1, x^2)$ with the metric components and  inverse metric components being smooth on their respective coordinate patch.  Also, the Jacobian of the change of coordinates  from the Boyer-Lindquist coordinates to the $(\tau, r, \th, \tphi)$ coordinates has determinant 1 so that $\det(\gam)$ coincides with the corresponding quantity in Boyer-Lindquist coordinates which implies \eqref{eq:assymptiticpropmetricKerrintaurxacoord:volumeform}. Moreover, using the fact that $\tmod'(r)=2\mu^{-1}-\frac{m^2}{r^2}$ and $\phimod'(r)=\frac{2a}{\De}$ on $r\geq 13m$, we have
\beaa
\tmod'=2+\frac{4m}{r}+\frac{7m^2}{r^2}+O(m^3r^{-3}), \qquad \phimod'=\frac{2a}{r^2}+O(am r^{-3}),
\eeaa
and hence \eqref{eq:assymptiticpropmetricKerrintaurxacoord:1} and \eqref{eq:assymptiticpropmetricKerrintaurxacoord:2} follow respectively from \eqref{eq:inverse:hypercoord} and \eqref{eq:metricinnormalizedcoordinatesgeneralformula}.

Finally, we consider the level hypersurfaces of $\tt$ in $\MM$. First, in view of the explicit choice of $\tau$ for $r\in [r_+(1-\dhor), +\infty)\setminus((r_+(1+\dbl), r_+(1+2\dbl))\cup(12m, 13m))$, we have
\beaa
\bsplit
\g^{\tt \tt}&=-\frac{m^2(\R)(r^2 - m^2)+r^4(m^2 - a^2(\sin\th)^2)+a^2r^2m^2+2m^5r}{|q|^2r^4} <0\quad\textrm{for}\quad r\leq r_+(1+\dbl),\\
\g^{\tau\tau} & =-\frac{(\R)|q|^2+2m ra^2\sin^2\theta}{|q|^2\Delta}<0\quad\textrm{for}\quad r\in[r_+(1+2\dbl), 12m],\\
\g^{\tau\tau} &= -\frac{m^2(\R)(r^2 - m^2)+r^4(m^2 - a^2(\sin\th)^2)+a^2r^2m^2+2m^5r}{|q|^2r^4} <0 \quad\textrm{for}\quad r\geq 13m,
\end{split}
\eeaa
so that, for $r\in [r_+(1-\dhor), +\infty)\setminus((r_+(1+\dbl), r_+(1+2\dbl))\cup(12m, 13m))$, the level hypersurfaces of $\tt$ are spacelike. Also, the above computation shows, for $r\geq 13m$,   
\beaa
-\frac{2m^2}{r^2}(1+O(m^2r^{-2}))\leq\g^{\tau\tau} = -\frac{(2m^2-a^2(\sin\th)^2)r^4+O(m^4r^2)}{|q|^2r^4}\leq -\frac{m^2}{r^2}(1+O(m^2r^{-2})),
\eeaa
so that the level hypersurfaces of $\tt$ are asymptotically null. 

It remains to consider the region $r\in (r_+(1+\dbl), r_+(1+2\dbl))\cup(12m, 13m)$. In order for the level hypersurfaces of $\tau$ to be strictly spacelike there, we need
\beaa
\g^{\tt\tt}=\frac{a^2\sin^2\th}{|q|^2} -\frac{2(\R)}{|q|^2}\tmod'+\frac{\Delta}{|q|^2}(\tmod')^2<0.
\eeaa
Thus, $\tmod'$ shall satisfy in the region $r\in (r_+(1+\dbl), r_+(1+2\dbl))\cup(12m, 13m)$ the following 
\bea\lab{eq:conditionontmodprimetobespacelike}
\frac{r^2+a^2 -\sqrt{(r^2+a^2)^2 -a^2\sin^2\th \Delta}}{\Delta}<\tmod'<\frac{r^2+a^2 +\sqrt{(r^2+a^2)^2 -a^2\sin^2\th \Delta}}{\Delta}.
\eea
Therefore, we extend smoothly $\tmod$ to $r\in (r_+(1+\dbl), r_+(1+2\dbl))\cup(12m, 13m)$ such that \eqref{eq:conditionontmodprimetobespacelike} is satisfied so that the level hypersurfaces of $\tt$ in $\MM$ are globally spacelike and asymptotically null as stated. Finally, we also extend smoothly $\phimod$ to $r\in (r_+(1+\dbl), r_+(1+2\dbl))\cup(12m, 13m)$ which concludes the proof of Lemma \ref{lem:specificchoice:normalizedcoord}.
 \end{proof}
  
Next, we consider the asymptotic of the induced metric on the level sets of $\tau$ in normalized coordinates in the region $r\geq 13m$.
 \begin{lemma}\lab{lemma:specificchoice:normalizedcoord:inducedmetricSitau}
Let $g_{a,m}$ denote the metric induced by $\gam$ on the level sets of $\tau$. Then, we have in the normalized coordinate systems $(r, x^1_0, x^2_0)$ and $(r, x^1_p, x^2_p)$, in $r\geq 13m$, 
\beaa
\bsplit
(g_{a,m})_{rr}&=O(m^2r^{-2}), \qquad (g_{a,m})_{rx^a}=O(m), \qquad (g_{a,m})_{x^ax^b}=O(r^2), \\
g_{a,m}^{rr}&=O(m^{-2}r^2), \qquad g_{a,m}^{rx^a}=O(m^{-1}), \qquad g_{a,m}^{x^ax^b}=O(r^{-2}),
\end{split} 
\eeaa
and
\beaa
\sqrt{\det(g_{a,m})}drdx^1dx^2 &= r\sqrt{2m^2 -a^2\sin^2\th +O(m^3r^{-1})}\sqrt{\det(\mathring{\ga})}drdx^1dx^2, 
\eeaa
with $(x^1, x^2)$ denoting either $(x^1_0, x^2_0)$ or $(x^1_p, x^2_p)$.
\end{lemma}

\begin{proof}
In view of \eqref{eq:metricinnormalizedcoordinatesgeneralformula} and noticing, in  $r\geq 13m$,
\beaa
\tmod'=2+\frac{4m}{r}+\frac{7m^2}{r^2}+O(m^3r^{-3}), \qquad \phimod'=\frac{2a}{r^2}+O(am r^{-3}),
\eeaa
we obtain by a straightforward computation
\beaa
\bsplit
g_{rr}&=\frac{2m^2}{r^2}+O(m^3r^{-3}), \qquad g_{r\th}=0, \qquad g_{r\tphi}=\big(a+O(amr^{-1})\big)\sin^2\th, \\
g_{\th\th}&=r^2(1+O(a^2r^{-2})), \qquad g_{\th\tphi}=0, \qquad g_{\tphi\tphi}=r^2\sin^2\th (1+O(a^2r^{-2})),
\end{split} 
\eeaa
which implies the statement for the asymptotic of the metric coefficients. 

Next, we compute the inverse of the induced metric. To this end, we first compute the minors of the matrix $g_{ij}$ and find in view of the above asymptotic of the induced metric 
\beaa
\bsplit
M_{rr}&=r^4\sin^2\th (1+O(a^2r^{-2})),\qquad M_{r\th}=0,\qquad M_{r\tphi}=-\big(a+O(amr^{-1})\big)r^2\sin^2\th,\\
M_{\th\th}&= \big(2m^2- a^2\sin^2\th+O(m^3r^{-1})\big)\sin^2\th,\qquad M_{\th\tphi}=0,\qquad M_{\tphi\tphi}=  2m^2+O(m^3r^{-1}).
\end{split}
\eeaa
The determinant of $g_{ij}$ is then given by
\beaa
\det(g) &=& g_{rr}M_{rr}-g_{r\th}M_{r\th}+g_{r\tphi}M_{r\tphi}=g_{rr}M_{rr}+g_{r\tphi}M_{r\tphi}\\
&=& r^2\left(2m^2-a^2\sin^2\th +O(m^3r^{-1})\right)\sin^2\th.
\eeaa
Thus, we deduce
\beaa
g^{-1} &=& \frac{1}{\det(g)}\left(\begin{array}{ccc}
M_{rr} & -M_{r\th} & M_{r\tphi}\\
-M_{r\th} & M_{\th\th} & -M_{\th\tphi}\\
M_{r\tphi} & -M_{\th\tphi} & M_{\tphi\tphi} 
\end{array}\right)\\
&=& \left(\begin{array}{ccc}
\frac{r^2(1+O(a^2r^{-2}))}{2m^2-a^2\sin^2\th +O(m^3r^{-1})} & 0 & -\frac{a+O(amr^{-1})}{2m^2-a^2\sin^2\th +O(m^3r^{-1})}\\
0 & \frac{1+O(a^2r^{-2})}{r^2} & 0\\
-\frac{a+O(amr^{-1})}{2m^2-a^2\sin^2\th +O(m^3r^{-1})} & 0 & \frac{2m^2+O(m^3r^{-1})}{r^2(2m^2-a^2\sin^2\th +O(m^3r^{-1}))\sin^2\th}
\end{array}\right)
\eeaa
where we also used 
\beaa
\frac{M_{\th\th}}{\det(g)} = \frac{g_{rr}g_{\tphi\tphi}-(g_{r\tphi})^2}{g_{rr}M_{rr}+g_{r\tphi}M_{r\tphi}}= \frac{g_{rr}g_{\tphi\tphi}-(g_{r\tphi})^2}{g_{rr}g_{\th\th}g_{\tphi\tphi}-g_{r\tphi}g_{r\tphi}g_{\th\th}}=\frac{1}{g_{\th\th}}.
\eeaa
We infer
\beaa
\bsplit
g^{rr}&=\frac{r^2(1+O(mr^{-1}))}{2m^2-a^2\sin^2\th} , \qquad g^{r\th}=0, \qquad g^{r\tphi}=-\frac{a+O(amr^{-1})}{2m^2-a^2\sin^2\th}, \\ 
g^{\th\th}&=\frac{1+O(a^2r^{-2})}{r^2},\qquad g^{\th\tphi}=0, \qquad g^{\tphi\tphi}=\frac{2m^2+O(m^3r^{-1})}{r^2(2m^2-a^2\sin^2\th)\sin^2\th},
\end{split} 
\eeaa
as stated. This concludes the proof of Lemma \ref{lemma:specificchoice:normalizedcoord:inducedmetricSitau}.
\end{proof} 
  
Next, we consider the induced metric on $\AA$.
\begin{lemma}\lab{lemma:specificchoice:normalizedcoord:inducedmetricAA}
Let $(g_{\AA})_{a,m}$ denote the metric induced by $\gam$ on $\AA$. Then,
\beaa
\sqrt{\det((g_{\AA})_{a,m})}d\tau dx^1 dx^2\simeq m\sqrt{\dhor} \sqrt{\det(\mathring{\ga})}d\tau dx^1 dx^2
\eeaa
with $(x^1, x^2)$ denoting either $(x^1_0, x^2_0)$ or $(x^1_p, x^2_p)$.
\end{lemma}

\begin{proof}
In view of \eqref{eq:metricinnormalizedcoordinatesgeneralformula}, we have 
\beaa
g_{\AA}=-\bigg(1-\frac{2mr}{|q|^2}\bigg) d\tt^2  -\frac{4amr\sin^2\th}{|q|^2}d\tt d\tphi +|q|^2 d\th^2+\frac{(\R)^2-a^2\sin^2\theta\Delta}{|q|^2}\sin^2\th d\tphi^2.
\eeaa
We infer
\beaa
\det(g_\AA) &=& |q|^2\left(\left(\frac{2mr}{|q|^2}-1\right)\frac{(\R)^2-a^2\sin^2\theta\Delta}{|q|^2}\sin^2\th - \frac{4a^2m^2r^2\sin^4\th}{|q|^4}\right)\\
&=& \frac{\sin^2\th(-\De)}{|q|^2}\Big((r^2+a^2)(r^2-a^2\sin^2\th+a^2\cos\th^2)+a^4\sin^4\th\Big)
\eeaa
and hence on $\AA$, we have $\det(g_\AA)\les m^2\dhor\sin^2\th$ and 
\beaa
{\det(g_\AA)} &\geq& \frac{\sin^2\th(-\De)}{|q|^2}\Big((m^2+a^2)(m^2-a^2)+a^4\Big)\geq \frac{\sin^2\th(-\De)m^4}{|q|^2}\\
&\gtrsim& m^2\dhor\sin^2\th
\eeaa
which concludes the proof of the lemma.
\end{proof}

%%%%%%%%%%%%%%%%%%%%%%%

\subsection{Choices of constants}
\lab{sec:smallnesconstants}

%%%%%%%%%%%%%%%%%%%%%%%

The following constants  are  involved in  the statement and in the proof of our main result:
\begin{itemize}
\item The constants $m>0$ and $a$, with $|a|<m$, are the mass and the angular momentum per unit mass of the Kerr solution relative to which the perturbation of the metric $\g$ is measured. 

\item The size of the metric perturbation  is measured by $\ep\geq 0$. 

\item The constant $\dhor$ is tied to the boundary of $\MM$ given by $\pr\MM=\AA=\{r=r_+(1-\dhor)\}$. 

\item The constant $\dred$ measures the width of the redshift region.

\item The constant $\dbl$ appears in the construction of normalized coordinates, see Lemma \ref{lem:specificchoice:normalizedcoord}.

\item The constant $\dec$ is tied to decay estimates in $(r, \tau)$ of the perturbed metric coefficients, see Section \ref{subsubsect:assumps:perturbedmetric}. 

\item The constant $\de$ is tied to $r$-weights in the Morawetz norm {$\M_\de[\psi]$}, see \eqref{def:variousMorawetzIntegrals}.

\item The constant $R$ measures the size of the spacetime region $\MM\cap\{r\leq R\}$ on which we derive microlocal energy-Morawetz estimates, see Section \ref{sect:CondEMF:Dynamic}.
\end{itemize}

These  constants are chosen such that \bea\lab{eq:constraintsonthemainsmallconstantsepanddelta}
0<\ep\ll\dhor\ll\dred\ll \dbl \ll 1-\frac{|a|}{m}, \quad \ep\ll \de\leq  1, \quad \ep\ll \dec,\quad \dred\ll \frac{1}{R}\leq \frac{1}{20m}.
\eea

From now on, in the rest of the paper, $\lesssim$ means bounded by a positive constant multiple, with this positive constant depending only on universal constants (such as constants arising from Sobolev embeddings, elliptic estimates,...) as well as the constants 
$$m,\,\, a, \,\, \dhor,\,\, \dred,\,\, \dbl, \,\, \dec, \,\, \de,\,\, R,$$
\textit{but not on} $\ep$. {Also, note that the} constants $\dhor, \dred$ and $\dbl$ can be {chosen} to be only dependent on $m$ and $a$, and that $R$ can be chosen to be only dependent on $m$.

Throughout this paper, ``LHS" and ``RHS" are {abbreviations} for ``left-hand side" and ``right-hand side", respectively,  ``w.r.t." is {an abbreviation} for ``with respect to", ``EMF" is {an abbreviation} for ``energy-Morawetz-flux",  and $\Re(\cdot)$ means taking the real part.

%%%%%%%%%%%%%%%%%%%%%%%%%%%%

\subsection{Subregions and hypersurfaces of $\MM$}
\lab{sect:subregionsandhypersurfaces}

%%%%%%%%%%%%%%%%%%%%%%%%%%%%

Let $(\MM, \g)$ be a four dimensional Lorentzian manifold covered by coordinate systems $(\tt, r, x^1_0, x^2_0)$ and {$(\tt, r, x^1_p, x^2_p)$}, defined respectively on $\th\neq 0, \pi$ and $\th\neq\frac{\pi}{2}$, with 
\beaa
\tau\in\mathbb{R}, \quad r_+(1-\dhor)\leq r<+\infty, \quad x^1_0=\th, \quad x^2_0=\tphi, \quad {x^1_p=\sin\th\cos\tphi, \quad x^2_p=\sin\th\sin\tphi}.
\eeaa

We define a few {subregions and hypersurfaces} of $\MM$. 
 \begin{definition}
 Define the following subregions and hypersurfaces of $\MM$:
 \bsub
 \begin{align}
 \MM(\tt_1,\tt_2):={}&\MM\cap\{\tt_1\leq \tt\leq \tt_2\}, \quad \forall\tau_1<\tau_2,\\
\MM_{r_1,r_2}:={}&\MM\cap \{r_1\leq r\leq r_2\}, \quad \forall r_+(1-\dhor)\leq r_1<r_2,\\
 \Sigma(\tt_1):={}&\MM\cap\{\tt=\tt_1\}, \quad \forall \tt_1\in \Reals,\\
  \Sigma_{r_1,r_2}(\tt_1):={}& \Sigma(\tt_1)\cap \{r_1\leq r\leq r_2\}, \quad \forall \tt_1\in \Reals, \forall r_+(1-\dhor)\leq r_1<r_2,\\
 H_{r_1}:={}&\MM\cap \{r=r_1\},\quad  \forall r_1\geq r_+(1-\dhor), \\
 \AA:={}&\MM\cap\{r=r_+(1-\dhor)\},\\
 \MM_{red}:={}&\MM\cap\{r\leq r_+(1+\dred)\},\\
 \Mtrap:={}&\MM_{r_+(1+2\dbl), 10m},\\ 
 \Mntrap:={}&\MM\setminus\Mtrap.
 \end{align}
 \esub
 \end{definition}

 %%%%%%%%%%%%%%%%%%%%%%%%%%%%%%%%%%%%
 
 \subsection{Assumptions on the perturbed metric and consequences}
 \label{subsect:assumps:perturbedmetric}
 
 %%%%%%%%%%%%%%%%%%%%%%%%%%%%%%%%%%%%

In this section, we introduce the assumptions for the metric perturbations relative to a subextremal Kerr and, under these metric perturbation assumptions, derive further estimates for the metric that are used in this work.

%%%%%%%%%%%%%%%%%%%%%%%%%%%%%%%%%

\subsubsection{Assumptions on the inverse metric perturbation}
 \label{subsubsect:assumps:perturbedmetric}

%%%%%%%%%%%%%%%%%%%%%%%%%%%%%%%%%

In order to introduce below our assumptions on the perturbed metric, we first introduce the notations $\Ga_g$ and $\Ga_b$ which denote any function of the coordinates on $\MM$ satisfying in the coordinates $(\tt,r,x^1_0, x^2_0)$ for $\th\in[\frac{\pi}{4}, \frac{3\pi}{4}]$, and in the coordinates {$(\tt,r,x^1_p, x^2_p)$} for $\th\in [0,\pi]\setminus(\frac{\pi}{3}, \frac{2\pi}{3})$, the following estimates
\bea\lab{eq:decaypropertiesofGabGag}
|\dk^{\leq 2}\Ga_g|\les \ep\min\Big\{r^{-2}\tau^{-\frac{1+\dec}{2}}, \, r^{-1}\tau^{-1-\dec}\Big\}, \qquad |\dk^{\leq 2}\Ga_b|\les \ep r^{-1}\tau^{-1-\dec},
\eea
where $\dec>0$, $\tau\in\mathbb{R}$, $r_+(1-\dhor)\leq r<+\infty$, and where the weighted derivatives $\dk$ are defined by
\bea\lab{eq:defweightedderivative}
\dk:=\{\rd_{\tt}, r\rd_{r}, r\nabla\}\quad\textrm{with}\quad r\nab=\langle \pr_{x^1}, \pr_{x^2}\rangle,
\eea
where $(x^1, x^2)$ denotes from now on either $(x^1_0, x^2_0)$ for $\th\in[\frac{\pi}{4}, \frac{3\pi}{4}]$ or $(x^1_p, x^2_p)$  for $\th\in [0,\pi]\setminus(\frac{\pi}{3}, \frac{2\pi}{3})$.

\begin{remark}
In view of \eqref{eq:decaypropertiesofGabGag}, we note that $\Ga_g$ satisfies the assumptions of $\Ga_b$ and that $r^{-1}\Ga_b$ satisfies the assumptions of $\Ga_g$. Thus, in the rest of the paper, we will systematically replace $\Ga_g+\Ga_b$ by $\Ga_b$ and $r^{-1}\Ga_b+\Ga_g$ by $\Ga_g$.  
\end{remark}

Using the notation $(\Ga_b, \Ga_g)$, we now introduce our assumptions on the perturbed inverse metric.

\begin{assumption}[Inverse metric assumptions]
\label{intro:assump:metric}
Let a subextremal Kerr metric $\gam$ be given and define, in the normalized coordinates $(\tt,r,x^1_0, x^2_0)$ and {$(\tt,r,x^1_p, x^2_p)$}, the inverse metric difference 
\bea
\gcheck^{\a\b}:=\g^{\a\b}-\gam^{\a\b}.
\eea
Then, with $(\Ga_b, \Ga_g)$ verifying \eqref{eq:decaypropertiesofGabGag}, we assume that $\gcheck^{\a\b}$ satisfies the following:
\bea\lab{eq:controloflinearizedinversemetriccoefficients}
\widecheck{\g}^{rr}=r\Ga_b, \quad \widecheck{\g}^{r\tau}=r\Ga_g,\quad \widecheck{\g}^{\tau\tau}=\Ga_g, \quad \widecheck{\g}^{ra}=\Ga_b, \quad \widecheck{\g}^{\tau a}=\Ga_g,\quad \widecheck{\g}^{ab}=r^{-1}\Ga_g.
\eea 
\end{assumption}

\begin{remark}
The decay assumptions \eqref{eq:controloflinearizedinversemetriccoefficients} on $\gcheck^{\a\b}$ are consistent with the decay estimates derived in the proof of the nonlinear stability of Kerr for small angular momentum in \cite{KS:Kerr}.
\end{remark}

The following immediate non-sharp consequence of \eqref{eq:assymptiticpropmetricKerrintaurxacoord:1}, \eqref{eq:controloflinearizedinversemetriccoefficients} and  \eqref{eq:decaypropertiesofGabGag} will be useful\footnote{This is non-sharp only for $\g^{ra}$ which satisfies in fact $\g^{ra}=O(\ep r^{-1}+mr^{-2})$.}
\bea\lab{eq:consequenceasymptoticKerrandassumptionsinverselinearizedmetric}
\begin{split}
\g^{rr}&=O(1), \qquad \g^{r\tau}=O(1), \qquad \g^{ra}=O(r^{-1}),\\
\g^{\tau\tau}&=O(m^2r^{-2}), \qquad \g^{\tau a}=O(mr^{-2}), \qquad \g^{ab}=O(r^{-2}).
\end{split}
\eea

%%%%%%%%%%%%%%%%%%%%%%%%%%%

\subsubsection{Control of the metric perturbation}

%%%%%%%%%%%%%%%%%%%%%%%%%%%

The following lemma provides the control of the perturbed metric coefficients which follows from the assumption \eqref{eq:controloflinearizedinversemetriccoefficients} on the perturbed inverse metric coefficients. 
\begin{lemma}\lab{lemma:controlofmetriccoefficients:bis}
Assume that $\widecheck{\g}^{\a\b}$ verifies \eqref{eq:controloflinearizedinversemetriccoefficients}. Then, $\widecheck{\g}_{\a\b}:=\g_{\a\b}-(\gam)_{\a\b}$ verifies  
\bea\lab{eq:controloflinearizedmetriccoefficients}
\begin{split}
\widecheck{\g}_{rr} &=\Ga_g, \qquad\quad \widecheck{\g}_{r\tau}=r\Ga_g, \qquad\quad \widecheck{\g}_{\tau\tau}=r\Ga_b, \\
\widecheck{\g}_{\tau a} &=r^2\Ga_b, \qquad\, \widecheck{\g}_{ra}=r^2\Ga_g, \qquad\,\,\,\, \widecheck{\g}_{ab}=r^3\Ga_g.
\end{split}
\eea
Also, we have
\bea\lab{eq:controloflinearizedmetriccoefficients:det}
\widecheck{\det(\g)}=\det(\g_{a,m})r^2\Ga_g, \qquad \widecheck{\det(\g)}:=\det(\g)-\det(\g_{a,m}).
\eea
\end{lemma}

\begin{proof}
First, using the following non-sharp consequence of \eqref{eq:controloflinearizedinversemetriccoefficients}, we have
\beaa
\widecheck{\g}^{rr}=r\Ga_b, \quad \widecheck{\g}^{r\tau}=r\Ga_b,\quad \widecheck{\g}^{\tau\tau}=r\Ga_b, \quad \widecheck{\g}^{ra}=\Ga_b, \quad \widecheck{\g}^{\tau a}=\Ga_b,\quad \widecheck{\g}^{ab}=r^{-1}\Ga_b,
\eeaa
which immediately implies, using the asymptotics \eqref{eq:assymptiticpropmetricKerrintaurxacoord:1} for $\gam^{\a\b}$, and  \eqref{eq:assymptiticpropmetricKerrintaurxacoord:2} for $(\gam)_{\a\b}$,  
\beaa
\begin{split}
\widecheck{\g}_{rr} &=r\Ga_b, \qquad \widecheck{\g}_{r\tau}=r\Ga_b, \qquad \widecheck{\g}_{\tau\tau}=r\Ga_b, \\
\widecheck{\g}_{\tau a} &=r^2\Ga_b, \qquad \widecheck{\g}_{ra}=r^2\Ga_b, \qquad \widecheck{\g}_{ab}=r^3\Ga_b.
\end{split}
\eeaa
This already yields the stated estimates for $\widecheck{\g}_{\tau\tau}$ and $\widecheck{\g}_{\tau a}$, and we still need to obtain the stated control of $\widecheck{\g}_{rr}$, $\widecheck{\g}_{r\tau}$, $\widecheck{\g}_{ra}$ and $\widecheck{\g}_{ab}$. 

Next, using again the asymptotic properties \eqref{eq:assymptiticpropmetricKerrintaurxacoord:1} for $\gam^{\a\b}$ and  \eqref{eq:assymptiticpropmetricKerrintaurxacoord:2} for $(\gam)_{\a\b}$,  
we have, using also \eqref{eq:controloflinearizedinversemetriccoefficients}, 
\beaa
0 &=& \g^{\tau\a}\g_{\tau\a}-1\\
&=& \Big(O(m^2r^{-2})+\Ga_g\Big)\widecheck{\g}_{\tau\tau}+\Ga_g +\Big(-1+O(m^2r^{-2})+r\Ga_g\Big)\widecheck{\g}_{r\tau}+r\Ga_g\\
&&+\Big(O(mr^{-2})+\Ga_g)\widecheck{\g}_{\tau a}+\Ga_g.
\eeaa
Since we have obtained above that $\widecheck{\g}_{\tau\tau}=r\Ga_b$ and $\widecheck{\g}_{\tau a}=r^2\Ga_b$, we infer $\widecheck{\g}_{r\tau} \in r\Ga_g$ as stated.

Next, proceeding as above, we have
\beaa
0 &=& \g^{\tau\a}\g_{a\a}\\
&=& \Big(O(m^2r^{-2})+\Ga_g\Big)\widecheck{\g}_{a\tau}+\Ga_g +\Big(-1+O(m^2r^{-2})+r\Ga_g\Big)\widecheck{\g}_{ar}+r\Ga_g\\
&&+\Big(O(mr^{-2})+\Ga_g\Big)\widecheck{\g}_{ab}+r^2\Ga_g.
\eeaa
Since we have obtained above that $\widecheck{\g}_{\tau a}=r^2\Ga_b$ and $\widecheck{\g}_{ab}=r^3\Ga_b$, we infer $\widecheck{\g}_{ra} \in r^2\Ga_g$ as stated.

Next, proceeding as above, we have
\beaa
0 &=& \g^{\tau\a}\g_{r\a}\\
&=& \Big(O(m^2r^{-2})+\Ga_g\Big)\widecheck{\g}_{r\tau}+\Ga_g +\Big(-1+O(m^2r^{-2})+r\Ga_g\Big)\widecheck{\g}_{rr}+r^{-1}\Ga_g\\
&&+\Big(O(mr^{-2})+\Ga_g\Big)\widecheck{\g}_{ra}+\Ga_g.
\eeaa
Since we have obtained above that $\widecheck{\g}_{r\tau} \in r\Ga_g$ and $\widecheck{\g}_{ra} \in r^2\Ga_g$, we infer $\widecheck{\g}_{rr} \in\Ga_g$ as stated.

Next, proceeding as above, we have
\beaa
0 &=& \g^{\a a}\g_{\a b}-\de^a\!_b\\
&=& \Big(O(mr^{-2})+\Ga_g\Big)\widecheck{\g}_{\tau b}+\Ga_g+\Big(O(mr^{-2})+\Ga_b\Big)\widecheck{\g}_{rb}+\Ga_b\\
&&+\Big(r^{-2}(1+O(m^2r^{-2})\mathring{\ga}^{ac}+r^{-1}\Ga_g\Big)\widecheck{\g}_{cb}+r\Ga_g.
\eeaa
Since we have obtained above that $\widecheck{\g}_{\tau a} \in r^2\Ga_b$ and $\widecheck{\g}_{ra} \in r^2\Ga_g$, we infer $\widecheck{\g}_{ab} \in r^3\Ga_g$ as stated. 

Finally, using the above estimates for $\g_{\a\b}$, as well as the asymptotics  \eqref{eq:assymptiticpropmetricKerrintaurxacoord:2} and \eqref{eq:assymptiticpropmetricKerrintaurxacoord:volumeform}, we immediately infer 
\beaa
\widecheck{\det(\g)}=\det(\g_{a,m})r^2\Ga_g, \qquad \widecheck{\det(\g)}=\det(\g)-\det(\g_{a,m}),
\eeaa
as stated. This concludes the proof of Lemma \ref{lemma:controlofmetriccoefficients:bis}.
\end{proof}

The following immediate non-sharp consequence of \eqref{eq:assymptiticpropmetricKerrintaurxacoord:2}, \eqref{eq:controloflinearizedmetriccoefficients} and  \eqref{eq:decaypropertiesofGabGag} will be useful\footnote{This is non-sharp only for $\g_{\tau a}$ which satisfies in fact $\g_{\tau a}=O(\ep r+m)$.}
\bea\lab{eq:consequenceasymptoticKerrandassumptionsinverselinearizedmetric:bis}
\bsplit
\g_{rr}&=O(m^2r^{-2}), \qquad \g_{r \tt}=O(1),\qquad \g_{ra}=O(m),\\
\g_{\tt \tt}&= O(1),\qquad\qquad\,\, \g_{\tau a}=O(r),\qquad \, \g_{ab}=O(r^2).
 \end{split}
\eea

%%%%%%%%%%%%%%%%%%%%%%%%%%%%%%%%%%%%%%%%%%

\subsubsection{Control of the induced metric on $\Si(\tau)$ and $\AA$}

%%%%%%%%%%%%%%%%%%%%%%%%%%%%%%%%%%%%%%%%%%

The following lemma provides the control of the induced metric on $\Si(\tau)$. 
\begin{lemma}\lab{lemma:controllinearizedmetric:inducedmetricSitau}
Let $g$ denote the metric induced by $\g$ on the level sets of $\tau$. Assume that $\widecheck{\g}^{\a\b}$ verifies \eqref{eq:controloflinearizedinversemetriccoefficients}. Then, $\widecheck{g}_{ij}:=g_{ij}-(g_{a,m})_{ij}$ and $\widecheck{g}^{ij}:=g^{ij}-g_{a,m}^{ij}$ verify
\beaa
\bsplit
\widecheck{g}_{rr} &=\Ga_g,  \qquad \widecheck{g}_{ra}=r^2\Ga_g, \qquad \widecheck{g}_{ab}=r^3\Ga_g,\\
\widecheck{g}^{rr} &=r^4\Ga_g,  \qquad \widecheck{g}^{ra}=r^2\Ga_g, \qquad \widecheck{g}^{ab}=\Ga_g.
\end{split}
\eeaa
Also, we have $\widecheck{\det(g)}=r^4\Ga_g$, with $\widecheck{\det(g)}:=\det(g)-\det(g_{a,m})$. 
\end{lemma}

\begin{proof}
Since $g_{ij}=\g_{ij}$ and $(g_{a,m})_{ij}=(\gam)_{ij}$, we have $\widecheck{g}_{ij}=\widecheck{\g}_{ij}$ and hence, we immediately infer from Lemma \ref{lemma:controlofmetriccoefficients:bis}
\beaa
\widecheck{g}_{rr} =\Ga_g,  \qquad \widecheck{g}_{ra}=r^2\Ga_g, \qquad \widecheck{g}_{ab}=r^3\Ga_g,
\eeaa
as stated. Together with the asymptotics for $g_{a,m}^{ij}$ in Lemma \ref{lemma:specificchoice:normalizedcoord:inducedmetricSitau}, this yields
\beaa
\widecheck{g}^{rr} =r^4\Ga_g,  \qquad \widecheck{g}^{ra}=r^2\Ga_g, \qquad \widecheck{g}^{ab}=\Ga_g,
\eeaa
as stated. Also, the above properties of $\widecheck{g}_{ij}$, together with the asymptotics for $(g_{a,m})_{ij}$ in Lemma \ref{lemma:specificchoice:normalizedcoord:inducedmetricSitau}, yields $\widecheck{\det(g)}=r^4\Ga_g$ as stated. This concludes the proof of Lemma \ref{lemma:controllinearizedmetric:inducedmetricSitau}. 
\end{proof}

The following lemma provides the control of the determinant of the induced metric on $\AA$. 
\begin{lemma}\lab{lemma:controllinearizedmetric:inducedmetricAA}
Let $g_\AA$ denote the metric induced by $\g$ on the spacelike hypersurface $\AA$.  Assume that $\widecheck{\g}^{\a\b}$ verifies \eqref{eq:controloflinearizedinversemetriccoefficients}. Then, $\widecheck{\det(g_\AA)}=O(\ep\tau^{-1-\de})$, with $\widecheck{\det(g_\AA)}:=\det(g)-\det(g_\AA)$. 
\end{lemma}

\begin{proof}
In view of the control for $\widecheck{\g}_{\tau\tau}$, $\widecheck{\g}_{\tau a}$ and $\widecheck{\g}_{ab}$ provided by Lemma \ref{lemma:controlofmetriccoefficients:bis}, and the control \eqref{eq:decaypropertiesofGabGag} for $(\Ga_b, \Ga_g)$, we have $\widecheck{g_{\AA}}=O(\ep\tau^{-1-\de})$, where $\widecheck{g_\AA}:=g_{\AA}-(g_{a,m})_\AA$. This immediately yields $\widecheck{\det(g_\AA)}=O(\ep\tau^{-1-\de})$ as stated. 
\end{proof}

%%%%%%%%%%%%%%%%%%%%%%%%%%%%%%%%%%%%%%%%%%

\subsubsection{Further consequences of the metric assumptions}

%%%%%%%%%%%%%%%%%%%%%%%%%%%%%%%%%%%%%%%%%%

In this section, we draw further consequences of the assumption \eqref{eq:controloflinearizedinversemetriccoefficients} on the perturbed inverse metric coefficients. 

\begin{lemma}\lab{lemma:computationofthederiveativeofsrqtg}
Let the 1-form $N_{det}$ be defined by  
\beaa
(N_{det})_\mu:=\frac{1}{\sqrt{|\g|}}\pr_\mu\sqrt{|\g|} - \frac{1}{\sqrt{|\g_{a,m}|}}\pr_\mu\sqrt{|\g_{a,m}|}. 
\eeaa
Then, we have 
\beaa
(N_{det})_r=\dk^{\leq 1}\Ga_g, \qquad (N_{det})_\tau=r\dk^{\leq 1}\Ga_g, \qquad (N_{det})_{x^a}=r\dk^{\leq 1}\Ga_g,
\eeaa 
and 
\beaa
(N_{det})^r=r\dk^{\leq 1}\Ga_g, \qquad (N_{det})^\tau=\dk^{\leq 1}\Ga_g, \qquad (N_{det})^{x^a}=r^{-1}\dk^{\leq 1}\Ga_g.
\eeaa
\end{lemma}

\begin{proof}
First, note that we have in view of \eqref{eq:consequenceasymptoticKerrandassumptionsinverselinearizedmetric}
\beaa
N^r &=& \g^{r\a}N_\a=O(1)N_r+O(1)N_\tau+O(r^{-1})N_{x^a},\\
N^\tau &=& \g^{\tau\a}N_\a=O(1)N_r+O(r^{-2})N_\tau+O(r^{-2})N_{x^a},\\
N^{x^a} &=& \g^{x^a\a}N_\a=O(r^{-1})N_r+O(r^{-2})N_\tau+O(r^{-2})N_{x^a}.
\eeaa
Thus, in order to prove the lemma, it suffices to focus, from now on, on proving the following 
\beaa
(N_{det})_r=\dk^{\leq 1}\Ga_g, \qquad (N_{det})_\tau=r\dk^{\leq 1}\Ga_g, \qquad (N_{det})_{x^a}=r\dk^{\leq 1}\Ga_g.
\eeaa
 
Next, we compute $(N_{det})_\mu$. We have 
\beaa
\frac{2}{\sqrt{|\g|}}\pr_\mu\sqrt{|\g|} &=& \g_{\a\b}\pr_\mu\g^{\a\b}\\
&=& \frac{2}{\sqrt{|\g_{a,m}|}}\pr_\mu\sqrt{|\g_{a,m}|}+\Big((\g_{a,m})_{\a\b}+\gcheck_{\a\b}\Big)\pr_\mu\gcheck^{\a\b}+\gcheck_{\a\b}\pr_\mu(\g^{a,m}_{\a\b})
\eeaa
and hence
\beaa
(N_{det})_\mu &=& \frac{1}{2}\Big((\g_{a,m})_{\a\b}+\gcheck_{\a\b}\Big)\pr_\mu\gcheck^{\a\b}+\frac{1}{2}\gcheck_{\a\b}\pr_\mu(\g_{a,m}^{\a\b}).
\eeaa

Next, we compute $((\g_{a,m})_{\a\b}+\gcheck_{\a\b})\pr_\mu\gcheck^{\a\b}$ and $\gcheck_{\a\b}\pr_\mu(\g^{a,m}_{\a\b})$. Using the asymptotic for large $r$ of the Kerr metric in $(\tau, r, x^1, x^2)$ coordinates given by \eqref{eq:assymptiticpropmetricKerrintaurxacoord:1}  \eqref{eq:assymptiticpropmetricKerrintaurxacoord:2} and the control of the perturbed metric coefficients provided by \eqref{eq:controloflinearizedinversemetriccoefficients} and \eqref{eq:controloflinearizedmetriccoefficients}, we have 
\beaa
((\g_{a,m})_{\a\b}+\gcheck_{\a\b})\pr_\mu\gcheck^{\a\b} &=& O(r^{-2})\pr_\mu\gcheck^{rr}+O(1)\pr_\mu\gcheck^{r\tau}+O(1)\pr_\mu\gcheck^{\tau\tau}\\
&&+O(m)\pr_\mu\gcheck^{rx^a}+(O(m)+O(\ep r))\pr_\mu\gcheck^{\tau x^a}\\
&&+O(r^2)\pr_\mu\gcheck^{x^ax^b}
\eeaa
and
\beaa
\gcheck_{\a\b}\pr_\mu(\g_{a,m}^{\a\b}) &=& \gcheck_{rr}\pr_\mu(O(mr^{-1}))
+\gcheck_{r\tau}\pr_\mu(O(m^2r^{-2}))+\gcheck_{\tau\tau}\pr_\mu(O(m^2r^{-2}))\\
&&+\gcheck_{rx^a}\pr_\mu(O(mr^{-2}))+\gcheck_{\tau x^a}\pr_\mu(O(mr^{-2}))+\gcheck_{x^ax^b}\pr_\mu(O(r^{-2})).
\eeaa

Since $\mathfrak{d}=(r\pr_r, \pr_\tau, \pr_{x^a})$, we infer
\beaa
&& \Big(r((\g_{a,m})_{\a\b}+\gcheck_{\a\b})\pr_r\gcheck^{\a\b},\,\, ((\g_{a,m})_{\a\b}+\gcheck_{\a\b})\pr_\tau\gcheck^{\a\b},\,\, ((\g_{a,m})_{\a\b}+\gcheck_{\a\b})\pr_{x^a}\gcheck^{\a\b}\Big)\\ 
&=& O(r^{-2})\dk\gcheck^{rr}+O(1)\dk\gcheck^{r\tau}+O(1)\dk\gcheck^{\tau\tau}+O(1)\dk\gcheck^{rx^a}+O(r)\dk\gcheck^{\tau x^a}+O(r^2)\dk\gcheck^{x^ax^b}.
\eeaa
Also, since $\pr_\tau$ is Killing for $\g_{a,m}$, and since $\pr_rO(r^{-p})=O(r^{-p-1})$ and  $\pr_{x^a}O(r^{-p})=O(r^{-p})$, we obtain
\beaa
\gcheck_{\a\b}\pr_\tau(\g_{a,m}^{\a\b}) &=& 0,\\
\eeaa
and 
\beaa
&&\Big(r\gcheck_{\a\b}\pr_r(\g_{a,m}^{\a\b}), \,\gcheck_{\a\b}\pr_{x^a}(\g_{a,m}^{\a\b})\Big)\\ 
&=& O(r^{-1})\gcheck_{rr}+O(r^{-2})\gcheck_{r\tau}+O(r^{-2})\gcheck_{\tau\tau}+O(r^{-2})\gcheck_{rx^a}+O(r^{-2})\gcheck_{\tau x^a}+O(r^{-2})\gcheck_{x^ax^b}.
\eeaa
This yields
\beaa
(N_{det})_\tau &=& O(r^{-2})\dk\gcheck^{rr}+O(1)\dk\gcheck^{r\tau}+O(1)\dk\gcheck^{\tau\tau}+O(1)\dk\gcheck^{rx^a}+O(r)\dk\gcheck^{\tau x^a}+O(r^2)\dk\gcheck^{x^ax^b}
\eeaa
and
\beaa
&&\Big(r(N_{det})_r, \, (N_{det})_{x^a}\Big)\\
&=& O(r^{-2})\dk\gcheck^{rr}+O(1)\dk\gcheck^{r\tau}+O(1)\dk\gcheck^{\tau\tau}+O(1)\dk\gcheck^{rx^a}+O(r)\dk\gcheck^{\tau x^a}+O(r^2)\dk\gcheck^{x^ax^b}\\
&&+O(r^{-1})\gcheck_{rr}+O(r^{-2})\gcheck_{r\tau}+O(r^{-2})\gcheck_{\tau\tau}+O(r^{-2})\gcheck_{rx^a}+O(r^{-2})\gcheck_{\tau x^a}+O(r^{-2})\gcheck_{x^ax^b}.
\eeaa
In view of the assumptions \eqref{eq:controloflinearizedinversemetriccoefficients} on $\gcheck^{\a\b}$ and the control for $\gcheck_{\a\b}$ provided by \eqref{eq:controloflinearizedmetriccoefficients}, we infer
\beaa
\Big(r(N_{det})_r, \, (N_{det})_\tau, \, (N_{det})_{x^a}\Big) &=& r\dk^{\leq 1}\Ga_g
\eeaa
as stated. This concludes the proof of Lemma \ref{lemma:computationofthederiveativeofsrqtg}.
\end{proof}

We have the following corollary of Lemma \ref{lemma:computationofthederiveativeofsrqtg}. 
\begin{corollary}\lab{cor:controloflinearizeddivergencecoordvectorfields}
We have
\beaa
\widecheck{\textbf{\textrm{Div}}(\pr_r)}=\dk^{\leq 1}\Ga_g, \qquad \widecheck{\textbf{\textrm{Div}}(\pr_\tau)}=r\dk^{\leq 1}\Ga_g, \qquad \widecheck{\textbf{\textrm{Div}}(\pr_{x^a})}=r\dk^{\leq 1}\Ga_g, \quad a=1,2.
\eeaa
\end{corollary}

\begin{proof}
In view of the definition of $(N_{det})_\mu$ in Lemma \ref{lemma:computationofthederiveativeofsrqtg}, we have 
\beaa
\textbf{\textrm{Div}}(\pr_\mu)=\frac{1}{\sqrt{|\g|}}\pr_\mu\sqrt{|\g|}= \frac{1}{\sqrt{|\g_{a,m}|}}\pr_\mu\sqrt{|\g_{a,m}|}+(N_{det})_\mu
\eeaa
and hence 
\beaa
\widecheck{\textbf{\textrm{Div}}(\pr_\mu)}=(N_{det})_\mu. 
\eeaa
In view of the control of $(N_{det})_\mu$ in Lemma \ref{lemma:computationofthederiveativeofsrqtg}, we deduce
\beaa
\widecheck{\textbf{\textrm{Div}}(\pr_r)}=\dk^{\leq 1}\Ga_g, \qquad \widecheck{\textbf{\textrm{Div}}(\pr_\tau)}=r\dk^{\leq 1}\Ga_g, \qquad \widecheck{\textbf{\textrm{Div}}(\pr_{x^a})}=r\dk^{\leq 1}\Ga_g, \quad a=1,2,
\eeaa
as stated. This concludes the proof of Corollary \ref{cor:controloflinearizeddivergencecoordvectorfields}.
\end{proof}

Next, we provide the control of deformation tensors involved in energy-Morawetz estimates{, where the deformation tensor of a vectorfield $X$ is given by
\bea
\label{def:deformationtensor:lastsect}
{}^{(X)}\pi_{\a\b} :=  \D_{\a}X_{\b} + \D_{\b}X_{\a}=\LL_X\g_{\a\b}.
\eea}

\begin{lemma}\lab{lemma:controlofdeformationtensorsforenergyMorawetz}
The deformation {tensors of $\pr_\tau$ and $\pr_{\tphi}$ satisfy}
\beaa
\begin{split}
{}^{(\pr_\tau)}\pi_{rr}{, {}^{(\pr_{\tphi})}\pi_{rr}} &=\dk^{\leq 1}\Ga_g, \qquad {}^{(\pr_\tau)}\pi_{r\tau}{, {}^{(\pr_{\tphi})}\pi_{r\tau}}=r\dk^{\leq 1}\Ga_g, \qquad {}^{(\pr_\tau)}\pi_{\tau\tau}{, {}^{(\pr_{\tphi})}\pi_{\tau\tau}}=r\dk^{\leq 1}\Ga_b, \\
{}^{(\pr_\tau)}\pi_{\tau a}{, {}^{(\pr_{\tphi})}\pi_{\tau a}} &=r^2\dk^{\leq 1}\Ga_b, \qquad {}^{(\pr_\tau)}\pi_{ra}{, {}^{(\pr_{\tphi})}\pi_{ra}}=r^2\dk^{\leq 1}\Ga_g, \qquad {}^{(\pr_\tau)}\pi_{ab}{, {}^{(\pr_{\tphi})}\pi_{ab}}=r^3\dk^{\leq 1}\Ga_g,
\end{split}
\eeaa
and 
\beaa
\begin{split}
{}^{(\pr_\tau)}\pi^{rr}{, {}^{(\pr_{\tphi})}\pi^{rr}} &=r\dk^{\leq 1}\Ga_b, \qquad {}^{(\pr_\tau)}\pi^{r\tau}{, {}^{(\pr_{\tphi})}\pi^{r\tau}}=r\dk^{\leq 1}\Ga_g, \qquad {}^{(\pr_\tau)}\pi^{\tau\tau}{, {}^{(\pr_{\tphi})}\pi^{\tau\tau}}=\dk^{\leq 1}\Ga_g, \\
{}^{(\pr_\tau)}\pi^{\tau a}{, {}^{(\pr_{\tphi})}\pi^{\tau a}} &=\dk^{\leq 1}\Ga_g, \qquad {}^{(\pr_\tau)}\pi^{ra}{, {}^{(\pr_{\tphi})}\pi^{ra}}=\dk^{\leq 1}\Ga_b, \qquad {}^{(\pr_\tau)}\pi^{ab}{, {}^{(\pr_{\tphi})}\pi^{ab}}=r^{-1}\dk^{\leq 1}\Ga_g.
\end{split}
\eeaa

Also, the perturbed deformation tensor of $\pr_r$ satisfies 
\beaa
\begin{split}
\widecheck{{}^{(\pr_r)}\pi}_{rr} &=r^{-1}\dk^{\leq 1}\Ga_g, \qquad \widecheck{{}^{(\pr_r)}\pi}_{r\tau}=\dk^{\leq 1}\Ga_g, \qquad \widecheck{{}^{(\pr_r)}\pi}_{\tau\tau}=\dk^{\leq 1}\Ga_b, \\
\widecheck{{}^{(\pr_r)}\pi}_{\tau a} &=r\dk^{\leq 1}\Ga_b, \qquad \widecheck{{}^{(\pr_r)}\pi}_{ra}=r\dk^{\leq 1}\Ga_g, \qquad \widecheck{{}^{(\pr_r)}\pi}_{ab}=r^2\dk^{\leq 1}\Ga_g,
\end{split}
\eeaa
and
\beaa
\begin{split}
\widecheck{{}^{(\pr_r)}\pi}^{rr} &=\dk^{\leq 1}\Ga_b, \qquad \widecheck{{}^{(\pr_r)}\pi}^{r\tau}=\dk^{\leq 1}\Ga_g, \qquad \widecheck{{}^{(\pr_r)}\pi}^{\tau\tau}=r^{-1}\dk^{\leq 1}\Ga_g, \\
\widecheck{{}^{(\pr_r)}\pi}^{\tau a} &=r^{-1}\dk^{\leq 1}\Ga_g, \qquad \widecheck{{}^{(\pr_r)}\pi}^{ra}=r^{-1}\dk^{\leq 1}\Ga_b, \qquad \widecheck{{}^{(\pr_r)}\pi}^{ab}=r^{-2}\dk^{\leq 1}\Ga_g.
\end{split}
\eeaa
\end{lemma}

\begin{proof}
We have 
{\beaa
{}^{(\pr_\mu)}\pi_{\a\b} = \LL_{\pr_\mu}\g(\pr_\a, \pr_\b)=\pr_{\mu}(\g_{\a\b}) - \g(\LL_{\pr_\mu}\pr_\a, \pr_\b) -\g(\pr_\a, \LL_{\pr_\mu}\pr_\b)=\pr_{\mu}(\g_{\a\b})
\eeaa}
and hence 
\beaa
{}^{(\pr_\mu)}\pi_{\a\b} = \pr_\mu(\g_{\a\b}), \qquad \widecheck{{}^{(\pr_\mu)}\pi}_{\a\b}=\pr_\mu\widecheck{\g}_{\a\b}.
\eeaa
Since $\pr_\tau((\g_{a,m})_{\a\b})=0${, $\pr_{\tphi}((\g_{a,m})_{\a\b})=0$,} and $\dk=(r\pr_r, \pr_\tau, \pr_{x^a})$, we infer
\beaa
{}^{(\pr_\tau)}\pi_{\a\b}{, {}^{(\pr_{\tphi})}\pi_{\a\b}} = \dk\widecheck{\g}_{\a\b}, \qquad \widecheck{{}^{(\pr_r)}\pi}_{\a\b}=r^{-1}\dk\widecheck{\g}_{\a\b}.
\eeaa
Together with \eqref{eq:controloflinearizedmetriccoefficients}, we deduce 
\beaa
\begin{split}
{}^{(\pr_\tau)}\pi_{rr}{, {}^{(\pr_{\tphi})}\pi_{rr}} &=\dk^{\leq 1}\Ga_g, \qquad {}^{(\pr_\tau)}\pi_{r\tau}{, {}^{(\pr_{\tphi})}\pi_{r\tau}}=r\dk^{\leq 1}\Ga_g, \qquad {}^{(\pr_\tau)}\pi_{\tau\tau}{, {}^{(\pr_{\tphi})}\pi_{\tau\tau}}=r\dk^{\leq 1}\Ga_b, \\
{}^{(\pr_\tau)}\pi_{\tau a}{, {}^{(\pr_{\tphi})}\pi_{\tau a}} &=r^2\dk^{\leq 1}\Ga_b, \qquad {}^{(\pr_\tau)}\pi_{ra}{, {}^{(\pr_{\tphi})}\pi_{ra}}=r^2\dk^{\leq 1}\Ga_g, \qquad {}^{(\pr_\tau)}\pi_{ab}{, {}^{(\pr_{\tphi})}\pi_{ab}}=r^3\dk^{\leq 1}\Ga_g,
\end{split}
\eeaa
and 
\beaa
\begin{split}
\widecheck{{}^{(\pr_r)}\pi}_{rr} &=r^{-1}\dk^{\leq 1}\Ga_g, \qquad \widecheck{{}^{(\pr_r)}\pi}_{r\tau}=\dk^{\leq 1}\Ga_g, \qquad \widecheck{{}^{(\pr_r)}\pi}_{\tau\tau}=\dk^{\leq 1}\Ga_b, \\
\widecheck{{}^{(\pr_r)}\pi}_{\tau a} &=r\dk^{\leq 1}\Ga_b, \qquad \widecheck{{}^{(\pr_r)}\pi}_{ra}=r\dk^{\leq 1}\Ga_g, \qquad \widecheck{{}^{(\pr_r)}\pi}_{ab}=r^2\dk^{\leq 1}\Ga_g,
\end{split}
\eeaa
as stated. 

Next, using \eqref{eq:consequenceasymptoticKerrandassumptionsinverselinearizedmetric}, we have, for a symmetric 2-tensor $\pi$, 
\beaa
\pi^{rr} &=& O(1)\big(\pi_{rr},\,\pi_{r\tau}, \,\pi_{\tau\tau}\big)+O(r^{-1})\big(\pi_{ra}, \pi_{\tau a}\big)+O(r^{-2})\pi_{ab},\\
\pi^{r\tau} &=& O(1)\big(\pi_{rr},\,\pi_{r\tau}\big)+O(r^{-1})\pi_{ra}+O(r^{-2})\big(\pi_{\tau\tau}, \pi_{\tau a}\big)+O(r^{-3})\pi_{ab},\\
\pi^{\tau\tau} &=& O(1)\pi_{rr}+O(r^{-2})\big(\pi_{r\tau}, \pi_{ra}\big)+O(r^{-4})\big(\pi_{\tau\tau}, \pi_{\tau a}, \pi_{ab}\big),\\
\pi^{ra} &=& O(r^{-1})\big(\pi_{rr},\,\pi_{r\tau}\big)+O(r^{-2})\big(\pi_{\tau\tau},\,\pi_{ra}, \,\pi_{\tau a}\big)+O(r^{-3})\pi_{ab},\\
\pi^{\tau a} &=& O(r^{-1})\pi_{rr}+O(r^{-2})\big(\pi_{r\tau},\, \pi_{ra}\big)+O(r^{-4})\big(\pi_{\tau\tau}, \, \pi_{\tau a}, \, \pi_{ab}\big),\\
\pi^{ab} &=& O(r^{-2})\pi_{rr}+O(r^{-3})\big(\pi_{r\tau},\,\pi_{ra}\big)+O(r^{-4})\big(\pi_{\tau\tau},\, \pi_{\tau a}, \,\pi_{ab}\big). 
\eeaa
Together with the above control of ${}^{(\pr_\tau)}\pi_{\a\b}$, ${}^{(\pr_{\tphi})}\pi_{\a\b}$ and $\widecheck{{}^{(\pr_r)}\pi}_{\a\b}$, we obtain 
\beaa
\begin{split}
{}^{(\pr_\tau)}\pi^{rr}{, {}^{(\pr_{\tphi})}\pi^{rr}} &=r\dk^{\leq 1}\Ga_b, \qquad {}^{(\pr_\tau)}\pi^{r\tau}{, {}^{(\pr_{\tphi})}\pi^{r\tau}}=r\dk^{\leq 1}\Ga_g, \qquad {}^{(\pr_\tau)}\pi^{\tau\tau}{, {}^{(\pr_{\tphi})}\pi^{\tau\tau}}=\dk^{\leq 1}\Ga_g, \\
{}^{(\pr_\tau)}\pi^{\tau a}{, {}^{(\pr_{\tphi})}\pi^{\tau a}} &=\dk^{\leq 1}\Ga_g, \qquad {}^{(\pr_\tau)}\pi^{ra}{, {}^{(\pr_{\tphi})}\pi^{ra}}=\dk^{\leq 1}\Ga_b, \qquad {}^{(\pr_\tau)}\pi^{ab}{, {}^{(\pr_{\tphi})}\pi^{ab}}=r^{-1}\dk^{\leq 1}\Ga_g,
\end{split}
\eeaa
and
\beaa
\begin{split}
\widecheck{{}^{(\pr_r)}\pi}^{rr} &=\dk^{\leq 1}\Ga_b, \qquad \widecheck{{}^{(\pr_r)}\pi}^{r\tau}=\dk^{\leq 1}\Ga_g, \qquad \widecheck{{}^{(\pr_r)}\pi}^{\tau\tau}=r^{-1}\dk^{\leq 1}\Ga_g, \\
\widecheck{{}^{(\pr_r)}\pi}^{\tau a} &=r^{-1}\dk^{\leq 1}\Ga_g, \qquad \widecheck{{}^{(\pr_r)}\pi}^{ra}=r^{-1}\dk^{\leq 1}\Ga_b, \qquad \widecheck{{}^{(\pr_r)}\pi}^{ab}=r^{-2}\dk^{\leq 1}\Ga_g,
\end{split}
\eeaa
as stated. This concludes the proof of Lemma \ref{lemma:controlofdeformationtensorsforenergyMorawetz}.
\end{proof}

%%%%%%%%%%%%%%%%%%%%%%%%%%%%%%%%%%%%%

\subsection{Future null infinity of the perturbed {spacetime}}
\label{subsect:nullinf}

%%%%%%%%%%%%%%%%%%%%%%%%%%%%%%%%%%%%%

We start by constructing an auxiliary ingoing optical function $\tauu$ in a subregion of $(\MM, \g)$. 
\begin{lemma}\lab{lemma:constructionoftheingoingopticalfunctiontauu}
There exists an ingoing optical function $\tauu$ defined in $\MM\cap\{r\geq |\tau|+ 10m\}$ by 
\bea
\tauu:=\tauu_0 +\tauut, \qquad \tauu_0:=\tau+2r+4m\log\left(\frac{r}{2m}\right),
\eea
where $\tauut$ satisfies 
\bea
|\dk^{\leq 2}\tauut|\les r^{-1}+\ep\quad\textrm{in}\quad\MM\cap\{r\geq |\tau|+ 10m\}.
\eea
\end{lemma}

\begin{proof}
Since $\tauu$ is an optical function, it satisfies by definition the eikonal equation
\beaa
\g^{\a\b}\pr_{\a}\tauu\pr_{\b}\tauu=0,
\eeaa
and plugging the decomposition $\tauu=\tauu_0 +\tauut$, this reduces to
\bea
\label{eq:eikonaltauut:nullinf}
&&2\g^{\a\b}\pr_{\a}\tauu_0 \pr_\b\tauut+\g^{\a\b}\pr_\a\tauut\pr_\b\tauut\nn\\
&=&-\g^{\a\b}\pr_\a\tauu_0\pr_\b\tauu_0 =-\gam^{\a\b}\pr_\a\tauu_0\pr_\b\tauu_0 +O(1)\big(\gcheck^{\tau\tau},  \gcheck^{r\tau}, \gcheck^{rr}\big) \nn\\
&=&O(r^{-2}) +r\Ga_b,
\eea
where we used the definition of $\tauu_0$, \eqref{eq:assymptiticpropmetricKerrintaurxacoord:2} and \eqref{eq:controloflinearizedmetriccoefficients} in the last equality. \eqref{eq:eikonaltauut:nullinf} is a nonlinear transport equations for $\tauut$ and we initialize it by $\tauut=0$ on $\Si(1)$. Then, integrating \eqref{eq:eikonaltauut:nullinf} from $\Sigma(1)$ in the region $\MM\cap\{r\geq |\tau|+ 10m\}$, together with \eqref{eq:controloflinearizedinversemetriccoefficients} and the control of $(\Ga_g, \Ga_b)$ in \eqref{eq:decaypropertiesofGabGag}, we easily obtain 
\beaa
|\dk^{\leq 2}\tauut|\les r^{-1}+\ep\quad\textrm{in}\quad\MM\cap\{r\geq |\tau|+ 10m\}
\eeaa
which concludes the proof of Lemma \ref{lemma:constructionoftheingoingopticalfunctiontauu}.
\end{proof}

Making use of the ingoing optical function $\tauu$, we may now define $\II_+$. 
\begin{definition}[Definition of $\II_+$]
\lab{def:howmathcalIplusisdefinedinMM} 
Consider the coordinates $(\tauu, \tau, x^1, x^2)$ covering the spacetime region $\MM\cap\{r\geq |\tau|+ 10m\}$, where $\tauu$ is the ingoing optical function constructed in Lemma \ref{lemma:constructionoftheingoingopticalfunctiontauu}. Then, the future null infinity of $(\MM, \g)$ is defined as
\bea
\II_+:=\MM\cap\{\tauu=+\infty\}.
\eea
\end{definition}

The following lemma provides the control of the induced geometry on $\II_+$ in the perturbed spacetime $(\MM, \g)$. 
\begin{lemma}\lab{lemma:controllinearizedmetric:inducedmetricII+}
Let $\II_+$ be given by Definition \ref{def:howmathcalIplusisdefinedinMM}. Consider the coordinates system $(\tau, x^1, x^2)$ covering $\II_+$, and denote by $(\pr_\tau^{\II_+}, \pr_{x^1}^{\II_+}, \pr_{x^2}^{\II_+})$ the corresponding coordinate vectorfields. Then, 
\begin{enumerate}
\item the coordinate vectorfields $\pr_{x^a}^{\II_+}$, $a=1,2,$ satisfy 
\bea\label{expression:prxaIIplus:nullinf}
\pr_{x^a}^{\II_+}=\pr_{x^a}+O(\ep)\pr_r,\,\,\, a=1,2,  \qquad  \nab^{\II_+}=\nab+O(\ep)r^{-1}\pr_r, \quad r\nab^{\II_+}:=\langle \pr_{x^1}^{\II_+}, \pr_{x^2}^{\II_+}\rangle,
\eea

\item the spheres $S^{\II_+}(\tau_1):=\II_+\cap\{\tau=\tau_1\}$ foliating $\II_+$ are round,

\item $\pr_\tau^{\II_+}$ is ingoing null and there exists a scalar function $b^r$ such that
\bea\label{expression:prtauIIplus:nullinf}
\pr_\tau^{\II_+}=\pr_\tau -\frac{1}{2}(1+b^r)\pr_r+O(\ep)\nab, \qquad |\dk^{\leq 1}b^r|\les \ep,
\eea

\item $\pr_r$ is an outgoing null vectorfield on $\II_+$ and satisfies 
\bea\label{expression:prrIIplus:nullinf}
\g(\pr_\tau^{\II_+}, \pr_r)=-1, \qquad \g(r^{-1}\pr_{x^a}^{\II_+}, \pr_r)=0.
\eea
\end{enumerate}
\end{lemma}

\begin{proof}
For the purpose of the proof, we introduce the notation $\err$ for any term satisfying 
\beaa
|\dk^{\leq 1}\err|\les r^{-1}+\ep\quad\textrm{in}\quad\MM\cap\{r\geq |\tau|+ 10m\}.
\eeaa
Then, in view of the definition and the control of $\tauu$ in Lemma \ref{lemma:constructionoftheingoingopticalfunctiontauu}, we have
\bea\lab{eq:thederivativesoftauuinrtauxacoord}
\pr_r\tauu = 2 +\frac{4m}{r}+r^{-1}\err, \qquad \pr_\tau\tauu = 1 +\err, \qquad \pr_{x^a}\tauu=\err, 
\eea
which implies, denoting by $\widehat{\pr}_\a$ the coordinate vectorfields associated to the coordinates $(\tauu, \tau, x^1, x^2)$,
\beaa
\pr_r=\left(2+\frac{4m}{r}+r^{-1}\err\right)\widehat{\pr}_{\tauu}, \qquad \pr_\tau=\widehat{\pr}_\tau+(1+\err)\widehat{\pr}_{\tauu}, \qquad \pr_{x^a}=\widehat{\pr}_{x^a}+\err\widehat{\pr}_{\tauu},
\eeaa
and in particular 
\beaa
\widehat{\pr}_{x^a}=\pr_{x^a}+\err\pr_r\quad\textrm{in}\quad\MM\cap\{r\geq |\tau|+ 10m\}.
\eeaa
Letting $r\to +\infty$, and using the control of $\err$, we infer  \eqref{expression:prxaIIplus:nullinf}. Also, \eqref{expression:prxaIIplus:nullinf} and the control of $\g_{\a\b}$ provided by \eqref{eq:assymptiticpropmetricKerrintaurxacoord:2} and \eqref{eq:controloflinearizedmetriccoefficients} implies that the spheres $S^{\II_+}(\tau_1)=\II_+\cap\{\tau=\tau_1\}$ foliating $\II_+$ are round. 

Next, we focus on the control of $\pr_\tau^{\II_+}$. Since $\tauu$ is an optical function, the vectorfield $e_3$ given by  
\beaa
e_3:=-\g^{\a\b}\pr_{\a}\tauu\pr_{\b}
\eeaa
is ingoing null and tangent to the level sets of $\tauu$. Also, in view of the above definition of $e_3$, \eqref{eq:thederivativesoftauuinrtauxacoord}, the control of $\g^{\a\b}$ provided by \eqref{eq:assymptiticpropmetricKerrintaurxacoord:1} and \eqref{eq:controloflinearizedinversemetriccoefficients}, and the fact that $\err$ contains in particular $r\Ga_b$ and $r^2\Ga_g$, we have
\beaa
e_3(r)=-1+\err, \qquad e_3(\tau)=2+\err, \qquad e_3(x^a)=r^{-1}\err,
\eeaa
and hence
\beaa
\frac{1}{e_3(\tau)}e_3=\pr_\tau -\frac{1}{2}(1+\err)\pr_r+\err\nab.
\eeaa
Letting $r\to +\infty$, and using the control of $\err$, we infer 
\beaa
\frac{1}{e_3(\tau)}e_3=\pr_\tau -\frac{1}{2}(1+b^r)\pr_r+O(\ep)\nab, \qquad |\dk^{\leq 1}b^r|\les \ep, \quad\textrm{on}\quad\II_+,
\eeaa
and
\beaa
\frac{1}{e_3(\tau)}e_3(\tau)=1, \qquad \frac{1}{e_3(\tau)}e_3(x^a)=0, \quad a=1,2, \quad\textrm{on}\quad\II_+.
\eeaa
Since $\frac{1}{e_3(\tau)}e_3$ is tangent to the level sets of $\tauu$, it is in particular tangent to $\II_+=\{\tauu=+\infty\}$, and we deduce
\beaa
\pr_\tau^{\II_+}=\frac{1}{e_3(\tau)}e_3\quad\textrm{on}\quad\II_+
\eeaa
so that $\pr_\tau^{\II_+}$ is ingoing null and satisfies \eqref{expression:prtauIIplus:nullinf}.

Finally, the fact that $\pr_r$ is an outgoing null vectorfield on $\II_+$ that satisfies \eqref{expression:prrIIplus:nullinf} follows immediately from the control of $\g_{\a\b}$ provided by \eqref{eq:assymptiticpropmetricKerrintaurxacoord:2} and \eqref{eq:controloflinearizedmetriccoefficients} and from  \eqref{expression:prtauIIplus:nullinf}. This concludes the proof of Lemma \ref{lemma:controllinearizedmetric:inducedmetricII+}. 
\end{proof}

%%%%%%%%%%%%%%%%%%%%%%%%%

\subsection{Energy, Morawetz and flux norms}
\label{subsect:norms}

%%%%%%%%%%%%%%%%%%%%%%%%%

We introduce in this section the energy, Morawetz and flux norms needed to state our main result. First, given any $(\tau, r)$, and for any scalar function $F$ on the spheres $S(\tau, r)$ of constant $\tau$ and $r$, we introduce the following notation 
\beaa
\int_{\mathbb{S}^2}F(\tau, r, \om)d\mathring{\ga} := \int F(\tau, r, x^1, x^2)\sqrt{\det(\mathring{\ga})}dx^1dx^2,
\eeaa 
{as well as the corresponding notation for the spheres $S^{\II_+}(\tau)$ of constant $\tau$ on $\II_+$.} Then, for $\tau_1<\tau_2$, we define flux norms\footnote{{For $\F_{\II_+}[\psi](\tau_1,\tau_2)$, recall that $\II_+=\MM\cap\{\tauu=+\infty\}$ where the ingoing optical function $\tauu$ has been constructed in Lemma \ref{lemma:constructionoftheingoingopticalfunctiontauu}, and recall that the notations $\pr_\tau^{\II_+}$ and $\nabla^{\II_+}$ on $\II_+$ have been introduced in Lemma \ref{lemma:controllinearizedmetric:inducedmetricII+}.}}
\bsub
\label{def:variousMorawetzIntegrals}
\begin{align}
\F_{\AA}[\psi](\tau_1,\tau_2)={}& \int_{\tau_1}^{\tau_2}\int_{\mathbb{S}^2}\big(|\mu| |\partial_r \psi|^2 +|\partial_{\tt} \psi|^2+|\nabla\psi|^2\big)(\tau, r=r_+(1-\dhor), \om){d\mathring{\ga}d\tau},\\
\F_{\II_+}[\psi](\tau_1,\tau_2)={}& \int_{\tau_1}^{\tau_2}\int_{\mathbb{S}^2}{\big(|\pr_\tau^{\II_+}\psi|^2+|\nabla^{\II_+}\psi|^2\big)(\tauu=+\infty, \tau, \om)r^2d\mathring{\ga}d\tau},\\
\F[\psi](\tau_1,\tau_2)={}& \F_{\II_+} [\psi](\tau_1,\tau_2)+\F_{\AA}[\psi](\tau_1,\tau_2),
\end{align}
the energy norm
\begin{align}
\E[\psi](\tau)={}& \int_{r_+(1-\dhor)}^{+\infty}\int_{\mathbb{S}^2}\Big((\pr_r\psi)^2+|\nab\psi|^2+r^{-2}(\pr_\tau\psi)^2\Big)r^2{d\mathring{\ga}dr},
\end{align}
and the Morawetz norms
\begin{align}
\M_\de[\psi](\tau_1,\tau_2)={}&\int_{\MM_{\nontrap}(\tau_1,\tau_2)}\left(\frac{\abs{\partial_{\tt}\psi}^2}{r^{1+\de}} +\frac{\abs{\nabla\psi}^2}{r}\right)
+\int_{\MM(\tau_1,\tau_2)}\left(\frac{\abs{\partial_r\psi}^2}{r^{1+\de}} +\frac{\abs{\psi}^2}{r^{3+\de}}\right),\\
\M[\psi](\tau_1,\tau_2)={}&\int_{\MM_{\nontrap}(\tau_1,\tau_2)}\left(\frac{\abs{\partial_{\tt}\psi}^2}{r^2} +\frac{\abs{\nabla\psi}^2}{r}\right)
+\int_{\MM(\tau_1,\tau_2)}\left(\frac{\abs{\partial_r\psi}^2}{r^2} +\frac{\abs{\psi}^2}{r^4}\right),
\end{align}
\esub
for any given $0<\de\leq 1$. We also define the norms of the right-hand side 
\bea\lab{eq:defmathcalNpsif}
\mathcal{N}_\de[ \psi, F](\tau_1, \tau_2) &:=& \!\!\int_{\MM(\tau_1, \tau_2)}r^{1+\de}|F|^2\\
\nn&+&\!\!\min\left[\left(\int_{\Mtrap(\tau_1, \tau_2)}|F|^2\right)^{\frac{1}{2}} \left(\int_{\Mtrap(\tau_1, \tau_2)}|\pr\psi|^2\right)^{\frac{1}{2}}, 
\int_{\Mtrap(\tau_1, \tau_2)}\tau^{1+\de}|F|^2\right],
\eea
and
\bea\lab{eq:defintionwidehatmathcalNfpsinormRHS}
&&\widehat{\mathcal{N}}[ \psi, F](\tau_1, \tau_2) \nn\\
\nn&:=&\sup_{\tau\in[\tau_1,\tau_2]}\bigg|\int_{\Mntrap(\tau_1, \tau)}\pr_\tau\psi F\bigg|+\int_{\Mntrap(\tau_1, \tau_2)}\big(|\pr_r\psi|+r^{-1}|\psi|\big)|F|+\int_{\MM(\tau_1, \tau_2)}|F|^2\\
&+&\!\!\min\left[\left(\int_{\Mtrap(\tau_1, \tau_2)}|F|^2\right)^{\frac{1}{2}} \left(\int_{\Mtrap(\tau_1, \tau_2)}|\pr\psi|^2\right)^{\frac{1}{2}}, 
\int_{\Mtrap(\tau_1, \tau_2)}\tau^{1+\de}|F|^2\right].
\eea

\begin{remark}\lab{rmk:controlofwidehatNfpsibyNfpsi}
In view of the above definitions, we immediately deduce the following bound
\beaa
\widehat{\mathcal{N}}[ \psi, F](\tau_1, \tau_2)\les \Big(\M_\de[\psi](\tau_1,\tau_2)\Big)^{\frac{1}{2}}\Big(\mathcal{N}_\de[ \psi, F](\tau_1, \tau_2)\Big)^{\frac{1}{2}}+\mathcal{N}_\de[ \psi, F](\tau_1, \tau_2).
\eeaa
Also, note that $\M[\psi](\tau_1,\tau_2)=\M_1[\psi](\tau_1,\tau_2)$. 
\end{remark}

Next, we introduce the notation 
\beaa
\pr\psi:=(\pr_\tau\psi, \pr_r\psi, \nab\psi), 
\eeaa
where $\nab$ is defined as in \eqref{eq:defweightedderivative}
 by $r\nab=\langle \pr_{x^1}, \pr_{x^2}\rangle$, and for any nonnegative integer $\reg$, let
\bea
\lab{def:normshighorder:FMdeandNNnorms}
\bsplit
\F^{(\reg)}[\psi](\tau_1,\tau_2)&:=\F[\pr^{\leq\reg}\psi](\tau_1,\tau_2),\qquad\qquad\qquad\quad \E^{(\reg)}[\psi](\tau):=\E[\pr^{\leq s}\psi](\tau), \\ 
\M^{(\reg)}_\de[\psi](\tau_1,\tau_2)&:=\M_\de[\pr^{\leq s}\psi](\tau_1,\tau_2),\qquad\qquad\, \M^{(\reg)}[\psi](\tau_1,\tau_2):=\M[\pr^{\leq s}\psi](\tau_1,\tau_2),\\
\mathcal{N}^{(\reg)}_\de[ \psi, F](\tau_1, \tau_2)&:=\mathcal{N}_\de[\pr^{\leq\reg}\psi, \pr^{\leq \reg}F](\tau_1, \tau_2),\quad\widehat{\mathcal{N}}^{(\reg)}[ \psi, F](\tau_1, \tau_2):=\widehat{\mathcal{N}}[\pr^{\leq\reg}\psi, \pr^{\leq \reg}F](\tau_1, \tau_2).
\end{split}
\eea

Finally, we define for any nonnegative integer $\reg$ the following combined norms 
\bea
\bsplit
\EMF^{(s)}_\de[\psi](\tau_1,\tau_2) := \sup_{\tt\in [\tau_1, \tau_2]} \E^{(s)}[\psi](\tt) + \M^{(s)}_\de[\psi](\tau_1,\tau_2)+\F^{(s)}[\psi](\tau_1,\tau_2),\\
\EMF^{(s)}[\psi](\tau_1,\tau_2) := \sup_{\tt\in [\tau_1, \tau_2]} \E^{(s)}[\psi](\tt) + \M^{(s)}[\psi](\tau_1,\tau_2)+\F^{(s)}[\psi](\tau_1,\tau_2),
\end{split}
\eea
with $\EM^{(s)}_\de[\psi](\tau_1,\tau_2)$, $\EM^{(s)}[\psi](\tau_1,\tau_2)$, $\MF^{(s)}[\psi](\tau_1,\tau_2)$ and $\EF^{(s)}[\psi](\tau_1,\tau_2)$ being defined in a similar way.

The reason we may choose flat volume elements in \eqref{def:variousMorawetzIntegrals} for the definition of $\F_\AA[\psi]$, $\F_{\II_+}[\psi]$ and $\E[\psi]$ is justified by the following lemma.
\begin{lemma}\lab{lemma:justificationboundarytermsflatvolume}
Let $\mathcal{Q}$ denote the energy momentum tensor of $\psi$, i.e., 
\beaa
\mathcal{Q}_{\a\b}:=\Re\bigg(\pr_\a\psi\ov{\pr_\b\psi }-\frac{1}{2}\g_{\a\b}\pr^\mu\psi\ov{\pr_\mu\psi}\bigg).
\eeaa
Also, let $X$ be a globally timelike vectorfield which coincides with $\pr_\tau$ in $r\geq 13m$. Then, the corresponding boundary terms generated by using $X$ as a multiplier in the energy estimates satisfy
\beaa
\F_{\AA}[\psi](\tau_1,\tau_2) &\simeq& \int_{\AA(\tau_1,\tau_2)}\mathcal{Q}(X, N_\AA),\\
\F_{\II_+}[\psi](\tau_1,\tau_2) &\simeq& \int_{\II_+(\tau_1,\tau_2)}\mathcal{Q}(X, N_{\II_+}),\\
\E[\psi](\tau) &\simeq&\int_{\Sigma(\tt)}\mathcal{Q}(X, N_{\Si(\tt)}),
\eeaa
where $f\simeq h$ if $f\les h$ and $h\les f$. 
\end{lemma}

\begin{proof}
{On $\II_+$, in view of Lemma \ref{lemma:controllinearizedmetric:inducedmetricII+}, $\pr^{\II_+}_\tau$ is a null ingoing vectorfield tangent to $\II_+$ and the spheres $S^{\II_+}(\tau)$ foliating $\II_+$ are round. Thus, noticing that $X=\pr_\tau$ on $\II_+$, we have
\beaa
\mathcal{Q}(\pr_\tau, N_{\II_+})d\II_+ = \mathcal{Q}(\pr_\tau, \pr_\tau^{\II_+})r^2d\mathring{\ga} d\tau
\eeaa
and the statement for $\F_{\II_+}[\psi]$ then follows immediately from the following computation
\beaa
\mathcal{Q}(\pr_\tau, \pr_\tau^{\II_+}) &=& \mathcal{Q}\left(\pr_\tau^{\II_+}+\frac{1}{2}(1+O(\ep))\pr_r+O(\ep)\nab^{\II_+}, \pr_\tau^{\II_+}\right)\\
&=& \mathcal{Q}(\pr_\tau^{\II_+}, \pr_\tau^{\II_+})+\frac{1}{2}(1+O(\ep))\QQ(\pr_r, \pr_\tau^{\II_+})+O(\ep)\QQ(\nab^{\II_+}, \pr^{\II_+}_\tau)\\
&=& (\pr_\tau^{\II_+}\psi)^2+\frac{1}{4}(1+O(\ep))|\nab^{\II_+}\psi|^2 +O(\ep)|\pr_\tau^{\II_+}\psi||\nabla^{\II_+}\psi|\\
&\simeq&  |\pr_\tau^{\II_+}\psi|^2+|\nabla^{\II_+}\psi|^2,
\eeaa
where we used the identities \eqref{expression:prtauIIplus:nullinf} and \eqref{expression:prrIIplus:nullinf}, as well as the fact that $\pr_r$ is outgoing null on $\II_+$ in view of Lemma \ref{lemma:controllinearizedmetric:inducedmetricII+}.}

Next, we consider the case of $\E[\psi](\tau)$ and focus on the region $r\geq 12m$ where $X=\pr_\tau$. Then, we notice that 
\beaa
\QQ(\pr_\tau, N_{\Si(\tau)})d\Si(\tau) &=& \QQ\left(\pr_\tau, -\frac{\D\tau}{|\D\tau|}\right)\sqrt{\det(g)}drdx^1dx^2\\
&\simeq& \Big((\pr_r\psi)^2+|\nab\psi|^2+r^{-2}(\pr_\tau\psi)^2\Big)\frac{\sqrt{\det(g)}}{|\D\tau|}drdx^1dx^2
\eeaa
and the statement for $\E[\psi](\tau)$ follows from 
\beaa
\frac{\sqrt{\det(g)}}{|\D\tau|}drdx^1dx^2 &=& \sqrt{\frac{\det(g)}{|\g^{\tau\tau}|}}drdx^1dx^2\simeq\sqrt{\frac{m^2r^2(1+ O(mr^{-1})+O(\ep))\det(\mathring{\ga})}{\frac{m^2}{r^2}(1+O(mr^{-1})+O(\ep))}}drdx^1dx^2\\
&\simeq& r^2\sqrt{\det(\mathring{\ga})}drdx^1dx^2=r^2drd\mathring{\ga}\quad\textrm{on}\quad\Si(\tau)
\eeaa
where we used Lemmas  \ref{lemma:specificchoice:normalizedcoord:inducedmetricSitau} and \ref{lemma:controllinearizedmetric:inducedmetricSitau} to control $\sqrt{\det(g)}$, and Lemma \ref{lem:specificchoice:normalizedcoord} and \eqref{eq:controloflinearizedinversemetriccoefficients} to control $\g^{\tau\tau}$.

Finally, we consider the case of $\F_\AA[\psi]$. We have
\beaa
\QQ(X, N_\AA)d\AA &=& \QQ\left(X, -\frac{\D r}{|\D r|}\right)\sqrt{\det(g_\AA)}d\tau dx^1dx^2\\
&\simeq& \big(|\mu||\pr_r\psi|^2+|\pr_\tau\psi|^2+|\nab\psi|^2\big)\frac{\sqrt{\det(g_\AA)}}{|\D r|}d\tau dx^1dx^2
\eeaa
and the statement for $\F_\AA[\psi]$ follows from  
\beaa
\frac{\sqrt{\det(g_\AA)}}{|\D r|}d\tau dx^1dx^2 &=& \sqrt{\frac{\det(g_\AA)}{|\g^{rr}|}}d\tau dx^1dx^2\simeq\sqrt{\frac{m^2(\dhor+O(\ep))\det(\mathring{\ga})}{\dhor+O(\ep)}}d\tau dx^1dx^2\\
&\simeq& \sqrt{\det(\mathring{\ga})}d\tau dx^1dx^2=d\tau d\mathring{\ga}\quad\textrm{on}\quad\AA
\eeaa
where we used Lemmas \ref{lemma:specificchoice:normalizedcoord:inducedmetricAA} and \ref{lemma:controllinearizedmetric:inducedmetricAA} to control $\det(g_\AA)$, and \eqref{eq:inverse:hypercoord} and \eqref{eq:controloflinearizedinversemetriccoefficients} to control $\g^{rr}$.  This concludes the proof of Lemma \ref{lemma:justificationboundarytermsflatvolume}.
\end{proof}

%%%%%%%%%%%%%%%%%%%%%%%%%%%%%%%%%%%%

\subsection{Functional inequalities on $\Si(\tau)$, $\II_+$ and $\AA$}
\label{subsect:functionalineqSigmatau}

%%%%%%%%%%%%%%%%%%%%%%%%%%%%%%%%%%%%

In this section, we derive estimates on $\Si(\tau)$, $\II_+$ and $\AA$. We start with Hardy estimates on $\Si(\tau)$ and $\II_+$.
\begin{lemma}[Hardy estimates on $\Si(\tau)$ and $\II_+(\tau_1, \tau_2)$]
\lab{lemma:HardySitauandIIplus}
For any $\tau_1<\tau_2$, the following Hardy estimates hold
\bea
\bsplit
\int_{\tau_1}^{\tau_2}\int_{\mathbb{S}^2}\frac{\psi^2}{r^2}{(\tauu=+\infty, \tau, \om)} r^2{d\mathring{\ga}d\tau} &\les \F_{\II_+}[\psi](\tau_1,\tau_2), \\ 
\int_{r_+(1-\dhor)}^{+\infty}\int_{\mathbb{S}^2}\frac{\psi^2}{r^2}r^2{d\mathring{\ga}dr}  &\les\E[\psi](\tau).
\end{split}
\eea
\end{lemma}

\begin{proof}
Given the definition of $\E[\psi](\tau)$ in \eqref{def:variousMorawetzIntegrals}, Hardy estimate on $\E[\psi](\tau)$ simply reduces to the standard Hardy estimate of the flat case. {Next, we focus on $\F_{\II_+}[\psi](\tau_1,\tau_2)$. Recalling from Lemma \ref{lemma:controllinearizedmetric:inducedmetricII+} that the vectorfield $\pr_\tau^{\II_+}$ is tangent to $\II_+$ and satisfies 
\beaa
\pr_\tau^{\II_+}=\pr_\tau -\frac{1}{2}(1+b^r)\pr_r+O(\ep)\nab, \qquad |\dk^{\leq 1}b^r|\les \ep,
\eeaa
we have, for $\psi$ compactly supported in $\II_+(\tau_1, \tau_2)$, 
\beaa
\int_{\tau_1}^{\tau_2}\int_{\mathbb{S}^2}\frac{\psi^2}{r^2} r^2d\mathring{\ga} d\tau  &=& -2\int_{\tau_1}^{\tau_2}\int_{\mathbb{S}^2}\pr_\tau^{\II_+}(r)\psi^2 d\mathring{\ga} d\tau  -\int_{\tau_1}^{\tau_2}\int_{\mathbb{S}^2}b^r\frac{\psi^2}{r^2}r^2 d\mathring{\ga} d\tau \\
&=& 4\int_{\tau_1}^{\tau_2}\int_{\mathbb{S}^2}\psi\pr_\tau^{\II_+}(\psi) r d\mathring{\ga}d\tau +O(\ep)\int_{\tau_1}^{\tau_2}\int_{\mathbb{S}^2}\frac{\psi^2}{r^2} r^2d\mathring{\ga} d\tau.
\eeaa
Thus, we infer
\beaa
\bsplit
(1+O(\ep))\int_{\tau_1}^{\tau_2}\int_{\mathbb{S}^2}\frac{\psi^2}{r^2} r^2d\mathring{\ga} d\tau  
&\leq 2\left(\int_{\tau_1}^{\tau_2}\int_{\mathbb{S}^2}\frac{\psi^2}{r^2} r^2d\mathring{\ga} d\tau  \right)^{\frac{1}{2}}\left(\int_{\tau_1}^{\tau_2}\int_{\mathbb{S}^2}|\pr_\tau^{\II_+}(\psi)|^2 r^2d\mathring{\ga} d\tau  \right)^{\frac{1}{2}}\\
&\leq 2\left(\int_{\tau_1}^{\tau_2}\int_{\mathbb{S}^2}\frac{\psi^2}{r^2} r^2d\mathring{\ga} d\tau  \right)^{\frac{1}{2}}\left(\F_{\II_+}[\psi](\tau_1,\tau_2)\right)^{\frac{1}{2}}
\end{split}
\eeaa}
and hence
\beaa
\int_{\tau_1}^{\tau_2}\int_{\mathbb{S}^2}\frac{\psi^2}{r^2} r^2{d\mathring{\ga}d\tau} \les \F_{\II_+}[\psi](\tau_1,\tau_2).
\eeaa
One then concludes using the density of compactly supported functions in the set of functions $\psi$ with $\F_{\II_+}[\psi](\tau_1,\tau_2)<+\infty$. This concludes the proof of the lemma.
\end{proof}

Also, we {will} use the trace estimate to control {lower order} terms on $\AA$ and $\II_+$ from Morawetz.
\begin{lemma}[Trace estimates on $\AA$ and $\II_+$]
\lab{lem:trace:fluxcontrolledbyMorawetz}
We have the following trace estimates on $\AA$ and $\II_+$
\beaa
\int_{\tau_1}^{\tau_2}\int_{\mathbb{S}^2}\psi^2(\tau, r=r_+(1-\dhor),\om)d\mathring{\ga} d\tau &\les& {\M}[\psi](\tau_1,\tau_2),\\
{\int_{\tau_1}^{\tau_2}\int_{\mathbb{S}^2}\psi^2(\tauu=+\infty, \tau, \om) r^2d\mathring{\ga} d\tau } &\les& \int_{\MM_{{11m}, \infty}(\tau_1,\tau_2)}|\pr^{\leq 1}\psi|^2.
\eeaa
\end{lemma}

\begin{proof}
Let $\chi=\chi(r)$ a smooth cut-off function such that $\chi(r)=1$ for $r\leq 3m$ and $\chi(r)=0$ for $r\geq 4m$. Then, we have
\beaa
\int_{\tau_1}^{\tau_2}\int_{\mathbb{S}^2}\psi^2(\tau, r=r_+(1-\dhor),\om)d\mathring{\ga} d\tau  &=& -\int_{\tau_1}^{\tau_2}\int_{r_+(1-\dhor)}^{4m}\int_{\mathbb{S}^2}\pr_r(\chi\psi^2)(\tau, r,\om)d\mathring{\ga}dr d\tau\\
&\les& \int_{\tau_1}^{\tau_2}\int_{r_+(1-\dhor)}^{4m}\int_{\mathbb{S}^2}\Big((\pr_r\psi)^2+\psi^2\Big)(\tau, r,\om)d\mathring{\ga}dr d\tau\\
&\les& \int_{\MM_{r_+(1-\dhor), 4m}(\tau_1, \tau_2)}\Big((\pr_r\psi)^2+\psi^2\Big)\\
&\les& {\M}[\psi](\tau_1,\tau_2)
\eeaa
as stated. The other estimate on $\II_+$ follows in the same manner.
\end{proof}

Finally, the following estimate on $\II_+$ will be useful.
\begin{lemma}\label{lem:nullnfFluxBdedByEnergy}
For any $\tau_1<\tau_2$ and any $\de>0$, we have
\bea
{\liminf_{\tauu\to+\infty}}\int_{\tau_1}^{\tau_2}\int_{\mathbb{S}^2}{(1+\tau-\tau_1)^{-1-\de}}r^{-1}|\mathfrak{d}^{\leq 1}\psi|^2{({\tauu}, \tau, \omega)}r^2d\mathring{\ga}d\tau   \les \sup_{\tau\in[\tau_1, \tau_2]}\E[\psi](\tau).
\eea
 \end{lemma}

{\begin{remark}\lab{rmk:limitingargumentwithliminfonII+}
In practice, Lemma \ref{lem:nullnfFluxBdedByEnergy} will be used to control error terms appearing when estimating quantities $A(\tauu)$ that have limits as $\tauu\to +\infty$, such as the restriction of $\EMF[\psi](\tau_1, \tau_2)$ to $\MM(\tauu\leq \tauu')$ which admits a limit as $\tauu'\to +\infty$. It then suffices to control lower order terms appearing in the estimate for $A(\tauu)$ using Lemma \ref{lem:nullnfFluxBdedByEnergy} along a particular sequence $\tauu_{(q)}$ with ${\tauu}_{(q)}\to +\infty$ as $q\to +\infty$ such that the bound of Lemma \ref{lem:nullnfFluxBdedByEnergy} holds uniformly along the sequence ${\tauu}_{(q)}$. This then yields an upper bound for $A(\tauu_{(q)})$ which is uniform in $q$ and the bound for $A(\tauu=+\infty)$ follows by taking the limit $q\to +\infty$. To ease notations, we will skip this limiting argument and directly use Lemma \ref{lem:nullnfFluxBdedByEnergy} at $\tauu=+\infty$. 
\end{remark}}

\begin{proof}
Using the definition of $\E[\psi](\tau)$ and $\dk$, as well as the Hardy estimate on $\Si(\tau)$ derived in Lemma \ref{lemma:HardySitauandIIplus}, we have
\beaa
\E[\psi](\tau) &\simeq& \int_{r_+(1-\dhor)}^{+\infty}\int_{\mathbb{S}^2}r^{-2}|\mathfrak{d}^{\leq 1}\psi|^2r^2d\mathring{\ga}dr.
\eeaa
We infer, using also \eqref{eq:assymptiticpropmetricKerrintaurxacoord:volumeform} and \eqref{eq:controloflinearizedmetriccoefficients:det}, 
\beaa
&&\int_{\MM(\tau_1, \tau_2)}{(1+\tau-\tau_1)^{-1-\de}}r^{-2}|\mathfrak{d}^{\leq 1}\psi|^2\\ &\les& \int_{\tau_1}^{\tau_2}{(1+\tau-\tau_1)^{-1-\de}}\left(\int_{r_+(1-\dhor)}^{+\infty}\int_{\mathbb{S}^2}r^{-2}|\mathfrak{d}^{\leq 1}\psi|^2r^2d\mathring{\ga}dr\right)d\tau\\
&\les& \left(\int_{\tau_1}^{\tau_2}{(1+\tau-\tau_1)^{-1-\de}}d\tau\right)\sup_{\tau\in[\tau_1, \tau_2]}\E[\psi](\tau)
\eeaa
and hence 
\bea\lab{eq:intermediaryboundonMMtoboundbulkonscri}
\int_{\MM(\tau_1, \tau_2)}{(1+\tau-\tau_1)^{-1-\de}}r^{-2}|\mathfrak{d}^{\leq 1}\psi|^2 &\les& \sup_{\tau\in[\tau_1, \tau_2]}\E[\psi](\tau).
\eea

{Next, recall from Lemma \ref{lemma:constructionoftheingoingopticalfunctiontauu} that the ingoing optical function $\tauu$ satisfies 
\beaa
\tauu=\tauu_0 +\tauut, \qquad \tauu_0=\tau+2r+4m\log\left(\frac{r}{2m}\right),\qquad |\dk^{\leq 2}\tauut|\les r^{-1}+\ep\quad\textrm{in}\quad\MM\cap\{r\geq |\tau|+ 10m\}.
\eeaa
Then, one easily checks that  
\beaa
\frac{\underline{\tau}}{8}\leq r\leq \underline{\tau}\quad\textrm{if}\quad\underline{\tau}\geq 2\max(|\tau|,40m), 
\eeaa
and hence, using \eqref{eq:assymptiticpropmetricKerrintaurxacoord:volumeform} and \eqref{eq:controloflinearizedmetriccoefficients:det}, as well as the following change of variable identity 
\beaa
drd\tau dx^1dx^2=\left|\frac{\pr\tauu}{\pr r}\right|d\tauu d\tau dx^1dx^2=\left(2 +\frac{4m}{r}+O(r^{-2}+r^{-1}\ep)\right)d\tauu d\tau dx^1dx^2,
\eeaa
we infer
\beaa
&&\int_{\MM(\tau_1, \tau_2)\cap\{\underline{\tau}\geq 2\max(|\tau_1|, |\tau_2|, 40m)\}}{(1+\tau-\tau_1)^{-1-\de}}r^{-2}|\mathfrak{d}^{\leq 1}\psi|^2\\ 
&\simeq& \int_{\tau_1}^{\tau_2}\int_{\underline{\tau}=2\max(|\tau_1|, |\tau_2|, 40m)}^{+\infty}\int_{\mathbb{S}^2}{(1+\tau-\tau_1)^{-1-\de}}r^{-2}|\mathfrak{d}^{\leq 1}\psi|^2r^2d\mathring{\ga} d\underline{\tau}d\tau\\
&\simeq& \int_{\underline{\tau}=2\max(|\tau_1|, |\tau_2|, 40m)}^{+\infty}\frac{1}{\underline{\tau}}\left(\int_{\tau_1}^{\tau_2}\int_{\mathbb{S}^2}{(1+\tau-\tau_1)^{-1-\de}}r^{-1}|\mathfrak{d}^{\leq 1}\psi|^2r^2d\mathring{\ga}d\tau\right)d\underline{\tau}.
\eeaa}
Together with \eqref{eq:intermediaryboundonMMtoboundbulkonscri}, we deduce
\beaa
\int_{\underline{\tau}=2\max(|\tau_1|, |\tau_2|, 40m)}^{+\infty}\frac{1}{\underline{\tau}}\left(\int_{\tau_1}^{\tau_2}\int_{\mathbb{S}^2}{(1+\tau-\tau_1)^{-1-\de}}r^{-1}|\mathfrak{d}^{\leq 1}\psi|^2r^2d\mathring{\ga}d\tau\right)d\underline{\tau} &\les& \sup_{\tau\in[\tau_1, \tau_2]}\E[\psi](\tau).
\eeaa
In particular, we infer the existence of a increasing sequence $(\underline{\tau}_{(q)})_{q\geq 1}$ such that $\underline{\tau}_{(q)}\to +\infty$ as $q\to +\infty$, and, for any $q\geq 1$, we have
\beaa
\int_{\tau_1}^{\tau_2}\int_{\mathbb{S}^2}{(1+\tau-\tau_1)^{-1-\de}}r^{-1}|\mathfrak{d}^{\leq 1}\psi|^2(\underline{\tau}=\underline{\tau}_{(q)}, \tau, \omega)r^2d\mathring{\ga}d\tau &\les& \sup_{\tau\in[\tau_1, \tau_2]}\E[\psi](\tau),
\eeaa
uniformly in $q$. In particular, letting $q\to +\infty$, we deduce 
\beaa
{\liminf_{\tauu\to+\infty}}\int_{\tau_1}^{\tau_2}\int_{\mathbb{S}^2}{(1+\tau-\tau_1)^{-1-\de}}r^{-1}|\mathfrak{d}^{\leq 1}\psi|^2{({\tauu}, \tau, \omega)}r^2d\mathring{\ga}d\tau   \les \sup_{\tau\in[\tau_1, \tau_2]}\E[\psi](\tau)
\eeaa
as stated. This concludes the proof of Lemma \ref{lem:nullnfFluxBdedByEnergy}.
\end{proof}

%%%%%%%%%%%%%%%%%%%%%%%%%%%%%

\section{Basic estimates for the wave equation}
\lab{sect:basicestimatesforwaveequations}

%%%%%%%%%%%%%%%%%%%%%%%%%%%%%

In this section, we collect estimates for solutions to the wave equation \eqref{intro:eq:scalarwave}, i.e.
\beaa
\Box_{\g} \psi = F, \qquad \MM,
\eeaa
which can be proved on perturbations of Kerr in the range $|a|<m$. {Some of these estimates are by now classical and proofs are provided for the convenience of the reader.}

%%%%%%%%%%%%%%%%%%%%%%%%%%%%%%%%%%

\subsection{Standard calculation for generalized currents}

%%%%%%%%%%%%%%%%%%%%%%%%%%%%%%%%%% 

{Recall from Lemma \ref{lemma:justificationboundarytermsflatvolume} the definition of the energy-momentum tensor $\QQ_{\a\b}$ for $\psi$
\beaa
\QQ_{\a\b} &=& \Re\bigg(\pr_{\a}\psi \ov{ \pr_{\b}\psi }- \frac{1}{2}\g_{\a\b} \pr_{\nu}\psi \ov{\pr^{\nu}\psi}\bigg),
\eeaa
and recall from \eqref{def:deformationtensor:lastsect} the definition of the deformation tensor of a vectorfield $X$
\beaa
{}^{(X)}\pi_{\a\b} =  \D_{\a}X_{\b} + \D_{\b}X_{\a}.
\eeaa

The following lemma provides a standard calculation for generalized currents associated to the scalar wave equation.
\begin{lemma}
\lab{lem:generalenergyidentity:waveeq}
Given a vectorfield $X$ and a scalar function $w$, we have
\bea
\label{eq:EnerIden:General:wave}
&&\D^{\a}\bigg(\QQ_{\a\b}[\psi]X^{\b}+\Re\bigg(w\ov\psi \rd_{\a} \psi -\frac{1}{2}\rd_{\a}w \psi\ov\psi\bigg)\bigg)\nn\\
 &=&\bigg(\frac{1}{2}{}^{(X)} \pi \cdot \QQ[\psi] +\Re \bigg(w\rd_{\a}\psi \ov{\rd^{\a}\psi} - \frac{1}{2} \Box_{\g} w \psi \ov\psi\bigg)\bigg)+ \Re\Big(\Box_{\g} \psi \ov{({X\psi} +w\psi)}\Big).
\eea
\end{lemma}}

%%%%%%%%%%%%%%%%%%%%%%%%%%%

\subsection{Control of error terms}
\lab{sec:controlforerrortermsinNRGMorawetz}

%%%%%%%%%%%%%%%%%%%%%%%%%%% 

In this section, we provide estimates for error terms appearing in energy-Morawetz estimates, as well as in {commutators} between first-order derivatives and the wave operator.

%%%%%%%%%%%%%%%%%%%%%%%%%%%%%%%%%%%%%%%%%%

\subsubsection{Control of error terms for energy-Morawetz estimates}

%%%%%%%%%%%%%%%%%%%%%%%%%%%%%%%%%%%%%%%%%%
 
In order to control error terms arising in the derivation of energy-Morawetz estimates, we start with the following basic lemma. 
\begin{lemma}\lab{lemma:basiclemmaforcontrolNLterms}
{For $s=0,1$, we have}
\beaa
&&\int_{\Mntrap(\tau_1, \tau_2)}\bigg[\tau^{-1-\dec}\left|\big(\pr_r, r^{-1}\pr_{x^a}, r^{-1}\big){\pr^{\leq s}}\psi\right|^2 
+ r^{-2}\left|\big(\pr_\tau, \pr_r, r^{-1}\pr_{x^a}, r^{-1}\big){\pr^{\leq s}}\psi\right|^2 \\
&&\qquad\qquad\qquad  +r^{-1}\tau^{-\frac{1+\dec}{2}}\left|\big(\pr_r, r^{-1}\pr_{x^a}{,}\, r^{-1}\big){\pr^{\leq s}}\psi\right||\pr_\tau({\pr^{\leq s}}\psi)|\bigg]\\ 
&\les& {\EM^{(s)}[\psi](\tau_1, \tau_2)}.
\eeaa
\end{lemma}

\begin{proof}
In view of the definition of the energy norm {and Lemma \ref{lemma:HardySitauandIIplus}}, we have 
\beaa
\int_{r_+(1-\dhor)}^{+\infty}\int_{\mathbb{S}^2}\Big(|\pr_r\psi|^2+r^{-2}\big(|\pr_{x^a}\psi|^2+|\psi|^2\big)\Big)r^2dr d\mathring{\ga}\les \E[\psi](\tau). 
\eeaa
Also, in view of the definition of the Morawetz norm, we have 
\beaa
\int_{\Mntrap(\tau_1, \tau_2)}\Big(r^{-2}\big(|\pr_\tau\psi|^2+|\pr_r\psi|^2\big)+r^{-3}|\pr_{x^a}\psi|^2+r^{-4}|\psi|^2\Big) \les \M[\psi].
\eeaa
We infer
\beaa
&&\int_{\Mntrap(\tau_1, \tau_2)}\Big[\tau^{-1-\dec}\left|\big(\pr_r, r^{-1}\pr_{x^a}, r^{-1}\big){\pr^{\leq s}}\psi\right|^2 + r^{-2}\left|\big(\pr_\tau, \pr_r, r^{-1}\pr_{x^a}, r^{-1}\big){\pr^{\leq s}}\psi\right|^2 \\
&&\qquad\qquad\qquad +r^{-1}\tau^{-\frac{1+\dec}{2}}\left|\big(\pr_r, r^{-1}\pr_{x^a}\, r^{-1}\big){\pr^{\leq s}}\psi\right||\pr_\tau({\pr^{\leq s}}\psi)|\Big]\\ 
&\les& {\sup_{\tau\in[\tau_1, \tau_2]} \E^{(s)}[\psi](\tau)+\sqrt{\sup_{\tau\in[\tau_1, \tau_2]} \E^{(s)}[\psi](\tau)}\sqrt{\M^{(s)}[\psi](\tau_1, \tau_2)}+\M^{(s)}[\psi](\tau_1, \tau_2)}\\
&\les& {\EM^{(s)}[\psi](\tau_1, \tau_2)}
\eeaa
as stated. This concludes the proof of Lemma \ref{lemma:basiclemmaforcontrolNLterms}.
\end{proof}

The next {two lemmas} will allow us to control all error terms arising in the derivation of energy-Morawetz estimates in $\MM(\tau_1,\tau_2)$. 
{\begin{lemma}\lab{lemma:basiclemmaforcontrolNLterms:ter}
Let $h\in r^{-1}\dk^{\leq 1}\Ga_b$ be a scalar function and let 
$M^{\a\b}$ be symmetric and satisfy
\beaa
&& M^{rr}\in r\dk^{\leq 1}\Ga_b, \qquad M^{r\tau}\in r\dk^{\leq 1}\Ga_g, \qquad M^{\tau\tau}\in \dk^{\leq 1}\Ga_g,\\
&& M^{rx^a}\in \dk^{\leq 1}\Ga_b, \qquad M^{\tau x^a}\in \dk^{\leq 1}\Ga_g,  \qquad M^{x^ax^b}\in r^{-1}\dk^{\leq 1}\Ga_g,
\eeaa
where $a,b=1,2$. Then, the following estimate holds
\beaa
\int_{\MM(\tau_1, \tau_2)}\Big(\big|M^{\a\b}\pr_\a\psi\pr_\b\psi\big|+h|\psi|^2\Big) &\les& \ep\EM[\psi](\tau_1, \tau_2).
\eeaa
\end{lemma}}

{\begin{remark}
In practice, Lemma \ref{lemma:basiclemmaforcontrolNLterms:ter} will be used to control error terms generated by the RHS of the divergence identity \eqref{eq:EnerIden:General:wave}.
\end{remark}}

\begin{proof}
{In view of the control of $h$ and $M^{\a\b}$, and the assumptions for $\Ga_g$ and $\Ga_b$, we have
\beaa
&&\int_{\MM(\tau_1, \tau_2)}\big|M^{\a\b}\pr_\a\psi\pr_\b\psi\big| \\
&\les& \ep\int_{\Mntrap(\tau_1, \tau_2)}\bigg[\tau^{-1-\dec}\left|\big(\pr_r, r^{-1}\pr_{x^a}, r^{-1}\big)\psi\right|^2  +r^{-1}\tau^{-\frac{1+\dec}{2}}\left|\big(\pr_r, r^{-1}\pr_{x^a}\big)\psi\right||\pr_\tau\psi|\\
&&\qquad\qquad\qquad\qquad\qquad\qquad + r^{-2}\left|\pr_\tau\psi\right|^2\bigg]\\ 
&\les& \ep\EM[\psi](\tau_1, \tau_2).
\eeaa
where we have used Lemma \ref{lemma:basiclemmaforcontrolNLterms} in the last estimate.} 
\end{proof}

\begin{lemma}\lab{lemma:basiclemmaforcontrolNLterms:bis}
Let $M^{\a\b}$ be symmetric and satisfy
\beaa
&& M^{rr}\in r\dk^{\leq 1}\Ga_b, \qquad M^{r\tau}\in r\dk^{\leq 1}\Ga_g, \qquad M^{\tau\tau}\in \dk^{\leq 1}\Ga_g,\\
&& M^{rx^a}\in \dk^{\leq 1}\Ga_b, \qquad M^{\tau x^a}\in \dk^{\leq 1}\Ga_g,  \qquad M^{x^ax^b}\in r^{-1}\dk^{\leq 1}\Ga_g,
\eeaa
where $a,b=1,2$. Then, the following estimate holds 
\beaa
\int_{\MM(\tau_1, \tau_2)}\big|M^{\a\b}\pr_\a\pr_\b\psi\big|^2
+\bigg|\int_{\MM(\tau_1, \tau_2)}M^{\a\b}\pr_\a\pr_\b\psi \pr_\tau(\pr^{\leq 1}\psi)\bigg|\\
+\int_{\MM(\tau_1, \tau_2)}\big|M^{\a\b}\pr_\a\pr_\b\psi\big| {\Big|\big(\pr_r, r^{-1}\pr_\tau, r^{-1}\pr_{x^a}, r^{-1}\big)(\pr^{\leq 1}\psi)\Big|} &\les& \ep\EM^{(1)}[\psi](\tau_1, \tau_2). 
\eeaa

Also, let $N$ be a spacetime vectorfield such that we have 
\beaa
N^r\in r\dk^{\leq 2}\Ga_g, \qquad N^\tau\in \dk^{\leq 2}\Ga_g, \qquad N^{x^a}\in \dk^{\leq 2}\Ga_g.
\eeaa
Then, the following holds
\beaa
\int_{\MM(\tau_1, \tau_2)}\big|N^\a\pr_\a\psi\big|^2+ \int_{\MM(\tau_1, \tau_2)}\big|N^\a\pr_\a\psi\big| \Big|\big(\pr_\tau, \pr_r, r^{-1}\pr_{x^a}, r^{-1}\big)\pr^{\leq 1}\psi\Big| \les \ep\EM^{(1)}[\psi](\tau_1, \tau_2). 
\eeaa
\end{lemma}

\begin{remark}
In practice, concerning the quantities estimated in Lemma \ref{lemma:basiclemmaforcontrolNLterms:bis}:
\begin{itemize}
\item $\big(\pr_\tau, \pr_r, r^{-1}\pr_{x^a}, r^{-1}\big)\pr^{\leq 1}\psi$ will be due to  energy-Morawetz multipliers, 

\item  $M^{\a\b}\pr_\a\pr_\b\psi$ and $N^\a\pr_\a\psi$ will come from the RHS of the wave equation, in particular after commutation with {various vectorfields such as} $\pr_\tau$.
\end{itemize}
\end{remark}

\begin{proof}
We introduce a smooth cut-off function $\chi=\chi(r)$ such that $0\leq\chi\leq1$, $\chi=1$ on $r\leq 10m$ and $\chi$ is supported in $r\leq 11m$. We then split all integrals to be estimated into two sub-integrals, one where the integrand is multiplied by $\chi$ and the other where it is multiplied by $1-\chi$. In the first case, we may bound all integrals by
\beaa
\int_{\MM_{r_+(1-\dhor), 11m}(\tau_1, \tau_2)}|r\dk^{\leq 1}\Ga_b||\pr^{\leq 2}\psi|^2&\les& \ep\left(\int \tau^{-1-\dec}d\tau\right)\sup_{\tau\in[\tau_1, \tau_2]}\E^{(1)}[\psi](\tau)\\
&\les& \ep\sup_{\tau\in[\tau_1, \tau_2]}\E^{(1)}[\psi](\tau)
\eeaa
as stated. We thus focus, from now on, on proving the estimates when the integrand is multiplied by $1-\chi$, and hence supported in $\Mntrap$.

We start with the first estimate. In view of the control of $M^{\a\b}$, and the assumptions for $\Ga_g$ and $\Ga_b$, we have
{\beaa
&&\int_{\MM(\tau_1, \tau_2)}(1-\chi)\big|M^{\a\b}\pr_\a\pr_\b\psi\big| {\Big|\big(\pr_r, r^{-1}\pr_\tau, r^{-1}\pr_{x^a}, r^{-1}\big)(\pr^{\leq 1}\psi)\Big|}\\
&&+\bigg|\int_{\MM(\tau_1, \tau_2)}(1-\chi)M^{\a\b}\pr_\a\pr_\b\psi \pr_\tau(\pr^{\leq 1}\psi)\bigg|+\int_{\MM(\tau_1, \tau_2)}(1-\chi)\big|M^{\a\b}\pr_\a\pr_\b\psi\big|^2\\
&\les&\bigg|\int_{\MM(\tau_1, \tau_2)}(1-\chi)M^{rr}\pr_r^2\psi \pr_\tau(\pr^{\leq 1}\psi)\bigg|+\bigg|\int_{\MM(\tau_1, \tau_2)}(1-\chi)M^{rx^a}\pr_r\pr_{x^a}\psi \pr_\tau(\pr^{\leq 1}\psi)\bigg|\\
&&+\ep\int_{\Mntrap(\tau_1, \tau_2)}\bigg[\tau^{-1-\dec}\left|\big(\pr_r, r^{-1}\pr_{x^a}, r^{-1}\big)\pr^{\leq 1}\psi\right|^2 
+ r^{-2}\left|\big(\pr_\tau, \pr_r, r^{-1}\pr_{x^a}, r^{-1}\big)\pr^{\leq 1}\psi\right|^2 \\
&&\qquad\qquad\qquad\quad\,\,
+r^{-1}\tau^{-\frac{1+\dec}{2}}\left|\big(\pr_r, r^{-1}\pr_{x^a}, r^{-1}\big)\pr^{\leq 1}\psi\right||\pr_\tau(\pr^{\leq 1}\psi)|\bigg]\\ 
&\les&\bigg|\int_{\MM(\tau_1, \tau_2)}(1-\chi)M^{rr}\pr_r^2\psi \pr_\tau(\pr^{\leq 1}\psi)\bigg|+\bigg|\int_{\MM(\tau_1, \tau_2)}(1-\chi)M^{rx^a}\pr_r\pr_{x^a}\psi \pr_\tau(\pr^{\leq 1}\psi)\bigg|\nn\\
&&+\ep\EM^{(1)}[\psi](\tau_1, \tau_2)
\eeaa}
where we have used Lemma \ref{lemma:basiclemmaforcontrolNLterms} in the last estimate. 

{Next, we control the remaining two terms on the RHS of the above estimate.  Integrating by parts first in $\pr_r$ and then in $\pr_{\tau}$ and using the control of $M^{\a\b}$ as well as Corollary \ref{cor:controloflinearizeddivergencecoordvectorfields}, we have 
\beaa
&& \left|\int_{\MM(\tau_1, \tau_2)}(1-\chi)M^{rr}\pr_r^2\psi \pr_\tau(\pr^{\leq 1}\psi)\right|+\bigg|\int_{\MM(\tau_1, \tau_2)}(1-\chi)M^{rx^a}\pr_r\pr_{x^a}\psi \pr_\tau(\pr^{\leq 1}\psi)\bigg|\\
&\les& \int_{\Mntrap(\tau_1, \tau_2)}\Big(|\dk^{\leq 1}(M^{rr})||\pr_r(\pr^{\leq 1}\psi)|^2+r^{-1}|\dk^{\leq 1}(M^{rr})||\pr_r\psi||\pr_\tau(\pr^{\leq 1}\psi)|\Big)\\
&&+\int_{\Mntrap(\tau_1, \tau_2)}|\dk^{\leq 1}(M^{rx^a})||\pr_r(\pr^{\leq 1}\psi)||(\pr_\tau, \pr_{x^a})(\pr^{\leq 1}\psi)|+\ep\int_{\MM_{10m,11m}(\tau_1, \tau_2)}\tau^{-1-\dec}|\pr^{\leq 2}\psi|^2\\
&& +\ep\int_{\Sigma(\tau_1)\cup\Sigma(\tau_2)}r^{-2}\tau^{-1-\dec}|\dk^{\leq 1}\pr^{\leq 1}\psi|^2+\ep\int_{\II_+(\tau_1,\tau_2)}r^{-1}\tau^{-1-\dec}|\dk^{\leq 1}\pr^{\leq 1}\psi|^2 \\
&\les&\ep\int_{\Mntrap(\tau_1, \tau_2)}\Big(\tau^{-1-\dec}\left|\big(\pr_r, r^{-1}\pr_{x^a}, r^{-1}\big)\pr^{\leq 1}\psi\right|^2+ r^{-2}\left|\pr_\tau\pr^{\leq 1}\psi\right|^2\Big)+ \ep\EM^{(1)}[\psi](\tau_1, \tau_2)\\
&\les& \ep\EM^{(1)}[\psi](\tau_1, \tau_2)
\eeaa}
where we have used Lemma \ref{lem:nullnfFluxBdedByEnergy} in the second last estimate and Lemma \ref{lemma:basiclemmaforcontrolNLterms} in the last estimate.  We deduce
{\beaa
&&\int_{\MM(\tau_1, \tau_2)}(1-\chi)\big|M^{\a\b}\pr_\a\pr_\b\psi\big| {\Big|\big(\pr_r, r^{-1}\pr_\tau, r^{-1}\pr_{x^a}, r^{-1}\big)(\pr^{\leq 1}\psi)\Big|}\\
&&+\bigg|\int_{\MM(\tau_1, \tau_2)}(1-\chi)M^{\a\b}\pr_\a\pr_\b\psi \pr_\tau(\pr^{\leq 1}\psi)\bigg|+\int_{\MM(\tau_1, \tau_2)}(1-\chi)\big|M^{\a\b}\pr_\a\pr_\b\psi\big|^2\\
&\les&\bigg|\int_{\MM(\tau_1, \tau_2)}(1-\chi)M^{rr}\pr_r^2\psi \pr_\tau(\pr^{\leq 1}\psi)\bigg|+\bigg|\int_{\MM(\tau_1, \tau_2)}(1-\chi)M^{rx^a}\pr_r\pr_{x^a}\psi \pr_\tau(\pr^{\leq 1}\psi)\bigg|\nn\\
&&+\ep\EM^{(1)}[\psi](\tau_1, \tau_2)\\
&\les& \ep\EM^{(1)}[\psi](\tau_1, \tau_2)
\eeaa
which concludes the proof of the first estimate.} 

Next, we consider the second estimate. In view of the control of $N^\a$ and the assumptions for $\Ga_g$, we have 
\beaa
&&{\int_{\MM(\tau_1, \tau_2)}(1-\chi)\big|N^\a\pr_\a\psi\big|^2+\int_{\MM(\tau_1, \tau_2)}(1-\chi)\big|N^\a\pr_\a\psi\big| \Big|\big(\pr_\tau, \pr_r, r^{-1}\pr_{x^a}, r^{-1}\big)\pr^{\leq 1}\psi\Big|}\\
&\les&\int_{\Mntrap(\tau_1, \tau_2)}\big|N^\a\pr_\a\psi\big|^2+\int_{\Mntrap(\tau_1, \tau_2)}\big|N^\a\pr_\a\psi\big| \Big|\big(\pr_\tau, \pr_r, r^{-1}\pr_{x^a}, r^{-1}\big)\pr^{\leq 1}\psi\Big|\\
&\les& \ep\int_{\Mntrap(\tau_1, \tau_2)}\Big(r^{-1}\tau^{-\frac{1+\dec}{2}}|\pr_r\psi|\left|\big(\pr_\tau, \pr_r, r^{-1}\pr_{x^a}, r^{-1}\big)\pr^{\leq 1}\psi\right|\\
&&\qquad\qquad\qquad\qquad +r^{-2}|\pr_\tau\psi|\left|\big(\pr_\tau, \pr_r, r^{-1}\pr_{x^a}, r^{-1}\big)\pr^{\leq 1}\psi\right|\\
&&\qquad\qquad\qquad\qquad +r^{-2}\tau^{-\frac{1+\dec}{2}}|\pr_{x^a}\psi|\left|\big(\pr_\tau, \pr_r, r^{-1}\pr_{x^a}, r^{-1}\big)\pr^{\leq 1}\psi\right|\Big)\\
&&+\ep\int_{\Mntrap(\tau_1, \tau_2)}r^{-2}\Big[|\pr_r\psi|^2+r^{-2}|\pr_{x^a}\psi|^2+|\pr_\tau\psi|^2\Big].
\eeaa
Together with the control provided by Lemma \ref{lemma:basiclemmaforcontrolNLterms}, we infer
{\beaa
&&\int_{\MM(\tau_1, \tau_2)}(1-\chi)\big|N^\a\pr_\a\psi\big|^2+\int_{\MM(\tau_1, \tau_2)}(1-\chi)\big|N^\a\pr_\a\psi\big| \Big|\big(\pr_\tau, \pr_r, r^{-1}\pr_{x^a}, r^{-1}\big)\pr^{\leq 1}\psi\Big|\\
&\les& \ep\EM^{(1)}[\psi](\tau_1, \tau_2)
\eeaa
which concludes the proof Lemma \ref{lemma:basiclemmaforcontrolNLterms:bis}.}
\end{proof}

%%%%%%%%%%%%%%%%%%%%%%%%%%%%%%%%%%%%%%%%%%

\subsubsection{Commutators between first-order derivatives and the wave operator}

%%%%%%%%%%%%%%%%%%%%%%%%%%%%%%%%%%%%%%%%%%

The following lemma {provides the structure of commutators between first-order derivatives and the wave operator.}

\begin{lemma}
\lab{lem:commutatorwithwave:firstorderderis}
The commutator between $\square_{\g}$ and $\pr_{\tau}$ satisfies
{\bea
\label{esti:commutatorBoxgandT:general}
\, [ \pr_{\tau}, \square_{\g}]\psi = \pr_{\tau}(\gcheck^{\a\b})\pr_{\a}\pr_{\b}\psi+
\dk^{\leq 2}\Ga_g\c\dk\psi.
\eea
Also, the} commutator between $\square_\g$ and $(\pr_r,r^{-1}\pr_{x^a})$ satisfies 
\bea
\lab{eq:localwavecommutators:withfirstordergoodderis}
{[(\pr_r,r^{-1}\pr_{x^a}), \square_{\g}]\psi} = O(r^{-1}) \pr^{\leq 1}\pr\psi
+r^{-1}\dk^{\leq 2}\Ga_b\dk\pr\psi+r^{-1}\dk^{\leq 2}\Ga_g\dk \psi.
\eea
\end{lemma}

\begin{proof}
{We compute for any coordinate vectorfield $\pr_\a$ 
\beaa
\pr_\a\Box_{\g}&=&\pr_\a\left(\frac{1}{\sqrt{|\g|}} \pr_{\mu}\Big(\sqrt{|\g|}\Big) \g^{\mu\nu}\pr_{\nu}
+\pr_{\mu} (\g^{\mu\nu}\pr_{\nu})\right)
\\
&=&
 \pr_\a\left(\frac{1}{\sqrt{|\g|}} \pr_{\mu}\Big(\sqrt{|\g|}\Big) \g^{\mu\nu}\right)\pr_{\nu}+ \pr_{\mu}(\pr_\a(\g^{\mu\nu})\pr_{\nu})
+\Box_{\g}\pr_{\a}
\\
&=&\pr_\a\left((N_{det})_{\mu}\g^{\mu\nu}
+\frac{1}{\sqrt{|\gam|}} \pr_{\mu}\Big(\sqrt{|\gam|}\Big) \gcheck^{\mu\nu}\right)\pr_{\nu}
+ \pr_{\mu}(\pr_\a(\gcheck^{\mu\nu})\pr_{\nu})+[\pr_\a, \Box_{\gam}]\\
&& +\Box_{\g}\pr_\a,
\eeaa
so that 
\beaa
[\pr_\a, \Box_{\g}] &=& \pr_\a\left((N_{det})_{\mu}\g^{\mu\nu}
+\frac{1}{\sqrt{|\gam|}} \pr_{\mu}\Big(\sqrt{|\gam|}\Big) \gcheck^{\mu\nu}\right)\pr_{\nu}
+ \pr_{\mu}(\pr_\a(\gcheck^{\mu\nu})\pr_{\nu})+[\pr_\a, \Box_{\gam}].
\eeaa
Now, using \eqref{eq:controloflinearizedinversemetriccoefficients}--\eqref{eq:consequenceasymptoticKerrandassumptionsinverselinearizedmetric} 
 and Lemma \ref{lemma:computationofthederiveativeofsrqtg}, we have
 \beaa
\pr_\a\left(\frac{1}{\sqrt{|\gam|}} \pr_{\mu}\Big(\sqrt{|\gam|}\Big)\gcheck^{\mu\nu}\right)\pr_{\nu}=(r^{-1}\dk^{\leq 1}\Ga_b +\dk^{\leq 1}\Ga_g)\c \dk=\dk^{\leq 1}\Ga_g \c\dk
\eeaa
and 
\beaa
\pr_\a\big((N_{det})_{\mu}\g^{\mu\nu}\big)\pr_{\nu}=\dk^{\leq 2}\Ga_g\c\dk, \qquad 
(\pr_{\mu}\pr_\a(\gcheck^{\mu\nu}))\pr_{\nu} = r^{-1}\dk^{\leq 2}\Ga_b\c\dk+\dk^{\leq 2}\Ga_g\c\dk=\dk^{\leq 2}\Ga_g\c\dk,
\eeaa
which yields in view of the above
\beaa
[\pr_\a, \Box_{\g}] &=& [\pr_\a, \Box_{\gam}] + \pr_\a(\gcheck^{\mu\nu})\pr_{\mu}\pr_{\nu}+\dk^{\leq 2}\Ga_g\c\dk.
\eeaa
We deduce, since $[\pr_\tau, \Box_{\gam}]=0$, 
\beaa
[ \pr_{\tau}, \square_{\g}]\psi&=&\pr_{\tau}(\gcheck^{\mu\nu})\pr_{\mu}\pr_{\nu}\psi+
\dk^{\leq 2}\Ga_g\c\dk\psi,
\eeaa
as stated in \eqref{esti:commutatorBoxgandT:general}, as well as 
\beaa
[(\pr_r,r^{-1}\pr_{x^a}), \Box_{\g}] &=& [(\pr_r,r^{-1}\pr_{x^a}), \Box_{\gam}] + r^{-1}\dk(\gcheck^{\mu\nu})\pr_{\mu}\pr_{\nu}+\dk^{\leq 2}\Ga_g\c\dk\\
&=& O(r^{-1}) \pr^{\leq 1}\pr\psi +r^{-1}\dk^{\leq 1}\Ga_b\dk\pr\psi +\dk^{\leq 2}\Ga_g\c\dk\psi
\eeaa
as stated in \eqref{eq:localwavecommutators:withfirstordergoodderis}, where we used again \eqref{eq:controloflinearizedinversemetriccoefficients} in the last line. This concludes the proof of Lemma \ref{lem:commutatorwithwave:firstorderderis}.}
\end{proof}

%%%%%%%%%%%%%%%%%%%%%%%%%%%

\subsection{Local energy estimate}

%%%%%%%%%%%%%%%%%%%%%%%%%%%

We have the following basic local (in time) energy estimate for solutions to the wave equation \eqref{intro:eq:scalarwave}.
\begin{lemma}[Local energy estimate]
\lab{lemma:localenergyestimate}
Let $\g$ satisfy the assumptions of Section \ref{subsubsect:assumps:perturbedmetric}. For any $\tau_0\in\mathbb{R}$ and $q>0$, we have for solutions to the wave equation \eqref{intro:eq:scalarwave} the following future directed local energy estimates, for $s=0,1$,  
\bea
\label{eq:localenergyestimate:future}
\EF^{(\reg)}[\psi](\tau_0, \tau_0+q) \les_q \E^{(\reg)}[\psi](\tau_0) + \widehat{\mathcal{N}}^{(\reg)}[\psi, F](\tau_0, \tau_0+q),
\eea
and the following past directed local energy estimates, for $s=0,1$, 
\bea
\label{eq:localenergyestimate:past}
\EF^{(\reg)}[\psi](\tau_0-q, \tau_0) \les_q \E^{(\reg)}[\psi](\tau_0) +\F^{(\reg)}[\psi](\tau_0-q, \tau_0)+ \widehat{\mathcal{N}}^{(\reg)}[\psi, F](\tau_0-q, \tau_0).
\eea
\end{lemma}

\begin{remark}\lab{rmk:whywearestatingallourresultforatmostonederivative}
Lemma \ref{lemma:localenergyestimate}, as well as Lemmas \ref{lemma:redshiftestimates}, \ref{lemma:redshiftestimates:general} and \ref{lemma:higherorderenergyMorawetzestimates} below, and our main result, see Theorem \ref{thm:main}, are all stated for at most $\reg=1$ {derivatives} of the solution $\psi$ to \eqref{intro:eq:scalarwave}. As mentioned in {Remark \ref{rem:thm1:highorderEMF}}, we could {easily} extend these results to higher order derivatives, but this would require to include more derivatives in the metric assumptions of Section \ref{subsubsect:assumps:perturbedmetric}. Such an extension is fairly standard for perturbations of Kerr in the range $|a|<m$, and we prefer in this work to focus on closing at the minimum number of derivatives for $\psi$, i.e., $\reg=1$.
\end{remark}

\begin{proof}
We first {focus on the case $s=0$}. In the {standard calculation for generalized currents} \eqref{eq:EnerIden:General:wave}, we {choose $w=0$,} and a vector field $X$ that is globally {uniformly} timelike in $\MM$ and equals\footnote{{Note, in view of \eqref{eq:metricinnormalizedcoordinatesgeneralformula} and \eqref{eq:controloflinearizedmetriccoefficients}, that $\g(\pr_\tau, \pr_\tau)=(\gam)_{\tau\tau}+r\Ga_b=-(1-\frac{2mr}{|q|^2})+O(\ep)\les -1$ in $r\geq 3m$.}} $\pr_{\tau}$ for $r\geq 3m${. By} integrating over $\MM(\tau_0, \tau_0+q)$, we infer
\bea
\lab{eq:localenergyestimate:zeroorder:general}
\EF[\psi](\tau_0, \tau_0+q) \les \E[\psi](\tau_0) +\bigg| \int_{\MM(\tau_0, \tau_0+q)} {X(\psi)F} \bigg|+\frac{1}{2} \bigg|\int_{\MM(\tau_0, \tau_0+q)}{}^{(X)} \pi \cdot \QQ[\psi]\bigg|{.}
\eea
{Next, we estimate the last integral in \eqref{eq:localenergyestimate:zeroorder:general}. Since we have in view of Lemma \ref{lemma:controlofdeformationtensorsforenergyMorawetz}, 
\beaa
&& \big({}^{(\pr_{\tau})} \pi\big)^{rr}\in r\dk^{\leq 1}\Ga_b, \qquad \big({}^{(\pr_{\tau})} \pi\big)^{r\tau}\in r\dk^{\leq 1}\Ga_g, \qquad \big({}^{(\pr_{\tau})} \pi\big)^{\tau\tau}\in \dk^{\leq 1}\Ga_g,\\  
&& \big({}^{(\pr_{\tau})} \pi\big)^{rx^a}\in \dk^{\leq 1}\Ga_b, \qquad \big({}^{(\pr_{\tau})} \pi\big)^{\tau x^a}\in \dk^{\leq 1}\Ga_g,  \qquad \big({}^{(\pr_{\tau})} \pi\big)^{x^ax^b}\in r^{-1}\dk^{\leq 1}\Ga_g,
\eeaa
we deduce, as $X=\pr_\tau$ for $r\geq 3m$, 
\beaa
\bigg|\int_{\MM(\tau_0, \tau_0+q)}{}^{(X)} \pi \cdot \QQ[\psi]\bigg| &\les& \int_{\MM_{r_+(1-\dhor), 3m}(\tau_0, \tau_0+q)}|\pr\psi|^2+\int_{\MM_{3m,+\infty}(\tau_0, \tau_0+q)}|\dk^{\leq 1}\Ga_g||\dk\psi|^2\\
&\les& \int_{\MM(\tau_0, \tau_0+q)}r^{-2}|\dk\psi|^2\\
&\les& \int_{\tau_0}^{\tau_0+q} \E[\psi](\tau) d\tau
\eeaa}
which thus yields
\beaa
\EF[\psi](\tau_0, \tau_0+q) \les \E[\psi](\tau_0) +\int_{\tau_0}^{\tau_0+q} \E[\psi](\tau) d\tau+ \bigg|\int_{\MM(\tau_0, \tau_0+q)}{X(\psi)F}\bigg|.
\eeaa
Applying the Gr\"onwall's inequality then proves the desired estimate \eqref{eq:localenergyestimate:future} {in the case $s=0$}. The other inequality \eqref{eq:localenergyestimate:past} {in the case $s=0$} follows in the same manner.

It remains to show \eqref{eq:localenergyestimate:future} and \eqref{eq:localenergyestimate:past} {in the case $s=1$}.  We first commute $\pr_{\tau}$ with the wave equation and derive
{\beaa
\square_{\g}\pr_{\tau}\psi&=&\pr_{\tau} F +[\square_{\g}, \pr_{\tau}]\psi,
\eeaa}
where by \eqref{esti:commutatorBoxgandT:general}, we have
\beaa
\, [ \pr_{\tau}, \square_{\g}]\psi&=& \pr_{\tau}(\gcheck^{\a\b})\pr_{\a}\pr_{\b}\psi+
\dk^{\leq 2}\Ga_g\c\dk\psi.
\eeaa
Applying the energy estimate \eqref{eq:localenergyestimate:zeroorder:general} with $(\psi, F)\rightarrow (\pr_{\tau}\psi, \pr_{\tau} F +[\square_{\g}, \pr_{\tau}]\psi)$, we deduce
\beaa
\EF[\pr_{\tau}\psi](\tau_0, \tau_0+q) &\les& \E[\pr_{\tau}\psi](\tau_0) +\bigg| \int_{\MM(\tau_0, \tau_0+q)} \big(\pr_{\tau} F -\pr_{\tau}(\gcheck^{\a\b})\pr_{\a}\pr_{\b}\psi+
\dk^{\leq 2}\Ga_g\c\dk\psi\big){X\pr_{\tau}\psi} \bigg|\nn\\
&&+\frac{1}{2} \bigg|\int_{\MM(\tau_0, \tau_0+q)}{}^{(X)} \pi \cdot \QQ[\pr_{\tau}\psi]\bigg|.
\eeaa
The last term is estimated in the same manner as above by $\int_{\tau_0}^{\tau_0+q} \E[\pr_{\tau}\psi](\tau) d\tau$ and, noticing that{, in view of \eqref{eq:controloflinearizedinversemetriccoefficients},} $\pr_{\tau}(\gcheck^{\a\b})$ satisfies the assumptions of $M^{\a\b}$ made in Lemma \ref{lemma:basiclemmaforcontrolNLterms:bis}, the {before to} last term is controlled by
\beaa
\widehat{\mathcal{N}}^{(1)}[F, \psi](\tau_0, \tau_0+q) 
+\ep{\EM^{(1)}[\psi](\tau_0,\tau_0+q)}
\eeaa
by using Lemma \ref{lemma:basiclemmaforcontrolNLterms:bis} to estimate
the integral of  $\pr_{\tau}(\gcheck^{\a\b})\pr_{\a}\pr_{\b}\psi{X\pr_{\tau}\psi}$. Consequently, we infer
\bea
\lab{eq:localenergyestimate:firstorder:T:general}
\EF[\pr_{\tau}\psi](\tau_0, \tau_0+q) &\les& \E[\pr_{\tau}\psi](\tau_0) +\widehat{\mathcal{N}}^{(1)}[\psi, F](\tau_0, \tau_0+q) \nn\\
&&
+\int_{\tau_0}^{\tau_0+q} \E[\pr_{\tau}\psi](\tau) d\tau
+\ep {\EM^{(1)}[\psi](\tau_0,\tau_0+q)}.
\eea

Next, we commute the wave equation with $(\pr_r, r^{-1}\pr_{x^a})$ and apply the energy estimate \eqref{eq:localenergyestimate:zeroorder:general} with $(\psi, F)\rightarrow ((\pr_r, r^{-1}\pr_{x^a})\psi, (\pr_r, r^{-1}\pr_{x^a}) F +[\square_{\g}, (\pr_r, r^{-1}\pr_{x^a})]\psi)$. {As above, we have}
\beaa
\frac{1}{2} \bigg|\int_{\MM(\tau_0, \tau_0+q)}{}^{(X)} \pi \cdot \QQ[(\pr_r, r^{-1}\pr_{x^a})\psi]\bigg| {\les \int_{\tau_0}^{\tau_0+q} \E[(\pr_r, r^{-1}\pr_{x^a})\psi](\tau) d\tau}\les  \int_{\tau_0}^{\tau_0+q} \E^{(1)}[\psi](\tau) d\tau. 
\eeaa
{Also, from \eqref{eq:localwavecommutators:withfirstordergoodderis},
we have}
\beaa
&&\bigg| \int_{\MM(\tau_0, \tau_0+q)}[\square_{\g},(\pr_r,r^{-1}\pr_{x^a})]\psi{X(\pr_r, r^{-1}\pr_{x^a})\psi}\bigg|\\ 
&\les& {\int_{\MM(\tau_0, \tau_0+q)}\big(r^{-1}|\pr^{\leq 1}\pr\psi|+\dk^{\leq 2}\Ga_g \dk\pr\psi^{\leq 1}\psi|\big)r^{-1}|\dk\pr\psi|}\\
&\les& {\int_{\MM(\tau_0, \tau_0+q)}r^{-2}|\dk\pr^{\leq 1}\psi|^2} \les \int_{\tau_0}^{\tau_0+q} \E^{(1)}[\psi](\tau) d\tau .
\eeaa
Based on the above{, we deduce}
\beaa
&&\EF[(\pr_r, r^{-1}\pr_{x^a})\psi](\tau_0, \tau_0+q)  \nn\\&\les& \E[(\pr_r, r^{-1}\pr_{x^a})\psi](\tau_0) +\widehat{\mathcal{N}}^{(1)}[\psi, F](\tau_0, \tau_0+q)
+\int_{\tau_0}^{\tau_0+q} \E^{(1)}[\psi](\tau) d\tau,
\eeaa
{which together with \eqref{eq:localenergyestimate:firstorder:T:general} yields, for $\ep$ small enough,
\beaa
\EF^{(1)}(\tau_0, \tau_0+q)  \les \E^{(1)}[\psi](\tau_0) +\widehat{\mathcal{N}}^{(1)}[\psi, F](\tau_0, \tau_0+q)
+\int_{\tau_0}^{\tau_0+q} \E^{(1)}[\psi](\tau) d\tau.
\eeaa
Applying} Gr\"onwall's inequality, we thus arrive at the desired estimate \eqref{eq:localenergyestimate:future} for $s=1$. The other inequality \eqref{eq:localenergyestimate:past} for $s=1$ follows in the same manner. {This concludes the proof of Lemma \ref{lemma:localenergyestimate}.}
 \end{proof}

%%%%%%%%%%%%%%%%%%%%%%%%%%%%

\subsection{Improved Morawetz estimates}

%%%%%%%%%%%%%%%%%%%%%%%%%%%%

The following lemma allows to improve the {$r$-weights} in Morawetz estimates. 
\begin{lemma}\lab{lemma:exteriorMorawetzestimates}
Let $\g$ satisfy the assumptions of Section \ref{subsubsect:assumps:perturbedmetric}. For any $1\leq\tau_1<\tau_2 <+\infty$, and for any $0<\de\leq 1$, we have for solutions to the wave equation \eqref{intro:eq:scalarwave} the improved estimates: 
\bea
\lab{eq:exteriorMorawetzestimates:strong}
\M_{\de}[\psi](\tau_1, \tau_2)&\les& \EMF[\psi](\tau_1, \tau_2)+\bigg|\int_{\MM_{r\geq 11m}(\tau_1, \tau_2)}(1+O(r^{-\de}))\pr_{\tau}\psi F\bigg|
\nn\\
&&
+\int_{\MM_{r\geq 11m}(\tau_1, \tau_2)}|F|\Big|\big(\pr_r, r^{-1}\pr_{x^a}, r^{-1}\big)\psi\Big|
\eea
and
\bea
\lab{eq:exteriorMorawetzestimates:weak}
\M_{\de}[\psi](\tau_1, \tau_2)&\les& \EMF[\psi](\tau_1, \tau_2)+\int_{\MM(\tau_1, \tau_2)}r^{1+\de}|F|^2.
\eea
\end{lemma}

\begin{proof}
Recall from Lemma \ref{lem:generalenergyidentity:waveeq} that
given a vectorfield $X$ and a scalar function $w$, we have the following {standard calculation for generalized currents}
\beaa
&&\nab^{\a}\bigg(\Re\bigg(\QQ_{\a\b}[\psi]X^{\b}+w\psi \ov{\rd_{\a} \psi} -\frac{1}{2}\rd_{\a}w{|\psi|^2}\bigg)\bigg)\\
 &=&\Re\bigg(\frac{1}{2}{}^{(X)} \pi \cdot \QQ[\psi] +w\rd_{\a}\psi \ov{\rd^{\a}\psi} - \frac{1}{2} \Box_{\g} w {|\psi|^2}\bigg)+\Re\big( \Box_{\g} \psi \ov{({X\psi} +w\psi)}\big){,}
\eeaa
{which we apply with the choices
\beaa
X=\mu f\bar{\pr}_r, \qquad w=\mu h, \qquad f=\chi_{R}(1-m^{\de}r^{-\de}),\qquad 
h=\chi_R r^{-1} (1-m^{\de}r^{-\de}),
\eeaa
where $\bar{\pr}_r$ is a coordinate derivative in Boyer--Lindquist coordinates, where $R\geq 14m$ is a constant that will be chosen large enough below, and where $\chi_R=\chi_R(r)$ is a smooth cutoff function that equals $1$ for $r\geq R$ and vanishes for $r\leq R-m$. Integrating the above divergence identity over $\MM(\tau_1, \tau_2)$, we infer
\beaa
&&\left|\int_{\MM(\tau_1, \tau_2)}\Re\bigg(\frac{1}{2}{}^{(X)} \pi \cdot \QQ[\psi] +w\rd_{\a}\psi \ov{\rd^{\a}\psi} - \frac{1}{2} \Box_{\g} w|\psi|^2\bigg)\right|\\ 
&\les& \EF[\psi](\tau_1, \tau_2)+\left|\int_{\MM(\tau_1, \tau_2)}\Re\big(F\ov{({X\psi} +w\psi)}\big)\right|.
\eeaa}

{Next, we introduce the following decomposition 
\beaa
\JJ:=\Re\bigg(\frac{1}{2}{}^{(X)} \pi \cdot \QQ[\psi] +w\rd_{\a}\psi \ov{\rd^{\a}\psi} - \frac{1}{2} \Box_{\g} w |\psi|^2\bigg) = \JJ_{a,m}+\widecheck{\JJ}
\eeaa 
where $\JJ_{a,m}$ denotes the computation of $\JJ$ in Kerr, and where $\widecheck{\JJ}$ is given by
\beaa
\widecheck{\JJ}:=\Re\left(\frac{1}{2}\widecheck{{}^{(X)}\pi}^{\a\b}\pr_a\psi\ov{\pr_\b\psi} -\frac{1}{2}\left(\div(X)\widecheck{\g}^{\mu\nu}+\widecheck{\div(X)}\gam^{\mu\nu}-2w\gcheck^{\mu\nu}\right)\pr_\mu\psi\ov{\pr_\nu\psi } - \frac{1}{2}\widecheck{\Box_{\g}w} |\psi|^2\right).
\eeaa
In view of the above, this yields
\bea\lab{eq:intermediarycomputation:exteriorMorawetzestimates1}
\left|\int_{\MM(\tau_1, \tau_2)}\JJ_{a,m}\right| \les \EF[\psi](\tau_1, \tau_2)+\left|\int_{\MM(\tau_1, \tau_2)}\Re\big(F\ov{({X\psi} +w\psi)}\big)\right|+\int_{\MM(\tau_1, \tau_2)}\left|\widecheck{\JJ}\right|.
\eea}

{Next, we estimate the last term on the RHS of \eqref{eq:intermediarycomputation:exteriorMorawetzestimates1}. To this end, we decompose $\widecheck{\JJ}$ as
\beaa
\bsplit
\widecheck{\JJ}=& \widecheck{\JJ}_1+\widecheck{\JJ}_2+\widecheck{\JJ}_3,\\
\widecheck{\JJ}_1:=& \Re\left(\frac{1}{2}\widecheck{{}^{(X)}\pi}^{\a\b}\pr_\a\psi\ov{\pr_\b\psi}\right),\\
\widecheck{\JJ}_2:=& \Re\left(-\frac{1}{2}\left(\div(X)\widecheck{\g}^{\mu\nu}+\widecheck{\div(X)}\gam^{\mu\nu}-2w\gcheck^{\mu\nu}\right)\pr_\mu\psi\pr_\nu\psi\right),\\
\widecheck{\JJ}_3:=& \Re\left(- \frac{1}{2}\widecheck{\Box_{\g}w} |\psi|^2\right),
\end{split}
\eeaa
and first consider $\widecheck{\JJ}_1$. Note that $\chi_R$ is supported in $[13m,+\infty)$ since $R\geq 14m$, and hence, in view of our choice of normalized coordinates, we have
\bea\lab{eq:decompositionofBLprrwrttocoordinatesvectorfieldnormalizedcoord}
\ov{\pr}_r &=& \left(-\mu^{-1}+\frac{m^2}{r^2}\right)\pr_\tau+\pr_r-\frac{a}{\De}\pr_{\tphi}\quad\textrm{on the support of}\,\,\chi_R.
\eea
Together with the choice of $X$, we deduce
\beaa
{}^{(X)}\pi_{\a\b} &=& \g\left(\D_\a\left(\mu \chi_{R}(1-m^{\de}r^{-\de})\left(\left(-\mu^{-1}+\frac{m^2}{r^2}\right)\pr_\tau+\pr_r-\frac{a}{\De}\pr_{\tphi}\right)\right), \pr_\b\right)\\
&=& O(1){}^{(\pr_\tau)}\pi_{\a\b}+O(1){}^{(\pr_r)}\pi_{\a\b}+O(r^{-2}){}^{(\pr_{\tphi})}\pi_{\a\b} +\de_{\a r}O(r^{-1-\de})(\g_{\tau\b}, \g_{r\b})\\
&&+\de_{\a r}O(r^{-3-\de})\g_{\tphi\b}+\de_{\b r}O(r^{-1-\de})(\g_{\tau\a}, \g_{r\a})+\de_{\b r}O(r^{-3-\de})\g_{\tphi\a},
\eeaa
and hence
\beaa
{}^{(X)}\pi^{\a\b} &=& O(1){}^{(\pr_\tau)}\pi^{\a\b}+O(1){}^{(\pr_r)}\pi^{\a\b}+O(r^{-2}){}^{(\pr_{\tphi})}\pi^{\a\b}+\g^{\a r}O(r^{-1-\de})(\de_{\tau\b}, \de_{r\b})\\
&&+\g^{\a r}O(r^{-3-\de})\de_{\tphi\b}+\g^{\b r}O(r^{-1-\de})(\de_{\tau\a}, \de_{r\a})+\g^{\b r}O(r^{-3-\de})\de_{\tphi\a},
\eeaa
which implies 
\beaa
\widecheck{{}^{(X)}\pi}^{\a\b} &=& O(1){}^{(\pr_\tau)}\pi^{\a\b}+O(1)\widecheck{{}^{(\pr_r)}\pi}^{\a\b}+O(r^{-2}){}^{(\pr_{\tphi})}\pi^{\a\b} +\widecheck{\g}^{\a r}O(r^{-1-\de})(\de_{\tau\b}, \de_{r\b})\\
&&+\widecheck{\g}^{\a r}O(r^{-3-\de})\de_{\tphi\b}+\widecheck{\g}^{\b r}O(r^{-1-\de})(\de_{\tau\a}, \de_{r\a})+\widecheck{\g}^{\b r}O(r^{-3-\de})\de_{\tphi\a}.
\eeaa
In view of Lemma \ref{lemma:controlofdeformationtensorsforenergyMorawetz} and \eqref{eq:controloflinearizedinversemetriccoefficients}, we infer
\beaa
\bsplit
&\widecheck{{}^{(X)}\pi}^{rr}=\chi_R r\dk^{\leq 1}\Ga_b +\chi_R'r\Ga_b, \quad \widecheck{{}^{(X)}\pi}^{r\tau}=\chi_R r\dk^{\leq 1}\Ga_g +\chi_R'r\Ga_g, \quad \widecheck{{}^{(X)}\pi}^{\tau\tau}=\chi_R\dk^{\leq 1}\Ga_g +\chi_R'\Ga_g,\\
&\widecheck{{}^{(X)}\pi}^{\tau a}=\chi_R \dk^{\leq 1}\Ga_g +\chi_R'\Ga_g,\quad \widecheck{{}^{(X)}\pi}^{ra}=\chi_R \dk^{\leq 1}\Ga_b +\chi_R'\Ga_b,\quad \widecheck{{}^{(X)}\pi}^{ab}=\chi_R r^{-1}\dk^{\leq 1}\Ga_g +\chi_R' r^{-1}\Ga_b,
\end{split}
\eeaa
which together with Lemma \ref{lemma:basiclemmaforcontrolNLterms:ter} implies, in view of the definition of $\widecheck{\JJ}_1$, 
\bea\lab{eq:intermediarysteptocontroJJc:controlofJJc1}
\int_{\MM(\tau_1, \tau_2)}\left|\widecheck{\JJ}_1\right|\les \int_{\MM(\tau_1, \tau_2)}\left|\Re\left(\frac{1}{2}\widecheck{{}^{(X)}\pi}^{\a\b}\pr_\a\psi\ov{\pr_\b\psi}\right)\right| \les\ep\EM[\psi](\tau_1, \tau_2).
\eea}

{Next, we consider $\widecheck{\JJ}_2$. We have
\beaa
\div(X)=\pr_\a(X^\a)+\frac{1}{\sqrt{|\g|}}\pr_\a\left(\sqrt{|\g|}\right)X^\a
\eeaa
and hence, in view of the definition of $X$, we infer  
\beaa
\widecheck{\div(X)} = (N_{det})_\a X^\a=O(1)(N_{det})_\tau+O(1)(N_{det})_r+O(r^{-2})(N_{det})_{\tphi}= r\dk^{\leq 1}\Ga_g
\eeaa
where we used Lemma \ref{lemma:computationofthederiveativeofsrqtg}. Together with  \eqref{eq:controloflinearizedinversemetriccoefficients}, \eqref{eq:assymptiticpropmetricKerrintaurxacoord:1}, and the fact that $\div(X)=O(r^{-1})$ and $w=O(r^{-1})$, this yields the following non sharp estimates 
\beaa
\bsplit
&M_2^{\a\b}:=\div(X)\widecheck{\g}^{\mu\nu}+\widecheck{\div(X)}\gam^{\mu\nu}-2w\gcheck^{\mu\nu},\\
&M_2^{rr}\in r\dk^{\leq 1}\Ga_b, \qquad M_2^{r\tau}\in r\dk^{\leq 1}\Ga_g, \qquad M_2^{\tau\tau}\in \dk^{\leq 1}\Ga_g,\\
& M_2^{rx^a}\in \dk^{\leq 1}\Ga_b, \qquad M_2^{\tau x^a}\in \dk^{\leq 1}\Ga_g,  \qquad M_2^{x^ax^b}\in r^{-1}\dk^{\leq 1}\Ga_g,
\end{split}
\eeaa
which together with Lemma \ref{lemma:basiclemmaforcontrolNLterms:ter} implies, in view of the fact that $2\widecheck{\JJ}_2=-\Re(M_2^{\a\b}\pr_\a\psi\ov{\pr_\b\psi})$, 
\bea\lab{eq:intermediarysteptocontroJJc:controlofJJc2}
\int_{\MM(\tau_1, \tau_2)}\left|\widecheck{\JJ}_2\right|\les \int_{\MM(\tau_1, \tau_2)}\left|M_2^{\a\b}\pr_a\psi\ov{\pr_\b\psi}\right| \les\ep\EM[\psi](\tau_1, \tau_2).
\eea}

{Next, we consider $\widecheck{\JJ}_3$. Recalling that $w=\chi_R\mu r^{-1}(1-m^\de r^{-\de})$, we have
\beaa
\bsplit
\widecheck{\square_\g w} =& \widecheck{\g}^{\a\b}\pr_\a\pr_\b w+\frac{1}{\sqrt{|\gam|}}\pr_\a\left(\sqrt{|\gam|}\widecheck{\g}^{\a\b}\right)\pr_\b w+(N_{det})^\a\pr_\a w\\
=& O(r^{-3})\widecheck{\g}^{rr}+O(r^{-2})\frac{1}{\sqrt{|\gam|}}\pr_\a\left(\sqrt{|\gam|}\widecheck{\g}^{\a r}\right)+ O(r^{-2})(N_{det})^r
\end{split}
\eeaa
which together with \eqref{eq:controloflinearizedinversemetriccoefficients} and Lemma \ref{lemma:computationofthederiveativeofsrqtg} yields
\beaa
\widecheck{\square_\g w} = r^{-2}\Ga_b+r^{-1}\dk^{\leq 1}\Ga_b+r^{-1}\dk^{\leq 1}\Ga_g = r^{-1}\dk^{\leq 1}\Ga_b.
\eeaa
As $2\widecheck{\JJ}_3=-\Re(\widecheck{\square_\g w})|\psi|^2$, we deduce, using Lemma \ref{lemma:basiclemmaforcontrolNLterms:ter} with the choice $h=\Re(\widecheck{\square_\g w})$, 
\beaa
\int_{\MM(\tau_1, \tau_2)}\left|\widecheck{\JJ}_3\right|\les \int_{\MM(\tau_1, \tau_2)}\left|\Re(\widecheck{\square_\g w)|\psi|^2}\right| \les\ep\EM[\psi](\tau_1, \tau_2),
\eeaa
which, together with \eqref{eq:intermediarysteptocontroJJc:controlofJJc1} and \eqref{eq:intermediarysteptocontroJJc:controlofJJc2}, and the fact that $\widecheck{\JJ}=\widecheck{\JJ}_1+\widecheck{\JJ}_2+\widecheck{\JJ}_3$, yields
\beaa
\int_{\MM(\tau_1, \tau_2)}\left|\widecheck{\JJ}\right|\les \ep\EM[\psi](\tau_1, \tau_2).
\eeaa
In view of \eqref{eq:intermediarycomputation:exteriorMorawetzestimates1}, we deduce 
\bea\lab{eq:intermediarycomputation:exteriorMorawetzestimates2}
\left|\int_{\MM(\tau_1, \tau_2)}\JJ_{a,m}\right| \les \EF[\psi](\tau_1, \tau_2)+\left|\int_{\MM(\tau_1, \tau_2)}\Re\big(F\ov{({X\psi} +w\psi)}\big)\right|+\ep\EM[\psi](\tau_1, \tau_2).
\eea}

{Next, we compute the LHS of  \eqref{eq:intermediarycomputation:exteriorMorawetzestimates2}, i.e., $\JJ_{a,m}$, where we recall that $\JJ_{a,m}$ denotes the computation of $\JJ$ in Kerr. Recalling that $X=\mu f\bar{\pr}_r$ and $w=\mu h$, with $f=f(r)$ and $h=h(r)$, we have in the Boyer--Lindquist coordinates}
\beaa
{|q|^2\JJ_{a,m}} &=& |q|^2\bigg(\frac{1}{2}{}^{(X)} \pi \cdot \QQ[\psi]+\Re(w\rd_{\a}\psi \ov{\rd^{\a}\psi}) - \frac{1}{2} \Box_{\g} w |\psi|^2\bigg)_{\gam}\nn\\
&=&\bigg(\Delta^2\pr_r\bigg(\frac{f}{2(r^2+a^2)}\bigg) +\mu^2 (r^2+a^2)h\bigg)(\bar\pr_r\psi)^2
\nn\\
&&+\bigg(\frac{1}{2}\pr_r\big(f(r^2+a^2)\big)-h(r^2+a^2) \bigg)(\bar\pr_t\psi)^2
+\bigg(-\frac{1}{2}\pr_r (\mu f) + \mu h\bigg)(\bar\pr_{\th}\psi)^2\nn\\
&&
-\bigg(\frac{1}{\sin^2\th}\bigg(\frac{1}{2}\pr_r (\mu f) - \mu h\bigg)
+h\frac{a^2}{r^2+a^2}-\pr_r\bigg(\frac{a^2 f}{2(r^2+a^2)}\bigg)\bigg)(\bar\pr_{\phi}\psi)^2\nn\\
&&+\bigg(\pr_r\bigg(\frac{2amr}{r^2+a^2}f\bigg) -\frac{4amr}{r^2+a^2}h \bigg)\Re\big(\bar\pr_t\psi \ov{\bar\pr_{\phi}\psi}\big)-\frac{1}{2}\pr_r\big(\pr_r (\mu h) \Delta\big) |\psi|^2,
\eeaa
where {$\bar{\pr}_\a$ denotes coordinate derivatives} in Boyer--Lindquist coordinates. {Recalling that we chosen the functions $f$ and $h$ as}
\beaa
f=\chi_{R}(1-m^{\de}r^{-\de}),\qquad 
h=\chi_R r^{-1} (1-m^{\de}r^{-\de}),
\eeaa
where $\chi_R=\chi_R(r)$ is a smooth cutoff function that equals $1$ for $r\geq R$ and vanishes for {$r\leq R-m$}, {we infer, for $r\geq R$,}
\beaa
{|q|^2\JJ_{a,m}} &=&\frac{1}{|q|^2}\bigg[\bigg(\frac{\de m^{\de}}{2}r^{1-\de} +O(r^{1-2\de})\bigg)(\pr_r\psi)^2
+\bigg(\frac{\de m^{\de}}{2}r^{1-\de} +O(r^{1-2\de})\bigg)(\pr_t\psi)^2
\nn\\
&&
+(r +O(r^{1-\de}))|\nab\psi|^2+O(r^{-2})\Re\big(\pr_t\psi\ov{ \pr_{\phi}\psi}\big)+\bigg(\frac{1}{2}\de (1+\de) m^{\de} r^{-1-\de} +O(r^{-1-2\de})\bigg) |\psi|^2\bigg]\nn\\
&\gtrsim& \de m^{\de} r^{-1-\de}\big( (\pr_{r}\psi)^2 + (\pr_{\tau}\psi)^2\big)
+ r^{-1}|\nab\psi|^2 + \de m^{\de} r^{-3-\de}  \psi^2,
\eeaa
{provided $R\geq 14m$ is chosen} large enough. {This yields 
\beaa
\int_{\MM(\tau_1, \tau_2)}\JJ_{a,m} &\gtrsim& \de\M_{\de, r\geq R}[\psi](\tau_1, \tau_2) - O(1)\M_{R-m, R}[\psi](\tau_1, \tau_2),
\eeaa
which together with \eqref{eq:intermediarycomputation:exteriorMorawetzestimates2} implies 
\beaa
\M_{\de}[\psi](\tau_1, \tau_2)\les \EMF[\psi](\tau_1, \tau_2)+\left|\int_{\MM(\tau_1, \tau_2)}\Re\big(F\ov{({X\psi} +w\psi)}\big)\right|.
\eeaa}
Applying Cauchy--Schwarz to the last term gives the desired estimate \eqref{eq:exteriorMorawetzestimates:weak}.
Since it holds that $X=O(1) \pr_r  + (1+O(r^{-\de})) \pr_{\tau} + O(r^{-1})\nab$ for $r$ large, we substitute this and $w=O(r^{-1})$  into the last term of the previous inequality and  hence prove the other desired estimate \eqref{eq:exteriorMorawetzestimates:strong}. {This concludes the proof of Lemma \ref{lemma:exteriorMorawetzestimates}.}
\end{proof}

%%%%%%%%%%%%%%%%%%%%%%%%%%%%

\subsection{Redshift estimates}

%%%%%%%%%%%%%%%%%%%%%%%%%%%%

To remove degeneracies of the energy in the neighborhood of the horizon, we make use of the Dafermos-Rodnianski redshift vectorfield. The goal of this section is to prove the following proposition.

\begin{lemma}[Redshift estimates]
\lab{lemma:redshiftestimates}
Let $\g$ satisfy the assumptions of Section \ref{subsubsect:assumps:perturbedmetric}. For any $\tau_1<\tau_2$, we have for solutions to the wave equation \eqref{intro:eq:scalarwave}, for $s=0,1$, 
\bea
\nn\EMF^{(s)}_{r\leq r_+(1+\dred)}[\psi](\tau_1, \tau_2)&\les& \E^{(s)}[\psi](\tau_1)+\dred^{-1}\M_{r_+(1+\dred), r_+(1+2\dred)}^{(s)}[\psi](\tau_1, \tau_2)\\
&&+\int_{\MM_{r\leq r_+(1+2\dred)}(\tau_1, \tau_2)}|\pr^{\leq s}F|^2.
\eea
\end{lemma}

{Following Section 9.4 of \cite{GKS}, we} shall in fact prove the {following} more general  statement which applies to equations   which, in the red shift region $r\leq r_+(1+2\dred)$, can be written in the  form,
\bea
\lab{eq:Redshift-gen.equation}
\square_\g\psi={-}\left(C_++{O\left(\bigg|{\frac{r}{r_+}}-1\bigg|\right)}\right)\pr_r\psi +O(1)\pr_\tau\psi+O(1)\pr_{x^a}\psi+F{,}
\eea
where $C_+$   is a {function} satisfying
 \bea
 \lab{eq:Redshift-gen.equation-1}
 C_+\ge   0, \qquad |\pr C_+|\les 1.
 \eea

\begin{lemma}[Redshift estimates-General]
\lab{lemma:redshiftestimates:general}
Let $\g$ satisfy the assumptions of Section \ref{subsubsect:assumps:perturbedmetric}. For any $1\leq\tau_1<\tau_2 <+\infty$, we have for solutions to the wave equation \eqref{eq:Redshift-gen.equation}, with $C_+$ satisfying \eqref{eq:Redshift-gen.equation-1}, for $s=0,1$,  
\bea
\lab{eq:redshift:general:highorderregularity}
\nn\EMF^{(s)}_{r\leq r_+(1+\dred)}[\psi](\tau_1, \tau_2)&\les& \E^{(s)}[\psi](\tau_1)+\dred^{-1}\M_{r_+(1+\dred), r_+(1+2\dred)}^{(s)}[\psi](\tau_1, \tau_2)\\
&&+\int_{\MM_{r\leq r_+(1+2\dred)}(\tau_1, \tau_2)}|\pr^{\leq s}F|^2.
\eea
\end{lemma}

\begin{proof}
We start with the case $s=0$ by first introducing {the} frame 
\beaa
\begin{split}
{\tilde{e}_3}&=-\pr_r +\frac{m^2}{r^2}\pr_{\tau}, \quad {\tilde{e}_4}=\left(\frac{2(r^2+a^2)}{|q|^2}-\frac{m^2}{r^2}\frac{\De}{|q|^2}\right)\pr_\tau +\frac{\De}{|q|^2}\pr_r +\frac{2a}{|q|^2}{\pr_{\tphi}}, \\
 \tilde{e}_1&=\frac{1}{|q|}\pr_{\th}, \qquad\quad\quad \,\,\tilde{e}_2= \frac{1}{|q|\sin\th}(\pr_{\tphi} +a(\sin\th)^2 \pr_{\tau}),
 \end{split}
\eeaa
which {corresponds to the ingoing principal null frame of Kerr. Then, we introduce as in Lemma 9.4.4 in \cite{GKS} the vectorfield 
\beaa
Y:=\underline{d}(r)\tilde{e}_3+d(r)\tilde{e}_4.
\eeaa
Under the condition
\beaa
\sup_{r\leq 4m}\Big(|d(r)|+|d'(r)|+|\underline{d}(r)|+|\underline{d}'(r)|\Big) &\les& 1,
\eeaa
the vectorfield $Y$ satisfies\footnote{{Lemma 9.4.4 in \cite{GKS} relies on the fact that the metric $\g$ is an $O(\ep)$ perturbations of the Kerr metric in the region $r\leq 4m$ which holds true here in view of assumption \eqref{eq:controloflinearizedinversemetriccoefficients}.}} according to Lemma 9.4.4 in \cite{GKS}, for $r\leq 4m$, 
\beaa
\QQ\c{}^{(Y)}\pi &=& \left(\pr_r\left(\frac{\De}{|q|^2}\right)\underline{d}(r) -\frac{\De}{|q|^2}\underline{d}'(r)\right)(\tilde{e}_3(\psi))^2 +d'(r)(\tilde{e}_4(\psi))^2\\
&&+\left(\underline{d}'(r) -\pr_r\left(\frac{\De}{|q|^2}\right)d(r) -\frac{\De}{|q|^2}d'(r)\right)|\tilde{\nab}\psi|^2\\
&& +O(1)(|\underline{d}(r)|+|d(r)|)|\tilde{e}_3(\psi)|\big(|\tilde{e}_4(\psi)|+|\tilde{\nab}(\psi)|\big)+O(\ep)\Big((\tilde{e}_3\psi)^2+(\tilde{e}_4\psi)^2+|\tilde{\nab}\psi|^2\Big).
\eeaa
Next, we introduce, as in Proposition 9.4.7 in \cite{GKS} the vectorfield 
\beaa
Y_\HH := \ka_\HH Y_{(0)}, \qquad Y_{(0)}=Y+2\pr_\tau, \quad \ka_\HH=\ka\left(\frac{\frac{r}{r_+}-1}{\dred}\right),
\eeaa
where $\ka(r)$ is a positive function supported in $[-2,2]$ and equal to 1 on $[-1,1]$, and we choose
\beaa
d(r)=d_0(r-r_+), \qquad \underline{d}(r)=1+\underline{d}_0(r-r_+), 
\eeaa
for constants $d_0>0$ and $\underline{d}_0>0$ to be chosen large enough below. This yields
\beaa
&&\QQ\c{}^{(Y_\HH)}\pi +2\square_\g\psi Y_\HH(\psi)\\
&=& \ka_\HH\Bigg\{\left(\pr_r\left(\frac{\De}{|q|^2}\right)_{r=r_+} +2C_+ \right)(\tilde{e}_3(\psi))^2 +\Big(d_0+O(1)\Big)(\tilde{e}_4(\psi))^2+\Big(\underline{d}_0 +O(1)\Big)|\tilde{\nab}\psi|^2\\
&& +O(1)|\tilde{e}_3(\psi)|\big(|\tilde{e}_4(\psi)|+|\tilde{\nab}(\psi)|\big)+O(\dred+\ep)\Big((\tilde{e}_3\psi)^2+(\tilde{e}_4\psi)^2+|\tilde{\nab}\psi|^2\Big)\Bigg\}\\
&&+O(\dred^{-1})\Big((\tilde{e}_3\psi)^2+(\tilde{e}_4\psi)^2+|\tilde{\nab}\psi|^2\Big){\bf{1}}_{r\in[r_+(1+\dred), r_+(1+2\dred)]}\\
&&+O(1)|\big(|\tilde{e}_3(\psi)|+|\tilde{e}_4(\psi)|+|\tilde{\nab}(\psi)|\big)|F|,
\eeaa
where we have used the fact that $\psi$ satisfies \eqref{eq:Redshift-gen.equation}. Since we have, for $|a|<m$, 
\beaa
\left(\pr_r\left(\frac{\De}{|q|^2}\right)\right)_{r=r_+}=\frac{r_+-r_-}{r_+^2+a^2(\cos\th)^2}\geq \frac{r_+-r_-}{r_+^2+a^2}=\frac{\sqrt{m^2-a^2}}{mr_+}>0,
\eeaa
and $C_+\geq 0$, we may choose $d_0$ and $\underline{d}_0$ large enough such that 
\beaa
&&\QQ\c{}^{(Y_\HH)}\pi +2\square_\g\psi Y_\HH(\psi)\\
&\geq& \frac{\sqrt{m^2-a^2}}{2mr_+}\Big((\pr_r\psi)^2+(\pr_\tau\psi)^2+|\nab\psi|^2\Big){\bf{1}}_{r\leq r_+(1+\dred)}\\
&&-O(\dred^{-1})\Big((\tilde{e}_3\psi)^2+(\tilde{e}_4\psi)^2+|\tilde{\nab}\psi|^2\Big){\bf{1}}_{r\in[r_+(1+\dred), r_+(1+2\dred)]}-O(1)|F|^2.
\eeaa
Using the energy identity of Lemma \ref{lem:generalenergyidentity:waveeq} with the choices $X=Y_\HH$ and $w=0$, and noticing that $Y_\HH$ is timelike, we infer  
\beaa
\nn\EMF_{r\leq r_+(1+\dred)}[\psi](\tau_1, \tau_2)&\les& \E[\psi](\tau_1)+\dred^{-1}\M_{r_+(1+\dred), r_+(1+2\dred)}[\psi](\tau_1, \tau_2)\\
&&+\int_{\MM_{r\leq r_+(1+2\dred)}(\tau_1, \tau_2)}|F|^2
\eeaa
which proves the} $\reg=0$ case of \eqref{eq:redshift:general:highorderregularity}.

It remains to  show {\eqref{eq:redshift:general:highorderregularity} in the case $\reg=1$. To control the commutators, we will simply use the following non-sharp consequence of \eqref{eq:controloflinearizedinversemetriccoefficients} and Lemma \ref{lemma:computationofthederiveativeofsrqtg} which yields, in $r\leq 4m$, 
\beaa
[\pr_\a, h(r,\cos\th)\square_\g]=[\pr_\a, h(r, \cos\th)\square_{\gam}]+O(\ep)\pr^{\leq 2},
\eeaa
for any smooth function $h(r, \cos\th)$. Now, we have, in $r\leq 4m$, 
\beaa
\bsplit
[\pr_\tau, \square_{\gam}]&=0, \qquad [\pr_{\tphi}, \square_{\gam}]=0,\qquad \,[\pr_{x^a}, \square_{\gam}]=O(1)\pr^{\leq 2}, \\
[\pr_r, |q|^2\square_{\gam}]&=[\pr_r, |q|^2\gam^{\mu\nu}\pr_\mu\pr_\nu]+O(1)\pr^{\leq 1}=2(r-m)\pr_r^2+O(1)\pr_\tau\pr+O(1)\pr_{\tphi}\pr+O(1)\pr^{\leq 1},
\end{split}
\eeaa
which implies for solutions $\psi$ of the wave equation \eqref{eq:Redshift-gen.equation} in $r\leq r_+(1+2\dred)$
\beaa
\square_\g(\pr_\a\psi)=-\left(C_{+,\pr_\a}+O\left(\frac{r}{r_+} -1\right)\right)\pr_r(\pr_\a\psi) +O(1)\pr_\tau(\pr_\a\psi)+O(1)\pr_{x^a}(\pr_\a\psi)+F_{\pr_\a},
\eeaa
where the functions $C_{+,\pr_\a}$ are given by 
\beaa
C_{+,\pr_\tau}=C_{+,\pr_{\tphi}}=C_{+,\pr_{x^a}}=C_+>0, \qquad C_{+,\pr_r}=C_++\frac{2(r_+-m)}{r_+^2+a^2(\cos\th)^2}>C_+>0,
\eeaa
and hence all satisfy \eqref{eq:Redshift-gen.equation-1}, and where the functions $F_{\pr_\a}$ are given in $r\leq r_+(1+2\dred)$ by
\beaa
\bsplit
F_{\pr_\tau}&=\pr_\tau F+O(1)\pr\psi+O(\ep)\pr^{\leq 2}\psi, \qquad F_{\pr_{\tphi}}=\pr_{\tphi}F+O(1)\pr\psi+O(\ep)\pr^{\leq 2}\psi, \\
F_{\pr_r}&=\pr_r F+O(1)\pr_\tau\pr+O(1)\pr_{\tphi}\pr+O(1)\pr^{\leq 1}+O(\ep)\pr^{\leq 2}\psi,\qquad F_{\pr_{x^a}}=\pr_{x^a}F+O(1)\pr^{\leq 2}\psi.
\end{split}
\eeaa
Applying \eqref{eq:redshift:general:highorderregularity} with $\reg=0$ to the above wave equations for} $\pr_{\tau}\psi$ and $\pr_{\tphi}\psi$, and {also using \eqref{eq:redshift:general:highorderregularity} with $\reg=0$} to control the {lower order} terms arising from $O(1)\pr\psi$ {in $F_{\pr_\tau}$ and $F_{\pr_{\tphi}}$},  we infer
\bea
\lab{eq:redshift:general:TandPhi:proof}
&&\EMF_{r\leq r_+(1+\dred)}[(\pr_{\tau}, \pr_{\tphi})^{\leq 1}\psi](\tau_1, \tau_2)\nn\\
&\les& \E[(\pr_{\tau}, \pr_{\tphi})^{\leq 1}\psi](\tau_1)+\dred^{-1}\M_{r_+(1+\dred), r_+(1+2\dred)}^{(1)}[\psi](\tau_1, \tau_2)\nn\\
&&+\int_{\MM_{r_+(1-\dhor), r_+(1+2\dred)}(\tau_1, \tau_2)}|(\pr_{\tau}, \pr_{\tphi})^{\leq 1}F|^2 \nn\\
&&
+\ep\M_{r_+(1-\dhor), {r_+(1+\dred)}}^{(1)}[\psi](\tau_1, \tau_2).
\eea

{Next, we apply \eqref{eq:redshift:general:highorderregularity} with $\reg=0$ to the above wave equations for $\pr_r\psi$, using also  \eqref{eq:redshift:general:TandPhi:proof} to control $F_{\pr_r}$. This yields}
\beaa
&&\EMF_{r\leq r_+(1+\dred)}[(\pr_{\tau}, \pr_{\tphi}, \pr_r)^{\leq 1}\psi](\tau_1, \tau_2)\nn\\
&\les& \E[(\pr_{\tau}, \pr_{\tphi},\pr_r)^{\leq 1}\psi](\tau_1)+\dred^{-1}\M_{r_+(1+\dred), r_+(1+2\dred)}^{(1)}[\psi](\tau_1, \tau_2)\nn\\
&&+\int_{\MM_{r_+(1-\dhor), r_+(1+2\dred)}(\tau_1, \tau_2)}|(\pr_{\tau}, \pr_{\tphi}, \pr_r)^{\leq 1}F|^2 
+\ep\M_{r_+(1-\dhor), {r_+(1+\dred)}}^{(1)}[\psi](\tau_1, \tau_2).
\eeaa
{Then,} expanding the wave equation into
\beaa
\mathring{\gamma}^{ab}\pr_{x^a}\pr_{x^b} \psi = O(1) (\pr_{\tau}, \pr_{\tphi}, \pr_r)^{\leq 1}\pr \psi +O(1)\nab\psi + O(1)F
+O(\ep) \pr^2\psi, \quad {\textrm{for}\,\,r\leq 4m,}
\eeaa
we make use of the above estimate{, together with elliptic estimates on the spheres $S(\tau, r)$ foliating $r\leq r_+(1+\dred)$,} to infer
\bea
\lab{eq:redshift:general:TandPhiandprr:proof}
\begin{split}
&{\EF_{r\leq r_+(1+\dred)}[(\pr_{\tau}, \pr_{\tphi}, \pr_r)^{\leq 1}\psi](\tau_1, \tau_2)
+\M^{(1)}_{r\leq r_+(1+\dred)}[\psi](\tau_1, \tau_2)}\\
\les {}&\E[(\pr_{\tau}, \pr_{\tphi},\pr_r)^{\leq 1}\psi](\tau_1)+\dred^{-1}\M_{r_+(1+\dred), r_+(1+2\dred)}^{(1)}[\psi](\tau_1, \tau_2)\\
&+\int_{\MM_{r_+(1-\dhor), r_+(1+2\dred)}(\tau_1, \tau_2)}|(\pr_{\tau}, \pr_{\tphi}, \pr_r)^{\leq 1}F|^2 
+\ep\M_{r_+(1-\dhor), {r_+(1+\dred)}}^{(1)}[\psi](\tau_1, \tau_2).
\end{split}
\eea

{Finally, we apply \eqref{eq:redshift:general:highorderregularity} with $\reg=0$ to the above wave equations for $\pr_{x^a}\psi$, using also  \eqref{eq:redshift:general:TandPhiandprr:proof}, to control $F_{\pr_{x^a}}$. This yields}
\beaa
&&\EMF^{(1)}_{r\leq r_+(1+\dred)}[\psi](\tau_1, \tau_2)
\nn\\
&\les& \E^{(1)}[\psi](\tau_1)+\dred^{-1}\M_{r_+(1+\dred), r_+(1+2\dred)}^{(1)}[\psi](\tau_1, \tau_2)\nn\\
&&+\int_{\MM_{r_+(1-\dhor), r_+(1+2\dred)}(\tau_1, \tau_2)}|\pr^{\leq 1}F|^2 
+\ep\M_{r_+(1-\dhor), {r_+(1+\dred)}}^{(1)}[\psi](\tau_1, \tau_2),
\eeaa
and the desired estimate \eqref{eq:redshift:general:highorderregularity} {for the case $\reg=1$} then follows from the above inequality  by taking $\ep$ small enough. {This concludes the proof of Lemma \ref{lemma:redshiftestimates:general}.} 
\end{proof}

%%%%%%%%%%%%%%%%%%%%%%%%%%%%%%%%%%%%%%%

\subsection{Conditional higher order derivatives energy-Morawetz estimates}

%%%%%%%%%%%%%%%%%%%%%%%%%%%%%%%%%%%%%%%

The following lemma allows to derive conditional higher order derivatives energy and Morawetz estimates. 

\begin{lemma}\lab{lemma:higherorderenergyMorawetzestimates}
Let $\g$ satisfy the assumptions of Section \ref{subsubsect:assumps:perturbedmetric}. For any $1\leq\tau_1<\tau_2 <+\infty$, we have for solutions to the wave equation \eqref{intro:eq:scalarwave} the improved estimate
\bea\lab{eq:higherorderenergymorawetzforlargerconditionalonsprtau}
\nn\EMF^{(1)}_{r\geq 11m}[\psi](\tau_1, \tau_2) &\les& \EMF[\psi](\tau_1, \tau_2)+\EMF[\pr_\tau\psi](\tau_1, \tau_2)
{+\widehat{\mathcal{N}}^{(1)}_{r\geq 10m}[\psi, F](\tau_1, \tau_2)}\\
&&+\big(\EMF[\psi](\tau_1, \tau_2)\big)^{\frac{1}{2}}\left(\EMF^{(1)}[\psi](\tau_1, \tau_2)\right)^{\frac{1}{2}},
\eea
and
\bea\lab{eq:higherorderenergymorawetzconditionalonsprtausprphis}
\nn\EMF^{(1)}[\psi](\tau_1, \tau_2) &\les& \E^{(1)}[\psi](\tau_1)+\EMF[\psi](\tau_1, \tau_2)+\EMF[\pr_\tau\psi](\tau_1, \tau_2)\\
&&+\EMF_{r_+(1-\dhor), 11m}[\pr_{\tphi}\psi](\tau_1, \tau_2)+{\widehat{\mathcal{N}}^{(1)}_{r\geq 10m}[\psi, F](\tau_1, \tau_2)}\nn\\
&&{+\int_{\MM(\tau_1, \tau_2)}|\pr^{\leq 1}F|^2.}
\eea
Furthermore, we have, for any $0<\de\leq 1$ and for $s=0,1$, 
\bea\lab{eq:imporvedhigherorderenergymorawetzconditionalonlowerrweigth}
\M^{(s)}_{\de}[\psi](\tau_1, \tau_2)\les \EMF^{(s)}[\psi](\tau_1, \tau_2)+\int_{\MM(\tau_1, \tau_2)}r^{1+\de}|\pr^{\leq s}F|^2.
\eea
\end{lemma}

\begin{proof}
The proof proceeds in the following steps.

{\noindent{\bf Step 1.} \textit{Proof of  \eqref{eq:higherorderenergymorawetzforlargerconditionalonsprtau}: Morawetz part}. In view of \eqref{eq:controloflinearizedinversemetriccoefficients}, \eqref{eq:assymptiticpropmetricKerrintaurxacoord:1} and Lemma \ref{lemma:computationofthederiveativeofsrqtg}, we have the following non-sharp identity
\beaa
\bsplit
\square_\g -\square_{\gam} =& \widecheck{\g}^{\a\b}\pr_\a\pr_\b +\frac{1}{\sqrt{|\gam|}}\pr_\a\left(\sqrt{|\gam|}\widecheck{\g}^{\a\b}\right)\pr_\b +(N_{det})^\a\pr_\a \\
=& O(\ep)(r^{-1}\pr_\tau, \pr_r, \nab)\pr+O(\ep r^{-1})\pr
\end{split}
\eeaa
which allows to rewrite the wave equation \eqref{intro:eq:scalarwave} as
\beaa
\gam^{\a\b}\pr_\a\pr_\b\psi &=& F+O(\ep)(r^{-1}\pr_\tau, \pr_r, \nab)\pr\psi+O(r^{-1})\pr\psi.
\eeaa
We infer
\bea\label{eq:commutewave:ellipticesti:high}
\gam^{ij}\pr_i\pr_j\psi &=& F+O(1)(\pr_r, r^{-1}\pr_\tau, \nab)\pr_\tau\psi
+O(\ep)(\pr_r, \nab)^2\psi+O(r^{-1})\pr\psi,
\eea
with $1\leq i,j\leq 3$ indices corresponding to the $(r, x^1, x^2)$ coordinates. Taking the square of both sides of \eqref{eq:commutewave:ellipticesti:high}, multiplying both squares} by $\chi_1(r) r^{-2}$ with $\chi_1$ a smooth cutoff function satisfying
\beaa
\chi_1(r)=1 \,\,\text{for } r\geq 10.5m, \qquad  \chi_1(r)=0 \,\,\text{for }r\leq 10m, 
\eeaa
and integrating over $\MM(\tau_1,\tau_2)$, we deduce
{\beaa
&&\int_{\MM(\tau_1,\tau_2)} \chi_1 r^{-2} |\gam^{ij}\pr_i \pr_j\psi |^2 \nn\\
&\les & {\M_{r\geq 10m}[\pr_{\tau}^{\leq 1}\psi](\tau_1,\tau_2)}
+\int_{{\MM_{r\geq 10m}}(\tau_1,\tau_2)}  r^{-2}  |F|^2 +\ep^2\int_{\MM(\tau_1,\tau_2)} \chi_1 r^{-2}|(\pr_r, \nab)^2\psi |^2.
\eeaa
Now, noticing from \eqref{eq:inverse:hypercoord} that 
\bea\lab{eq:ellipticityofinversemetricrestritectorxaforrgeq10m}
\gam^{ij}\pr_i\psi\pr_j\psi\gtrsim |(\pr_r, \nab)\psi|^2\quad\textrm{for}\,\, r\geq 10m,
\eea
and using integration by parts, we have
\bea
\label{eq:highorderMora:rgeq11m:otherderis}
&&\int_{\MM(\tau_1,\tau_2)}  \chi_1 r^{-2} |(  \pr_r, \nab)^2\psi|^2 \nn\\
&\les&\int_{\MM(\tau_1,\tau_2)} \chi_1 r^{-2} |\gam^{ij}\pr_i \pr_j\psi |^2+\sqrt{\M[\psi](\tau_1,\tau_2)}\sqrt{\M^{(1)}[\psi](\tau_1,\tau_2)},
\eea
where the lower order terms are absorbed in the last term on the RHS, and where the boundary terms vanish on $\Si(\tau_1)$ and $\Si(\tau_2)$ because $\pr_i$ are tangent and on $\II_+$ because of the $r^{-2}$ weight. Together} with the previous estimate, this  yields{, for $\ep$ small enough,
\beaa
&&\int_{\MM(\tau_1,\tau_2)}   r^{-2}\chi_1 |(  \pr_r, \nab)^2\psi|^2 \nn\\
&\les&{\M_{r\geq 10m}[\pr_{\tau}^{\leq 1}\psi](\tau_1,\tau_2)}
+\int_{\MM_{{r\geq 10m}}(\tau_1,\tau_2)} r^{-2}  | F|^2 +\sqrt{\M[\psi](\tau_1,\tau_2)}\sqrt{\M^{(1)}[\psi](\tau_1,\tau_2)}.
\eeaa
Also, multiplying \eqref{eq:commutewave:ellipticesti:high} with $r^{-3}\gam^{bc}\pr_{x^b}\pr_{x^c}\psi$, integrating over $\MM(\tau_1, \tau_2)$, and using the following consequence of integration by parts 
\bea\label{eq:highorderMora:rgeq11m:otherderis:angularderivatives}
&&\int_{\MM(\tau_1, \tau_2)}\chi_1r^{-1}|\nab(\pr, \nab)\psi|^2\nn\\
&\les& \int_{\MM(\tau_1, \tau_2)}\chi_1r^{-3}\gam^{ij}\pr_i\psi\pr_j\psi \gam^{bc}\pr_{x^b}\pr_{x^c}\psi+\sqrt{\M[\psi](\tau_1,\tau_2)}\sqrt{\M^{(1)}[\psi](\tau_1,\tau_2)},
\eea
where the lower order terms are absorbed in the last term on the RHS, and where the boundary terms vanish on $\Si(\tau_1)$ and $\Si(\tau_2)$ because $\pr_i$ are tangent and on $\II_+$ because of the $r^{-1}$ weight, we obtain, for $\ep$ small enough, and after integrating by parts various terms on the RHS to distribute angular derivatives, 
\beaa
&&\int_{\MM(\tau_1, \tau_2)}\chi_1r^{-1}|\nab(\pr, \nab)\psi|^2\\
&\les& {\M_{r\geq 10m}[\pr_{\tau}^{\leq 1}\psi](\tau_1,\tau_2)}
+\int_{\MM_{{r\geq 10m}}(\tau_1,\tau_2)} r^{-1}  | F|^2 +\sqrt{\EM[\psi](\tau_1,\tau_2)}\sqrt{\EM^{(1)}[\psi](\tau_1,\tau_2)}.
\eeaa
Together with the above estimate for $\int_{\MM(\tau_1,\tau_2)}   r^{-2}\chi_1 |(  \pr_r, \nab)^2\psi|^2$, we infer
\bea
\lab{eq:EM:AsympKerr:highMora:v1:mid}
\M^{(1)}_{r\geq 10.5m}[\psi](\tau_1,\tau_2) 
&\les& {\M_{r\geq 10m}[\pr_{\tau}^{\leq 1}\psi](\tau_1,\tau_2)}
+\int_{\MM_{{r\geq 10m}}(\tau_1,\tau_2)} r^{-1}  | F|^2\nn\\
&& +\sqrt{\EM[\psi](\tau_1,\tau_2)}\sqrt{\EM^{(1)}[\psi](\tau_1,\tau_2)}.
\eea}

{\noindent{\bf Step 2.} \textit{Proof of  \eqref{eq:higherorderenergymorawetzforlargerconditionalonsprtau}: energy part}. Let $n$ be any integer such that $[n,n+1]\subset(\tau_1, \tau_2)$. We take the square of both sides of equation \eqref{eq:commutewave:ellipticesti:high}, multiply both squares by $\chi_2^2$ where 
\beaa
\chi_2(r)=1 \,\,\text{for } r\geq 11m, \qquad  \chi_2(r)=0 \,\,\text{for }r\leq 10.5m, 
\eeaa
integrate\footnote{{While the LHS of \eqref{eq:commutewave:ellipticesti:high} only contains derivatives that are tangential to $\Si(\tau)$, integrating its square by parts  on $\Si(\tau)$ would generate boundary terms on $\Si(\tau)\cap\II_+$. To avoid such problematic boundary terms, we instead perform in this step integrations on $\MM(n,n+1)$.}} over $\MM(n,n+1)$, and use an estimate analog to \eqref{eq:highorderMora:rgeq11m:otherderis}. This yields
\beaa
&&\int_{\MM(n,n+1)}\chi_2^2|(\pr_r, \nab)^2 \psi |^2 \nn\\
&\les &
\int_{\MM(n,n+1)}\chi_2^2|\g^{ij}\pr_i\pr_j \psi |^2 +\sqrt{\MF{[\psi]}(\tau_1,\tau_2)}\sqrt{\MF^{(1)}{[\psi]}(\tau_1,\tau_2)}\\
\nn\\
&\les & \sup_{\tau\in[\tau_1, \tau_2]}\E_{{r\geq 10.5m}}[{\pr_{\tau}^{\leq 1}}\psi](\tau)+ \int_{\MM{_{r\geq 10.5m}(n,n+1)}}|F|^2
 \\
&&+\ep^2\int_{\MM(n,n+1)}\chi_2|(\pr_r, \nab)^2 \psi |^2+\sqrt{\MF{[\psi]}(\tau_1,\tau_2)}\sqrt{\MF^{(1)}{[\psi]}(\tau_1,\tau_2)},
\eeaa
and hence, for $\ep$ small enough,
\beaa
\int_{\MM(n,n+1)}\chi_2^2|(\pr_r, \nab)^2 \psi |^2 
&\les & \sup_{\tau\in[\tau_1, \tau_2]}\E_{{r\geq 10.5m}}[{\pr_{\tau}^{\leq 1}}\psi](\tau)+ \int_{\MM{_{r\geq 10.5m}(n,n+1)}}|F|^2 \nn\\
&&+\sqrt{\MF{[\psi]}(\tau_1,\tau_2)}\sqrt{\MF^{(1)}{[\psi]}(\tau_1,\tau_2)}.
\eeaa}
 
By the mean-value theorem, there exists $\tau_n\in [n,n+1]$ such that
\beaa
\int_{\Sigma(\tau_n)}\chi_2^2|(\pr_r, \nab)^2 \psi |^2
&\les &\int_{\MM(n,n+1)}\chi_2^2|(\pr_r,\nab)^2 \psi |^2.
\eeaa
Hence, it follows that
\beaa
\E^{(1)}[\chi_2\psi] (\tau_n)
&\les &\E[{\chi_2}\pr_{\tau}\psi](\tau_n) 
+\int_{\Sigma(\tau_n)}\chi_2^2|(\pr_r, \nab)^2 \psi |^2
+\E_{{r\geq 10.5m}}[\psi](\tau_n)
\nn\\
&\les & \sup_{\tau\in[\tau_1, \tau_2]}\E_{{r\geq 10.5m}}[{(\pr_{\tau})^{\leq 1}}\psi](\tau)+ \int_{\MM{_{r\geq 10.5m}(n,n+1)}}|F|^2 \nn\\
&&+\sqrt{\MF{[\psi]}(\tau_1,\tau_2)}\sqrt{\MF^{(1)}{[\psi]}(\tau_1,\tau_2)}.
\eeaa
{Applying local energy estimate to $\pr^{\leq 1}(\chi_2\psi)$ using the vectorfield $\pr_{\tt}$,} we deduce
\beaa
&&\EF^{(1)}[\chi_2\psi](\tau_n, \min(\tau_n+2, \tau_2))\\ 
&\les& \E^{(1)}[\chi_2\psi](\tau_n)+{\bigg|\int_{\MM(\tau_n, \min(\tau_n+2, \tau_2))}\Re\Big(\ov{\pr_{\tt}\pr^{\leq 1}(\chi_2\psi)} \pr^{\leq 1}\big(\chi_2F+[\square_\g, \chi_2]\psi\big)\Big)\bigg|}\\
&&{+\bigg|\int_{\MM(\tau_n, \min(\tau_n+2, \tau_2))}\Re\Big(\ov{\pr_{\tt}\pr^{\leq 1}(\chi_2\psi)} \pr^{\leq 2}(\chi_2\psi)\Big)\bigg|}\\
&\les& \sup_{\tau\in[\tau_1, \tau_2]}\E_{{r\geq 10.5m}}[{(\pr_{\tau})^{\leq 1}}\psi](\tau)
+{\bigg|\int_{\MM(\tau_n, \min(\tau_n+2, \tau_2))}\Re\Big(\ov{\pr_{\tt}\pr^{\leq 1}(\chi_2\psi)} \pr^{\leq 1}\big(\chi_2F\big)\Big)\bigg|} \\
&& {+\int_{\MM_{r\geq 10.5m}(n,n+1)}|F|^2} +\sqrt{\MF{[\psi]}(\tau_1,\tau_2)}\sqrt{\MF^{(1)}{[\psi]}(\tau_1,\tau_2)}\\
&&+\M^{(1)}_{10.5m, 11m}[\psi]{(\tau_n, \min(\tau_n+2, \tau_2))}.
\eeaa
Plugging \eqref{eq:EM:AsympKerr:highMora:v1:mid} to control the last term on the RHS, we infer
\beaa
&&\EF^{(1)}[\chi_2\psi](\tau_n, \min(\tau_n+2, \tau_2))\\ 
&\les& \EM_{{r\geq 10m}} [{(}\pr_{\tau}{)^{\leq 1}}\psi](\tau_1,\tau_2)+\sqrt{\EMF{[\psi]}(\tau_1,\tau_2)}\sqrt{\EMF^{(1)}{[\psi]}(\tau_1,\tau_2)}\\
&&+{\bigg|\int_{\MM(\tau_n, \min(\tau_n+2, \tau_2))}\Re\Big(\ov{\pr_{\tt}\pr^{\leq 1}(\chi_2\psi)} \pr^{\leq 1}\big(\chi_2F\big)\Big)\bigg|}+ {\sup_{\tau\in[\tau_1,\tau_2-2]} \int_{\MM_{r\geq {10m}}(\tau,\tau+2)}}|F|^2.
\eeaa
Proceeding similarly on $(\tau_1, \tau_1+2)$ and noticing that the above time intervals cover $(\tau_1, \tau_2)$, we infer
\bea\lab{eq:highorderEMF:energy:pre2}
&&\sup_{\tau\in[\tau_1, \tau_2]}\E^{(1)}[\chi_2\psi](\tau)+\F_{\II_+}^{(1)}[\psi](\tau_1, \tau_1+1)+\F_{\II_+}^{(1)}[\psi](\tau_2-1, \tau_2)\nn\\ 
\nn&\les&  \EM_{{r\geq 10m}} [{(\pr_{\tau})^{\leq 1}\psi}](\tau_1,\tau_2) +\sqrt{\EMF{[\psi]}(\tau_1,\tau_2)}\sqrt{\EMF^{(1)}{[\psi]}(\tau_1,\tau_2)}\\
&&+{\bigg|\int_{\MM(\tau_n, \min(\tau_n+2, \tau_2))}\Re\Big(\ov{\pr_{\tt}\pr^{\leq 1}(\chi_2\psi)} \pr^{\leq 1}\big(\chi_2F\big)\Big)\bigg|}\nn\\
&&+ {\sup_{\tau\in[\tau_1,\tau_2-2]} \int_{\MM_{r\geq {10.5m}}(\tau,\tau+2)}}|F|^2,
\eea
and hence{, in view of the definition of $\widehat{\mathcal{N}}^{(1)}_{r\geq 10m}[\psi, F](\tau_1, \tau_2)$ in Section \ref{subsect:norms},}
\bea
\lab{eq:highorderEMF:energylarger}
\nn&&\sup_{\tau\in[\tau_1,\tau_2]}\E^{(1)}_{r\geq 11m}[\psi](\tau){+\F_{\II_+}^{(1)}[\psi](\tau_1, \tau_1+1)+\F_{\II_+}^{(1)}[\psi](\tau_2-1, \tau_2)}\\
&\les&  \EM_{{r\geq 10m}} [{(\pr_{\tau})^{\leq 1}\psi}](\tau_1,\tau_2)  +\sqrt{\EMF{[\psi]}(\tau_1,\tau_2)}\sqrt{\EMF^{(1)}{[\psi]}(\tau_1,\tau_2)}\nn\\
&&+\widehat{\mathcal{N}}^{(1)}_{r\geq 10m}[\psi, F](\tau_1, \tau_2).
\eea

{\noindent{\bf Step 3.} \textit{Proof of  \eqref{eq:higherorderenergymorawetzforlargerconditionalonsprtau}: flux part}. We start by using \eqref{eq:controloflinearizedinversemetriccoefficients}, \eqref{eq:assymptiticpropmetricKerrintaurxacoord:1} and Lemma \ref{lemma:computationofthederiveativeofsrqtg} to write the wave operator as 
\beaa
\square_\g &=& \pr_r^2-2\pr_\tau\pr_r+\frac{1}{r^2}\mathring{\ga}^{ab}\pr_{x^a}\pr_{x^b}+O(r^{-1})\pr(\pr^{\leq 1}\psi).
\eeaa
Recalling from \eqref{expression:prxaIIplus:nullinf} and \eqref{expression:prtauIIplus:nullinf} that
\beaa
\pr_{x^a}^{\II_+}=\pr_{x^a}+O(\ep)\pr_r,\,\,\, a=1,2,  \qquad \pr_\tau^{\II_+}=\pr_\tau -\frac{1}{2}(1+b^r)\pr_r+O(\ep)\nab, \qquad |\dk^{\leq 1}b^r|\les \ep,
\eeaa
we infer on $\II_+$
\beaa
\square_\g &=& -2\pr_\tau^{\II_+}\pr_r+\frac{1}{r^2}\mathring{\ga}^{ab}\pr_{x^a}^{\II_+}\pr_{x^b}^{\II_+}+O(\ep)(\pr_r, \nab)\pr_r+O(r^{-1})\pr\pr^{\leq 1}\\
&=& 4(\pr^{\II_+}_\tau)^2+\frac{1}{r^2}\mathring{\ga}^{ab}\pr_{x^a}^{\II_+}\pr_{x^b}^{\II_+}+O(1)\pr^{\II_+}_\tau\pr_\tau +O(\ep)(\pr_r, \pr^{\II_+}_\tau, \nab^{\II_+})\pr_r+O(r^{-1})\pr\pr^{\leq 1}
\eeaa
and hence
\beaa
\left(4(\pr^{\II_+}_\tau)^2+\frac{1}{r^2}\mathring{\ga}^{ab}\pr_{x^a}^{\II_+}\pr_{x^b}^{\II_+}\right)\psi = F+O(1)\pr^{\II_+}_\tau(\pr_\tau\psi) +O(\ep)(\pr_r, \pr^{\II_+}_\tau, \nab^{\II_+})\pr_r\psi+O(r^{-1})\pr(\pr^{\leq 1}\psi).
\eeaa
Taking the square on both sides and integrating on $\II_+(\tau_1, \tau_2)$, and multiplying both squares with $\chi_3^2$ where $\chi_3=\chi_3(\tau)$ is a smooth cut-off supported in $(\tau_1, \tau_2)$ and $\chi_3=1$ on $(\tau_1+1, \tau_2-1)$,  we infer
\beaa
&&\int_{\tau_1}^{\tau_2}\int_{\mathbb{S}^2}\chi_3^2\left|\left(4(\pr^{\II_+}_\tau)^2+\frac{1}{r^2}\mathring{\ga}^{ab}\pr_{x^a}^{\II_+}\pr_{x^b}^{\II_+}\right)\psi\right|^2(\tauu=+\infty, \tau, \omega)r^2d\mathring{\ga}d\tau\\
&\les& \int_{\tau_1}^{\tau_2}\int_{\mathbb{S}^2}\chi_3^2\Big(|F|^2+|\pr^{\II_+}_\tau(\pr_\tau\psi)|^2+\ep^2|(\pr^{\II_+}_\tau, \nab^{\II_+})\pr_r\psi|^2\Big)(\tauu=+\infty, \tau, \omega)r^2d\mathring{\ga}d\tau,
\eeaa
where the remaining terms vanish on $\II_+(\tau_1, \tau_2)$ in view of Lemma \ref{lem:nullnfFluxBdedByEnergy}. Integrating by parts the LHS, we obtain
\beaa
\F_{\II_+}[(\pr^{\II_+}_\tau, \nab^{\II_+})\psi](\tau_1+1, \tau_2-1)&\les&  {\int_{\II_+(\tau_1,\tau_2)}|F|^2} + \F_{\II_+}[\pr_\tau\psi](\tau_1, \tau_2)\\
&&+\sqrt{\F_{\II_+}[\psi](\tau_1, \tau_2)}\sqrt{\F^{(1)}_{\II_+}[\psi](\tau_1, \tau_2)}+\ep^2\F^{(1)}_{\II_+}[\psi](\tau_1, \tau_2).
\eeaa
Using again the fact that $\pr_\tau^{\II_+}=\pr_\tau -\frac{1}{2}(1+b^r)\pr_r+O(\ep)\nab$, as well as the control of $\F_{\II_+}^{(1)}[\psi](\tau_1, \tau_1+1)$ and $\F_{\II_+}^{(1)}[\psi](\tau_2-1, \tau_2)$ provided by \eqref{eq:highorderEMF:energy:pre2}, we infer
\beaa
\F^{(1)}_{\II_+}[\psi](\tau_1, \tau_2)&\les& \ep^2\F^{(1)}_{\II_+}[\psi](\tau_1, \tau_2)+ {\EMF_{r\geq 10m}[\pr_{\tau}^{\leq 1}\psi](\tau_1,\tau_2)}\\
&&+{\sup_{\tau\in[\tau_1,\tau_2-2]}}\widehat{\mathcal{N}}^{(1)}_{r\geq {10.5}m}[\psi, F]({\tau, \tau+2})+ \int_{{\II_+(\tau_1,\tau_2)}}|F|^2\\
&&+\sqrt{\EMF{[\psi]}(\tau_1,\tau_2)}\sqrt{\EMF^{(1)}{[\psi]}(\tau_1,\tau_2)},
\eeaa
which yields for $\ep>0$ small enough
{\bea
\F^{(1)}_{\II_+}[\psi](\tau_1, \tau_2)&\les&  \EMF_{r\geq 10m}[\pr_{\tau}^{\leq 1}\psi](\tau_1,\tau_2)+{\sup_{\tau\in[\tau_1,\tau_2-2]}}\widehat{\mathcal{N}}^{(1)}_{r\geq 10.5m}[\psi, F]({\tau, \tau+2})\nn\\
&&+ \int_{{\II_+(\tau_1,\tau_2)}}|F|^2+\sqrt{\EMF{[\psi]}(\tau_1,\tau_2)}\sqrt{\EMF^{(1)}{[\psi]}(\tau_1,\tau_2)}.
\eea
Applying  Lemma \ref{lem:trace:fluxcontrolledbyMorawetz} to control the term $ \int_{\II_+(\tau_1,\tau_2)}|F|^2$ on the RHS by  
\beaa
\int_{\II_+(\tau_1,\tau_2)}|F|^2\les \int_{\MM{_{r\geq {10.5m}}}(\tau_1,\tau_2)}|\pr^{\leq 1}F|^2 \les \widehat{\mathcal{N}}^{(1)}_{r\geq 10m}[\psi, F](\tau_1, \tau_2),
\eeaa
where we have also used the definitions \eqref{eq:defintionwidehatmathcalNfpsinormRHS} and \eqref{def:normshighorder:FMdeandNNnorms} in the last step,
we infer}
\bea\lab{eq:highorderEMF:fluxnullinf:final}
\F^{(1)}_{\II_+}[\psi](\tau_1, \tau_2)&\les& \EMF[\pr_{\tau}\psi](\tau_1,\tau_2)+\EMF[\psi](\tau_1,\tau_2)+\widehat{\mathcal{N}}^{(1)}_{r\geq 10m}[\psi, F](\tau_1, \tau_2)\nn\\
&& +\sqrt{\EMF{[\psi]}(\tau_1,\tau_2)}\sqrt{\EMF^{(1)}{[\psi]}(\tau_1,\tau_2)}.
\eea}

{\noindent{\bf Step 4.} \textit{End of the proof of  \eqref{eq:higherorderenergymorawetzforlargerconditionalonsprtau}}}. Combining the estimates \eqref{eq:EM:AsympKerr:highMora:v1:mid}, \eqref{eq:highorderEMF:energylarger} and \eqref{eq:highorderEMF:fluxnullinf:final}, {and in view of the following bound
\beaa
\int_{\MM_{r\geq 10m}(\tau_1,\tau_2)}|\pr^{\leq 1} F|^2 \les \widehat{\mathcal{N}}^{(1)}_{r\geq 10m}[\psi, F](\tau_1, \tau_2)
\eeaa
which follows from the definition $\widehat{\mathcal{N}}^{(1)}_{r\geq 10m}[\psi, F](\tau_1, \tau_2)$ in Section \ref{subsect:norms},} we get
the desired estimate \eqref{eq:higherorderenergymorawetzforlargerconditionalonsprtau}.

{\noindent{\bf Step 5.} \textit{Proof of  \eqref{eq:higherorderenergymorawetzconditionalonsprtausprphis}: Morawetz part and flux part on $\AA$}. In view of \eqref{eq:controloflinearizedinversemetriccoefficients}, \eqref{eq:inverse:hypercoord} and Lemma \ref{lemma:computationofthederiveativeofsrqtg}, we have the following non-sharp identity in $r\leq 12m$ 
\beaa
|q|^2\square_\g &=& \De\pr_r^2+\mathring{\ga}^{ab}\pr_{x^a}\pr_{x^b}+O(1)(\pr_\tau, \pr_{\tphi}){(\pr_\tau, \pr_{\tphi},\pr_r)}+O(1)\pr\psi+O(\ep)\pr^2
\eeaa
which yields
\bea\lab{eq:elliptic:waveform:outsidehorizon}
\Big(\De\pr_r^2+\mathring{\ga}^{ab}\pr_{x^a}\pr_{x^b}\Big)\psi &=& |q|^2F+O(1)(\pr_\tau, \pr_{\tphi}){(\pr_\tau, \pr_{\tphi},\pr_r)}\psi+O(1)\pr\psi+O(\ep)\pr^2\psi.
\eea} 

Next, we consider a cut-off function $\chi_4$ such hat 
\beaa
\chi_4(r)=1 \,\,\text{for }r_+(1+\dred)\leq r\leq 11m, \qquad  \chi_4(r)=0 \,\,\text{for }r\geq 12m\,\,\textrm{and }r\leq r_+(1+\dred/2).
\eeaa
Then, we multiply both sides of \eqref{eq:elliptic:waveform:outsidehorizon} by $\chi_4^2\pr^2_r\psi$ and integrate over $\MM(\tau_1, \tau_2)$ which yields, after integration by parts on the RHS, 
\bea
&&\int_{\MM(\tau_1, \tau_2)}\chi_4^2\Big(\De\pr_r^2+\mathring{\ga}^{ab}\pr_{x^a}\pr_{x^b}\Big)\psi\pr^2_r\psi\nn\\
&\les& \left(\int_{\MM(\tau_1, \tau_2)}|F|^2 {+\M[\pr_\tau\psi](\tau_1, \tau_2)}+{\M[\chi_4\pr_{\tphi}\psi]}(\tau_1, \tau_2)\right)^{\frac{1}{2}}\Big(\M[\chi_4\pr_r\psi](\tau_1, \tau_2)\Big)^{\frac{1}{2}}\nn\\
&&+\sqrt{\EMF[\psi](\tau_1, \tau_2)}\sqrt{\EMF^{(1)}[\psi](\tau_1, \tau_2)}+\ep\M^{(1)}[\psi](\tau_1, \tau_2)\nn\\
&&{+\M[\pr_\tau\psi](\tau_1, \tau_2)+\M_{r_+(1+\dred/2), 12m}[\pr_{\tphi}\psi](\tau_1, \tau_2).} 
\eea
Integrating by parts the LHS, we infer 
\beaa
&& \M[\chi_4\pr_r\psi](\tau_1, \tau_2)\\
&\les& \left(\int_{\MM(\tau_1, \tau_2)}|F|^2 {+\M[\pr_\tau\psi](\tau_1, \tau_2)}+{\M[\chi_4\pr_{\tphi}\psi]}(\tau_1, \tau_2)\right)^{\frac{1}{2}}\Big(\M[\chi_4\pr_r\psi](\tau_1, \tau_2)\Big)^{\frac{1}{2}}\\
&&+\sqrt{\EMF[\psi](\tau_1, \tau_2)}\sqrt{\EMF^{(1)}[\psi](\tau_1, \tau_2)} +\ep\M^{(1)}[\psi](\tau_1, \tau_2)\nn\\
&& +\M[\pr_\tau\psi](\tau_1, \tau_2) {+\M_{r_+(1+\dred/2), 12m}[\pr_{\tphi}\psi](\tau_1, \tau_2).}
\eeaa
and hence
\begin{align}
\lab{eq:controlofchi4prrpsi:Morawetz:recoveringderivatives}
& \M{[\chi_4}\pr_r\psi](\tau_1, \tau_2)\nn\\
\les{}& \int_{\MM(\tau_1, \tau_2)}|F|^2{+\M[\pr_\tau\psi](\tau_1, \tau_2)}+\M_{r_+(1+{\dred/2}), 11m}[\pr_{\tphi}\psi](\tau_1, \tau_2)\nn\\
{}&+\M^{(1)}_{r\geq 11m}[\psi](\tau_1, \tau_2)+\sqrt{\EMF[\psi](\tau_1, \tau_2)}\sqrt{\EMF^{(1)}[\psi](\tau_1, \tau_2)}+\ep\M^{(1)}[\psi](\tau_1, \tau_2).
\end{align}

Next, we use the following consequence of \eqref{eq:elliptic:waveform:outsidehorizon}
\beaa
\mathring{\ga}^{ab}\pr_{x^a}\pr_{x^b}\psi &=& |q|^2F+O(1)(\pr_\tau, \pr_{\tphi}, \pr_r)\pr\psi+O(1)\pr\psi+O(\ep)\pr^2\psi.
\eeaa
Taking the square on both sides and integrating on $\MM(r_+(1+\dred), 11m)\setminus\Mtrap$, we infer
\beaa
&&\int_{\MM(r_+(1+\dred), 11m)\setminus\Mtrap}|\mathring{\ga}^{ab}\pr_{x^a}\pr_{x^b}\psi|^2\\
&\les& \int_{{\Mntrap_{r_+(1+\dred), 11m}(\tau_1, \tau_2)}}|F|^2+\M[\psi](\tau_1, \tau_2)+\M[\pr_\tau\psi](\tau_1, \tau_2)\\
&&+\M_{r_+(1+\dred), 11m}[\pr_{\tphi}\psi](\tau_1, \tau_2) {+\M_{r_+(1+\dred), 11m}[\pr_r\psi](\tau_1, \tau_2)}\\
&&{+\ep\M_{r_+(1+\dred), 11m}^{(1)}[\psi](\tau_1, \tau_2)}+\sqrt{\EMF[\psi](\tau_1, \tau_2)}\sqrt{\EMF^{(1)}[\psi](\tau_1, \tau_2)}.
\eeaa
{Integrating the LHS by parts,} we infer, for $\ep>0$ small enough, 
\bea
&& \M^{(1)}_{r_+(1+\dred), 11m}[\psi](\tau_1, \tau_2)\nn\\
&\les& {\int_{\Mntrap_{r_+(1+\dred), 11m}(\tau_1, \tau_2)}|F|^2+\M[\pr_\tau\psi](\tau_1, \tau_2)+\M_{r_+(1+\dred), 11m}[\pr_{\tphi}\psi](\tau_1, \tau_2)} \nn\\
&&{+\M[\psi](\tau_1, \tau_2)+\M_{r_+(1+\dred), 11m}[\pr_r\psi](\tau_1, \tau_2)} \nn\\
&&{+\sqrt{\EMF[\psi](\tau_1, \tau_2)}\sqrt{\EMF^{(1)}[\psi](\tau_1, \tau_2)}}\nn\\
&\les& \int_{{\MM}(\tau_1, \tau_2)}|F|^2+\M[\psi](\tau_1, \tau_2)+\M[\pr_\tau\psi](\tau_1, \tau_2)+\M_{r_+(1+{\dred/2}), 11m}[\pr_{\tphi}\psi](\tau_1, \tau_2)\nn\\
&&+\M^{(1)}_{r\geq 11m}[\psi](\tau_1, \tau_2)+\sqrt{\EMF[\psi](\tau_1, \tau_2)}\sqrt{\EMF^{(1)}[\psi](\tau_1, \tau_2)}\nn\\
&&{+\ep\M_{r\leq r_+(1+\dred)}^{(1)}[\psi](\tau_1, \tau_2)},
\eea
{where in the last step we have used the above control of $\M_{r_+(1+\dred), 11m}[\pr_r\psi](\tau_1, \tau_2)$ in \eqref{eq:controlofchi4prrpsi:Morawetz:recoveringderivatives}.}
Together with  the red-shift estimate {of} Lemma \ref{lemma:redshiftestimates} with $s=1$, we infer{, for $\ep$ small enough,}
\bea
&&\sup_{\tau\in[\tau_1,\tau_2]}\E^{(1)}_{r\leq r_+ (1+\dred)}[\psi](\tau)
+\F_{\AA(\tau_1,\tau_2)}^{(1)}[\psi](\tau_1, \tau_2)+ \M^{(1)}_{r_+(1-\dhor), 11m}[\psi](\tau_1, \tau_2)\nn\\
&\les& \E^{(1)}[\psi](\tau_1)+\M[\psi](\tau_1, \tau_2)+\M[\pr_\tau\psi](\tau_1, \tau_2)+\M_{r_+(1+{\dred/2}), 11m}[\pr_{\tphi}\psi](\tau_1, \tau_2)\nn\\
&&+\sqrt{\EMF[\psi](\tau_1, \tau_2)}\sqrt{\EMF^{(1)}[\psi](\tau_1, \tau_2)}\nn\\
&&+\int_{\MM(\tau_1, \tau_2)}|\pr^{\leq 1}F|^2+\M^{(1)}_{r\geq 11m}[\psi](\tau_1, \tau_2).
\eea
Plugging the control of \eqref{eq:EM:AsympKerr:highMora:v1:mid} for $\M^{(1)}_{r\geq 11m}[\psi](\tau_1, \tau_2)$, we deduce 
\bea\lab{eq:highorderMora:improveMora:v3}
\nn&&\sup_{\tau\in[\tau_1,\tau_2]}\E^{(1)}_{r\leq r_+ (1+\dred)}[\psi](\tau)
+\F_{\AA(\tau_1,\tau_2)}^{(1)}[\psi](\tau_1, \tau_2)+ \M^{(1)}[\psi](\tau_1, \tau_2)\\
\nn&\les& \E^{(1)}[\psi](\tau_1)+\M[\psi](\tau_1, \tau_2)+\M[\pr_\tau\psi](\tau_1, \tau_2)+\M_{r_+(1+{\dred/2}), 11m}[\pr_{\tphi}\psi](\tau_1, \tau_2)\\
&&+\sqrt{\EMF[\psi](\tau_1, \tau_2)}\sqrt{\EMF^{(1)}[\psi](\tau_1, \tau_2)} + \int_{\MM(\tau_1, \tau_2)}|\pr^{\leq 1}F|^2.
\eea

\noindent{\bf Step 6.} \textit{Proof of  \eqref{eq:higherorderenergymorawetzconditionalonsprtausprphis}: energy part}. We square both sides of \eqref{eq:elliptic:waveform:outsidehorizon} and multiply both squares by $\chi_4^2$, where, as in Step 5, the smooth cut-off function $\chi_4$ satisfies 
\beaa
\chi_4(r)=1 \,\,\text{for }r_+(1+\dred)\leq r\leq 11m, \qquad  \chi_4(r)=0 \,\,\text{for }r\geq 12m\,\,\textrm{and }r\leq r_+(1+\dred/2).
\eeaa
Integrating on $\Si(\tau)$, we obtain, for any $\tau_1\leq\tau\leq \tau_2$, 
\beaa
\int_{\Si(\tau)}\chi_4^2\left|\Big(\De\pr_r^2+\mathring{\ga}^{ab}\pr_{x^a}\pr_{x^b}\Big)\psi\right|^2&\les& \int_{\Si(\tau)}\chi_4^2|F|^2+\E[\pr_\tau\psi](\tau)+\E_{r_+(1+{\dred/2}), 11m}[\pr_{\tphi}\psi](\tau)\\
&&{+\E_{r\geq 11m}[\pr^{\leq 1}\psi](\tau)}+\E[\psi](\tau)+\ep^2\E^{(1)}[\psi](\tau)\\
&\les&{ \int_{\Si(\tau)}|F|^2}+\E[\pr_\tau\psi](\tau)+\E_{r_+(1+{\dred/2}), 11m}[\pr_{\tphi}\psi](\tau)\\
&&{+\E_{r\geq 11m}[\pr^{\leq 1}\psi](\tau)}+\E[\psi](\tau)+\ep^2\E^{(1)}[\psi](\tau).
\eeaa
Integrating by parts the LHS, we infer
\beaa
\int_{\Si(\tau)}\chi_4^2|(\pr_r, \nab)^2\psi|^2 &\les& { \int_{\Si(\tau)}|F|^2}+\E[\pr_\tau\psi](\tau)+\E_{r_+(1+{\dred/2}), 11m}[\pr_{\tphi}\psi](\tau)+\E[\psi](\tau)\\
&&{+\E_{r\geq 11m}[\pr^{\leq 1}\psi](\tau)}+\ep^2\E^{(1)}[\psi](\tau)+\sqrt{\E[\psi](\tau)}\sqrt{\E^{(1)}[\psi](\tau)},
\eeaa
and hence
\begin{align}
\E_{r_+(1+\dred), 11m}^{(1)}[\psi](\tau) \les& {\int_{\Si(\tau)}|F|^2}+\E[\pr_\tau\psi](\tau)+\E_{r_+(1+{\dred/2}), 11m}[\pr_{\tphi}\psi](\tau)+\E[\psi](\tau)\nn\\
&{+\E_{r\geq 11m}[\pr^{\leq 1}\psi](\tau)}+\ep^2\E^{(1)}[\psi](\tau)+\sqrt{\E[\psi](\tau)}\sqrt{\E^{(1)}[\psi](\tau)}.
\end{align}
Taking the supremum in $\tau\in[\tau_1, \tau_2]$ and using also \eqref{eq:highorderEMF:energylarger} and \eqref{eq:highorderMora:improveMora:v3}, {and applying a trace estimate to control $\int_{\Si(\tau)}|F|^2$ by $\int_{\MM(\tau_1, \tau_2)}|\pr^{\leq 1}F|^2$,}
 we infer, for $\ep$ small enough, 
 \bea\lab{eq:highorder:energyneartrap:total}
\nn&&\EM^{(1)}[\psi](\tau_1, \tau_2)
+\F_{\AA(\tau_1,\tau_2)}^{(1)}[\psi](\tau_1, \tau_2)\\
\nn&\les& \E^{(1)}[\psi](\tau_1)+\EM[\psi](\tau_1, \tau_2)+\EM[\pr_\tau\psi](\tau_1, \tau_2)+\EM_{r_+(1+{\dred/2}), 11m}[\pr_{\tphi}\psi](\tau_1, \tau_2)\\
&&+\sqrt{\EMF[\psi](\tau_1, \tau_2)}\sqrt{\EMF^{(1)}[\psi](\tau_1, \tau_2)}\nn\\
&&+\widehat{\mathcal{N}}^{(1)}_{r\geq 10m}[\psi, F](\tau_1, \tau_2)+ \int_{\MM(\tau_1, \tau_2)}|\pr^{\leq 1}F|^2.
\eea

{\noindent{\bf Step 7.} \textit{End of the proof of  \eqref{eq:higherorderenergymorawetzconditionalonsprtausprphis}}}. 
By adding {\eqref{eq:highorderEMF:fluxnullinf:final}} and \eqref{eq:highorder:energyneartrap:total} together,  we prove the desired estimate \eqref{eq:higherorderenergymorawetzconditionalonsprtausprphis}.

{\noindent{\bf Step 8.} \textit{Proof of  \eqref{eq:imporvedhigherorderenergymorawetzconditionalonlowerrweigth}}. As the  $s=0$ case has been shown in Lemma \ref{lemma:exteriorMorawetzestimates}, we focus on the case $s=1$. We rely on the following immediate consequence of  \eqref{eq:commutewave:ellipticesti:high}, \eqref{eq:inverse:hypercoord} and \eqref{eq:assymptiticpropmetricKerrintaurxacoord:1}
\bea
\pr^2_r\psi = F+O(1)\nab\pr\psi+O(1)\pr_r\pr_\tau\psi
+O(\ep)\pr_r^2\psi+O(r^{-1})\pr\pr^{\leq 1}\psi, \quad \text{for }\,\, r\geq 10m.
\eea
We square both sides of the above identity, multiply both squares by $r^{-1-\de}$ and integrate over $\MM_{r\geq 10m}(\tau_1, \tau_2)$ which yields
\beaa
\int_{\MM_{r\geq 10m}(\tau_1, \tau_2)}\frac{|\pr_r^2\psi|^2}{r^{1+\de}} \les \int_{\MM(\tau_1, \tau_2)}\frac{|F|^2}{r^{1+\de}}+\M^{(1)}[\psi](\tau_1, \tau_2)+\M_\de[\pr_\tau\psi](\tau_1, \tau_2)+\ep^2\M^{(1)}_\de[\psi](\tau_1, \tau_2).
\eeaa
Since 
\beaa
\M^{(1)}_\de[\psi](\tau_1, \tau_2) \les \M^{(1)}[\psi](\tau_1, \tau_2)+\int_{\MM_{r\geq 10m}(\tau_1, \tau_2)}\frac{|\pr_r^2\psi|^2}{r^{1+\de}}+\M_\de[\pr_\tau\psi](\tau_1, \tau_2),
\eeaa
we infer, for $\ep$ small enough, 
\bea\lab{eq:highorderMora:improveMora:byTpsi}
\M^{(1)}_\de[\psi](\tau_1, \tau_2) \les \int_{\MM(\tau_1, \tau_2)}|F|^2+\M^{(1)}[\psi](\tau_1, \tau_2)+\M_\de[\pr_\tau\psi](\tau_1, \tau_2).
\eea}

{In view of \eqref{eq:highorderMora:improveMora:byTpsi}, it remains to control $\M_{\de}  [\pr_{\tau}\psi](\tau_1,\tau_2)$.} By Lemma \ref{lem:commutatorwithwave:firstorderderis}, $\pr_{\tau}\psi$ satisfies the following wave equation
\beaa
\square_{\g}\pr_{\tau}\psi= \pr_{\tau} F + [\square_{\g}, \pr_{\tau}]\psi,
\eeaa
with
\bea\lab{eq:structurecommutatorwaveequationwithprtauusedtocontrolMdeprtaupsi}
\, [ \pr_{\tau}, \square_{\g}]\psi&=& \pr_{\tau}(\gcheck^{\mu\nu})\pr_{\mu}\pr_{\nu}\psi+
\dk^{\leq 2}\Ga_g\c\dk\psi.
\eea
Applying Lemma \ref{lemma:exteriorMorawetzestimates} to this wave equation of $\pr_{\tau}\psi$, we obtain
\beaa
&&\M_{{\de}}  [\pr_{\tau}\psi](\tau_1,\tau_2)\nn\\
&\les&\EMF[\pr_{\tau}\psi](\tau_1,\tau_2) + \int_{\MM(\tau_1, \tau_2)} r^{1+\de} |\pr_{\tau} F|^2 
+ \bigg|\int_{\MM_{r\geq 11m}(\tau_1, \tau_2)} [\pr_{\tau},\square_{\g}]\psi (1+O(r^{-\de}))\pr_{\tau}\pr_{\tau}\psi \bigg|\nn\\
&&
+
\int_{\MM_{r\geq 11m}(\tau_1, \tau_2)}|[\pr_{\tau},\square_{\g}]\psi|
\Big|\big(\pr_r, r^{-1}\pr_{x^a}, r^{-1}\big)\pr_{\tau}\psi\Big|.
\eeaa
{Now, in view of \eqref{eq:structurecommutatorwaveequationwithprtauusedtocontrolMdeprtaupsi}, the last two terms on the RHS are controlled as follows 
\beaa
&&  \bigg|\int_{\MM_{r\geq 11m}(\tau_1, \tau_2)} [\pr_{\tau},\square_{\g}]\psi (1+O(r^{-\de}))\pr_{\tau}\pr_{\tau}\psi \bigg| \\
&+&
\int_{\MM_{r\geq 11m}(\tau_1, \tau_2)}|[\pr_{\tau},\square_{\g}]\psi|
\Big|\big(\pr_r, r^{-1}\pr_{x^a}, r^{-1}\big)\pr_{\tau}\psi\Big|\\
&\les&  \bigg|\int_{\MM_{r\geq 11m}(\tau_1, \tau_2)}\pr_{\tau}(\gcheck^{\mu\nu})\pr_{\mu}\pr_{\nu}\psi(1+O(r^{-\de}))\pr_{\tau}\pr_{\tau}\psi \bigg| \\
&+&
\int_{\MM(\tau_1, \tau_2)}|\pr_{\tau}(\gcheck^{\mu\nu})\pr_{\mu}\pr_{\nu}\psi|
\Big|\big(\pr_r, r^{-1}\pr_{x^a}, r^{-1}\big)\pr_{\tau}\psi\Big|+\ep\int_{\MM(\tau_1, \tau_2)}\tau^{-\frac{1}{2}-\dec}r^{-2}|\dk\psi|
\Big|\big(\pr, r^{-1}\big)\pr_{\tau}\psi\Big|\\
&\les&  \ep \EM^{(1)}[\psi](\tau_1,\tau_2)+\ep\sqrt{\sup_{\tau\in[\tau_1, \tau_2]}\E[\psi](\tau)}\sqrt{\M^{(1)}[\psi](\tau_1, \tau_2)}\\
&\les& \ep \EM^{(1)}[\psi](\tau_1,\tau_2),
\eeaa
where we have in particular used Lemma \ref{lemma:basiclemmaforcontrolNLterms:bis} with $M^{\a\b}=\pr_\tau(\gcheck^{\a\b})$ which satisfies the needed assumptions in view of \eqref{eq:controloflinearizedinversemetriccoefficients}. We deduce} 
\beaa
\M_{{\de}}  [\pr_{\tau}\psi](\tau_1,\tau_2)
&\les&\EMF[\pr_{\tau}\psi](\tau_1,\tau_2) + \int_{\MM(\tau_1, \tau_2)} r^{1+\de} |\pr_{\tau} F|^2 
+\ep {\EM^{(1)}[\psi](\tau_1,\tau_2)},
\eeaa
{which together with \eqref{eq:highorderMora:improveMora:byTpsi} concludes the proof of the desired estimate \eqref{eq:imporvedhigherorderenergymorawetzconditionalonlowerrweigth}. This concludes the proof of Lemma \ref{lemma:higherorderenergyMorawetzestimates}.}
\end{proof}

%%%%%%%%%%%%%%%%%%%%%%%%%%%

\section{Statement and proof of the main theorem}
\lab{sect:maintheoremandproof}

%%%%%%%%%%%%%%%%%%%%%%%%%%%

%%%%%%%%%%%%%%%%%%%%%%%

In this section, we state a precise version of our main theorem on energy-Morawetz estimates in perturbations of Kerr and provide its proof.

%%%%%%%%%%%%%%%%%%%%%%%

\subsection{Statement of the main theorem}

%%%%%%%%%%%%%%%%%%%%%%%

We now provide a precise version of our main theorem on the derivation of energy-Morawetz estimates for solutions to the inhomogeneous scalar wave equation on $(\MM, \g)$, i.e.
\begin{equation}
\label{eq:scalarwave}
\Box_{\g} \psi = F, \qquad \MM,
\end{equation}
where $\g$ is a perturbation of a Kerr metric $\gam$ with $|a|<m$. 

\begin{theorem}[Energy-Morawetz for scalar waves, precise version] 
\label{thm:main}
Let $\g$ satisfy the assumptions of Section \ref{subsubsect:assumps:perturbedmetric}. There exists a suitably small constant $\ep'>0$ such that for any $\ep\leq \ep'$, we have for solutions to the inhomogeneous wave equation \eqref{eq:scalarwave} the following energy-Morawetz-flux estimates, for any $1\leq\tau_1<\tau_2 <+\infty$ and any $0<\de\leq 1$,
\bea
\label{MainEnerMora:psi}
\EMF^{(1)}_\de[{\psi}](\tau_1,\tau_2) \lesssim \E^{(1)}[\psi](\tau_1)+\mathcal{N}^{(1)}_\de[\psi, F](\tau_1, \tau_2),
\eea
where the norms $\EMF^{(1)}_\de[{\psi}](\tau_1,\tau_2)$, $\E^{(1)}[\psi](\tau_1)$ and $\mathcal{N}^{(1)}_\de[\psi, F](\tau_1, \tau_2)$ have been introduced in Section \ref{subsect:norms}, and where the implicit constant in $\lesssim$ only depends on $a$, $m$, $\de$ and $\dec$ (with $\dec$ appearing in \eqref{eq:decaypropertiesofGabGag}).
\end{theorem}

%%%%%%%%%%%%%%%%%%%%%%%%%

\subsection{Global energy-Morawetz estimates}

%%%%%%%%%%%%%%%%%%%%%%%%%

In order to prove our main Theorem \ref{thm:main}, i.e., the derivation of energy-Morawetz estimates for $\tau$ in $(\tau_1, \tau_2)$, we first state in this section global energy-Morawetz estimates, i.e., energy-Morawetz estimates for $\tau$ in $\mathbb{R}$. 
\begin{theorem}\lab{th:main:intermediary}
Let $\g$ satisfy the assumptions of Section \ref{subsubsect:assumps:perturbedmetric}, and let $\psi$ be a solution to the inhomogeneous wave equation \eqref{eq:scalarwave} with RHS $F$. Assume that $\psi$ can be smoothly extended by $0$ for $\tau\leq 1$. Also, assume that the metric $\g$ coincides with $\gam$ for $\tau\leq 1$ and for $\tau\geq \tau_*$ with $\tau_*$ arbitrarily large. Finally, assume that $F$ is supported in $(1,\tau_*)$.
Then, we have 
\bea\lab{eq:th:main:intermediary1}
\EMF[\psi](\mathbb{R})+\EMF[\pr_\tau\psi](\mathbb{R}) \les \widetilde{\mathcal{N}}^{(1)}[ \psi, F](\mathbb{R})+\ep \EM^{(1)}[\psi](\mathbb{R})
\eea
and
\bea\lab{eq:th:main:intermediary2}
{\EMF_{r_+(1-\dhor), 11m}[\pr_{\tphi}\psi](\mathbb{R})} &\les& \widetilde{\mathcal{N}}^{(1)}[\psi, F](\mathbb{R})+\ep \EM^{(1)}[\psi](\mathbb{R})+{\M_{11m, 12m}[\pr\psi](\mathbb{R})},
\eea
where $\widetilde{\mathcal{N}}[\psi, F](\mathbb{R})$ is defined in \eqref{eq:definitionwidetildemathcalNpsijnormRHS}.
\end{theorem}

The proof of Theorem \ref{th:main:intermediary} requires microlocal estimates in $\Mtrap$ and is postponed to Section \ref{sec:proofofth:main:intermediary}. Based on Theorem \ref{th:main:intermediary}, we are now ready to prove our main theorem, i.e., Theorem \ref{thm:main}.

%%%%%%%%%%%%%%%%%%%%%%%%%%%

\subsection{Proof of {Theorem \ref{thm:main}}}
\lab{subsect:proofofThm4.1}

%%%%%%%%%%%%%%%%%%%%%%%%%%%

Let $\g$ satisfy the assumptions of Section \ref{subsubsect:assumps:perturbedmetric},  let $\tau_1$ and $\tau_2$ be such that $1\leq\tau_1<\tau_2 <+\infty$, and let $\psi$ be a solution to the inhomogeneous scalar wave \eqref{eq:scalarwave} with RHS $F$. If $\tau_2\leq\tau_1+4$, then \eqref{MainEnerMora:psi} follows immediately from {the local energy estimates of Lemma \ref{lemma:localenergyestimate} together with  \eqref{eq:imporvedhigherorderenergymorawetzconditionalonlowerrweigth}}, so we may assume from now on that $\tau_2\geq \tau_1+4$. {Similarly, if $1\leq \tau_1<2$, we may use the local energy estimates of Lemma \ref{lemma:localenergyestimate} together \eqref{eq:imporvedhigherorderenergymorawetzconditionalonlowerrweigth} on $(\tau_1, 2)$ to reduce to the case $\tau_1\geq 2$. Thus, we assume from now on that $2\leq \tau_1<\tau_1+4\leq \tau_2<+\infty$.} We proceed in the following steps.

{\noindent{\bf Step 0.} We will first prove Theorem \ref{thm:main} under the assumption that $\psi$ is compactly supported in $\Si(\tau_1)$, see Steps 1--8, and we will then extend it to the general case by density, see Step 9. This assumption of compact support will be used in Step 7 in conjunction with the following lemma.}

{\begin{lemma}\lab{lemma:howtoexploitthecompactsupportofpsionSigramtau}
Let $\tau_0\in\mathbb{R}$ and $q>0$, and assume that $\psi$ vanishes on $\Si(\tau_0)\cap\II_+$. Then, we have 
\bea
{\liminf_{\tauu\to +\infty}}\int_{\tau_0-q}^{\tau_0}\int_{\mathbb{S}^2}|\pr^{\leq 2}\psi|^2({\tauu}, \tau, \om)r^2d\mathring{\ga}d\tau \les_q \EF^{(1)}[\psi](\tau_0-q,\tau_0).
\eea
\end{lemma}} 

{\begin{proof}
In view of Lemma \ref{lem:nullnfFluxBdedByEnergy}, we have
\beaa
{\liminf_{\tauu\to +\infty}}\int_{\tau_0-q}^{\tau_0}\int_{\mathbb{S}^2}r^{-1}|\dk^{\leq 1}(\pr^{\leq 1}\psi)|^2({\tauu}, \tau, \om)r^2d\mathring{\ga}d\tau \les_q \sup_{\tau\in[\tau_0-q, \tau_0]}\E^{(1)}[\psi](\tau).
\eeaa
We infer
\beaa
{\liminf_{\tauu\to +\infty}}\int_{\tau_0-q}^{\tau_0}\int_{\mathbb{S}^2}r|\pr_r(\pr^{\leq 1}\psi)|^2({\tauu}, \tau, \om)r^2d\mathring{\ga}d\tau \les_q \sup_{\tau\in[\tau_0-q, \tau_0]}\E^{(1)}[\psi](\tau),
\eeaa
which together with the definition of $\F_\II^{(1)}[\psi](\tau_0-q,\tau_0)$ implies
\beaa
{\liminf_{\tauu\to +\infty}}\int_{\tau_0-q}^{\tau_0}\int_{\mathbb{S}^2}|(\pr_\tau, \pr_r, \nab)\pr^{\leq 1}\psi|^2({\tauu}, \tau, \om)r^2d\mathring{\ga}d\tau \les_q  \EF^{(1)}[\psi](\tau_0-q,\tau_0).
\eeaa
It thus remains to control $\psi$. Since $\psi$ vanishes on $\Si(\tau_0)\cap\II_+$, we have
\beaa
\psi(\tauu=+\infty, \tau, \om)= -\frac{1}{2}\int_\tau^{\tau_0}{\pr_\tau^{\II_+}}\psi(\tauu=+\infty, \tau', \om)d\tau', \quad \tau_0-q\leq\tau\leq\tau_0,
\eeaa
and hence
\beaa
\int_{\tau_0-q}^{\tau_0}\int_{\mathbb{S}^2}|\psi|^2(\tauu=+\infty, \tau, \om)r^2d\mathring{\ga}d\tau  &\les_q& \int_{\tau_0-q}^{\tau_0}\int_{\mathbb{S}^2}|{\pr_\tau^{\II_+}}\psi|^2(\tauu=+\infty, \tau, \om)r^2d\mathring{\ga}d\tau \\
&\les_q& \F_\II[\psi](\tau_0-q,\tau_0)
\eeaa
which concludes the proof of Lemma \ref{lemma:howtoexploitthecompactsupportofpsionSigramtau}. 
\end{proof}}

\noindent{\bf Step 1.} Our goal is to deduce the proof of Theorem \ref{thm:main} as a consequence of Theorem \ref{th:main:intermediary}. To this end, we need to introduce an auxiliary wave equation defined in $\MM$. We start with the construction of the corresponding RHS. We define the scalar function $\widetilde{F}$ 
\bea\label{def:tildef}
\widetilde{F}=\chi_{\tau_1, \tau_2}F
\eea
where $\chi_{\tau_1, \tau_2}=\chi_{\tau_1, \tau_2}(\tau)$ is a smooth cut-off function satisfying 
\bea\lab{eq:propertieschitoextendmetricg}
\chi_{\tau_1, \tau_2}(\tau)=0\,\,\,\textrm{on}\,\,\, \mathbb{R}\setminus(\tau_1, \tau_2), \quad\chi_{\tau_1, \tau_2}(\tau)=1\,\,\,\textrm{on}\,\,\, [\tau_1+1, \tau_2-1], \quad \|\chi_{\tau_1, \tau_2}\|_{W^{2,+\infty}(\mathbb{R})}\les 1.
\eea

\noindent{\bf Step 2.} Next, we introduce a new Lorentzian metric $\g_{\chi_{\tau_1, \tau_2}}$ defined as follows 
\bea
\g_{\chi_{\tau_1, \tau_2}}^{\a\b}=\chi_{\tau_1, \tau_2} \g^{\a\b}+(1-\chi_{\tau_1, \tau_2})\g_{a,m}^{\a\b},
\eea
where $\chi_{\tau_1, \tau_2}$ is defined in \eqref{eq:propertieschitoextendmetricg}. In particular, we have
\beaa
\widecheck{\g}_{\chi_{\tau_1, \tau_2}}^{\a\b}=\chi_{\tau_1, \tau_2} \g^{\a\b}+(1-\chi_{\tau_1, \tau_2})\gam^{\a\b} -\gam^{\a\b}=\chi_{\tau_1, \tau_2}\left(\g^{\a\b}-\gam^{\a\b}\right)=\chi_{\tau_1, \tau_2}\widecheck{\g}^{\a\b}
\eeaa
and hence
\beaa
|\dk^{\leq 2}\widecheck{\g}_{\chi_{\tau_1, \tau_2}}^{\a\b}|\les |\dk^{\leq 2}\chi_{\tau_1, \tau_2}||\dk^{\leq 2}\widecheck{\g}^{\a\b}|\les \|\chi_{\tau_1, \tau_2}\|_{W^{2,+\infty}}|\dk^{\leq 2}\widecheck{\g}^{\a\b}| \les |\dk^{\leq 2}\widecheck{\g}^{\a\b}|,
\eeaa
where we used the fact that $\dk\chi_{\tau_1, \tau_2}=(\pr_\tau, r\pr_r, \pr_{x^a})\chi_{\tau_1, \tau_2}=\pr_\tau\chi_{\tau_1, \tau_2}=\chi_{\tau_1, \tau_2}'$ since $\chi_{\tau_1, \tau_2}=\chi_{\tau_1, \tau_2}(\tau)$ and in view of the definition of the weighted derivatives $\dk$. Since $\g$ satisfies the assumptions of Section \ref{subsubsect:assumps:perturbedmetric}, and in view of the properties \eqref{eq:propertieschitoextendmetricg} of $\chi_{\tau_1, \tau_2}$, we deduce that $\g_{\chi_{\tau_1, \tau_2}}$ also satisfies the assumptions of Section \ref{subsubsect:assumps:perturbedmetric}, and in addition coincides with $\gam$ in $\mathbb{R}\setminus(\tau_1, \tau_2)$. 

\noindent{\bf Step 3.}  Next, we introduce the solution $\widetilde{\psi}$ to the following auxiliary scalar wave equation
\bea\lab{eq:waveequationdefiningfirstextensionwidetildepsi}
\bsplit
\square_{\g_{\chi_{\tau_1, \tau_2}}}\widetilde{\psi} &= \widetilde{F}\quad\textrm{on}\quad\MM(\tau_1, +\infty),\\
\widetilde{\psi} &= \psi, \quad N_{\Sigma(\tau_1+1)}\widetilde{\psi}=N_{\Sigma(\tau_1+1)}\psi\quad\textrm{on}\quad\Si(\tau_1+1),\\
\widetilde{\psi} &= \psi\quad\textrm{on}\quad(\AA_+\cup\II_+)\cap\{\tau_1\leq\tau\leq\tau_1+1\}.
\end{split}
\eea 
Then, we have in view of the second local energy estimate of Lemma \ref{lemma:localenergyestimate} with $s=1$ for \eqref{eq:waveequationdefiningfirstextensionwidetildepsi}
\beaa
\EF^{(1)}[\widetilde{\psi}](\tau_1, \tau_1+1)\les \EF^{(1)}[\psi](\tau_1+1)+\widehat{\mathcal{N}}^{(1)}[\widetilde{\psi}, \widetilde{F}](\tau_1, \tau_1+1).
\eeaa
Together with the local energy estimate of Lemma \ref{lemma:localenergyestimate} with $s=1$ for $\psi$, we infer
\bea\lab{eq:localenergyestimateforwidetildepsionSigma}
\EF^{(1)}[\widetilde{\psi}](\tau_1, \tau_1+1)\les \E^{(1)}[\psi](\tau_1)+\widehat{\mathcal{N}}^{(1)}[\psi, F](\tau_1, \tau_1+1)+\widehat{\mathcal{N}}^{(1)}[\widetilde{\psi}, \widetilde{F}](\tau_1, \tau_1+1).
\eea

Also, we introduce the solution $\psi_{aux}$ to the following auxiliary scalar wave equation
\bea\lab{eq:waveequationdefiningpsiaux}
\bsplit
\square_{\g_{\chi_{\tau_1, \tau_2}}}\psi_{aux} &= 0\quad\textrm{on}\quad\MM(\tau_1-1, \tau_1),\\
\psi_{aux}&=\widetilde{\psi}, \quad N_{\Sigma(\tau_1)}\psi_{aux}=N_{\Sigma(\tau_1)}\widetilde{\psi}\quad\textrm{on}\quad\Si(\tau_1),\\
\psi_{aux}&= \psi_\AA\quad\textrm{on}\quad\AA_+\cap\{\tau_1-1\leq\tau\leq\tau_1\},\\ \psi_{aux}&=\psi_\II\quad\textrm{on}\quad\II_+\cap\{\tau_1-1\leq\tau\leq\tau_1\},
\end{split}
\eea 
where $\psi_\AA$ and $\psi_\II$ are smooth extensions of $\psi$ respectively from $\AA\cap\{\tau\geq \tau_1\}$ to $\AA\cap\{\tau_1-1\leq\tau\leq\tau_1\}$ and from $\II\cap\{\tau\geq \tau_1\}$ to $\II\cap\{\tau_1-1\leq\tau\leq\tau_1\}$ satisfying
\bea\lab{eq:propertyextensionpsiAAofpsitotauleqtau1}
\F_\AA^{(1)}[\psi_\AA](\tau_1-1, \tau_1)\les \F_\AA^{(1)}[\psi](\tau_1, \tau_1+1), \quad {\F_\II^{(1)}[\psi_\II](\tau_1-1, \tau_1)\les \F_\II^{(1)}[\psi](\tau_1, \tau_1+1).}
\eea 
The second local energy estimate of Lemma \ref{lemma:localenergyestimate} with $s=1$ for  \eqref{eq:waveequationdefiningpsiaux} yields 
\beaa
\EF^{(1)}[\psi_{aux}](\tau_1-1, \tau_1) &\les& \E^{(1)}[\psi_{aux}](\tau_1)+\F^{(1)}[\psi_{aux}](\tau_1-1, \tau_1)\\
&\les& \E^{(1)}[\widetilde{\psi}](\tau_1)+\F^{(1)}_\AA[\psi_\AA](\tau_1-1, \tau_1)+\F^{(1)}_\II[\psi_\II](\tau_1-1, \tau_1)
\eeaa
which together with \eqref{eq:localenergyestimateforwidetildepsionSigma} and \eqref{eq:propertyextensionpsiAAofpsitotauleqtau1} implies
\beaa
\EF^{(1)}[\psi_{aux}](\tau_1-1, \tau_1) \les \EF^{(1)}[\psi](\tau_1, \tau_1+1)+\widehat{\mathcal{N}}^{(1)}[\psi, F](\tau_1, \tau_1+1)+\widehat{\mathcal{N}}^{(1)}[\widetilde{\psi}, \widetilde{F}](\tau_1, \tau_1+1).
\eeaa
Using again the {local} energy estimate of Lemma \ref{lemma:localenergyestimate} with $s=1$ for {$\psi$}, we deduce
\bea\lab{eq:localenergyestimateforpsiauxonSigma}
\EF^{(1)}[\psi_{aux}](\tau_1-1, \tau_1) \les \E^{(1)}[\psi](\tau_1)+\widehat{\mathcal{N}}^{(1)}[\psi, F](\tau_1, \tau_1+1)+\widehat{\mathcal{N}}^{(1)}[\widetilde{\psi}, \widetilde{F}](\tau_1, \tau_1+1).
\eea

\noindent{\bf Step 4.} Next, we define
\bea
\label{def:tildef0}
\widetilde{F}_{(0)} = \left\{\begin{array}{l}
\square_{\g_{\chi_{\tau_1, \tau_2}}}(\chi_{\tau_1}\psi_{aux}) \quad\textrm{on}\quad \MM(\tau_1-1, \tau_1),\\[2mm]
0\quad\textrm{on}\quad \MM\setminus\MM(\tau_1-1, \tau_1),
\end{array}\right.
\eea
where the smooth cut-off $\chi_{\tau_1}=\chi_{\tau_1}(\tau)$ is such that $\chi_{\tau_1}=1$ for $\tau\geq \tau_1$ and 
$\chi_{\tau_1}=0$ for $\tau\leq \tau_1-1$. In particular, \eqref{eq:waveequationdefiningpsiaux} and \eqref{def:tildef0} imply that, for all $\tau\in\mathbb{R}$, 
\beaa
\widetilde{F}_{(0)} &=& \square_{\g_{\chi_{\tau_1, \tau_2}}}(\chi_{\tau_1}\psi_{aux})\\
&=& 2\g_{\chi_{\tau_1, \tau_2}}^{\a\b}\pr_\a(\chi_{\tau_1})\pr_\b(\psi_{aux})+\square_{\g_{\chi_{\tau_1, \tau_2}}}(\chi_{\tau_1})\psi_{aux}\\
&=&  2\g_{\chi_{\tau_1, \tau_2}}^{\tau\tau}\chi_{\tau_1}'(\tau)\pr_\tau(\psi_{aux})+2\g_{\chi_{\tau_1, \tau_2}}^{\tau r}\chi_{\tau_1}'(\tau)\pr_r(\psi_{aux})+2\g_{\chi_{\tau_1, \tau_2}}^{\tau x^a}\chi_{\tau_1}'(\tau)\pr_{x^a}(\psi_{aux})\\
&& +\left(\g_{\chi_{\tau_1, \tau_2}}^{\tau\tau}\chi_{\tau_1}''(\tau)+\frac{1}{\sqrt{|\g_{\chi_{\tau_1, \tau_2}}|}}\pr_\a(\sqrt{|\g_{\chi_{\tau_1, \tau_2}}|}\g_{\chi_{\tau_1, \tau_2}}^{\a\tau})\chi_{\tau_1}'(\tau)\right)\psi_{aux}.
\eeaa
Now, since $\g_{\chi_{\tau_1, \tau_2}}$ satisfies the assumptions of Section \ref{subsubsect:assumps:perturbedmetric} in view of Step 2, we infer from  \eqref{eq:assymptiticpropmetricKerrintaurxacoord:1} and \eqref{eq:controloflinearizedinversemetriccoefficients} that 
\beaa
\g_{\chi_{\tau_1, \tau_2}}^{\tau\tau}=O(m^2r^{-2}), \qquad \g_{\chi_{\tau_1, \tau_2}}^{\tau r}=-1+O(m^2r^{-2})+r\Ga_g, \qquad \g_{\chi_{\tau_1, \tau_2}}^{\tau a}=O(mr^{-2}),
\eeaa
and from \eqref{eq:assymptiticpropmetricKerrintaurxacoord:volumeform} and Lemma \ref{lemma:computationofthederiveativeofsrqtg} that 
\beaa
\bsplit
\frac{1}{\sqrt{|\g_{\chi_{\tau_1, \tau_2}}|}}\pr_\tau\Big(\sqrt{|\g_{\chi_{\tau_1, \tau_2}}|}\Big)&=r\dk^{\leq 1}\Ga_g, \quad \frac{1}{\sqrt{|\g_{\chi_{\tau_1, \tau_2}}|}}\pr_r\Big(\sqrt{|\g_{\chi_{\tau_1, \tau_2}}|}\Big)=\frac{2}{r}(1+O(m^2r^{-2}))+\dk^{\leq 1}\Ga_g, \\ 
\frac{1}{\sqrt{|\g_{\chi_{\tau_1, \tau_2}}|}}\pr_{x^a}\Big(\sqrt{|\g_{\chi_{\tau_1, \tau_2}}|}\Big)&=O(1).
\end{split}
\eeaa
{Hence} we deduce 
\bea\lab{eq:structureoftildef0}
\widetilde{F}_{(0)} &=&  -2\chi_{\tau_1}'(\tau)\left(\pr_r(\psi_{aux}) +\frac{1}{r}\psi_{aux}\right) +O(r^{-2})\Big(\chi_{\tau_1}''(\tau), \chi_{\tau_1}'(\tau)\Big)\dk^{\leq 1}\psi_{aux}.
\eea

\noindent{\bf Step 5.} Next, we introduce the solution $\widetilde{\psi}_1$ of the following scalar wave equation 
\bea\lab{eq:waveeqwidetildepsi1}
\bsplit
\square_{\g_{\chi_{\tau_1, \tau_2}}}\widetilde{\psi}_1 &= \widetilde{F}+\widetilde{F}_{(0)}\quad\textrm{on}\quad\MM,\\
\widetilde{\psi}_1&= \widetilde{\psi}, \quad N_{\Sigma(\tau_1)}\widetilde{\psi}_1=N_{\Sigma(\tau_1)}\widetilde{\psi}\quad\textrm{on}\quad\Si(\tau_1),\\
\widetilde{\psi}_1&= \chi_{\tau_1}\psi_\AA\quad\textrm{on}\quad\AA\cap\{\tau\leq\tau_1\},\\
\widetilde{\psi}_1&= \chi_{\tau_1}\psi_\II\quad\textrm{on}\quad\II_+\cap\{\tau\leq\tau_1\},
\end{split}
\eea
where we recall that $\psi_\AA$ {and $\psi_\II$} are smooth extensions of $\psi$ respectively from $\AA\cap\{\tau\geq \tau_1\}$ to $\AA\cap\{\tau_1-1\leq\tau\leq\tau_1\}$ {and from $\II\cap\{\tau\geq \tau_1\}$ to $\II\cap\{\tau_1-1\leq\tau\leq\tau_1\}$} satisfying \eqref{eq:propertyextensionpsiAAofpsitotauleqtau1}.

In particular, note by causality that we have
\beaa
\widetilde{\psi}_1=\widetilde{\psi}\quad\textrm{on}\quad \MM(\tau_1, +\infty), \qquad \widetilde{\psi}_1 = \chi_{\tau_1}\psi_{aux}\quad\textrm{on}\quad \MM(-\infty, \tau_1).
\eeaa
On the other hand, we have 
\beaa
\widetilde{\psi}=\psi\quad\textrm{on}\quad \MM(\tau_1+1, \tau_2-1)
\eeaa
by causality in view of \eqref{eq:waveequationdefiningfirstextensionwidetildepsi}, and we thus deduce
\bea\lab{eq:causlityrelationsforwidetildepsi1}
\widetilde{\psi}_1=\psi\quad\textrm{on}\quad \MM(\tau_1+1, \tau_2-1), \qquad \widetilde{\psi}_1 = 0\quad\textrm{on}\quad \MM(-\infty, \tau_1-1).
\eea

\noindent{\bf Step 6.} We have obtained so far the following:
\begin{itemize}
\item in view of Step 2, 
$\g_{\chi_{\tau_1, \tau_2}}$ satisfies the assumptions of Section \ref{subsubsect:assumps:perturbedmetric} and coincides with Kerr in $\MM\setminus(\tau_1, \tau_2)$, where we recall that $1\leq\tau_1<\tau_2<+\infty$,

\item in view of \eqref{eq:waveeqwidetildepsi1}, $\widetilde{\psi}_1$ is a solution to the scalar wave equation with RHS $\widetilde{F}+\widetilde{F}_{(0)}$,

\item in view of \eqref{eq:causlityrelationsforwidetildepsi1}, $\widetilde{\psi}_1$ can be smoothly extended by $0$ for $\tau\leq \tau_1-1$, {and hence for $\tau\leq 1$ since $\tau_1-1\geq 1$,}

\item {in view of \eqref{def:tildef} and \eqref{def:tildef0}, $ \widetilde{F}+\widetilde{F}_{(0)}$ is supported in $(\tau_1-1,\tau_2)$.} 
\end{itemize}
We may thus apply Theorem \ref{th:main:intermediary} which yields
\bea\lab{eq:th:main:intermediary1:application}
\EMF[\widetilde{\psi}_1](\mathbb{R})+\EMF[\pr_\tau\widetilde{\psi}_1](\mathbb{R}) \les \widetilde{\mathcal{F}}+\ep \EM^{(1)}[\widetilde{\psi}_1](\mathbb{R})
\eea
and
\bea\lab{eq:th:main:intermediary2:application}
{\EMF_{r_+(1-\dhor), 11m}[\pr_{\tphi}\widetilde{\psi}_1](\mathbb{R})} &\les& \widetilde{\mathcal{F}}+\ep \EM^{(1)}[\widetilde{\psi}_1](\mathbb{R})+{\M_{11m,12m}[\pr\widetilde{\psi}_1](\mathbb{R})},
\eea
where $\widetilde{\mathcal{F}}$ is defined by
\beaa
\widetilde{\mathcal{F}}:=\widetilde{\mathcal{N}}^{(1)}[ \widetilde{\psi}_1, \widetilde{F}+\widetilde{F}_{(0)}](\mathbb{R}).
\eeaa
In particular, \eqref{eq:th:main:intermediary1:application} implies, in view of \eqref{eq:higherorderenergymorawetzforlargerconditionalonsprtau}, 
\beaa
\nn\EMF^{(1)}_{r\geq 11m}[\widetilde{\psi}_1](\mathbb{R}) &\les& \EMF[\widetilde{\psi}_1](\mathbb{R})+\EMF[\pr_\tau\widetilde{\psi}_1](\mathbb{R})+\int_{\MM}|\pr^{\leq 1}(\widetilde{F}+\widetilde{F}_{(0)})|^2\\
&&{+\widehat{\mathcal{N}}^{(1)}_{r\geq 10m}[\psi, \widetilde{F}+\widetilde{F}_{(0)}](\tau_1, \tau_2)}
+\big(\EMF[\widetilde{\psi}_1](\tau_1, \tau_2)\big)^{\frac{1}{2}}\left(\EMF^{(1)}[\widetilde{\psi}_1](\mathbb{R})\right)^{\frac{1}{2}}\\
&\les& \widetilde{\mathcal{F}}  +\ep \EM^{(1)}[\widetilde{\psi}_1](\mathbb{R})+\big(\widetilde{\mathcal{F}}+\ep \EM^{(1)}[\widetilde{\psi}_1](\mathbb{R})\big)^{\frac{1}{2}}\left(\EMF^{(1)}[\widetilde{\psi}_1](\mathbb{R})\right)^{\frac{1}{2}},
\eeaa
where we used the {following consequence of the definitions \eqref{eq:definitionwidetildemathcalNpsijnormRHS} and \eqref{eq:defintionwidehatmathcalNfpsinormRHS}}
\bea\lab{eq:upperboundofL2normprm1fbywidetildemathcalF}
\int_{\MM}|\pr^{\leq 1}(\widetilde{F}+\widetilde{F}_{(0)})|^2 {+\widehat{\mathcal{N}}^{(1)}_{r\geq 10m}[\psi, \widetilde{F}+\widetilde{F}_{(0)}](\tau_1, \tau_2)} &\les& \widetilde{\mathcal{N}}^{(1)}[\widetilde{\psi}_1, \widetilde{F}+\widetilde{F}_{(0)}](\mathbb{R})\les \widetilde{\mathcal{F}}.
\eea
Together with  \eqref{eq:th:main:intermediary2:application}, this implies 
\beaa
{\EMF_{r_+(1-\dhor),11m}[\pr_{\tphi}\widetilde{\psi}_1](\mathbb{R})} \les \widetilde{\mathcal{F}}+\ep \EM^{(1)}[\widetilde{\psi}_1](\mathbb{R})+\big(\widetilde{\mathcal{F}}+\ep \EM^{(1)}[\widetilde{\psi}_1](\mathbb{R})\big)^{\frac{1}{2}}\left(\EMF^{(1)}[\widetilde{\psi}_1](\mathbb{R})\right)^{\frac{1}{2}}.
\eeaa
Together with \eqref{eq:th:main:intermediary1:application} and \eqref{eq:higherorderenergymorawetzconditionalonsprtausprphis}, we obtain
\beaa
\EMF^{(1)}[\widetilde{\psi}_1](\mathbb{R}) &\les& \widetilde{\mathcal{F}}+\ep \EM^{(1)}[\widetilde{\psi}_1](\mathbb{R})
+\big(\widetilde{\mathcal{F}}+\ep \EM^{(1)}[\widetilde{\psi}_1](\mathbb{R})\big)^{\frac{1}{2}}\left(\EMF^{(1)}[\widetilde{\psi}_1](\mathbb{R})\right)^{\frac{1}{2}},
\eeaa
where we have used again \eqref{eq:upperboundofL2normprm1fbywidetildemathcalF}, as well as the fact that $\widetilde{\psi}_1 = 0$ on $\MM(-\infty, \tau_1-1)$ in view of \eqref{eq:causlityrelationsforwidetildepsi1} which implies $\E^{(1)}[\widetilde{\psi}_1](-\infty)=0$. For $\ep>0$ small enough, we deduce 
\bea\lab{eq:intermedirayesimtaforEMF1ofwidetildepsi1intermsofwidehatF:bis}
\EMF^{(1)}[\widetilde{\psi}_1](\mathbb{R}) \les  \widetilde{\mathcal{F}}.
\eea

Next, we use the comparison of $\widetilde{\mathcal{N}}[\psi, F](\Reals)$ with $\widehat{\mathcal{N}}[\psi, F]$ provided by Lemma \ref{lemma:additivepropertyofwidetildeN} with $N=4$ and $\tau^{(1)}=\tau_1-2$, $\tau^{(2)}=\tau_1+2$, $\tau^{(3)}=\tau_2-2$, $\tau^{(4)}=\tau_2+1$, to obtain, for any $0<\la\leq 1$,
\beaa
\widetilde{\mathcal{F}} &=& \widetilde{\mathcal{N}}^{(1)}[\widetilde{\psi}_1, \widetilde{F}+\widetilde{F}_{(0)}](\mathbb{R})\\
&\les& \la^{-1}\widehat{\mathcal{N}}^{(1)}[\widetilde{\psi}_1, \widetilde{F}+\widetilde{F}_{(0)}](-\infty,\tau_1-1)+\la^{-1}\widehat{\mathcal{N}}^{(1)}[ \widetilde{\psi}_1, \widetilde{F}+\widetilde{F}_{(0)}](\tau_1-3, \tau_1+3)\\
&&+\la^{-1}\widehat{\mathcal{N}}^{(1)}[ \widetilde{\psi}_1, \widetilde{F}+\widetilde{F}_{(0)}](\tau_1+1, \tau_2-1)+\la^{-1}\widehat{\mathcal{N}}^{(1)}[ \widetilde{\psi}_1, \widetilde{F}+\widetilde{F}_{(0)}](\tau_2-3, \tau_2+2)\\
&&+\la^{-1}\widehat{\mathcal{N}}^{(1)}[ \widetilde{\psi}_1, \widetilde{F}+\widetilde{F}_{(0)}](\tau_2, +\infty) +\la\EM^{(1)}[\widetilde{\psi}_1](\Reals)\\
&&+\bigg(\int_{\Mtrap}|\widetilde{\psi}_1|^2\bigg)^{\frac{1}{2}}\bigg(\int_{\Mtrap}|\widetilde{F}+\widetilde{F}_{(0)}|^2\bigg)^{\frac{1}{2}}.
\eeaa
Using the properties of the support of $\widetilde{F}$ and $\widetilde{F}_{(0)}$, as well as \eqref{eq:causlityrelationsforwidetildepsi1} and \eqref{def:tildef}, we infer
\beaa
\bsplit
\widetilde{\mathcal{F}} &\les \la^{-1}\widehat{\mathcal{N}}^{(1)}[\widetilde{\psi}_1, \widetilde{F}+\widetilde{F}_{(0)}](\tau_1-1, \tau_1+3)+\la^{-1}\widehat{\mathcal{N}}^{(1)}[\psi, F](\tau_1+1, \tau_2-1)\\
&+\la^{-1}\widehat{\mathcal{N}}^{(1)}[ \widetilde{\psi}_1, \widetilde{F}](\tau_2-3, \tau_2+2) +\la\EM^{(1)}[\widetilde{\psi}_1](\Reals) +\bigg(\int_{\Mtrap}|\widetilde{\psi}_1|^2\bigg)^{\frac{1}{2}}\bigg(\int_{\Mtrap}|\widetilde{F}+\widetilde{F}_{(0)}|^2\bigg)^{\frac{1}{2}}.
\end{split}
\eeaa
Since we obviously have
\beaa
\widehat{\mathcal{N}}[\phi, F_1+F_2](\tau_1, \tau_2)\leq \widehat{\mathcal{N}}[\phi, F_1](\tau_1, \tau_2)+\widehat{\mathcal{N}}[\phi, F_2](\tau_1, \tau_2),
\eeaa
we deduce
\beaa
\widetilde{\mathcal{F}} &\les& \la^{-1}\widehat{\mathcal{N}}^{(1)}[\widetilde{\psi}_1, \widetilde{F}_{(0)}](\tau_1-1, \tau_1+3) +\la^{-1}\widehat{\mathcal{N}}^{(1)}[\widetilde{\psi}_1,\widetilde{F}](\tau_1-1, \tau_1+3)\\
&&+\la^{-1}\widehat{\mathcal{N}}^{(1)}[\psi, F](\tau_1+1, \tau_2-1)+\la^{-1}\widehat{\mathcal{N}}^{(1)}[ \widetilde{\psi}_1,\widetilde{F}](\tau_2-3, \tau_2+2)\\
&& +\la\EM^{(1)}[\widetilde{\psi}_1](\Reals) +\bigg(\int_{\Mtrap}|\widetilde{\psi}_1|^2\bigg)^{\frac{1}{2}}\bigg(\int_{\Mtrap}|\widetilde{F}+\widetilde{F}_{(0)}|^2\bigg)^{\frac{1}{2}}.
\eeaa
Now, note from \eqref{eq:defintionwidehatmathcalNfpsinormRHS} that 
\bea\lab{eq:usefulupperboundwidehatmathcalNonboundedtimeintervals}
\nn\widehat{\mathcal{N}}^{(1)}[{\phi,h}](\tau, \tau+q) 
&\les& \sup_{\tau'\in[\tau,\tau+q]}\bigg|\int_{\Mntrap(\tau, \tau')}\pr_\tau(\pr^{\leq 1}\phi)\pr^{\leq 1}h\bigg|+\int_{\MM(\tau, \tau+q)}|\pr^{\leq 1}h|^2
\nn\\
&&+\sqrt{q}\bigg(\sup_{\tau'\in[\tau, \tau+q]}\E^{(1)}[\phi](\tau')\bigg)^{\frac{1}{2}}\bigg(\int_{\MM(\tau, \tau+q)}|\pr^{\leq 1}h|^2\bigg)^{\frac{1}{2}}
\eea
and hence, taking also the support in $\tau$ of $\widetilde{F}$ and $\widetilde{F}^{(0)}$ into account, we obtain
\beaa
\bsplit
\widetilde{\mathcal{F}} \les&  \la^{-1}\sup_{\tau'\in[\tau_1-1, \tau_1]} \bigg|\int_{\Mntrap(\tau_1-1, \tau')}\pr_\tau(\pr^{\leq 1}\widetilde{\psi}_1)\pr^{\leq 1}\widetilde{F}_{(0)}\bigg|
\nn\\
&+\la^{-1}\int_{\Mntrap(\tau_1, \tau_2)}|\pr_\tau(\pr^{\leq 1}\widetilde{\psi}_1)||\pr^{\leq 1}F| +\la^{-1}\bigg(\sup_{\tau\in[\tau_1, \tau_2]}\E^{(1)}[\widetilde{\psi}_1](\tau)\bigg)^{\frac{1}{2}}\bigg(\int_{\MM(\tau_1, \tau_2)}|\pr^{\leq 1}F|^2\bigg)^{\frac{1}{2}}\\
&+\la^{-1}\widehat{\mathcal{N}}^{(1)}[\psi, F](\tau_1+1, \tau_2-1)+\la^{-1}\sup_{\tau\in[\tau_1-1, \tau_1]}\E^{(1)}[\psi_{aux}](\tau) +\la^{-1}\int_{\MM(\tau_1-1, \tau_1)}|\pr^{\leq 1}\widetilde{F}_{(0)}|^2\\
&+\la^{-1}\int_{\MM(\tau_1, \tau_2)}|\pr^{\leq 1}F|^2 +\la\EM^{(1)}[\widetilde{\psi}_1](\Reals)+\bigg(\int_{\Mtrap}|\widetilde{\psi}_1|^2\bigg)^{\frac{1}{2}}\bigg(\int_{\Mtrap}|\widetilde{F}+\widetilde{F}_{(0)}|^2\bigg)^{\frac{1}{2}},
\end{split}
\eeaa
which together with \eqref{eq:intermedirayesimtaforEMF1ofwidetildepsi1intermsofwidehatF:bis} yields, for $0<\la\leq 1$ small enough,
\bea\lab{eq:intermedirayesimtaforEMF1ofwidetildepsi1intermsofwidehatF:ter}
\nn&&\EMF^{(1)}[\widetilde{\psi}_1](\mathbb{R})\\ 
&\les& \sup_{\tau'\in[\tau_1-1, \tau_1]} \bigg|\int_{\Mntrap(\tau_1-1, \tau')}\pr_\tau(\pr^{\leq 1}\widetilde{\psi}_1)\pr^{\leq 1}\widetilde{F}_{(0)}\bigg|+\int_{\Mntrap(\tau_1, \tau_2)}|\pr_\tau(\pr^{\leq 1}\widetilde{\psi}_1)||\pr^{\leq 1}F|
\nn\\
\nn&&+\widehat{\mathcal{N}}^{(1)}[\psi, F](\tau_1+1, \tau_2-1)+\bigg(\sup_{\tau\in[\tau_1, \tau_2]}\E^{(1)}[\widetilde{\psi}_1](\tau)\bigg)^{\frac{1}{2}}\bigg(\int_{\MM(\tau_1, \tau_2)}|\pr^{\leq 1}F|^2\bigg)^{\frac{1}{2}}\\
&&+\sup_{\tau\in[\tau_1-1, \tau_1]}\E^{(1)}[\psi_{aux}](\tau) +\int_{\MM(\tau_1-1, \tau_1)}|\pr^{\leq 1}\widetilde{F}_{(0)}|^2+\int_{\MM(\tau_1, \tau_2)}|\pr^{\leq 1}F|^2,
\eea
{where we have also used} \eqref{def:tildef} and \eqref{def:tildef0} to infer
\beaa
&&\bigg(\int_{\Mtrap}|\widetilde{\psi}_1|^2\bigg)^{\frac{1}{2}}\bigg(\int_{\Mtrap}|\widetilde{F}+\widetilde{F}_{(0)}|^2\bigg)^{\frac{1}{2}}\nn\\
&\les& {\Big(\EMF^{(1)}[\widetilde{\psi}_1](\mathbb{R})\Big)^{\frac{1}{2}}\bigg(\int_{\MM(\tau_1-1, \tau_1)}|\pr^{\leq 1}\widetilde{F}_{(0)}|^2+\int_{\MM(\tau_1, \tau_2)}|\pr^{\leq 1}F|^2\bigg)^{\frac{1}{2}}}
\eeaa
{which allows to absorb the term $\EMF^{(1)}[\widetilde{\psi}_1](\mathbb{R})$ from the LHS.}

\noindent{\bf Step 7.} In this step, we estimate the terms involving $\widetilde{F}^{(0)}$ in the RHS of the final estimate of Step 6. It is at this stage that the assumption that $\psi$ is compactly supported in $\Si(\tau_1)$ plays an important role. Indeed, together with \eqref{eq:waveequationdefiningpsiaux} and \eqref{eq:waveeqwidetildepsi1}, this assumption on $\psi$ implies
\beaa
\widetilde{\psi}_1(\tauu=+\infty, \tau_1, \om)=\psi_{aux}(\tauu=+\infty, \tau_1, \om)=\psi(\tauu=+\infty, \tau_1, \om)=0, \quad \forall \om\in\mathbb{S}^2,
\eeaa
so that both $\psi_{aux}$ and $\widetilde{\psi}_1$ vanish at $\Si(\tau_1)\cap\II_+$. We may thus apply Lemma \ref{lemma:howtoexploitthecompactsupportofpsionSigramtau} with $\tau_0=\tau_1$ and $q=1$ which yields
\bea\lab{eq:estimatetotreatboundarytermsonII+inintegrationbypartscontroltildef0}
\bsplit
{\liminf_{\tauu\to +\infty}}\int_{\tau_1-1}^{\tau_1}\int_{\mathbb{S}^2}|\pr^{\leq 2}\psi_{aux}|^2({\tauu}, \tau, \om)r^2d\mathring{\ga}d\tau &\les \EF^{(1)}[\psi_{aux}](\tau_1-1,\tau_1),\\
{\liminf_{\tauu\to +\infty}}\int_{\tau_1-1}^{\tau_1}\int_{\mathbb{S}^2}|\pr^{\leq 2}\widetilde{\psi}_1|^2({\tauu}, \tau, \om)r^2d\mathring{\ga}d\tau &\les \EF^{(1)}[\widetilde{\psi}_1](\tau_1-1,\tau_1),
\end{split}
\eea
where \eqref{eq:estimatetotreatboundarytermsonII+inintegrationbypartscontroltildef0} will be used to control a priori dangerous boundary terms on $\II_+$ appearing in the various integrations by parts of Step 7. {Also, to simplify the exposition, as explained in Remark \ref{rmk:limitingargumentwithliminfonII+}, we will skip the limiting argument needed to apply \eqref{eq:estimatetotreatboundarytermsonII+inintegrationbypartscontroltildef0} and simply consider that \eqref{eq:estimatetotreatboundarytermsonII+inintegrationbypartscontroltildef0} holds at $\tauu=+\infty$.}

Next, recall from \eqref{eq:structureoftildef0} that $\widetilde{F}^{(0)}$ satisfies 
\beaa
\widetilde{F}_{(0)} &=&  -2\chi_{\tau_1}'(\tau)\left(\pr_r(\psi_{aux}) +\frac{1}{r}\psi_{aux}\right) +O(r^{-2})\Big(\chi_{\tau_1}''(\tau), \chi_{\tau_1}'(\tau)\Big)\dk^{\leq 1}\psi_{aux}.
\eeaa
Then, we have for any $\tau'\in [\tau_1-1,\tau_1]$,
\beaa
&&\bigg|\int_{\Mntrap(\tau_1-1, \tau')}\pr_\tau(\pr^{\leq 1}\widetilde{\psi}_1)\pr^{\leq 1}\left(\widetilde{F}_{(0)}+2\chi_{\tau_1}'(\tau)\left(\pr_r(\psi_{aux}) +\frac{1}{r}\psi_{aux}\right)\right)\bigg|\\
&&+\int_{\MM(\tau_1-1, \tau_1)}|\pr^{\leq 1}\widetilde{F}_{(0)}|^2\\
&\les& \int_{\Mntrap(\tau_1-1, \tau')}r^{-2}|\pr_\tau(\pr^{\leq 1}\widetilde{\psi}_1)||\dk^{\leq 1}\pr^{\leq 1}\psi_{aux}|+\int_{\MM(\tau_1-1, \tau')}r^{-2}|\dk^{\leq 1}\pr^{\leq 1}\psi_{aux}|^2\\
&\les& \left(\sup_{\tau\in[\tau_1-1, \tau_1]}\E^{(1)}[\widetilde{\psi}_1](\tau)\right)^{\frac{1}{2}}\left(\sup_{\tau\in[\tau_1-1, \tau_1]}\E^{(1)}[\psi_{aux}](\tau)\right)^{\frac{1}{2}}+\sup_{\tau\in[\tau_1-1, \tau_1]}\E^{(1)}[\psi_{aux}](\tau),
\eeaa
which together with \eqref{eq:intermedirayesimtaforEMF1ofwidetildepsi1intermsofwidehatF:ter} yields
\bea\lab{eq:intermediaryestimatewithmaintermoftildef0stilltobeestimated}
\nn\EMF^{(1)}[\widetilde{\psi}_1](\mathbb{R})\!\! &\les&  \sup_{\tau'\in [\tau_1-1,\tau_1]}\mathcal{J}(\tau')+\int_{\Mntrap(\tau_1, \tau_2)}|\pr_\tau(\pr^{\leq 1}\widetilde{\psi}_1)||\pr^{\leq 1}F| +\widehat{\mathcal{N}}^{(1)}[\psi, F](\tau_1+1, \tau_2-1)\\
&&+\sup_{\tau\in[\tau_1-1, \tau_1]}\E^{(1)}[\psi_{aux}](\tau) +\int_{\MM(\tau_1, \tau_2)}|\pr^{\leq 1}F|^2,
\eea
where $\mathcal{J}(\tau')$ is  given  by
\beaa
\mathcal{J}(\tau'):=\bigg|\int_{\Mntrap(\tau_1-1, \tau')}\pr_\tau(\pr^{\leq 1}\widetilde{\psi}_1)\pr^{\leq 1}\left(\chi_{\tau_1}'(\tau)\left(\pr_r(\psi_{aux}) +\frac{1}{r}\psi_{aux}\right)\right)\bigg|.
\eeaa

Next, we estimate $\mathcal{J}(\tau')$, which we rewrite as follows,
\beaa
\mathcal{J}(\tau')=\bigg|\int_{\Mntrap(\tau_1-1, \tau')}\pr_\tau(\pr^{\leq 1}\widetilde{\psi}_1)\pr^{\leq 1}\left(\chi_{\tau_1}'(\tau)\frac{1}{r}\pr_r(r\psi_{aux})\right)\bigg|,
\eeaa
and integrating by parts in $r$, we infer 
\beaa
\mathcal{J}(\tau') &\les& \left|\int_{\Mntrap(\tau_1-1, \tau_1)}\pr^{\leq 1}(\chi_{\tau_1}'(\tau)\psi_{aux})\frac{r}{\sqrt{|\g_{\chi_{\tau_1, \tau_2}}|}}\pr_r\left(\frac{\sqrt{|\g_{\chi_{\tau_1, \tau_2}}|}}{r}\pr_\tau(\pr^{\leq 1}\widetilde{\psi}_1)\right)\right|\\
&&+\left(\EF^{(1)}[\widetilde{\psi}_1](\tau_1-1,\tau_1)\right)^{\frac{1}{2}}\left(\EF^{(1)}[\psi_{aux}](\tau_1-1,\tau_1)\right)^{\frac{1}{2}},
\eeaa
where we have used \eqref{eq:estimatetotreatboundarytermsonII+inintegrationbypartscontroltildef0} to control the boundary terms generated on $\II_+(\tau_1-1,\tau')$. In view of
\beaa
&&\frac{r}{\sqrt{|\g_{\chi_{\tau_1, \tau_2}}|}}\pr_r\left(\frac{\sqrt{|\g_{\chi_{\tau_1, \tau_2}}|}}{r}\pr_\tau(\pr^{\leq 1}\widetilde{\psi}_1)\right)\\ 
&=& \pr^{\leq 1}\left(\frac{r}{\sqrt{|\g_{\chi_{\tau_1, \tau_2}}|}}\pr_r\left(\frac{\sqrt{|\g_{\chi_{\tau_1, \tau_2}}|}}{r}\pr_\tau\widetilde{\psi}_1\right)\right)+\left[\frac{r}{\sqrt{|\g_{\chi_{\tau_1, \tau_2}}|}}\pr_r\left(\frac{\sqrt{|\g_{\chi_{\tau_1, \tau_2}}|}}{r}\right), \pr\right]\pr_\tau(\widetilde{\psi}_1)\\
&=& \pr^{\leq 1}\left(\frac{r}{\sqrt{|\g_{\chi_{\tau_1, \tau_2}}|}}\pr_r\left(\frac{\sqrt{|\g_{\chi_{\tau_1, \tau_2}}|}}{r}\pr_\tau\widetilde{\psi}_1\right)\right)+\left[\frac{r}{\sqrt{|\gam|}}\pr_r\left(\frac{\sqrt{|\gam|}}{r}\right)+(N_{det})_r, \pr\right]\pr_\tau(\widetilde{\psi}_1)\\
&=& \pr^{\leq 1}\left(\frac{r}{\sqrt{|\g_{\chi_{\tau_1, \tau_2}}|}}\pr_r\left(\frac{\sqrt{|\g_{\chi_{\tau_1, \tau_2}}|}}{r}\pr_\tau\widetilde{\psi}_1\right)\right)+\left[\frac{1}{r}(1+O(m^2r^{-2}))+\dk^{\leq 1}\Ga_g, \pr\right]\pr_\tau(\widetilde{\psi}_1)\\
&=& \pr^{\leq 1}\left(\frac{r}{\sqrt{|\g_{\chi_{\tau_1, \tau_2}}|}}\pr_r\left(\frac{\sqrt{|\g_{\chi_{\tau_1, \tau_2}}|}}{r}\pr_\tau\widetilde{\psi}_1\right)\right)+O(r^{-2})\pr_\tau(\widetilde{\psi}_1)
\eeaa
where we used the estimate for $(N_{det})_r$ provided by Lemma \ref{lemma:computationofthederiveativeofsrqtg} and the control of  $\sqrt{|\det(\gam)|}$ provided by \eqref{eq:assymptiticpropmetricKerrintaurxacoord:volumeform}, we infer
\bea\lab{eq:controlofmathcalJwithmaintermisolatedandmathcalKremainder}
\mathcal{J}(\tau') \les \left|\int_{\Mntrap(\tau_1-1, \tau')}\pr^{\leq 1}(\chi_{\tau_1}'(\tau)\psi_{aux})\pr^{\leq 1}\left(\frac{r}{\sqrt{|\g_{\chi_{\tau_1, \tau_2}}|}}\pr_r\left(\frac{\sqrt{|\g_{\chi_{\tau_1, \tau_2}}|}}{r}\pr_\tau\widetilde{\psi}_1\right)\right)\right|+\mathcal{K},
\eea
where
\beaa
\mathcal{K} &:=& \left(\EF^{(1)}[\widetilde{\psi}_1](\tau_1-1,\tau_1)\right)^{\frac{1}{2}}\left(\EF^{(1)}[\psi_{aux}](\tau_1-1,\tau_1)\right)^{\frac{1}{2}}+\EF^{(1)}[\psi_{aux}](\tau_1-1,\tau_1).
\eeaa

Next, in order to control the first term at the RHS of \eqref{eq:controlofmathcalJwithmaintermisolatedandmathcalKremainder}, we provide an identity for $\square_\g\phi$ with $\phi$ a scalar function and $\g$ a metric satisfying \eqref{eq:controloflinearizedinversemetriccoefficients}. In view of \eqref{eq:controloflinearizedinversemetriccoefficients}, \eqref{eq:assymptiticpropmetricKerrintaurxacoord:1}, Lemma \ref{lemma:computationofthederiveativeofsrqtg} and \eqref{eq:assymptiticpropmetricKerrintaurxacoord:volumeform}, we have
\beaa
\g^{\a\b}\pr_\a\pr_\b+(N_{det})^\a\pr_\a\phi &=& O(1)\pr_r^2\phi +(-2+O(r^{-1}))\pr_r\pr_\tau\phi+O(r^{-1})\pr_r\pr_{x^a}\phi\\
&&+O(m^2r^{-2})\pr_\tau^2\phi+O(mr^{-2})\pr_{\tau}\pr_{x^a}\phi+O(r^{-2})\pr_{x^a}\pr_{x^b}
\eeaa
and
\beaa
&&\pr_\a(\g^{\a\b})\pr_\b\phi+\g^{\a\b}\left(\frac{1}{\sqrt{|\gam|}}\pr_\a\left(\sqrt{|\gam|}\right)+(N_{det})_\a\right)\pr_\b\phi\\
&=& O(r^{-1})\pr_r\phi+\left(-\frac{2}{r}+O(r^{-2})\right)\pr_\tau\phi+O(r^{-2})\pr_{x^a}\phi
\eeaa
which yields
\bea\lab{eq:expressionforthewaveoperatorisolatingprincipalterms}
\nn\square_\g\phi &=& \g^{\a\b}\pr_\a\pr_\b\phi+\pr_\a(\g^{\a\b})\pr_\b\phi+\g^{\a\b}\left(\frac{1}{\sqrt{|\gam|}}\pr_\a\left(\sqrt{|\gam|}\right)+(N_{det})_\a\right)\pr_\b\phi\\
\nn&=& -2\left(\pr_r\pr_\tau\phi +\frac{1}{r}\pr_\tau\phi\right) +O(1)\big(\pr_r^2, r^{-1}\pr_{x^a}\pr_r, r^{-2}\pr_{x^a}\pr_{x^b}\big)\phi\\
&&+O(r^{-1})\big(\pr_r, r^{-1}\pr_{x^a}, r^{-1}\pr_\tau\big)\pr^{\leq 1}\phi.
\eea
Now, we have, using again Lemma \ref{lemma:computationofthederiveativeofsrqtg} and \eqref{eq:assymptiticpropmetricKerrintaurxacoord:volumeform},
\beaa
\frac{r}{\sqrt{|\g_{\chi_{\tau_1, \tau_2}}|}}\pr_r\left(\frac{\sqrt{|\g_{\chi_{\tau_1, \tau_2}}|}}{r}\pr_\tau\widetilde{\psi}_1\right) &=& \pr_r\pr_\tau(\widetilde{\psi}_1)+\left(-\frac{1}{r}+\frac{1}{\sqrt{|\g_{\chi_{\tau_1, \tau_2}}|}}\pr_r\left(\sqrt{|\g_{\chi_{\tau_1, \tau_2}}|}\right)\right)\pr_\tau\widetilde{\psi}_1\\
&=& \pr_r\pr_\tau(\widetilde{\psi}_1)+\left(\frac{r}{\sqrt{|\gam|}}\pr_r\left(\frac{\sqrt{|\gam|}}{r}\right)+(N_{det})_r\right)\pr_\tau\widetilde{\psi}_1\\
&=& \pr_r\pr_\tau(\widetilde{\psi}_1)+\frac{1}{r}\pr_\tau(\widetilde{\psi}_1)+O(r^{-2})\pr_\tau\widetilde{\psi}_1,
\eeaa
and hence, applying \eqref{eq:expressionforthewaveoperatorisolatingprincipalterms} with $\g\to\g_{\chi_{\tau_1, \tau_2}}$ and $\phi\to \widetilde{\psi}_1$, we infer
\beaa
\frac{r}{\sqrt{|\g_{\chi_{\tau_1, \tau_2}}|}}\pr_r\left(\frac{\sqrt{|\g_{\chi_{\tau_1, \tau_2}}|}}{r}\pr_\tau\widetilde{\psi}_1\right) &=& -\frac{1}{2}\square_{\tilde{g}}(\widetilde{\psi}_1)+O(1)\big(\pr_r^2, r^{-1}\pr_{x^a}\pr_r, r^{-2}\pr_{x^a}\pr_{x^b}\big)\widetilde{\psi}_1\\
&&+O(r^{-1})\big(\pr_r, r^{-1}\pr_{x^a}, r^{-1}\pr_\tau\big)\pr^{\leq 1}\widetilde{\psi}_1\\
&=& -\frac{1}{2}(\widetilde{F}+\widetilde{F}^{(0)})+O(1)\big(\pr_r^2, r^{-1}\pr_{x^a}\pr_r, r^{-2}\pr_{x^a}\pr_{x^b}\big)\widetilde{\psi}_1\\
&&+O(r^{-1})\big(\pr_r, r^{-1}\pr_{x^a}, r^{-1}\pr_\tau\big)\pr^{\leq 1}\widetilde{\psi}_1,
\eeaa
where we have used \eqref{eq:waveeqwidetildepsi1}. Together with \eqref{eq:controlofmathcalJwithmaintermisolatedandmathcalKremainder}, and using also the fact that $\widetilde{F}=0$ for $\tau\in (\tau_1-1, \tau_1)$, we deduce 
\beaa
\mathcal{J}(\tau') &\les& \bigg|\int_{\Mntrap(\tau_1-1, \tau')}\pr^{\leq 1}(\chi_{\tau_1}'(\tau)\psi_{aux})\pr^{\leq 1}\widetilde{F}^{(0)}\bigg|\\
&&+\bigg|\int_{\Mntrap(\tau_1-1, \tau')}\pr^{\leq 1}(\chi_{\tau_1}'(\tau)\psi_{aux})\pr^{\leq 1}\Big(O(1)\big(\pr_r^2, r^{-1}\pr_{x^a}\pr_r, r^{-2}\pr_{x^a}\pr_{x^b}\big)\widetilde{\psi}_1\Big)\bigg|\\
&&+\bigg|\int_{\Mntrap(\tau_1-1, \tau')}\pr^{\leq 1}(\chi_{\tau_1}'(\tau)\psi_{aux})\pr^{\leq 1}\Big(O(r^{-1})\big(\pr_r, r^{-1}\pr_{x^a}, r^{-1}\pr_\tau\big)\pr^{\leq 1}\widetilde{\psi}_1\Big)\bigg|+\mathcal{K}.
\eeaa
Integrating by parts in the second and the third term at the RHS, we infer 
\beaa
\mathcal{J}(\tau') &\les& \bigg|\int_{\Mntrap(\tau_1-1, \tau')}\pr^{\leq 1}(\chi_{\tau_1}'(\tau)\psi_{aux})\pr^{\leq 1}\widetilde{F}^{(0)}\bigg|\\
&&+\int_{\Mntrap(\tau_1-1, \tau')}\Big|\big(\pr_r, r^{-1}\pr_{x^a}\big)\pr^{\leq 1}\widetilde{\psi}_1\Big|\Big|\big(\pr_r, r^{-1}\pr_{x^a}\big)\pr^{\leq 1}\psi_{aux}\Big|\\
&&+\int_{\Mntrap(\tau_1-1, \tau')}r^{-1}|\pr^{\leq 2}\psi_{aux}|\Big|\big(\pr_r, r^{-1}\pr_{x^a}, r^{-1}\pr_\tau\big)\pr^{\leq 1}\widetilde{\psi}_1\Big|+\mathcal{K}\\
&\les& \bigg|\int_{\Mntrap(\tau_1-1, \tau')}\pr^{\leq 1}(\chi_{\tau_1}'(\tau)\psi_{aux})\pr^{\leq 1}\widetilde{F}^{(0)}\bigg| +\mathcal{K},
\eeaa
where we also used the definition of $\mathcal{K}$, the estimate  \eqref{eq:estimatetotreatboundarytermsonII+inintegrationbypartscontroltildef0} to control the boundary terms generated on $\II_+(\tau_1-1,\tau_1)$, and the fact that integration by parts w.r.t. to $\pr_r$ and $\pr_{x^a}$ do not produce boundary terms on $\Si(\tau_1-1)$ and $\Si(\tau')$. Recalling from \eqref{eq:structureoftildef0} that $\widetilde{F}^{(0)}$ satisfies 
\beaa
\widetilde{F}_{(0)} &=&  -2\chi_{\tau_1}'(\tau)r^{-1}\pr_r(r\psi_{aux}) +O(r^{-2})\Big(\chi_{\tau_1}''(\tau), \chi_{\tau_1}'(\tau)\Big)\dk^{\leq 1}\psi_{aux},
\eeaa
we obtain 
\beaa
\mathcal{J}(\tau') &\les& \bigg|\int_{\Mntrap(\tau_1-1, \tau')}\pr^{\leq 1}(\chi_{\tau_1}'(\tau)\psi_{aux})\pr^{\leq 1}(\chi_{\tau_1}'(\tau)r^{-1}\pr_r(r\psi_{aux}))\bigg| +\mathcal{K}\\
&\les& \bigg|\int_{\Mntrap(\tau_1-1, \tau')}\frac{1}{r^2}\pr_r\Big(\big(\pr^{\leq 1}(\chi_{\tau_1}'(\tau)r\psi_{aux})\big)^2\Big)\bigg| +\mathcal{K}.
\eeaa
Integrating by parts in $\pr_r$ again, we deduce 
\beaa
\mathcal{J}(\tau') &\les& \left|\int_{\Mntrap(\tau_1-1, \tau')}(\chi_{\tau_1}'(\tau))^2\frac{r^2}{\sqrt{|\g_{\chi_{\tau_1, \tau_2}}|}}\pr_r\left(\frac{\sqrt{|\g_{\chi_{\tau_1, \tau_2}}|}}{r^2}\right)(\psi_{aux})^2\right| +\mathcal{K},
\eeaa
which together with Lemma \ref{lemma:computationofthederiveativeofsrqtg} and \eqref{eq:assymptiticpropmetricKerrintaurxacoord:volumeform}, as well as the definition of $\mathcal{K}$, yields
\beaa
\sup_{\tau'\in[\tau_1-1,\tau_1]}\mathcal{J}(\tau') &\les& \left(\EF^{(1)}[\widetilde{\psi}_1](\tau_1-1,\tau_1)\right)^{\frac{1}{2}}\left(\EF^{(1)}[\psi_{aux}](\tau_1-1,\tau_1)\right)^{\frac{1}{2}}\nn\\
&&+\EF^{(1)}[\psi_{aux}](\tau_1-1,\tau_1).
\eeaa
Plugging this estimate back into  \eqref{eq:intermediaryestimatewithmaintermoftildef0stilltobeestimated}, we infer
\beaa
\nn\EMF^{(1)}[\widetilde{\psi}_1](\mathbb{R}) &\les& \int_{\Mntrap(\tau_1, \tau_2)}|\pr_\tau(\pr^{\leq 1}\widetilde{\psi}_1)||\pr^{\leq 1}F| +\widehat{\mathcal{N}}^{(1)}[\psi, F](\tau_1+1, \tau_2-1)\\
&&+\EF^{(1)}[\psi_{aux}](\tau_1-1,\tau_1) +\int_{\MM(\tau_1, \tau_2)}|\pr^{\leq 1}F|^2
\eeaa
which together with \eqref{eq:localenergyestimateforpsiauxonSigma} yields
\beaa
\nn\EMF^{(1)}[\widetilde{\psi}_1](\mathbb{R}) &\les& \E^{(1)}[\psi](\tau_1) +\int_{\Mntrap(\tau_1, \tau_2)}|\pr_\tau(\pr^{\leq 1}\widetilde{\psi}_1)||\pr^{\leq 1}F| +\widehat{\mathcal{N}}^{(1)}[\psi, F](\tau_1+1, \tau_2-1)\\
&&+\int_{\MM(\tau_1, \tau_2)}|\pr^{\leq 1}F|^2+\widehat{\mathcal{N}}^{(1)}[\psi, F](\tau_1, \tau_1+1)+\widehat{\mathcal{N}}^{(1)}[\widetilde{\psi}, \widetilde{F}](\tau_1, \tau_1+1).
\eeaa
Now, in view of Step 5, we have $\widetilde{\psi}_1=\widetilde{\psi}$ on $\MM(\tau_1,+\infty)$ and hence, using also \eqref{eq:usefulupperboundwidehatmathcalNonboundedtimeintervals}, we have 
\beaa
\widehat{\mathcal{N}}^{(1)}[\widetilde{\psi}, \widetilde{F}](\tau_1, \tau_1+1) &=& \widehat{\mathcal{N}}^{(1)}[\widetilde{\psi}_1, \widetilde{F}](\tau_1, \tau_1+1)\\
&\les& \int_{\Mntrap(\tau_1, \tau_1+1)}|\pr_\tau(\pr^{\leq 1}\widetilde{\psi}_1)||\pr^{\leq 1}F| +\int_{\MM(\tau_1, \tau_1+1)}|\pr^{\leq 1}F|^2\\
&&+\left(\sup_{\tau\in[\tau_1, \tau_1+1]}\E^{(1)}[\widetilde{\psi}_1](\tau)\right)^{\frac{1}{2}}\left(\int_{\MM(\tau_1, \tau_1+1)}|\pr^{\leq 1}F|^2\right)^{\frac{1}{2}},
\eeaa
which implies 
\beaa
\nn\EMF^{(1)}[\widetilde{\psi}_1](\mathbb{R}) &\les& \E^{(1)}[\psi](\tau_1) +\int_{\Mntrap(\tau_1, \tau_2)}|\pr_\tau(\pr^{\leq 1}\widetilde{\psi}_1)||\pr^{\leq 1}F| +\widehat{\mathcal{N}}^{(1)}[\psi, F](\tau_1, \tau_2-1)\\
&&+\int_{\MM(\tau_1, \tau_2)}|\pr^{\leq 1}F|^2
\eeaa
and thus
\beaa
\nn\EMF^{(1)}[\widetilde{\psi}_1](\tau_1, \tau_2) &\les& \E^{(1)}[\psi](\tau_1) +\int_{\Mntrap(\tau_1, \tau_2)}|\pr_\tau(\pr^{\leq 1}\widetilde{\psi}_1)||\pr^{\leq 1}F| +\widehat{\mathcal{N}}^{(1)}[\psi, F](\tau_1, \tau_2-1)\\
&&+\int_{\MM(\tau_1, \tau_2)}|\pr^{\leq 1}F|^2.
\eeaa

\noindent{\bf Step 8.} We now upgrade the control for $\EMF^{(1)}[\widetilde{\psi}_1](\tau_1, \tau_2)$ in Step 7 to a control of $\EMF^{(1)}_\de[\widetilde{\psi}_1](\tau_1, \tau_2)$. In view of \eqref{eq:imporvedhigherorderenergymorawetzconditionalonlowerrweigth} and \eqref{eq:waveeqwidetildepsi1}, we have
\beaa
\M^{(1)}_{\de}[\widetilde{\psi}_1](\tau_1, \tau_2)\les \EMF^{(1)}[\widetilde{\psi}_1](\tau_1, \tau_2)+\int_{\MM(\tau_1, \tau_2)}r^{1+\de}|\pr^{\leq s}(\widetilde{F}+\widetilde{F}_{(0)})|^2.
\eeaa
Together with the properties of the support of $\widetilde{F}_{(0)}$, we deduce 
\beaa
\M^{(1)}_{\de}[\widetilde{\psi}_1](\tau_1, \tau_2)\les \EMF^{(1)}[\widetilde{\psi}_1](\tau_1, \tau_2)+\int_{\MM(\tau_1, \tau_2)}r^{1+\de}|\pr^{\leq s}\widetilde{F}|^2.
\eeaa
Hence, using also the definition of $\widetilde{F}$, we obtain 
\beaa
\EMF^{(1)}_{\de}[\widetilde{\psi}_1](\tau_1, \tau_2)\les \EMF^{(1)}[\widetilde{\psi}_1](\tau_1, \tau_2)+\int_{\MM(\tau_1, \tau_2)}r^{1+\de}|\pr^{\leq s}F|^2
\eeaa
which together with the control for $\EMF^{(1)}[\widetilde{\psi}_1](\tau_1, \tau_2)$ in Step 7 implies
\beaa
\nn\EMF^{(1)}_\de[\widetilde{\psi}_1](\tau_1, \tau_2) &\les& \E^{(1)}[\psi](\tau_1) +\int_{\Mntrap(\tau_1, \tau_2)}|\pr_\tau(\pr^{\leq 1}\widetilde{\psi}_1)||\pr^{\leq 1}F| +\widehat{\mathcal{N}}^{(1)}[\psi, F](\tau_1, \tau_2-1)\\
&&+\int_{\MM(\tau_1, \tau_2)}r^{1+\de}|\pr^{\leq s}F|^2.
\eeaa
Now, we have from Remark \ref{rmk:controlofwidehatNfpsibyNfpsi} the following comparison between $\widehat{\mathcal{N}}[\psi, F](\tau_1, \tau_2)$ and $\mathcal{N}_\de[\psi, F](\tau_1, \tau_2)$
\bea\lab{eq:comparaisonwideharNpsifwithNdeltapsif}
\widehat{\mathcal{N}}[F, \psi](\tau_1, \tau_2) &\les& \mathcal{N}_\de[\psi, F](\tau_1, \tau_2)+\left(\M_\de[\psi](\tau_1, \tau_2)\right)^{\frac{1}{2}}\left(\mathcal{N}_\de[\psi, F](\tau_1, \tau_2)\right)^{\frac{1}{2}}
\eea
which yields
\beaa
\nn\EMF^{(1)}_\de[\widetilde{\psi}_1](\tau_1, \tau_2) &\les& \E^{(1)}[\psi](\tau_1) +\mathcal{N}_\de[\psi, F](\tau_1, \tau_2).
\eeaa
Together with \eqref{eq:causlityrelationsforwidetildepsi1}, we infer
\bea\lab{eq:almosttotheendforcontrolEMFreg1depsitau1p1tau2m1}
\EMF^{(1)}_\de[\psi](\tau_1+1, \tau_2-1) &\les& \E^{(1)}[\psi](\tau_1) +\mathcal{N}_\de[\psi, F](\tau_1, \tau_2).
\eea
Now, applying the first local energy estimate of Lemma \ref{lemma:localenergyestimate} respectively on $(\tau_1, \tau_1+1)$ and on $(\tau_2-1, \tau_2)$, and using again \eqref{eq:comparaisonwideharNpsifwithNdeltapsif}, we have  
\beaa
\bsplit
\EF^{(1)}[\psi](\tau_1, \tau_1+1) &\les \E^{(1)}[\psi](\tau_1)+\mathcal{N}^{(1)}_\de[\psi, F](\tau_1, \tau_1+1)\\
&+\left(\M^{(1)}_\de[\psi](\tau_1, \tau_1+1)\right)^{\frac{1}{2}}\left(\mathcal{N}^{(1)}_\de[\psi, F](\tau_1, \tau_1+1)\right)^{\frac{1}{2}},\\
\EF^{(1)}[\psi](\tau_2-1, \tau_2) &\les  \E^{(1)}[\psi](\tau_2-1)+\mathcal{N}^{(1)}_\de[\psi, F](\tau_2-1, \tau_2)\\
&+\left(\M^{(1)}_\de[\psi](\tau_2-1, \tau_2)\right)^{\frac{1}{2}}\left(\mathcal{N}^{(1)}_\de[\psi, F](\tau_2-1, \tau_2)\right)^{\frac{1}{2}}.
\end{split}
\eeaa
Together with the fact that 
\beaa
\M^{(1)}[\psi](\tau_0, \tau_0+q) \les_q \sup_{\tau\in[\tau_0, \tau_0+q]}\E^{(1)}[\psi](\tau),
\eeaa
and \eqref{eq:imporvedhigherorderenergymorawetzconditionalonlowerrweigth}, we infer
\beaa
\bsplit
\EMF_\de^{(1)}[\psi](\tau_1, \tau_1+1) &\les \E^{(1)}[\psi](\tau_1)+\mathcal{N}^{(1)}_\de[\psi, F](\tau_1, \tau_1+1),\\
\EMF_\de^{(1)}[\psi](\tau_2-1, \tau_2) &\les  \E^{(1)}[\psi](\tau_2-1)+\mathcal{N}^{(1)}_\de[\psi, F](\tau_2-1, \tau_2).
\end{split}
\eeaa
In view of \eqref{eq:almosttotheendforcontrolEMFreg1depsitau1p1tau2m1}, we deduce
\bea\lab{eq:almosttotheendforcontrolEMFreg1:onlycompacsupport}
\EMF^{(1)}_\de[\psi](\tau_1, \tau_2) &\les&  \E^{(1)}[\psi](\tau_1)+\mathcal{N}^{(1)}_\de[\psi, F](\tau_1, \tau_2)
\eea
which concludes the proof of Theorem \ref{thm:main} under the assumption that $\psi$ is compactly supported in $\Si(\tau_1)$.

\noindent{\bf Step 9.} We now extend  \eqref{eq:almosttotheendforcontrolEMFreg1:onlycompacsupport} to any $\psi$ such that $\E^{(1)}[\psi](\tau_1)<+\infty$. To this end, we rely on the following lemma.
\begin{lemma}\lab{lemma:densityofcompactlysupportedfunctionsinSitau1}
Let $\psi$ be a scalar function such that $\E^{(1)}[\psi](\tau_1)<+\infty$. Then, there exists a sequence $(\psi_p)_{p\geq 1}$ of scalar functions such that $\psi_p$ is compactly supported in $\Si(\tau_1)$ and 
\beaa
\lim_{p\to +\infty}\E^{(1)}[\psi_p-\psi](\tau_1)=0.
\eeaa
\end{lemma}

\begin{proof}
We consider a cut-off function $\chi$ such that $\chi(r)=1$ for $r\leq 1$, $\chi(r)=0$ for $r\geq 2$, and $|\pr_r\chi|\les 1$, and for $p\geq 1$, we define $\chi_p(r):=\chi(r/p)$. Then, we consider the sequence $\psi_p$ satisfying 
\beaa
\psi_p(\tau_1,\cdot)=\chi_p(r)\psi(\tau_1, \cdot), \qquad {\pr_\tau(\psi_p)(\tau_1,\cdot)=\chi_p(r)\pr_\tau(\psi)(\tau_1,\cdot)},
\eeaa
{which prescribes the initial data $(\psi_p(\tau_1,\cdot), N_{\Si(\tau_1)}(\psi_p)(\tau_1,\cdot))$ of $\psi_p$ on $\Si(\tau_1)$, and} so that $\psi_p$ is compactly supported in $\Si(\tau_1)$. We have
\beaa
\E^{(1)}[\psi-\psi_p](\tau_1) &\les&  \int_p^{+\infty}\int_{\mathbb{S}^2}\Big((\pr_r\pr^{\leq 1}\psi)^2+|\nab\pr^{\leq 1}\psi|^2\\
&&\qquad\qquad\qquad+r^{-2}\big((\pr_\tau\pr^{\leq 1}\psi)^2+(\pr^{\leq 1}\psi)^2\big)\Big)(\tau=\tau_1, r, \om)
r^2d\mathring{\ga}dr
\eeaa
which converges to $0$ as $p\to +\infty$ by Lebesgue dominated convergence theorem. This concludes the proof of Lemma \ref{lemma:densityofcompactlysupportedfunctionsinSitau1}.
\end{proof}

Let $\psi$ be a scalar function satisfying the wave equation \eqref{eq:scalarwave}  such that $\E[\psi](\tau_1)<+\infty$. Then, according to Lemma \ref{lemma:densityofcompactlysupportedfunctionsinSitau1}, there exists a sequence $(\psi_p)_{p\geq 1}$ of scalar functions satisfying the wave equation \eqref{eq:scalarwave} such that $\psi_p$ is compactly supported in $\Si(\tau_1)$ and 
\beaa
\lim_{p\to +\infty}\E^{(1)}[\psi_p-\psi](\tau_1)=0.
\eeaa
Since $\psi_p$ is compactly supported in $\Si(\tau_1)$, $\psi_p-\psi_q$ are also  compactly supported functions in $\Si(\tau_1)$, so we may apply  \eqref{eq:almosttotheendforcontrolEMFreg1:onlycompacsupport} with $F=0$ which implies 
\beaa
\EMF^{(1)}_\de[\psi_p-\psi_q](\tau_1, \tau_2)\les \E^{(1)}[\psi_p-\psi_q](\tau_1).
\eeaa
We deduce that $(\psi_p)_{p\geq 1}$ is a Cauchy sequence for the norm $\EMF_\de^{(1)}[\cdot](\tau_1, \tau_2)$, which must converge to $\psi$ by uniqueness for the wave equation. Hence, using also \eqref{eq:almosttotheendforcontrolEMFreg1:onlycompacsupport} for each $\psi_p$, 
\beaa
\EMF_\de^{(1)}[\psi](\tau_1, \tau_2) &=& \lim_{p\to +\infty}\EMF_\de^{(1)}[\psi_p](\tau_1, \tau_2)\\
&\les& \limsup_{p\to +\infty}\Big(\E^{(1)}[\psi_p](\tau_1)+\mathcal{N}_\de^{(1)}[\psi_p, F](\tau_1, \tau_2)\Big)\\
&\les&  \E^{(1)}[\psi](\tau_1)+\limsup_{p\to +\infty}\Big(\mathcal{N}_\de^{(1)}[\psi_p, F](\tau_1, \tau_2)\Big).
\eeaa
Now, since
\beaa
&&\left|\mathcal{N}_\de^{(1)}[\psi_p, F](\tau_1, \tau_2)-\mathcal{N}_\de^{(1)}[\psi, F](\tau_1, \tau_2)\right|\\
&\les& (\tau_2-\tau_1)^{\frac{1}{2}}\left(\sup_{\tau\in[\tau_1, \tau_2]}\E^{(1)}[\psi_p-\psi](\tau)\right)^{\frac{1}{2}}\left(\int_{\Mtrap(\tau_1, \tau_2)}|\pr^{\leq 1}F|^2\right)^{\frac{1}{2}}
\eeaa
we infer that $\mathcal{N}_\de^{(1)}[\psi_p, F](\tau_1, \tau_2)-\mathcal{N}_\de^{(1)}[\psi, F](\tau_1, \tau_2)\to 0$ as $p\to +\infty$ since $(\psi_p)_{p\geq 1}$ converges to $\psi$ in the norm $\EMF_\de^{(1)}[\cdot](\tau_1, \tau_2)$. Thus, we deduce from the above that 
\beaa
\EMF_\de^{(1)}[\psi](\tau_1, \tau_2) &\les&  \E^{(1)}[\psi](\tau_1)+\mathcal{N}_\de^{(1)}[\psi, F](\tau_1, \tau_2)
\eeaa
which proves that \eqref{eq:almosttotheendforcontrolEMFreg1:onlycompacsupport}  extends to  any $\psi$ such that $\E^{(1)}[\psi](\tau_1)<+\infty$. This concludes the proof of Theorem \ref{thm:main}.

%%%%%%%%%%%%%%%%%%%%%%%%%%%%%%%%

\subsection{Structure of the rest of the paper}

%%%%%%%%%%%%%%%%%%%%%%%%%%%%%%%%

In Section \ref{sect:microlocalcalculus}, we introduce a microlocal calculus on $\MM$ {that will be used throughout the rest of the paper. Then}, we provide in Section \ref{sec:proofofth:main:intermediary} a proof of our {global} energy-Morawetz estimate Theorem \ref{th:main:intermediary} {relying in particular on the combination of: 
\begin{enumerate}
\item the microlocal energy-Morawetz estimate of Theorem \ref{th:mainenergymorawetzmicrolocal} which is conditional on the control of lower order terms,
\item the energy-Morawetz estimate of Proposition \ref{prop:energymorawetzmicrolocalwithblackbox} which provides the control of lower order terms at the expense of the loss of one derivative.
\end{enumerate}
Finally, Theorem \ref{th:mainenergymorawetzmicrolocal} is proved in Section \ref{sect:CondEMF:Dynamic} and Proposition \ref{prop:energymorawetzmicrolocalwithblackbox}} is proved in Section \ref{sec:proofofprop:energymorawetzmicrolocalwithblackbox}.

%%%%%%%%%%%%%%%%%%%%%%%%%%%%%%%%

\section{Microlocal calculus on $\MM$}
\lab{sect:microlocalcalculus}

%%%%%%%%%%%%%%%%%%%%%%%%%%%%%%%%

In this section, we introduce a microlocal calculus on $\MM$ that will be used in Sections \ref{sec:proofofth:main:intermediary}, \ref{sect:CondEMF:Dynamic} and \ref{sec:proofofprop:energymorawetzmicrolocalwithblackbox}. We start by reviewing pseudo-differential operators (PDOs) in $\Reals^n$ in Section \ref{sect:PDO:Rn} and then extend both the definitions and the statements from $\Reals^n$ to our spacetime $\MM$ in Section \ref{sect:PDO:rfoliation}.

%%%%%%%%%%%%%%%%%%%%%%%%%%%%%%%%%%%

\subsection{Pseudodifferential operators on $\mathbb{R}^n$}
\label{sect:PDO:Rn}

%%%%%%%%%%%%%%%%%%%%%%%%%%%%%%%%%%%

This section is devoted to a basic introduction of pseudodifferential operators.  The material presented here is standard and can be found for example in textbooks \cite[Chapter 18.5]{hormander}, \cite{Taylor91}, \cite[Chapter I]{AG07} and \cite[Appendix E]{DZ19}.

%%%%%%%%%%%%%%%%%%%%%%%%%%%%%%%%%%%%%

\subsubsection{Symbols and symbol classes}

%%%%%%%%%%%%%%%%%%%%%%%%%%%%%%%%%%%%%

We first define the symbols on $\Reals^n$.
\begin{definition}[Symbols and symbol classes]
\label{def:symbols:Rn}
For $m\in \mathbb{R}$, let $S^m =S^m (\mathbb{R}^n\times \mathbb{R}^n)$ denote the set of functions $a(x,\xi)\in C^{\infty} (\mathbb{R}^n\times \mathbb{R}^n)$ such that
 \begin{align}
 \forall \alpha, \forall \beta,\quad  \abs{\partial_x^{\alpha} \partial_\xi^{\beta}a(x,\xi)}\leq C_{\alpha,\beta} \langle \xi\rangle^{m-\abs{\beta}}
 \end{align}
 where $C_{\alpha,\beta}<+\infty$, and where $\langle \xi\rangle :=\sqrt{1+\abs{\xi}^2}$. $S^m$ is called the symbol class of order $m$ and $a\in S^m$ is called a {symbol of order $m$}. We also denote $S^{-\infty} :=\cap_{m\in \Reals} S^m$. 
 \end{definition}

%%%%%%%%%%%%%%%%%%%%%%%%%%%%%%%%%%%%%

\subsubsection{Weyl quantization on $\mathbb{R}^n$}

%%%%%%%%%%%%%%%%%%%%%%%%%%%%%%%%%%%%%

In this paper, we will always rely on the Weyl quantization of PDOs which we recall  below.
\begin{definition}[PDO in the Weyl quantization]
\label{def:PDO:Rn:Weylquan}
If $a\in S^m$ and $u\in \mathcal{S}$, where $\mathcal{S}$ is the set of Schwartz functions, then the formula
\begin{align}
\label{PDO:Rn:Opw}
\Opw(a) u(x) :={}&  (2\pi)^{-n} \int_{\Reals^n}  \int_{\Reals^n} e^{i(x-y)\cdot\xi} a\bigg(\frac{x+y}{2},\xi\bigg)  u(y)d\xi d y
\end{align}
defines a function of $\mathcal{S}$, the mapping 
\begin{align}
(a,u)\rightarrow \Opw(  a) u
\end{align}
is continuous, and the linear mapping $\Opw$ from $S^m$ to the linear operators of $\mathcal{S}$ is injective. Moreover, $a$ is said to be the symbol of the operator $\Opw(a)$, and  the operator $\Opw(a)$ is called the {Weyl quantization} of the symbol $a$. Finally, a PDO is said to be a PDO of order $m$ if its symbol belongs to $S^m$.
\end{definition}

%%%%%%%%%%%%%%%%%%%%%%%%%%%%%%%%%%%%%

\subsubsection{Properties of the Weyl quantization on $\mathbb{R}^n$}

%%%%%%%%%%%%%%%%%%%%%%%%%%%%%%%%%%%%%

The following lemma provides the Weyl quantization of symbols that are polynomials in $\xi$. 
\begin{lemma}\lab{lemma:particularcaseWeylquantizationpolynomialxi}
The Weyl quantization of a symbol $a(x, \xi)$ which is a polynomial in $\xi$ is given by
\beaa
\Opw(a)\psi(x)=\sum_{|\a|\leq k}\sum_{\ga\leq\a}2^{-|\ga|}\left(\!\!\begin{array}{c}
\a\\
\ga
\end{array}\!\!\right)(D^\ga a_\a)(x)D^{\a-\ga}\psi(x), \qquad a(x,\xi)=\sum_{|\a|\leq k}a_\a(x)\xi^\a,
\eeaa
where $D_x$ is defined by
\begin{align}
D_x := -i \partial_x.
\end{align}
In particular, we have
\beaa
\Opw (a(x))\psi =a(x)\psi(x), \qquad \Opw (a^j(x) \xi_j)\psi = a^j(x)D_{x^j}\psi(x)+ \frac{1}{2}(D_{x^j} a^j)(x)\psi(x),
\eeaa
and the Weyl quantization of the symbol $a^{jk}(x)\xi_j\xi_k$, where $a^{jk}(x)=a^{kj}(x)$, is
\beaa
\Opw (a^{jk} \xi_j \xi_k)\psi(x) =a^{jk}(x)D_{x^j}D_{x^k}\psi(x) + (D_{x^j}a^{jk})(x)D_{x^k}\psi(x) +\frac{1}{4}(D_{x^j}D_{x^k}a^{jk})(x)\psi(x).
\eeaa
\end{lemma}

\begin{proof}
We recall this classical proof for the convenience of the reader. For 
\beaa
a(x,\xi)=\sum_{|\a|\leq k}a_\a(x)\xi^\a,
\eeaa
we have
\beaa
\Opw(a)\psi(x) &=& \sum_{|\a|\leq k}(2\pi)^{-n} \int_{\Reals^n}  \int_{\Reals^n} e^{i(x-y)\cdot\xi} a_\a\left(\frac{x+y}{2}\right)\xi^\a  \psi(y)d\xi dy \\
&=& \sum_{|\a|\leq k}(2\pi)^{-n} \int_{\Reals^n}  \int_{\Reals^n} (-D_y)^\a\big(e^{i(x-y)\cdot\xi}\big) a_\a\left(\frac{x+y}{2}\right)\psi(y)d\xi dy\\
&=& \sum_{|\a|\leq k}(2\pi)^{-n} \int_{\Reals^n}  \int_{\Reals^n}e^{i(x-y)\cdot\xi}D_y^\a\left(a_\a\left(\frac{x+y}{2}\right)\psi(y)\right)d\xi dy\\
&=& \sum_{|\a|\leq k}\left(D_y^\a\left(a_\a\left(\frac{x+y}{2}\right)\psi(y)\right)\right)_{|_{y=x}}\\
&=&\sum_{|\a|\leq k}\sum_{\ga\leq\a}2^{-|\ga|}\left(\!\!\begin{array}{c}
\a\\
\ga
\end{array}\!\!\right)(D^\ga a_\a)(x)D^{\a-\ga}\psi(x), 
\eeaa
as stated.
\end{proof}

Next, we recall the properties of the Weyl quantization concerning composition and adjoint. 
\begin{proposition}
\label{prop:PDO:Rn:Weylquan}
The Weyl quantization satisfies the following properties: 
\begin{enumerate}[label=\arabic*)] 
\item For symbols $a_1$ and $a_2$ of orders $m_1$ and $m_2$, we have
\beaa
\Opw(a_1)\circ\Opw(a_2)=\Opw(a_3)
\eeaa
where the symbol $a_3$ has the following asymptotic expansion 
\bea\lab{eq:asymptoticexpansionofsymbolofproductinWeylquantization}
a_3(x,\xi)\sim\sum_{j\geq 0}\frac{1}{j!}\left(-\frac{i}{2}\right)^j(\pr_y\c\pr_\xi-\pr_x\c\pr_\eta)^j(a_1(x,\xi)a_2(y,\eta))_{|_{y=x,\eta=\xi}}.
\eea
In particular, we infer
\bea\lab{eq:propWeylquantizationmathbbRn:composition}
\bsplit
[\Opw(a_1), \Opw(a_2)]=\Opw(a_3),\quad a_3=\frac{1}{i}\{a_1, a_2\} +S^{m_1+m_2-3},\\
\Opw(a_1)\circ\Opw(a_2) +\Opw(a_2)\circ\Opw(a_1)=\Opw(a_3),\quad a_3=2a_1a_2 + S^{m_1+m_2-2},
\end{split}
\eea
where the Poisson bracket of symbols $a_1$ and $a_2$ is given by
\beaa
\{a_1, a_2\} :=\partial_{\xi} a_1\c \partial_{x} a_2 -  \partial_{x} a_1\c
 \partial_{\xi} a_2.
 \eeaa

\item For a symbol $a(x,\xi)$ of order $m$, the adjoint of its Weyl quantization is given by 
\bea\lab{eq:propWeylquantizationmathbbRn:adjoint}
(\Opw(a))^{\star} = \Opw(\bar{a}).
\eea
In particular, the Weyl quantization of a real-valued symbol is a self-adjoint operator.
\end{enumerate}
\end{proposition}

\begin{remark}\lab{rmk:Weylquantizationdoesitbetter}
Proposition \ref{prop:PDO:Rn:Weylquan} contains the properties of Weyl quantization which improve w.r.t. other quantizations such as the standard one. Indeed, for other quantizations, the remainders corresponding to \eqref{eq:propWeylquantizationmathbbRn:composition} and \eqref{eq:propWeylquantizationmathbbRn:adjoint} would be respectively only in  $S^{m_1+m_2-2}$, $S^{m_1+m_2-1}$ and $S^{m-1}$.
\end{remark}

\begin{proof}
This proposition is classical. For the convenience of the reader, we recall the proof of \eqref{eq:asymptoticexpansionofsymbolofproductinWeylquantization}. We have 
\beaa
\Opw(a_1)\circ\Opw(a_2)\psi = (2\pi)^{-2n} \int_{\Reals^{4n}} e^{i(x-y)\cdot\xi}e^{i(y-z)\cdot\eta} a_1\bigg(\frac{x+y}{2},\xi\bigg)a_2\bigg(\frac{y+z}{2},\eta\bigg)\psi(z)d\xi dy d\eta dz
\eeaa
and hence $\Opw(a_1)\circ\Opw(a_2)=\Opw(a_3)$ if 
\beaa
\int_{\Reals^n}e^{i\rho\c(x-z)}a_3\bigg(\frac{x+z}{2},\rho\bigg)d\rho = (2\pi)^{-n} \int_{\Reals^{3n}} e^{i(x-y)\cdot\xi}e^{i(y-z)\cdot\eta} a_1\bigg(\frac{x+y}{2},\xi\bigg)a_2\bigg(\frac{y+z}{2},\eta\bigg)d\xi dy d\eta 
\eeaa
or
\beaa
\int_{\Reals^n}e^{-i\rho\c z}a_3(x,\rho)d\rho = (2\pi)^{-n} \int_{\Reals^{3n}} e^{i\left(x-\frac{z}{2}-y\right)\cdot\xi}e^{i\left(y-x-\frac{z}{2}\right)\cdot\eta} a_1\bigg(\frac{x-\frac{z}{2}+y}{2},\xi\bigg)a_2\bigg(\frac{y+x+\frac{z}{2}}{2},\eta\bigg)d\xi dy d\eta 
\eeaa
which, after taking the inverse Fourier transform, holds provided $a_3(x,\xi)$ is given by
\beaa
a_3(x, \xi)=(2\pi)^{-2n} \int_{\Reals^{4n}} e^{iz\c\xi}e^{i\left(x-\frac{z}{2}-y\right)\cdot\rho}e^{i\left(y-x-\frac{z}{2}\right)\cdot\eta} a_1\bigg(\frac{x-\frac{z}{2}+y}{2},\rho\bigg)a_2\bigg(\frac{y+x+\frac{z}{2}}{2},\eta\bigg)d\rho dy d\eta dz.
\eeaa
Changing variables to $u=\frac{y}{2}-\frac{z}{4}-\frac{x}{2}$, $u'=\frac{y}{2}+\frac{z}{4}-\frac{x}{2}$, $v=\eta-\xi$ and $v'=\rho-\xi$, we infer
\bea\lab{eq:formulafora3incompositionWeylquantization}
a_3(x, \xi)=\pi^{-2n} \int_{\Reals^{4n}}e^{2i(u\c v-u'\c v')} a_1(x+u, \xi+v')a_2(x+u', \xi+v)dudu'dvdv'
\eea
and \eqref{eq:asymptoticexpansionofsymbolofproductinWeylquantization} follows from the stationary phase method. 
\end{proof}

We also recall the action of PDOs on Sobolev spaces.
\begin{lemma}
\label{lem:PDO:Sobolev}
Let $H^s=H^s(\mathbb{R}^n)$ for $s\in\mathbb{R}$ denote the standard Sobolev space on $\mathbb{R}^n$. If $a\in S^m$, then the PDO $\Opw(a)$ maps $H^s$ to $H^{s-m}$ for all $s\in \mathbb{R}$. That is, for any $\psi\in H^s$, we have
\begin{align}
\|\Opw(a)\psi\|_{H^{s-m}}\lesssim \|\psi\|_{H^s}.
\end{align}
\end{lemma}

Finally, we provide a basic (non-sharp) commutator estimate. 
\begin{lemma}\lab{lemma:bascicommutatorlemmawithelementaryproof}
Let $P=\Opw(p)$ with $p\in S^1(\Reals^n)$ and let $f$ be a scalar function. Then, we have 
\beaa
\|[P, f]\psi\|_{L^2(\Reals^n)} &\les& \|f\|_{W^{2,+\infty}(\Reals^n)}\|\psi\|_{L^2(\Reals^n)}.
\eeaa
\end{lemma}

\begin{proof}
This type of commutator lemma is classical, although usually stated in a sharper form, see for example \cite{Taylor06}. For the convenience of the reader, we provide a proof. First, we rewrite $[P, f]\psi$ as follows
\beaa
\bsplit
[P, f]\psi(x) =& \int_{\mathbb{R}^n}K_{[P, f]}(x,y)\psi(y)dy,\\
K_{[P, f]}(x,y):=& (f(x)-f(y))K_P(x,y),\qquad K_P(x,y):= \frac{1}{(2\pi)^n}\int_{\mathbb{R}^n}e^{i(x-y)\c\xi}p\left(\frac{x+y}{2}, \xi\right)d\xi, 
\end{split}
\eeaa
where $P=\Opw(p)$. Then, we decompose
\beaa
f(x)-f(y) &=& (x-y)\c\nab f(x)+(x-y)\c H(x,y), \\ 
H(x,y) &:=& \int_0^1\big(\nab f(x+s(y-x))-\nab f(x)\big)ds,
\eeaa
which yields
\beaa
\bsplit
K_{[P, f]}(x,y)=& K^{(1)}_{[P, f]}(x,y)+K^{(2)}_{[P, f]}(x,y),\\
K^{(1)}_{[P, f]}(x,y) :=& ((x-y)\c\nab f(x))K_P(x,y),\quad K^{(2)}_{[P, f]}(x,y) := ((x-y)\c H(x,y))K_P(x,y), 
\end{split}
\eeaa
and hence
\beaa
[P, f]\psi(x) &=& \sum_{j=1}^N\pr_jf(x)\Opw(i\pr_{\xi_j}p)\psi(x)+{[P, f]}^{(2)}\psi(x),\\
{[P, f]}^{(2)}\psi(x) &:=& \int_{\mathbb{R}^n}K^{(2)}_{[P, f]}(x,y)\psi(y)dy.
\eeaa
Since $\pr_{\xi_j}p\in S^0(\Reals^n)$ for all $j=1,\cdots, N$, we infer in view of Lemma \ref{lem:PDO:Sobolev}
\beaa
\|[P, f]\psi\|_{L^2(\Reals^n)} &\les& \|\pr f\|_{L^\infty(\Reals^n)}\|\psi\|_{L^2(\Reals^n)}+\|{[P, f]}^{(2)}\psi\|_{L^2(\Reals^n)}.
\eeaa

It then remains to control $\|{[P, f]}^{(2)}\psi\|_{L^2(\Reals^n)}$. To this end, we start by deriving a classical bounds for the kernel $K_P(x,y)$ of the PDO $P$. We introduce a dyadic partition of $\mathbb{R}^n$ satisfying 
\beaa
\bsplit
&\sum_{j\geq -1}\chi_j(\xi)=1\,\,\,\textrm{on}\,\,\mathbb{R}^n, \qquad \textrm{supp}(\chi_{-1})\subset\{|\xi|\leq 1\}, \qquad \textrm{supp}(\chi_j)\subset\{2^{j-1}|\xi|\leq 2^{j+1}\}\quad j\geq 0,\\
&0\leq \chi_j\leq 1, \qquad |\nab^N_\xi\chi_j(\xi)|\les_N 2^{-jN}, \quad \forall j\geq -1,
\end{split}
\eeaa
and decompose 
\beaa
K_P(x,y) &=& \sum_{j=-1}^{+\infty}K_{P,j}(x,y),\\
K_{P,j}(x,y) &:=& \frac{1}{(2\pi)^n}\int_{\mathbb{R}^n}e^{i(x-y)\c\xi}p\left(\frac{x+y}{2}, \xi\right)\chi_j(\xi)d\xi, \quad j\geq -1. 
\eeaa
We then use integration by parts to obtain, since $p\in S^1(\Reals^n)$ and in view of the properties of $\chi_j$,  
\beaa
|x-y|^{n+2}|K_{P,j}(x,y)|\les \int_{\mathbb{R}^n}\left|\nab_\xi^{n+2}\left(\chi_j(\xi)p\left(\frac{x+y}{2}, \xi\right)\right)\right|d\xi \les 2^{-j},\\
|x-y|^n|K_{P,j}(x,y)|\les \int_{\mathbb{R}^n}\left|\nab_\xi^n\left(\chi_j(\xi)p\left(\frac{x+y}{2}, \xi\right)\right)\right|d\xi \les 2^j,
\eeaa
which implies the following classical bound
\beaa
|x-y|^{n+\frac{3}{2}}|K_P(x,y)|&\les&\sum_{j\geq -1}\left(|x-y|^{n+2}|K_{P,j}(x,y)|\right)^{\frac{3}{4}}\left(|x-y|^n|K_{P,j}(x,y)|\right)^{\frac{1}{4}}\\
&\les& \sum_{j\geq -1}2^{-\frac{j}{2}}\les 1.
\eeaa
We deduce 
\beaa
|K_{[P, f]}^{(2)}(x,y)| &=& |((x-y)\c H(x,y))K_P(x,y)|\les \frac{|x-y||H(x,y)|}{|x-y|^{n+\frac{3}{2}}}\\
&\les& \frac{1}{|x-y|^{n+\frac{1}{2}}}\left|\int_0^1\big(\nab f(x+s(y-x))-\nab f(x)\big)ds\right|\\
&\les& \|f\|_{W^{2,+\infty}(\Reals^n)}\frac{\min(1,|x-y|)}{|x-y|^{n+\frac{1}{2}}}
\eeaa
which immediately yields
\beaa
\sup_{x\in\Reals^n}\left(\int_{\Reals^n}|K_{[P, f]}^{(2)}(x,y)|dy\right)\les \|f\|_{W^{2,+\infty}(\Reals^n)}, \qquad \sup_{y\in\Reals^n}\left(\int_{\Reals^n}|K_{[P, f]}^{(2)}(x,y)|dx\right)\les \|f\|_{W^{2,+\infty}(\Reals^n)}.
\eeaa
Together with Schur's lemma, we infer
\beaa
\|{[P, f]}^{(2)}\psi\|_{L^2(\Reals^n)}\les \|f\|_{W^{2,+\infty}(\Reals^n)}\|\psi\|_{L^2(\Reals^n)}
\eeaa
and hence
\beaa
\|[P, f]\psi\|_{L^2(\Reals^n)} &\les& \|\pr f\|_{L^\infty(\Reals^n)}\|\psi\|_{L^2(\Reals^n)}+\|{[P, f]}^{(2)}\psi\|_{L^2(\Reals^n)}\\
&\les& \|f\|_{W^{2,+\infty}(\Reals^n)}\|\psi\|_{L^2(\Reals^n)}
\eeaa
as stated. This concludes the proof of Lemma \ref{lemma:bascicommutatorlemmawithelementaryproof}.
\end{proof}

%%%%%%%%%%%%%%%%%%%%%%%%%%%%%%%%%%%%%

\subsubsection{Change of coordinates and Weyl quantization}

%%%%%%%%%%%%%%%%%%%%%%%%%%%%%%%%%%%%%

In order to establish good composition rules concerning PDOs on manifolds in the Weyl quantization, we will rely on  the following property of the Weyl quantization on $\mathbb{R}^n$ with respect to change of variables.
\begin{lemma}[Change of variables on $\mathbb{R}^n$ for the Weyl quantization]
\lab{lemma:compositionPDOWeylanddiffeoconservative}
For a symbol $a\in S^m(\mathbb{R}^n)$ and a diffeomorphism $\varphi$ of $\mathbb{R}^n$ such that $|\det(d\varphi)|=1$, there exists a symbol $b$ such that 
\beaa
\varphi^\#\Opw(a)[(\varphi^{-1})^\#\psi]=\Opw(b)\psi,
\eeaa 
where ${}^\#$ denotes the pullback of a map and where
\beaa
b(x, \xi)=a(\varphi(x), d\varphi_x^{-1}(\xi))+S^{m-2}.
\eeaa
\end{lemma}

\begin{remark}\lab{rmk:Weylquantizationdoesitbetter:bis}
The property of the Weyl quantization in Lemma \ref{lemma:compositionPDOWeylanddiffeoconservative} improves w.r.t. other quantizations such as the standard one for which the remainder would only be in $S^{m-1}$. Note also that the assumption $|\det(d\varphi)|=1$ will be satisfied  thanks to our choices of isochore coordinates on $\MM$, see Section \ref{sec:isochorecoordinatesonHr}. 
\end{remark}

\begin{proof}
For the convenience of the reader, we provide below the proof of Lemma \ref{lemma:compositionPDOWeylanddiffeoconservative}. We have
\beaa
\varphi^\#\Opw(a)[(\varphi^{-1})^\#\psi](x) &=& \frac{1}{(2\pi)^n}\int_{\mathbb{R}^n}\int_{\mathbb{R}^n}e^{i(\varphi(x)-y)\c\xi}a\left(\frac{\varphi(x)+y}{2}, \xi\right)\psi\circ\varphi^{-1}(y)dy d\xi\\
&=& \frac{1}{(2\pi)^n}\int_{\mathbb{R}^n}\int_{\mathbb{R}^n}e^{i(\varphi(x)-\varphi(y))\c\xi}a\left(\frac{\varphi(x)+\varphi(y)}{2}, \xi\right)\psi(y)dy d\xi,
\eeaa
where we used the fact that $|\textrm{det}(d\varphi)|=1$. Now, we look for $b(x,\xi)$ such that for all $(x, y)\in \mathbb{R}^n\times \mathbb{R}^n$, we have
\beaa
\int_{\mathbb{R}^n}e^{i(x-y)\c\xi}b\left(\frac{x+y}{2}, \xi\right)d\xi = \int_{\mathbb{R}^n}e^{i(\varphi(x)-\varphi(y))\c\xi}a\left(\frac{\varphi(x)+\varphi(y)}{2}, \xi\right)d\xi
\eeaa
which, setting
\beaa
z=\frac{x+y}{2}, \qquad u=x-y, \qquad x=z+\frac{u}{2}, \qquad y=z-\frac{u}{2},
\eeaa
is equivalent to  
\beaa
\int_{\mathbb{R}^n}e^{iu\c\xi}b(z, \xi)d\xi = \int_{\mathbb{R}^n}e^{i\left(\varphi\left(z+\frac{u}{2}\right)-\varphi\left(z-\frac{u}{2}\right)\right)\c\xi}a\left(\frac{\varphi\left(z+\frac{u}{2}\right)+\varphi\left(z-\frac{u}{2}\right)}{2}, \xi\right)d\xi,
\eeaa
and hence, for all $(x, \xi)\in \mathbb{R}^n\times \mathbb{R}^n$, we infer
\beaa
b(x, \xi)=\frac{1}{(2\pi)^n}\int_{\mathbb{R}^n}\int_{\mathbb{R}^n}e^{-iu\c\xi+i\left(\varphi\left(x+\frac{u}{2}\right)-\varphi\left(x-\frac{u}{2}\right)\right)\c\eta}a\left(\frac{\varphi\left(x+\frac{u}{2}\right)+\varphi\left(x-\frac{u}{2}\right)}{2}, \eta\right)d\eta du
\eeaa
or
\beaa
\bsplit
b(x, \xi) =& \frac{1}{(2\pi)^n}\int_{\mathbb{R}^n}\int_{\mathbb{R}^n}e^{i\Phi(x, \xi, u, \eta)}c(x, u, \eta)d\eta du, \\ 
\Phi(x, \xi, u, \eta) :=& -u\c\xi+\left(\varphi\left(x+\frac{u}{2}\right)-\varphi\left(x-\frac{u}{2}\right)\right)\c\eta,\\
c(x, u, \eta) :=& a\left(\frac{\varphi\left(x+\frac{u}{2}\right)+\varphi\left(x-\frac{u}{2}\right)}{2}, \eta\right),
\end{split}
\eeaa
which proves the existence of $b$. 

Next, since 
\beaa
\pr_\eta\Phi=\varphi\left(x+\frac{u}{2}\right)-\varphi\left(x-\frac{u}{2}\right), \qquad \pr_u\Phi=-\xi+\frac{1}{2}\left(d\varphi_{x+\frac{u}{2}}(\eta) + d\varphi_{x-\frac{u}{2}}(\eta)\right), 
\eeaa
the phase $\Phi$ has a unique stationary point at 
\beaa
u=0, \qquad \eta=(d\varphi_x)^{-1}(\xi). 
\eeaa
We write 
\beaa
\bsplit
\Phi(x, \xi, u, \eta)=& -u\c\xi+d\vphi_x(u)\c\eta+\widetilde{\Phi}(x, u)\c\eta, \\
\widetilde{\Phi}(x, u):=& \varphi\left(x+\frac{u}{2}\right)-\varphi\left(x-\frac{u}{2}\right)-d\vphi_x(u), \qquad \widetilde{\Phi}(x, u)=O(u^3),
\end{split}
\eeaa
and then, setting $v=(d\varphi_x)^T(\eta) -\xi$, we obtain 
\beaa
\bsplit
b(x, \xi) =& \frac{1}{(2\pi)^n}\int_{\mathbb{R}^n}\int_{\mathbb{R}^n}e^{iu\c v}d(x,\xi,u,v)dv du, \\ 
d(x,\xi,u,v) :=& e^{i \widetilde{\Phi}(x, u)\c(d\varphi_x^{-1})^T(v+\xi)} c\Big(x, u, (d\varphi_x^{-1})^T(v+\xi)\Big)\\
=&e^{i\left(\varphi\left(x+\frac{u}{2}\right)-\varphi\left(x-\frac{u}{2}\right)-d\vphi_x(u)\right)\c(d\varphi_x^{-1})^T(v+\xi)}a\left(\frac{\varphi\left(x+\frac{u}{2}\right)+\varphi\left(x-\frac{u}{2}\right)}{2}, (d\varphi_x^{-1})^T(v+\xi)\right).
\end{split}
\eeaa
The stationary phase method implies immediately  
\beaa
b(x, \xi) \sim \sum_{j\geq 0}\frac{1}{j!}(\pr_u\c\pr_v)^jd(x,\xi,u,v)_{|_{(u,v)=(0,0)}}.
\eeaa
Since we have
\beaa
d(x,\xi,0,0)=a(\varphi(x), d\varphi_x^{-1}(\xi)), \qquad \pr_u\c\pr_v d(x,\xi,0,0)=0, 
\eeaa
the above asymptotic expansion implies  
\beaa
b(x, \xi)=a(\varphi(x), d\varphi_x^{-1}(\xi))+S^{m-2},
\eeaa
which concludes the proof of Lemma \ref{lemma:compositionPDOWeylanddiffeoconservative}.
\end{proof}

%%%%%%%%%%%%%%%%%%%%%%%%%%%%%%%%

\subsubsection{Mixed symbols and their Weyl quantization}
\lab{sec:mixedsymbolsonRn}

%%%%%%%%%%%%%%%%%%%%%%%%%%%%%%%%

In view of our latter applications to microlocal energy-Morawetz estimates, we decompose $x=(x', x^n)\in\mathbb{R}^{n-1}\times\mathbb{R}$ and consider mixed operators which are PDO in $x'$ and differential in $x^n$. We first define $x^n$-tangential symbols on $\mathbb{R}^n$.
\begin{definition}[$x^n$-tangential symbols on $\mathbb{R}^n$]
 \label{def:symbols:Rn:rtangent}
For $m\in \mathbb{R}$, let $S^m_{tan}(\mathbb{R}^n)$ denote the set of functions $a$ which are $C^{\infty}(\mathbb{R}^n\times\mathbb{R}^{n-1})$ such that for all multi-indices $\alpha$, $\beta$,
 \beaa
\forall x=(x', x^n)\in\mathbb{R}^n, \,\, \forall \xi\in\mathbb{R}^{n-1},\quad \abs{\partial_x^{\alpha} \partial_\xi^{\beta}a(x,\xi)}\leq C_{\alpha,\beta} \langle \xi\rangle^{m-\abs{\beta}},
\eeaa
with $C_{\alpha,\beta}<+\infty$. An element $a\in S^m_{tan}(\mathbb{R}^n)$ is called an $x^n$-tangential symbol of order $m$. 
\end{definition}

Next, we introduce a class of mixed symbols on $\mathbb{R}^n$.
\begin{definition}[Mixed symbols on $\mathbb{R}^n$]
\label{PDO:Rn:Shom}
For $m\in\mathbb{R}$ and $N\in\mathbb{N}$, we define the class $\widetilde{S}^{m,N}(\mathbb{R}^n)$ of symbols as $a\in C^{\infty}(\mathbb{R}^n\times\mathbb{R}^n)$ such that for all $(x, \xi)\in\mathbb{R}^n\times\mathbb{R}^n$, we have, for $\xi=(\xi', \xi_n)$,  
\beaa
a(x,\xi)=\sum_{j=0}^N v_{m-j}(x,\xi')(\xi_n)^j,  \qquad v_{m-j}\in S^{m-j}_{tan}(\mathbb{R}^n).
\eeaa
An element $a\in\widetilde{S}^{m,N}(\mathbb{R}^n)$ is called a mixed symbol of order $(m, N)$. 
\end{definition}

\begin{remark}
Notice that $S^m_{tan}(\Reals^n)=\widetilde{S}^{m,0}(\Reals^n)$. 
\end{remark}

The following lemma provides a formula for the Weyl quantization of symbols in $\widetilde{S}^{m,N}(\mathbb{R}^n)$.
\begin{lemma}\lab{lemma:WeylquantizationmidexsymbolRn}
Let $m\in\mathbb{R}$, $N\in\mathbb{N}$, and $a\in\widetilde{S}^{m,N}(\mathbb{R}^n)$ with
\beaa
a(x,\xi)=\sum_{j=0}^N v_{m-j}(x,\xi')(\xi_n)^j,  \qquad v_{m-j}\in S^{m-j}_{tan}(\mathbb{R}^n).
\eeaa
Then, the Weyl quantization of $a$ is given by
\beaa
\Opw(a)=\sum_{j=0}^N\sum_{k=0}^j2^{-k}\left(\!\!\begin{array}{c}
j\\
k
\end{array}\!\!\right)\Opw(D_{x_n}^kv_{m-j})D_{x^n}^{j-k},
\eeaa
where $\Opw(D_{x_n}^kv_{m-j})$ is the Weyl quantization in $\mathbb{R}^{n-1}$ of the $x^n$-tangential symbol $D_{x_n}^kv_{m-j}$. 
\end{lemma}

\begin{proof}
The proof follows along the same lines as the one of Lemma \ref{lemma:particularcaseWeylquantizationpolynomialxi}.
\end{proof}

\begin{remark}\lab{rmk:WeylquantizationofmixedsymbolsonRnlocalinxn}
Lemma \ref{lemma:WeylquantizationmidexsymbolRn} shows that the Weyl quantization of mixed symbols is pseudo-differential w.r.t $x'$ but differential w.r.t. $x^n$ so that it can be applied to functions that are defined on $\mathbb{R}^{n-1}\times I$ where $I$ is an open set of $\mathbb{R}$. 
\end{remark}

We now consider the properties of the Weyl quantization of symbols in $\widetilde{S}^{m,N}(\mathbb{R}^n)$. 
\begin{proposition}
\label{prop:PDO:Rn:Weylquan:mixedoperators}
The Weyl quantization satisfies the following properties for symbols in the class 
$\widetilde{S}^{m,N}(\mathbb{R}^n)$: 
\begin{enumerate}[label=\arabic*)] 
\item For mixed symbols $a_1$ and $a_2$ of respective orders $(m_1,N_1)$ and $(m_2,N_2)$, we have
\bea\lab{eq:propWeylquantizationmathbbRn:composition:mixedsymbols}
\bsplit
[\Opw(a_1), \Opw(a_2)]=\Opw(a_3),\quad a_3=\frac{1}{i}\{a_1, a_2\} +\widetilde{S}^{m_1+m_2-3,N_1+N_2},\\
\Opw(a_1)\circ\Opw(a_2) +\Opw(a_2)\circ\Opw(a_1)=\Opw(a_3),\quad a_3=2a_1a_2 + \widetilde{S}^{m_1+m_2-2,N_1+N_2}.
\end{split}
\eea

\item In the particular case where $a_1(x,\xi)=v_1(x^n)\xi_n^{N_1}$, which is a mixed symbol of order $(m_1, N_1)$ with $m_1=N_1$, we have, with $a_2$ of order $(m_2,N_2)$
\bea\lab{eq:propWeylquantizationmathbbRn:composition:mixedsymbols:specialcase}
\bsplit
&[\Opw(a_1), \Opw(a_2)]=\Opw(a_3),\quad a_3=\frac{1}{i}\{a_1, a_2\} +\tilde{a}_3,\\
&\tilde{a}_3=0\quad\textrm{if}\quad \max(N_1, N_2)\leq 2, \qquad \tilde{a}_3\in\widetilde{S}^{m_1+m_2-3,N_1+N_2-3}\quad\textrm{if}\quad \max(N_1, N_2)\geq 3,\\
&\Opw(a_1)\circ\Opw(a_2) +\Opw(a_2)\circ\Opw(a_1)=\Opw(a_4),\quad a_4=2a_1a_2 + \tilde{a}_4,\\
&\tilde{a}_4=0\quad\textrm{if}\quad \max(N_1, N_2)\leq 1, \qquad \tilde{a}_4\in\widetilde{S}^{m_1+m_2-2,N_1+N_2-2}\quad\textrm{if}\quad \max(N_1, N_2)\geq 2.
\end{split}
\eea

\item In the particular case where $a_1$ and $a_2$ are mixed symbols of respective order $(m_1, 1)$ and $(m_2, 1)$, and $f=f(x^n)$, we have
\bea\lab{eq:propWeylquantizationmathbbRn:composition:mixedsymbols:specialcase:1}
\bsplit
&[\Opw(a_1), \Opw(f(x^n)a_2)]=\Opw(a_3),\quad a_3=\tilde{a}_3,\\
&\tilde{a}_3=\frac{1}{i}\{a_1, f(x^n)a_2\} +f(x^n)\widetilde{S}^{m_1+m_2-3,2}+\widetilde{S}^{m_1+m_2-3,1},\\
&\Opw(a_1)\circ\Opw(f(x^n)a_2) +\Opw(f(x^n)a_2)\circ\Opw(a_1)=\Opw(a_4),\quad a_4=2f(x^n)a_1a_2 + \tilde{a}_4,\\
&\tilde{a}_4=f(x^n)\widetilde{S}^{m_1+m_2-2,2}+\widetilde{S}^{m_1+m_2-2,1}.
\end{split}
\eea

\item For a mixed symbol $a(x,\xi)$, the adjoint of its Weyl quantization is given by 
\bea\lab{eq:propWeylquantizationmathbbRn:adjoint:mixedsymbols}
(\Opw(a))^{\star} = \Opw(\bar{a}).
\eea
In particular, the Weyl quantization of a real-valued symbol is a self-adjoint operator.
\end{enumerate}
\end{proposition}

\begin{proof}
The fourth property is immediate, so we focus on the first three properties. Recall from the proof of Proposition \ref{prop:PDO:Rn:Weylquan} that 
\beaa
\Opw(a_1)\circ\Opw(a_2)=\Opw(a_3)
\eeaa
where $a_3$ is given by \eqref{eq:formulafora3incompositionWeylquantization}. Using the stationary phase method, this yields the asymptotic expansion \eqref{eq:asymptoticexpansionofsymbolofproductinWeylquantization}, i.e., 
\beaa
a_3(x,\xi)\sim\sum_{j\geq 0}\left(-\frac{i}{2}\right)^j(\pr_y\c\pr_\xi-\pr_x\c\pr_\eta)^j(a_1(x,\xi)a_2(y,\eta))_{|_{y=x,\eta=\xi}}.
\eeaa
Now, if $a_1$ and $a_2$ are mixed symbols of respective orders $(m_1,N_1)$ and $(m_2,N_2)$, then, for $j\geq 0$, 
\beaa
(\pr_y\c\pr_\xi-\pr_x\c\pr_\eta)^j(a_1(x,\xi)a_2(y,\eta))_{|_{y=x,\eta=\xi}}
\eeaa
is a mixed symbol of orders $(m_1+m_2-j,N_1+N_2)$ which yields the first property.

Next, in the particular case where $a_1(x,\xi)=v_1(x^n)\xi_n^{N_1}$ and $a_2$ is a mixed symbol of order $(m_2,N_2)$, we have, for $j\geq 0$, 
\beaa
(\pr_y\c\pr_\xi-\pr_x\c\pr_\eta)^j(a_1(x,\xi)a_2(y,\eta))_{|_{y=x,\eta=\xi}} &=& (\pr_y\c\pr_\xi-\pr_x\c\pr_\eta)^j(v_1(x^n)\xi_n^{N_1}a_2(y,\eta))_{|_{y=x,\eta=\xi}}\\
&=& (\pr_{y^n}\pr_{\xi^n}-\pr_{x^n}\pr_{\eta^n})^j(v_1(x^n)\xi_n^{N_1}a_2(y,\eta))_{|_{y=x,\eta=\xi}}
\eeaa
so that
\beaa
&&(\pr_y\c\pr_\xi-\pr_x\c\pr_\eta)^j(a_1(x,\xi)a_2(y,\eta))_{|_{y=x,\eta=\xi}}=0\quad\textrm{for}\quad j>\max(N_1, N_2)\\
&&(\pr_y\c\pr_\xi-\pr_x\c\pr_\eta)^j(a_1(x,\xi)a_2(y,\eta))_{|_{y=x,\eta=\xi}}\in \widetilde{S}^{m_1+m_2-j,N_1+N_2-j}\quad\textrm{for}\quad j\leq\max(N_1, N_2)
\eeaa
which implies the second property. 

Finally, in the particular case where $a_1$ and $a_2$ are mixed symbols of respective order $(m_1, 1)$ and $(m_2, 1)$, and $f=f(x^n)$, we have
\beaa
&&(\pr_y\c\pr_\xi-\pr_x\c\pr_\eta)^j(a_1(x,\xi)f(x^n)a_2(y,\eta))_{|_{y=x,\eta=\xi}}\\ 
&=& f(x^n)(\pr_{y'}\c\pr_{\xi'}-\pr_{x'}\c\pr_{\eta'})^j(a_1(x,\xi)a_2(y,\eta))_{|_{y=x,\eta=\xi}}+\widetilde{S}^{m_1+m_2-j,1}\\
&=& f(x^n)\widetilde{S}^{m_1+m_2-j,2}+\widetilde{S}^{m_1+m_2-j,1}
\eeaa
which implies the third property. This concludes the proof of Proposition \ref{prop:PDO:Rn:Weylquan:mixedoperators}.
\end{proof}

We also consider the action of the Weyl quantization of mixed symbols on Sobolev spaces. 
\begin{lemma}\lab{lemma:actionmixedsymbolsSobolevspaces}
Let $m\in\mathbb{R}$, $N\in\mathbb{N}$, let $I$ be an interval of $\mathbb{R}$, and let $a\in\widetilde{S}^{m, N}(\mathbb{R}^n)$ be of the form $a(x,\xi)=v_{m-N}(x,\xi')\xi^N_n$ with $v_{m-N}\in S_{tan}^{m-N}(\mathbb{R}^n)$. Then, we have for all $s\in\mathbb{R}$
\beaa
\|\Opw(a)\psi\|_{L^2_{x^n}(I, H_{x'}^{s-m+N}(\mathbb{R}^{n-1}))}\les \sum_{j=0}^N\|\pr_{x^n}^j\psi\|_{L^2_{x^n}(I, H_{x'}^s(\mathbb{R}^{n-1}))}.
\eeaa
\end{lemma}

\begin{proof}
In view of Lemma \ref{lemma:WeylquantizationmidexsymbolRn}, we have
\beaa
\Opw(a)=\sum_{k=0}^N2^{-k}\left(\!\!\begin{array}{c}
N\\
k
\end{array}\!\!\right)\Opw(D_{x_n}^kv_{m-N})D_{x^n}^{N-k}
\eeaa
and the lemma follows immediately from that fact that $D_{x_n}^kv_{m-N}(x^x,\cdot, \cdot)\in S^{m-N}(\mathbb{R}^{n-1})$ together with Lemma \ref{lem:PDO:Sobolev}.
\end{proof}

Next, we prove a non-sharp G\aa rding type inequality for tangential symbols.
\begin{lemma}\lab{lemma:basicgardingfortangentialsymbols}
 Let $I$ be an interval of $\mathbb{R}$ and let $a\in\widetilde{S}^{1,0}$ be such that
\beaa
a(x, \xi)\geq c_1\langle\xi'\rangle, \qquad \forall (x,\xi)\in\mathbb{R}^{2n}, \quad \xi=(\xi', \xi_n),
\eeaa
for some constant $c_1>0$. Then, there exist constants $c_2>0$ and $c_3>0$ such that we have
\beaa
\|\Opw(a)\psi\|_{L^2_{x^n}(I, L^2_{x'}(\mathbb{R}^{n-1}))}^2\geq c_2\|\psi\|_{L^2_{x^n}(I, H_{x'}^1(\mathbb{R}^{n-1}))}^2 -c_3\|\psi\|_{L^2_{x^n}(I, L^2_{x'}(\mathbb{R}^{n-1}))}^2.
\eeaa
\end{lemma}

\begin{proof}
Since 
\beaa
(a(x, \xi'))^2-\frac{c_1^2}{2}\langle\xi'\rangle^2\geq \frac{c_1^2}{2}\langle\xi'\rangle^2,
\eeaa
we may rewrite $a$ as
\beaa
a(x, \xi')^2=\frac{c_1^2}{2}\langle\xi'\rangle^2+e(x,\xi')^2, \qquad e(x,\xi'):=\sqrt{(a(x, \xi'))^2-\frac{c_1^2}{2}\langle\xi'\rangle^2}, \qquad e\in\widetilde{S}^{1,0}(\mathbb{R}^n).
\eeaa
Together with Proposition \ref{prop:PDO:Rn:Weylquan:mixedoperators}, this implies 
\beaa
\Opw(a)^2=\frac{c_1^2}{2}\Opw(\langle\xi'\rangle)^2+\Opw(e)^2+\Opw(\widetilde{S}^{0,0}(\mathbb{R}^n))
\eeaa
which, together with Lemma \ref{lemma:actionmixedsymbolsSobolevspaces} applied to $\Opw(\widetilde{S}^{0,0}(\mathbb{R}^n))$, implies
\beaa
\|\Opw(a)\psi\|_{L^2_{x^n}(I, L^2_{x'}(\mathbb{R}^{n-1}))}^2\geq \frac{c_1^2}{2}\|\psi\|_{L^2_{x^n}(I, H_{x'}^1(\mathbb{R}^{n-1}))}^2 - c_3\|\psi\|_{L^2_{x^n}(I, L^2_{x'}(\mathbb{R}^{n-1}))}^2
\eeaa
as stated.
\end{proof}

Finally, we provide a basic commutator estimate for mixed symbols. 
\begin{lemma}\lab{lemma:bascicommutatorlemmawithelementaryproof:mixedsymbols}
 Let $I$ be an interval of $\mathbb{R}$, let $P=\Opw(p)$ with $p\in\widetilde{S}^{1,1}(\Reals^n)$ and let $f$ be a scalar function. Then, we have 
\beaa
\|[P, f]\psi\|_{L^2_{x^n}(I,H^1_{x'}(\Reals^{n-1}))} &\les& \|f\|_{W^{2,+\infty}(I\times\Reals^{n-1})}\|\pr^{\leq 1}\psi\|_{L^2_{x^n}(I,L^2_{x'}(\Reals^{n-1}))}.
\eeaa
\end{lemma}

\begin{proof}
As $P=\Opw(p)$ with $p\in\widetilde{S}^{1,1}(\Reals^n)$, we may decompose $P$ as  
\beaa
P=P_0\pr_{x^n}+P_1+\Opw(\widetilde{S}^{0,0}(\Reals^n)), \quad P_0=\Opw(p_0), \quad P_1=\Opw(p_1), \quad p_j\in\widetilde{S}^{j,0}(\Reals^n), \,\, j=0,1,
\eeaa
which together with Lemma \ref{lemma:actionmixedsymbolsSobolevspaces} yields
\beaa
\|[P, f]\psi\|_{L^2_{x^n}(I,H^1_{x'}(\Reals^{n-1}))} &\les& \|[P_0, f]\pr_{x^n}\psi\|_{L^2_{x^n}(I,H^1_{x'}(\Reals^{n-1}))}+\|P_0((\pr_{x^n}f)\psi)\|_{L^2_{x^n}(I,H^1_{x'}(\Reals^{n-1}))}\\
&&+\|[P_1, f]\psi\|_{L^2_{x^n}(I,H^1_{x'}(\Reals^{n-1}))}+\|f\Opw(\widetilde{S}^{0,0}(\Reals^n))\psi\|_{L^2_{x^n}(I,H^1_{x'}(\Reals^{n-1}))}\\
&&+\|\Opw(\widetilde{S}^{0,0}(\Reals^n))(f\psi)\|_{L^2_{x^n}(I,H^1_{x'}(\Reals^{n-1}))}\\
&\les&  \|[P_0, f]\pr_{x^n}\psi\|_{L^2_{x^n}(I,H^1_{x'}(\Reals^{n-1}))}+\|[P_1, f]\psi\|_{L^2_{x^n}(I,H^1_{x'}(\Reals^{n-1}))}\\
&&+\|f\|_{W^{2,+\infty}(I\times\Reals^{n-1})}\|\psi\|_{L^2_{x^n}(I,H^1_{x'}(\Reals^{n-1}))}.
\eeaa
Also, we have
\beaa
&&\|[P_0, f]\pr_{x^n}\psi\|_{L^2_{x^n}(I,H^1_{x'}(\Reals^{n-1}))}+\|[P_1, f]\psi\|_{L^2_{x^n}(I,H^1_{x'}(\Reals^{n-1}))}\\
&\les& \sum_{j=1}^{n-1}\Big(\|\pr_{x^j}[P_0, f]\pr_{x^n}\psi\|_{L^2_{x^n}(I,L^2_{x'}(\Reals^{n-1}))}+\|\pr_{x^j}[P_1, f]\psi\|_{L^2_{x^n}(I,L^2_{x'}(\Reals^{n-1}))}\Big)\\
&\les& \sum_{j=1}^{n-1}\Big(\|[\Opw(ip_0\xi_j), f]\pr_{x^n}\psi\|_{L^2_{x^n}(I,L^2_{x'}(\Reals^{n-1}))}+\|[P_1, f]\pr_{x^j}\psi\|_{L^2_{x^n}(I,L^2_{x'}(\Reals^{n-1}))}\Big)\\
&&+\|f\Opw(\widetilde{S}^{1,1}(\Reals^n))\psi\|_{L^2_{x^n}(I,L^2_{x'}(\Reals^{n-1}))}+\|\Opw(\widetilde{S}^{1,0}(\Reals^n))(f\psi)\|_{L^2_{x^n}(I,L^2_{x'}(\Reals^{n-1}))}\\
&&+\|(\pr f)\Opw(\widetilde{S}^{1,0}(\Reals^n))\psi\|_{L^2_{x^n}(I,L^2_{x'}(\Reals^{n-1}))}+\|\Opw(\widetilde{S}^{1,0}(\Reals^n))((\pr f)\psi)\|_{L^2_{x^n}(I,L^2_{x'}(\Reals^{n-1}))}\\
&\les& \sum_{j=1}^{n-1}\Big(\|[\Opw(ip_0\xi_j), f]\pr_{x^n}\psi\|_{L^2_{x^n}(I,L^2_{x'}(\Reals^{n-1}))}+\|[P_1, f]\pr_{x^j}\psi\|_{L^2_{x^n}(I,L^2_{x'}(\Reals^{n-1}))}\Big)\\
&&+\|f\|_{W^{2,+\infty}(I\times\Reals^{n-1})}\|\pr^{\leq 1}\psi\|_{L^2_{x^n}(I,L^2_{x'}(\Reals^{n-1}))},
\eeaa
where we used repeatedly Lemma \ref{lemma:actionmixedsymbolsSobolevspaces}. We deduce 
\beaa
\|[P, f]\psi\|_{L^2_{x^n}(I,H^1_{x'}(\Reals^{n-1}))} \les \|[\widetilde{P}, f]\pr\psi\|_{L^2_{x^n}(I,L^2_{x'}(\Reals^{n-1}))}+\|f\|_{W^{2,+\infty}(I\times\Reals^{n-1})}\|\pr^{\leq 1}\psi\|_{L^2_{x^n}(I,L^2_{x'}(\Reals^{n-1}))},
\eeaa
where $\widetilde{P}=\Opw(\widetilde{p})$, $\widetilde{p}=(ip_0\xi_1,\cdots, ip_0\xi_{n-1}, p_1)\in(\widetilde{S}^{1,0}(\Reals^n))^n$. Applying Lemma \ref{lemma:bascicommutatorlemmawithelementaryproof} with $P\to \widetilde{P}$ and $\psi\to\pr\psi$, we infer
\beaa
\|[P, f]\psi\|_{L^2_{x^n}(I,H^1_{x'}(\Reals^{n-1}))} &\les& \|f\|_{W^{2,+\infty}(I\times\Reals^{n-1})}\|\pr^{\leq 1}\psi\|_{L^2_{x^n}(I,L^2_{x'}(\Reals^{n-1}))}
\eeaa
which concludes the proof of Lemma \ref{lemma:bascicommutatorlemmawithelementaryproof:mixedsymbols}.
\end{proof}

%%%%%%%%%%%%%%%%%%%%%%%%%%%%%%%

\subsection{PDOs adapted to the $r$-foliation of $\MM$}
\label{sect:PDO:rfoliation}

%%%%%%%%%%%%%%%%%%%%%%%%%%%%%%%

In this section, we extend the Weyl quantization for mixed symbols introduced in Section \ref{sec:mixedsymbolsonRn} to the case of mixed symbols adapted to the level sets $H_r$ of $r$ in $\MM$ denoted by
\beaa
H_{r_1}:=\MM\cap\{r=r_1\}, \qquad \forall r_1\geq r_+(1-\dhor),
\eeaa
with $\MM$ covered by $(\tau, r, x^1, x^2)$ coordinates. Rather than working with half densities, which is the standard approach to define a Weyl quantization on a manifold, see for example Chapter 14 in \cite{Zwo}, we will instead adapt the approach in \cite{Bon} relying on isochore coordinates. To this end, we first introduce isochore coordinates on $\mathbb{S}^2$.

%%%%%%%%%%%%%%%%%%%%%%%%%%%%%%%

\subsubsection{Isochore coordinates on $\mathbb{S}^2$}

%%%%%%%%%%%%%%%%%%%%%%%%%%%%%%%

In order to define a Weyl quantization on $\MM$ that preserves  the good properties in terms of adjoint and composition of the Weyl quantization on $\mathbb{R}^n$, we first need  to introduce local coordinates on $\mathbb{S}^2$ for which the corresponding density is the one of the Lebesgue measure. This is done in the following lemma. 
\begin{lemma}\lab{lemma:isochorecoordinates}
Let the coordinates  $(x^1_0, x^2_0)$ and $(x^1_p, x^2_p)$ be defined by  
\beaa
x_0^1=\cos\th, \qquad x_0^2=\tphi, \qquad x^1_p=\sin\th\cos\tphi, \qquad x^2_p=\arcsin\left(\frac{\sin\th\sin\tphi}{\sqrt{1-(\sin\th)^2(\cos\tphi)^2}}\right).
\eeaa
 with the corresponding coordinates patches 
\beaa
\mathbb{S}^2=\mathring{U}_0\cup\mathring{U}_p, \quad \mathring{U}_0:=\left\{(x_0^1, x_0^2)\,\,/\,\,\frac{\pi}{4}<\th<\frac{3\pi}{4}\right\}, \quad \mathring{U}_p:=\left\{(x^1_p, x^2_p)\,\,/\,\,\th\in [0,\pi]\setminus\left[\frac{\pi}{3}, \frac{2\pi}{3}\right]\right\}.
\eeaa
Then, the measure of the unit round 2-sphere in these coordinates has the density of the Lebesgue measure, i.e., 
\beaa
d\mathring{\ga}=dx^1_0 dx^2_0\quad\textrm{on}\quad U_0, \qquad d\mathring{\ga}=dx^1_p dx^2_p\quad\textrm{on}\quad U_p.
\eeaa
\end{lemma}

\begin{remark}
While {$(x^1_q, x^2_q)$}, $q=0,p$, referred until now to the choice in Lemma \ref{lem:specificchoice:normalizedcoord}, from now on, $(x^1_q, x^2_q)$, $q=0,p$,  will always refer to the choice in Lemma \ref{lemma:isochorecoordinates}. Also, the notation $(x^1, x^2)$ will be used to denote either $(x^1_0, x^2_0)$ or $(x^1_p, x^2_p)$.
\end{remark}

\begin{remark}
Coordinates systems such as the ones in Lemma \ref{lemma:isochorecoordinates} are called isochore. Local isochore coordinates exist on a Riemannian manifold as a consequence of Moser's trick, see \cite{Moser}.  
\end{remark}

\begin{proof}
We start with $(x^1_0, x^2_0)$. We have
\beaa
d\th=-\frac{dx^1_0}{\sin\th}, \qquad d\tphi=dx^2_0,
\eeaa 
and hence 
\beaa
\mathring{\ga}=(\sin\th)^{-2}(dx^1_0)^2+(\sin\th)^2(dx^2_0)^2.
\eeaa
In particular, we have
\beaa
d\mathring{\ga}=dx^1_0 dx^2_0
\eeaa
as desired. 

Next, we focus on $(x^1_p, x^2_p)$. To this end, we first introduce $(y^1, y^2)$ given by 
\beaa
y^1=\sin\th\cos\tphi, \qquad y^2=\sin\th\sin\tphi. 
\eeaa
Then, we have
\beaa
dy^1=\cot\th y^1 d\th - y^2 d\tphi, \qquad  dy^2=\cot\th y^2 d\th + y^1 d\tphi.
\eeaa
Hence 
\beaa
d\th= \frac{\sin\th}{|y|^2\cos\th}(y^1dy^1+y^2dy^2), \qquad d\tphi= \frac{1}{|y|^2}(-y^2dy^1+y^1dy^2),
\eeaa
which yields 
\beaa
\mathring{\ga} &=& \left(\frac{\sin\th}{|y|^2\cos\th}(y^1dy^1+y^2dy^2)\right)^2+(\sin\th)^2\left(\frac{1}{|y|^2}(-y^2dy^1+y^1dy^2)\right)^2\\
&=& \frac{1-(y^2)^2}{1-|y|^2}(dy^1)^2+\frac{2y^1y^2}{1-|y|^2}dy^1dy^2+\frac{1-(y^1)^2}{1-|y|^2}(dy^2)^2.
\eeaa

Next, we notice that 
\beaa
x^1_p=y^1, \qquad  x^2_p=\arcsin\left(\frac{y^2}{\sqrt{1-(y^1)^2}}\right)
\eeaa
which yields
\beaa
dx^1_p=dy^1_p, \qquad dx^2=\frac{y^1y^2}{(1-(y^1)^2)\sqrt{1-|y|^2}}dy^1+\frac{1}{\sqrt{1-|y|^2}}dy^2
\eeaa
and hence 
\beaa
dy^1=dx^1_p, \qquad dy^2=-Adx^1_p+\sqrt{1-|y|^2}dx^2_p, \qquad A:=\frac{y^1y^2}{1-(y^1)^2}.
\eeaa
This yields  
\beaa
\mathring{\gamma} &=& \frac{1-(y^2)^2}{1-|y|^2}(dy^1)^2+\frac{2y^1y^2}{1-|y|^2}dy^1dy^2+\frac{1-(y^1)^2}{1-|y|^2}(dy^2)^2\\
&=& \left(\frac{1-(y^2)^2}{1-|y|^2} -\frac{2y^1y^2}{1-|y|^2}A +\frac{1-(y^1)^2}{1-|y|^2}A^2\right)(dx^1_p)^2\\
&&+\left(\frac{2y^1y^2}{1-|y|^2}\sqrt{1-|y|^2} -2A\sqrt{1-|y|^2}\frac{1-(y^1)^2}{1-|y|^2}\right)dx^1_pdx^2_p\\
&&+\frac{1-(y^1)^2}{1-|y|^2}(1-|y|^2)(dx^2_p)^2\\
\eeaa
and {we obtain in particular}
\beaa
d\mathring{\ga} &=&   \Bigg[\left(\frac{1-(y^2)^2}{1-|y|^2} -\frac{2y^1y^2}{1-|y|^2}A +\frac{1-(y^1)^2}{1-|y|^2}A^2\right)\frac{1-(y^1)^2}{1-|y|^2}(1-|y|^2)\\
&& - \left(\frac{y^1y^2}{1-|y|^2}\sqrt{1-|y|^2} -A\sqrt{1-|y|^2}\frac{1-(y^1)^2}{1-|y|^2}\right)^2\Bigg]dx^1_pdx^2_p\\
&=& dx^1_pdx^2_p
\eeaa
as stated. This concludes the proof of Lemma \ref{lemma:isochorecoordinates}.
\end{proof}

%%%%%%%%%%%%%%%%%%%%%%%%%%%%%%%%

\subsubsection{Coordinates systems on $H_r$ and $\MM$}
\lab{sec:isochorecoordinatesonHr}

%%%%%%%%%%%%%%%%%%%%%%%%%%%%%%%%

We consider on $H_r$ the coordinates systems $(\tau, x^1_0, x^2_0)$ and $(\tau, x^1_p, x^2_p)$ with $(x^1_0, x^2_0)$ and $(x^1_p, x^2_p)$ constructed in Lemma \ref{lemma:isochorecoordinates}. This induces on $\MM$ coordinates $(\tau, r, x^1, x^2)$ with $(x^1, x^2)$ denoting  either $(x^1_0, x^2_0)$ or $(x^1_p, x^2_p)$ with
\beaa
x=(x', r), \qquad x':=(x^0, x^1, x^2), \qquad x^0:=\tau, \qquad x^3:=r,
\eeaa
and the corresponding coordinate patches
\beaa
\MM=U_0\cup U_p, \qquad U_q=\mathbb{R}_\tau\times \mathring{U}_q\times [r_+(1-\dhor), +\infty), \quad q=0,p,
\eeaa
with $\mathring{U}_q$, $q=0,p$ the coordinate patches on $\mathbb{S}^2$ introduced in Lemma \ref{lemma:isochorecoordinates}. Also, we denote by $\varphi_q: U_q\to \varphi_q(U_q)\subset\mathbb{R}^4$, $q=0,p$, the corresponding coordinates charts.

Next, we denote by $(\chi_q)_{q=0,p}$, a partition of unity subordinated to the covering by the coordinates patches $U_q$, $q=0,p$, i.e., $\chi_q$ are smooth scalar functions on $\MM$ satisfying  
\bea\lab{eq:partictionofuniityonHr}
\chi_0+\chi_p=1\,\,\,\textrm{on}\,\,\,\MM, \qquad \textrm{supp}(\chi_q)\subset U_q,\,\, q=0,p,  \qquad \pr_r\chi_q=\pr_\tau\chi_q=0,\,\, q=0,p.
\eea
Moreover, we also introduce smooth scalar functions $\widetilde{\chi}_q$, $q=0,p$ on $\MM$ satisfying 
\bea\lab{eq:partictionofuniityonHr:bis}
\widetilde{\chi}_q=1\,\,\,\textrm{on}\,\,\,\textrm{supp}(\chi_q),\,\, q=0,p, \quad \textrm{supp}(\widetilde{\chi_q})\subset U_q,\,\, q=0,p, \quad \pr_r\widetilde{\chi}_q=\pr_\tau\widetilde{\chi}_q=0,\,\, q=0,p.
\eea

To define symbols on $T^\star H_r$, we will need to introduce a norm of a co-vector $\xi'$ on the cotangent bundle. To this end, we introduce the following Riemannian metric $h_r$ on $H_r$
\bea
h_r=(d\tau)^2+\mathring{\ga}.
\eea 
Then, we define the length of a co-vector $\xi'$ by 
\bea
\abs{\xi'} := \sqrt{h_r^{ij} \xi_i'\xi_j'}=\sqrt{(\xi_0')^2+\mathring{\ga}^{ab}\xi_a'\xi_b'},
\eea
where latin indices $i,j$ take values $0,1,2$ and $a,b$ take values $1,2$.

Finally, we have the following lemma concerning the properties of $\sqrt{|\det(\g)|}$ in the coordinates systems $(\tau, r, x^1, x^2)$.
\begin{lemma}\lab{lemma:spacetimevolumeformusingisochorecoordinates}
There exists a well-defined scalar function $f_0$ on $\MM$ such that $\g$ satisfies in the coordinates systems $(\tau, r, x^1, x^2)$
\bea\lab{eq:spacetimevolumeformusingisochorecoordinates}
f_0:=\sqrt{|\det(\g)|}, \qquad f_0=|q|^2(1+r^2\Ga_g).
\eea
\end{lemma}

\begin{remark}\lab{rmk:linkvolumeformMwithf0dVref}
Note that \eqref{eq:spacetimevolumeformusingisochorecoordinates} implies that 
\beaa
\sqrt{|\det(\g)|}d\tau dr dx^1dx^2=f_0 d\Vref, \qquad d\Vref:=d\tau dr dx^1dx^2,
\eeaa
with $d\Vref$ denoting the Lebesgue measure in the coordinates system $(\tau, r, x^1, x^2)$. This allows us to reduce integrals on $\MM$ w.r.t. the volume form of $\g$ in the coordinates systems $(\tau, r, x^1, x^2)$, after multiplication of the integrand by the scalar function $f_0$, to integrals w.r.t. the Lebesgue measure in the coordinates system $(\tau, r, x^1, x^2)$. This fact will play an important role in order to preserve the good properties of Weyl quantization on $\MM$.
\end{remark}

\begin{proof}
We introduce the scalar functions $h_q$, $q=0,p$ defined respectively on $U_0$ and $U_p$ as 
\beaa
h_q:=\sqrt{|\det(\g)|}\quad\textrm{in the coordinates}\quad(\tau, r, x^1_q, x^2_q)\quad\textrm{on}\quad U_q, \quad q=0,p.
\eeaa
Note that we have 
\beaa
h_p=|\det(J_{0p})|h_0\quad\textrm{on}\quad U_0\cap U_p,
\eeaa
where $J_{0p}$ denotes the Jacobian matrix of the change of coordinates. Also, note that 
\beaa
|\det(J_{0p})|=|\det(J^0_{0p})|\quad\textrm{on}\quad U_0\cap U_p,
\eeaa 
where $J^0_{0p}$ denotes the Jacobian matrix of the change of coordinates from $(x_0^1, x_0^2)$ to $(x_p^1, x_p^2)$. Now, since since the coordinates $(x^1_0, x^2_0)$ and $(x^1_p, x^2_p)$ constructed in Lemma \ref{lemma:isochorecoordinates} are isochore, we have $|\det(J^0_{0p})|=1$ and hence
\beaa
h_p=h_0\quad\textrm{on}\quad U_0\cap U_p.
\eeaa
We may thus define the scalar $f_0$ on $\MM$ by
\beaa
f_0=h_0\quad\textrm{on}\quad U_0, \qquad f_0=h_p\quad\textrm{on}\quad U_p,
\eeaa
so that there exists indeed a well-defined scalar function $f_0$ on $\MM$ satisfying $f_0=\sqrt{|\det(\g)|}$  in the coordinates systems $(\tau, r, x^1, x^2)$. Finally, in view of the definition of $f_0$, together with  \eqref{eq:assymptiticpropmetricKerrintaurxacoord:volumeform} and \eqref{eq:controloflinearizedmetriccoefficients:det}, we infer $f_0=|q|^2(1+r^2\Ga_g)$ as stated. This concludes the proof of Lemma \ref{lemma:spacetimevolumeformusingisochorecoordinates}. 
\end{proof}

%%%%%%%%%%%%%%%%%%%%%%%%%%%%

\subsubsection{Classes of mixed symbols on $\MM$}

%%%%%%%%%%%%%%%%%%%%%%%%%%%%

In this section, we extend the tangential and mixed symbol classes on $\mathbb{R}^n$ defined in Section \ref{sec:mixedsymbolsonRn} to $\MM$. We first define $r$-tangential symbols on $\MM$.
\begin{definition}[$r$-tangential symbols on $\MM$]
 \label{def:symbols:mflds:rtangent}
For $m\in \mathbb{R}$, let $S^m_{tan}(\MM)$ denote the set of functions $a$ which are $C^\infty$ in $r$ with values in $C^{\infty}(T^{\star}H_r)$ such that in $x=(x',r)=(\tau, x^1, x^2, r)$ coordinates of $\MM$,  for all multi-indices $\alpha$, $\beta$, and for all $q=0,p$,
 \beaa
\forall x=(x',r)\in\varphi_q(U_q), \,\, \forall \xi'\in T^{\star}H_r,\quad \abs{\partial_x^{\alpha} \partial_{\xi'}^\beta a(\varphi_q^{-1}(x),\xi_j'(dx^j_q)_{|_{\varphi_q^{-1}(x)}})}\leq C_{\alpha,\beta} \langle \xi'\rangle^{m-\abs{\beta}},
\eeaa
with $C_{\alpha,\beta}<+\infty$ and $\langle \xi'\rangle :=\sqrt{1+\abs{\xi'}^2}$. An element $a\in S^m_{tan}{(\MM)}$ is called an $r$-tangential symbol of order $m$. We also denote $S^{-\infty}_{tan}(\MM) :=\cap_{m\in \Reals} S^m_{tan}(\MM)$.
 \end{definition}

Next, we introduce a class of mixed symbols on $\MM$.
\begin{definition}[Mixed symbols on $\MM$]
\label{PDO:MM:Shom}
For $m\in\mathbb{R}$ and $N\in\mathbb{N}$, we define the class $\widetilde{S}^{m,N}(\MM)$ of symbols as $a\in C^{\infty}(T^\star\MM)$ such that for all $q=0,p$, for all $x\in\varphi_q(U_q)$ and for all $\xi\in\mathbb{R}^4$, 
\beaa
a(\varphi_q^{-1}(x),\xi_\a dx^\a)=\sum_{j=0}^N v_{m-j}(\varphi_q^{-1}(x),\xi_i'(dx^i_q)_{|_{\varphi_q^{-1}(x)}})(\xi_3)^j,  \qquad v_{m-j}\in S^{m-j}_{tan}(\MM),
\eeaa 
where $\xi=(\xi', \xi_3)$. An element $a\in\widetilde{S}^{m,N}(\MM)$ is called a mixed symbol of order $(m,N)$. We also denote $\widetilde{S}^{-\infty,N}(\MM) :=\cap_{m\in \Reals}\widetilde{S}^{m,N}(\MM)$.
\end{definition}

\begin{remark}
Notice that $S^m_{tan}(\MM)=\widetilde{S}^{m,0}(\MM)$. 
\end{remark}

%%%%%%%%%%%%%%%%%%%%%%%%%%%%%%%%%%%%

\subsubsection{Weyl quantization of mixed symbols on $\MM$}

%%%%%%%%%%%%%%%%%%%%%%%%%%%%%%%%%%%%

This section adapts to our setting the classical definition of Weyl quantization on a manifold, see e.g. \cite[Chapter 14]{Zwo} \cite[Appendix E]{DZ19}, replacing the presentation using half-densities with the use of isochore coordinates as in \cite{Bon}. 

Recall the coordinates charts $(\varphi_q)_{q=0,p}$ on $\MM$, as well as the partition of unity $(\chi_q)_{q=0,p}$, introduced in Section \ref{sec:isochorecoordinatesonHr}. For $m\in\mathbb{R}$ and $N\in\mathbb{N}$, given $a\in\widetilde{S}^{m,N}(\MM)$, we introduce the following notation, for all $q=0,p$, $x\in \varphi_{{q}}(U_q)$ and $\xi\in\mathbb{R}^4$,
\bea\lab{eq:defsymbolonRnaqchiq}
a_{q,\chi_q}(x, \xi):=\chi_q(\varphi_q^{-1}(x))a\big(\varphi_q^{-1}(x), \xi_\a(dx_q^\a)_{|_{\varphi_q^{-1}(x)}}\big), \qquad a_{q,\chi_q}\in \widetilde{S}^{m,N}(\mathbb{R}^4),
\eea
where the class of mixed symbols $\widetilde{S}^{m,N}(\mathbb{R}^n)$ has been introduced in Definition \ref{PDO:Rn:Shom}.

\begin{definition}[Weyl quantization of mixed symbols on $\MM$]
\lab{def:weylquantizationforrhomsymbolsMM}
Let $m\in\mathbb{R}$, $N\in\mathbb{N}$ and $a\in\widetilde{S}^{m,N}(\MM)$. We associate to $a$ the operator $\Opw(a)$ in the Weyl quantization as follows 
\bea\lab{eq:defintionWeylquantizationmixedsymbolonMM}
\Opw(a)\psi := \sum_{q=0,p}\widetilde{\chi}_q\varphi_q^{\#}\Opw(a_{q,\chi_q})[(\varphi_q^{-1})^\#(\widetilde{\chi}_q\psi)],
\eea
where $(\widetilde{\chi}_q)_{q=0,p}$ is given by \eqref{eq:partictionofuniityonHr:bis} and $a_{q,\chi_q}$ is given by \eqref{eq:defsymbolonRnaqchiq}, i.e., for $x\in\varphi_{q'}(U_{q'})$, $q'=0,p$, 
\beaa
\Opw(a)\psi(\varphi_{q'}^{-1}(x)) &=& \frac{1}{(2\pi)^4}\sum_{q=0,p}\widetilde{\chi}_q(\varphi_{q'}^{-1}(x))\\
&&\times\int_{\mathbb{R}^4}\int_{\mathbb{R}^4}e^{i(x_{q,q'}-y)\c\xi}a_{q,\chi_q}\left(\frac{x_{q,q'}+y}{2}, \xi\right)(\widetilde{\chi}_q\psi)\circ\varphi_q^{-1}(y)dy d\xi,
\eeaa
where $x_{q,q'}=\varphi_q\circ\varphi_{q'}^{-1}(x)$ if $\widetilde{\chi}_q(\varphi_{q'}^{-1}(x))\neq 0$.
\end{definition} 

\begin{remark}
Since $a_{q,\chi_q}\in \widetilde{S}^{m,N}(\mathbb{R}^4)$, Remark \ref{rmk:WeylquantizationofmixedsymbolsonRnlocalinxn} applies to $\Opw(a_{q,\chi_q})$. In view of \eqref{eq:defintionWeylquantizationmixedsymbolonMM}, we infer that $\Opw(a)$ for $a\in\widetilde{S}^{m,N}(\MM)$ is pseudo-differential on $H_r$ but differential w.r.t. $r$ so that it can be applied to functions that are defined on $\MM_{r_1, r_2}$ for $r_+(1-\dhor)\leq r_1<r_2<+\infty$. 
\end{remark}

\begin{remark}
Note that Definition \ref{def:weylquantizationforrhomsymbolsMM} is invariant modulo a smoothing operator under change of coordinates that preserve the isochore property of Lemma \ref{lemma:isochorecoordinates}, but not under general change of coordinates.
\end{remark}

%%%%%%%%%%%%%%%%%%%%%%%%%%%%%%%%%%%%%%%%%

\subsubsection{Properties of the Weyl quantization of mixed symbols on $\MM$}

%%%%%%%%%%%%%%%%%%%%%%%%%%%%%%%%%%%%%%%%%

We start with the following lemma which provides the Weyl quantization of functions, 1-forms and symmetric 2-tensors.
\begin{lemma}\lab{lemma:particularcaseWeylquantizationpolynomialxi:onMM}
We introduce the following mixed symbols $a_0\in\widetilde{S}^{0,0}(\MM)$, $a_1\in\widetilde{S}^{1,1}(\MM)$ and $a_2\in\widetilde{S}^{2,2}(\MM)$ defined, for all $q=0,p$, for all $x\in\varphi_q(U_q)$ and for all $\xi\in\mathbb{R}^4$, by
\beaa
\bsplit
a_0(\varphi_q^{-1}(x),\xi_\a dx^\a)&=\tilde{a}_{0,q}(x),\qquad a_1(\varphi_q^{-1}(x),\xi_\a dx^\a)=\tilde{a}^\a_{1,q}(x)\xi_\a,\\
a_2(\varphi_q^{-1}(x),\xi_\a dx^\a)&=\tilde{a}^{\a\b}_{2,q}(x)\xi_\a\xi_\b,\qquad \tilde{a}^{\a\b}_{2,q}=\tilde{a}^{\b\a}_{2,q}.
\end{split}
\eeaa
Then, we have
\beaa
\bsplit
\Opw (a_0)\psi =& a_0\psi, \\ 
\Opw (a_1)\psi(\varphi^{-1}_q(x)) =& \tilde{a}^\a_{1,q}(x)D_{x^\a}\psi(\varphi^{-1}_q(x))+ \frac{1}{2}(D_{x^\a}\tilde{a}^\a_{1,q})(x)\psi(\varphi^{-1}_q(x)),\\
\Opw (a_2)\psi(\varphi^{-1}_q(x)) =&\tilde{a}^{\a\b}_{2,q}(x)D_{x^\a}D_{x^\b}\psi(\varphi^{-1}_q(x)) + (D_{x^\a}\tilde{a}^{\a\b}_{2,q})(x)D_{x^\b}\psi(\varphi^{-1}_q(x))\\
& +\frac{1}{4}(D_{x^\a}D_{x^\b}\tilde{a}^{\a\b}_{2,q})(x)\psi(\varphi^{-1}_q(x)).
\end{split}
\eeaa
\end{lemma}

\begin{proof}
For the symbols $a_0$, $a_1$ and $a_2$, $a_{q,\chi_q}(x, \xi)$ as defined in \eqref{eq:defsymbolonRnaqchiq} is polynomial in $\xi$ and we may thus apply Lemma \ref{lemma:particularcaseWeylquantizationpolynomialxi}. We then plug the resulting formula in \eqref{eq:defintionWeylquantizationmixedsymbolonMM} and use the fact that $\widetilde{\chi}_q=1$ on the support of $\chi_q$ and the fact that $\chi_0+\chi_1=1$ to conclude the proof of the lemma.
\end{proof}

Next, we consider the properties of the Weyl quantization of symbols in $\widetilde{S}^{m,N}(\MM)$ w.r.t. composition and adjoint. 
\begin{proposition}
\label{prop:PDO:MM:Weylquan:mixedoperators}
The Weyl quantization satisfies the following properties for symbols in the class 
$\widetilde{S}^{m,N}(\MM)$: 
\begin{enumerate}[label=\arabic*)] 
\item For mixed symbols $a_1$ and $a_2$ of respective orders $(m_1,N_1)$ and $(m_2,N_2)$, we have
\bea\lab{eq:propWeylquantization:MM:composition:mixedsymbols}
\bsplit
[\Opw(a_1), \Opw(a_2)]=\Opw(a_3),\quad a_3=\frac{1}{i}\{a_1, a_2\} +\widetilde{S}^{m_1+m_2-3,N_1+N_2}(\MM),\\
\Opw(a_1)\circ\Opw(a_2) +\Opw(a_2)\circ\Opw(a_1)=\Opw(a_3),\quad a_3=2a_1a_2 + \widetilde{S}^{m_1+m_2-2,N_1+N_2}(\MM).
\end{split}
\eea

\item In the particular case where $a_1(\varphi_q^{-1}(x),\xi_\a dx^\a)=v_1(r)\xi_3^{N_1}$ for $x=(r,x')\in\varphi_q(U_q)$ and $\xi=(\xi', \xi_3)\in\mathbb{R}^4$, which is a mixed symbol of order $(m_1, N_1)$ with $m_1=N_1$, we have, with $a_2$ of order $(m_2,N_2)$
\bea\lab{eq:propWeylquantization:MM:composition:mixedsymbols:specialcase}
\bsplit
[\Opw(a_1), \Opw(a_2)]=\Opw(a_3),\quad a_3=\frac{1}{i}\{a_1, a_2\} +\tilde{a}_3,\\
\tilde{a}_3=0\quad\textrm{if}\quad \max(N_1, N_2)\leq 2, \qquad \tilde{a}_3\in\widetilde{S}^{m_1+m_2-3,N_1+N_2-3}(\MM)\quad\textrm{if}\quad \max(N_1, N_2)\geq 3,\\
\Opw(a_1)\circ\Opw(a_2) +\Opw(a_2)\circ\Opw(a_1)=\Opw(a_4),\quad a_4=2a_1a_2 + \tilde{a}_4,\\
\tilde{a}_4=0\quad\textrm{if}\quad \max(N_1, N_2)\leq 1, \qquad \tilde{a}_4\in\widetilde{S}^{m_1+m_2-2,N_1+N_2-2}(\MM)\quad\textrm{if}\quad \max(N_1, N_2)\geq 2.
\end{split}
\eea

\item In the particular case where $a_1$ and $a_2$ are mixed symbols of respective orders $(m_1, 1)$ and $(m_2, 1)$, and $f=f({r})$, we have
\bea\lab{eq:propWeylquantization:MM:composition:mixedsymbols:specialcase:1}
\bsplit
&[\Opw(a_1), \Opw(f(r)a_2)]=\Opw(a_3),\quad a_3=\frac{1}{i}\{a_1, f(r)a_2\} +\tilde{a}_3,\\
&\tilde{a}_3=f(r)\widetilde{S}^{m_1+m_2-3,2}(\MM)+\widetilde{S}^{m_1+m_2-3,1}(\MM),\\
&\Opw(a_1)\circ\Opw(f(r)a_2) +\Opw(f(r)a_2)\circ\Opw(a_1)=\Opw(a_4),\quad a_4=2f(r)a_1a_2 + \tilde{a}_4,\\
&\tilde{a}_4=f(r)\widetilde{S}^{m_1+m_2-2,2}(\MM)+\widetilde{S}^{m_1+m_2-2,1}(\MM).
\end{split}
\eea

\item For a mixed symbol $a(x,\xi)$, the adjoint, w.r.t. the Lebesgue measure $d\Vref$ in $(\tau, r, x^1, x^2)$ coordinates, of its Weyl quantization is given by 
\bea\lab{eq:propWeylquantization:MM:adjoint:mixedsymbols}
(\Opw(a))^{\star} = \Opw(\bar{a}).
\eea
In particular, the Weyl quantization of a real-valued symbol is a self-adjoint operator  w.r.t. the Lebesgue measure $d\Vref$ in $(\tau, r, x^1, x^2)$ coordinates.
\end{enumerate}
\end{proposition}

\begin{proof}
The proof of \eqref{eq:propWeylquantization:MM:adjoint:mixedsymbols} is immediate. Next, adapting \cite{Bon} to the case of mixed symbols, the proof of \eqref{eq:propWeylquantization:MM:composition:mixedsymbols} follows, in view of the definition \eqref{eq:defintionWeylquantizationmixedsymbolonMM} of the Weyl quantization for  mixed symbols, from the corresponding result in $\mathbb{R}^n$ in \eqref{eq:propWeylquantizationmathbbRn:composition:mixedsymbols} applied to the mixed symbols $a_{q,\chi_q}\in \widetilde{S}^{m,N}(\mathbb{R}^4)$ introduced in \eqref{eq:defsymbolonRnaqchiq}, together with Lemma \ref{lemma:compositionPDOWeylanddiffeoconservative} on change of isochore coordinates. Finally, for \eqref{eq:propWeylquantization:MM:composition:mixedsymbols:specialcase} and \eqref{eq:propWeylquantization:MM:composition:mixedsymbols:specialcase:1}, we proceed as in the proof of \eqref{eq:propWeylquantization:MM:composition:mixedsymbols} using additionally  the corresponding properties \eqref{eq:propWeylquantizationmathbbRn:composition:mixedsymbols:specialcase} and \eqref{eq:propWeylquantizationmathbbRn:composition:mixedsymbols:specialcase:1} in $\mathbb{R}^n$.
\end{proof}

We also consider the action of the Weyl quantization of mixed symbols on Sobolev spaces.  
\begin{lemma}\lab{lemma:actionmixedsymbolsSobolevspaces:MM}
Let $m\in\mathbb{R}$, $N\in\mathbb{N}$, let $I$ be an interval of $[r_+(1-\dhor), +\infty)$, and let $a\in\widetilde{S}^{m,N}(\MM)$ be of the form, for all $q=0,p$, $x\in \varphi(U_q)$ and $\xi\in\mathbb{R}^4$,
\beaa
a(\varphi_q^{-1}(x),\xi_\a dx^\a)=v_{m-N}(\varphi_q^{-1}(x),\xi_i'(dx^i_q)_{|_{\varphi_q^{-1}(x)}})(\xi_3)^N,  \qquad v_{m-N}\in S^{m-N}_{tan}(\MM).
\eeaa 
Then we have for all $s\in\mathbb{R}$
\beaa
\|\Opw(a)\psi\|_{H^{s-m+N}(H_r)}&\les& \sum_{j=0}^N\|\pr_r^j\psi\|_{H^s(H_r)},\\
\|\Opw(a)\psi\|_{L^2_r(I, H^{s-m+N}(H_r))}&\les& \sum_{j=0}^N\|\pr_r^j\psi\|_{L^2_r(I, H^s(H_r))}.
\eeaa
\end{lemma}

\begin{proof}
This follows immediately by relying on the definition  \eqref{eq:defintionWeylquantizationmixedsymbolonMM} of $\Opw(a)$ and then  applying Lemma \ref{lemma:actionmixedsymbolsSobolevspaces} to $\Opw(a_{q,\chi_q})$ where $a_{q,\chi_q}\in\widetilde{S}^{m,N}(\mathbb{R}^4)$ is given by  \eqref{eq:defsymbolonRnaqchiq}.
\end{proof}

Finally, we prove a G\aa rding type inequality for tangential symbols.
\begin{lemma}\lab{lemma:basicgardingfortangentialsymbols:MM}
Let $I$ be an interval of $[r_+(1-\dhor), +\infty)$ and let $a\in\widetilde{S}^{1,0}(\MM)$ be such that, for all $q=0,p$, $x\in \varphi(U_q)$ and $\xi\in\mathbb{R}^4$,
\beaa
a(\varphi_q^{-1}(x),\xi_\a dx^\a)\geq c_1\langle\xi'\rangle,
\eeaa
for some constant $c_1>0$. Then, there exist constants $c_2>0$ and $c_3>0$ such that we have
\beaa
\|\Opw(a)\psi\|_{L^2_r(I, L^2(H_r))}^2\geq c_2\|\psi\|_{L^2_r(I, H^1(H_r))}^2 -c_3\|\psi\|_{L^2_r(I, L^2(H_r))}^2.
\eeaa
\end{lemma}

\begin{proof}
The proof of Lemma \ref{lemma:basicgardingfortangentialsymbols:MM} is analogous to the one of Lemma \ref{lemma:basicgardingfortangentialsymbols} replacing the use of Proposition \ref{prop:PDO:Rn:Weylquan:mixedoperators} with the one of Proposition \ref{prop:PDO:MM:Weylquan:mixedoperators} and the use of Lemma \ref{lemma:actionmixedsymbolsSobolevspaces} with the one of Lemma \ref{lemma:actionmixedsymbolsSobolevspaces:MM}.
\end{proof}

Finally, we provide a basic commutator estimate for mixed symbols. 
\begin{lemma}\lab{lemma:bascicommutatorlemmawithelementaryproof:mixedsymbols:MMcase}
Let $I$ be an interval of $[r_+(1-\dhor), +\infty)$, let $P=\Opw(p)$ with $p\in\widetilde{S}^{1,1}(\MM)$ and let $f$ be a scalar function on $\MM\cap\{r\in I\}$. Then, we have 
\beaa
\|[P, f]\psi\|_{L^2_r(I,H^1(H_r))} &\les& \|f\|_{W^{2,+\infty}(\MM\cap\{r\in I\})}\|\pr^{\leq 1}\psi\|_{L^2_r(I,L^2(H_r))}.
\eeaa
\end{lemma}

\begin{proof}
The proof is a simple adaptation of the one of Lemma \ref{lemma:bascicommutatorlemmawithelementaryproof:mixedsymbols}.
\end{proof}

%%%%%%%%%%%%%%%%%%%%%%%%%%%%%%%%%%%%

\subsubsection{Weyl quantization of particular symbols}

%%%%%%%%%%%%%%%%%%%%%%%%%%%%%%%%%%%%

The vectorfields $\partial_r$, $\partial_{\tt}$ and $\partial_{\tphi}$ are globally smooth vectorfields on $\MM$, and the Carter $2$-tensor field, denoted by 
\bea
\Carter := \mathring{\ga}^{bc}\pr_{x^b}\otimes \pr_{x^c}+a^2 \sin^2\th \partial_{\tt}\otimes \partial_{\tt}, 
\eea 
is a globally smooth $2$-tensor field on $\MM$. Recalling that our symbols are defined w.r.t. the $(\tau, r, x^1, x^2)$ coordinates systems of Section \ref{sec:isochorecoordinatesonHr}, we introduce the following definitions:
\begin{subequations}\lab{eq:basicparticularsymbolsusedeverywhere}
\begin{align}
\xit:={}& \langle \xi, \partial_{\tt}\rangle=\xi_0,\\
\xi_r : ={}& \langle \xi, \partial_{r} \rangle = \xi_3,\\
{\xiphi}:={}& \langle \xi, \partial_{\tphi}\rangle=\frac{\partial x^a}{\partial \tphi} \xi_a,\\
\Lambda:={}& \sqrt{\langle \xi\otimes \xi, \Carter\rangle}=\sqrt{\mathring{\ga}^{bc} \langle \xi, \pr_{x^b}\rangle \langle \xi, \pr_{x^c}\rangle+ a^2\sin^2\th \xi_0^2}\geq 0.
\end{align}
\end{subequations}
We denote $\Xi:=(\xit,  \xiphi,\Lambda)$, and denote $\Gtz$ as the space of the frequency triplet $\Xi=(\xit,\xiphi,\Lambda)$.

Since $\pr_{r^*}$, defined w.r.t. the tortoise coordinates $(t,r^*, \theta,\phi)$, is a globally smooth vectorfield in $\MM$, the following definition is well-defined
\begin{align}
\label{def:xirstar}
\xirstar:=& \langle \xi, \pr_{r^*}\rangle \nn\\
=& \mu \xi_r + \frac{1}{\R}\big((\R)(1-\mu \tmod')\xit+(a-\Delta\phimod')\xiphi\big).
\end{align}

We now introduce two additional classes of symbols.
\begin{definition}[Classes of symbols $\widetilde{S}^{m,N}_{hom}(\MM)$ and $\widetilde{S}^{N}_{pol}(\MM)$]
\lab{def:thesymbolswereallyuse}
For $m\in\mathbb{R}$ and $N\in\mathbb{N}$, we define the class $\widetilde{S}^{m,N}_{hom}(\MM)$ of symbols as $a\in C^{\infty}(T^\star\MM\setminus\{\xi'=0\})$ such that for all $q=0,p$, for all $x\in\varphi_q(U_q)$ and for all $\xi=(\xi', \xi_3)\in\mathbb{R}^4$, 
\beaa
a(\varphi_q^{-1}(x),\xi_\a dx^\a)=\sum_{j=0}^N \tilde{v}_{m-j}(r,\xi_\tau, \xi_{\tphi}, \La)(\xi_3)^j,  
\eeaa 
where $\tilde{v}_{m-j}$ is smooth w.r.t. $r$ and homogenous of order $m-j$ w.r.t. $(\xi_\tau, \xi_{\tphi}, \La)$. Also, we define the class $\widetilde{S}^{N}_{pol}(\MM)$ of symbols as the subclass of $\widetilde{S}^{N,N}_{hom}(\MM)$ where $\tilde{v}_{N-j}$ is smooth w.r.t. $r$ and a  homogenous polynomial of order $N-j$ w.r.t. $(\xi_\tau, \xi_{\tphi}, \La)$ for $N-j$ even, and here $\tilde{v}_{N-j}$ is smooth w.r.t. $r$ and a  homogenous polynomial of order $N-j$ w.r.t. $(\xi_\tau, \xi_{\tphi})$ for $N-j$ odd.
\end{definition}

The above symbols satisfy the following properties.
\begin{lemma}\lab{lemma:simplepropertiesspecialcaseofsymbols:recoveringoperators}
We have the following properties:
\begin{enumerate}
\item We have $\widetilde{S}^{N}_{pol}(\MM)\subset\widetilde{S}^{N,N}(\MM)$. Also, any symbol $a$ in $\widetilde{S}^{m,N}_{hom}(\MM)$ is such that $\chi(\Xi)a\in\widetilde{S}^{m,N}(\MM)$ where $\Xi=(\xit, \xiphi, \La)$, and where $\chi$ denotes a smooth cut-off in $\mathbb{R}^3$ such that $0\leq\chi\leq 1$, $\chi=1$ for $|\Xi|\geq 2$ and $\chi=0$ for $|\Xi|\leq 1$.

\item Let $a^{(1)}\in\widetilde{S}^{N_1}_{pol}(\MM)$ and $a^{(2)}\in\widetilde{S}^{m, N_2}_{hom}(\MM)$. Also, let $\chi$ be chosen as above. Then, the Poisson bracket $\{a^{(1)}, \chi(\Xi)a^{(2)}\}$ is given by
\bea\lab{eq:basiccomputationPoissonbracketsymbolsusedinwaveandmultipliers}
\{a^{(1)}, \chi(\Xi)a^{(2)}\}=\pr_{\xi_r}a^{(1)}\pr_r(\chi(\Xi)a^{(2)}) - \pr_r a^{(1)}\pr_{\xi_r}(\chi(\Xi)a^{(2)}).
\eea

\item The following identities hold true
\beaa
\bsplit
\Opw(\xi_\tau)&=D_\tau, \qquad \Opw(\xi_r)=D_r, \qquad \Opw(\xi_{\tphi})=D_{\tphi}, \qquad \Opw(\xi_{r^*})=D_{r^*}-\frac{i}{2}\mu',\\
\Opw(\xi_\tau^j\xi_r^{2-j})&=D_\tau^jD_r^{2-j}, \,\, j=0,1,2,\qquad \Opw(\xi_\tau\xi_{\tphi})=D_\tau D_{\tphi}, \qquad \Opw(\xi_r\xi_{\tphi})=D_rD_{\tphi},\\
\Opw(\xi_{\tphi}^2)&=D_{\tphi}^2+\Opw(\widetilde{S}^{0,0}(\MM)), \qquad \Opw(\La^2)=-\De_{\mathring{\ga}}-a^2\sin^2\th\pr_{\tau}^2+\Opw(\widetilde{S}^{0,0}(\MM)).
\end{split}
\eeaa 
\end{enumerate}
\end{lemma}

\begin{proof}
The first property is obvious and we focus on the two other ones. For the second property, we first use Leibniz rule which yields
\beaa
\{a^{(1)}, \chi(\Xi)a^{(2)}\}=\chi(\Xi)\{a^{(1)}, a^{(2)}\}+\{a^{(1)}, \chi(\Xi)\}a^{(2)}.
\eeaa
Next, we focus on the computation of $\{a^{(1)}, a^{(2)}\}$. To this end, we rely on the fact that the Poisson bracket is invariant under coordinates changes to compute $\{a^{(1)}, a^{(2)}\}$ in the $(\tau, r, \th, \tphi)$ coordinates system. Since 
\beaa
\bsplit
&\pr_\tau(r,\xi_r, \xi_\tau, \xi_{\tphi}, \La)=0, \qquad \pr_r(r)=1, \qquad \pr_r(\xi_r, \xi_\tau, \xi_{\tphi}, \La)=0,\\
&\pr_\th(r,\xi_r, \xi_\tau, \xi_{\tphi})=0, \qquad \pr_{\tphi}(r,\xi_r, \xi_\tau, \xi_{\tphi}, \La)=0,
\end{split}
\eeaa
and since the symbols $a^{(1)}$ and $a^{(2)}$ are functions of $(r, \xi_r, \xit, \xiphi, \La)$, we obtain 
\beaa
\{a^{(1)}, a^{(2)}\} &=& \pr_{\xi_r}a^{(1)}\pr_r a^{(2)} - \pr_r a^{(1)}\pr_{\xi_r}a^{(2)}+\pr_{\xi_\th}(\La)\pr_\La a^{(1)}\pr_\th(\La)\pr_\La a^{(2)}\\
&&-\pr_\th(\La)\pr_\La a^{(1)}\pr_{\xi_\th}(\La)\pr_\La a^{(2)}\\
&=& \pr_{\xi_r}a^{(1)}\pr_r a^{(2)} - \pr_r a^{(1)}\pr_{\xi_r}a^{(2)}.
\eeaa
Similarly, we have
\beaa
\{a^{(1)}, \chi(\Xi)\}=\pr_{\xi_\th}(\La)\pr_\La a^{(1)}\pr_\th(\La)\pr_\La\chi(\Xi)-\pr_\th(\La)\pr_\La a^{(1)}\pr_{\xi_\th}(\La)\pr_\La\chi(\Xi)= 0
\eeaa
and hence
\beaa
\{a^{(1)}, \chi(\Xi)a^{(2)}\}&=& \chi(\Xi)\{a^{(1)}, a^{(2)}\}+\{a^{(1)}, \chi(\Xi)\}a^{(2)}\\
&=& \pr_{\xi_r}a^{(1)}\pr_r(\chi(\Xi)a^{(2)}) - \pr_r a^{(1)}\pr_{\xi_r}(\chi(\Xi)a^{(2)})
\eeaa
as stated.

It remains to prove the third property. The identities concerning the Weyl quantization of $\xi_\tau$, $\xi_r$, $\xi_{r*}$ and $\xi_\tau^j\xi_r^{2-j}$ follow immediately from Lemma \ref{lemma:particularcaseWeylquantizationpolynomialxi:onMM}. Also, the identities concerning the Weyl quantization of $\xi_{\tphi}$, $\xi_\tau\xi_{\tphi}$ and $\xi_r\xi_{\tphi}$ follow from Lemma \ref{lemma:particularcaseWeylquantizationpolynomialxi:onMM}, using also $\pr_{\tphi}=\frac{\pr x^a}{\pr\tphi}\pr_{x^a}$ and the fact that the isochore coordinates of Lemma \ref{lemma:isochorecoordinates} satisfy the following identities\footnote{These identities follow immediately from the following identities for the isochore coordinates of Lemma \ref{lemma:isochorecoordinates} 
\beaa
\frac{\partial x^1_0}{\partial \tphi}=0, \qquad \frac{\partial x^2_0}{\partial \tphi}=1,\qquad \frac{\partial x^1_p}{\partial \tphi} = -\sin(x_p^2)\sqrt{1-(x_p^1)^2},\qquad
\frac{\partial x^2_p}{\partial \tphi} = \frac{x_p^1\cos(x^2_p)}{\sqrt{1-(x^1_p)^2}}.
\eeaa}   
\bea\lab{eq:niceidentity:isochorecoordinatesandnormalizedcoordinates}
\pr_{x^a_0}\left(\frac{\pr x^a_0}{\pr\tphi}\right)=0, \qquad \pr_{x^a_p}\left(\frac{\pr x^a_p}{\pr\tphi}\right)=0. 
\eea

Next, the identity concerning the Weyl quantization of $\xi_{\tphi}^2$ follows immediately from the one for $\xi_{\tphi}\in\widetilde{S}^{1,0}(\MM)$ and \eqref{eq:propWeylquantization:MM:composition:mixedsymbols} which implies 
\beaa
\Opw(\xi_{\tphi}^2) &=& \Opw(\xi_{\tphi})\circ\Opw(\xi_{\tphi})+\Opw(\widetilde{S}^{0,0}(\MM))\\
&=& D_{\tphi}^2+\Opw(\widetilde{S}^{0,0}(\MM))
\eeaa
as stated.

Finally, in view of Lemma \ref{lemma:particularcaseWeylquantizationpolynomialxi:onMM}, we have
\beaa
\Opw(\La^2) &=& \Opw(\mathring{\ga}^{bc}\xi_b\xi_c)+a^2\Opw(\sin^2\th \eta_0^2)\\
&=& \mathring{\ga}^{bc}D_bD_c+D_b(\mathring{\ga}^{bc})D_c+\Opw(\widetilde{S}^{0,0}(\MM))+a^2\sin^2\th D_\tau^2\\
&=& -\mathring{\ga}^{bc}\pr_b\pr_c -\pr_b(\mathring{\ga}^{bc})\pr_c-a^2\sin^2\th \pr_\tau^2+\Opw(\widetilde{S}^{0,0}(\MM)).
\eeaa
Now, since the coordinates of Lemma \ref{lemma:isochorecoordinates} are isochore for $\mathring{\ga}$, we have 
\beaa
\De_{\mathring{\ga}}\psi=\pr_b(\mathring{\ga}^{bc}\pr_c\psi) =\mathring{\ga}^{bc}\pr_b\pr_c\psi+\pr_b(\mathring{\ga}^{bc})\pr_c\psi
\eeaa
and hence
\beaa
\Opw(\La^2)=-\De_{\mathring{\ga}}-a^2\sin^2\th\pr_{\tau}^2+\Opw(\widetilde{S}^{0,0}(\MM)).
\eeaa
This concludes the proof of Lemma \ref{lemma:simplepropertiesspecialcaseofsymbols:recoveringoperators}. 
\end{proof}

Next, the following lemma allows to write, to main order, a rescaling of the scalar wave operator on $\MM$ as the Weyl quantization of a mixed symbol of order $(2,2)$.
\begin{lemma}
\label{lem:WeylquantizedtoBoxg}
Let $f_0$ be the scalar function on $\MM$ constructed in Lemma  \ref{eq:spacetimevolumeformusingisochorecoordinates}, i.e., $f_0=\sqrt{|\det(\g)|}$ in the coordinates systems of Section \ref{sec:isochorecoordinatesonHr}. Then, we have
\bea
f_0\Box_{\g}=\Opw (-f_0\g^{\a\b}\xi_{\a}\xi_{\b})+\Opw(\widetilde{S}^{0,0}(\MM)).
\eea
\end{lemma}

\begin{proof}
Relying on Lemma \ref{lemma:particularcaseWeylquantizationpolynomialxi:onMM}, we have
\beaa
\Opw (f_0\g^{\a\b}\xi_{\a}\xi_{\b}) &=& f_0\g^{\a\b}D_{x^\a}D_{x^\b}+D_{x^\a}(f_0\g^{\a\b})D_{x^\b}+\Opw(\widetilde{S}^{0,0}(\MM))\\
&=& -f_0\g^{\a\b}\pr_{x^\a}\pr_{x^\b}-\pr_{x^\a}(f_0\g^{\a\b})\pr_{x^\b}+\Opw(\widetilde{S}^{0,0}(\MM)).
\eeaa
Now, since $f_0=\sqrt{|\det(\g)|}$ in the coordinates systems of Section \ref{sec:isochorecoordinatesonHr}, see \eqref{eq:spacetimevolumeformusingisochorecoordinates}, we have
\beaa
f_0\g^{\a\b}\pr_{x^\a}\pr_{x^\b}\psi+\pr_{x^\a}(f_0\g^{\a\b})\pr_{x^\b}\psi &=& \sqrt{|\det(\g)|}\g^{\a\b}\pr_{x^\a}\pr_{x^\b}\psi+\pr_{x^\a}(\sqrt{|\det(\g)|}\g^{\a\b})\pr_{x^\b}\psi\\
&=& \pr_{x^\a}(\sqrt{|\det(\g)|}\g^{\a\b}\pr_{x^\b}\psi)= \sqrt{|\det(\g)|}\square_\g\psi\\
&=& f_0\square_\g\psi
\eeaa
and hence 
\beaa
\Opw (f_0\g^{\a\b}\xi_{\a}\xi_{\b}) &=& -f_0\square_\g+\Opw(\widetilde{S}^{0,0}(\MM))
\eeaa
as stated.
\end{proof}

Finally, we introduce a notation for the symbol of an operator corresponding to the Weyl quantization of a mixed symbol. 
\begin{definition}\lab{def:notationforthesymbolofWeylquantizationmixedsymbol}
If $A=\Opw(a)$ for some mixed symbol $a\in\widetilde{S}^{m,N}(\MM)$, then we denote the symbol $a$ of $A$ by $\sigma(A)$, i.e., 
\beaa
\sigma(A):=a, \qquad A=\Opw(a).
\eeaa
\end{definition}

\begin{remark}\lab{rmk:symbolesofimportantoperators}
In view of Definition \ref{def:notationforthesymbolofWeylquantizationmixedsymbol} and Lemma \ref{lemma:simplepropertiesspecialcaseofsymbols:recoveringoperators}, 
we have
\beaa
\bsplit
\sigma(D_\tau)&=\xi_\tau, \qquad \sigma(D_r)=\xi_r, \qquad \sigma(D_{\tphi})=\xi_{\tphi}, \qquad \sigma(D_{r^*})=\xi_{r^*}+\frac{i}{2}\mu',\\
\sigma(D_\tau^jD_r^{2-j})&=\xi_\tau^j\xi_r^{2-j}, \,\, j=0,1,2,\qquad \sigma(D_\tau D_{\tphi})=\xi_\tau\xi_{\tphi}, \qquad \sigma(D_rD_{\tphi})=\xi_r\xi_{\tphi},\\
\sigma(D_{\tphi}^2)&=\xi_{\tphi}^2+\widetilde{S}^{0,0}(\MM), \qquad \sigma(\De_{\mathring{\ga}}+a^2\sin^2\th\pr_{\tau}^2)=-\La^2+\widetilde{S}^{0,0}(\MM).
\end{split}
\eeaa
Also, in view of Definition \ref{def:notationforthesymbolofWeylquantizationmixedsymbol} and Lemma  \ref{lem:WeylquantizedtoBoxg}, we have
\beaa
\sigma(f_0\Box_{\g})= -f_0\g^{\a\b}\xi_{\a}\xi_{\b}+\widetilde{S}^{0,0}(\MM).
\eeaa
\end{remark}

%%%%%%%%%%%%%%%%%%%%%%%%%%%%%%%%%%%%%

\section{Proof of global energy-Morawetz estimates}
\lab{sec:proofofth:main:intermediary}

%%%%%%%%%%%%%%%%%%%%%%%%%%%%%%%%%%%%%

The goal of this section is to prove Theorem \ref{th:main:intermediary}. To this end, we first introduce microlocal energy-Morawetz norms.

%%%%%%%%%%%%%%%%%%%%%%%%%%%%%%%%

\subsection{Definition of microlocal energy-Morawetz norms}

%%%%%%%%%%%%%%%%%%%%%%%%%%%%%%%%

\begin{definition}[Microlocal energy-Morawetz norms]
\lab{def:microlocalenergyMorawetznorms}
Let $\sigma_{\trap}\in\widetilde{S}^{1,0}(\MM)$ be defined, for all $q=0,p$, for all $x=(x', r)\in\varphi_q(U_q)$ and for all $\xi\in\mathbb{R}^4$, by
\bea\lab{eq:definitionofthesymbolsigmatrap}
\sigma_{\trap}(\varphi_q^{-1}(x),\xi_\a dx^\a):=(r-{r_{\trap}}){\upsilon},
\eea 
where {$\upsilon\in\widetilde{S}^{1,0}(\MM)$ is given by} 
{\bea
\lab{def:upsilonsymbol}
\upsilon:=\sqrt{1+\xi_0^2+\mathring{\ga}^{bc} \langle \xi, \pr_{x^b}\rangle \langle \xi, \pr_{x^c}\rangle}
\eea}
  and $r_{\trap}\in\widetilde{S}^{0,0}(\MM)$ is defined in \eqref{eq:definitionofrtrapinfunctionofrmaxandcutoffinmathcalG5}{, and let $e\in \widetilde{S}^{1,0}(\MM)$ be given as in \eqref{eq:defintionofthesymboleasasquarerootofsigma2TXEmsumofsquares:1:defe}.} Then, we define the microlocal Morawetz norm $\widetilde {\M}[\psi]$ by
\bea\lab{eq:definitionofmicrolocalMorawetznormwidetildeM}
\widetilde {\M}[\psi]:={}{\M}[\psi](\Reals) +\int_{\MM_{r_+{(1+2\dhor)}, 10m}}|\Opw(\sigma_{\trap})\psi|^2{+\int_{\MM_{r_+{(1+2\dhor)}, 10m}}|\Opw(e)\psi|^2}.
\eea
We also introduce 
\bea\lab{eq:definitionwidetildemathcalNpsijnormRHS}
\widetilde{\mathcal{N}}[\psi, F](\mathbb{R})&:=&\widetilde{\mathcal{N}}_{trap}[\psi, F](\mathbb{R}) +\sup_{\tau\in\Reals}\bigg|\int_{\Mntrap(-\infty, \tau)}{\ov{\pr_{\tau}\psi}}F \bigg|\nn\\
&&+\int_{\Mntrap}\big(|\pr_r\psi|+r^{-1}|\psi|\big)|F|+\int_{\MM}|F|^2,
\eea
where $\widetilde{\mathcal{N}}_{trap}[\psi, F]$ is defined by
\bea
\lab{def:widetildeNtrap}
\widetilde{\NN}_{trap}[\psi, F](\Reals):=\int_{\Mtrap}|F|{|S^1\psi|},
\eea
with {$S^1=(S^{1,0}, S^{1,1},\cdots, S^{1,\iota})\in(\Opw(\widetilde{S}^{1,0}(\MM)))^{1+\iota}$} being defined by 
\beaa
{S^1:=\Big(\pr_\tau, \Opw(\Theta_1)V_1\Opw(\Theta_1),\cdots, \Opw(\Theta_{\iota})V_{\iota} \Opw(\Theta_{\iota})\Big),}
\eeaa
with the operators {$V_i\in\Opw(\widetilde{S}^{1,0}(\MM))$}, $1\leq i\leq \iota$, being defined as in  \eqref{eq:defintionofthedifferentialoperatorsViformicrolocalNRGestimates}, and with the {symbols} $\Theta_i\in\widetilde{S}^{0,0}(\MM)$, $1\leq i\leq \iota$, being introduced in \eqref{eq:definitionofpartitionofunityThetajj=1toiota}.

Finally, for any nonnegative integer $\reg$, let
\bea
\widetilde{\M}^{(s)}[\psi]:=\sum_{0\leq \abs{i}\leq s}\widetilde{\M}[\pr^{i}\psi],\qquad \widetilde{\mathcal{N}}^{(\reg)}[\psi, F](\mathbb{R}):=\sum_{0\leq \abs{i}\leq s}\widetilde{\mathcal{N}}[\pr^i\psi,\pr^iF](\mathbb{R})
\eea
and
\bea
\widetilde{\EMF}^{(s)}[\psi]:= \sup_{\tt\in \Reals} \E^{(s)}[\psi](\tt) + \widetilde{\M}^{(s)}[\psi]+\F^{(s)}[\psi](\Reals).
\eea
\end{definition}

We have the following comparison of $\widetilde{\mathcal{N}}[\psi, F](\Reals)$ with $\widehat{\mathcal{N}}[\psi, F]$.
\begin{lemma}\lab{lemma:additivepropertyofwidetildeN}
Given $N\geq 2$, consider any partition of $\mathbb{R}$ in intervals of the following form
\beaa
\Reals=(-\infty, \tau^{(1)}]\cup\bigcup_{j=1}^{N-1}( \tau^{(j)},  \tau^{(j+1)}]\cup( \tau^{(N)}, +\infty), \qquad -\infty< \tau^{(1)}<\cdots< \tau^{(N)}<+\infty.
\eeaa
{Then,} for $F$ supported in $\tau\geq 1$, $\widetilde{\mathcal{N}}[\psi, F](\mathbb{R})$ satisfies the following, for any $0<\la\leq 1$,
\beaa
\bsplit
\widetilde{\mathcal{N}}[\psi, F](\mathbb{R}) \les_N& \la^{-1}\widehat{\mathcal{N}}[\psi, F](-\infty, \tau^{(1)}+1)+\la^{-1}\sum_{j=1}^{N-1}\widehat{\mathcal{N}}[\psi, F](\tau^{(j)}-1, \tau^{(j+1)}+1)\\
&+\la^{-1}\widehat{\mathcal{N}}[\psi, F](\tau^{(N)}-1, +\infty) +\la\EM[\psi](\Reals) +\left(\int_{\Mtrap}|F|^2\right)^{\frac{1}{2}}\left(\int_{\Mtrap}|\psi|^2\right)^{\frac{1}{2}},
\end{split}
\eeaa
where $\widehat{\mathcal{N}}[\psi, F](\tau_1, \tau_2)$, for any $\tau_1<\tau_2$, is defined in \eqref{eq:defintionwidehatmathcalNfpsinormRHS}.
\end{lemma}

\begin{proof}
Let us denote the intervals $I_j$ by
\beaa
I_0:=(-\infty, \tau^{(1)}], \qquad I_j:=(\tau^{(j)}, \tau^{(j+1)}], \quad j=1,\cdots,N-1, \qquad I_N:=(\tau^{(N)}, +\infty).
\eeaa
Then, notice from the definition of $\widetilde{\mathcal{N}}[\psi, F](\mathbb{R})$ that 
\beaa
\widetilde{\mathcal{N}}[\psi, F](\mathbb{R})\leq \sum_{j=0}^N\widetilde{\mathcal{N}}[\psi, F](I_j),
\eeaa
where we defined more generally $\widetilde{\mathcal{N}}[\psi, F](\tau_1,\tau_2)$ on an arbitrary interval $(\tau_1, \tau_2)$ of $\mathbb{R}$ by 
\begin{align}\lab{eq:definitionwidetildemathcalNpsijnormRHS:local}
\widetilde{\mathcal{N}}[\psi, F](\tau_1,\tau_2):={}&\widetilde{\mathcal{N}}_{trap}[\psi, F](\tau_1,\tau_2) +\sup_{\tau\in[\tau_1,\tau_2]}\bigg|\int_{\Mntrap(\tau_1, \tau)}{\pr_{\tau}\psi} F \bigg|\nn\\
&+\int_{\Mntrap(\tau_1,\tau_2)}\big(|\pr_r\psi|+r^{-1}|\psi|\big)|F|+\int_{\MM(\tau_1,\tau_2)}|F|^2,
\end{align}
with
\beaa
\widetilde{{\mathcal{N}}}_{trap}[\psi, F](\tau_1, \tau_2):=\int_{\Mtrap(\tau_1, \tau_2)}|F|{|S^1\psi|}, \quad -\infty\leq\tau_1<\tau_2\leq+\infty,
\eeaa
where {$S^1\in(\Opw(\widetilde{S}^{1,0}(\MM)))^{1+\iota}$} is defined as in \eqref{def:widetildeNtrap}. Next, we consider smooth cut-off functions $\chi_j=\chi_j(\tau)$, $j=0, \cdots, N$ such that $0\leq\chi_j\leq 1$, with $\chi_j=1$ on $I_j$ and
\beaa
\bsplit
\textrm{supp}(\chi_0)&\subset(-\infty, \tau^{(1)}+1), \qquad \textrm{supp}(\chi_j)\subset(\tau^{(j)}-1, \tau^{(j+1)}+1),\,1\leq j\leq N-1, \\
\textrm{supp}(\chi_N)&\subset(\tau^{(N)}-1, +\infty).
\end{split}
\eeaa
Then, using the fact that $\chi_j=1$ on $I_j$, we have, for a PDO $P^1\in\Opw(\widetilde{S}^{1,0}(\MM))$,
\bea\lab{eq:intermediaryestimateproofaubadditivityoverintervalsofwidetildeNN:1}
\nn\int_{\Mtrap(I_j)}|F||P^1\psi| &\leq&   \int_{\Mtrap(I_j)}|F||P^1(\chi_j\psi)|+\int_{\Mtrap(I_j)}|F||[P^1, 1-\chi_j]\psi|\\
\nn&\les& \int_{\Mtrap(I_j)}|F||P^1(\chi_j\psi)|+\left(\int_{\Mtrap(I_j)}|F|^2\right)^{\frac{1}{2}}\left(\int_{\Mtrap}|P^0_j\psi|^2\right)^{\frac{1}{2}}\\
&\les& \int_{\Mtrap(I_j)}|F||P^1(\chi_j\psi)|+\left(\int_{\Mtrap}|F|^2\right)^{\frac{1}{2}}\left(\int_{\Mtrap}|\psi|^2\right)^{\frac{1}{2}},
\eea
where $P^0_j:=[P^1, 1-\chi_j]\in\Opw(\widetilde{S}^{0,0}(\MM))$, where we used  Lemma \ref{lemma:actionmixedsymbolsSobolevspaces:MM}, and where the intervals $I_j^1$ are given by
\beaa
I_0^1:=(-\infty, \tau^{(1)}+1), \quad I_j^1:=(\tau^{(j)}-1, \tau^{(j+1)}+1), \,\,\, j=1,\cdots,N-1, \quad I_N^1:=(\tau^{(N)}-1, +\infty),
\eeaa
so that $\chi_j$ is supported in $I^1_j$. 

Next, we estimate the first term of the RHS of  \eqref{eq:intermediaryestimateproofaubadditivityoverintervalsofwidetildeNN:1}. Using the fact that $F$ is supported in $\tau\geq 1$, and relying on Lemma \ref{lem:gpert:MMtrap} with $h(\tau)=\tau^{-1-\dec}$ and $\de_0=\de$, we have 
\beaa
\int_{\Mtrap(I_j)}|F||P^1(\chi_j\psi)| &\les& \left(\int_{\Mtrap(I_j)}\tau^{1+\de}|F|^2\right)^{\frac{1}{2}}\left(\int_{\Mtrap(I_j)}\tau^{-1-\de}|P^1(\chi_j\psi)|\right)^{\frac{1}{2}}\\
&\les& \left(\int_{\Mtrap(I_j)}\tau^{1+\de}|F|^2\right)^{\frac{1}{2}}\left(\EM[\chi_j\psi](\Reals)\right)^{\frac{1}{2}}\\
&\les&  \left(\int_{\Mtrap(I_j^1)}\tau^{1+\de}|F|^2\right)^{\frac{1}{2}}\left(\EM[\psi](I^1_j)\right)^{\frac{1}{2}} \\
&\les& \la^{-1}\int_{\Mtrap(I_j^1)}\tau^{1+\de}|F|^2+\la\EM[\psi](I^1_j)
\eeaa
for any $0<\la\leq 1$. Also, using Lemma \ref{lemma:actionmixedsymbolsSobolevspaces:MM}, we have
\beaa
\int_{\Mtrap(I_j)}|F||P^1(\chi_j\psi)| &\les& \left(\int_{\Mtrap(I_j)}|F|^2\right)^{\frac{1}{2}}\left(\int_{\Mtrap(I_j)}|P^1\chi_j\psi|^2\right)^{\frac{1}{2}}\\
&\les& \left(\int_{\Mtrap(I_j)}|F|^2\right)^{\frac{1}{2}}\left(\left(\int_{\Mtrap(I_j^1)}|\pr(\chi_j\psi)|^2\right)^{\frac{1}{2}}+\left(\int_{\Mtrap}|\psi|^2\right)^{\frac{1}{2}}\right)\\
&\les& \left(\int_{\Mtrap(I_j)}|F|^2\right)^{\frac{1}{2}}\left(\left(\int_{\Mtrap(I_j^1)}|\pr\psi|^2\right)^{\frac{1}{2}}+\left(\int_{\Mtrap}|\psi|^2\right)^{\frac{1}{2}}\right)\\
&\les& \left(\int_{\Mtrap(I_j^1)}|F|^2\right)^{\frac{1}{2}}\left(\int_{\Mtrap(I_j^1)}|\pr\psi|^2\right)^{\frac{1}{2}}\\
&&+\left(\int_{\Mtrap}|F|^2\right)^{\frac{1}{2}}\left(\int_{\Mtrap}|\psi|^2\right)^{\frac{1}{2}}.
\eeaa
Since we have chosen $0<\la\leq 1$, the two last estimates imply
\beaa
&&\int_{\Mtrap(I_j)}|F||P^1(\chi_j\psi)|\\
&\les& \la^{-1}\min\left[\left(\int_{\Mtrap(I_j^1)}|F|^2\right)^{\frac{1}{2}}\left(\int_{\Mtrap(I_j^1)}|\pr\psi|^2\right)^{\frac{1}{2}}, \int_{\Mtrap(I_j^1)}\tau^{1+\de}|F|^2\right] +\la\EM[\psi](I^1_j)\\
&&+\left(\int_{\Mtrap}|F|^2\right)^{\frac{1}{2}}\left(\int_{\Mtrap}|\psi|^2\right)^{\frac{1}{2}}\\
&\les& \la^{-1}\widehat{\mathcal{N}}[\psi, F](I_j^1)+\la\EM[\psi](I^1_j)+\left(\int_{\Mtrap}|F|^2\right)^{\frac{1}{2}}\left(\int_{\Mtrap}|\psi|^2\right)^{\frac{1}{2}}
\eeaa
which, together with \eqref{eq:intermediaryestimateproofaubadditivityoverintervalsofwidetildeNN:1}, yields 
\beaa
\int_{\Mtrap(I_j)}|F||P^1\psi| &\les& \int_{\Mtrap(I_j)}|F||P^1(\chi_j\psi)|+\left(\int_{\Mtrap}|F|^2\right)^{\frac{1}{2}}\left(\int_{\Mtrap}|\psi|^2\right)^{\frac{1}{2}}\\
&\les& \la^{-1}\widehat{\mathcal{N}}[\psi, F](I_j^1)+\la\EM[\psi](I^1_j)+\left(\int_{\Mtrap}|F|^2\right)^{\frac{1}{2}}\left(\int_{\Mtrap}|\psi|^2\right)^{\frac{1}{2}}.
\eeaa
In view of the expression \eqref{eq:defintionwidehatmathcalNfpsinormRHS}
 for $\widehat{\mathcal{N}}[ \psi, F](\tau_1, \tau_2)$ and the expression \eqref{eq:definitionwidetildemathcalNpsijnormRHS:local} for $\widetilde{\mathcal{N}}[ \psi, F](\tau_1, \tau_2)$, we infer from the above, for any $0<\la\leq 1$,
\beaa
\widetilde{\mathcal{N}}[\psi, F](\mathbb{R}) &\leq& \sum_{j=0}^N\widetilde{\mathcal{N}}[\psi, F](I_j)\\
&\les& \sum_{j=0}^N\left( \la^{-1}\widehat{\mathcal{N}}[\psi, F](I_j^1)+\la\EM[\psi](I^1_j)+\left(\int_{\Mtrap}|F|^2\right)^{\frac{1}{2}}\left(\int_{\Mtrap}|\psi|^2\right)^{\frac{1}{2}}\right)\\
&\les_N& \la^{-1}\sum_{j=0}^N\widehat{\mathcal{N}}[\psi, F](I_j^1)+\la\EM[\psi](\Reals)+\left(\int_{\Mtrap}|F|^2\right)^{\frac{1}{2}}\left(\int_{\Mtrap}|\psi|^2\right)^{\frac{1}{2}}
\eeaa
as stated. This concludes the proof of Lemma \ref{lemma:additivepropertyofwidetildeN}.
\end{proof}

We will also rely on the following lemma to absorb lower order terms.
\begin{lemma}\lab{lemma:controllowerordertermusingtrickprrrmrtrap}
We have
\beaa
\int_{\Mtrap}|\pr_\tau\psi|^2 &\les& \Big(\M_{r\leq 11m}[\pr_\tau\psi](\Reals)\Big)^{\frac{1}{2}}\left(\widetilde{\M}[\psi]\right)^{\frac{1}{2}}, \\ 
\int_{\Mtrap}|\pr_{\tphi}\psi|^2 &\les& \Big(\M_{r\leq 11m}[\pr_{\tphi}\psi](\Reals)\Big)^{\frac{1}{2}}\left(\widetilde{\M}[\psi]\right)^{\frac{1}{2}}.
\eeaa
\end{lemma}

\begin{proof}
Let $\chi(r)$ be a smooth cut-off function supported in $(r_+(1-\dhor), 11m)$ with $\chi=1$ on the support of $\Mtrap$. Then, since $r_{\trap}\in\widetilde{S}^{0,0}(\MM)$ defined in \eqref{eq:definitionofrtrapinfunctionofrmaxandcutoffinmathcalG5} is such that $r_{\trap}=r_{\trap}(\Xi)$ with $\Xi=(\xit, \xi_{\tphi}, \La)$, we have $\pr_r(r_{trap})=0$ and hence $\pr_r(r-r_{\trap})=1$. Thus, we have, for any scalar function $\phi$,
\beaa
\pr_r(\Opw(r-r_{\trap})\phi) &=& \Opw(r-r_{\trap})\pr_r\phi+\Opw(\pr_r(r-r_{\trap}))\phi\\
&=& \Opw(r-r_{\trap})\pr_r\phi+\phi,
\eeaa
and hence, recalling that $d\Vref$ denotes the Lebesgue measure $d\tau d r dx^1 dx^2$ in $(\tau, r, x^1, x^2)$ coordinates (see Remark \ref{rmk:linkvolumeformMwithf0dVref}), we infer, using also Lemma \ref{lemma:spacetimevolumeformusingisochorecoordinates}, 
\beaa
\bsplit
\int_{\Mtrap}|\phi|^2\leq & \int_{\MM}\chi(r)\phi^2\\
\les& \int_{\MM}\chi(r)\phi^2d\Vref\\
=& \int_{\MM}\chi(r)\Big(\pr_r(\Opw(r-r_{\trap})\phi) - \Opw(r-r_{\trap})\pr_r\phi\Big)\phi d\Vref\\
=& -2\int_{\MM}\chi(r)\pr_r\phi\Opw(r-r_{\trap})\phi d\Vref - \int_{\MM}\chi'(r)\phi\Opw(r-r_{\trap})\phi d\Vref,
\end{split}
\eeaa
where the integration by parts in $r$ does not produce boundary terms since $\chi(r)$ is supported in $(r_+(1-\dhor), 11m)$, and where we used in the last step that $\Opw(r-r_{\trap})$ is self-adjoint w.r.t. $d\Vref$ and that $\Opw(r-r_{\trap})$ commutes with $\chi(r)$ since $\Opw(r-r_{\trap})$ is a tangential operator. This yields
\beaa
\int_{\Mtrap}|\phi|^2  &\les& \left(\int_{\MM_{r\leq 11m}}\big((\pr_r\phi)^2+\phi^2\big)d\Vref\right)^{\frac{1}{2}}\left(\int_{\MM_{r\leq 11m}}|\Opw(r-r_{\trap})\phi|^2d\Vref\right)^{\frac{1}{2}},
\eeaa
and hence, using again Lemma \ref{lemma:spacetimevolumeformusingisochorecoordinates}, 
\beaa
\int_{\Mtrap}|\phi|^2 &\les& \left(\int_{\MM_{r\leq 11m}}\big((\pr_r\phi)^2+\phi^2\big)\right)^{\frac{1}{2}}\left(\int_{\MM_{r\leq 11m}}|\Opw(r-r_{\trap})\phi|^2\right)^{\frac{1}{2}}\\
&\les& \Big(\M_{r\leq 11m}[\phi](\Reals)\Big)^{\frac{1}{2}}\left(\int_{\MM_{r\leq 11m}}|\Opw(r-r_{\trap})\phi|^2\right)^{\frac{1}{2}}.
\eeaa
Applying this inequality with $\phi=\pr_\tau\psi$ and $\phi=\pr_{\tphi}\psi$, we infer 
\beaa
\int_{\Mtrap}|\pr_\tau\psi|^2 &\les& \Big(\M_{r\leq 11m}[\pr_\tau\psi](\Reals)\Big)^{\frac{1}{2}}\left(\widetilde{\M}[\psi]\right)^{\frac{1}{2}}, \\ 
\int_{\Mtrap}|\pr_{\tphi}\psi|^2 &\les& \Big(\M_{r\leq 11m}[\pr_{\tphi}\psi](\Reals)\Big)^{\frac{1}{2}}\left(\widetilde{\M}[\psi]\right)^{\frac{1}{2}},
\eeaa
as stated. This concludes the proof of Lemma \ref{lemma:controllowerordertermusingtrickprrrmrtrap}.
\end{proof}

Relying on the microlocal energy-Morawetz norms of Definition \ref{def:microlocalenergyMorawetznorms}, we may now state microlocal energy-Morawetz estimates.

%%%%%%%%%%%%%%%%%%%%%%%%%%%

\subsection{Microlocal energy-Morawetz estimates}

%%%%%%%%%%%%%%%%%%%%%%%%%%%

The following theorem states our main microlocal energy-Morawetz estimate which is conditional on the control of lower order terms.
\begin{theorem}\lab{th:mainenergymorawetzmicrolocal}
Assuming that $\psi$ satisfies the same assumptions as in Theorem \ref{th:main:intermediary}, we have  
\bea\lab{th:eq:mainenergymorawetzmicrolocal}
\widetilde{\EMF}[\psi]\les \widetilde{\mathcal{N}}[\psi, F](\Reals)+\int_{\MM}r^{-4}|\psi|^2.
\eea
\end{theorem}

We will also rely on the following proposition to control lower order terms.
\begin{proposition}\lab{prop:energymorawetzmicrolocalwithblackbox}
Assuming that $\psi$ satisfies the same assumptions as in Theorem \ref{th:main:intermediary}, we have 
\beaa
\EMF[\psi](\Reals) &\les& \int_{\Mtrap}|\pr F|^2
+\sup_{\tau\in\Reals}\bigg|\int_{\Mntrap(-\infty, \tau)}{F\ov{\pr_{\tau}\psi}}\bigg|\\
&&+\int_{\Mntrap}\big(|\pr_r\psi|+r^{-1}|\psi|\big)|F|+\int_{\MM}|F|^2+\ep\EM^{(1)}[\psi]{(\Reals)}.
\eeaa
\end{proposition}

The proof of Theorem \ref{th:mainenergymorawetzmicrolocal} is postponed to Section \ref{sect:CondEMF:Dynamic}, while the proof of Proposition \ref{prop:energymorawetzmicrolocalwithblackbox} is postponed to Section \ref{sec:proofofprop:energymorawetzmicrolocalwithblackbox}. We are now ready to prove Theorem \ref{th:main:intermediary}.

%%%%%%%%%%%%%%%%%%%%%%%%%%%%%%%%%%

\subsection{Proof of Theorem \ref{th:main:intermediary}}
\lab{sec:proofofTheoremth:main:intermediary}

%%%%%%%%%%%%%%%%%%%%%%%%%%%%%%%%%%

The proof of Theorem \ref{th:main:intermediary} will rely in particular on the microlocal energy-Morawetz estimate of Theorem \ref{th:mainenergymorawetzmicrolocal} which we rewrite below for convenience as follows 
\bea
\lab{eq:microEMF:psi:robust}
\widetilde{\EMF}[\psi]&\les&\NNtilde[\psi, F](\Reals)+\A[\psi], \quad \A[\psi]:=\int_{\MM}r^{-4}\psi^2,
\eea
where $\psi$ is a solution of \eqref{eq:scalarwave}.

We prove the estimates \eqref{eq:th:main:intermediary1} and \eqref{eq:th:main:intermediary2} in sections \ref{sect:microlocalEMFesti:Tpsi} and \ref{sect:EMFesti:chiZ:intermediary}, respectively.

%%%%%%%%%%%%%%%%%%%%%%%%%%%%%%%%%%%%%%%%%%

\subsubsection{EMF estimate \eqref{eq:th:main:intermediary1} for $\pr_{\tau}\psi$}
\lab{sect:microlocalEMFesti:Tpsi}

%%%%%%%%%%%%%%%%%%%%%%%%%%%%%%%%%%%%%%%%%%

We commute the wave equation \eqref{eq:scalarwave} for $\psi$ with $\pr_{\tau}$ and derive
\beaa
\square_{\g}\pr_{\tau}\psi&=&\pr_{\tau} F +[\square_{\g}, \pr_{\tau}]\psi.
\eeaa
Applying \eqref{eq:microEMF:psi:robust} to the above wave equation for  $\pr_{\tau}\psi$, we deduce
\bea
\lab{eq:highEMFglobal:commT:1}
\widetilde{\EMF}[\pr_{\tau}\psi]
&\les  &\NNtilde\Big[\pr_{\tau}\psi,  \pr_{\tau} F + [\square_{\g}, \pr_{\tau}]\psi\Big](\Reals)+\A[\pr_{\tau}\psi]\nn\\
&\les & \NNtilde\big[\pr_{\tau}\psi, \pr_{\tau}F\big](\Reals) +\NNtilde\big[\pr_{\tau}\psi,  [\square_{\g}, \pr_{\tau}]\psi\big](\Reals)+\A[\pr_{\tau}\psi].
\eea

Next, we estimate the second term on the RHS of \eqref{eq:highEMFglobal:commT:1}. In view of the definition \eqref{eq:definitionwidetildemathcalNpsijnormRHS} of
 $\widetilde{\mathcal{N}}[\psi, F](\Reals)$, we have by Cauchy--Schwarz, using also the fact that $\psi$ is supported for $\tau\geq 1$, 
\bea
\lab{eq:error:commT:pre}
&&\NNtilde\big[\pr_{\tau}\psi, [\pr_{\tau}, \square_{\g}]\psi\big](\Reals)\nn\\
&\les& \left(\int_{\Mtrap(\tau\geq 1)}\tau^{-1-\dec}{|S^1\pr_\tau\psi|^2}\right)^{\frac{1}{2}}\left(\int_{\Mtrap}\tau^{1+\dec}|[\pr_{\tau}, \square_{\g}]\psi|^2\right)^{\frac{1}{2}} 
+\int_{\MM}|[\pr_{\tau}, \square_{\g}]\psi|^2\nn\\
&&+\sup_{\tau\in\Reals}\bigg|\int_{\Mntrap(-\infty, \tau)}[ \pr_{\tau}, \square_{\g}]\psi \pr_{\tau}(\pr_{\tau}\psi)\bigg|
+ \int_{\Mntrap}\big|[\pr_{\tau}, \square_{\g}]\psi\big|\big|(\pr_r, r^{-1})\pr_{\tau}\psi\big|.
\eea
In order to estimate the last three terms on the RHS of \eqref{eq:error:commT:pre}, we     use Lemma \ref{lem:commutatorwithwave:firstorderderis} which yields 
\beaa
\, [ \square_{\g}, \pr_{\tau}]\psi&=&-\pr_{\tau}(\gcheck^{\a\b})\pr_{\a}\pr_{\b}\psi+
\dk^{\leq 2}\Ga_g\c\dk\psi,
\eeaa
so that, in view of \eqref{eq:controloflinearizedinversemetriccoefficients}, we may apply Lemma \ref{lemma:basiclemmaforcontrolNLterms:bis} to obtain
\bea\lab{eq:error:commT:pre:bissss}
\nn&&\int_{\MM}|[\pr_{\tau}, \square_{\g}]\psi|^2+\sup_{\tau\in\Reals}\bigg|\int_{\Mntrap(-\infty, \tau)}[ \pr_{\tau}, \square_{\g}]\psi \pr_{\tau}(\pr_{\tau}\psi)\bigg|
+ \int_{\Mntrap}\big|[\pr_{\tau}, \square_{\g}]\psi\big|\big|(\pr_r, r^{-1})\pr_{\tau}\psi\big|\\
&\les& \ep\EM^{(1)}[\psi](\Reals).
\eea
Also, applying Lemma \ref{lem:gpert:MMtrap} with $h=\tau^{-1-\dec}\mathbf{1}_{\tau\geq 1}$, $r_0=10m$, {$S=S^1$} and $\de_0=\dec$, we obtain 
\beaa
\int_{\Mtrap(\tau\geq 1)}\tau^{-1-\dec}{|S^1\pr_\tau\psi|^2}\les \EM^{(1)}[\psi](\Reals),
\eeaa
and we also have
\beaa
\int_{\Mtrap}\tau^{1+\dec}|[\pr_{\tau}, \square_{\g}]\psi|^2\les \ep^2\int_{\Mtrap}\tau^{-1-\dec}|\pr^{\leq 2}\psi|^2\les \ep^2\sup_{\tau\in\Reals}\E^{(1)}[\psi](\tau).
\eeaa
Together with \eqref{eq:error:commT:pre} and \eqref{eq:error:commT:pre:bissss}, we infer
\bea\lab{eq:controlofwidetildeNprtaucommprtauandsquaregonReals}
\NNtilde\big[\pr_{\tau}\psi, [\pr_{\tau}, \square_{\g}]\psi\big](\Reals) \les \ep\EM^{(1)}[\psi](\Reals).
\eea
Substituting this back to \eqref{eq:highEMFglobal:commT:1}, we infer
\beaa
\widetilde{\EMF}[\pr_{\tau}\psi]
&\les & \NNtilde[\pr_{\tau}\psi,  \pr_{\tau} F](\Reals) +\ep \EM^{(1)}[\psi](\Reals)+\A[\pr_{\tau}\psi].
\eeaa
Now, using Lemma \ref{lemma:controllowerordertermusingtrickprrrmrtrap}, we have
\beaa
\A[\pr_{\tau}\psi] &=& \int_{\MM}r^{-4}|\pr_\tau\psi|^2\les \int_{\Mtrap}|\pr_\tau\psi|^2+\M[\psi](\Reals)\\
&\les& (\M[\pr_{\tau}\psi](\Reals))^{\frac{1}{2}}\big(\widetilde{\M}[\psi]\big)^{\frac{1}{2}}+\M[\psi](\Reals)
\eeaa
which yields
\beaa
\widetilde{\EMF}[\pr_{\tau}\psi] &\les&  \NNtilde[\pr_{\tau}\psi,  \pr_{\tau} F ](\Reals) +\ep \EM^{(1)}[\psi](\Reals)+(\M[\pr_{\tau}\psi](\Reals))^{\frac{1}{2}}\big(\widetilde{\M}[\psi]\big)^{\frac{1}{2}}+\M[\psi](\Reals)
\eeaa
and hence
\beaa
\widetilde{\EMF}[\pr_{\tau}\psi] &\les & \NNtilde[\pr_{\tau}\psi,  \pr_{\tau} F ](\Reals) +\ep \EM^{(1)}[\psi](\Reals)+\widetilde{\M}[\psi].
\eeaa
Together with Theorem \ref{th:mainenergymorawetzmicrolocal}, we deduce
\beaa
\widetilde{\EMF}[\psi]+\widetilde{\EMF}[\pr_{\tau}\psi]
&\les & \NNtilde^{(1)}[\psi,  F](\Reals) +\ep \EM^{(1)}[\psi](\Reals)+\int_{\MM}r^{-4}|\psi|^2.
\eeaa
Finally, using Proposition \ref{prop:energymorawetzmicrolocalwithblackbox} to control the lower order term on the RHS, we deduce
\beaa
\EMF[\psi](\Reals)+\EMF[\pr_{\tau}\psi](\Reals)
&\les & \NNtilde^{(1)}[\psi,  F ](\Reals) +\ep \EM^{(1)}[\psi](\Reals),
\eeaa
which concludes the proof of \eqref{eq:th:main:intermediary1}.

%%%%%%%%%%%%%%%%%%%%%%%%%%%%%%%%%%%%%%%%%%

\subsubsection{EMF estimate \eqref{eq:th:main:intermediary2} for $\pr_{\tphi}\psi$}
\label{sect:EMFesti:chiZ:intermediary}

%%%%%%%%%%%%%%%%%%%%%%%%%%%%%%%%%%%%%%%%%%

We commute the wave equation \eqref{eq:scalarwave} with $\chi_0(r)\pr_{\tphi}$, where $\chi_0$ is a smooth cutoff function that equals $1$ for $r\leq 11m$ and vanishes for $r\geq 12m$, and we obtain the following wave equation for $\chi_0\pr_{\tphi} \psi$
\beaa
\square_{\g}(\chi_0\pr_{\tphi}\psi)&=& \chi_0 \pr_{\tphi} F + [\square_{\g}, \chi_0]\pr_{\tphi}\psi +\chi_0[\square_{\g}, \pr_{\tphi}]\psi\\
&=& \chi_0 \pr_{\tphi} F + (\chi_0', \chi_0'')\pr^{\leq 2}\psi +\chi_0[\square_{\g}, \pr_{\tphi}]\psi.
\eeaa
Applying \eqref{eq:microEMF:psi:robust} to the above wave equation for $\chi_0\pr_{\tphi}\psi$, and using the support properties of $\chi_0$, $\chi_0'$ and $\chi_0''$, we deduce
\bea
\lab{eq:EM:AsympKerr:highLiePhi:v0}
&&\widetilde{\EMF}[\chi_0 \pr_{\tphi}\psi]\nn\\
&\les  &\NNtilde\Big[\chi_0 \pr_{\tphi}\psi, \chi_0 \pr_{\tphi} F + (\chi_0', \chi_0'')\pr^{\leq 2}\psi +\chi_0[\square_{\g}, \pr_{\tphi}]\psi\Big](\Reals) +\A[\chi_0 \pr_{\tphi}\psi]\nn\\
&\les & \NNtilde[\chi_0 \pr_{\tphi}\psi,  \chi_0 \pr_{\tphi}F](\Reals) 
+\int_{\MM_{11m, 12m}}|\pr^{\leq 2}\psi|^2 +\NNtilde\big[\chi_0 \pr_{\tphi}\psi, \chi_0[\square_{\g}, \pr_{\tphi}]\psi\big](\Reals)+\A[\chi_0\pr_{\tphi}\psi]\nn\\
\nn&\les&\NNtilde[\chi_0 \pr_{\tphi}\psi,  \chi_0 \pr_{\tphi}F ](\Reals) 
+\M^{(1)}_{11m,12m} [\psi](\Reals) +\NNtilde\big[\chi_0 \pr_{\tphi}\psi, \chi_0[\square_{\g}, \pr_{\tphi}]\psi\big](\Reals)\\
&& +\A[\chi_0 \pr_{\tphi}\psi].
\eea

Next, we estimate the last two terms on the RHS of \eqref{eq:EM:AsympKerr:highLiePhi:v0}. First, proceeding exactly as in the proof of \eqref{eq:controlofwidetildeNprtaucommprtauandsquaregonReals}, we obtain the analog estimate for $\chi_0 \pr_{\tphi}\psi$, i.e., 
\beaa
\NNtilde\big[\chi_0 \pr_{\tphi}\psi, \chi_0[\square_{\g}, \pr_{\tphi}]\psi\big](\Reals) \les \ep\EM^{(1)}[\psi](\Reals).
\eeaa
Also, using Lemma \ref{lemma:controllowerordertermusingtrickprrrmrtrap}, we have
\beaa
\A[\chi_0 \pr_{\tphi}\psi] &=& \int_{\MM_{r\leq 12m}}|\pr_{\tphi}\psi|^2\les \int_{\Mtrap}|\pr_{\tphi}\psi|^2+\M[\psi](\Reals)\\
&\les& \Big(\M_{r\leq 11m}[\pr_{\tphi}\psi](\Reals)\Big)^{\frac{1}{2}}\Big(\widetilde{\M}[\psi]\Big)^{\frac{1}{2}}+\M[\psi](\Reals).
\eeaa
Plugging the above estimates in \eqref{eq:EM:AsympKerr:highLiePhi:v0}, we infer
\beaa
\widetilde{\EMF}[\chi_0 \pr_{\tphi}\psi] &\les&\NNtilde[\chi_0 \pr_{\tphi}\psi,  \chi_0 \pr_{\tphi}F ](\Reals) 
+\M^{(1)}_{11m,12m} [\psi](\Reals) +\ep\EM^{(1)}[\psi](\Reals)\\
&& +\Big(\M_{r\leq 11m}[\pr_{\tphi}\psi](\Reals)\Big)^{\frac{1}{2}}\Big(\widetilde{\M}[\psi]\Big)^{\frac{1}{2}}+\M[\psi](\Reals)
\eeaa
and hence
\beaa
\widetilde{\EMF}_{r_+(1-\dhor), 11m}[\pr_{\tphi}\psi] &\les &\NNtilde[\chi_0 \pr_{\tphi}\psi,  \chi_0 \pr_{\tphi}F](\Reals) 
+\M^{(1)}_{11m,12m} [\psi](\Reals)+\ep\EM^{(1)}[\psi](\Reals)+\widetilde{\M}[\psi].
\eeaa
Together with Theorem \ref{th:mainenergymorawetzmicrolocal}, we deduce
\beaa
\widetilde{\EMF}_{r_+(1-\dhor), 11m}[\pr_{\tphi}\psi]
&\les & \NNtilde^{(1)}[\psi,  F](\Reals) +\ep \EM^{(1)}[\psi](\Reals)+\int_{\MM}r^{-4}|\psi|^2.
\eeaa
Finally, using Proposition \ref{prop:energymorawetzmicrolocalwithblackbox} to control the lower order term on the RHS, we deduce
\beaa
\EMF_{r_+(1-\dhor), 11m}[\pr_{\tphi}\psi](\mathbb{R}) &\les& \widetilde{\mathcal{N}}^{(1)}[\psi, F](\Reals)+\ep \EM^{(1)}[\psi](\mathbb{R})+\M_{11m, 12m}[\pr\psi](\mathbb{R}),
\eeaa
which concludes the proof of \eqref{eq:th:main:intermediary2}.

%%%%%%%%%%%%%%%%%%%%%%%%%%%%%%%%%%%%%%%%%%

\section{Proof of Theorem \ref{th:mainenergymorawetzmicrolocal}}
\label{sect:CondEMF:Dynamic}

%%%%%%%%%%%%%%%%%%%%%%%%%%%%%%%%%%%%%%%%%%

The proof of Theorem \ref{th:mainenergymorawetzmicrolocal} relies on a microlocal approach:
\begin{itemize}
\item inspired by the one in \cite{Ta-Toh}, where we replace the mixed symbols differential in $\tau$ and microlocal on $\Si(\tau)$ of \cite{Ta-Toh} by the mixed symbols differential in $r$ and microlocal on $H_r$ of Section \ref{sect:PDO:rfoliation},

\item closely following the way of handling high frequencies in phase space in \cite{DRSR} for the inhomogeneous scalar wave equation in a subextremal Kerr spacetime. 
\end{itemize}
This approach allows us to derive the microlocal energy-Morawetz estimates conditional on lower order derivatives of Theorem \ref{th:mainenergymorawetzmicrolocal} on a dynamic background that is asymptotically approaching any subextremal Kerr background.

The rest of this section is organized as follows. We start by discussing in Section \ref{subsect:discussionofsquareintegrabilityintimetoapplyPDOop} the square integrability properties of $\psi$ w.r.t. $\tau$ that are necessary to justify the various computations involving PDOs with mixed symbols on $\MM$. We then discuss in Section \ref{subsect:PrinSym:RescaleWave} the principal symbol of the rescaled wave operator $|q|^2 \Box_{\gam}$ in the normalized coordinates and derive a pseudodifferential version of the energy identity in Section \ref{subsect:GeneralEnerIden:PD}. Next, a few general choices of microlocal multipliers in the energy identity are given in Section \ref{subsect:PDcurrents}, and we make appropriate choices of these microlocal multipliers to derive conditional degenerate Morawetz estimates in Kerr spacetimes,  nondegenerate Morawetz-flux estimates in perturbations of Kerr spacetimes, and nondegenerate energy-Morawetz-flux estimates in perturbations of Kerr spacetimes in Sections \ref{subsect:scalarwave:Kerr:condiMora}--\ref{subsect:CondDegMFesti:Kerr}, \ref{subsect:CondMoraFlux:Kerrperturb}, and \ref{subsect:proofoftheorem64}, respectively. The proof of Theorem \ref{th:mainenergymorawetzmicrolocal} is then concluded at the end of Section \ref{subsect:proofoftheorem64}.

%%%%%%%%%%%%%%%%%%%%%%%%%%%%%%%%%%%%%%%

\subsection{Square integrability in $\tau$ for $\psi$}
\label{subsect:discussionofsquareintegrabilityintimetoapplyPDOop}

%%%%%%%%%%%%%%%%%%%%%%%%%%%%%%%%%%%%%%%

By the assumptions of Theorem \ref{th:mainenergymorawetzmicrolocal}, $\psi$ is a solution to the following inhomogeneous wave equation on $(\MM, \g)$
\begin{align}
\Box_{\g}\psi=F,
\end{align}
where $\g$ satisfies the assumptions of Section \ref{subsubsect:assumps:perturbedmetric} and coincides with $\gam$ for $\tau\leq 1$ and for $\tau\geq \tau_*$ with $\tau_*$ arbitrarily large, where $\psi=0$ for $\tau\leq 1$, and where $F$ is supported in $(1,\tau_*)$. In particular, under the additional assumption
\bea\lab{eq:additionalnonquantiativeassumptionforFtoproveaprioriintegrabilityintime}
\int_{\MM(1,\tau_*)}r^{1+\de}|\pr^{\leq 1}F|^2<+\infty,
\eea
as a consequence of:
\begin{itemize}
\item the fact that $\psi=0$ for $\tau\leq 1$,

\item the local energy estimates of Lemma \ref{lemma:localenergyestimate} applied with $\tau_0=1$, $q=\tau_*-1$ and $s=1$, which, together with Lemma \ref{lemma:higherorderenergyMorawetzestimates}, implies
\beaa
\EMF^{(1)}_\de[\psi](1,\tau_*)\les_{\tau_*} \int_{\MM(1,\tau_*)}r^{1+\de}|\pr^{\leq 1}F|^2,
\eeaa

\item and the fact that $\psi$ satisfies $\square_{\gam}\psi=0$ for $\tau\geq \tau_*$, so that we may apply \cite[Theorem 3.2]{DRSR} in $\MM(\tau_*, +\infty)$ which  implies
\beaa
\EMF^{(1)}_\de[\psi](\tau_*,+\infty)\les \E^{(1)}[\psi](\tau_*),
\eeaa
\end{itemize}
we infer 
\beaa
\EMF^{(1)}_\de[\psi](\Reals)\les_{\tau_*} \int_{\MM(1,\tau_*)}r^{1+\de}|\pr^{\leq 1}F|^2<+\infty.
\eeaa
This implies in particular that $\pr^{\leq 1}\psi$ is square integrable in $\tau\in\Reals$ on any region $\MM_{r\leq R}$ for $20m\leq R<+\infty$ and thus allows to justify the various computations in this section involving PDOs with mixed symbols on $\MM_{r\leq R}$ introduced in Section \ref{sect:PDO:rfoliation}. Finally, note that we may always reduce to the case where \eqref{eq:additionalnonquantiativeassumptionforFtoproveaprioriintegrabilityintime} holds by a density argument.

%%%%%%%%%%%%%%%%%%%%%%%%%%%%%%%%%%%%%%%

\subsection{The rescaled wave operator in normalized coordinates}
\label{subsect:PrinSym:RescaleWave}

%%%%%%%%%%%%%%%%%%%%%%%%%%%%%%%%%%%%%%%

We have introduced the rescaled wave operator $f_0\square_{\gam}$ in Lemma \ref{lem:WeylquantizedtoBoxg}, where $f_0=|q|^2$ in Kerr, see Lemma  \ref{eq:spacetimevolumeformusingisochorecoordinates}. We start with the computation of the symbol of $|q|^2\square_{\gam}$.

%%%%%%%%%%%%%%%%%%%%%%%%%%%%%%%%%%%%%%%%%

\subsubsection{Symbol of the rescaled wave operator $|q|^2\Box_{\gam}$ in normalized coordinates}

%%%%%%%%%%%%%%%%%%%%%%%%%%%%%%%%%%%%%%%%%

In view of Lemma \ref{lem:WeylquantizedtoBoxg}, taking into account that $f_0=|q|^2$ in Kerr, see Lemma  \ref{eq:spacetimevolumeformusingisochorecoordinates}, we have
\beaa
|q|^2\Box_{\gam}=\Opw (-|q|^2\gam^{\a\b}\xi_{\a}\xi_{\b})+\Opw(\widetilde{S}^{0,0}(\MM)),
\eeaa
which we may also rewrite, in view of Remark \ref{rmk:symbolesofimportantoperators}, as 
\beaa
\sigma(|q|^2\Box_{\gam})= -|q|^2\gam^{\a\b}\xi_{\a}\xi_{\b}+\widetilde{S}^{0,0}(\MM),\qquad \sigma(|q|^2\Box_{\gam})\in\widetilde{S}^{2,2}(\MM),
\eeaa
where the notation $\sigma$ for the symbol of an operator has been introduced in  Definition \ref{def:notationforthesymbolofWeylquantizationmixedsymbol}. By the formula \eqref{eq:inverse:hypercoord} for the components of the inverse metric $\gam^{\a\b}$ in the normalized coordinates $(\tt, r, \th,\tphi)$, we infer
\bea\lab{eq:computationofsymbolmodqsquareboxgam}
\sigma(\qs{\Box}_{\gam} )&=&-\Delta\xi_r^2 -2\S_1\xi_r+\S_2 +\widetilde{S}^{0,0}(\MM),
\eea
where the symbols $\S_1$ and $\S_2$ are
\bea\lab{eq:computationofsymbolmodqsquareboxgam:1}
\bsplit
\S_1:=& (\R)(1-\mu \tmod')\xit+(a-\Delta\phimod')\xiphi, \qquad 
\S_1\in\widetilde{S}^{1,0}(\MM),\\
\S_2:=& - \Lambda^2
-\big(2a(1-\tmod')-2(\R)\phimod' (1-\mu \tmod')\big)\xit\xiphi\\
&+\big(2(\R)\tmod'-\Delta(\tmod')^2\big)\xit^2
-({\Delta}(\phimod')^2-2a\phimod' )\xiphi^2, \qquad \S_2\in\widetilde{S}^{2,0}(\MM),
\end{split}
\eea
with $\xit$, $\xi_r$, $\xiphi$ and $\La$ defined in \eqref{eq:basicparticularsymbolsusedeverywhere}.

We can alternatively express $\S_2$ as
\beaa
\S_2&=&\frac{(\R)^2}{\Delta}\xit^2 +\frac{4amr}{\Delta}\xit\xiphi +\frac{a^2}{\Delta}\xiphi^2 -\Lambda^2 -\frac{1}{\Delta} (\S_1)^2.
\eeaa
By denoting
\bea
\label{def:2symbol:wavepotential}
V:=\frac{\Delta\Lambda^2 -4amr\xit\xiphi-a^2 \xiphi^2}{(\R)^2}, \qquad V\in \widetilde{S}^{2,0}(\MM),
\eea
one finds 
\bea\lab{eq:alternativeexpressionforS2}
\S_2=\frac{(\R)^2}{\Delta} (\xit^2- V) - \frac{1}{\Delta}(\S_1)^2.
\eea 

Next, note that we have, in view of \eqref{def:xirstar}, 
\beaa
\xirstar = \mu \xi_r +\frac{\S_1}{\R}.
\eeaa
Thus we may rewrite the symbol of the rescaled wave operator $\qs{\Box}_{\gam}$ as follows 
\bea
\label{eq:ps:rescalewave:normalize}
\sigma(\qs{\Box}_{\gam}) =-\mu^{-1} (\R) \xirstar^2 +\BLS_2 +\widetilde{S}^{0,0}(\MM), \qquad \BLS_2 := \frac{(\R)^2}{\Delta}(\xit^2-V),
\eea
where we have used the fact that $\S_2=\BLS_2-\frac{1}{\Delta}(\S_1)^2$.

Also, a  simple calculation shows that\footnote{Notice that $ik_+$ is the symbol of $\partial_{t}+\omega_{\HH}\partial_{\phi}$ which is the Killing null generator on the future event horizon.} near $r=r_+(1-\dhor)$, 
\bea
\label{property:F1symbol}
\S_1 =(\R)k_+ + O(|\mu|)(\xit, a\xiphi), \qquad k_+:=\xit + \om_{\HH}\xiphi, \qquad \om_{\HH}:=\frac{a}{2mr_+}.
\eea

\begin{remark}\lab{rmk:wemayalwaysreducetothecaseageq0!!}
Notice that $\S_1$, $\S_2$ and $V$ defined above are invariant under the change 
\beaa
(a, \phimod', \xi_{\tphi})\to (-a, -\phimod', -\xi_{\tphi}),
\eeaa
and notice also that $\phimod'\to -\phimod'$ if $a\to -a$ in view of our choice for $\phimod$, see Remark \ref{rmk:phimodprimeisproportionaltoa!!}. This observation allows to reduce the analysis, from now on, to the case $a\geq 0$. 
\end{remark}

Notice that we have
\beaa
\Lambda^2 &=& \mathring{\ga}^{bc}\xi_b\xi_c+a^2\sin^2\th\xi_0^2\\
&=& \xitheta^2 +(\sin\th)^{-2} \xiphi^2 +a^2\sin^2\theta \xit^2,
\eeaa
where $\xitheta:=\langle \xi, \pr_{\th}\rangle$ and ${\xiphi}=\langle \xi, \partial_{\tphi}\rangle$, which immediately implies
\bea
\label{eq:Lambdasquare:bound}
\Lambda^2 \geq \max\big\{|\xiphi|^2, 2|a\xiphi\xit|\big\}. 
\eea

%%%%%%%%%%%%%%%%%%%%%%%%

\subsubsection{Properties of the symbol $V$}

%%%%%%%%%%%%%%%%%%%%%%%%

We discuss in the following lemmas a few properties of the symbol $V$ as defined in \eqref{def:2symbol:wavepotential}. Most of these properties are shown in \cite[Section 6]{DRSR}, and we provide the proof of the additional properties below. Notice that some of these properties rely on the reduction to the case $a\geq 0$ assumed for convenience  in view of Remark \ref{rmk:wemayalwaysreducetothecaseageq0!!}.

First, we collect some monotonicity properties for the symbol $V$. 

\begin{lemma}
\label{lem:potentialproperty:PS}
For any $(\xit, \xiphi, \Lambda)$,  the symbol $V= V(r, \xit, \xiphi, \Lambda)$ defined in \eqref{def:2symbol:wavepotential}, viewed as a scalar function of $r$ on $(r_+, +\infty)$, either
\begin{itemize}
\item is strictly decreasing,
\item or has a unique critical point $r_{\text{max}}$, which is a global maximum,
\item or has exactly two critical points $r_+< r_{\text{min}}<r_{\text{max}}<\infty$ that are a local minimum and maximum, respectively.
\end{itemize}

If $r_{\text{min}}$ exists, then $\xit^2 >V(r_{\text{min}})$ since $\xit^2 -V(r_+)=(\xit + \omega_{\HH} \xiphi)^2 \geq 0$.

The critical point $r_{\text{max}}$, if it exists, is bounded by
\bea
\lab{eq:UpperBoundforrmaxtrapOfthePotential}
r_{\text{max}}\leq 8m.
\eea
\end{lemma}

\begin{proof}
All the statements, except for the upper bound 
\bea\lab{eq:explicitupperboundforrmatrapleq8mtobeproved}
{r_{\text{max}}\leq 8m,}
\eea
are proved in \cite[Section 6]{DRSR}. The proof uses the fact that,  in view of formula \eqref{eq:firstandsecondorderderiofV} below, the quantity $\frac{d}{dr}((\R)^3\frac{d}{dr} V)$ either is negative in {$(r_+, +\infty)$}, or is positive for $r\in {(r_+, r_1)}$ and negative for $r\in [r_1,+\infty)$, where 
\beaa
r_1 := m\left(1+\frac{2a\xit\xiphi}{\Lambda^2}\right)+\left(
m^2\left(1+\frac{2a\xit\xiphi}{\Lambda^2}\right)^2-\frac{a^2}{3}\left(1-\frac{2\xiphi^2}{\Lambda^2}\right)\right)^{\frac{1}{2}}
\eeaa
 is the largest root of equation $\frac{d}{dr}((\R)^3\frac{d}{dr} V)=0$. 

In the following, we focus on the proof of the explicit upper bound \eqref{eq:explicitupperboundforrmatrapleq8mtobeproved} for {$r_{\text{max}}$}. Recall the expression of the potential $V$ from \eqref{def:2symbol:wavepotential}:
\beaa
V=\frac{\Delta\Lambda^2 -4amr\xit\xiphi-a^2 \xiphi^2}{(\R)^2}.
\eeaa
Its first and second order derivatives can be computed as
\bea
\label{eq:firstandsecondorderderiofV}
\begin{split}
&(\R)^3\frac{d}{dr} V=-2(r^3- 3mr^2 +a^2 r + a^2m) \Lambda^2 +4a^2 r \xiphi^2 + 4am(3r^2-a^2)\xiphi\xit,\\
&\frac{d}{dr}\bigg((\R)^3\frac{d}{dr} V\bigg)=-2 (3r^2 - 6mr +a^2) \Lambda^2 +4a^2 \xiphi^2 +24amr\xit\xiphi .
\end{split}
\eea
In view of \eqref{eq:Lambdasquare:bound}, one finds for $r\geq 6m$ and $\La^2>0$ that 
\beaa
\frac{d}{dr}\bigg((\R)^3\frac{d}{dr} V\bigg)= -2(3r^2 -15mr+a^2)\Lambda^2 - 12mr(\Lambda^2 - 2a\xit\xiphi) - (6mr\Lambda^2 - 4a^2\xiphi^2)< 0.
\eeaa
Hence, for $r> 8m$ and $\La^2>0$, we deduce
\beaa
&&\left[(\R)^3\frac{d}{dr} V\right]_{\vert_{r> 8m}}\\
&<& \left[(\R)^3\frac{d}{dr} V\right]_{\vert_{r=8m}}\nn\\
&= & \left[-2r^2\bigg(5m\Lambda^2 -6am\xit\xiphi - \frac{a^2}{r}\xiphi^2\bigg) -2a^2 r(\Lambda^2 -\xiphi^2) -2a^2m(\Lambda^2+2a\xit\xiphi)\right]_{\vert_{r=8m}}\\
&<&0,
\eeaa
which then yields $r_{\text{max}}\leq 8m$ for all triplets $(\xit, \xiphi, \Lambda)$ with $\La^2>0$. This concludes the proof of \eqref{eq:explicitupperboundforrmatrapleq8mtobeproved} and of Lemma \ref{lem:potentialproperty:PS}.
\end{proof}

Next, we introduce and discuss the superradiant frequencies and trapped frequencies. 

\begin{definition}[Superradiant and trapped frequencies]
\label{def:SRfreq:PS}
Superradiant and trapped frequencies are defined as follows:
\begin{enumerate}[label=\arabic*)]
\item Superradiant frequencies are defined as the set of frequency triplets $(\xit, \xiphi, \Lambda)$ that satisfy
\bea
\xit(\xit+\omega_{\HH}\xiphi) <0.
\eea
This inequality is equivalent to $0<-\xit \xiphi<\omega_{\HH}\xiphi^2$.

\item The trapped frequencies are defined as the set of  the frequency triplets $(\xit,\xiphi,\Lambda)$ such that there exists a radii $\tilde{r}\in (r_+, \infty)$ such that $\xit^2 - V(\tilde{r})=0$ and $\pr_rV(\tilde{r})=0$. The set of all such radius  $\tilde{r}$ is called the trapping region.
\end{enumerate}
\end{definition}

\begin{remark}
For trapped frequencies $(\xit,\xiphi,\Lambda)$, notice that $\tilde{r}$, as introduced in Definition \ref{def:SRfreq:PS}, is unique and coincides with $r_{max}$. 
\end{remark}

The following lemma contains useful properties of superradiant and trapped frequencies. Note that some of these properties can be found in \cite[Section 6]{DRSR}.

\begin{lemma}
\label{lem:SRnottrap}
The following properties hold true:
\begin{enumerate}[label=\arabic*)]
\item\lab{point1:potentialV:maxandminbounds:frequencies} {There exists a constant $\b=\b(m,a)>0$, where $\b\gtrsim m^{-2}(m-a)$ and $\b$ degenerates as $a\rightarrow m$,} such that for any $(\xit, \xiphi,\Lambda)$ satisfying  $-\xit\xiphi \leq \omega_{\HH} \xiphi^2 +\b\Lambda^2$, {$V$ has a unique critical point $r_{\text{max}}$, which is a global maximum with $r_{\text{max}}-r_+\gtrsim m-a$, and satisfies} 
\bsub
\lab{property:potential:superradiant:derivatives}
\begin{align}
\label{eq:V:case1:property2}
-(r-r_{\text{max}})\partial_r V\geq{}& b\Lambda^2 \frac{(r-r_{\text{max}})^2}{r^4}, &\forall r\in &{(r_+,\infty)},\\
\lab{eq:V:case1:property3:secondorderderi}
\frac{d^2 V}{dr^2}(r_{\text{max}})\leq{}& -b\Lambda^2,
\end{align}
{where $b\gtrsim m^{-4}(m-a)$ is a constant depending only on $m$ and $a$.}
\esub

\item\lab{point2:potentialV:maxandminbounds:frequencies}  {There exists a constant $\b=\b(m,a)>0$, where $\b\gtrsim m^{-2}(m-a)$ and $\b$ degenerates as $a\rightarrow m$,} such that  for any ${0}< -\xit\xiphi \leq \omega_{\HH} \xiphi^2 +\b\Lambda^2$, {$V$ has a unique critical point $r_{\text{max}}$, which is a global maximum {with} $r_{\text{max}}-r_+\gtrsim m-a$, and satisfies
\bea 
\lab{eq:V:case2:property4:superradiantisnottrapped}
V(r_{\text{max}})-\xit^2 \geq b\Lambda^2,
\eea
where $b\gtrsim m^{-4}(m-a)^2$ is a constant depending only on $m$ and $a$.}
This estimate {is a quantitative version of the well-known fact that superradiant frequencies are not trapped.}
\end{enumerate}
\end{lemma}

\begin{proof}
To begin with, note that throughout this proof, the constant $b$ may vary from line to line but it always depends only on the values of $m$ and $a$.

We first show point \ref{point1:potentialV:maxandminbounds:frequencies}. By the formula \eqref{eq:firstandsecondorderderiofV}, one finds
\beaa
(r^2+a^2)^3\frac{d}{dr}V\vert_{r=r_+}=2(r_+^2+a^2)(r_+-m)\Lambda^2 +4a^2r_+\xiphi^2 +4am(3r_+^2-a^2)\xiphi\xit.
\eeaa
If $-\xiphi\xit\leq 0$, then \eqref{eq:property:V:case1} below clearly holds true. Instead, if $0<-\xiphi\xit\leq \omega_{\HH} \xiphi^2+\b\Lambda^2$, we deduce from the above equality that
\beaa
&&(r^2+a^2)^3\frac{d}{dr}V\vert_{r=r_+}\\
&\geq&2(r_+^2+a^2)(r_+-m)\Lambda^2 +4a^2r_+\xiphi^2 -\frac{2a^2(3r_+^2-a^2)}{r_+}\xiphi^2
-4am(3r_+^2-a^2)\b\Lambda^2\\
&\geq&2(r_+-m)((r_+^2+a^2)\Lambda^2 -2a^2\xiphi^2)-4am(3r_+^2-a^2)\b\Lambda^2.
\eeaa
Using \eqref{eq:Lambdasquare:bound}, it then follows that there exists a small constant $\b=\b(m,a)>0$, where $\b\gtrsim {m^{-2}(m-a)}$ and $\b\to 0$ as $|a|\to m$, such that for all $-\xit\xiphi \leq \omega_{\HH} \xiphi^2 +\b\Lambda^2$, the following estimate 
{\begin{align}
\lab{eq:property:V:case1}
\pr_r V (r_+)\geq b\Lambda^2,
\end{align}
where $b\gtrsim m^{-4}(m-a)$ is a constant depending only on $m$ and $a$.}
 Since the derivative of $V$ at $r_+$ is positive, by Lemma \ref{lem:potentialproperty:PS}, $V$ has a unique critical point $r_{\text{max}}$, which is a global maximum. {Furthermore, since $|\pr_{rr}V|\les \Lambda^2$, this critical point $r_{\text{max}}$ is uniformly bounded away from $r_+$ and satisfies $r_{\text{max}}-r_+\gtrsim m-a$. }  

Recall from the proof of Lemma \ref{lem:potentialproperty:PS} that $\frac{d}{dr}\big((r^2+a^2)^3\frac{d}{dr}V\big)$ is either nonpositive in $(r_+, \infty)$ or has a unique point $r_+\leq r_1<r_{\text{max}}$ such that it is positive for $r<r_1$ and negative for $r>r_1$. Since $|\pr_{rr} V|\lesssim r^{-4}\Lambda^2$, we have for $\dhor\ll \frac{m-a}{m}$ sufficiently small that the estimate \eqref{eq:V:case1:property2} holds true for $r\in [r_+(1-\dhor), r_+]$, thus we only show the estimate \eqref{eq:V:case1:property2} for $r\in [r_+, \infty)$. Recall from the proof of Lemma \ref{lem:potentialproperty:PS} that $\frac{d}{dr}\big((r^2+a^2)^3\frac{d}{dr}V\big)$ is either nonpositive in $[r_+, \infty)$ or has a unique point $r_+\leq r_1<r_{\text{max}}$ such that it is positive for $r<r_1$ and negative for $r>r_1$. Since $|\pr_{rr} V|\lesssim r^{-4}\Lambda^2$, there exists a $\tilde{r}_1$ with $\tilde{r}_1-r_+\gtrsim m-a$ such that
\bea
\lab{eq:property:pr_rV:rleqtilder_1}
\pr_r V\geq \frac{b}{2} \Lambda^2 , \qquad \forall r\in[r_+, \tilde{r}_1],
\eea 
where $b\gtrsim m^{-4}(m-a)$ is a constant depending only on $m$ and $a$.

Let us consider the first case. For $r\geq \tilde{r}_1$, it follows that
there is a constant $c>0$ such that
\beaa
\frac{d}{dr}\left((r^2+a^2)^3\frac{d}{dr}V\right)\leq -cr^2\Lambda^2, \qquad \forall r\geq \tilde{r}_1
\eeaa
since $\frac{d}{dr}\big((r^2+a^2)^3\frac{d}{dr}V\big)$ is negative for $r\geq \tilde{r}_1$ and since $\frac{d}{dr}\big((r^2+a^2)^3\frac{d}{dr}V\big)\lesssim -r^2\Lambda^2$ for $r$ large,
hence
\beaa
-(r-r_{\text{max}})\partial_r V\geq b\Lambda^2 \frac{(r-r_{\text{max}})^2}{r^4}, \qquad\forall r\in [\tilde{r}_1,\infty),
\eeaa
with $b\gtrsim m^{-4}(m-a)$.
Together with the proven estimate \eqref{eq:property:pr_rV:rleqtilder_1}, this proves \eqref{eq:V:case1:property2} and \eqref{eq:V:case1:property3:secondorderderi} in the first case.

We next consider the  second case that $\frac{d}{dr}\big((r^2+a^2)^3\frac{d}{dr}V\big)$  has a unique point $r_+\leq r_1<r_{\text{max}}$ such that it is positive for $r<r_1$ and negative for $r>r_1$. Since it is positive for $r<r_1$, we infer
\beaa
\frac{d}{dr}V(r)\geq \frac{1}{(r_1^2+a^2)^3} (r_+^2+a^2)^3 \frac{d}{dr}V\vert_{r=r_+}\gtrsim b\Lambda^2, \qquad \forall r\in[r_+, r_1].
\eeaa
Next, we consider the case $r\in [r_1,+\infty)$ and we may repeat the argument of the first case by replacing $r=r_+$ with $r=r_1$, hence proving the estimates \eqref{eq:V:case1:property2} and \eqref{eq:V:case1:property3:secondorderderi}  in this second case and concluding the proof of point \ref{point1:potentialV:maxandminbounds:frequencies}. 

It remains to prove point \ref{point2:potentialV:maxandminbounds:frequencies}. As the assumptions of point \ref{point2:potentialV:maxandminbounds:frequencies} imply the ones of point \ref{point1:potentialV:maxandminbounds:frequencies}, point \ref{point1:potentialV:maxandminbounds:frequencies} applies and hence $V$ has a unique critical point $r_{\text{max}}$ which is a global maximum {and satisfies $r_{\text{max}}-r_+\gtrsim m-a$}. For the estimate \eqref{eq:V:case2:property4:superradiantisnottrapped}, it suffices to show it for $\b=0$, and the existence of $\b\gtrsim m^{-2}(m-a)$ such that \eqref{eq:V:case2:property4:superradiantisnottrapped} holds true follows manifestly in the same manner.

{First, in the case that $\xiphi(\xit +\frac{a}{2mr_+}\xiphi)\leq \ep_0 m^{-2}|\xiphi|\Lambda$, {noticing also that $\xiphi\neq 0$ since $\xit\xiphi<0$,} it holds
\beaa
\xit^2 - V(r_+)=\bigg(\xit +\frac{a}{2mr_+} \xiphi\bigg)^2\leq \ep_0^2m^{-4}\Lambda^2.
\eeaa}
{Since we have by} \eqref{eq:V:case1:property2} that $\pr_r V\gtrsim m^{-4}(r_{\text{max}}-r)\Lambda^2\gtrsim m^{-4}(m-a) \Lambda^2$ for $r\leq r_{\text{max}}$, we {infer}
\beaa
V(r_++\de_0m)-\xit^2\gtrsim m^{-4}(m-a)^2\Lambda^2
\eeaa
for $\de_0\gtrsim m-a$ sufficiently small and $\ep_0\gtrsim m-a$ even smaller,
which then yields \eqref{eq:V:case2:property4:superradiantisnottrapped}. 

{Next, consider the case where $\xiphi(\xit{+}\frac{a}{2mr_+}\xiphi)\geq \ep_0 m^{-2}|\xiphi|\Lambda$. Since $\xit\xiphi<0$, we have from $\xiphi(\xit{+}\frac{a}{2mr_+}\xiphi)\geq \ep_0 m^{-2} |\xiphi|\Lambda$ that $\xiphi^2\gtrsim -m\xit\xiphi$.
Let $r_2=-\frac{a\xiphi}{2m\xit}$, then
\beaa
r_2-r_+=\frac{(2mr_+\xit +a\xiphi)\xiphi}{-2m\xit\xiphi}\gtrsim  \frac{\ep_0|\xiphi|\Lambda}{-m\xit \xiphi}\gtrsim \ep_0\frac{\xiphi^2}{-m\xit \xiphi}\gtrsim \ep_0\gtrsim m-a,
\eeaa
which indicates that} $r_2>r_+$ is bounded away from $r_+$ by a constant that depends only on $m$ and $a$, but not on the frequencies. We compute
\beaa
(V-\xit^2)\vert_{r=r_2}&=&\bigg(\frac{\Delta\Lambda^2}{(\R)^2}-\frac{(a\xiphi+ 2mr\xit)^2}{(\R)^2} -\frac{\Delta (r^2+2mr+a^2)}{(\R)^2} \xit^2 \bigg)\bigg|_{r=r_2}\nn\\
&=&\frac{\Delta(r_2)}{(r_2^2+a^2)^2}\bigg(\Lambda^2 - \frac{a^2\xiphi^2}{4m^2r_2^2}(r_2^2 +2mr_2 + a^2)\bigg)\nn\\
&\gtrsim& {m^{-4}(m-a)^2\Lambda^2,}
\eeaa
where we have used in the last step that $r_2>m$, $\Lambda^2\geq \xiphi^2$ and $r_2-r_+\gtrsim m-a$. The estimate \eqref{eq:V:case2:property4:superradiantisnottrapped} is thus proved which concludes the proof of Lemma \ref{lem:SRnottrap}.
\end{proof}

%%%%%%%%%%%%%%%%%%%%%%%%%%%%%%%%%

\subsection{Pseudo-differential version of the energy identity}
\label{subsect:GeneralEnerIden:PD}

%%%%%%%%%%%%%%%%%%%%%%%%%%%%%%%%%

Let $E$ and $X$ be defined by 
\bea\lab{eq:generalformofthePDOmultipliersXandE}
\bsplit
E=&\Opw(e_0), \qquad\qquad e_0\in \widetilde{S}^{0,0}(\MM),\\
X=&\Opw(i \mu s_0 \xi_r) + \Opw\bigg(\frac{is_0 \S_1}{\R} + i s_1\bigg), \qquad s_j\in\widetilde{S}^{j,0}(\MM),\,\,\, j=0,1,
\end{split}
\eea
with $e_0$, $s_0$, and $s_1$ all being real symbols {of the form\footnote{In particular, $\sigma(X)$ is of the form $\chi(\Xi)\widetilde{S}^{1,1}_{hom}(\MM)$.}
\beaa
e_0=\chi(\Xi)\tilde{e}_0, \qquad s_0=\chi(\Xi)\tilde{s}_0, \qquad s_1=\chi(\Xi)\tilde{s}_1, \qquad \tilde{e}_0, \, \tilde{s}_0\in\widetilde{S}^{0,0}_{hom}(\MM), \qquad \tilde{s}_1\in\widetilde{S}^{1,0}_{hom}(\MM),
\eeaa
where $\chi$ denotes a smooth cut-off in $\mathbb{R}^3$ such that $0\leq\chi\leq 1$, $\chi=1$ for $|\Xi|\geq 2$ and $\chi=0$ for $|\Xi|\leq 1$, and where the class of symbols $\widetilde{S}^{m,N}_{hom}(\MM)$ has been introduced in Definition \ref{def:thesymbolswereallyuse}.} With respect to {the measure $d\Vref$ introduced in Remark \ref{rmk:linkvolumeformMwithf0dVref}}, $X$ is a skew-adjoint operator in $
\Opw(\widetilde{S}^{1,1}(\MM))$ and $E$ is a self-adjoint operator in $
\Opw(\widetilde{S}^{0,0}(\MM))$. Their symbols are given by
\bea\lab{eq:symbolsofPDOmultipliersXandE}
\sigma(X)= i\left(\mu s_0 \xi_r+\frac{s_0 \S_1}{\R}+ s_1\right), \qquad \sigma(E)=e_0.
\eea 

By utilizing the above PDOs $X$ and $E$ as multipliers, we derive a pseudodifferential version of the standard energy identity in the following lemma. It is stated for a Lorentzian metric $\g$ in the normalized coordinates systems $(\tau, r, x^1, x^2)$ of Section \ref{sec:isochorecoordinatesonHr}.

\begin{lemma}
\label{lem:EnerIden:PDO}
Recall that $\MM_{r_1,r_2}=\MM\cap \{r_1\leq r\leq r_2\}$, {and recall also the notation $f_0$, introduced in Lemma \ref{lemma:spacetimevolumeformusingisochorecoordinates}, associated to a Lorentzian metric $\g$ in the normalized coordinates systems $(\tau, r, x^1, x^2)$ of Section \ref{sec:isochorecoordinatesonHr}.} Then, the following pseudodifferential energy identity holds
\bea
\label{eq:EnerIden:PDO}
-\int_{\MM_{r_1,r_2}}\Re\big(\Box_{\g}\psi\overline{(X+E)\psi}\big)=\int_{\MM_{r_1,r_2}}  \Re\big({T}_{X,E}\psi\bar{\psi} \big) d\Vref+\textbf{BDR}[\psi]\Big|_{r=r_1}^{r=r_2},
\eea
where $d\Vref=d\tt drdx^1dx^2$ is the Lebesgue measure in the coordinates $(\tau, r, x^1, x^2)$ and 
\bsub
\bea
\label{def:hatTXE:PDEnerIden}
{T}_{X,E}&:=&- \f12 \Big([f_0{\Box}_{\g}, X]+\big(E f_0{\Box}_{\g} +f_0{\Box}_{\g}E\big)\Big),\\
\label{def:BDR:PDEnerIden}
\textbf{BDR}[\psi]&:=& \frac{1}{2}\int_{H_r}\Re\Big(
\g^{\a r}\psi \overline{\pr_{\a} (X{+E})\psi}
-\g^{r\a}\pr_{\a}\psi \overline{(X{+E})\psi}
\nn\\
&& \qquad\qquad\qquad\qquad\qquad\qquad\qquad -\mu \Opw(s_0)\psi \overline{{\Box}_{\g}\psi}
\Big)f_0 d\tt dx^1dx^2.
\eea
\esub
\end{lemma}

\begin{proof}
Using the skew-adjointness of $X$ and the self-adjointness of $E$ with respect to $d\Vref={d\tt drdx^1dx^2}$, as well as the self-adjointness of $\Box_{\g}$ with respect to  $f_0d\Vref$ in view of Remark \ref{rmk:linkvolumeformMwithf0dVref}, we derive
\beaa
&&\int_{\MM_{r_1,r_2}} 2\Re\big(\Box_{\g}\psi\overline{(X+E)\psi}\big)\nn\\
&=&\Re\bigg(\int_{\MM_{r_1,r_2}}\overline{ \Box_{\g}\psi }{(X+E)\psi}f_0d\Vref
+\int_{\MM_{r_1, r_2}} \overline{(X+E)\psi} {\Box_{\g}\psi } f_0d\Vref\bigg)\nn\\
&=& \Re\bigg(\int_{\MM_{r_1,r_2}} \bar{\psi} \Box_{\g}(X+E)\psi f_0d\Vref
+\int_{\MM_{r_1, r_2}} \bar{\psi} {(-X +E) (f_0\Box_{\g}\psi )}d\Vref\bigg)+BDR\nn\\
&=&\Re\bigg(\int_{\MM_{r_1,r_2}}\Big([f_0\Box_{\g}, X]\psi \bar{\psi} +E(f_0\Box_{\g}\psi )\bar{\psi} + f_0\Box_{\g}E\psi \bar{\psi}\Big)d\Vref \bigg) +BDR\\
&=& -2\int_{\MM_{r_1,r_2}}  \Re\big({T}_{X,E}\psi\bar{\psi} \big) d\Vref +BDR,
\eeaa
where $BDR$ indicates the boundary terms arising in the integrations by parts {from the second to the third line.} Examining the above integration by parts, {and noticing in particular that 
\beaa
\Opw(i \mu s_0 \xi_r)\phi=\frac{1}{2}(\Opw(i \mu s_0)D_r\phi+D_r(\Opw(i \mu s_0)\phi))
\eeaa
in view of Proposition \ref{prop:PDO:MM:Weylquan:mixedoperators},} we find that the arising boundary terms are as in \eqref{def:BDR:PDEnerIden}.
\end{proof}

Next, {compute the RHS of \eqref{eq:EnerIden:PDO}} in Kerr {up to lower order terms}. 
\begin{proposition}
\label{prop:esti:GeneralEnerIden:PDO}
Let 
\bea\label{eq:symbolofBulk:EnerIden}
\bsplit
\sigma_2({T}_{X,E}):=& \f12 \Bigg\{2{(\R)\partial_r(s_0)}\xirstar^2 + 2(\R) \pr_r(s_1)\xirstar\\
&+ \mu s_0{\left(\big(-4r\mu^{-1}+2(r-m)\mu^{-2}\big)\Big(\xirstar^2 - (\xi_\tau^2-V)\Big)   - (\R)\mu^{-1}\pr_rV\right)}\Bigg\}\\
& +\mu^{-1}(\R)e_0\Big(\xirstar^2-(\xit^2-V)\Big), \qquad \sigma_2({T}_{X,E})\in\widetilde{S}^{2,2}(\MM),
\end{split}
\eea
and
 \bea
 \label{eq:symbolofBDR:EnerIden}
 \sigma_{2, \textbf{BDR}}^{X,E}:= -\mu s_0\xi_r \S_1 -\frac{s_0 \S_1^2}{\R}- \Delta s_1 \xi_r  -  s_1 \S_1 -\frac{1}{2}\mu s_0 \S_2, \qquad \sigma_{2, \textbf{BDR}}^{X,E}\in\widetilde{S}^{2,1}(\MM).
 \eea
{Then, it holds 
\bea
\label{eq:esti:GeneralEnerIden:PDO}
\nn&&\int_{\MM_{r_1,r_2}}  \Re\big(\bar{\psi}\Opw(\sigma_2({T}_{X,E}))\psi \big) d\Vref\\
\nn&&+\Bigg[\int_{H_r}\Re\Bigg(-\frac{1}{2}\Opw(\mu^2(\R)s_0)\pr_r\psi \ov{\pr_r\psi} + \ov{\psi}\Opw\Big( \sigma_{2, \textbf{BDR}}^{X,E}\Big)\psi\Bigg) d\tt dx^1dx^2\Bigg]_{r=r_1}^{r=r_2}\\
\nn&\leq& \int_{\MM_{r_1,r_2}}  \Re\big({T}_{X,E}\psi\bar{\psi} \big) d\Vref+\textbf{BDR}[\psi]\Big|_{r=r_1}^{r=r_2} +C_{r_2} \Bigg\{\left(\int_{\MM_{r_1, r_2}}|\pr_r\psi|^2\right)^{\frac{1}{2}}\left(\int_{\MM_{r_1, r_2}} |\psi|^2\right)^{\frac{1}{2}}\\
\nn&+&\int_{\MM_{r_1, r_2}} |\psi|^2 {+\left|\int_{\MM_{r_1,r_2}}  \Re\left(\bar{\psi}\Opw(\mu\widetilde{S}^{0,2}(\MM))\psi \right) d\Vref\right|}\\
&+&\left(\int_{H_{r_1}}|\pr\psi|^2+|\psi|^2\right)^{\frac{1}{2}}\left(\int_{H_{r_1}}|\psi|^2\right)^{\frac{1}{2}}+\left(\int_{H_{r_2}}|\pr\psi|^2+|\psi|^2\right)^{\frac{1}{2}}\left(\int_{H_{r_2}}|\psi|^2\right)^{\frac{1}{2}}\Bigg\},
\eea
where $C_{r_2}$ is a constant that depends on the value of $r_2$.}
\end{proposition}

\begin{remark}[Choice of constants $r_1$ an $r_2$]\label{rmk:choiceofconstantRbymeanvalue}
In {practice}, we will take $r_1=r_+(1+\dhor')$ and $r_2=R$, where $\dhor'\in [\dhor, 2\dhor]$ verifies
{\bea
\lab{de:choiceofdhor'}
\nn&&\int_{H_{r_+(1+\dhor')}}\big(|{\pr^{\leq 1}}\psi|^2+|\square_\g\psi|^2\big) d\tt dx^1dx^2\\ 
&\leq& \frac{1}{\dhor}\int_{\MM_{r_+(1+\dhor), r_+(1+2\dhor)}}\big(|{\pr^{\leq 1}}\psi|^2+|\square_\g\psi|^2\big) d\Vref,
\eea}
and $R\in [Nm, (N+1)m]$, with $N\geq 20$ a large enough integer, verifies 
{\bea
\label{eq:choiceofRvalue:Kerr}
\int_{H_R}\big(|{\pr^{\leq 1}}\psi|^2+|\square_\g\psi|^2\big) d\tt dx^1dx^2\leq \frac{1}{m}\int_{\MM_{Nm, (N+1)m}}\big(|{\pr^{\leq 1}}\psi|^2+|\square_\g\psi|^2\big) d\Vref, \,\,\,  R\geq 20m.
\eea}
\end{remark}

\begin{proof}
{The proof proceeds in the following steps.}

{\noindent{\bf Step 1.}} Recall from  Lemma \ref{lemma:spacetimevolumeformusingisochorecoordinates} that $f_0=|q|^2$ in Kerr. In particular, we may rewrite ${T}_{X,E}$ in \eqref{def:hatTXE:PDEnerIden} as
\beaa
{T}_{X,E}&=&- \f12 \Big([{|q|^2{\Box}_{\gam}}, X]+\big(E{|q|^2{\Box}_{\gam}} +{|q|^2{\Box}_{\gam}}E\big)\Big).
\eeaa
Also, recall from \eqref{eq:computationofsymbolmodqsquareboxgam} and \eqref{eq:computationofsymbolmodqsquareboxgam:1} that 
\beaa
\sigma(\qs{\Box}_{\gam} )&=&-\Delta\xi_r^2 -2\S_1\xi_r+\S_2, \qquad \S_1\in\widetilde{S}^{1,0}(\MM), \qquad \S_2\in\widetilde{S}^{2,0}(\MM)
\eeaa
and from \eqref{eq:symbolsofPDOmultipliersXandE} that 
\beaa
\sigma(X)= i\left(\mu s_0 \xi_r+\frac{s_0 \S_1}{\R}+ s_1\right), \qquad \sigma(E)=e_0.
\eeaa
We decompose ${T}_{X,E}$ as follows
\beaa
\bsplit
{T}_{X,E}:=&{T}_{X,E}^{(1)}+{T}_{X,E}^{(2)}+{T}_{X,E}^{(3)},\\
{T}_{X,E}^{(1)}:=& \f12 \Big([\Opw(\Delta\xi_r^2), X]+\big(E\Opw(\Delta\xi_r^2) +\Opw(\Delta\xi_r^2)E\big)\Big)\\
{T}_{X,E}^{(2)}:=& - \f12 [\Opw(-2\S_1\xi_r+\S_2), \Opw(i\mu s_0 \xi_r)])\\
{T}_{X,E}^{(3)}:=&  - \f12 \Bigg(\left[\Opw( -2\S_1\xi_r+\S_2),  i\left(\frac{s_0 \S_1}{\R}+ s_1\right)\right]\\
&+\big(E\Opw( -2\S_1\xi_r+\S_2) +\Opw( -2\S_1\xi_r+\S_2)E\big)\Bigg).
\end{split}
\eeaa
We now compute the symbol of ${T}_{X,E}$:
\begin{itemize}
\item For ${T}_{X,E}^{(1)}$, we rely on 
\eqref{eq:propWeylquantization:MM:composition:mixedsymbols:specialcase}, noticing that, in local coordinates, $\Delta\xi_r^2$ is of the type $v_1(r)\xi_3^{N_1}$ with $v_1(r)=\De$ and $N_1=2$.
\item For ${T}_{X,E}^{(2)}$, we rely on 
\eqref{eq:propWeylquantization:MM:composition:mixedsymbols:specialcase:1} with $f(r)=\mu$.
\item For ${T}_{X,E}^{(3)}$, we rely on \eqref{eq:propWeylquantization:MM:composition:mixedsymbols}.
\end{itemize}
{We obtain 
\beaa
\sigma({T}_{X,E})= -\frac{1}{2i} \{\sigma(|q|^2{\Box}_{\gam} ), \sigma(X)\}-|q|^2\sigma(E)\sigma({\Box}_{\gam} ) +\mu\widetilde{S}^{0,2}(\MM)+\widetilde{S}^{0,1}(\MM).
\eeaa}

{\noindent{\bf Step 2.}} {Next, we compute $\{\sigma(|q|^2{\Box}_{\gam} ), \sigma(X)\}$. First, recall from \eqref{eq:computationofsymbolmodqsquareboxgam} and \eqref{eq:computationofsymbolmodqsquareboxgam:1} that 
\beaa
\bsplit
\sigma(\qs{\Box}_{\gam} )=&-\Delta\xi_r^2 -2\S_1\xi_r+\S_2 +\widetilde{S}^{0,0}(\MM),\\
\S_1=& (\R)(1-\mu \tmod')\xit+(a-\Delta\phimod')\xiphi,\\
\S_2=& - \Lambda^2
-\big(2a(1-\tmod')-2(\R)\phimod' (1-\mu \tmod')\big)\xit\xiphi\\
&+\big(2(\R)\tmod'-\Delta(\tmod')^2\big)\xit^2
-({\Delta}(\phimod')^2-2a\phimod' )\xiphi^2,
\end{split}
\eeaa
which implies that $-\Delta\xi_r^2 -2\S_1\xi_r+\S_2\in\widetilde{S}^2_{pol}(\MM)$, see Definition \ref{def:thesymbolswereallyuse}, and 
\beaa
\{\sigma(|q|^2{\Box}_{\gam} ), \sigma(X)\}=\{-\Delta\xi_r^2 -2\S_1\xi_r+\S_2,  \sigma(X)\}+\widetilde{S}^{0,1}(\MM).
\eeaa
Since $\sigma(X)$ is of the form $\chi(\Xi)\widetilde{S}^{1,1}_{hom}(\MM)$ and  $-\Delta\xi_r^2 -2\S_1\xi_r+\S_2\in\widetilde{S}^2_{pol}(\MM)$, we may apply \eqref{eq:basiccomputationPoissonbracketsymbolsusedinwaveandmultipliers} which yields
\beaa
&&\{-\Delta\xi_r^2 -2\S_1\xi_r+\S_2, \sigma(X)\} \\
&=& \pr_{\xi_r}(-\Delta\xi_r^2 -2\S_1\xi_r+\S_2)\pr_r(\sigma(X)) - \pr_r(-\Delta\xi_r^2 -2\S_1\xi_r+\S_2)\pr_{\xi_r}(\sigma(X))
\eeaa
and hence
\beaa
\{\sigma(|q|^2{\Box}_{\gam} ), \sigma(X)\} = \pr_{\xi_r}(\sigma(|q|^2{\Box}_{\gam} ))\pr_r(\sigma(X)) - \pr_r(\sigma(|q|^2{\Box}_{\gam} ))\pr_{\xi_r}(\sigma(X))+\widetilde{S}^{0,1}(\MM).
\eeaa
Now, in view of  \eqref{eq:ps:rescalewave:normalize},  \eqref{def:2symbol:wavepotential}, \eqref{eq:symbolsofPDOmultipliersXandE} and \eqref{def:xirstar}, we have
\beaa
\bsplit
\sigma(\qs{\Box}_{\gam}) &=-\mu^{-1} (\R) \xirstar^2 +\BLS_2+\widetilde{S}^{0,0}(\MM), \qquad \sigma(X)= i\left(s_0 \xirstar+ s_1\right),\\
\BLS_2 &= (\R)\mu^{-1}(\xit^2-V),\qquad V=\frac{\Delta\Lambda^2 -4amr\xit\xiphi-a^2 \xiphi^2}{(\R)^2},\qquad \xirstar = \mu \xi_r +\frac{\S_1}{\R}.
\end{split}
\eeaa
Hence, using also the fact that  $\pr_{\xi_r}\xirstar=\mu$, as well as
\beaa
\pr_{\xi_r}\big(-\mu^{-1} (\R) \xirstar^2 +\BLS_2\big) &=& -2(\R)\xirstar,\\
\pr_r\big(-\mu^{-1} (\R) \xirstar^2 +\BLS_2\big) &=& -\pr_r(\mu^{-1} (\R))\xirstar^2  -2\mu^{-1} (\R) \xirstar\pr_r\xirstar+\pr_r\BLS_2,\\
&=& (-4r\mu^{-1}+2(r-m)\mu^{-2})(\xirstar^2 - (\xi_\tau^2-V)) \\
&& -2\mu^{-1} (\R) \xirstar\pr_r\xirstar -(r^2+a^2)\mu^{-1}\pr_rV,\\
\pr_{\xi_r}(\sigma(X)) &=& i\mu s_0,\\
\pr_r(\sigma(X)) &=& i\left(\pr_r(s_0)\xirstar+s_0\pr_r\xirstar +\pr_r(s_1)\right),
\eeaa
we infer, noticing the cancellation of the terms involving $\pr_r\xirstar$,  
\beaa
\bsplit
\{\sigma(|q|^2{\Box}_{\gam} ), \sigma(X)\} =&  -2i(\R)\xirstar \left(\pr_r(s_0)\xirstar +\pr_r(s_1)\right)\\
& -i\mu s_0\left(\Big(-4r\mu^{-1}+2(r-m)\mu^{-2}\Big)\Big(\xirstar^2 - (\xi_\tau^2-V)\Big)   -(r^2+a^2)\mu^{-1}\pr_rV\right)\\
& +\widetilde{S}^{0,1}(\MM).
\end{split}
\eeaa
We deduce} 
\beaa
\sigma({T}_{X,E})&=&  -\frac{1}{2i} \{\sigma(|q|^2{\Box}_{\gam} ), \sigma(X)\}-|q|^2\sigma(E)\sigma({\Box}_{\gam} ) +\mu\widetilde{S}^{0,2}(\MM)+\widetilde{S}^{0,1}(\MM)\nn\\
&=& \f12 \Bigg\{2{(\R)\partial_r(s_0)}\xirstar^2 {+} 2(\R) \pr_r(s_1)\xirstar\nn\\
&&+ \mu s_0{\left(\Big(-4r\mu^{-1}+2(r-m)\mu^{-2}\Big)\Big(\xirstar^2 - (\xi_\tau^2-V)\Big)   - (\R)\mu^{-1}\pr_rV\right)}\Bigg\}\nn\\
&& +\mu^{-1}(\R)e_0\Big(\xirstar^2-(\xit^2-V)\Big)+\mu\widetilde{S}^{0,2}(\MM)+\widetilde{S}^{0,1}(\MM),
\eeaa
which we rewrite as 
{\beaa
\bsplit
\sigma({T}_{X,E}) =& \sigma_2({T}_{X,E})+\mu\widetilde{S}^{0,2}(\MM)+\widetilde{S}^{0,1}(\MM),\\
\sigma_2({T}_{X,E}):=& \f12 \Bigg\{2{(\R)\partial_r(s_0)}\xirstar^2 + 2(\R) \pr_r(s_1)\xirstar\\
&+ \mu s_0{\left(\Big(-4r\mu^{-1}+2(r-m)\mu^{-2}\Big)\Big(\xirstar^2 - (\xi_\tau^2-V)\Big)   - (\R)\mu^{-1}\pr_rV\right)}\Bigg\}\\
& +\mu^{-1}(\R)e_0\Big(\xirstar^2-(\xit^2-V)\Big),
\end{split}
\eeaa
where $\sigma_2({T}_{X,E})$ is as stated in \eqref{eq:symbolofBulk:EnerIden}.}

{\noindent{\bf Step 3.}} {Next, we consider the boundary term which is given by \eqref{def:BDR:PDEnerIden}, i.e., 
\beaa
\textbf{BDR}[\psi]&=& \frac{1}{2}\int_{H_r}\Re\Big(
|q|^2\gam^{\a r}\psi \overline{\pr_{\a} (X{+E})\psi}
-|q|^2\gam^{r\a}\pr_{\a}\psi \overline{(X{+E})\psi}
\nn\\
&& \qquad\qquad\qquad\qquad\qquad\qquad\qquad -\mu \Opw(s_0)\psi \overline{|q|^2{\Box}_{\gam}\psi}
\Big) d\tt dx^1dx^2,
\eeaa
where we used the fact that $f_0=|q|^2$ in Kerr. Using \eqref{eq:inverse:hypercoord}, we have
\beaa
|q|^2\gam^{\a r}\pr_{\a} &=& \De\pr_r+(\R)(1-\mu \tmod')\pr_{\tt}+\big(a - \Delta\phimod'\big)\pr_{\tphi}\\
&=& i\Opw(\De\xi_r+\S_1)+(r-m)
\eeaa
where we have used the definition of $\S_1$. We infer
\beaa
&& |q|^2\gam^{\a r}\pr_{\a} (X\psi) - \mu \Opw(s_0)|q|^2{\Box}_{\gam}\psi\\ 
&=& i\Opw(\De\xi_r+\S_1)\circ\left(\Opw(i \mu s_0 \xi_r) + \Opw\bigg(\frac{is_0 \S_1}{\R} + i s_1\bigg)\right)\\ 
&& - \mu \Opw(s_0)\circ\Opw(-\Delta\xi_r^2 -2\S_1\xi_r+\S_2)\psi +(r-m)X\psi
\eeaa
which together with Proposition \ref{prop:PDO:MM:Weylquan:mixedoperators} implies
\beaa
&& |q|^2\gam^{\a r}\pr_{\a} (X\psi) - \mu \Opw(s_0)|q|^2{\Box}_{\gam}\psi\\ 
&=& \Opw\Bigg\{-(\De\xi_r+\S_1)\left(\mu s_0 \xi_r + \frac{s_0 \S_1}{\R} + s_1\right) +\mu s_0(\Delta\xi_r^2 +2\S_1\xi_r-\S_2)\Bigg\}\psi \\
&&+(r-m)X\psi +\Opw(\widetilde{S}^{0,0}(\MM))\\
&=& \Opw\left(-\De s_1\xi_r -\S_1\left(\frac{s_0 \S_1}{\R} + s_1\right) -\mu s_0\S_2\right)\psi +(r-m)X\psi +\Opw(\widetilde{S}^{0,0}(\MM)).
\eeaa}
Plugging in $\textbf{BDR}[\psi]$, we infer 
{\beaa
\textbf{BDR}[\psi]&=& \frac{1}{2}\int_{H_r}\Re\Bigg(-|q|^2\g^{r\a}\pr_{\a}\psi \overline{X\psi} +\ov{\psi}\Opw\left(-\De s_1\xi_r -\S_1\left(\frac{s_0 \S_1}{\R} + s_1\right) -\mu s_0\S_2\right)\psi\\
&&+\ov{\psi}\Opw(\widetilde{S}^{1,1}(\MM))\psi\Bigg) d\tt dx^1dx^2.
\eeaa
Using again the above identity for $|q|^2\g^{\a r}\pr_{\a}$, we deduce
\beaa
\textbf{BDR}[\psi]&=& \frac{1}{2}\int_{H_r}\Re\Bigg(-i\Opw(\De\xi_r+\S_1)\psi \overline{X\psi} +\ov{\psi}\Opw\left(-\De s_1\xi_r -\S_1\left(\frac{s_0 \S_1}{\R} + s_1\right) -\mu s_0\S_2\right)\psi\\
&&+\ov{\psi}\Opw(\widetilde{S}^{1,1}(\MM))\psi\Bigg) d\tt dx^1dx^2.
\eeaa
Now, in view of the definition of $X$, and using repeatedly Proposition \ref{prop:PDO:MM:Weylquan:mixedoperators}, we have
\beaa
&&\int_{H_r}\Re\Big(-i\Opw(\De\xi_r+\S_1)\psi \overline{X\psi}\Big) d\tt dx^1dx^2\\ 
&=& \int_{H_r}\Re\Big(-\Opw(\De\xi_r+\S_1)\psi \overline{\left(\Opw(\mu s_0 \xi_r) + \Opw\bigg(\frac{s_0 \S_1}{\R} + s_1\bigg)\right)\psi}\Big) d\tt dx^1dx^2\\
&=& \int_{H_r}\Re\Bigg(-\Opw((\R)s_0\mu^2)\pr_r\psi \ov{\pr_r\psi}-\ov{\psi}\Opw\left(\mu s_0\S_1\xi_r+\bigg(\frac{s_0 \S_1}{\R} + s_1\bigg)(\De\xi_r+\S_1)\right)\psi\\
&& +\ov{\psi}\Opw(\widetilde{S}^{1,1}(\MM))\psi\Bigg) d\tt dx^1dx^2
\eeaa
and hence
\bea\lab{eq:formulaforBDRbracketpsithatwillbeusedagainlaterforboundarytermr=R}
\nn\textbf{BDR}[\psi]&=& \int_{H_r}\Re\Bigg(-\frac{1}{2}\Opw(\mu^2(\R)s_0)\pr_r\psi \ov{\pr_r\psi} + \ov{\psi}\Opw\Big( \sigma_{2, \textbf{BDR}}^{X,E}\Big)\psi \\
&&\qquad\qquad\qquad\qquad+\ov{\psi}\Opw(\widetilde{S}^{1,1}(\MM))\psi\Bigg) d\tt dx^1dx^2,
\eea
where
 \beaa
 \sigma_{2, \textbf{BDR}}^{X,E}:= -\mu s_0\xi_r \S_1 -\frac{s_0 \S_1^2}{\R}- \Delta s_1 \xi_r  -  s_1 \S_1 -\frac{1}{2}\mu s_0 \S_2, \qquad \sigma_{2, \textbf{BDR}}^{X,E}\in\widetilde{S}^{2,1}(\MM)
 \eeaa
 as stated in \eqref{eq:symbolofBDR:EnerIden}.}

{\noindent{\bf Step 4.}} {We are now ready to conclude. {Recall from} above that we have obtained 
\beaa
\sigma({T}_{X,E}) &=& \sigma_2({T}_{X,E})+\mu\widetilde{S}^{0,2}(\MM)+\widetilde{S}^{0,1}(\MM),
\eeaa
and
\beaa
\textbf{BDR}[\psi]&=& \int_{H_r}\Re\Bigg(-\frac{1}{2}\Opw(\mu^2(\R)s_0)\pr_r\psi \ov{\pr_r\psi} + \ov{\psi}\Opw\Big( \sigma_{2, \textbf{BDR}}^{X,E}\Big)\psi \\
&&\qquad\qquad\qquad\qquad+\ov{\psi}\Opw(\widetilde{S}^{1,1}(\MM))\psi\Bigg) d\tt dx^1dx^2,
\eeaa
where $\sigma_2({T}_{X,E})$ and $\sigma_{2, \textbf{BDR}}^{X,E}$ are given respectively by \eqref{eq:symbolofBulk:EnerIden} and \eqref{eq:symbolofBDR:EnerIden}. This implies
\beaa
&& {\int_{\MM_{r_1,r_2}}  \Re\big({T}_{X,E}\psi\bar{\psi} \big) d\Vref+\textbf{BDR}[\psi]\Big|_{r=r_1}^{r=r_2}}\\ 
&=&\int_{\MM_{r_1,r_2}}  \Re\left(\bar{\psi}\Big(\Opw(\sigma_2({T}_{X,E}))+\Opw(\mu\widetilde{S}^{0,2}(\MM)+\widetilde{S}^{0,1}(\MM))\psi \Big)\right) d\Vref\\
&&+\Bigg[\int_{H_r}\Re\Bigg(-\frac{1}{2}\Opw(\mu^2(\R)s_0)\pr_r\psi \ov{\pr_r\psi} + \ov{\psi}\Opw\Big( \sigma_{2, \textbf{BDR}}^{X,E}\Big)\psi \\
&&\qquad\qquad\qquad\qquad+\ov{\psi}\Opw(\widetilde{S}^{1,1}(\MM))\psi\Bigg) d\tt dx^1dx^2\Bigg]_{r=r_1}^{r=r_2}.
\eeaa
{In view of Lemma \ref{lemma:actionmixedsymbolsSobolevspaces:MM}, we infer}
\beaa
&&\int_{\MM_{r_1,r_2}}  \Re\big(\bar{\psi}\Opw(\sigma_2({T}_{X,E}))\psi \big) d\Vref\\
&&+\Bigg[\int_{H_r}\Re\Bigg(-\frac{1}{2}\Opw(\mu^2(\R)s_0)\pr_r\psi \ov{\pr_r\psi} + \ov{\psi}\Opw\Big( \sigma_{2, \textbf{BDR}}^{X,E}\Big)\psi\Bigg) d\tt dx^1dx^2\Bigg]_{r=r_1}^{r=r_2}\\
&\leq& {\int_{\MM_{r_1,r_2}}  \Re\big({T}_{X,E}\psi\bar{\psi} \big) d\Vref+\textbf{BDR}[\psi]\Big|_{r=r_1}^{r=r_2}} +C_{r_2} \Bigg\{\left(\int_{\MM_{r_1, r_2}}|\pr_r\psi|^2\right)^{\frac{1}{2}}\left(\int_{\MM_{r_1, r_2}} |\psi|^2\right)^{\frac{1}{2}}\\
&&+\int_{\MM_{r_1, r_2}} |\psi|^2 {+\left|\int_{\MM_{r_1,r_2}}  \Re\left(\bar{\psi}\Opw(\mu\widetilde{S}^{0,2}(\MM))\psi \right) d\Vref\right|}\\
&&+\left(\int_{H_{r_1}}|\pr\psi|^2+|\psi|^2\right)^{\frac{1}{2}}\left(\int_{H_{r_1}}|\psi|^2\right)^{\frac{1}{2}}+\left(\int_{H_{r_2}}|\pr\psi|^2+|\psi|^2\right)^{\frac{1}{2}}\left(\int_{H_{r_2}}|\psi|^2\right)^{\frac{1}{2}}\Bigg\}
\eeaa
as stated, where $C_{r_2}$ is a constant that depends on the value of $r_2$.} This concludes the proof of Proposition \ref{prop:esti:GeneralEnerIden:PDO}.
\end{proof}

%%%%%%%%%%%%%%%%%%%%%%%%%%%%%%%%%%%%%%%%%%

\subsection{General choices {for the microlocal multipliers} in the energy identity}
\label{subsect:PDcurrents}

%%%%%%%%%%%%%%%%%%%%%%%%%%%%%%%%%%%%%%%%%%

Based on different choices of the symbols $s_0, s_1$ and $e_0$, we define
\bsub\lab{eq:definitionofQhQyQfQzintermsofTXEands0s1e0}
\bea
Q^h&:=&\sigma_2 ({T}_{X,E}), \quad \text{with }\, (s_0,s_1,e_0)=(0,0, \mu h),\\
Q^y&:=&\sigma_2 ({T}_{X,E}), \quad \text{with }\,  (s_0,s_1,e_0)=\left(2y,0,\frac{2\mu r}{\R}y-\pr_r(\mu y)\right),\\
Q^f&:=&\sigma_2 ({T}_{X,E}), \quad \text{with }\,  (s_0,s_1,e_0)=\left( 2f,0,\frac{2\mu r}{\R}f-\pr_r(\mu f)+ \mu \rd_r f\right),\\
Q^z&:=&\sigma_2 ({T}_{X,E}), \quad \text{with }\,  (s_0,s_1,e_0)=(0, z, 0),
\eea
\esub
where $h,y,f\in\widetilde{S}^{0,0}(\MM)$ and $z = \xit+\chi_z\omega_{\HH}\xiphi\in\widetilde{S}^{1,0}(\MM)$ with $\chi_z\in\widetilde{S}^{0,0}(\MM)$. Similarly, we define $\sigma_{2,\textbf{BDR}}^h$, $\sigma_{2,\textbf{BDR}}^y$, $\sigma_{2,\textbf{BDR}}^f$, and $\sigma_{2,\textbf{BDR}}^z$
to be $\sigma_{2, \textbf{BDR}}^{X,E}$ with the above corresponding choices of $(s_0, s_1,e_0)$. 

{\begin{remark}
In fact, in Section \ref{subsect:scalarwave:Kerr:condiMora}, we will choose $h,y,f\in\widetilde{S}^{0,0}_{hom}(\MM)$ and $z\in\widetilde{S}^{1,0}_{hom}(\MM)$. Based on these choices, we will then produce symbols $h,y,f\in\widetilde{S}^{0,0}(\MM)$ and $z\in\widetilde{S}^{1,0}(\MM)$ in Section \ref{subsect:CondDegMFesti:Kerr}. 
\end{remark}}

These defined symbols can be computed from the formulas \eqref{eq:symbolofBulk:EnerIden} and \eqref{eq:symbolofBDR:EnerIden}, and we list them as follows:
\bsub
\label{eq:bulk:currents:EnerIden}
\bea
Q^h&=&(\R)h (\xirstar^2+V-\xit^2),\\
Q^y&=&(\R)\big(\rd_ry \xirstar^2+\rd_r y(\xit^2-V)-y\rd_r V\big),\\
Q^f&=&(\R)\big(2\rd_r f \xirstar^2- f\rd_r V\big),\\
Q^{z}&=& {(\R) \omega_{\HH} \pr_r\chi_z \xiphi\xirstar}
\eea
\esub
and 
\bsub
\label{eq:flux:fixr}
\bea
\sigma_{2,\textbf{BDR}}^h&=&0,\\
\sigma_{2,\textbf{BDR}}^y&=& {- 2\mu y \xi_r \S_1-\frac{2y \S_1^2}{\R}-\mu y \S_2},\\
\sigma_{2,\textbf{BDR}}^f&=&  { -2\mu f \xi_r \S_1-\frac{2f \S_1^2}{\R}-\mu f \S_2},\\
\sigma_{2,\textbf{BDR}}^z&=&-(\xit+{\chi_z}\omega_{\HH}\xiphi) (\Delta \xi_r +\S_1).
\eea
\esub

%%%%%%%%%%%%%%%%%%%%%%%%%%%%%%%%%%%%%%%%%%

\subsection{Conditional degenerate Morawetz estimate in {all} frequency regimes}
\label{subsect:scalarwave:Kerr:condiMora}

%%%%%%%%%%%%%%%%%%%%%%%%%%%%%%%%%%%%%%%%%%

Let {$\df\gtrsim m^{-2}(m-a)>0$}, which depends only on the values of $m$ and $a$ and might degenerates as $|a|\to m$,  be a small constant to be fixed. Recalling that $\Gtz$ denotes the space of the frequency triplets $\Xi=(\xit,\xiphi,\Lambda)$, we decompose $\Gtz$ as 
\beaa
\Gtz=\GG_{{SR}}\cup\GG_{{A}}\cup\GG_{{T}}\cup\GG_{{TR}},
\eeaa
where the four open sets $\GG_{{SR}}$, $\GG_{{A}}$, $\GG_{{T}}$ and $\GG_{{TR}}$ of $\Gtz$ are given by: 
{\begin{subequations}
\lab{eq:defofthefourfrequencyregimes}
\begin{align}
\GG_{{SR}}:={}& \big\{0< -\xit\xiphi < \omega_{\HH}\xiphi^2 +2\df\Lambda^2\big\},\\
\GG_{A}:={}&\big\{\Lambda^2>\df^{-1} m\xit^2\big\}\setminus  \bigg\{\frac{1}{4}\df\Lambda^2\leq -\xit\xiphi \leq \omega_{\HH}\xiphi^2 +\df\Lambda^2\bigg\},\\
\GG_{T}:={}&\big\{ \xit^2>\df^{-1} m^{-3}\Lambda^2\big\}\setminus  \bigg\{\frac{1}{4}\df\Lambda^2\leq -\xit\xiphi \leq \omega_{\HH}\xiphi^2 +\df\Lambda^2\bigg\},\\
\GG_{TR}:={}&\left\{\frac{1}{2}\df m^3\xit^2< \Lambda^2< 2\df^{-1} m\xit^2\right\}\setminus \bigg\{\frac{1}{4}\df\Lambda^2\leq -\xit\xiphi \leq \omega_{\HH}\xiphi^2 +\df\Lambda^2\bigg\}.
\end{align}
\end{subequations}}
These four regimes are interpreted {respectively as the superradiant  frequency regime\footnote{Actually, this is a slightly larger frequency regime than the regime containing only the superradiant frequencies defined in Definition \ref{def:SRfreq:PS}.}, the angular-dominated frequency regime (i.e., the angular frequency $\Lambda^2$ is much larger than $\xit^2$), the time-dominated frequency regime (i.e., the time frequency $\xit^2$ is much larger than $\Lambda^2$),  and the trapped frequency regime\footnote{All the trapped frequencies are contained in this regime, though not all frequencies in this regime are trapped.} where $\Lambda^2$ and $\xit^2$ are comparable.} 

To {derive a} Morawetz estimate in 
{$\MM_{r_+(1+\dhor'), R}$, where the constants $\dhor'$ and $R$ have been introduced in Remark \ref{rmk:choiceofconstantRbymeanvalue},} in each of {the above} frequency regimes, it is {crucial to carry out the proof in the following order}:
\begin{enumerate}[label= \arabic*)]
\item one among the {symbols} $Q^f$, $Q^h$ and $Q^y$ in \eqref{eq:bulk:currents:EnerIden} is made globally nonnegative (but might vanish {in a subregion of the type $\MM_{r_1, r_2}$ for specific values $r_+(1+\dhor')< r_1<r_2\leq R$}, by making appropriate choices {for} the symbols $f$, $h$, and $y$;

\item if the above {symbol} only guarantees positivity in a certain subregion, we utilize the other symbols among $Q^f$, $Q^h$ and $Q^y$, which may exhibit negativity in another  subregion {where it can} be controlled by the above step, to gain positivity globally\footnote{The achieved estimate may have degeneracy in the trapping region.} {for the sum $Q^f+Q^h+Q^y$};

\item finally, we make a choice of the symbol $z$ such that  we obtain (almost) non-negativity for the sum {$\sigma_{2,\textbf{BDR}}^{f}+\sigma_{2,\textbf{BDR}}^{h}+\sigma_{2,\textbf{BDR}}^{y}+\sigma_{2,\textbf{BDR}}^{z}$} on the boundary $r=r_+(1+\dhor')$, {while} the principal symbol $Q^z$ is controlled by {the sum  $Q^f+Q^h+Q^y$ so that $Q^f+Q^h+Q^y+Q^z$ is also positive globally}. 
\end{enumerate}

Overall, suitable choices of the symbols $f, h, y$, and $z$  in {each of the four frequency regimes \eqref{eq:defofthefourfrequencyregimes}} are made such that the above three points are satisfied in the region $r_+(1+\dhor') \leq r\leq R$, {for the constants $\dhor'$ and $R$ introduced in Remark \ref{rmk:choiceofconstantRbymeanvalue}.} We now realize the above three steps in the four frequency regimes one by one.

{\begin{remark}
In addition to the above considerations, we will choose the symbols $(f,h,y,z)$ such that their sum at $r=R$ only depends on $r$ but not on frequencies $\Xi=(\xit, \xiphi, \La)$. This will allow us to complement the microlocal Morawetz estimates in $r_+(1+\dhor')\leq r\leq R$ with physical space Morawetz estimates in $ r\geq R$, see Section \ref{subsect:CondMoraFlux:Kerrperturb}.
\end{remark}}

%%%%%%%%%%%%%%%%%%%%%%%

\subsubsection{Estimates in $\GG_{SR}$}
\label{sect:GGSR:pseudo}

%%%%%%%%%%%%%%%%%%%%%%%

By Lemma \ref{lem:SRnottrap}, there is a unique root $r_{\text{max}}$ of the symbol $\pr_rV$, with {$r_{\text{max}}-r_+\gtrsim m-a$} being a global maximum of $V$, {and it holds for $\df=\df(m,a)\sim \b\gtrsim m^{-2}(m-a), \delta_0=\delta_0(m,a)\gtrsim m-a$, and $\dhor'$ positive but much smaller than $m^{-1}(m-a)$ that
\bsub
\label{property:potential:SR:Kerr}
\begin{align}
V-\xit^2 \geq{}& b_0\Lambda^2, & \forall r\in &[r_{\text{max}}-\delta_0, r_{\text{max}}+\delta_0],
\label{eq:GGSR:potential:1}
\\
-(r-r_{\text{max}})\partial_r V\geq{}& b\Lambda^2 \frac{(r-r_{\text{max}})^2}{r^4}, &\forall r\in &[r_+(1+\dhor'),\infty),
\label{eq:GGSR:potential:2:1}
\end{align}
\esub
for universal constants $b_0\gtrsim m^{-4}(m-a)^2$ and  $b\gtrsim m^{-4}(m-a)$. }

In view of the property \eqref{eq:GGSR:potential:2:1}, by choosing the symbol $f\in\widetilde{S}^{0,0}_{hom}(\MM)$ such that $f\vert_{r=r_+(1+\dhor')}=-1$, $f\vert_{r=R}=1-{m}R^{-1}$, $f$ vanishes at\footnote{{Note that $f\in\widetilde{S}^{0,0}_{hom}(\MM)$ as a consequence of the fact that $r_{\text{max}}$ is homogeneous of order 0 w.r.t. $(\xit, \xiphi, \La)$.}} $r_{\text{max}}$, and $\pr_r f\gtrsim m r^{-2}$ holds globally in $[r_+(1+\dhor'), R]$, we see from the expression of $Q^{f}$ in \eqref{eq:bulk:currents:EnerIden} that 
{\bea\lab{eq:lowerboundforQfforwellsuitedchoicefinGGSR}
Q^f=(\R)\big(2\rd_r f \xirstar^2- f\rd_r V\big)\gtrsim \xirstar^2+b\La^2\frac{(r-r_{\text{max}})^2}{r^3}.
\eea}

Next, we choose $h\in\widetilde{S}^{0,0}_{hom}(\MM)$ under the form $h=B\tilde{h}_0$, where $\tilde{h}_0=1$ in $[r_{\text{max}}-\delta_0/2, r_{\text{max}}+\delta_0/2]$, $\tilde{h}_0=0$ in $[r_+(1+\dhor'), r_{\text{max}} -\delta_0)\cup(r_{\text{max}}+\delta_0,\infty)$ {and $\tilde{h}_0\geq 0$ which yields in view of the property \eqref{eq:GGSR:potential:1} of the symbol $V$ near $r_{\text{max}}$ and the expression of $Q^{h}$ in \eqref{eq:bulk:currents:EnerIden}
\beaa
Q^h=(\R)h (\xirstar^2+V-\xit^2)\geq B(\xirstar^2+(r_{\text{max}})^2 b\La^2 ){\bf 1}_{[r_{\text{max}}-\delta_0/2, r_{\text{max}}+\delta_0/2]}(r).
\eeaa}
Hence, {together with \eqref{eq:lowerboundforQfforwellsuitedchoicefinGGSR},} we achieve for any fixed $B>0$ that 
\bea\lab{eq:lowerboundforQfplusQhforwellsuitedchoicefandhinGGSR}
Q^f + Q^{h=B\tilde{h}_0}\gtrsim  \xirstar^2 + r^{-1} \Lambda^2  \quad \text{in\, }
[r_+(1+\dhor'), R]
\eea
and
\bea\lab{eq:lowerboundforQfplusQhforwellsuitedchoicefandhinGGSR:1}
Q^f +Q^{h=B\tilde{h}_0}\gtrsim B (\xirstar^2 +\Lambda^2) 
\quad \text{in } [r_{\text{max}}-\delta_0/2, r_{\text{max}}+\delta_0/2].
\eea

We also need to control $\xit^2$. Note that
\bea\lab{eq:usefulidentityforxit2minusV}
\xit^2- V=\frac{(a\xiphi+ 2mr\xit)^2}{(\R)^2} +\frac{\Delta (r^2+2mr+a^2)}{(\R)^2} \xit^2 -\frac{\Delta\Lambda^2}{(\R)^2},
\eea 
which implies,  {using the fact that $r_{\text{max}}\geq r_++\de_0 (m-a)>r_+(1+\dhor')$ in view of Lemma \ref{lem:SRnottrap}},    
\beaa
\xit^2 - V(r_{\text{max}})\geq c_0\xit^2 - \frac{\Delta\Lambda^2}{(\R)^2}, \qquad c_0>0.
\eeaa
{Then, we can add the symbol $Q^{h=\tilde{h}_1}$, where $\tilde{h}_1=- c' mr^{-2}$ with $c'>0$ a small constant, and the symbol $Q^{h=B\tilde{h}_2}$ with $\tilde{h}_2=-b'$ in $[r_{\text{max}}-\delta_0/2, r_{\text{max}}+\delta_0/2]$ and $\tilde{h}_2=0$ in $[r_+(1+\dhor'), r_{\text{max}}-\delta_0)$ and $(r_{\text{max}}+\delta_0,R]$, $b'>0$ being a small constant. Together with
\eqref{eq:lowerboundforQfplusQhforwellsuitedchoicefandhinGGSR} and \eqref{eq:lowerboundforQfplusQhforwellsuitedchoicefandhinGGSR:1}, this implies}
\begin{align}
Q^f +Q^{h=B\tilde{h}_0} + Q^{h=\tilde{h}_1}+ Q^{h=B\tilde{h}_2}
\gtrsim & \xirstar^2 +\xit^2+ r^{-1} \Lambda^2  \quad \text{in\, }
[r_+(1+\dhor'), R],
\lab{eq:lowerboundforQfplusQhforwellsuitedchoicefandhinGGSR:2}\\
Q^f +Q^{h=B\tilde{h}_0} + Q^{h=\tilde{h}_1}+ Q^{h=B\tilde{h}_2}
\gtrsim & B (\xirstar^2 +\xit^2+\Lambda^2) 
\quad \text{in } [r_{\text{max}}-\delta_0/2, r_{\text{max}}+\delta_0/2]
\lab{eq:lowerboundforQfplusQhforwellsuitedchoicefandhinGGSR:3}.
\end{align}

Eventually, {we choose $z=A(\xit+\chi_z \omega_{\HH}\xiphi)$ with {$A>2$} to be picked large enough in Section \ref{sect:GGTR:pseudo}} and with the smooth cutoff $\chi_z=1$ in $[r_+(1+\dhor'), r_{\text{max}}-\delta_0/2]$ and $\chi_z=0$ in $[{r_{\text{max}}-\de_0/4}, R]$. {Then, for $B\gg A$ large enough, we obtain, in view of  \eqref{eq:lowerboundforQfplusQhforwellsuitedchoicefandhinGGSR:2} and \eqref{eq:lowerboundforQfplusQhforwellsuitedchoicefandhinGGSR:3},} the following estimate
\beaa
Q^f +Q^{h=B\tilde{h}_0+\tilde{h}_1+B\tilde{h}_2} + Q^z\gtrsim \xirstar^2 +\xit^2+ r^{-1} \Lambda^2  \quad \text{in\, }
{\{r_+(1+\dhor')\leq r\leq R\}\times\GG_{SR}}.
\eeaa
{Furthermore, a quick inspection of the above proof allows us to rewrite the above inequality in the following more precise form
\bea\lab{eq:finallowerboundforQfplusQhpluszforsuitablefhzinGGSR}
&& Q^f +Q^{h=B\tilde{h}_0+\tilde{h}_1+B\tilde{h}_2} + Q^z - d(\xirstar+\eta)^2\in\widetilde{S}^{2,0}_{hom}(\MM),\qquad d\in\widetilde{S}^{0,0}_{hom}(\MM), \quad \eta\in\widetilde{S}^{1,0}_{hom}(\MM),\nn\\
&& Q^f +Q^{h=B\tilde{h}_0+\tilde{h}_1+B\tilde{h}_2} + Q^z - d(\xirstar+\eta)^2\gtrsim \xit^2+ r^{-1} \Lambda^2\quad\textrm{and}\,\, d\gtrsim 1\nn\\ 
&&\qquad\qquad\text{in\, }\{r_+(1+\dhor')\leq r\leq R\}\times\GG_{SR},
\eea
where 
\beaa
d:=(r^2+a^2)\left(2\pr_rf+B\tilde{h}_0 -\frac{c'}{r^2} +B\tilde{h}_2\right), \qquad \eta:=\frac{(\R)A\om_\HH\pr_r\chi_z}{2d}\xiphi.
\eeaa
Notice from our choice of $\chi_z$ above that 
\bea\lab{eq:remarkthatsymbolbvanishesatr=rmaxatinfiniteorder}
\pr^l_r\eta=0 \quad\forall l\geq 0\quad\textrm{on}\quad\{r=r_{\text{max}}\}\times\GG_{SR}.
\eea}

It remains to consider the {symbol} of the boundary terms. {We start with the boundary term on $r=r_+(1+\dhor')$.} By \eqref{eq:flux:fixr}, we have 
{\beaa
\sigma_{2,\textbf{BDR}}^{h=B\tilde{h}_0}=0, \qquad \sigma_{2,\textbf{BDR}}^{h=\tilde{h}_1}=0, \qquad \sigma_{2,\textbf{BDR}}^{h=B\tilde{h}_2}=0\quad\textrm{on}\quad {\{r=r_+(1+\dhor')\}\times\GG_{SR}},
\eeaa}
and
\beaa
\sigma_{2,\textbf{BDR}}^z  
&=& {-A(\xit+\omega_{\HH}\xiphi)(\Delta\xi_r +\S_1)}=-A(\Delta k_+ \xi_r +k_+\S_1)\quad\textrm{on}\quad{\{r=r_+(1+\dhor')\}\times\GG_{SR}},\\
\sigma_{2,\textbf{BDR}}^f &=& {2\mu \xi_r\S_1+\frac{2\S_1^2}{\R}+\mu\S_2}\\
&=& (\R)\bigg(\xit^2 -V+\bigg(\frac{\S_1}{\R}\bigg)^2\bigg) {+2\mu \xi_r\S_1}\quad\textrm{on}\quad{\{r=r_+(1+\dhor')\}\times\GG_{SR}},
\eeaa
where we used the fact that $f=-1$ and $\chi_z=1$ on $r=r_+(1+\dhor')$, $k_+=\xit+\omega_{\HH}\xiphi$ and \eqref{eq:alternativeexpressionforS2}. Together with \eqref{property:F1symbol}{,} the following equation 
\bea\lab{eq:usefulidentityforxit2minusV:v2}
\xit^2- V=\frac{(a\xiphi+ (r^2+a^2)\xit)^2}{(\R)^2}  -\frac{\Delta(\Lambda^2+2a\xi_\tau\xiphi)}{(\R)^2},
\eea 
using also $\Lambda^2+2a\xi_\tau\xiphi\geq 0$ from \eqref{eq:Lambdasquare:bound} and the fact that $A>2$, we  conclude that 
{\bea\lab{eq:finallowerboundfBRDforsuitablefhzinGGSR:onAA}
\nn&&\sigma_{2,\textbf{BDR}}^f +\sigma_{2,\textbf{BDR}}^{h=B\tilde{h}_0+\tilde{h}_1+B\tilde{h}_2}+\sigma_{2,\textbf{BDR}}^z +\varrho^2 +\varpi^2\\
&=& \mu\varrho\widetilde{S}^{1,1}_{hom}(\MM) + \mu^2\widetilde{S}^{2,1}_{hom}(\MM),\,\,\,\textrm{on}\,\,{\{r=r_+(1+\dhor')\}\times\GG_{SR}}, \quad\varrho,\, \varpi\in\widetilde{S}^{1,0}_{hom}(\MM),
\eea
where\footnote{{Notice that 
\beaa
\frac{A-2}{2}(\R)k_+^2-\frac{\Delta(\Lambda^2+2a\xi_\tau\xiphi)}{\R}>0\quad\textrm{on}\quad\{r\geq r_+(1+\dhor'/2)\}\times\{\xi'\neq 0\}
\eeaa
so that we have indeed $\varpi\in\widetilde{S}^{1,0}_{hom}(\MM)$.}}
\beaa
\varrho:=\sqrt{\frac{A-2}{2}(\R)}k_+, \qquad \varpi:=\chi(r)\sqrt{\frac{A-2}{2}(\R)k_+^2-\frac{\Delta(\Lambda^2+2a\xi_\tau\xiphi)}{\R}},
\eeaa
with $\chi$ a smooth cut-off function such that $0\leq \chi\leq 1$, $\chi=1$ at $r=r_+(1+\dhor')$ and $\chi$ supported in $r\geq r_+(1+\frac{\dhor'}{2})$.}

{Finally, we consider the boundary term on $r=R$. Since} we have chosen on $r=R$ that $f=1- mR^{-1}$, ${h=-c'mR^{-2}}$, and {$z=A\xit$, we have, using also the fact that $\sigma_{2,\textbf{BDR}}^{h}=0$ in view of \eqref{eq:flux:fixr},
\bea\lab{eq:finallowerboundfBRDforsuitablefhzinGGSR:onr=R}
\sigma_{2,\textbf{BDR}}^f {+\sigma_{2,\textbf{BDR}}^{h=B\tilde{h}_0+\tilde{h}_1+B\tilde{h}_2}}+\sigma_{2,\textbf{BDR}}^z = \sigma_{2,\textbf{BDR}}^{f=1-{mR^{-1}}} +\sigma_{2,\textbf{BDR}}^{z=A\xit}\quad\textrm{on}\quad\{r=R\}\times\GG_{SR}. 
\eea}

%%%%%%%%%%%%%%%%%%%%%%

\subsubsection{Estimates in $\GG_{A}$}

%%%%%%%%%%%%%%%%%%%%%%

{{Since $\Lambda^2> \df^{-1}m\xit^2$, and since $\Lambda^2\geq \xiphi^2$, it then follows that $\Lambda^2>\df^{-\frac{1}{2}}m^{\frac{1}{2}}|\xit\xiphi|$. Then, by taking $m^{-3} (m-a)^2\les \df\leq m\b^2\les m^{-3} (m-a)^2$ where $\b$ is the constant in point \ref{point1:potentialV:maxandminbounds:frequencies}  of Lemma \ref{lem:SRnottrap}, the triplet $(\xit,\xiphi, \Lambda)$ automatically satisfy $-\xit\xiphi\leq \omega_{\HH}\xiphi^2+\b\Lambda^2$, and hence \eqref{property:potential:superradiant:derivatives} is valid with $r_{\text{max}}- r_+\gtrsim m-a$. In addition, we compute 
\beaa
(V-\xit^2)\vert_{r=5m}&=&\bigg(\frac{\Delta\Lambda^2 -4amr\xit\xiphi-a^2 \xiphi^2}{(\R)^2}-\xit^2\bigg)\vert_{r=5m}\nn\\
&\geq&\bigg(\frac{(r^2-2mr+a^2)\Lambda^2 -2mr\Lambda^2-a^2 \Lambda^2}{(\R)^2}-\xit^2\bigg)\vert_{r=5m}\nn\\
&\geq&\bigg(\frac{(r-4m)r\Lambda^2 }{(\R)^2}-\xit^2\bigg)\vert_{r=5m}\nn\\
&\gtrsim& m^{-2}\Lambda^2
\eeaa
for $\df$ small enough, which yields that  \eqref{eq:V:case2:property4:superradiantisnottrapped} is also valid. Hence,  the symbol $V$ satisfies the inequalities \eqref{property:potential:SR:Kerr} as well. } }

We can thus apply the same argument as the one in Section \ref{sect:GGSR:pseudo} and {deduce that the estimates \eqref{eq:finallowerboundforQfplusQhpluszforsuitablefhzinGGSR},  \eqref{eq:finallowerboundfBRDforsuitablefhzinGGSR:onAA} and \eqref{eq:finallowerboundfBRDforsuitablefhzinGGSR:onr=R} hold in $\GG_A$ as well. More precisely, for the same choice of symbols $(f, h, z)$ as in Section \ref{sect:GGSR:pseudo}, there holds
\bea\lab{eq:finallowerboundforQfplusQhpluszforsuitablefhzinGGA}
&& Q^f +Q^{h=B\tilde{h}_0+\tilde{h}_1+B\tilde{h}_2} + Q^z - d(\xirstar+\eta)^2\in\widetilde{S}^{2,0}_{hom}(\MM),\qquad d\in\widetilde{S}^{0,0}_{hom}(\MM), \quad \eta\in\widetilde{S}^{1,0}_{hom}(\MM),\nn\\
&& Q^f +Q^{h=B\tilde{h}_0+\tilde{h}_1+B\tilde{h}_2} + Q^z - d(\xirstar+\eta)^2\gtrsim \xit^2+ r^{-1} \Lambda^2\quad\textrm{and}\,\, \eta\gtrsim 1\nn\\ 
&&\qquad\qquad\text{in\, }\{r_+(1+\dhor')\leq r\leq R\}\times\GG_{A},
\eea
{\bea\lab{eq:finallowerboundfBRDforsuitablefhzinGGA:onAA}
\nn&&\sigma_{2,\textbf{BDR}}^f +\sigma_{2,\textbf{BDR}}^{h=B\tilde{h}_0+\tilde{h}_1+B\tilde{h}_2}+\sigma_{2,\textbf{BDR}}^z +\varrho^2 +\varpi^2\\
&=& \mu\varrho\widetilde{S}^{1,1}_{hom}(\MM) + \mu^2\widetilde{S}^{2,1}_{hom}(\MM),\,\,\,\textrm{on}\,\,{\{r=r_+(1+\dhor')\}\times\GG_{A}},\quad\varrho,\, \varpi\in\widetilde{S}^{1,0}_{hom}(\MM),
\eea}
and
\bea\lab{eq:finallowerboundfBRDforsuitablefhzinGGA:onr=R}
\sigma_{2,\textbf{BDR}}^f {+\sigma_{2,\textbf{BDR}}^{h=B\tilde{h}_0+\tilde{h}_1+B\tilde{h}_2}}+\sigma_{2,\textbf{BDR}}^z = \sigma_{2,\textbf{BDR}}^{f=1-mR^{-1}} +\sigma_{2,\textbf{BDR}}^{z=A\xit}\quad\textrm{on}\quad\{r=R\}\times\GG_{A}. 
\eea}

%%%%%%%%%%%%%%%%%%%%%%

\subsubsection{Estimates in $\GG_{T}$}

%%%%%%%%%%%%%%%%%%%%%%

In this regime, {recalling \eqref{eq:Lambdasquare:bound}, the lower bound} $\xit^2- V\geq 2c\xit^2 \geq c (\xit^2 +\xiphi^2+\Lambda^2)$ holds globally for a constant $c>0$.  {Therefore, choosing $y=y(r)$ to be increasing from $y(r_+(1+\dhor'))=0$ to $y(R)=1- mR^{-1}$ and to satisfy $\pr_r y \gtrsim r^{-2}$, we obtain, in view of \eqref{eq:bulk:currents:EnerIden},
\beaa
Q^y&\gtrsim& \xirstar^2 + \xit^2+\xiphi^2 +\Lambda^2.
\eeaa}
For the sum $Q^y +Q^{h=-c' mr^{-2}}+ {Q^{z=A\xit}}$, where {$A>2$} and  $c'>0$ is a suitably small constant, we infer{, noticing that $Q^{z=A\xit}=0$,} 
\beaa
Q^y +Q^{h=-c'mr^{-2}}+Q^{z=A\xit}\gtrsim ( \xirstar^2 + \xit^2+\xiphi^2 +\Lambda^2)\quad \text{in\, }{\{r_+(1+\dhor')\leq r\leq R\}\times\GG_{T}}.
\eeaa
{Furthermore, a quick inspection of the above proof allows us to rewrite the above inequality in the following more precise form
\bea\lab{eq:finallowerboundforQfplusQhpluszforsuitablefhzinGGT}
&& Q^y +Q^{h=-c'mr^{-2}} + Q^{z=A\xit} - d(r)\xirstar^2\in\widetilde{S}^{2,0}(\MM),\nn\\
&& Q^y +Q^{h=-c'mr^{-2}} + Q^{z=A\xit} - d(r)\xirstar^2\gtrsim \xit^2+\xiphi^2 +\Lambda^2\quad\textrm{and}\,\, d(r)\gtrsim 1\nn\\ 
&&\qquad\qquad\text{in\, }\{r_+(1+\dhor')\leq r\leq R\}\times\GG_{T},
\eea
where 
\beaa
d(r):=(r^2+a^2)\left(\pr_ry-\frac{c'}{r^2}\right).
\eeaa}

{Next, we estimate the symbol} of the boundary terms starting with the boundary term on $r=r_+(1+\dhor')$. Since $y=0$ on $r=r_+(1+\dhor')$,  it follows that 
{\beaa
\sigma_{2,\textbf{BDR}}^y =0\quad\textrm{on}\quad\{r=r_+(1+\dhor')\}\times\GG_{T}.
\eeaa}
{Also}, using \eqref{property:F1symbol} and \eqref{eq:flux:fixr}, we have $\sigma_{2,\textbf{BDR}}^{h=-c'mr^{-2}}=0$ {and}  
\beaa
\bsplit
\sigma_{2,\textbf{BDR}}^{z=A\xit} 
=A\big(-\Delta \xit\xi_r - \xit\S_1\big) =& -(\R)A\xit^2 -A\xit\Big(\De\xi_r+(\R)\om_\HH\xiphi+O(|\mu|)(\xit, a\xiphi)\Big)\\
&\textrm{on}\quad\{r=r_+(1+\dhor')\}\times\GG_{T}.
\end{split}
\eeaa
{We deduce  
{\bea\lab{eq:finallowerboundfBRDforsuitablefhzinGGT:onAA}
\nn&&\sigma_{2,\textbf{BDR}}^y+ \sigma_{2,\textbf{BDR}}^{h=-c'r^{-2}}+\sigma_{2,\textbf{BDR}}^{z=A\xit} +\varrho^2 +\varpi^2\\ 
&=& \mu\varrho\widetilde{S}^{1,1}_{hom}(\MM) \,\,\,\textrm{on}\,\,\,\{r=r_+(1+\dhor')\}\times\GG_{T},\qquad \varrho,\,\varpi\in\widetilde{S}^{1,0}_{hom}(\MM),
\eea
where\footnote{{In fact, $\varpi$ is in $\widetilde{S}^{1,0}_{hom}(\MM)$ only when restricted to $\GG_T$ which is sufficient for our applications in Section \ref{subsect:CondDegMFesti:Kerr}.}}
\beaa
\varrho:=\sqrt{\frac{A}{2}(\R)}\xit, \qquad \varpi:=\sqrt{\frac{A}{2}(\R)\Big(\xit^2+2\om_\HH\xit\xiphi\Big)}.
\eeaa} 
Finally, we consider the boundary term on $r=R$. Since we have chosen on $r=R$ that $y=1-mR^{-1}$, we obtain
\bea\lab{eq:finallowerboundfBRDforsuitablefhzinGGT:onr=R}
\sigma_{2,\textbf{BDR}}^y + \sigma_{2,\textbf{BDR}}^{h=-c'mr^{-2}}+\sigma_{2,\textbf{BDR}}^{z=A\xit} = \sigma_{2,\textbf{BDR}}^{y=1-mR^{-1}} +\sigma_{2,\textbf{BDR}}^{z=A\xit}\quad\textrm{on}\quad\{r=R\}\times\GG_{T}. 
\eea}

%%%%%%%%%%%%%%%%%%%%%%%

\subsubsection{Estimates in $\GG_{TR}$}
\label{sect:GGTR:pseudo}

%%%%%%%%%%%%%%%%%%%%%%%

This is the only frequency regime containing the trapped frequencies, and all the frequencies in this regime are not superradiant. For frequencies in this regime, the symbol $V$ may belong to either of the three cases listed in Lemma \ref{lem:potentialproperty:PS}. In particular, there might be two critical points $r_{\text{min}}=r_{\text{min}}(\xit,\xiphi,\Lambda)$ and $r_{\text{max}}=r_{\text{max}}(\xit,\xiphi,\Lambda)$ of the potential $V$. 

Next, recall from {\eqref{eq:UpperBoundforrmaxtrapOfthePotential}} that if the symbol $V$ has a maximum at $r_{\text{max}}${, then $r_{\text{max}}\leq 8m$}. Additionally, {we have in view of \eqref{eq:usefulidentityforxit2minusV}, 
\bea\lab{eq:potifvedefinitelowerboundforxit2mVatr=r+}
(\xit^2-V)\vert_{r=r_+}=k_+^2>\left(\frac{\de_\FF\La^2}{|\xiphi|}\right)^2\geq \de_\FF^2\La^2
\eea
where we used the fact that $|k_+\xiphi|>\de_\FF\La^2$ in $\GG_{TR}$, as well as \eqref{eq:Lambdasquare:bound}.} Thus, for $\dhor>0$ sufficiently small, there exists:
\begin{itemize}
\item $r_3\in (r_+(1+2\dhor),+\infty)$, $r_3-r_+\geq b_0(\df)$, such that $\xit^2-V> {\frac{1}{4}\df^2}\Lambda^2$ is valid for all $r\in [r_+(1+\dhor'),r_3]$,

\item $r_4\in (r_+(1+2\dhor),+\infty)$, $r_4-r_+\geq b_0(\df)$, such that $\xit^2-V> {\frac{1}{2}\df^2}\Lambda^2$ is valid for all $r\in [r_+(1+\dhor'),r_4]$,
\end{itemize}
where $b_0(\df)>0$ {approaches} $0$ as $\df\to 0^+$.

Hereafter, we consider two sub-regimes of $\GG_{TR}$  based on if $r_3$ (or $r_4$) is larger than $R$ or not.
Let 
\bsub
\begin{align}
\GG_{TR, 1}:={}&\GG_{TR}\cap\{(\xit,\xiphi,\Lambda): \sup r_3> R \},\\
\GG_{TR, 2}:={}&\GG_{TR}\cap\{(\xit,\xiphi,\Lambda): \sup r_4<R\}.
\end{align}
\esub
Clearly, these two sub-regimes are open sets and it holds that  
\beaa
\GG_{TR}=\GG_{TR, 1}\cup \GG_{TR, 2}, \qquad \GG_{TR, 1}\cap \GG_{TR, 2}\neq \emptyset.
\eeaa

First, we consider the frequency sub-regime $\GG_{TR, 1}$. 
In this case, these frequencies are not trapped, and one can simply consider the symbol $Q^y$, where {$y=y(r)$} is monotonically increasing with $y\vert_{r=r_+(1+\dhor')}=0$, $y\vert_{r=R}=1-mR^{-1}$, $y$ satisfies\footnote{{For instance, the following choice for $y$ works, for a constant $C_2\gg 1$ large enough, 
\beaa
y(r)=C_1\left(e^{-\frac{C_2}{\df^2r}}-e^{-\frac{C_2}{\df^2r_+}}\right), \qquad C_1:=(1-R^{-1})\left(e^{-\frac{C_2}{\df^2R}}-e^{-\frac{C_2}{\df^2r_+}}\right)^{-1}.
\eeaa}} 
\beaa
\pr_r y (\xit^2 -V)-y\pr_r V\geq c_0(\df) r^{-2} (\xit^2 +r^{-2}\Lambda^2 +r^{-2}\xiphi^2), \qquad r\in[r_+(1+\dhor'), R],
\eeaa
{and $y$ is positive for $r\in[r_+(1+\dhor'), R]$}. We subsequently add $Q^{z=A\xit}${, with {$A>2$}, in order} to get good control of the {symbol} for the boundary term at $r=r_+(1+\dhor')$. This implies
\beaa
Q^y +Q^{h=-c'mr^{-2}}+{Q^{z=A\xit}}
 &\geq&c_0(\df)  (\xirstar^2 + \xit^2+r^{-2}\xiphi^2 +r^{-2}\Lambda^2)\nn\\
&&{\text{in\, }\{r_+(1+\dhor')\leq r\leq R\}\times\GG_{TR, 1}},
\eeaa
{which we rewrite in the following more precise form
\beaa
&& Q^y +Q^{h=-c'mr^{-2}} + Q^{z=A\xit} - d(r)\xirstar^2\in\widetilde{S}^{2,0}(\MM),\nn\\
&& Q^y +Q^{h=-c'mr^{-2}} + Q^{z=A\xit} - d(r)\xirstar^2\gtrsim\xit^2+r^{-2}\xiphi^2 +r^{-2}\Lambda^2\quad\textrm{and}\,\, d(r)\gtrsim 1\nn\\ 
&&\qquad\qquad\text{in\, }\{r_+(1+\dhor')\leq r\leq R\}\times\GG_{TR,1},
\eeaa
where $d(r):=(r^2+a^2)(\pr_ry-\frac{c'}{r^2})$ and where $c'>0$ is a small constant that depends on $c_0(\df)$. Also, we have} 
{\beaa
\sigma_{2,\textbf{BDR}}^y {+\sigma_{2,\textbf{BDR}}^{h=-c'mr^{-2}}}+ \sigma_{2,\textbf{BDR}}^{z=A\xit}  
=  -A\Delta \xit\xi_r -A\xit \S_1 \qquad{\textrm{on}\quad\{r=r_+(1+\dhor')\}\times\GG_{TR,1}},
\eeaa
and hence}
{\bea\lab{eq:finallowerboundfBRDforsuitablefhzinGGTR1:onAA}
\nn&&\sigma_{2,\textbf{BDR}}^y {+\sigma_{2,\textbf{BDR}}^{h=-c'mr^{-2}}}+ \sigma_{2,\textbf{BDR}}^{z=A\xit} + \varrho^2+\varpi^2\\
&= & \mu\varrho\widetilde{S}^{1,1}_{hom}(\MM)\quad{\textrm{on}\quad\{r=r_+(1+\dhor')\}\times\GG_{TR,1}}, \quad \varrho,\, \varpi\in\widetilde{S}^{1,0}_{hom}(\MM),
\eea
where 
\beaa
\varrho:=\frac{1}{2}\df^2\sqrt{A(\R)}\xit, \qquad \varpi:=\sqrt{A(\R)\left(\xit k_+-\frac{1}{4}\df^4\xit^2\right)}
\eeaa
where the fact that}\footnote{{To check \eqref{eq:TR:tauk+:bound}, notice first that it holds trivially if $\xit\xiphi\geq 0$. Next, if $\xit\xiphi<0$, we have either $\xit\xiphi>-\frac{1}{4}\df\Lambda^2$ or $\xiphi k_+<-\df\La^2$ in $\GG_{TR}$. In the case where $\xit\xiphi>-\frac{1}{4}\df\Lambda^2$,  we have $\xit k_+ >\xit^2 -\frac{a}{8mr_+}\df\Lambda^2> \frac{3}{8}\df m^{-1} \Lambda^2$. In the other case where $\xiphi k_+<-\df\La^2$ in $\GG_{TR}$, this  implies $\xit k_+>-\frac{\xit}{\xiphi}\df\La^2$ and \eqref{eq:TR:tauk+:bound} follows from $\xit^2\geq \frac{1}{2}\df\La^2$ in $\GG_{TR}$ and {$|\xiphi|\leq\La$} in view of \eqref{eq:Lambdasquare:bound}.}} 
\begin{align}
\label{eq:TR:tauk+:bound}
\xit k_+\geq {\frac{1}{\sqrt{2}}\df^{\frac{3}{2}}\Lambda^2 \geq\frac{1}{2\sqrt{2}}\df^{\frac{5}{2}}\xit^2} \qquad \text{in \,\, }\GG_{TR}.
\end{align}
{In addition, we have on $r=R$ that $y=1-mR^{-1}$ and hence
\bea\lab{eq:finallowerboundfBRDforsuitablefhzinGGTR1:onr=R}
\sigma_{2,\textbf{BDR}}^y + \sigma_{2,\textbf{BDR}}^{h=-c'mr^{-2}}+\sigma_{2,\textbf{BDR}}^{z=A\xit} = \sigma_{2,\textbf{BDR}}^{y=1-mR^{-1}} +\sigma_{2,\textbf{BDR}}^{z=A\xit}\quad\textrm{on}\quad\{r=R\}\times\GG_{TR,1}. 
\eea}

Next, we consider the frequency sub-regime $\GG_{TR, 2}$.
In this case, $\sup r_4< R$, so there is a maximum critical point $r_{\text{max}}\in (\sup r_4,{8m}]$ of the potential $V$. {Since $|\pr_r V|\les r^{-3}\Lambda^2$, {and in view of the bound  $(\xit^2-V)\vert_{r=r_+}> \de_\FF^2\La^2$ from \eqref{eq:potifvedefinitelowerboundforxit2mVatr=r+}, it follows that we can take $r_4$ such that 
\beaa
r_4-r_+ \geq \frac{1}{|\pr_r V|}\bigg((\xit^2-V)\vert_{r=r_+}- \frac{1}{2}\df^2\Lambda^2\bigg)\geq \frac{\df^2\Lambda^2}{2|\pr_r V|}\gtrsim m^3 \df^2 \gtrsim m^{-3}(m-a)^4,
\eeaa
 and hence 
 \bea
 \lab{eq:rmaxbound:GGTR2}
 r_{\text{max}}-r_+\gtrsim m^{-3}(m-a)^4.
 \eea} 

{Choosing $R\geq 16m$, we infer}  $r_{\text{max}}\in (\sup r_4,R/2]$. {Note also that, if $r_{\text{min}}$ exists, then  we have in view of \eqref{eq:potifvedefinitelowerboundforxit2mVatr=r+} 
\beaa
(\xit^2-V)\vert_{r=r_{\text{min}}}\geq (\xit^2-V)\vert_{r=r_+}\geq \de_\FF^2\La^2
\eeaa
so that there exists a constant $c_0(\df)>0$ such that $r_{\text{min}}+c_0(\df)<\sup r_4$. If $r_{\text{min}}$ does not exist, we only need a symbol $Q^f$ as constructed below, and we deal now with the case where $r_{\text{min}}$ exists in which case we also need to introduce a symbol $Q^y$. More precisely, in this case, we use the symbol $Q^y$, where $y\in\widetilde{S}^{0,0}_{hom}(\MM)$, $y=0$ for $r\geq \sup r_4$, $y$ is increasing, and $y$ satisfies\footnote{{For instance, as $r_{\text{min}}+c_0(\df)<\sup r_4$, the following choice works, for constants $C_1\gg C_2\gg 1$ large enough, 
\beaa
y(r)=C_1e^{-\frac{C_2}{\df^2(r-(r_{\text{min}}+c_0(\df)))}}\quad\textrm{for}\quad r<r_{\text{min}}+c_0(\df), \qquad y(r)=0 \quad \textrm{for}\quad r\geq r_{\text{min}}+c_0(\df).
\eeaa}}, for a large enough constant $C_0\gg 1$,  
\beaa
\pr_r y\geq \frac{C_0}{r^2}, \qquad \pr_r y (\xit^2 -V)-y\pr_r V\geq C_0 r^{-2} (\xit^2 +r^{-2}\Lambda^2 +r^{-2}\xiphi^2), \qquad r\in[r_+(1+\dhor'), r_{\text{min}}],
\eeaa
as well as
\beaa
\pr_r y (\xit^2 -V)-y\pr_r V\geq 0, \qquad r\in[r_{\text{min}}, r_4],
\eeaa
which implies 
\beaa
Q^y\geq C_0\big(\xirstar^2 +\xit^2 +r^{-2}\Lambda^2 +r^{-2}\xiphi^2\big)\quad{\text{in\, }\{r_+(1+\dhor')\leq r\leq r_{\text{min}}\}\times\GG_{TR, 2}},
\eeaa
and
\beaa
Q^y\geq 0\quad{\text{in\, }\{r_+(1+\dhor')\leq r\leq R\}\times\GG_{TR, 2}}.
\eeaa
Next, we introduce $Q^f$
\begin{itemize}
\item either to obtain coercivity on $[r_+(1+\dhor'), R]$ if $r_{\text{min}}$ does not exist,
\item or to compensate the lack of coercivity of $Q^y$ in the region $[r_{\text{min}}, R]$ if $r_{\text{min}}$ exists.
\end{itemize}}
{In view of the property\footnote{This is a quantitative characterization of the well-known fact that the  trapped null geodesic flow in a subextremal Kerr spacetime is unstable.} $\frac{d^2 V}{dr^2}(r_{\text{max}})\leq -b_1(\df)\Lambda^2$  that is proven in \cite[Lemma 8.6.1]{DRSR} for a constant $b_1(\df)>0$ depending only on $\df$, $a$ and $m$,} we use the symbol $Q^f$, where {$f\in\widetilde{S}^{0,0}_{hom}(\MM)$,} $f(r_+(1+\dhor'))=0$, $f(r_{\text{max}})=0$, $f(R)= 1- mR^{-1}$ and $\pr_r f\gtrsim r^{-2}$ {on $[r_{min}, R]$}, to obtain 
{\beaa
Q^f\gtrsim \xirstar^2 + (r-r_{\text{max}})^2(r^{-2}\xit^2+r^{-4}\xiphi^2 +r^{-4}\Lambda^2)\quad\text{in\, }\{ r_{\text{min}}\leq r\leq R\}\times\GG_{TR, 2}
\eeaa
and, if $r_{\text{min}}$ exists, 
\beaa
Q^f\gtrsim -(\xit^2+r^{-2}\xiphi^2 +r^{-2}\Lambda^2)\quad\text{in\, }\{r_+(1+\dhor')\leq r\leq r_{min}\}\times\GG_{TR, 2}.
\eeaa
The negative contribution of $Q^f$ in $r\in [r_+(1+\dhor'), r_{\text{min}}]$, in the case where $r_{\text{min}}$ exists, is absorbed by the coercivity of $Q^y$ on $r\in [r_+(1+\dhor'), r_{\text{min}}]$ provided the constant $C_0$ is chosen large enough, and we thus obtain, for $c'>0$ a small constant, and $\chi(r)$ is a smooth nonnegative cutoff function\footnote{{The cut-off $\chi$ allows to have $h=-c'R^{-2}$ at $r=R$ and also ensures that $Q^h=0$ for $r\leq 10m$ and hence in a neighborhood of $r_{\text{max}}$ since $r_{\text{max}}\leq 8m$.}} that vanishes for $r\leq 10m$ and equals $1$ for $r\geq 11m$,} 
\beaa
&&Q^y +Q^{{h=-c'\chi(r)mr^{-2}}}+Q^f+  Q^{z=A\xit}\nn\\
&\gtrsim &  \xirstar^2 + (r-r_{\text{max}})^2(r^{-2}\xit^2+r^{-4}\xiphi^2 +r^{-4}\Lambda^2) \nn\\
&&{\text{in\, }\{r_+(1+\dhor')\leq r\leq R\}\times\GG_{TR, 2}},
\eeaa
{which we rewrite in the following more precise form
\bea\lab{eq:finallowerboundforQfplusQhpluszforsuitablefhzinGGTR2}
&& Q^y +Q^{{h=-c'\chi(r)mr^{-2}}}+Q^f+  Q^{z=A\xit} - d\xirstar^2\in\widetilde{S}^{2,0}_{hom}(\MM),\qquad d\in\widetilde{S}^{0,0}_{hom}(\MM),\nn\\
&& Q^y +Q^{{h=-c'\chi(r)mr^{-2}}}+Q^f+  Q^{z=A\xit} - d\xirstar^2\gtrsim  (r-r_{\text{max}})^2(r^{-2}\xit^2+r^{-4}\xiphi^2 +r^{-4}\Lambda^2),\nn\\
&&\textrm{and}\,\, d\gtrsim 1\quad\text{in\, }\{r_+(1+\dhor')\leq r\leq R\}\times\GG_{TR,2},
\eea
where 
\beaa
d:=(r^2+a^2)\left(\pr_ry+2\pr_rf -\frac{c'}{r^2}\right).
\eeaa
Furthermore}, we compute the {symbol} of the boundary term at $r=r_+(1+\dhor')$. {Using in particular $f(r_+(1+\dhor'))=0$, we have}
\begin{align*}
&(\sigma_{2,\textbf{BDR}}^y+\sigma_{2,\textbf{BDR}}^{h=-c'r^{-2}}+\sigma_{2,\textbf{BDR}}^f+ \sigma_{2,\textbf{BDR}}^{z=A\xit} )\vert_{r_+(1+\dhor')} \nn\\
= {}&
c_{y} \bigg({2\mu \xi_r \S_1} + \frac{2\S_1^2}{\R} +\mu \S_2\bigg)\bigg\vert_{r_+(1+\dhor')}
-(A\Delta \xit\xi_r +A\xit \S_1)\vert_{r_+(1+\dhor')},
\end{align*}
where $c_y=-y\vert_{r=r_+(1+\dhor')}>0$ by the above construction. {Then,} we can take {$A>2$} suitably large, which depends on {$c_y$ and} $\df$,  such that  
{\bea\lab{eq:finallowerboundfBRDforsuitablefhzinGGTR2:onAA}
&&\sigma_{2,\textbf{BDR}}^y+\sigma_{2,\textbf{BDR}}^{h}
+\sigma_{2,\textbf{BDR}}^f+ \sigma_{2,\textbf{BDR}}^{z=A\xit} +\varrho^2+\varpi^2
\nn\\
&=&\mu\varrho\widetilde{S}^{1,1}_{hom}(\MM)+\mu^2\widetilde{S}^{2,1}_{hom}(\MM)\,\,\,\textrm{on}\,\,\{r=r_+(1+\dhor')\}\times\GG_{TR,2}, \,\,\,\, \varrho, \,\varpi\in\widetilde{S}^{1,0}_{hom}(\MM),
\eea
where, using in particular \eqref{eq:usefulidentityforxit2minusV:v2}, we have\footnote{{The fact that $\varpi\in\widetilde{S}^{1,0}_{hom}(\MM)$ relies on  \eqref{eq:TR:tauk+:bound} provided we choose $A>2$ large enough. Also, notice that $\varrho$ and $\varpi$ are in $\widetilde{S}^{1,0}_{hom}(\MM)$ only when restricted to $\GG_T$ which is sufficient for our applications in Section \ref{subsect:CondDegMFesti:Kerr}.}}
\beaa
\bsplit
\varrho:=& \sqrt{(\R)(\xit^2+a^2\xiphi^2)}, \\ 
\varpi:=& \chi(r)\sqrt{(\R)\left(A\xit k_+-2c_yk_+^2  +c_y\frac{\Delta(\Lambda^2+2a\xi_\tau\xiphi)}{\R} -\xit^2-a^2\xiphi^2\right)},
\end{split}
\eeaa
with $\chi$ a smooth cut-off function such that $0\leq \chi\leq 1$, $\chi=1$ at $r=r_+(1+\dhor')$ and $\chi$ supported in $r\geq r_+(1+\frac{\dhor'}{2})$.} 
In addition, we have on $r=R$ that $f=1-mR^{-1}$ and $y=0$ and hence
\bea\lab{eq:finallowerboundfBRDforsuitablefhzinGGTR2:onr=R}
\sigma_{2,\textbf{BDR}}^y + \sigma_{2,\textbf{BDR}}^{h}+ \sigma_{2,\textbf{BDR}}^{f}+\sigma_{2,\textbf{BDR}}^{z=A\xit} = \sigma_{2,\textbf{BDR}}^{y=1-mR^{-1}} +\sigma_{2,\textbf{BDR}}^{z=A\xit}\quad\textrm{on}\quad\{r=R\}\times\GG_{TR,2}. 
\eea

\begin{remark}
It is in the frequency sub-regime $\GG_{TR,2}$ that the trapping degeneracy is present in {the lower bound for the contribution of the sum of symbols $Q$, see the term $(r-r_{\text{max}})^2$ {on} the RHS of \eqref{eq:finallowerboundforQfplusQhpluszforsuitablefhzinGGTR2}. This} is in turn reflected in the choice of the symbol $f$ by requiring {that it vanishes at $r=r_{\text{max}}$}.
\end{remark}

%%%%%%%%%%%%%%%%%%%%%%%%%%%%%%%%%

\subsection{Conditional degenerate Morawetz estimate in Kerr}
\label{subsect:CondDegMFesti:Kerr}

%%%%%%%%%%%%%%%%%%%%%%%%%%%%%%%%%

{The following proposition proves a conditional degenerate Morawetz estimate in Kerr on $\MM_{r_+(1+\dhor'),R}$. 
\begin{proposition}[Conditional degenerate Morawetz estimate in Kerr on $\MM_{r_+(1+\dhor'),R}$]
\lab{prop:microlocalenergyMorawetzinKerronMMrp1dhorpR}
Let $\psi$ {be} a scalar function on $\MM$, and let $\dhor'$ and $R$ be the constants defined respectively by \eqref{de:choiceofdhor'} and \eqref{eq:choiceofRvalue:Kerr}. Then, there {exists a constant $c>0$ and} choices of operators $X\in\Opw(\widetilde{S}^{1,1}(\MM))$ and $E\in\Opw(\widetilde{S}^{0,0}(\MM))$ of the form \eqref{eq:generalformofthePDOmultipliersXandE} such that there holds \bea\lab{eq:microlocal:currentEMF:v0}
\nn&{c\Bigg[}& \int_{\MM_{r_+(1+\dhor'),R}}\frac{\mu^2|\pr_r\psi|^2}{r^2} +\int_{\MM_{r_+(1+\dhor'),10m}}\big(|\Opw(\sigma_{\trap})\psi|^2+|\Opw(x_1)\psi|^2+|\Opw(e)\psi|^2\big)\\
\nn&+&\int_{{\Mntrap_{r_+(1+\dhor'),R}}}\frac{|\pr_\tau\psi|^2+|\nab\psi|^2}{r^2}{\Bigg]}\\
\nn&+&\int_{H_R}\Re\Bigg(-\frac{1}{2}\Opw(\mu^2(\R)s_0)\pr_r\psi \ov{\pr_r\psi} + \ov{\psi}\Opw\Big(\sigma_{2,\textbf{BDR}}^{f+y=1-mR^{-1}}+\sigma_{2,\textbf{BDR}}^{z = A\xit}\Big)\psi\Bigg) d\tt dx^1dx^2\\
\nn&-&{\bigg( \int_{\MM_{r_+(1+\dhor'),R}}  \Re\big({T}^{a,m}_{X,E}\psi\bar{\psi} \big) d\Vref+\textbf{BDR}[\psi]\Big|_{r=r_+(1+\dhor')}^{r=R} \bigg)}\\
\nn&\les_R& \left(\int_{\MM_{r_+(1+\dhor'),R}}|\pr_r\psi|^2\right)^{\frac{1}{2}}\left(\int_{\MM_{r_+(1+\dhor'),R}} |\psi|^2\right)^{\frac{1}{2}}+\int_{\MM_{r_+(1+\dhor'),R}} |\psi|^2\\
&&+\dhor^2\int_{H_{r_+(1+\dhor')}}|\pr\psi|^2+\bigg|\int_{\MM_{r_+(1+\dhor'),R}}  \Re\left(\bar{\psi}\Opw(\mu\widetilde{S}^{0,2}(\MM))\psi \right) d\Vref\bigg|
\nn\\
&&+\left(\int_{H_{r_+(1+\dhor')}}|\pr\psi|^2\right)^{\frac{1}{2}}\left(\int_{\MM_{r_+(1+\dhor'),R}}\big(|\pr_r\psi|^2+|\psi|^2\big)\right)^{\frac{1}{4}}\left(\int_{\MM_{r_+(1+\dhor'),R}}|\psi|^2\right)^{\frac{1}{4}}\nn\\
&&+\left(\int_{H_{R}}\big(|\pr\psi|^2+|\psi|^2\big)\right)^{\frac{1}{2}}\left(\int_{H_{R}}|\psi|^2\right)^{\frac{1}{2}},
\eea
where the symbol $\sigma_{\trap}\in\widetilde{S}^{1,0}(\MM)$ is defined in  \eqref{eq:definitionofthesymbolsigmatrap} depending on the symbol $r_{\trap}\in\widetilde{S}^{0,0}(\MM)$ introduced below in \eqref{eq:definitionofrtrapinfunctionofrmaxandcutoffinmathcalG5}, and where the symbols $x_1, e\in\widetilde{S}^{1,0}(\MM)$ are introduced below respectively in \eqref{eq:definitionofthesymbolx1patXinwidetildeS10} and \eqref{eq:defintionofthesymboleasasquarerootofsigma2TXEmsumofsquares:1:defe}.
\end{proposition}}

{\begin{remark}
In view of Lemma \ref{lem:EnerIden:PDO} with $\g=\gam$, the {last line on the LHS} of \eqref{eq:microlocal:currentEMF:v0} is equal to 
\beaa
{\int_{\MM_{r_+(1+\dhor'),R}}\Re\big(\Box_{\gam}\psi\overline{(X+E)\psi}\big)}
\eeaa
so that \eqref{eq:microlocal:currentEMF:v0} corresponds to a conditional degenerate Morawetz estimate in Kerr.
\end{remark}}

\begin{proof}
{The proof proceeds in the following steps.}

{\noindent{\bf Step 1.}} Let $\{\GG_j\}_{j=1,2,3,4,5}$ be the open sets $\{\GG_{SR}, \GG_{A}, \GG_T, \GG_{TR,1}, \GG_{TR,2}\}$, respectively. In  the above Sections \ref{sect:GGSR:pseudo}--\ref{sect:GGTR:pseudo}, we have made choices of symbols $(h_j, y_j, f_j, z_j)$ {such that $h_j, y_j, f_j\in\widetilde{S}^{0,0}_{hom}(\MM)$ and $z_j\in\widetilde{S}^{1,0}_{hom}(\MM)$,} such that the symbols $\{Q(h_j, y_j, f_j, z_j)\}_{j=1,2,3,4,5}$ satisfy
{\bea
\label{eq:bulksymbol:eachregime:PD}
&& Q^{y_j}+Q^{h_j}+Q^{f_j}+Q^{z_j} -(d_j\xirstar+\eta_j)^2\in\widetilde{S}^{2,0}_{hom}(\MM),\quad d_j\in\widetilde{S}^{0,0}_{hom}(\MM), \quad \eta_j\in\widetilde{S}^{1,0}_{hom}(\MM),\nn\\
&& Q^{y_j}+Q^{h_j}+Q^{f_j}+Q^{z_j} -(d_j\xirstar+\eta_j)^2 \gtrsim P_{j}\quad\textrm{and}\quad d_j\gtrsim 1, \nn\\ 
&&\text{on }\quad\{r_+(1+\dhor')\leq r\leq R\}\times\GG_j\quad\text{for\,\, } j=1,2,3,4,5,
\eea}
where 
\bsub\lab{eq:propertiessatisfiedbyPjforj=12345}
\begin{align}
P_{j}&={\xit^2+r^{-2}\xiphi^2 +r^{-2}\Lambda^2} \quad \text{for\,\, } j=1,2,3,4,\\
P_{j}&={(r-r_{\text{max}})^2(r^{-2}\xit^2+r^{-4}\xiphi^2 +r^{-4}\Lambda^2)} \quad \text{for\,\, } j=5,\\
{\pr_r^l\eta_j}&=0\quad\forall l\geq 0\quad\textrm{on}\quad\{r=r_{\text{max}}\}\times{\GG_j \quad \text{for\,\, }j=1,2,} \qquad \eta_{j}=0 \quad \text{for\,\, } j=3,4,5,
\end{align}
\esub
such that the symbols $\{\sigma_{2,\textbf{BDR}}(h_j, y_j, f_j, z_j)\}_{j=1,2,3,4,5}$ satisfy\footnote{{In fact, $\varrho_j$ and $\varpi_j$ are in $\widetilde{S}^{1,0}_{hom}(\MM)$ when restricted to $\GG_j$ which is sufficient as they will be multiplied in Step 2 below by a smooth cut-off function $\chi_j$ supported in $\GG_j$.}}
{\bea
\label{eq:fluxsymbol:eachregime:PD}
&&{\sigma_{2,\textbf{BDR}}^{y_j}+\sigma_{2,\textbf{BDR}}^{h_j}+\sigma_{2,\textbf{BDR}}^{f_j}+\sigma_{2,\textbf{BDR}}^{z_j}}+\varrho_j^2+\varpi_j^2\nn\\
\nn&=& \mu\varrho_j\widetilde{S}^{1,1}_{hom}(\MM)+\mu^2\widetilde{S}^{2,1}_{hom}(\MM)\quad {\textrm{on}\quad\{r=r_+(1+\dhor')\}\times\GG_j},\\
&&\varrho_j,\, \varpi_j\in\widetilde{S}^{1,0}_{hom}(\MM),\quad\text{for\,\, } j=1,2,3,4,5,
\eea}
{and such that} 
\bea
\label{eq:choicesofsymbols:r=R}
h_j=-c'mr^{-2},\,\,\,  y_j+f_j = 1-mR^{-1}, \,\,\,  z_j = A\xit, \,\,\, {\textrm{on}\,\,\, \{r=R\}\times\GG_j}, \,\,\,  \forall j\in\{1,2,3,4,5\}.
\eea

Since there holds $\cup_{j=1}^5 \GG_j=\Gtz$ with $\{\GG_j\}_{j=1,2,3,4,5}$ open sets, there exist $\{\chi_j\}_{j=1,2,3,4,5}=\{\chi_j(\xit,\xiphi,\Lambda)\}_{j=1,2,3,4,5}$, {with $\chi_j\in\widetilde{S}^{0,0}(\MM)$ such that}
\begin{align*}
\text{supp} (\chi_j)\Subset \GG_j, \qquad \sum_{j=1}^5{\chi_j^2} =1 \quad\text{on}\quad \Gtz{\cap\{|\Xi|\geq 2\}, \qquad \chi_j=0\quad\textrm{on}\quad |\Xi|\leq 1}.
\end{align*}
Also, we define $r_{\trap}$ as 
\bea
\lab{eq:definitionofrtrapinfunctionofrmaxandcutoffinmathcalG5}
r_{\trap}:=3m(1-\widetilde{\chi}_5)+\widetilde{\chi}_5r_{max}, \qquad r_{\trap}\in\widetilde{S}^{0,0}(\MM), 
\eea
where $\widetilde{\chi}_5\in\widetilde{S}^{0,0}(\MM)$ is supported in $\mathcal{G}_5$ and $\widetilde{\chi}_5=1$ on the support of $\chi_5$, {and  in view of \eqref{eq:UpperBoundforrmaxtrapOfthePotential} and the bound  $r_+ +{\de_0 m^{-3}(m-a)^4} \leq  r_{\text{max}}\leq 8m$ from \eqref{eq:rmaxbound:GGTR2} in $\GG_{TR,2}=\GG_5$,  it follows
\begin{equation}
\lab{eq:UpperBoundforrmaxtrapOfthePotential:consequencertrap}
\begin{split}
&r_+ +\de_0 m^{-3}(m-a)^4 \leq  r_{\trap}\leq 8m
\end{split}
\end{equation}
for a universal constant $\de_0>0$ that {is independent from $m$ and $a$}.}

{\noindent{\bf Step 2.}} We may now define our choice of symbols {globally in $\Gtz$ as follows}
\bea
\lab{eq:finalchoicesofhyfzsymbols}
(h,y,f,z):=\sum_{j=1}^5 {\chi_j^2}(h_j, y_j, f_j, z_j), \qquad {h,y,f\in\widetilde{S}^{0,0}(\MM), \qquad z\in\widetilde{S}^{1,0}(\MM)}{,}
\eea
{which, in view of \eqref{eq:definitionofQhQyQfQzintermsofTXEands0s1e0}, uniquely prescribes operators $X\in\Opw(\widetilde{S}^{1,1}(\MM))$ and $E\in\Opw(\widetilde{S}^{0,0}(\MM))$ of the form \eqref{eq:generalformofthePDOmultipliersXandE}.} We also have 
\beaa
\bsplit
&\chi_jd_j\in\widetilde{S}^{0,0}(\MM), \quad j=1,2,3,4,5, \qquad \sum_j\chi_j^2d_j^2\gtrsim 1\quad\textrm{on}\,\,\Gtz,\\
&{\chi_j\eta_j\in\widetilde{S}^{1,0}(\MM),  \quad \pr_r^l(\eta_j)=0\quad\forall l\geq 0\quad\textrm{on}\quad\{r=r_{\text{max}}\}\times\GG_j, \quad \text{for\,\, } j=1,2,}\\
&\eta_{j}=0 \quad \text{for\,\, } j=3,4,5,\\
&{\chi_j\varrho_j, \, \chi_j\varpi_j\in\widetilde{S}^{1,0}(\MM)\quad \text{for\,\, } j=1,2,3,4,5.}
\end{split}
\eeaa
Since, in view of \eqref{eq:symbolofBulk:EnerIden} and \eqref{eq:definitionofQhQyQfQzintermsofTXEands0s1e0}, $Q^y$, $Q^h$, $Q^f$ and $Q^z$ depend linearly respectively on $(y, \pr_ry)$, $h$, $(f, \pr_rf)$, and $\pr_rz$, and since $\pr_r\chi_j=0$ for $j=1,2,3,4,5$, we deduce
\beaa
\sigma_2({T^{a.m}_{X,E}}) -\sum_{j=1}^5(\chi_jd_j\xirstar+\chi_j\eta_j)^2 &=& Q^{y}+Q^{h}+Q^{f}+Q^{z}  -\sum_{j=1}^5(\chi_jd_j\xirstar+\chi_j\eta_j)^2\\
&=& \sum_{j=1}^5\chi_j^2\big(Q^{y_j}+Q^{h_j}+Q^{f_j}+Q^{z_j} -(d_j\xirstar+\eta_j)^2\big)\\
&\gtrsim& \sum_{j=1}^5\chi_j^2P_j.
\eeaa
In view of \eqref{eq:bulksymbol:eachregime:PD} and \eqref{eq:propertiessatisfiedbyPjforj=12345}, this yields
\bea\lab{eq:intermadiarysigmaTXEmsumofsqquaresisinwtildeS20MM}
\bsplit
&\sigma_2({T^{a.m}_{X,E}}) -\sum_{j=1}^5(\chi_jd_j\xirstar+\chi_j\eta_j)^2\in\widetilde{S}^{2,0}(\MM),\\
&\sigma_2({T^{a.m}_{X,E}}) -\sum_{j=1}^5(\chi_jd_j\xirstar+\chi_j\eta_j)^2 \gtrsim P_0, \quad\textrm{on}\quad\{r_+(1+\dhor')\leq r\leq R\}\times\Gtz,
\end{split}
\eea
{where $P_0\in\widetilde{S}^{2,0}(\MM)$ is given by
\bea
\lab{eq:defofsymbolP0inwtildeS20MMforlowerboindsigma2TXE}
P_0:= (1-\chi_5^2)r^2(r^{-2}\xit^2+r^{-4}\xiphi^2 +r^{-4}\Lambda^2)+\chi_5^2(r-r_{\trap})^2(r^{-2}\xit^2+r^{-4}\xiphi^2 +r^{-4}\Lambda^2),
\eea}
{with $r_{\trap}$ given by \eqref{eq:definitionofrtrapinfunctionofrmaxandcutoffinmathcalG5}. Now, in view of \eqref{eq:intermadiarysigmaTXEmsumofsqquaresisinwtildeS20MM}, we may introduce
\bea\lab{eq:defintionofthesymboleasasquarerootofsigma2TXEmsumofsquares:1:defe}
e:=\sqrt{\sigma_2(T_{X,E})+1 -\sum_{j=1}^5(\chi_jd_j\xirstar+\chi_j\eta_j)^2}
\eea}
{which satisfies
\bea
\lab{eq:defintionofthesymboleasasquarerootofsigma2TXEmsumofsquares:1}
{e\in\widetilde{S}^{1,0}(\MM), \qquad e\gtrsim 1+ \left(\sqrt{1-\chi_5^2}+|\chi_5||r-r_{\trap}|\right)\sqrt{r^{-2}\xit^2+r^{-4}\xiphi^2 +r^{-4}\Lambda^2},}
\eea}
and, recalling that {$\eta_3=\eta_4=\eta_5=0$,}  
\bea\lab{eq:defintionofthesymboleasasquarerootofsigma2TXEmsumofsquares:2}
\sigma_2({T^{a.m}_{X,E}}) = {(\chi_1d_1\xirstar+\chi_1\eta_1)^2+(\chi_2d_2\xirstar+\chi_2\eta_2)^2+\sum_{j=3}^5(\chi_jd_j\xirstar)^2+e^2 -1. }
\eea
Also, since {$\pr_r^l(\eta_j)=0$ for all $l\geq 0$ and $j=1,2$ on $\{r=r_{\text{max}}\}\times\GG_j$}, and since $r_{\trap}=r_{\text{max}}$ on the support of $\chi_5$, we have
\bea\lab{eq:necessarypropertyforchi1b1tobecontrolledbye}
{|\chi_1\eta_1|+|\chi_2\eta_2|}\les e. 
\eea
Using \eqref{eq:defintionofthesymboleasasquarerootofsigma2TXEmsumofsquares:1} and \eqref{eq:defintionofthesymboleasasquarerootofsigma2TXEmsumofsquares:2}, together with Proposition \ref{prop:PDO:MM:Weylquan:mixedoperators} and the fact that $\xirstar=\mu\xi_r+\widetilde{S}^{1,0}(\MM)$ in view of \eqref{def:xirstar}, we infer
{\beaa
\Opw(\sigma_2({T^{a.m}_{X,E}})) &=& (\Opw(\chi_1d_1\xirstar+\chi_1\eta_1))^2+
(\Opw(\chi_2d_2\xirstar+\chi_2\eta_2))^2+
\sum_{j=3}^5\Opw(\chi_jd_j\xirstar)^2\\
&&+\Opw(e)^2+{\Opw\big(}\mu\widetilde{S}^{0,2}(\MM)+\widetilde{S}^{0,1}(\MM){\big)}.
\eeaa}
Since the symbols $\xirstar$, $\chi_j$, $d_j$, {$\eta_1$, $\eta_2$} and $e$ are real valued, we deduce, using again Proposition \ref{prop:PDO:MM:Weylquan:mixedoperators} as well as Lemma \ref{lemma:actionmixedsymbolsSobolevspaces:MM}, 
\bea\lab{eq:intermediarycontrolofH1normpsibysigma2TXE}
&&\int_{\MM_{r_+(1+\dhor'),R}}  \Re\big(\bar{\psi}\Opw(\sigma_2({T^{a.m}_{X,E}}))\psi \big) d\Vref\nn\\ 
&=& \int_{\MM_{r_+(1+\dhor'),R}}\left({\sum_{i=1}^2|\Opw(\chi_id_i\xirstar+\chi_i\eta_i)\psi|^2+\sum_{j=3}^5|\Opw(\chi_jd_j\xirstar)\psi|^2}+|\Opw(e)\psi|^2\right)d\Vref\nn\\
&&+O(1)\Bigg\{\left(\int_{\MM_{r_+(1+\dhor'),R}}|\pr_r\psi|^2\right)^{\frac{1}{2}}\left(\int_{\MM_{r_+(1+\dhor'),R}} |\psi|^2\right)^{\frac{1}{2}}
+\int_{\MM_{r_+(1+\dhor'),R}} |\psi|^2\\
&&
+\bigg|\int_{\MM_{r_+(1+\dhor'),R}}  \Re\left(\bar{\psi}\Opw(\mu\widetilde{S}^{0,2}(\MM))\psi \right) d\Vref\bigg|\Bigg\}.\nn
\eea
In particular, we infer
\bea\lab{eq:intermediarycontrolofH1normpsibysigma2TXE:1}
&& \int_{\MM_{r_+(1+\dhor'),R}}|\Opw(e)\psi|^2d\Vref\nn\\ 
&\les& \int_{\MM_{r_+(1+\dhor'),R}}  \Re\big(\bar{\psi}\Opw(\sigma_2({T^{a.m}_{X,E}}))\psi \big) d\Vref
+\left(\int_{\MM_{r_+(1+\dhor'),R}}|\pr_r\psi|^2\right)^{\frac{1}{2}}\left(\int_{\MM_{r_+(1+\dhor'),R}} |\psi|^2\right)^{\frac{1}{2}}\nn\\
&&
+\int_{\MM_{r_+(1+\dhor'),R}} |\psi|^2
+\bigg|\int_{\MM_{r_+(1+\dhor'),R}}  \Re\left(\bar{\psi}\Opw(\mu\widetilde{S}^{0,2}(\MM))\psi \right) d\Vref\bigg|,
\eea
which together with \eqref{eq:necessarypropertyforchi1b1tobecontrolledbye} implies
\beaa
&& \int_{\MM_{r_+(1+\dhor'),R}}{\Big(|\Opw(\chi_1\eta_1)\psi|^2+|\Opw(\chi_2\eta_2)\psi|^2\Big)}d\Vref\\ 
&\les& \int_{\MM_{r_+(1+\dhor'),R}}  \Re\big(\bar{\psi}\Opw(\sigma_2({T^{a.m}_{X,E}}))\psi \big) d\Vref
+\left(\int_{\MM_{r_+(1+\dhor'),R}}|\pr_r\psi|^2\right)^{\frac{1}{2}}\left(\int_{\MM_{r_+(1+\dhor'),R}} |\psi|^2\right)^{\frac{1}{2}}\\
&&
+\int_{\MM_{r_+(1+\dhor'),R}} |\psi|^2
+\bigg|\int_{\MM_{r_+(1+\dhor'),R}}  \Re\left(\bar{\psi}\Opw(\mu\widetilde{S}^{0,2}(\MM))\psi \right) d\Vref\bigg|.
\eeaa
In view of \eqref{eq:intermediarycontrolofH1normpsibysigma2TXE}, we deduce
\beaa
&&  \int_{\MM_{r_+(1+\dhor'),R}}\left(\sum_{j=1}^5|\Opw(\chi_jd_j\xirstar)\psi|^2\right)d\Vref\nn\\
&\les& \int_{\MM_{r_+(1+\dhor'),R}}  \Re\big(\bar{\psi}\Opw(\sigma_2({T^{a.m}_{X,E}}))\psi \big) d\Vref
+\left(\int_{\MM_{r_+(1+\dhor'),R}}|\pr_r\psi|^2\right)^{\frac{1}{2}}\left(\int_{\MM_{r_+(1+\dhor'),R}} |\psi|^2\right)^{\frac{1}{2}}\nn\\
&&
+\int_{\MM_{r_+(1+\dhor'),R}} |\psi|^2
+\bigg|\int_{\MM_{r_+(1+\dhor'),R}}  \Re\left(\bar{\psi}\Opw(\mu\widetilde{S}^{0,2}(\MM))\psi \right) d\Vref\bigg|,
\eeaa
which together with Proposition \ref{prop:PDO:MM:Weylquan:mixedoperators} and Lemma \ref{lemma:actionmixedsymbolsSobolevspaces:MM} implies
\beaa
&&  \int_{\MM_{r_+(1+\dhor'),R}}\left|\Opw\left(\sqrt{\sum_{j=1}^5\chi_j^2d_j^2}\right)\circ\Opw(\xirstar)\psi\right|^2d\Vref\nn\\
&\les& \int_{\MM_{r_+(1+\dhor'),R}}  \Re\big(\bar{\psi}\Opw(\sigma_2({T^{a.m}_{X,E}}))\psi \big) d\Vref
+\left(\int_{\MM_{r_+(1+\dhor'),R}}|\pr_r\psi|^2\right)^{\frac{1}{2}}\left(\int_{\MM_{r_+(1+\dhor'),R}} |\psi|^2\right)^{\frac{1}{2}}\nn\\
&&
+\int_{\MM_{r_+(1+\dhor'),R}} |\psi|^2
+\bigg|\int_{\MM_{r_+(1+\dhor'),R}}  \Re\left(\bar{\psi}\Opw(\mu\widetilde{S}^{0,2}(\MM))\psi \right) d\Vref\bigg|.
\eeaa
Since $\sum_{j=1}^5\chi_j^2d_j^2\gtrsim 1$, we infer
\beaa
&&  \int_{\MM_{r_+(1+\dhor'),R}}\left|\Opw(\xirstar)\psi\right|^2d\Vref\nn\\
&\les& \int_{\MM_{r_+(1+\dhor'),R}}  \Re\big(\bar{\psi}\Opw(\sigma_2({T^{a.m}_{X,E}}))\psi \big) d\Vref
+\left(\int_{\MM_{r_+(1+\dhor'),R}}|\pr_r\psi|^2\right)^{\frac{1}{2}}\left(\int_{\MM_{r_+(1+\dhor'),R}} |\psi|^2\right)^{\frac{1}{2}}\nn\\
&&
+\int_{\MM_{r_+(1+\dhor'),R}} |\psi|^2
+\bigg|\int_{\MM_{r_+(1+\dhor'),R}}  \Re\left(\bar{\psi}\Opw(\mu\widetilde{S}^{0,2}(\MM))\psi \right) d\Vref\bigg|,
\eeaa
which {together with \eqref{eq:intermediarycontrolofH1normpsibysigma2TXE:1} implies
\bea\lab{eq:controloffirstorderderivativesinMMwithtrapdegbysigma2TXE:almostthere!0}
\nn&&  \int_{\MM_{r_+(1+\dhor'),R}}\Big(|\pr_{r^*}\psi|^2+|\Opw(e)\psi|^2\Big)d\Vref\\
&\les& \int_{\MM_{r_+(1+\dhor'),R}}  \Re\big(\bar{\psi}\Opw(\sigma_2({T^{a.m}_{X,E}}))\psi \big) d\Vref
+\left(\int_{\MM_{r_+(1+\dhor'),R}}|\pr_r\psi|^2\right)^{\frac{1}{2}}\left(\int_{\MM_{r_+(1+\dhor'),R}} |\psi|^2\right)^{\frac{1}{2}}\nn\\
&&
+\int_{\MM_{r_+(1+\dhor'),R}} |\psi|^2
+\bigg|\int_{\MM_{r_+(1+\dhor'),R}}  \Re\left(\bar{\psi}\Opw(\mu\widetilde{S}^{0,2}(\MM))\psi \right) d\Vref\bigg|.
\eea}

{\noindent{\bf Step 3.} Next, we have, in view of \eqref{eq:generalformofthePDOmultipliersXandE} {and \eqref{eq:definitionofQhQyQfQzintermsofTXEands0s1e0}}, 
\bea\lab{eq:definitionofthesymbolx1patXinwidetildeS10}
\bsplit
X=&\Opw(\widetilde{S}^{0,0}(\MM))\mu\pr_r+A\pr_\tau+\Opw(x_1)+\Opw(\widetilde{S}^{0,0}(\MM)),\\
x_1:=&\frac{is_0 \S_1}{\R} + i s_1 -iA\xit, \quad x_1\in\widetilde{S}^{1,0}(\MM),
\end{split}
\eea
and hence, since $z_j=A(\xit+\chi_z\om_\HH\xiphi)$ for $j=1,2$, and $z_j=A(\xit+\chi_z\om_\HH\xiphi)$ for $j=3,4,5$, 
\bea\lab{eq:definitionofthesymbolx1patXinwidetildeS10:bis}
x_1 &=& i\sum_{j=1}^2\chi_j^2\left(\frac{2(y_j+f_j)}{r^2+a^2}S_1+{A}\chi_z\om_\HH\xiphi\right)+ i\sum_{j=3}^5\chi_j^2\frac{2(y_j+f_j)}{r^2+a^2}S_1.
\eea
Then, notice from \eqref{eq:defintionofthesymboleasasquarerootofsigma2TXEmsumofsquares:1}, \eqref{eq:UpperBoundforrmaxtrapOfthePotential:consequencertrap}, \eqref{eq:definitionofthesymbolx1patXinwidetildeS10:bis} and the definition of the symbol $\sigma_{\trap}\in\widetilde{S}^{1,0}(\MM)$ introduced in \eqref{eq:definitionofthesymbolsigmatrap} that there exists a constant $c>0$ small enough such that
\bea\lab{eq:auxiliarrye1toboundOpwsigmatrap}
e_1:=\sqrt{e^2-c\chi_0(r)(\sigma_{\trap}^2+x_1^2)-c(1-\chi_0(r))\Big(r^{-2}\xit^2+r^{-4}\xiphi^2 +r^{-4}\Lambda^2\Big)}, \quad e_1\in\widetilde{S}^{1,0}(\MM),
\eea
where $\chi_0$ is a smooth cut-off function satisfying $0\leq \chi_0\leq 1$, $\chi_0=1$ for $r\leq 10m$ and $\chi_0=0$ for $r\geq 11m$. In view of Proposition \ref{prop:PDO:MM:Weylquan:mixedoperators} and Lemma \ref{lemma:simplepropertiesspecialcaseofsymbols:recoveringoperators}, and using again \eqref{eq:UpperBoundforrmaxtrapOfthePotential:consequencertrap}, we infer
\beaa
&&\int_{\MM_{r_+(1+\dhor'),10m}}\big(|\Opw(\sigma_{\trap})\psi|^2+|\Opw(x_1)\psi|^2\big)+\int_{\Mntrap_{r\leq R}}\frac{|\pr_\tau\psi|^2+|\nab\psi|^2}{r^2}\\
&\les& \int_{\MM_{r_+(1+\dhor'),R}}\Re\left(\ov{\psi}\Opw\left(c\chi_0(r)\big(\sigma_{\trap}^2+x_1^2\big)+c(1-\chi_0(r))\Big(r^{-2}\xit^2+r^{-4}\xiphi^2 +r^{-4}\Lambda^2\Big)\right)\psi\right)\\
&&\times d\Vref +\int_{\MM_{r_+(1+\dhor'),R}} |\psi|^2\\
&=&  \int_{\MM_{r_+(1+\dhor'),R}}\Re\left(\ov{\psi}\Opw\big(e^2-e_1^2\big)\psi\right)d\Vref +\int_{\MM_{r_+(1+\dhor'),R}} |\psi|^2\\
&\les& \int_{\MM_{r_+(1+\dhor'),R}}|\Opw(e)\psi|^2+\int_{\MM_{r_+(1+\dhor'),R}}|\psi|^2,
\eeaa
which together with \eqref{eq:controloffirstorderderivativesinMMwithtrapdegbysigma2TXE:almostthere!0} finally yields} 
\bea\lab{eq:controloffirstorderderivativesinMMwithtrapdegbysigma2TXE}
\nn&&  {\int_{\MM_{r_+(1+\dhor'),R}}\frac{\mu^2|\pr_r\psi|^2}{r^2} +\int_{\MM_{r_+(1+\dhor'),10m}}\big(|\Opw(\sigma_{\trap})\psi|^2+|\Opw(x_1)\psi|^2+|\Opw(e)\psi|^2\big)}
\\
\nn&&{+\int_{{\Mntrap_{r_+(1+\dhor'),R}}}\frac{|\pr_\tau\psi|^2+|\nab\psi|^2}{r^2}}\\
&\les& \int_{\MM_{r_+(1+\dhor'),R}}  \Re\big(\bar{\psi}\Opw(\sigma_2({T^{a.m}_{X,E}}))\psi \big) d\Vref
+\left(\int_{\MM_{r_+(1+\dhor'),R}}|\pr_r\psi|^2\right)^{\frac{1}{2}}\left(\int_{\MM_{r_+(1+\dhor'),R}} |\psi|^2\right)^{\frac{1}{2}}\nn\\
&&
+\int_{\MM_{r_+(1+\dhor'),R}} |\psi|^2
+\bigg|\int_{\MM_{r_+(1+\dhor'),R}}  \Re\left(\bar{\psi}\Opw(\mu\widetilde{S}^{0,2}(\MM))\psi \right) d\Vref\bigg|,
\eea
{where we have also used the fact that 
\beaa
\int_{\MM_{r_+(1+\dhor'),R}}\frac{\mu^2|\pr_r\psi|^2}{r^2}\les \int_{\MM_{r_+(1+\dhor'),R}}|\pr_{r^*}\psi|^2+\int_{\Mntrap_{r\leq R}}\frac{|\pr_\tau\psi|^2+|\nab\psi|^2}{r^2}
\eeaa
in view of our choice of normalized coordinates in Lemma \ref{lem:specificchoice:normalizedcoord} which ensures that $\pr_{r^*}=\mu\pr_r$ in $\Mtrap$.}

{\noindent{\bf Step 4.}} Next, we consider the boundary terms, starting with the one at $r=r_+(1+\dhor')$. Since, in view of \eqref{eq:flux:fixr}, $\sigma_{2,\textbf{BDR}}^h=0$, and $\sigma_{2,\textbf{BDR}}^y$, $\sigma_{2,\textbf{BDR}}^f$ and $\sigma_{2,\textbf{BDR}}^z$ depend linearly respectively on $y$, $f$, and $z$, we deduce
\beaa
 \sigma_{2, \textbf{BDR}}^{X,E}\vert_{r=r_+(1+\dhor')} &=& \Big(\sigma_{2,\textbf{BDR}}^{y}+\sigma_{2,\textbf{BDR}}^{h}+\sigma_{2,\textbf{BDR}}^{f}+\sigma_{2,\textbf{BDR}}^{z}\Big)\vert_{r=r_+(1+\dhor')}\\
 &=&\sum_{j=1}^5\chi_j^2\Big(\sigma_{2,\textbf{BDR}}^{y_j}+\sigma_{2,\textbf{BDR}}^{h_j}+\sigma_{2,\textbf{BDR}}^{f_j}+\sigma_{2,\textbf{BDR}}^{z_j}\Big)\vert_{r=r_+(1+\dhor')}
\eeaa
which together with \eqref{eq:fluxsymbol:eachregime:PD} implies
{\beaa
&&\int_{H_{r_+(1+\dhor')}}\Re\left(\ov{\psi}\left(\Opw\Big( \sigma_{2, \textbf{BDR}}^{X,E}\Big) +\sum_{j=1}^5\Big(\Opw(\chi_j^2\varrho_j^2)+\Opw(\chi_j^2\varpi_j^2)\Big)\right)\psi \right) d\tt dx^1dx^2\\
&=& \int_{H_{r_+(1+\dhor')}}\Re\left(\ov{\psi}\left(\sum_{j=1}^5\Opw\big(\mu\chi_j\varrho_j\widetilde{S}^{1,1}(\MM)\big)+\Opw\big(\mu^2\widetilde{S}^{2,1}(\MM)\big)\right)\psi \right) d\tt dx^1dx^2
\eeaa
which yields
\beaa
&&\int_{H_{r_+(1+\dhor')}}\left(\Re\left(\ov{\psi}\Opw\Big( \sigma_{2, \textbf{BDR}}^{X,E}\Big)\psi\right) +\sum_{j=1}^5|\chi_j\varrho_j\psi|^2 \right) d\tt dx^1dx^2\\
&=& \left(\sum_{j=1}^5\int_{H_{r_+(1+\dhor')}}|\chi_j\varrho_j\psi|^2 d\tt dx^1dx^2\right)^{\frac{1}{2}}\left(\int_{H_{r_+(1+\dhor')}}\mu^2|\pr\psi|^2 d\tt dx^1dx^2\right)^{\frac{1}{2}}\\
&&+\int_{H_{r_+(1+\dhor')}}\big(\mu^2|\pr\psi|^2+|\psi|^2\big) d\tt dx^1dx^2
\eeaa
and hence
\bea\lab{eq:concludingestimateforsigma2BDRXE:AA}
\int_{H_{r_+(1+\dhor')}}\Re\left(\ov{\psi}\left(\Opw\Big( \sigma_{2, \textbf{BDR}}^{X,E}\Big)\right)\psi \right) d\tt dx^1dx^2 \les \int_{H_{r_+(1+\dhor')}}\big(\mu^2|\pr\psi|^2+|\psi|^2\big) d\tt dx^1dx^2.
\eea
Also, we have, for $s_0\in\widetilde{S}^{0,0}(\MM)$, 
\beaa
\left|\int_{H_{r_+(1+\dhor')}}\Re\Bigg(-\frac{1}{2}\Opw(\mu^2(\R)s_0)\pr_r\psi \ov{\pr_r\psi}\Bigg) d\tt dx^1dx^2\right|\les\int_{H_{r_+(1+\dhor')}}\big(\mu^2 |\pr\psi|^2 +|\psi|^2),
\eeaa
which, together with \eqref{eq:concludingestimateforsigma2BDRXE:AA}, yields
\bea\lab{eq:concludingestimateforsigma2BDRXE:AA:1}
&&\int_{H_{r_+(1+\dhor')}}\Re\Bigg(-\frac{1}{2}\Opw(\mu^2(\R)s_0)\pr_r\psi \ov{\pr_r\psi} + \ov{\psi}\Opw\Big( \sigma_{2, \textbf{BDR}}^{X,E}\Big)\psi\Bigg) d\tt dx^1dx^2\nn\\
&\les&\int_{H_{r_+(1+\dhor')}}\big(\mu^2 |\pr\psi|^2 +|\psi|^2)\nn\\
\nn&\les& \dhor^2\int_{H_{r_+(1+\dhor')}}|\pr\psi|^2 \\
&&+\left(\int_{\MM_{r_+(1+\dhor'), 3m}}\big(|\pr_r\psi|^2+|\psi|^2\big)\right)^{\frac{1}{2}}\left(\int_{\MM_{r_+(1+\dhor'), 3m}}|\psi|^2\right)^{\frac{1}{2}},
\eea
where we have also used the following trace estimate 
\bea\lab{eq:precisetraceestimateonr=rplus1plusdhoprimewithsqrt}
 \int_{H_{r_+(1+\dhor')}}|\psi|^2
 \les \left(\int_{\MM_{r_+(1+\dhor'), 3m}}\big(|\pr_r\psi|^2+|\psi|^2\big)\right)^{\frac{1}{2}}\left(\int_{\MM_{r_+(1+\dhor'), 3m}}|\psi|^2\right)^{\frac{1}{2}}.
 \eea}
 Also, we have in view of \eqref{eq:choicesofsymbols:r=R}
\bea\lab{eq:concludingestimateforsigma2BDRXE:r=R:0}
h=-c'mr^{-2},\quad  y+f = 1-mR^{-1}, \quad z = A\xit, \quad\textrm{on}\quad\{r=R\}
\eea
as well as 
\bea\lab{eq:concludingestimateforsigma2BDRXE:r=R}
 \sigma_{2, \textbf{BDR}}^{X,E} = \sigma_{2,\textbf{BDR}}^{f+y=1-mR^{-1}}+\sigma_{2,\textbf{BDR}}^{z = A\xit}\quad\textrm{on}\quad\{r=R\}. 
\eea

{\noindent{\bf Step 5.}} {Now, we apply \eqref{eq:esti:GeneralEnerIden:PDO} with $r_1=r_+(1+\dhor')$ and $r_2=R$, which yields 
\beaa
\nn&&\int_{\MM_{r_+(1+\dhor'),R}}  \Re\big(\bar{\psi}\Opw(\sigma_2({T^{a.m}_{X,E}}))\psi \big) d\Vref\\
\nn&&+\Bigg[\int_{H_r}\Re\Bigg(-\frac{1}{2}\Opw(\mu^2(\R)s_0)\pr_r\psi \ov{\pr_r\psi} + \ov{\psi}\Opw\Big( \sigma_{2, \textbf{BDR}}^{X,E}\Big)\psi\Bigg) d\tt dx^1dx^2\Bigg]_{r=r_+(1+\dhor')}^{r=R}\\
\nn&\leq& \int_{\MM_{r_+(1+\dhor'),R}}  \Re\big({T^{a.m}_{X,E}}\psi\bar{\psi} \big) d\Vref+\textbf{BDR}[\psi]\Big|_{r=r_+(1+\dhor')}^{r=R}\\
&& +C_R\Bigg\{\left(\int_{\MM_{r_+(1+\dhor'),R}}|\pr_r\psi|^2\right)^{\frac{1}{2}}\left(\int_{\MM_{r_+(1+\dhor'),R}} |\psi|^2\right)^{\frac{1}{2}}\\
\nn&&+\int_{\MM_{r_+(1+\dhor'),R}} |\psi|^2 +\left|\int_{\MM_{r_1,r_2}}  \Re\left(\bar{\psi}\Opw(\mu\widetilde{S}^{0,2}(\MM))\psi \right) d\Vref\right|\\
&&+\left(\int_{H_{r_+(1+\dhor')}}|\pr\psi|^2+|\psi|^2\right)^{\frac{1}{2}}\left(\int_{H_{r_+(1+\dhor')}}|\psi|^2\right)^{\frac{1}{2}}+\left(\int_{H_R}|\pr\psi|^2+|\psi|^2\right)^{\frac{1}{2}}\left(\int_{H_R}|\psi|^2\right)^{\frac{1}{2}}\Bigg\}.
\eeaa
Together with \eqref{eq:controloffirstorderderivativesinMMwithtrapdegbysigma2TXE}, \eqref{eq:concludingestimateforsigma2BDRXE:AA:1} and \eqref{eq:concludingestimateforsigma2BDRXE:r=R}, we infer {the existence of a constant $c>0$ such that}
\beaa
\nn&{c\Bigg[}& \int_{\MM_{r_+(1+\dhor'),R}}\frac{\mu^2|\pr_r\psi|^2}{r^2} +\int_{\MM_{r_+(1+\dhor'),10m}}\big(|\Opw(\sigma_{\trap})\psi|^2+|\Opw(x_1)\psi|^2+|\Opw(e)\psi|^2\big)\\
\nn&+&\int_{{\Mntrap_{r_+(1+\dhor'),R}}}\frac{|\pr_\tau\psi|^2+|\nab\psi|^2}{r^2}{\Bigg]}\\
\nn&+&\int_{H_R}\Re\Bigg(-\frac{1}{2}\Opw(\mu^2(\R)s_0)\pr_r\psi \ov{\pr_r\psi} + \ov{\psi}\Opw\Big(\sigma_{2,\textbf{BDR}}^{f+y=1-mR^{-1}}+\sigma_{2,\textbf{BDR}}^{z = A\xit}\Big)\psi\Bigg) d\tt dx^1dx^2\\
\nn&-&{\bigg( \int_{\MM_{r_+(1+\dhor'),R}}  \Re\big({T}^{a,m}_{X,E}\psi\bar{\psi} \big) d\Vref+\textbf{BDR}[\psi]\Big|_{r=r_+(1+\dhor')}^{r=R} \bigg)}\\
\nn&\les_R& \left(\int_{\MM_{r_+(1+\dhor'),R}}|\pr_r\psi|^2\right)^{\frac{1}{2}}\left(\int_{\MM_{r_+(1+\dhor'),R}} |\psi|^2\right)^{\frac{1}{2}}
+\int_{\MM_{r_+(1+\dhor'),R}} |\psi|^2
\\
&&+\dhor^2\int_{H_{r_+(1+\dhor')}}|\pr\psi|^2+\bigg|\int_{\MM_{r_+(1+\dhor'),R}}  \Re\left(\bar{\psi}\Opw(\mu\widetilde{S}^{0,2}(\MM))\psi \right) d\Vref\bigg|
\nn\\
&&+\left(\int_{H_{r_+(1+\dhor')}}|\pr\psi|^2+|\psi|^2\right)^{\frac{1}{2}}\left(\int_{H_{r_+(1+\dhor')}}|\psi|^2\right)^{\frac{1}{2}}+\left(\int_{H_{R}}\big(|\pr\psi|^2+|\psi|^2\big)\right)^{\frac{1}{2}}\left(\int_{H_{R}}|\psi|^2\right)^{\frac{1}{2}}.
\eeaa
Now, using again the trace estimate \eqref{eq:precisetraceestimateonr=rplus1plusdhoprimewithsqrt}, we have 
\beaa
&& \left(\int_{H_{r_+(1+\dhor')}}|\pr\psi|^2+|\psi|^2\right)^{\frac{1}{2}}\left(\int_{H_{r_+(1+\dhor')}}|\psi|^2\right)^{\frac{1}{2}}\\
&\les&  \left(\int_{H_{r_+(1+\dhor')}}|\pr\psi|^2\right)^{\frac{1}{2}}\left(\int_{\MM_{r_+(1+\dhor'),R}}\big(|\pr_r\psi|^2+|\psi|^2\big)\right)^{\frac{1}{4}}\left(\int_{\MM_{r_+(1+\dhor'),R}}|\psi|^2\right)^{\frac{1}{4}}\\
&&+\left(\int_{\MM_{r_+(1+\dhor'),R}}\big(|\pr_r\psi|^2+|\psi|^2\big)\right)^{\frac{1}{2}}\left(\int_{\MM_{r_+(1+\dhor'),R}}|\psi|^2\right)^{\frac{1}{2}},
\eeaa
which finally yields
\beaa
\nn&{c\Bigg[}& \int_{\MM_{r_+(1+\dhor'),R}}\frac{\mu^2|\pr_r\psi|^2}{r^2} +\int_{\MM_{r_+(1+\dhor'),10m}}\big(|\Opw(\sigma_{\trap})\psi|^2+|\Opw(x_1)\psi|^2+|\Opw(e)\psi|^2\big)\\
\nn&+&\int_{{\Mntrap_{r_+(1+\dhor'),R}}}\frac{|\pr_\tau\psi|^2+|\nab\psi|^2}{r^2}{\Bigg]}\\
\nn&+&\int_{H_R}\Re\Bigg(-\frac{1}{2}\Opw(\mu^2(\R)s_0)\pr_r\psi \ov{\pr_r\psi} + \ov{\psi}\Opw\Big(\sigma_{2,\textbf{BDR}}^{f+y=1-mR^{-1}}+\sigma_{2,\textbf{BDR}}^{z = A\xit}\Big)\psi\Bigg) d\tt dx^1dx^2\\
\nn&-&{\bigg( \int_{\MM_{r_+(1+\dhor'),R}}  \Re\big({T}^{a,m}_{X,E}\psi\bar{\psi} \big) d\Vref+\textbf{BDR}[\psi]\Big|_{r=r_+(1+\dhor')}^{r=R} \bigg)}\\
\nn&\les_R& \left(\int_{\MM_{r_+(1+\dhor'),R}}|\pr_r\psi|^2\right)^{\frac{1}{2}}\left(\int_{\MM_{r_+(1+\dhor'),R}} |\psi|^2\right)^{\frac{1}{2}}+\int_{\MM_{r_+(1+\dhor'),R}} |\psi|^2\\
&&+\dhor^2\int_{H_{r_+(1+\dhor')}}|\pr\psi|^2+\bigg|\int_{\MM_{r_+(1+\dhor'),R}}  \Re\left(\bar{\psi}\Opw(\mu\widetilde{S}^{0,2}(\MM))\psi \right) d\Vref\bigg|
\nn\\
&&+\left(\int_{H_{r_+(1+\dhor')}}|\pr\psi|^2\right)^{\frac{1}{2}}\left(\int_{\MM_{r_+(1+\dhor'),R}}\big(|\pr_r\psi|^2+|\psi|^2\big)\right)^{\frac{1}{4}}\left(\int_{\MM_{r_+(1+\dhor'),R}}|\psi|^2\right)^{\frac{1}{4}}\nn\\
&&+\left(\int_{H_{R}}\big(|\pr\psi|^2+|\psi|^2\big)\right)^{\frac{1}{2}}\left(\int_{H_{R}}|\psi|^2\right)^{\frac{1}{2}},
\eeaa
as stated in \eqref{eq:microlocal:currentEMF:v0}. This concludes the proof of Proposition \ref{prop:microlocalenergyMorawetzinKerronMMrp1dhorpR}.}
\end{proof}

{\begin{remark}
In view of \eqref{eq:generalformofthePDOmultipliersXandE},  \eqref{eq:definitionofQhQyQfQzintermsofTXEands0s1e0}  and \eqref{eq:finalchoicesofhyfzsymbols}, we have
\bea\lab{eq:propertiesofthesymbolEusedinTeuk}
\bsplit
E=&\Opw(e_0), \qquad e_0\in\widetilde{S}^{0,0}(\MM),\\
e_0 =&\sum_{j=1}^5\chi_j^2\left(\mu h_j+\frac{2\mu r}{r^2+a^2}y_j-\pr_r(\mu y_j)+\frac{2\mu r}{r^2+a^2}f_j -\pr_r(\mu)f_j\right), \qquad 
\end{split}
\eea
and hence, since $h_5$ and $y_5$ vanish in a neighborhood of $r_{\max}$ and $f_5(r_{\max})=0$, we infer the existence of a small enough constant $c>0$ such that 
\bea\lab{eq:definitionofthesymbolx1patXinwidetildeS10:bis}
e_2:=\sqrt{e^2-c\chi_0(r)(e_0\langle\Xi\rangle)^2}, \qquad e_2\in\widetilde{S}^{1,0}(\MM)
\eea
where $\chi_0$ is a smooth cut-off function satisfying $0\leq \chi_0\leq 1$, $\chi_0=1$ for $r\leq 10m$ and $\chi_0=0$ for $r\geq 11m$, and where we have used the property \eqref{eq:defintionofthesymboleasasquarerootofsigma2TXEmsumofsquares:1} of the symbol $e$ given by \eqref{eq:defintionofthesymboleasasquarerootofsigma2TXEmsumofsquares:1:defe}.
\end{remark}}

%%%%%%%%%%%%%%%%%%%%%%%%%%%%%%%%%%%%%%%%%%

\subsection{Conditional nondegenerate Morawetz-flux estimate in perturbations of Kerr}
\label{subsect:CondMoraFlux:Kerrperturb}

%%%%%%%%%%%%%%%%%%%%%%%%%%%%%%%%%%%%%%%%%%

{In this section, we prove a conditional nondegenerate Morawetz-flux estimate for the wave equation \eqref{eq:scalarwave} in $(\MM, \g)$ with  $\g$ satisfying the assumptions of Section \ref{subsubsect:assumps:perturbedmetric}. For convenience, we first introduce a notation for error terms in the region $\MM_{r_+(1+\dhor'), R}$ where  our microlocal energy-Morawetz are derived.}

%%%%%%%%%%%%%%%%%%%%%%%%%%%%%%%%%%%%%%%%%%

\subsubsection{{Notation for error terms in $\MM_{r_+(1+\dhor'), R}$}}

%%%%%%%%%%%%%%%%%%%%%%%%%%%%%%%%%%%%%%%%%%

{As our microlocal energy-Morawetz are derived on $\MM_{r_+(1+\dhor'), R}$, where the constants $\dhor'$ and $R$ are introduced in Remark \ref{rmk:choiceofconstantRbymeanvalue}, it is convenient to introduce the following notation $\Gac$ for error terms 
\bea\lab{eq:decaypropertiesofGac:microlocalregion}
|\dk^{\leq 2}\Gac|\les \ep\tau^{-1-\dec}\qquad\textrm{on}\,\,\MM_{r_+(1+\dhor'), R}.
\eea
In particular, we have, in view of \eqref{eq:controloflinearizedinversemetriccoefficients}, Lemma \ref{lemma:controlofmetriccoefficients:bis}, Lemma \ref{lemma:computationofthederiveativeofsrqtg} and Lemma \ref{lemma:spacetimevolumeformusingisochorecoordinates},
\bea\lab{eq:controlofperturbedmetric:microlocalregion}
\widecheck{\g}^{\a\b}=\Gac, \qquad \widecheck{\g}_{\a\b}=\Gac, \qquad N_{det}=\dk^{\leq 1}\Gac, \qquad \widecheck{f_0}=\Gac, \qquad\textrm{on}\,\,\MM_{r_+(1+\dhor'), R},
\eea
where {$\widecheck{f_0}:=f_0-|q|^2$}, which yields the following decomposition for the wave operator
\bea\lab{eq:decompositionwaveoperatorpertrub:microlocalregion}
\square_\g\psi &=& \square_{\gam}\psi+\Gac\pr^2\psi+\dk^{\leq 1}(\Gac)\pr\psi, \quad\textrm{on}\,\,\MM_{r_+(1+\dhor'), R}.
\eea}

{Also, we introduce a notation for all tangential derivatives to $H_r$ 
\bea
\lab{def:tangentialderivativeonHr:Kerrpert}
\prtan:=\pr\setminus\{\pr_r\},
\eea
which together with \eqref{eq:decompositionwaveoperatorpertrub:microlocalregion} and \eqref{eq:inverse:hypercoord} allows to decompose $\pr_r^2\psi$ as follows 
\bea\lab{eq:decompositionofpr2psiinfunctionwaveandprprtan:microlocal}
\left(\frac{\De}{|q|^2}+\Gac\right)\pr_r^2\psi = \square_\g\psi+\Big(O(1)+\Gac\Big)\prtan\pr\psi+\Big(O(1)+\dk^{\leq 1}(\Gac)\Big)\pr\psi\quad\textrm{on}\,\,\MM_{r_+(1+\dhor'), R},
\eea
where we will use the fact that, in view of $\dhor'\geq\dhor$ and $\ep\ll\dhor$,  
\bea\lab{eq:lowerboundDeltaovermodqsquare}
\frac{\De}{|q|^2}+\Gac\gtrsim \dhor -O(\ep)\gtrsim \dhor\quad\textrm{on}\,\,\MM_{r_+(1+\dhor'), R}.
\eea}

%%%%%%%%%%%%%%%%%%%%%%%%%%%%%%%%%%%%%%%%%%

\subsubsection{{Control of error terms in microlocal energy-Morawetz estimates}}

%%%%%%%%%%%%%%%%%%%%%%%%%%%%%%%%%%%%%%%%%%

{Recall that Section \ref{sec:controlforerrortermsinNRGMorawetz} provides the control of error terms arising in the derivation of standard energy-Morawetz estimates. In this section, we consider the control of error terms arising in the derivation of microlocal energy-Morawetz estimates.}

{We start with the following lemma which will be used to control error terms in the bulk.}
\begin{lemma}\label{lem:gpert:MMtrap}
{Let  $h$ be a scalar function in $\MM_{r_1, r_2}$ supported in $\tau\geq 1$} with $r_+(1+\dhor)\leq r_1<r_2<+\infty$, let $S\in \Opw(\widetilde{S}^{1,1}(\MM))$, and let $\psi$ be supported on $\{\tt\geq 1\}$. Then, {for any $\dec>0$, we have} 
{\bea
\int_{\MM_{{r_1,r_2}}} |h|\abs{S\psi}^2 \les_{r_2, \dec}\|\tt^{1+\dec}h\|_{L^\infty(\MM_{r_1, r_2})}{\EM}[\psi](\Reals).
\eea}
\end{lemma} 

\begin{proof}
Let $\chi(x)\geq 0$ be a smooth cutoff function supported on {$[\frac{1}{2},2]$} and satisfying 
\bea\lab{eq:dyadicdecompositionwithacubicpower}
\sum_{j\geq 0}\big(\chi(2^{-j}x)\big)^3=1, \quad \forall x\geq 1.
\eea
{Then, introducing for convenience the notation {$c_{h,\dec, r_2}:=\|\tt^{1+\dec}h\|_{L^\infty(\MM_{r_1, r_2})}$}, and using the fact that $h$ is supported in $\tau\geq 1$, we have
\beaa
&&\int_{{\MM_{r_1, r_2}}} |h| \abs{S\psi}^2=\sum_{j\geq 0} \int_{{\MM_{r_1, r_2}}} \chi(2^{-j}\tt) |h|  \abs{\chi(2^{-j}\tt) S\psi}^2\\
&\lesssim & c_{h,\dec, {{r_2}}}\sum_{j\geq 0}\frac{1}{2^{j(1+\dec)}}\int_{{\MM_{r_1, r_2}}} \Big(\chi(2^{-j}\tt)  \abs{S(\chi(2^{-j}\tt) \psi)}^2 +\chi(2^{-j}\tt)\abs{[S,\chi(2^{-j}\tt)] \psi}^2\Big)\\
&\lesssim & c_{h,\dec, {r_2}}\left(\sum_{j\geq 0}\frac{1}{2^{j(1+\dec)}}\int_{{\MM_{r_1, r_2}}} \chi(2^{-j}\tt)\abs{S(\chi(2^{-j}\tt) \psi)}^2 +\int_{{\MM_{r_1, r_2}}}  \abs{\pr_r^{\leq 1}\psi}^2\right)\\
&\lesssim_{{r_2}} & c_{h,\dec, {r_2}}\left(\sum_{j\geq 0}\frac{1}{2^{j(1+\dec)}}\|S(\chi(2^{-j}\tt) \psi)\|_{L^2 ({\MM_{r_1, r_2}}(2^{j-1}\leq \tt\leq 2^{j+1}))}^2  +{\M}[\psi](\Reals)\right)\\
&\lesssim_{{r_2}} & c_{h,\dec,{r_2}}\left(\sum_{j\geq 0} \frac{1}{2^{j(1+\dec)}}\|\rd^{\leq 1}(\chi(2^{-j}\tt) \psi)\|_{L^2 ({\MM_{r_1, r_2}}(2^{j-1}\leq \tt\leq 2^{j+1}))}^2  +{\M}[\psi](\Reals)\right)\\
&\lesssim_{{r_2}} & c_{h,\dec, {r_2}}\left(\sum_{j\geq 0} \frac{1}{2^{j(1+\dec)}}2^j\sup_{\tt\in [2^{j-1}, 2^{j+1}]}\E[\psi](\tt) +{\M}[\psi](\Reals)\right)\\
&\lesssim_{{r_2}} & c_{h,\dec, {r_2}}\left(\left(\sum_{j\geq 0} \frac{1}{2^{j\dec}}\right)\sup_{\tt\in\Reals}\E[\psi](\tt) +{\M}[\psi](\Reals)\right)\\
&\lesssim_{{r_2}, \dec}&  c_{h,\dec, {{r_2}}}{\EM}[\psi](\Reals),
\eeaa}
where  we have used  in the third step the fact that $[S, \chi(2^{-j}\tt)]$ is an operator in $\Opw(\widetilde{S}^{0,1}(\MM))$, and {where  we have used Lemma \ref{lemma:actionmixedsymbolsSobolevspaces:MM} in both the third and fifth steps. In view of the fact that $c_{h,\dec, {{r_2}}}=\|\tt^{1+\dec}h\|_{L^\infty({\MM_{r_1, r_2}})}$, this concludes the proof of Lemma \ref{lem:gpert:MMtrap}.}
\end{proof}

{In the next two lemmas, we provide the control of error terms arising in the microlocal energy identity of Lemma \ref{lem:EnerIden:PDO}. We start with the control of the error in the bulk.
\begin{lemma}\lab{lemma:controloferrortermsinTXEmircolocalenergybulk}
The operator ${T}_{X,E}$ introduced in \eqref{def:hatTXE:PDEnerIden} can be decomposed as 
\beaa
\bsplit
{T}_{X,E} =& {T}^{a,m}_{X,E}+\widecheck{T}_{X,E},\\
{T}_{X,E}^{a,m} =:& - \f12 \Big([f_0{\Box}_{\gam}, X]+\big(E |q|^2{\Box}_{\gam} +|q|^2{\Box}_{\gam}E\big)\Big),
\end{split}
\eeaa
where $\widecheck{T}_{X,E}$ satisfies 
\beaa
\left|\int_{\MM_{r_+(1+\dhor'),R}}  \Re\big(\widecheck{T}_{X,E}\psi\bar{\psi} \big) d\Vref\right| &\les& {\ep\int_{\Mntrap}|\square_\g\psi|^2+\ep\int_{\Mtrap}\tau^{-1-\dec}|\square_\g\psi|^2}\\
&&{+\ep\int_{\Mtrap}\left|\Opw(\widetilde{S}^{-1,0}(\MM))\square_\g\psi\right|^2}+\ep\EM[\psi](\Reals).
\eeaa
\end{lemma}}

{\begin{remark}\lab{rmk:whydowederiveapreciseestimatehere:neededTeuk}
In fact, in order to derive energy-Morawetz estimates for the scalar wave equation, the following non-sharp consequence of Lemma \ref{lemma:controloferrortermsinTXEmircolocalenergybulk} is sufficient
\beaa
\left|\int_{\MM_{r_+(1+\dhor'),R}}  \Re\big(\widecheck{T}_{X,E}\psi\bar{\psi} \big) d\Vref\right| &\les& \ep\int_{\MM}|\square_\g\psi|^2+\ep\EM[\psi](\Reals).
\eeaa
The more refined estimate in Lemma \ref{lemma:controloferrortermsinTXEmircolocalenergybulk} will be needed in the derivation of energy-Morawetz estimates for Teukolsky in perturbations of Kerr in \cite{MaSz25}.
\end{remark}}

\begin{proof}
{The proof proceeds in the following steps.}

{\noindent{\bf Step 1.} Recall from \eqref{def:hatTXE:PDEnerIden} that 
\beaa
{T}_{X,E} =- \f12 \Big([f_0{\Box}_{\g}, X]+\big(E f_0{\Box}_{\g} +f_0{\Box}_{\g}E\big)\Big)
\eeaa
so that ${T}_{X,E} = {T}^{a,m}_{X,E}+\widecheck{T}_{X,E}$ where $\widecheck{T}_{X,E}$ is given by 
\beaa
\widecheck{T}_{X,E} = - \f12 \Big([f_0{\Box}_{\g}-|q|^2{\Box}_{\gam}, X]+\big(E (f_0{\Box}_{\g}-|q|^2{\Box}_{\gam}) +(f_0{\Box}_{\g}-|q|^2{\Box}_{\gam})E\big)\Big).
\eeaa
In view of \eqref{eq:controlofperturbedmetric:microlocalregion} and \eqref{eq:decompositionwaveoperatorpertrub:microlocalregion}, we infer
\beaa
\widecheck{T}_{X,E} = - \f12 \Big([\Gac\pr^2+\dk^{\leq 1}(\Gac)\pr, X]+\big(E(\Gac\pr^2+\dk^{\leq 1}(\Gac)\pr) +(\Gac\pr^2+\dk^{\leq 1}(\Gac)\pr)E\big)\Big)
\eeaa
which we decompose as follows 
\beaa
\bsplit
\widecheck{T}_{X,E} =& \widecheck{T}_{X,E}^{(1)}+\widecheck{T}_{X,E}^{(2)}, \qquad \widecheck{T}_{X,E}^{(1)}:=  - \f12 [\Gac, X]\pr^2,\\
\widecheck{T}_{X,E}^{(2)}:=&  - \f12 \Big(\Gac[\pr^2, X]+[\dk^{\leq 1}(\Gac)\pr, X]+\big(E(\Gac\pr^2+\dk^{\leq 1}(\Gac)\pr) +(\Gac\pr^2+\dk^{\leq 1}(\Gac)\pr)E\big)\Big).
\end{split}
\eeaa}

{\noindent{\bf Step 2.} Next, we first focus on $\widecheck{T}_{X,E}^{(2)}$. Using the fact that, with respect to the measure $d\Vref$, $X$ is a skew-adjoint operator in $
\Opw(\widetilde{S}^{1,1}(\MM))$ and $E$ is a self-adjoint operator in $
\Opw(\widetilde{S}^{0,0}(\MM))$, and integrating by parts, we obtain 
\beaa
\int_{\MM_{r_+(1+\dhor'),R}}  \Re\big(\widecheck{T}_{X,E}^{(2)}\psi\bar{\psi} \big) d\Vref &=& \int_{\MM_{r_+(1+\dhor'),R}}\Re\left(\dk^{\leq 2}(\Gac)\Opw(\widetilde{S}^{1,1}(\MM))\psi \ov{\Opw(\widetilde{S}^{1,1}(\MM))\psi}\right)\\
&&+\int_{H_{r_+(1+\dhor')}}\Re\left(\dk^{\leq 1}(\Gac)\Opw(\widetilde{S}^{1,1}(\MM))\psi\ov{\Opw(\widetilde{S}^{0,0}(\MM))\psi}\right)\\
&&+\int_{H_{R}}\Re\left(\dk^{\leq 1}(\Gac)\Opw(\widetilde{S}^{1,1}(\MM))\psi\ov{\Opw(\widetilde{S}^{0,0}(\MM))\psi}\right)
\eeaa 
and hence
\beaa
&&\left|\int_{\MM_{r_+(1+\dhor'),R}}  \Re\big(\widecheck{T}_{X,E}^{(2)}\psi\bar{\psi} \big) d\Vref\right|\\
&\les&\int_{\MM_{r_+(1+\dhor'),R}}|\dk^{\leq 2}(\Gac)||\Opw(\widetilde{S}^{1,1}(\MM))\psi|^2+\int_{H_{r_+(1+\dhor')}}|\dk^{\leq 1}(\Gac)||\Opw(\widetilde{S}^{1,1}(\MM))\psi|^2\\
&&+\int_{H_{R}}|\dk^{\leq 1}(\Gac)||\Opw(\widetilde{S}^{1,1}(\MM))\psi|^2.
\eeaa 
Together with Lemma \ref{lem:gpert:MMtrap},  \eqref{eq:decaypropertiesofGac:microlocalregion} and Lemma \ref{lemma:actionmixedsymbolsSobolevspaces:MM}, we infer 
\beaa
&&\left|\int_{\MM_{r_+(1+\dhor'),R}}  \Re\big(\widecheck{T}_{X,E}^{(2)}\psi\bar{\psi} \big) d\Vref\right|\\ 
&\les_{R, \dec}& \dhor^{-1}\|\tau^{1+\dec}\dk^{\leq 2}(\Gac)\|_{L^\infty(\MM_{r_+(1+\dhor'),R})}\EM[\psi](\Reals)\\
&&+\ep\int_{H_{r_+(1+\dhor')}}|\Opw(\widetilde{S}^{1,1}(\MM))\psi|^2+\ep\int_{H_{R}}|\Opw(\widetilde{S}^{1,1}(\MM))\psi|^2\\
&\les_{R, \dec}& \dhor^{-1}\ep\EM[\psi](\Reals)+\ep\int_{H_{r_+(1+\dhor')}}|\pr^{\leq 1}\psi|^2+{\ep}\int_{H_{R}}|\pr^{\leq 1}\psi|^2.
\eeaa
Together with \eqref{de:choiceofdhor'} and \eqref{eq:choiceofRvalue:Kerr}, this implies\footnote{{Recall that $\ep\ll \dhor$ and $\ep\ll R^{-1}$, see Section \ref{sec:smallnesconstants}, and that, by convention, once $\ep$ appears on the RHS of an inequality, we may simply use the notation $\les$ and forget about the dependance on $\dhor$ and $R$.}}
\beaa
\left|\int_{\MM_{r_+(1+\dhor'),R}}  \Re\big(\widecheck{T}_{X,E}^{(2)}\psi\bar{\psi} \big) d\Vref\right| \les \ep\EM[\psi](\Reals)+\ep\int_{{\Mntrap}}|\square_\g\psi|^2,
\eeaa
and thus
\beaa
\left|\int_{\MM_{r_+(1+\dhor'),R}}  \Re\big(\widecheck{T}_{X,E}\psi\bar{\psi} \big) d\Vref\right| \les \left|\int_{\MM_{r_+(1+\dhor'),R}}  \Re\big(\widecheck{T}_{X,E}^{(1)}\psi\bar{\psi} \big) d\Vref\right|+\ep\EM[\psi](\Reals)+\ep\int_{{\Mntrap}}|\square_\g\psi|^2.
\eeaa}

{\noindent{\bf Step 3.} In view of the above, it remains to consider $\widecheck{T}_{X,E}^{(1)}$. Using \eqref{eq:decompositionofpr2psiinfunctionwaveandprprtan:microlocal} and \eqref{eq:lowerboundDeltaovermodqsquare} to decompose $\pr_r^2$, we rewrite $\widecheck{T}_{X,E}^{(1)}$ as follows 
\beaa
\bsplit
\widecheck{T}_{X,E}^{(1)}=&\widecheck{T}_{X,E}^{(1,1)}+\widecheck{T}_{X,E}^{(1,2)}\quad\textrm{on}\,\,\MM_{r_+(1+\dhor'),R},\\
\widecheck{T}_{X,E}^{(1,1)}:=&  [\Gac, X]\Big\{O(\dhor^{-1})\prtan\pr\psi\Big\},\qquad \widecheck{T}_{X,E}^{(1,2)}:= [\Gac, X]\Big\{O(\dhor^{-1})\square_\g\psi+O(\dhor^{-1})\pr\psi\Big\},
\end{split}
\eeaa
where the notation $\prtan$, introduced in \eqref{def:tangentialderivativeonHr:Kerrpert}, denotes tangential derivatives to $H_r$. We first focus on $\widecheck{T}_{X,E}^{(1,2)}$. Using the fact that, with respect to the measure $d\Vref$, $X$ is a skew-adjoint operator in $
\Opw(\widetilde{S}^{1,1}(\MM))$, we obtain 
\beaa
\bsplit
\int_{\MM_{r_+(1+\dhor'),R}}  \Re\big(\widecheck{T}_{X,E}^{(1,2)}\psi\bar{\psi} \big) d\Vref =& \int_{\MM_{r_+(1+\dhor'),R}}\Re\left(\Gac O(\dhor^{-1})\square_\g\psi\ov{\Opw(\widetilde{S}^{1,1}(\MM))\psi}\right)\\
+&\int_{\MM_{r_+(1+\dhor'),R}}\Re\left(O(\dhor^{-1})\square_\g\psi\ov{{X}(\Gac \psi)}\right)\\
+&\int_{\MM_{r_+(1+\dhor'),R}}\Re\left(O(\dhor^{-1})\Gac\Opw(\widetilde{S}^{1,1}(\MM))\psi \ov{\Opw(\widetilde{S}^{1,1}(\MM))\psi}\right)\\
+&\int_{\MM_{r_+(1+\dhor'),R}}\Re\left(O(\dhor^{-1})\Opw(\widetilde{S}^{1,1}(\MM))\psi\ov{\Opw(\widetilde{S}^{1,1}(\MM))(\Gac\psi)}\right)
\end{split}
\eeaa 
and hence, using Cauchy-Schwarz, Lemma \ref{lem:gpert:MMtrap} and \eqref{eq:decaypropertiesofGac:microlocalregion}, we deduce 
\bea\lab{eq:morepresiceNLestimatetoimprovesquaregpsiterm}
\nn&&\left|\int_{\MM_{r_+(1+\dhor'),R}}  \Re\big(\widecheck{T}_{X,E}^{(1,2)}\psi\bar{\psi} \big) d\Vref\right|\\
\nn&\les& \dhor^{-1}\|\tau^{1+\dec}\Gac\|_{L^\infty(\MM_{r_+(1+\dhor'),R})}\left(\int_{\MM_{r_+(1+\dhor'),R}}{\tau^{-1-\dec}}|\square_\g\psi|^2\right)^{\frac{1}{2}}\left(\EM[\psi](\Reals)\right)^{\frac{1}{2}}\\
\nn&&+\dhor^{-1}\|\tau^{1+\dec}\Gac\|_{L^\infty(\MM_{r_+(1+\dhor'),R})}\EM[\psi](\Reals){+\left|\int_{\MM_{r_+(1+\dhor'),R}}O(\dhor^{-1})\square_\g\psi\ov{X(\Gac \psi)}\right|}\\
\nn&\les& \ep\int_{\MM_{r_+(1+\dhor'),R}}{\tau^{-1-\dec}}|\square_\g\psi|^2+\ep\EM[\psi](\Reals){+\left|\int_{\MM_{r_+(1+\dhor'),R}}O(\dhor^{-1})\square_\g\psi\ov{X(\Gac \psi)}\right|}\\
\nn&\les& {\ep\int_{\Mntrap}|\square_\g\psi|^2+\ep\int_{\Mtrap}\tau^{-1-\dec}|\square_\g\psi|^2+\ep\EM[\psi](\Reals)}\\
&&{+\left|\int_{\MM_{r_+(1+\dhor'),R}}O(\dhor^{-1})\square_\g\psi\ov{X(\Gac \psi)}\right|.}
\eea

{Next, we focus on the last term on the RHS of \eqref{eq:morepresiceNLestimatetoimprovesquaregpsiterm}. First, we have 
\beaa
\left|\int_{\MM_{r_+(1+\dhor'),R}}O(\dhor^{-1})\square_\g\psi\ov{X(\Gac \psi)}\right| &\les& \left|\int_{\Mtrap}O(\dhor^{-1})\square_\g\psi\ov{X(\Gac \psi)}\right|\\
&&+\ep\int_{\Mntrap}|\square_\g\psi|^2+\ep\M[\psi](\Reals).
\eeaa
Next, notice, in view of \eqref{eq:definitionofthesymbolx1patXinwidetildeS10} \eqref{eq:definitionofthesymbolx1patXinwidetildeS10:bis} and the definition of $\S_1$ in \eqref{eq:computationofsymbolmodqsquareboxgam:1}, that $X$ takes the form 
\beaa
X=\Opw(\widetilde{S}^{0,0}(\MM))\mu\pr_r+\Opw(\widetilde{S}^{0,0}(\MM))\pr_\tau+\Opw(\widetilde{S}^{0,0}(\MM))\pr_{\tphi}+\Opw(\widetilde{S}^{0,0}(\MM))
\eeaa
which yields, after taking the adjoint of the operators $\Opw(\widetilde{S}^{0,0}(\MM))$ and using \eqref{eq:decaypropertiesofGac:microlocalregion}, 
\beaa
&&\left|\int_{\Mtrap}O(\dhor^{-1})\square_\g\psi\ov{X(\Gac \psi)}\right|\\ 
&\les& \ep\int_{\Mtrap}\tau^{-1-\dec}|\Opw(\widetilde{S}^{0,0}(\MM))(O(\dhor^{-1})\square_\g\psi)||\pr^{\leq 1}\psi|\\
&\les& \ep\sup_{\tau\in\Reals}\E[\psi]+\ep\int_{\Mtrap}(1+\tau)^{-1-\dec}|\Opw(\widetilde{S}^{0,0}(\MM))(O(\dhor^{-1})\square_\g\psi)|^2\\
&\les& \ep\sup_{\tau\in\Reals}\E[\psi]+\ep\int_{\Mtrap}\left|\Opw(\widetilde{S}^{0,0}(\MM))\left(O(\dhor^{-1})(1+\tau)^{-\frac{1+\dec}{2}}\square_\g\psi\right)\right|^2\\
&&+\ep\int_{\Mtrap}\left|\Opw(\widetilde{S}^{-1,0}(\MM))\square_\g\psi\right|^2\\
&\les& \ep\sup_{\tau\in\Reals}\E[\psi]+\ep\int_{\Mtrap}\tau^{-1-\dec}|\square_\g\psi|^2+\ep\int_{\Mtrap}\left|\Opw(\widetilde{S}^{-1,0}(\MM))\square_\g\psi\right|^2,
\eeaa
and hence
\beaa
&&\left|\int_{\MM_{r_+(1+\dhor'),R}}O(\dhor^{-1})\square_\g\psi\ov{X(\Gac \psi)}\right| \\
&\les& \left|\int_{\Mtrap}O(\dhor^{-1})\square_\g\psi\ov{X(\Gac \psi)}\right|+\ep\int_{\Mntrap}|\square_\g\psi|^2+\ep\M[\psi](\Reals)\\
&\les& \ep\EM[\psi](\Reals)+\ep\int_{\Mntrap}|\square_\g\psi|^2+\ep\int_{\Mtrap}\tau^{-1-\dec}|\square_\g\psi|^2\\
&&+\ep\int_{\Mtrap}\left|\Opw(\widetilde{S}^{-1,0}(\MM))\square_\g\psi\right|^2.
\eeaa
Plugging in \eqref{eq:morepresiceNLestimatetoimprovesquaregpsiterm}, we obtain 
\beaa
\left|\int_{\MM_{r_+(1+\dhor'),R}}  \Re\big(\widecheck{T}_{X,E}^{(1,2)}\psi\bar{\psi} \big) d\Vref\right| &\les& \ep\int_{\Mntrap}|\square_\g\psi|^2+\ep\int_{\Mtrap}\tau^{-1-\dec}|\square_\g\psi|^2\\
&&+\ep\int_{\Mtrap}\left|\Opw(\widetilde{S}^{-1,0}(\MM))\square_\g\psi\right|^2+\ep\EM[\psi](\Reals)
\eeaa}
which implies 
\beaa
&&\left|\int_{\MM_{r_+(1+\dhor'),R}}  \Re\big(\widecheck{T}_{X,E}\psi\bar{\psi} \big) d\Vref\right|\\ 
&\les& \left|\int_{\MM_{r_+(1+\dhor'),R}}  \Re\big(\widecheck{T}_{X,E}^{(1)}\psi\bar{\psi} \big) d\Vref\right|+\ep\EM[\psi](\Reals)
{+\ep\int_{{\Mntrap}}|\square_\g\psi|^2}\\
&\les& \left|\int_{\MM_{r_+(1+\dhor'),R}}  \Re\big(\widecheck{T}_{X,E}^{(1,1)}\psi\bar{\psi} \big) d\Vref\right| {+\ep\int_{\Mntrap}|\square_\g\psi|^2+\ep\int_{\Mtrap}\tau^{-1-\dec}|\square_\g\psi|^2}\\
&&{+\ep\int_{\Mtrap}\left|\Opw(\widetilde{S}^{-1,0}(\MM))\square_\g\psi\right|^2}+\ep\EM[\psi](\Reals).
\eeaa
}

{\noindent{\bf Step 4.} In view of the above, it remains to consider $\widecheck{T}_{X,E}^{(1,1)}$. Using the fact that, with respect to the measure $d\Vref$, $X$ is a skew-adjoint operator, and integrating by parts in $\prtan$, we have
\beaa
\bsplit
\int_{\MM_{r_+(1+\dhor'),R}}  \Re\big(\widecheck{T}_{X,E}^{(1,1)}\psi\bar{\psi} \big) d\Vref =& \int_{\MM_{r_+(1+\dhor'),R}}  \Re\big(\pr\psi\ov{\prtan(O(\dhor^{-1})[\Gac, X]\psi)} \big) d\Vref.
\end{split}
\eeaa
Then, using \eqref{eq:dyadicdecompositionwithacubicpower} and proceeding as in the proof of Lemma \ref{lem:gpert:MMtrap}, we have
\beaa
\bsplit
&\int_{\MM_{r_+(1+\dhor'),R}}  \Re\big(\widecheck{T}_{X,E}^{(1,1)}\psi\bar{\psi} \big) d\Vref \\
=& \sum_{j\geq 0}\int_{\MM_{r_+(1+\dhor'),R}}  \Re\big(\pr\psi\ov{\prtan(O(\dhor^{-1})[\chi^3(2^{-j}\tau)\Gac, X]\psi)} \big) d\Vref\\
=& \sum_{j\geq 0}\int_{\MM_{r_+(1+\dhor'),R}}  \Re\big(\chi(2^{-j}\tau)\pr\psi\ov{\prtan(O(\dhor^{-1})[\chi(2^{-j}\tau)\Gac, X](\chi(2^{-j}\tau)\psi)} \big) d\Vref\\
&+\sum_{j\geq 0}\int_{\MM_{r_+(1+\dhor'),R}}  \Re\big(\pr\psi\ov{\prtan(\chi(2^{-j}\tau))(O(\dhor^{-1})[\chi(2^{-j}\tau)\Gac, X](\chi(2^{-j}\tau)\psi)} \big) d\Vref\\
& +\sum_{j\geq 0}\int_{\MM_{r_+(1+\dhor'),R}}  \Re\big(\pr\psi\ov{\prtan(O(\dhor^{-1})\chi^2(2^{-j}\tau)\Gac[\chi(2^{-j}\tau), X]\psi)} \big) d\Vref\\
& +\sum_{j\geq 0}\int_{\MM_{r_+(1+\dhor'),R}}  \Re\big(\chi^2(2^{-j}\tau)\pr\psi\ov{\prtan(O(\dhor^{-1})[\chi(2^{-j}\tau), X](\chi(2^{-j}\tau)\Gac\psi)} \big) d\Vref\\
& +\sum_{j\geq 0}\int_{\MM_{r_+(1+\dhor'),R}}  \Re\big(\pr\psi\ov{\prtan(O(\dhor^{-1})[[\chi(2^{-j}\tau), X], \chi(2^{-j}\tau)](\chi(2^{-j}\tau)\Gac\psi)} \big) d\Vref\\
\end{split}
\eeaa
and hence, after decomposing $\pr\psi=(\pr_r\psi, \prtan\psi)$ in the last term on the RHS and  integrating by parts in $\prtan$ for the term $\prtan\psi$, 
\beaa
\bsplit
&\left|\int_{\MM_{r_+(1+\dhor'),R}}  \Re\big(\widecheck{T}_{X,E}^{(1,1)}\psi\bar{\psi} \big) d\Vref\right| \\
\les& \sum_{j\geq 0}\left(\int_{\MM_{r_+(1+\dhor'),R}}|\chi(2^{-j}\tau)\pr\psi|^2\right)^{\frac{1}{2}}\left(\int_{\MM_{r_+(1+\dhor'),R}}|\prtan(O(\dhor^{-1})[\chi(2^{-j}\tau)\Gac, X](\chi(2^{-j}\tau)\psi)|^2\right)^{\frac{1}{2}}\\
& +\dhor^{-1}\sum_{j\geq 0}\|\chi(2^{-j}\tau)\dk^{\leq 1}(\Gac)\|_{L^\infty(\MM_{r_+(1+\dhor'),R})}\Bigg(\int_{\MM_{r_+(1+\dhor'),R}}|\Opw(\widetilde{S}^{1,1}(\MM))(\chi(2^{-j}\tau)\psi)|^2\\
&+\M[\psi](\Reals)\Bigg)\\
\les_{R, \dhor}& \sum_{j\geq 0}\left(2^j\sup_{\tau\in[2^{j-1}, 2^{j+1}]}\E[\psi](\tau)\right)^{\frac{1}{2}}\|[\chi(2^{-j}\tau)\Gac, X](\chi(2^{-j}\tau)\psi)\|_{L^2_r([r_+(1+\dhor'),R],H^1(H_r))}\\
&+\ep\sum_{j\geq 0}2^{-j(1+\dec)}\Bigg(\int_{\MM_{r_+(1+\dhor'),R}}|\chi(2^{-j}\tau)\pr^{\leq 1}\psi|^2+\M[\psi](\Reals)\Bigg)\\
\les_{R, \dhor}& \sum_{j\geq 0}2^{\frac{j}{2}}\left(\sup_{\tau\in[2^{j-1}, 2^{j+1}]}\E[\psi](\tau)\right)^{\frac{1}{2}}\|[\chi(2^{-j}\tau)\Gac, X](\chi(2^{-j}\tau)\psi)\|_{L^2_r([r_+(1+\dhor'),R],H^1(H_r))}\\
&+\ep\sum_{j\geq 0}2^{-j(1+\dec)}\Bigg(2^j\sup_{\tau\in[2^{j-1}, 2^{j+1}]}\E[\psi](\tau)+\M[\psi](\Reals)\Bigg)\\
\les_{R, \dhor}& \sum_{j\geq 0}2^{\frac{j}{2}}\left(\sup_{\tau\in[2^{j-1}, 2^{j+1}]}\E[\psi](\tau)\right)^{\frac{1}{2}}\|[\chi(2^{-j}\tau)\Gac, X](\chi(2^{-j}\tau)\psi)\|_{L^2_r([r_+(1+\dhor'),R],H^1(H_r))}\\
+&\ep\EM[\psi](\Reals)
\end{split}
\eeaa
where  we have used the fact that $[X, \chi(2^{-j}\tt)]$ is an operator in $\Opw(\widetilde{S}^{0,1}(\MM))$, the fact that $[[\chi(2^{-j}\tau), X], \chi(2^{-j}\tau)]$ is an operator in $\Opw(\widetilde{S}^{-1,1}(\MM))$, and where  we have used Lemma \ref{lemma:actionmixedsymbolsSobolevspaces:MM}. Next, we rely on Lemma \ref{lemma:bascicommutatorlemmawithelementaryproof:mixedsymbols:MMcase} with $f=\chi(2^{-j}\tau)\Gac$ and $P=X$ to estimate the remaining commutator term on the RHS and obtain 
\beaa
\bsplit
&\left|\int_{\MM_{r_+(1+\dhor'),R}}  \Re\big(\widecheck{T}_{X,E}^{(1,1)}\psi\bar{\psi} \big) d\Vref\right| \\
\les_{R, \dhor}& \sum_{j\geq 0}2^{\frac{j}{2}}\left(\sup_{\tau\in[2^{j-1}, 2^{j+1}]}\E[\psi](\tau)\right)^{\frac{1}{2}}\|\chi(2^{-j}\tau)\Gac\|_{W^{2,+\infty}(\MM_{r_+(1+\dhor'),R})}\\
&\times\|\pr^{\leq 1}(\chi(2^{-j}\tau)\psi)\|_{L^2_r([r_+(1+\dhor'),R],L^2(H_r))}+\ep\EM[\psi](\Reals)\\
\les_{R, \dhor}& \ep\left(\sum_{j\geq 0}2^{\frac{j}{2}}2^{\frac{j}{2}}2^{-j(1+\dec)}\left(\sup_{\tau\in[2^{j-1}, 2^{j+1}]}\E[\psi](\tau)\right)\right)+\ep\EM[\psi](\Reals)\\
\les_{R, \dhor}& \ep\left(1+\sum_{j\geq 0}2^{-j\dec}\right)\EM[\psi](\Reals)\\
\les_{R, \dhor}&\ep\EM[\psi](\Reals)
\end{split}
\eeaa
which yields 
\beaa
&&\left|\int_{\MM_{r_+(1+\dhor'),R}}  \Re\big(\widecheck{T}_{X,E}\psi\bar{\psi} \big) d\Vref\right|\\ 
&\les& \left|\int_{\MM_{r_+(1+\dhor'),R}}  \Re\big(\widecheck{T}_{X,E}^{(1,1)}\psi\bar{\psi} \big) d\Vref\right| {+\ep\int_{\Mntrap}|\square_\g\psi|^2+\ep\int_{\Mtrap}\tau^{-1-\dec}|\square_\g\psi|^2}\\
&&{+\ep\int_{\Mtrap}\left|\Opw(\widetilde{S}^{-1,0}(\MM))\square_\g\psi\right|^2}+\ep\EM[\psi](\Reals)\\
&\les& {\ep\int_{\Mntrap}|\square_\g\psi|^2+\ep\int_{\Mtrap}\tau^{-1-\dec}|\square_\g\psi|^2+\ep\int_{\Mtrap}\left|\Opw(\widetilde{S}^{-1,0}(\MM))\square_\g\psi\right|^2}\\
&&+\ep\EM[\psi](\Reals)
\eeaa
as stated. This concludes the proof of Lemma \ref{lemma:controloferrortermsinTXEmircolocalenergybulk}.}
\end{proof}

{Finally, we provide the control of error terms arising on the boundary in the microlocal energy identity of Lemma \ref{lem:EnerIden:PDO}. 
\begin{lemma}\lab{lemma:controloferrortermsinBDRmircolocalenergyboundary}
The quantity $\textbf{BDR}[\psi]$ introduced in \eqref{def:BDR:PDEnerIden} can be decomposed as 
\beaa
\bsplit
\textbf{BDR}[\psi] =& \textbf{BDR}^{a,m}[\psi]+\widecheck{\textbf{BDR}}[\psi],\\
\textbf{BDR}^{a,m}[\psi] :=&  \frac{1}{2}\int_{H_r}\Re\Big(
\gam^{\a r}\psi \overline{\pr_{\a} (X{+E})\psi}
-\gam^{r\a}\pr_{\a}\psi \overline{(X{+E})\psi}  -\mu \Opw(s_0)\psi \overline{{\Box}_{\gam}\psi}
\Big)|q|^2 d\tt dx^1dx^2,
\end{split}
\eeaa
where $\widecheck{\textbf{BDR}}[\psi]$ satisfies
\beaa
\left|\Big(\widecheck{\textbf{BDR}}[\psi]\Big)_{r=r_+(1+\dhor')}\right|+\left|\Big(\widecheck{\textbf{BDR}}[\psi]\Big)_{r=R}\right| &\les& \ep\int_{{\Mntrap}}|\square_\g\psi|^2+\ep\M[\psi](\Reals).
\eeaa
\end{lemma}}

\begin{proof}
{Recall from \eqref{def:BDR:PDEnerIden} that 
\beaa
\textbf{BDR}[\psi] = \frac{1}{2}\int_{H_r}\Re\Big(
\g^{\a r}\psi \overline{\pr_{\a} (X{+E})\psi}
-\g^{r\a}\pr_{\a}\psi \overline{(X{+E})\psi}  -\mu \Opw(s_0)\psi \overline{{\Box}_{\g}\psi}
\Big)f_0 d\tt dx^1dx^2,
\eeaa
so that 
{\beaa
\textbf{BDR}[\psi] = \textbf{BDR}^{a,m}[\psi]+\widecheck{\textbf{BDR}}[\psi]
\eeaa}
where $\widecheck{\textbf{BDR}}[\psi]$ is given by 
\beaa
\widecheck{\textbf{BDR}}[\psi] &:=& \frac{1}{2}\int_{H_r}\Re\Big(
\widecheck{\g}^{\a r}\psi \overline{\pr_{\a} (X{+E})\psi}
-\widecheck{\g}^{r\a}\pr_{\a}\psi \overline{(X{+E})\psi}
\nn\\
&& \qquad\qquad\qquad\qquad\qquad\qquad\qquad -\mu \Opw(s_0)\psi \overline{({\Box}_{\g}-{\Box}_{\gam})\psi}
\Big)|q|^2d\tt dx^1dx^2\\
&&+\frac{1}{2}\int_{H_r}\Re\Big(
\g^{\a r}\psi \overline{\pr_{\a} (X{+E})\psi}
-\g^{r\a}\pr_{\a}\psi \overline{(X{+E})\psi}  -\mu \Opw(s_0)\psi \overline{{\Box}_{\g}\psi}
\Big)\widecheck{f_0} d\tt dx^1dx^2.
\eeaa
In view of \eqref{eq:controlofperturbedmetric:microlocalregion} and \eqref{eq:decompositionwaveoperatorpertrub:microlocalregion}, we infer
\beaa
\widecheck{\textbf{BDR}}[\psi] &=& \int_{H_r}\Re\Big(
\Gac\psi \overline{\pr(X{+E})\psi}+\Gac\pr\psi \overline{(X{+E})\psi}
\nn\\
&& \qquad\qquad\qquad\qquad\qquad\qquad\qquad -\mu \Opw(s_0)\psi \overline{(\Gac\pr^2\psi+\dk^{\leq 1}(\Gac)\pr\psi)}\Big)d\tt dx^1dx^2.
\eeaa}

{Next, we use the fact that 
\beaa
X=\Opw(i\mu s_0\xi_r)+\Opw(\widetilde{S}^{1,0}(\MM))=\mu\Opw(\widetilde{S}^{0,0}(\MM))\pr_r+\Opw(\widetilde{S}^{1,0}(\MM)),
\eeaa 
and the definition \eqref{def:tangentialderivativeonHr:Kerrpert} for $\prtan$ to rewrite $\widecheck{\textbf{BDR}}[\psi]$ as
\beaa
\widecheck{\textbf{BDR}}[\psi] &=& \int_{H_r}\Re\Big(\Gac\psi\ov{\pr_r(\mu\Opw(\widetilde{S}^{0,0}(\MM))\pr_r\psi)}+\Gac\psi\ov{\pr_r(\Opw(\widetilde{S}^{1,0}(\MM))\psi)}\\
&&+\Gac\psi\ov{\prtan(\Opw(\widetilde{S}^{1,1}(\MM)\psi)}+\Gac\pr\psi \overline{\Opw(\widetilde{S}^{1,1}(\MM))\psi}
\nn\\
&& \qquad\qquad\qquad -\mu \Opw(\widetilde{S}^{0,0}(\MM))\psi \overline{(\Gac\pr_r^2\psi+\Gac\prtan\pr\psi+\dk^{\leq 1}(\Gac)\pr\psi)}\Big) d\tt dx^1dx^2\\
&=& \int_{H_r}\Re\Big(\Gac\psi\ov{\Opw(\widetilde{S}^{0,0}(\MM))\mu\pr^2_r\psi}+\Gac\psi\ov{\Opw(\widetilde{S}^{0,0}(\MM))\pr_r\psi}\\
&&+\Gac\psi\ov{\Opw(\widetilde{S}^{1,0}(\MM))\pr_r\psi}+\Gac\psi\ov{\Opw(\widetilde{S}^{1,1}(\MM))\psi}\\
&&+\Gac\psi\ov{\prtan(\Opw(\widetilde{S}^{1,1}(\MM)\psi)}+\Gac\pr\psi \overline{\Opw(\widetilde{S}^{1,1}(\MM))\psi}
\nn\\
&& \qquad\qquad\qquad -\Opw(\widetilde{S}^{0,0}(\MM))\psi \overline{(\Gac\mu\pr_r^2\psi+\Gac\prtan\pr\psi+\dk^{\leq 1}(\Gac)\pr\psi)}\Big) d\tt dx^1dx^2.
\eeaa
We now decompose the two terms involving $\mu\pr_r^2\psi$ using \eqref{eq:decompositionofpr2psiinfunctionwaveandprprtan:microlocal} and \eqref{eq:lowerboundDeltaovermodqsquare} and obtain}\footnote{{Note that the $\mu\pr^2_r\psi$ terms in $\widecheck{\textbf{BDR}}[\psi]$ are explicitly given by
\beaa
\int_{H_r}\Big(\psi\ov{[\widecheck{\g}^{rr}, \Opw(s_0)](\mu\pr_r^2\psi)}|q|^2+\psi\ov{[\g^{rr}, \Opw(s_0)](\mu\pr_r^2\psi)}\widecheck{{f_0}}\Big)d\tau dx^1dx^2. 
\eeaa
In particular, they do not cancel and thus need to be estimated.}}
{\beaa
\widecheck{\textbf{BDR}}[\psi] &=& \int_{H_r}\Re\Big(\Gac\psi\ov{\Opw(\widetilde{S}^{0,0}(\MM))\big(\square_\g\psi+\prtan\pr\psi+\pr\psi\big)}+\Gac\psi\ov{\Opw(\widetilde{S}^{0,0}(\MM))\pr_r\psi}\\
&&+\Gac\psi\ov{\Opw(\widetilde{S}^{1,0}(\MM))\pr_r\psi}+\Gac\psi\ov{\Opw(\widetilde{S}^{1,1}(\MM))\psi}\\
&&+\Gac\psi\ov{\prtan(\Opw(\widetilde{S}^{1,1}(\MM)\psi)}+\Gac\pr\psi \overline{\Opw(\widetilde{S}^{1,1}(\MM))\psi}
\nn\\
&&  -\Opw(\widetilde{S}^{0,0}(\MM))\psi \overline{(\Gac\big(\square_\g\psi+\prtan\pr\psi+\pr\psi\big)+\Gac\prtan\pr\psi+\dk^{\leq 1}(\Gac)\pr\psi)}\Big) d\tt dx^1dx^2.
\eeaa
Integrating by parts the tangential derivatives and PDO, and using Lemma \ref{lemma:actionmixedsymbolsSobolevspaces:MM}, we deduce
\beaa
\left|\widecheck{\textbf{BDR}}[\psi]\right| &\les& \|\dk^{\leq 1}(\Gac)\|_{L^\infty(H_r)}
\int_{H_r}\big(|\Opw(\widetilde{S}^{1,1}(\MM))\psi|^2+|\square_\g\psi|^2+|\pr^{\leq 1}\psi|^2\Big)\\
&\les& \|\dk^{\leq 1}(\Gac)\|_{L^\infty(H_r)}
\int_{H_r}\big(|\square_\g\psi|^2+|\pr^{\leq 1}\psi|^2\Big).
\eeaa
In view of \eqref{de:choiceofdhor'}, \eqref{eq:choiceofRvalue:Kerr} and \eqref{eq:decaypropertiesofGac:microlocalregion}, we infer
\beaa
\left|\Big(\widecheck{\textbf{BDR}}[\psi]\Big)_{r=r_+(1+\dhor')}\right|+\left|\Big(\widecheck{\textbf{BDR}}[\psi]\Big)_{r=R}\right| &\les& \ep\int_{{\Mntrap}}|\square_\g\psi|^2+\ep\M[\psi](\Reals)
\eeaa
as stated. This concludes the proof of Lemma \ref{lemma:controloferrortermsinBDRmircolocalenergyboundary}.} 
\end{proof}

%%%%%%%%%%%%%%%%%%%%%%%%%%%%%%%%%%%%%%%%%%

\subsubsection{A conditional degenerate Morawetz-flux estimate {in perturbations of Kerr}}

%%%%%%%%%%%%%%%%%%%%%%%%%%%%%%%%%%%%%%%%%%

{We first provide a conditional degenerate Morawetz-flux estimate on $\MM_{r_+(1+\dhor'), R}$.
\begin{proposition}\lab{prop:conditionaldegenerateMorawetzflux:pertKerrrp1pdhorpR}
Assume that $\psi$, $\g$ and $F$ satisfy the same assumptions as in Theorem \ref{th:main:intermediary}, and let $\dhor'$ and $R$ be the constants introduced in Remark \ref{rmk:choiceofconstantRbymeanvalue}. Then, we have  \bea\lab{eq:finalestimateformicrolocalconditiondegenerateMorrp1pdhorpR}
\nn&& {c\Bigg[}\int_{\MM_{r_+(1+\dhor'),R}}\frac{\mu^2|\pr_r\psi|^2}{r^2} +\int_{\MM_{r_+(1+\dhor'),10m}}\big(|\Opw(\sigma_{\trap})\psi|^2+|\Opw(e)\psi|^2\big)\\
\nn&&+\int_{{\Mntrap_{r_+(1+\dhor'),R}}}\frac{|\pr_\tau\psi|^2+|\nab\psi|^2}{r^2}{\Bigg]}\\
\nn&+&\int_{H_R}\Re\Bigg(-\frac{1}{2}\Opw(\mu^2(\R)s_0)\pr_r\psi \ov{\pr_r\psi} + \ov{\psi}\Opw\Big(\sigma_{2,\textbf{BDR}}^{f+y=1-mR^{-1}}+\sigma_{2,\textbf{BDR}}^{z = A\xit}\Big)\psi\Bigg) d\tt dx^1dx^2\\
\nn&\les_R& \int_{\MM_{r_+(1+\dhor'),R}}|F||\pr_\tau\psi|   +\int_{\MM}|F|^2 +\frac{1}{\dhor^6}\int_{\MM_{r_+(1+\dhor'),R}}|\psi|^2+\ep\EM[\psi](\Reals) +\dhor\M[\psi](\Reals)\\
 &&+\left(\int_{H_{R}}\big(|\pr\psi|^2+|\psi|^2\big)\right)^{\frac{1}{2}}\left(\int_{H_{R}}|\psi|^2\right)^{\frac{1}{2}},
\eea
{where $c>0$ is a constant,} where the symbol $\sigma_{\trap}\in\widetilde{S}^{1,0}(\MM)$ is defined in  \eqref{eq:definitionofthesymbolsigmatrap}, and where the symbol $e\in\widetilde{S}^{1,0}(\MM)$ is introduced in  \eqref{eq:defintionofthesymboleasasquarerootofsigma2TXEmsumofsquares:1:defe}.
\end{proposition}}

\begin{proof}
{The proof proceeds in the following steps.}

{\noindent{\bf Step 1.} We apply Lemma \ref{lem:EnerIden:PDO} with $r_1=r_+(1+\dhor')$ and $r_2=R$ which yields
\beaa
-\int_{\MM_{r_+(1+\dhor'),R}}\Re\big(\Box_{\g}\psi\overline{(X+E)\psi}\big)=\int_{\MM_{r_+(1+\dhor'),R}}  \Re\big({T}_{X,E}\psi\bar{\psi} \big) d\Vref+\textbf{BDR}[\psi]\Big|_{r=r_+(1+\dhor')}^{r=R}.
\eeaa
Since $\psi$ satisfies \eqref{eq:scalarwave}, we infer
\beaa
\int_{\MM_{r_+(1+\dhor'),R}}  \Re\big({T}_{X,E}\psi\bar{\psi} \big) d\Vref+\textbf{BDR}[\psi]\Big|_{r=r_+(1+\dhor')}^{r=R} = -\int_{\MM_{r_+(1+\dhor'),R}}\Re\big(F\overline{(X+E)\psi}\big).
\eeaa}

{Next, recall from Lemmas \ref{lemma:controloferrortermsinTXEmircolocalenergybulk} and \ref{lemma:controloferrortermsinBDRmircolocalenergyboundary} that we have
\beaa
&&\left|\int_{\MM_{r_+(1+\dhor'),R}}  \Re\big(T_{X,E}\psi\bar{\psi} \big) d\Vref - \int_{\MM_{r_+(1+\dhor'),R}}  \Re\big(T_{X,E}^{a,m}\psi\bar{\psi} \big) d\Vref\right|\\ 
&\les& {\ep\int_{\Mntrap}|\square_\g\psi|^2+\ep\int_{\Mtrap}\tau^{-1-\dec}|\square_\g\psi|^2}\\
&&{+\ep\int_{\Mtrap}\left|\Opw(\widetilde{S}^{-1,0}(\MM))\square_\g\psi\right|^2}+\ep\EM[\psi](\Reals)
\eeaa
and  
\beaa
&&\left|\Big(\textbf{BDR}[\psi]\Big)_{r=r_+(1+\dhor')}-\Big(\textbf{BDR}^{a,m}[\psi]\Big)_{r=r_+(1+\dhor')}\right|+\left|\Big(\textbf{BDR}[\psi] - \textbf{BDR}^{a,m}[\psi]\Big)_{r=R}\right|\\ 
&\les& \ep\int_{{\Mntrap}}|\square_\g\psi|^2+\ep\M[\psi](\Reals).
\eeaa
As $\psi$ satisfies \eqref{eq:scalarwave}, we deduce from the above
{
\bea\lab{eq:intermediaryestimateformicrolocalconditiondegenerateMorrp1pdhorpR:precis}
\nn&& \int_{\MM_{r_+(1+\dhor'),R}}  \Re\big({T}_{X,E}^{a,m}\psi\bar{\psi} \big) d\Vref+\textbf{BDR}^{a,m}[\psi]\Big|_{r=r_+(1+\dhor')}^{r=R} +\int_{\MM_{r_+(1+\dhor'),R}}\Re\big(F\ov{(X+E)\psi\big)}\\
&\les& \ep\int_{\Mntrap}|F|^2+\ep\int_{\Mtrap}\tau^{-1-\dec}|F|^2+\ep\int_{\Mtrap}\left|\Opw(\widetilde{S}^{-1,0}(\MM))F\right|^2+\ep\EM[\psi](\Reals).
\eea}

\noindent{\bf Step 2.} Next, notice that the {first two terms on the LHS of \eqref{eq:intermediaryestimateformicrolocalconditiondegenerateMorrp1pdhorpR:precis} coincide (up to sign)} with the {last two terms on the LHS} of \eqref{eq:microlocal:currentEMF:v0}. Thus, we infer the following estimate from {\eqref{eq:intermediaryestimateformicrolocalconditiondegenerateMorrp1pdhorpR:precis}} and  \eqref{eq:microlocal:currentEMF:v0}
{\bea\lab{eq:intermediaryestimateformicrolocalconditiondegenerateMorrp1pdhorpR:1}
\nn&c\Bigg[& \int_{\MM_{r_+(1+\dhor'),R}}\frac{\mu^2|\pr_r\psi|^2}{r^2} +\int_{\MM_{r_+(1+\dhor'),10m}}\big(|\Opw(\sigma_{\trap})\psi|^2+|\Opw(x_1)\psi|^2+|\Opw(e)\psi|^2\big)\\
\nn&&+\int_{\Mntrap_{r_+(1+\dhor'),R}}\frac{|\pr_\tau\psi|^2+|\nab\psi|^2}{r^2}\Bigg]\\
\nn&+&\int_{H_R}\Re\Bigg(-\frac{1}{2}\Opw(\mu^2(\R)s_0)\pr_r\psi \ov{\pr_r\psi} + \ov{\psi}\Opw\Big(\sigma_{2,\textbf{BDR}}^{f+y=1-mR^{-1}}+\sigma_{2,\textbf{BDR}}^{z = A\xit}\Big)\psi\Bigg) d\tt dx^1dx^2\\
&&+\int_{\MM_{r_+(1+\dhor'),R}}\Re\big(F\ov{(X+E)\psi}\big)\nn\\
\nn&\les_R& {\ep\int_{\Mntrap}|F|^2+\ep\int_{\Mtrap}\tau^{-1-\dec}|F|^2}{+\ep\int_{\Mtrap}\left|\Opw(\widetilde{S}^{-1,0}(\MM))F\right|^2+\ep\EM[\psi](\Reals)}\nn\\
\nn&&+ \left(\int_{\MM_{r_+(1+\dhor'),R}}|\pr_r\psi|^2\right)^{\frac{1}{2}}\left(\int_{\MM_{r_+(1+\dhor'),R}} |\psi|^2\right)^{\frac{1}{2}}+\int_{\MM_{r_+(1+\dhor'),R}} |\psi|^2\\
&&+\dhor^2\int_{H_{r_+(1+\dhor')}}|\pr\psi|^2+\bigg|\int_{\MM_{r_+(1+\dhor'),R}}  \Re\left(\bar{\psi}\Opw(\mu\widetilde{S}^{0,2}(\MM))\psi \right) d\Vref\bigg|
\nn\\
&&+\left(\int_{H_{r_+(1+\dhor')}}|\pr\psi|^2\right)^{\frac{1}{2}}\left(\int_{\MM_{r_+(1+\dhor'),R}}\big(|\pr_r\psi|^2+|\psi|^2\big)\right)^{\frac{1}{4}}\left(\int_{\MM_{r_+(1+\dhor'),R}}|\psi|^2\right)^{\frac{1}{4}}\nn\\
&&+\left(\int_{H_{R}}\big(|\pr\psi|^2+|\psi|^2\big)\right)^{\frac{1}{2}}\left(\int_{H_{R}}|\psi|^2\right)^{\frac{1}{2}},
\eea}
{where $c>0$ is a constant.}

\noindent{\bf Step 3.} Next, we control the boundary terms on $H_{r_+(1+\dhor')}$ in the RHS of \eqref{eq:intermediaryestimateformicrolocalconditiondegenerateMorrp1pdhorpR:1}. In view of \eqref{de:choiceofdhor'}, using also the fact that $\psi$ satisfies \eqref{eq:scalarwave}, we have
\beaa
&&\dhor^2\int_{H_{r_+(1+\dhor')}}|\pr\psi|^2\\
&&+\left(\int_{H_{r_+(1+\dhor')}}|\pr\psi|^2\right)^{\frac{1}{2}}\left(\int_{\MM_{r_+(1+\dhor'),R}}\big(|\pr_r\psi|^2+|\psi|^2\big)\right)^{\frac{1}{4}}\left(\int_{\MM_{r_+(1+\dhor'),R}}|\psi|^2\right)^{\frac{1}{4}}\\
&\les& \dhor\int_{\MM_{r_+(1+\dhor), r_+(1+2\dhor)}}\big(|{\pr^{\leq 1}}\psi|^2+|\square_\g\psi|^2\big)\\
&&+\left(\frac{1}{\dhor}\int_{\MM_{r_+(1+\dhor), r_+(1+2\dhor)}}\big(|{\pr^{\leq 1}}\psi|^2+|\square_\g\psi|^2\big)\right)^{\frac{1}{2}}\left(\int_{\MM_{r_+(1+\dhor'),R}}\big(|\pr_r\psi|^2+|\psi|^2\big)\right)^{\frac{1}{4}}\\
&&\times\left(\int_{\MM_{r_+(1+\dhor'),R}}|\psi|^2\right)^{\frac{1}{4}}\\
&\les& \dhor\M[\psi](\Reals)+\dhor\int_{{\Mntrap}}|F|^2+\frac{1}{\dhor^2}\left(\int_{\MM_{r_+(1+\dhor'),R}}\big(|\pr_r\psi|^2+|\psi|^2\big)\right)^{\frac{1}{2}}\left(\int_{\MM_{r_+(1+\dhor'),R}}|\psi|^2\right)^{\frac{1}{2}}
\eeaa
which together with \eqref{eq:intermediaryestimateformicrolocalconditiondegenerateMorrp1pdhorpR:1} implies
{\beaa
\nn&& c\Bigg[\int_{\MM_{r_+(1+\dhor'),R}}\frac{\mu^2|\pr_r\psi|^2}{r^2} +\int_{\MM_{r_+(1+\dhor'),10m}}\big(|\Opw(\sigma_{\trap})\psi|^2+|\Opw(x_1)\psi|^2+|\Opw(e)\psi|^2\big)\\
\nn&&+\int_{\Mntrap_{r_+(1+\dhor'),R}}\frac{|\pr_\tau\psi|^2+|\nab\psi|^2}{r^2}\Bigg]\\
\nn&&+\int_{H_R}\Re\Bigg(-\frac{1}{2}\Opw(\mu^2(\R)s_0)\pr_r\psi \ov{\pr_r\psi} + \ov{\psi}\Opw\Big(\sigma_{2,\textbf{BDR}}^{f+y=1-mR^{-1}}+\sigma_{2,\textbf{BDR}}^{z = A\xit}\Big)\psi\Bigg) d\tt dx^1dx^2\\
&&+\int_{\MM_{r_+(1+\dhor'),R}}\Re\big(F\ov{(X+E)\psi}\big)\nn\\
\nn&\les_R& (\ep+\dhor)\int_{\Mntrap}|F|^2+\ep\int_{\Mtrap}\tau^{-1-\dec}|F|^2+\ep\int_{\Mtrap}\left|\Opw(\widetilde{S}^{-1,0}(\MM))F\right|^2\\
&&
+\ep\EM[\psi](\Reals)+\dhor\M[\psi](\Reals)+\int_{\MM_{r_+(1+\dhor'),R}} |\psi|^2\nn\\
&&+\bigg|\int_{\MM_{r_+(1+\dhor'),R}}  \Re\left(\bar{\psi}\Opw(\mu\widetilde{S}^{0,2}(\MM))\psi \right) d\Vref\bigg|
\nn\\
\nn&& +\frac{1}{\dhor^2}\left(\int_{\MM_{r_+(1+\dhor'),R}}\big(|\pr_r\psi|^2+|\psi|^2\big)\right)^{\frac{1}{2}}\left(\int_{\MM_{r_+(1+\dhor'),R}}|\psi|^2\right)^{\frac{1}{2}}\\
&&+\left(\int_{H_{R}}\big(|\pr\psi|^2+|\psi|^2\big)\right)^{\frac{1}{2}}\left(\int_{H_{R}}|\psi|^2\right)^{\frac{1}{2}}
\eeaa}
and hence
{
\bea\lab{eq:intermediaryestimateformicrolocalconditiondegenerateMorrp1pdhorpR:2}
\nn&& c\Bigg[\int_{\MM_{r_+(1+\dhor'),R}}\frac{\mu^2|\pr_r\psi|^2}{r^2} +\int_{\MM_{r_+(1+\dhor'),10m}}\big(|\Opw(\sigma_{\trap})\psi|^2+|\Opw(x_1)\psi|^2+|\Opw(e)\psi|^2\big)\\
\nn&&+\int_{\Mntrap_{r_+(1+\dhor'),R}}\frac{|\pr_\tau\psi|^2+|\nab\psi|^2}{r^2}\Bigg]\\
\nn&&+\int_{H_R}\Re\Bigg(-\frac{1}{2}\Opw(\mu^2(\R)s_0)\pr_r\psi \ov{\pr_r\psi} + \ov{\psi}\Opw\Big(\sigma_{2,\textbf{BDR}}^{f+y=1-mR^{-1}}+\sigma_{2,\textbf{BDR}}^{z = A\xit}\Big)\psi\Bigg) d\tt dx^1dx^2\\
&&+\int_{\MM_{r_+(1+\dhor'),R}}\Re\big(F\ov{(X+E)\psi}\big)\nn\\
\nn&\les_R& (\ep+\dhor)\int_{\Mntrap}|F|^2+\ep\int_{\Mtrap}\tau^{-1-\dec}|F|^2+\ep\int_{\Mtrap}\left|\Opw(\widetilde{S}^{-1,0}(\MM))F\right|^2\\
&&+\ep\EM[\psi](\Reals)+\dhor\M[\psi](\Reals)
+\frac{1}{\dhor^6}\int_{\MM_{r_+(1+\dhor'),R}}|\psi|^2\nn \\
\nn&&+\bigg|\int_{\MM_{r_+(1+\dhor'),R}}  \Re\left(\bar{\psi}\Opw(\mu\widetilde{S}^{0,2}(\MM))\psi \right) d\Vref\bigg|\\
&&+\left(\int_{H_{R}}\big(|\pr\psi|^2+|\psi|^2\big)\right)^{\frac{1}{2}}\left(\int_{H_{R}}|\psi|^2\right)^{\frac{1}{2}}.
\eea}

\noindent{\bf Step 4.} Next, we provide the control of the before to last term in the RHS of \eqref{eq:intermediaryestimateformicrolocalconditiondegenerateMorrp1pdhorpR:2}. We have
\beaa
&&\left|\int_{\MM_{r_+(1+\dhor'),R}}  \Re\left(\bar{\psi}\Opw(\mu\widetilde{S}^{0,2}(\MM))\psi \right) d\Vref\right|\\ 
&\les& \left|\int_{\MM_{r_+(1+\dhor'),R}}  \Re\left(\bar{\psi}\Opw(\widetilde{S}^{-2,0}(\MM))(\mu\pr_r^2\psi) \right) d\Vref\right|\\
&& + \int_{\MM_{r_+(1+\dhor'),R}}|\psi||\Opw(\widetilde{S}^{0,1}(\MM))\psi| d\Vref\\
&\les& \left|\int_{\MM_{r_+(1+\dhor'),R}}  \Re\left(\bar{\psi}\Opw(\widetilde{S}^{-2,0}(\MM))(\mu\pr_r^2\psi) \right) d\Vref\right|\\
&&+\left(\int_{\MM_{r_+(1+\dhor'),R}}|\psi|^2\right)^{\frac{1}{2}}\left(\int_{\MM_{r_+(1+\dhor'),R}}(|\pr_r\psi|^2+|\psi|^2)\right)^{\frac{1}{2}}.
\eeaa
Then, we decompose $\mu\pr_r^2$ using  \eqref{eq:decompositionofpr2psiinfunctionwaveandprprtan:microlocal} and \eqref{eq:lowerboundDeltaovermodqsquare} which yields
\beaa
\bsplit
&\left|\int_{\MM_{r_+(1+\dhor'),R}}  \Re\left(\bar{\psi}\Opw(\widetilde{S}^{-2,0}(\MM))(\mu\pr_r^2\psi) \right) d\Vref\right| \\
\les& \left|\int_{\MM_{r_+(1+\dhor'),R}}  \Re\left(\bar{\psi}\Opw(\widetilde{S}^{-2,0}(\MM))\Big(\square_\g\psi+O(1)\prtan\pr\psi+O(1)\pr\psi\Big) \right) d\Vref\right|\\
\les& \left(\int_{\MM_{r_+(1+\dhor'),R}}|\psi|^2\right)^{\frac{1}{2}}\left(\int_{\MM_{r_+(1+\dhor'),R}}(|{\Opw(\widetilde{S}^{-2,0}(\MM))\square_\g\psi}|^2+|\pr_r\psi|^2+|\psi|^2)\right)^{\frac{1}{2}}.
\end{split}
\eeaa
and hence
\beaa
&&\left|\int_{\MM_{r_+(1+\dhor'),R}}  \Re\left(\bar{\psi}\Opw(\mu\widetilde{S}^{0,2}(\MM))\psi \right) d\Vref\right|\\ 
&\les& \left(\int_{\MM_{r_+(1+\dhor'),R}}|\psi|^2\right)^{\frac{1}{2}}\left(\int_{\MM_{r_+(1+\dhor'),R}}(|{\Opw(\widetilde{S}^{-2,0}(\MM))\square_\g\psi}|^2+|\pr_r\psi|^2+|\psi|^2)\right)^{\frac{1}{2}}.
\eeaa
Plugging in \eqref{eq:intermediaryestimateformicrolocalconditiondegenerateMorrp1pdhorpR:2}, we infer
{\beaa
\nn&& c\Bigg[\int_{\MM_{r_+(1+\dhor'),R}}\frac{\mu^2|\pr_r\psi|^2}{r^2} +\int_{\MM_{r_+(1+\dhor'),10m}}\big(|\Opw(\sigma_{\trap})\psi|^2+|\Opw(x_1)\psi|^2+|\Opw(e)\psi|^2\big)\\
\nn&&+\int_{\Mntrap_{r_+(1+\dhor'),R}}\frac{|\pr_\tau\psi|^2+|\nab\psi|^2}{r^2}\Bigg]\\
\nn&&+\int_{H_R}\Re\Bigg(-\frac{1}{2}\Opw(\mu^2(\R)s_0)\pr_r\psi \ov{\pr_r\psi} + \ov{\psi}\Opw\Big(\sigma_{2,\textbf{BDR}}^{f+y=1-mR^{-1}}+\sigma_{2,\textbf{BDR}}^{z = A\xit}\Big)\psi\Bigg) d\tt dx^1dx^2\\
&&+\int_{\MM_{r_+(1+\dhor'),R}}\Re\big(F\ov{(X+E)\psi}\big)\nn\\
\nn&\les_R& (\ep+\dhor)\int_{\Mntrap}|F|^2+\ep\int_{\Mtrap}\tau^{-1-\dec}|F|^2+\ep\int_{\Mtrap}\left|\Opw(\widetilde{S}^{-1,0}(\MM))F\right|^2\\
&&+\ep\EM[\psi](\Reals)+\dhor\M[\psi](\Reals)
+\frac{1}{\dhor^6}\int_{\MM_{r_+(1+\dhor'),R}}|\psi|^2\nn \\&&+\left(\int_{\MM_{r_+(1+\dhor'),R}}|\psi|^2\right)^{\frac{1}{2}}\left(\int_{\MM_{r_+(1+\dhor'),R}}(|\Opw(\widetilde{S}^{-2,0}(\MM))\square_\g\psi|^2+|\pr_r\psi|^2+|\psi|^2)\right)^{\frac{1}{2}}\\
&&+\left(\int_{H_{R}}\big(|\pr\psi|^2+|\psi|^2\big)\right)^{\frac{1}{2}}\left(\int_{H_{R}}|\psi|^2\right)^{\frac{1}{2}}
\eeaa}
and hence 
{\bea
\lab{eq:intermediaryestimateformicrolocalconditiondegenerateMorrp1pdhorpR:robust}
\nn&& c\Bigg[\int_{\MM_{r_+(1+\dhor'),R}}\frac{\mu^2|\pr_r\psi|^2}{r^2} +\int_{\MM_{r_+(1+\dhor'),10m}}\big(|\Opw(\sigma_{\trap})\psi|^2+|\Opw(x_1)\psi|^2+|\Opw(e)\psi|^2\big)\\
\nn&&+\int_{\Mntrap_{r_+(1+\dhor'),R}}\frac{|\pr_\tau\psi|^2+|\nab\psi|^2}{r^2}\Bigg]\\
\nn&&+\int_{H_R}\Re\Bigg(-\frac{1}{2}\Opw(\mu^2(\R)s_0)\pr_r\psi \ov{\pr_r\psi} + \ov{\psi}\Opw\Big(\sigma_{2,\textbf{BDR}}^{f+y=1-mR^{-1}}+\sigma_{2,\textbf{BDR}}^{z = A\xit}\Big)\psi\Bigg) d\tt dx^1dx^2\\
&&+\int_{\MM_{r_+(1+\dhor'),R}}\Re\big(F\ov{(X+E)\psi}\big)\nn\\
\nn&\les_R& (\ep+\dhor)\int_{\Mntrap}|F|^2+\ep\int_{\Mtrap}\tau^{-1-\dec}|F|^2+(\ep+\dhor^6)\int_{\Mtrap}\left|\Opw(\widetilde{S}^{-1,0}(\MM))F\right|^2\\
&&+\ep\EM[\psi](\Reals)+\dhor\M[\psi](\Reals)
+\frac{1}{\dhor^6}\int_{\MM_{r_+(1+\dhor'),R}}|\psi|^2\nn \\
&&+\left(\int_{H_{R}}\big(|\pr\psi|^2+|\psi|^2\big)\right)^{\frac{1}{2}}\left(\int_{H_{R}}|\psi|^2\right)^{\frac{1}{2}}.
\eea}

{The precise control of $F$ on both sides of  \eqref{eq:intermediaryestimateformicrolocalconditiondegenerateMorrp1pdhorpR:robust} will be needed in the derivation of energy-Morawetz estimates for Teukolsky in \cite{MaSz25}, see also Remark \ref{rmk:whydowederiveapreciseestimatehere:neededTeuk}, while for the derivation of energy-Morawetz estimates for the scalar wave equation, we give below a  non-sharp consequence of \eqref{eq:intermediaryestimateformicrolocalconditiondegenerateMorrp1pdhorpR:robust} by further estimating the terms involving $F$ on both sides. We start by estimating the last  term on the LHS} of \eqref{eq:intermediaryestimateformicrolocalconditiondegenerateMorrp1pdhorpR:robust}. Notice from \eqref{eq:generalformofthePDOmultipliersXandE} and \eqref{eq:definitionofthesymbolx1patXinwidetildeS10} that 
\beaa
X+E=\Opw(\widetilde{S}^{0,0}(\MM))\mu\pr_r+A\pr_\tau+\Opw(x_1)+\Opw(\widetilde{S}^{0,0}(\MM)), \qquad x_1\in\widetilde{S}^{1,0}(\MM),
\eeaa
which together with Lemma \ref{lemma:actionmixedsymbolsSobolevspaces:MM} yields
\beaa
\bsplit
&{\bigg|\int_{\MM_{r_+(1+\dhor'),R}}\Re\big(F\ov{(X+E)\psi}\big)\bigg|}\\
\les&\int_{\MM_{r_+(1+\dhor'),R}}|F||(X+E)\psi| \\
\les& \int_{\MM_{r_+(1+\dhor'),R}}|F||\pr_\tau\psi| +\left(\int_{\MM}|F|^2\right)^{\frac{1}{2}}\Bigg(\int_{\MM_{r_+(1+\dhor'),R}}|\Opw(x_1)\psi|^2\\
&+|\Opw(\widetilde{S}^{0,0}(\MM))\mu\pr_r\psi|^2+|\Opw(\widetilde{S}^{0,0}(\MM))\psi|^2\Bigg)^{\frac{1}{2}}\\
\les&  \int_{\MM_{r_+(1+\dhor'),R}}|F||\pr_\tau\psi| +\left(\int_{\MM}|F|^2\right)^{\frac{1}{2}}\left(\int_{\MM_{r_+(1+\dhor'),R}}|\Opw(x_1)\psi|^2+|\mu\pr_r\psi|^2+|\psi|^2\right)^{\frac{1}{2}}.
\end{split}
\eeaa
{Plugging in \eqref{eq:intermediaryestimateformicrolocalconditiondegenerateMorrp1pdhorpR:robust}, and noticing
\beaa
(\ep+\dhor)\int_{\Mntrap}|F|^2+\ep\int_{\Mtrap}\tau^{-1-\dec}|F|^2
+(\ep+\dhor^6)\int_{\Mtrap}\left|\Opw(\widetilde{S}^{-1,0}(\MM))F\right|^2
\les\int_{\MM}|F|^2,
\eeaa}
we deduce
{\bea\lab{eq:intermediaryestimateformicrolocalconditiondegenerateMorrp1pdhorpR:xo}
\nn&& c\Bigg[\int_{\MM_{r_+(1+\dhor'),R}}\frac{\mu^2|\pr_r\psi|^2}{r^2} +\int_{\MM_{r_+(1+\dhor'),10m}}\big(|\Opw(\sigma_{\trap})\psi|^2+|\Opw(x_1)\psi|^2+|\Opw(e)\psi|^2\big)\\
\nn&&+\int_{\Mntrap_{r_+(1+\dhor'),R}}\frac{|\pr_\tau\psi|^2+|\nab\psi|^2}{r^2}\Bigg]\\
\nn&+&\int_{H_R}\Re\Bigg(-\frac{1}{2}\Opw(\mu^2(\R)s_0)\pr_r\psi \ov{\pr_r\psi} + \ov{\psi}\Opw\Big(\sigma_{2,\textbf{BDR}}^{f+y=1-mR^{-1}}+\sigma_{2,\textbf{BDR}}^{z = A\xit}\Big)\psi\Bigg) d\tt dx^1dx^2\\
\nn&\les_R&  \int_{\MM_{r_+(1+\dhor'),R}}|F||\pr_\tau\psi| +\int_{\MM}|F|^2+\ep\EM[\psi](\Reals)+\dhor\M[\psi](\Reals)
+\frac{1}{\dhor^6}\int_{\MM_{r_+(1+\dhor'),R}}|\psi|^2\nn \\
&&+\left(\int_{H_{R}}\big(|\pr\psi|^2+|\psi|^2\big)\right)^{\frac{1}{2}}\left(\int_{H_{R}}|\psi|^2\right)^{\frac{1}{2}}
\eea}
as stated in \eqref{eq:finalestimateformicrolocalconditiondegenerateMorrp1pdhorpR}. This concludes the proof of Proposition \ref{prop:conditionaldegenerateMorawetzflux:pertKerrrp1pdhorpR}.}
\end{proof}

{We are now ready to prove a conditional degenerate Morawetz-flux estimate on $\MM$.
\begin{proposition}[Conditional degenerate Morawetz-flux estimate]
\lab{prop:conditionaldegenerateMorawetzflux:pertKerrr:MM}
Assume that $\psi$, $\g$ and $F$ satisfy the same assumptions as in Theorem \ref{th:main:intermediary}, and let $\dhor'$ {be} a constant satisfying \eqref{de:choiceofdhor'}. Then, we have 
\bea\lab{eq:globalmicro:degMF:current}
\nn&& \int_{\MM_{r_+(1+\dhor'),10m}}\frac{\mu^2|\pr_r\psi|^2}{r^2} +\int_{\MM_{r_+(1+\dhor'),10m}}\big(|\Opw(\sigma_{\trap})\psi|^2+|\Opw(e)\psi|^2\big)\\
\nn&&+\MF_{r\geq 10m}[\psi](\Reals)\\
\nn&\les& \int_{\MM_{r_+(1+\dhor'),10m}}|F||\pr_\tau\psi| +\int_{\Mntrap}|F|\big(|\pr_r\psi|+r^{-1}|\psi|\big) +\bigg|\int_{\Mntrap}{F\ov{\pr_{\tau}\psi}}\bigg| +\int_{\MM}|F|^2\\
&& +\frac{1}{\dhor^6}\int_{\MM}r^{-4}|\psi|^2 +\ep\EM[\psi](\Reals) +\dhor\M[\psi](\Reals),
\eea
where the symbol $\sigma_{\trap}\in\widetilde{S}^{1,0}(\MM)$ is defined in  \eqref{eq:definitionofthesymbolsigmatrap}, and where the symbol $e\in\widetilde{S}^{1,0}(\MM)$ is introduced in  \eqref{eq:defintionofthesymboleasasquarerootofsigma2TXEmsumofsquares:1:defe}. 
\end{proposition}}

\begin{proof}
{The proof proceeds in the following steps.}

{\noindent{\bf Step 1.} Let $R$ be a constant satisfying \eqref{eq:choiceofRvalue:Kerr} that will be chosen large enough below. Given that we have already derived a conditional degenerate Morawetz-flux estimate on $\MM_{r_+(1+\dhor'), R}$ in Proposition \ref{prop:conditionaldegenerateMorawetzflux:pertKerrrp1pdhorpR}, we now focus on deriving an analog estimate in the region $\MM_{r\geq R}$. To this end, recall} from the proof of Lemma \ref{lemma:exteriorMorawetzestimates} that  with\footnote{This choice for $(X_0, w)$ coincides with the one for $(X, w)$ in Lemma \ref{lemma:exteriorMorawetzestimates} in the particular case {$\de=1$, up to a factor of $2$.}}  
$$
{X_0={2}\mu (1-mr^{-1})\bar{\pr}_r},\qquad \quad w={2}\mu r^{-1} (1-mr^{-1}),
$$
where $\bar{\pr}_r$ is the $r$-coordinate derivative in Boyer--Lindquist coordinates{, we have the following Morawetz} estimate in the region $\MM_{r\geq R}${, for $R\geq 20m$ large enough,} 
\bea
\label{eq:Morawetznearinf:Kerr:1:00}
\nn c\M_{r\geq R}[\psi](\Reals) {- \BB^{X_0}_{r=R}[\psi]}
&\leq & -\int_{\MM_{r\geq R}}\Re\big({\Box}_{\g}\psi \overline{({X_0}+w)\psi}\big)
+O(1)\F_{\II_+}[\psi](\Reals) \\
&+&O(\ep)\EM[\psi](\Reals)+C_R\bigg(\int_{H_{R}}|\pr^{\leq 1}\psi|^2\bigg)^{\frac{1}{2}}\bigg(\int_{H_{R}}|\psi|^2\bigg)^{\frac{1}{2}},
\eea
{where $c>0$, where the boundary term $\BB^{X_0}_r[\psi]$ is given by 
\beaa
\BB^{X_0}_r[\psi] :=\int_{H_r}\QQ_{\a\b}[\psi]X_0^{\b}N_r^\a,
\eeaa
with $N_r$ denoting the unit outward normal to $H_r$, and where we have controlled the boundary terms involving $w$ by $C_R(\int_{H_{R}}|\pr^{\leq 1}\psi|^2)^{\frac{1}{2}}(\int_{H_{R}}|\psi|^2)^{\frac{1}{2}}$. Next, we introduce 
\beaa
X:=X_0  +A\pr_{\tt},
\eeaa 
where $A>2$ is a large constant that has been fixed at the end of Section \ref{sect:GGTR:pseudo}. Together with \eqref{eq:Morawetznearinf:Kerr:1:00}, we infer, for some $c>0$, 
\bea
\label{eq:Morawetznearinf:Kerr:1}
\nn c\M_{r\geq R}[\psi](\Reals) +c\F_{\II_+}[\psi](\Reals) - \BB^X_{r=R}[\psi]
&\leq & -\int_{\MM_{r\geq R}}\Re\big({\Box}_{\g}\psi \overline{(X+w)\psi}\big)+O(\ep)\EM[\psi](\Reals)\\
&&+C_R\bigg(\int_{H_{R}}|\pr^{\leq 1}\psi|^2\bigg)^{\frac{1}{2}}\bigg(\int_{H_{R}}|\psi|^2\bigg)^{\frac{1}{2}},
\eea
where the extra error term generated by ${}^{(\pr_\tau)}\pi\c\QQ[\psi]$ is controlled by $O(\ep)\EM[\psi](\Reals)$ in view of Lemmas \ref{lemma:controlofdeformationtensorsforenergyMorawetz} and \ref{lemma:basiclemmaforcontrolNLterms:ter}, and where 
\bea\lab{eq:boundarytermforMorawetzregionrgeqR}
\BB^X_r[\psi] :=\int_{H_r}\QQ_{\a\b}[\psi]X^{\b}N_r^\a. 
\eea}

\noindent{\bf Step 2.} {Next, we compare the boundary term 
\bea
\BB^{X}_{r=R}[\psi]=\int_{H_R}\QQ_{\a\b}[\psi]X^{\b}N_r^\a,
\eea
 which {appears} on the LHS of \eqref{eq:Morawetznearinf:Kerr:1} and defined in \eqref{eq:boundarytermforMorawetzregionrgeqR}, with the boundary term 
\bea\lab{eq:boundarytermforMorawetzregionrleqR}
\nn\widetilde{\BB}^{X}_{r=R}[\psi] &:=&\int_{H_R}\Re\bigg(-\frac{1}{2}\Opw(\mu^2(\R)s_0)\pr_r\psi \ov{\pr_r\psi} \\
&&\qquad\qquad + \ov{\psi}\Opw\Big(\sigma_{2,\textbf{BDR}}^{f+y=1-mR^{-1}}+\sigma_{2,\textbf{BDR}}^{z = A\xit}\Big)\psi\bigg) d\tt dx^1dx^2,
\eea
which {appears} on the LHS of \eqref{eq:finalestimateformicrolocalconditiondegenerateMorrp1pdhorpR}. }

{Recall from {\eqref{eq:formulaforBDRbracketpsithatwillbeusedagainlaterforboundarytermr=R}} that 
\bea
\lab{eq:microlocalMora:boundarytermatR:diff:Kerr}
\big|\widetilde{\BB}^{X}_{r=R}[\psi] - \textbf{BDR}^X_{r=R}[\psi]\big|\les \bigg(\int_{H_{R}}|\pr^{\leq 1}\psi|^2\bigg)^{\frac{1}{2}}\bigg(\int_{H_{R}}|\psi|^2\bigg)^{\frac{1}{2}},
\eea
where ${\textbf{BDR}}^X_{r=R}[\psi]$ is defined as in \eqref{def:BDR:PDEnerIden} for $r=R$ and $\g=\gam$
\bea
\lab{def:widetildeBDR:X:Kerr}
{\textbf{BDR}}^X_{r=R}[\psi]&=& \frac{1}{2}\int_{H_R}\Re\Big(
\gam^{\a r}\psi \overline{\pr_{\a} X\psi}
-\gam^{r\a}\pr_{\a}\psi \overline{X\psi}
\nn\\
&&\qquad\qquad\qquad\qquad\qquad\qquad -\mu \Opw(s_0)\psi \overline{{\Box}_{\gam}\psi}
\Big)|q|^2 d\tt dx^1dx^2,
\eea
and where we have controlled the terms involving {$E\in\Opw(\widetilde{S}^{0,0}(\MM))$} in \eqref{def:BDR:PDEnerIden} for $r=R$ and $\g=\gam$ by $C_R(\int_{H_{R}}|\pr^{\leq 1}\psi|^2)^{\frac{1}{2}}(\int_{H_{R}}|\psi|^2)^{\frac{1}{2}}$.
In view of {the fact that $\gcheck^{\a\b}=O(\ep)$ thanks to  \eqref{eq:controloflinearizedinversemetriccoefficients}, and the fact that the outward unit normal $N_r$ and the induced metric $g_{H_r}$ on $H_r$ satisfy\footnote{Recall also from Lemma \ref{lemma:isochorecoordinates} that the coordinates $(x^1, x^2)$ are isochore so that $d\mathring{\ga}=dx^1x^2$.}
\beaa
\sqrt{|\det(g_{H_r})|}N^r &=& -\sqrt{\frac{|\det(g_{H_r})|}{\g^{rr}}}\g^{r\a}\pr_\a = -\sqrt{|\det((g_{a,m})_{H_r})|}(1+O(\ep))\Big(\gam^{r\a}\pr_\a+O(\ep)\pr\Big)\\
&=& -(1+O(\ep))|q|^2\Big(\gam^{r\a}\pr_\a+O(\ep)\pr\Big)
 \eeaa
 thanks to \eqref{eq:metricinnormalizedcoordinatesgeneralformula}, \eqref{eq:controloflinearizedinversemetriccoefficients} and \eqref{eq:controloflinearizedmetriccoefficients},} the absolute value of the difference between $\BB^X_{r=R}[\psi]$ {and its corresponding value $\BB^{X, \gam}_{r=R}[\psi]$ in Kerr} is bounded by $C_{R}\ep \int_{H_{R}} |\pr^{\leq 1}\psi|^2$, which together with \eqref{eq:microlocalMora:boundarytermatR:diff:Kerr} yields
\bea
\lab{eq:boundarytermdifference:microlocalMoraandMoranearinf}
\Big|\BB^X_{r=R}[\psi] - \widetilde{\BB}^{X}_{r=R}[\psi]\Big| &\leq &\Big|\BB^{X,\gam}_{r=R}[\psi] - {\textbf{BDR}}^X_{r=R}[\psi]\Big|  \nn\\
&+&C_{R}\ep \int_{H_{R}} |\pr^{\leq 1}\psi|^2+C_R\bigg(\int_{H_{R}}|\pr^{\leq 1}\psi|^2\bigg)^{\frac{1}{2}}\bigg(\int_{H_{R}}|\psi|^2\bigg)^{\frac{1}{2}},
\eea
where
\bea
\lab{def:boundaryrequalsR:EMtensor:larger:Kerr}
\BB^{X,\gam}_{r=R}[\psi] &=& {-\int_{H_R}\Re\bigg(\pr_{\a}\psi\ov{\pr_{\b}\psi}-\frac{1}{2}(\gam)_{\a\b}\gam^{\ga\de}\pr_{\ga}\psi \ov{\pr_{\de}\psi}\bigg)X^{\b}\gam^{r\a} |q|^2d\tau dx^1dx^2}\nn\\
&=&\int_{H_R}\Re\bigg(-\gam^{r\a}\pr_{\a}\psi \ov{X\psi}+\frac{1}{2}{X^r}\gam^{\ga\de}\pr_{\ga}\psi \ov{\pr_{\de}\psi}\bigg)|q|^2d\tau dx^1dx^2.
\eea}

{On $r=R$, the operator $X$ we choose in the region $r\leq R$ is given in view of  \eqref{eq:generalformofthePDOmultipliersXandE}, \eqref{eq:computationofsymbolmodqsquareboxgam:1}, \eqref{eq:definitionofQhQyQfQzintermsofTXEands0s1e0} and \eqref{eq:concludingestimateforsigma2BDRXE:r=R:0} by
\beaa
X &=& \Opw(i \mu s_0 \xi_r) + \Opw\bigg(\frac{is_0 \S_1}{R^2+a^2} + i s_1\bigg)\\
&=& \Opw(s_0)\mu\pr_r + \Opw\bigg(\frac{i2\left(1-\frac{m}{R}\right)\S_1}{R^2+a^2} + i A\xit\bigg)+\Opw(\widetilde{S}^{0,0}(\MM))\\
&=& 2\left(1-\frac{m}{R}\right)\mu\left(\pr_r+\left(-\mu^{-1}+\frac{m^2}{R^2}\right)\pr_\tau -\frac{a}{\De}\pr_{\tphi}\right)+A\pr_\tau +\Opw(\widetilde{S}^{0,0}(\MM))\\
&=& 2\left(1-\frac{m}{R}\right)\mu\ov{\pr}_r +A\pr_\tau +\Opw(\widetilde{S}^{0,0}(\MM))\quad\textrm{on}\,\, \{r=R\},
\eeaa
where we have also used the definition of the normalized coordinates of Lemma \ref{lem:specificchoice:normalizedcoord} in $r\geq 13m$,  \eqref{eq:decompositionofBLprrwrttocoordinatesvectorfieldnormalizedcoord} which is valid in $r\geq 13m$, and the fact that $R\geq 20m$. Thus,} the operator $X$ we choose in region $r\geq R$ coincides in its first-order part with the one of the operator $X$ we choose in region $r\leq R$ in the above discussions. Using the definition \eqref{def:tangentialderivativeonHr:Kerrpert} for $\prtan=\pr\setminus\{\pr_r\}$, we decompose 
{\beaa
X=X_1+X_2+\Opw(\widetilde{S}^{0,0}(\MM)), \quad \text{where}\quad X_1=X^r\pr_r, \quad X_2=X^{\text{tan}}\prtan,
\eeaa}
and hence {obtain} the following decomposition
\beaa
\BB^{X,\gam}_{r=R}[\psi] -  {\textbf{BDR}}^X_{r=R}[\psi] {=} \BB_1+\BB_2 {+\BB_3},
\eeaa
where
\beaa
\BB_1:=\BB^{X_1,\gam}_{r=R}[\psi] -  {\textbf{BDR}}^{X_1}_{r=R}[\psi],\qquad
\BB_2:=\BB^{X_2,\gam}_{r=R}[\psi] -{\textbf{BDR}}^{X_2}_{r=R}[\psi],
\eeaa
and where $\BB_3$ is the contribution {due} to the part in $\Opw(\widetilde{S}^{0,0}(\MM))$ of $X$ which satisfies 
\beaa
|\BB_3|\les_R\bigg(\int_{H_{R}}|\pr^{\leq 1}\psi|^2\bigg)^{\frac{1}{2}}\bigg(\int_{H_{R}}|\psi|^2\bigg)^{\frac{1}{2}}. 
\eeaa

{Next, we} compute $\BB_2$. {We have in view of \eqref{def:widetildeBDR:X:Kerr} and \eqref{def:boundaryrequalsR:EMtensor:larger:Kerr}
\beaa
\BB_2 &=& \BB^{X_2,\gam}_{r=R}[\psi] -  {\textbf{BDR}}^{X_2}_{r=R}[\psi]\\
&=& \int_{H_R}\Re\bigg(-\gam^{r\a}\pr_{\a}\psi \ov{X^{\text{tan}}\prtan\psi}\bigg)|q|^2d\tau dx^1dx^2\\
&& - \frac{1}{2}\int_{H_R}\Re\Big(
\gam^{\a r}\psi \overline{\pr_{\a} X^{\text{tan}}\prtan\psi}
-\gam^{r\a}\pr_{\a}\psi \overline{X^{\text{tan}}\prtan\psi}\Big)|q|^2 d\tt dx^1dx^2\\
&=&  -\frac{1}{2}\int_{H_R}\Re\Big(
\gam^{\a r}\psi \overline{\pr_{\a} X^{\text{tan}}\prtan\psi}
+\gam^{r\a}\pr_{\a}\psi \overline{X^{\text{tan}}\prtan\psi}\Big)|q|^2 d\tt dx^1dx^2.
\eeaa
Then, integrating by parts in $\prtan$ in the first term on the RHS, the higher order  terms, i.e., the ones that are quadratic in $\pr\psi$, cancel and we deduce}
 \beaa
 |\BB_2|\les_{R}\bigg(\int_{H_{R}}|\pr^{\leq 1}\psi|^2\bigg)^{\frac{1}{2}}\bigg(\int_{H_{R}}|\psi|^2\bigg)^{\frac{1}{2}}. 
\eeaa
 
Next, we estimate $\BB_1$. {We have in view of \eqref{def:widetildeBDR:X:Kerr} and \eqref{def:boundaryrequalsR:EMtensor:larger:Kerr}
\beaa
\BB_1 &=& \BB^{X_1,\gam}_{r=R}[\psi] -  {\textbf{BDR}}^{X_1}_{r=R}[\psi]\\
&=& \int_{H_R}\Re\bigg(-\gam^{r\a}\pr_{\a}\psi \ov{X^r\pr_r\psi}+\frac{1}{2}X^r\gam^{\ga\de}\pr_{\ga}\psi \ov{\pr_{\de}\psi}\bigg)|q|^2d\tau dx^1dx^2\\
&& -\frac{1}{2}\int_{H_R}\Re\Big(
\gam^{\a r}\psi \overline{\pr_{\a} X^r\pr_r\psi}
-\gam^{r\a}\pr_{\a}\psi \overline{X^r\pr_r\psi} -\mu \Opw(s_0)\psi \overline{{\Box}_{\gam}\psi}
\Big)|q|^2 d\tt dx^1dx^2\\
&=&  -\frac{1}{2}\int_{H_R}\Re\Big(
\gam^{\a r}\psi \overline{\pr_{\a} X^r\pr_r\psi}
+\gam^{r\a}\pr_{\a}\psi \overline{X^r\pr_r\psi} \\
&&\qquad\qquad\qquad\qquad\qquad - X^r\gam^{\ga\de}\pr_{\ga}\psi \ov{\pr_{\de}\psi} -\mu \Opw(s_0)\psi \overline{{\Box}_{\gam}\psi}
\Big)|q|^2 d\tt dx^1dx^2.
\eeaa
Also, recall that $s_0=2(1-mR^{-1})$ on $r=R$ so that $\mu\Opw(s_0)=\mu s_0=X^r$ and hence
\beaa
-2\BB_1 &=& \int_{H_R}\Re\Big(
\gam^{\a r}\psi \overline{\pr_{\a} X^r\pr_r\psi}
+\gam^{r\a}\pr_{\a}\psi \overline{X^r\pr_r\psi} \\
&&\qquad\qquad\qquad\qquad\qquad - X^r\gam^{\ga\de}\pr_{\ga}\psi \ov{\pr_{\de}\psi} -X^r\psi \overline{{\Box}_{\gam}\psi}
\Big)|q|^2 d\tt dx^1dx^2\\
&=& \int_{H_R}X^r\Re\Big(
\gam^{\a r}\psi \overline{\pr_{\a} \pr_r\psi}
+\gam^{r\a}\pr_{\a}\psi \overline{\pr_r\psi}  - \gam^{\ga\de}\pr_{\ga}\psi \ov{\pr_{\de}\psi} -\psi \overline{\gam^{\ga\de}\pr_\ga\pr_\de\psi}
\Big)|q|^2 d\tt dx^1dx^2\\
&&+O_R(1)\bigg(\int_{H_{R}}|\pr^{\leq 1}\psi|^2\bigg)^{\frac{1}{2}}\bigg(\int_{H_{R}}|\psi|^2\bigg)^{\frac{1}{2}}\\
&=& \int_{H_R}X^r\Re\Big(
\gam^{rr}\psi \overline{\pr_r\pr_r\psi}
+\gam^{rr}\pr_r\psi \overline{\pr_r\psi}  - \gam^{rr}\pr_r\psi \ov{\pr_r\psi} -\psi \overline{\gam^{rr}\pr_r^2\psi}
\Big)|q|^2 d\tt dx^1dx^2\\
&+&\int_{H_R}X^r\Re\Big(
\gam^{\tan r}\psi \overline{\pr_{\tan} \pr_r\psi}
+\gam^{r\tan}\pr_{\tan}\psi \overline{\pr_r\psi}  - 2\gam^{r\tan}\pr_{\tan}\psi\ov{\pr_r\psi}  -2\psi\overline{\gam^{r\tan}\pr_{\tan}\pr_r\psi} \\
&&- \gam^{\tan\tan}\pr_{\tan}\psi\ov{\pr_{\tan}\psi}  -\psi \overline{\gam^{\tan\tan}\pr_{\tan}\pr_{\tan}\psi}
\Big)|q|^2 d\tt dx^1dx^2\\
&&+O_R(1)\bigg(\int_{H_{R}}|\pr^{\leq 1}\psi|^2\bigg)^{\frac{1}{2}}\bigg(\int_{H_{R}}|\psi|^2\bigg)^{\frac{1}{2}}\\
&=& \int_{H_R}X^r\Re\Big(
\gam^{\tan r}\psi \overline{\pr_{\tan} \pr_r\psi}
+\gam^{r\tan}\pr_{\tan}\psi \overline{\pr_r\psi}  - 2\gam^{r\tan}\pr_{\tan}\psi\ov{\pr_r\psi}  -2\psi\overline{\gam^{r\tan}\pr_{\tan}\pr_r\psi} \\
&&- \gam^{\tan\tan}\pr_{\tan}\psi\ov{\pr_{\tan}\psi}  -\psi \overline{\gam^{\tan\tan}\pr_{\tan}\pr_{\tan}\psi}
\Big)|q|^2 d\tt dx^1dx^2\\
&&+O_R(1)\bigg(\int_{H_{R}}|\pr^{\leq 1}\psi|^2\bigg)^{\frac{1}{2}}\bigg(\int_{H_{R}}|\psi|^2\bigg)^{\frac{1}{2}}.
\eeaa
Finally, integrating by parts once in $\prtan$ in the first, the fourth and the last term of the RHS, the higher order  terms, i.e., the ones that are quadratic in $\pr\psi$, cancel and we deduce}
\beaa
|\BB_1|\les_{R}\bigg(\int_{H_{R}}|\pr^{\leq 1}\psi|^2\bigg)^{\frac{1}{2}}\bigg(\int_{H_{R}}|\psi|^2\bigg)^{\frac{1}{2}}. 
\eeaa
 
{In view of the above two estimates for $\BB_3$, $\BB_2$ and $\BB_1$, we infer
 \beaa
\Big|\BB^{X,\gam}_{r=R}[\psi] -  {\textbf{BDR}}^X_{r=R}[\psi]\Big|\les_{R} \bigg(\int_{H_{R}}|\pr^{\leq 1}\psi|^2\bigg)^{\frac{1}{2}}\bigg(\int_{H_{R}}|\psi|^2\bigg)^{\frac{1}{2}},
\eeaa
which together with \eqref{eq:boundarytermdifference:microlocalMoraandMoranearinf} implies
\bea
\label{eq:Difference:boundaryatR}
\Big|\BB^X_{r=R}[\psi] - \widetilde{\BB}^{X}_{r=R}[\psi]\Big| \les_{R} \left(\int_{H_{R}}\big(|\pr\psi|^2+|\psi|^2\big)\right)^{\frac{1}{2}}\left(\int_{H_{R}}|\psi|^2\right)^{\frac{1}{2}}+\ep\int_{H_R}|\pr^{\leq 1}\psi|^2.
\eea}

\noindent{\bf Step 3.} Now, we add \eqref{eq:Morawetznearinf:Kerr:1} and {\eqref{eq:intermediaryestimateformicrolocalconditiondegenerateMorrp1pdhorpR:robust}}, and rely on the comparison \eqref{eq:Difference:boundaryatR} of the boundary terms at $r=R$. We obtain, for a solution $\psi$ to \eqref{eq:scalarwave},
{\beaa
\nn&& \int_{\MM_{r_+(1+\dhor'),10m}}\frac{\mu^2|\pr_r\psi|^2}{r^2} +\int_{\MM_{r_+(1+\dhor'),10m}}{\big(|\Opw(\sigma_{\trap})\psi|^2+|\Opw(x_1)\psi|^2+|\Opw(e)\psi|^2\big)}\\
\nn&&+\MF_{r\geq 10m}[\psi](\Reals)\\
\nn&\les_R& \left|\int_{\MM_{r_+(1+\dhor'),R}}\Re\Big(F\ov{(X+E)\psi}\Big)+\int_{\MM_{r\geq R}}{\Re\Big(F\ov{(X+w)\psi}\Big)}\right| \\
&&+(\ep+\dhor)\int_{\Mntrap}|F|^2+\ep\int_{\Mtrap}\tau^{-1-\dec}|F|^2+(\ep+\dhor^6)\int_{\Mtrap}\left|\Opw(\widetilde{S}^{-1,0}(\MM))F\right|^2\nn \\
&&+\ep\EM[\psi](\Reals)+\dhor\M[\psi](\Reals)
+\frac{1}{\dhor^6}\int_{\MM_{r_+(1+\dhor'),R}}|\psi|^2\nn\\
&&+\left(\int_{H_{R}}\big(|\pr\psi|^2+|\psi|^2\big)\right)^{\frac{1}{2}}\left(\int_{H_{R}}|\psi|^2\right)^{\frac{1}{2}}+\ep\int_{H_R}|\pr^{\leq 1}\psi|^2.
\eeaa}
Together with the fact that $R$ satisfies \eqref{eq:choiceofRvalue:Kerr}, we infer
{\beaa
\nn&& \int_{\MM_{r_+(1+\dhor'),10m}}\frac{\mu^2|\pr_r\psi|^2}{r^2} +\int_{\MM_{r_+(1+\dhor'),10m}}{\big(|\Opw(\sigma_{\trap})\psi|^2+|\Opw(x_1)\psi|^2+|\Opw(e)\psi|^2\big)}\\
\nn&&+\MF_{r\geq 10m}[\psi](\Reals)\\
\nn&\les_R& \left|\int_{\MM_{r_+(1+\dhor'),R}}\Re\Big(F\ov{(X+E)\psi}\Big)+\int_{\MM_{r\geq R}}{\Re\Big(F\ov{(X+w)\psi}\Big)}\right| \\
&&+(\ep+\dhor)\int_{\Mntrap}|F|^2+\ep\int_{\Mtrap}\tau^{-1-\dec}|F|^2+(\ep+\dhor^6)\int_{\Mtrap}\left|\Opw(\widetilde{S}^{-1,0}(\MM))F\right|^2\nn \\
&&+\ep\EM[\psi](\Reals)+\dhor\M[\psi](\Reals)
+\frac{1}{\dhor^6}\int_{\MM_{r_+(1+\dhor'),R}}|\psi|^2\nn\\
&&+\left(\int_{\MM_{R,R+m}}\big(|\pr^{\leq 1}\psi|^2+|F|^2\big)\right)^{\frac{3}{4}}\left(\int_{\MM_{R,R+m}}|\psi|^2\right)^{\frac{1}{4}} +\ep\int_{\MM_{R,R+m}}\big(|\pr^{\leq 1}\psi|^2+|F|^2\big)\nn\\
\nn&\les_R& \left|\int_{\MM_{r_+(1+\dhor'),R}}\Re\Big(F\ov{(X+E)\psi}\Big)+\int_{\MM_{r\geq R}}{\Re\Big(F\ov{(X+w)\psi}\Big)}\right| \\
&&+(\ep+\dhor)\int_{\Mntrap}|F|^2+\ep\int_{\Mtrap}\tau^{-1-\dec}|F|^2+(\ep+\dhor^6)\int_{\Mtrap}\left|\Opw(\widetilde{S}^{-1,0}(\MM))F\right|^2\nn \\
&&+\ep\EM[\psi](\Reals)+\dhor\M[\psi](\Reals)
+\frac{1}{\dhor^6}\int_{\MM}r^{-4}|\psi|^2,
\eeaa}
where we have also used in the first step the following trace estimate 
\beaa
 \int_{H_R}|\psi|^2
 \les \left(\int_{\MM_{R, R+{m}}}\big(|\pr_r\psi|^2+|\psi|^2\big)\right)^{\frac{1}{2}}\left(\int_{\MM_{R, R+{m}}}|\psi|^2\right)^{\frac{1}{2}}.
\eeaa
{Now, since $E$ is self-adjoint, $E=\Opw(\widetilde{S}^{0,0}(\MM))$ and $w=O(r^{-1})$, we have
\beaa
&& \left|\int_{\MM_{r_+(1+\dhor'),R}}\Re\Big(F\ov{(X+E)\psi}\Big)+\int_{\MM_{r\geq R}}{\Re\Big(F\ov{(X+w)\psi}\Big)}\right|\\
&\les& \left|\int_{\MM_{r_+(1+\dhor'),+\infty}}\Re\Big(F\ov{X\psi}\Big)\right|+\int_{\MM_{r_+(1+\dhor'),R}}|EF||\psi|+\int_{\MM_{r\geq R}}r^{-1}|F||\psi|\\
&\les_R& \left|\int_{\MM_{r_+(1+\dhor'),+\infty}}\Re\Big(F\ov{X\psi}\Big)\right|+\left(\int_{\Mtrap}|EF|^2+\int_{\Mntrap}|F|^2\right)^{\frac{1}{2}}\left(\int_{\MM}r^{-4}|\psi|^2\right)^{\frac{1}{2}}\\
&&+\int_{\Mntrap}r^{-1}|F||\psi|
\eeaa
and hence
\bea
\lab{eq:globalmicro:degMF:current:robust}
\nn&& \int_{\MM_{r_+(1+\dhor'),10m}}\frac{\mu^2|\pr_r\psi|^2}{r^2} +\int_{\MM_{r_+(1+\dhor'),10m}}{\big(|\Opw(\sigma_{\trap})\psi|^2+|\Opw(x_1)\psi|^2+|\Opw(e)\psi|^2\big)}\\
\nn&&+\MF_{r\geq 10m}[\psi](\Reals)\\
\nn&\les& \left|\int_{\MM_{r_+(1+\dhor'),+\infty}}\Re\Big(F\ov{X\psi}\Big)\right|+\left(\int_{\Mtrap}|EF|^2+\int_{\Mntrap}|F|^2\right)^{\frac{1}{2}}\left(\int_{\MM}r^{-4}|\psi|^2\right)^{\frac{1}{2}}\\
\nn&&+\int_{\Mntrap}r^{-1}|F||\psi| +(\ep+\dhor)\int_{\Mntrap}|F|^2+\ep\int_{\Mtrap}\tau^{-1-\dec}|F|^2\\
\nn&&+(\ep+\dhor^6)\int_{\Mtrap}\left|\Opw(\widetilde{S}^{-1,0}(\MM))F\right|^2+\ep\EM[\psi](\Reals)+\dhor\M[\psi](\Reals)\\
&&+\frac{1}{\dhor^6}\int_{\MM}r^{-4}|\psi|^2,
\eea
where we have compressed the dependence on $R$ in $\les_{R}$ in the second step as $R\geq 20m$ has been fixed large enough (only depending on $m$) in order to derive \eqref{eq:Morawetznearinf:Kerr:1:00}.}

{The precise control of $F$ on both sides of  \eqref{eq:globalmicro:degMF:current:robust} will be needed in the derivation of energy-Morawetz estimates for Teukolsky in \cite{MaSz25},  while for the derivation of energy-Morawetz estimates for the scalar wave equation, we can further estimate the terms involving $F$ and show below a non-sharp consequence of \eqref{eq:globalmicro:degMF:current:robust}. Since we have 
\beaa
&&\left|\int_{\MM_{r_+(1+\dhor'),+\infty}}\Re\Big(F\ov{X\psi}\Big)\right|+\left(\int_{\Mtrap}|EF|^2+\int_{\Mntrap}|F|^2\right)^{\frac{1}{2}}\left(\int_{\MM}r^{-4}|\psi|^2\right)^{\frac{1}{2}}\\
\nn&&+\int_{\Mntrap}r^{-1}|F||\psi| +(\ep+\dhor)\int_{\Mntrap}|F|^2+\ep\int_{\Mtrap}\tau^{-1-\dec}|F|^2\\
\nn&&+(\ep+\dhor^6)\int_{\Mtrap}\left|\Opw(\widetilde{S}^{-1,0}(\MM))F\right|^2\\
&\les& \int_{\MM_{r_+(1+\dhor'),10m}}|F||\pr_\tau\psi| +\int_{\Mntrap}|F|\big(|\pr_r\psi|+r^{-1}|\psi|\big) +\bigg|\int_{\Mntrap}{F\ov{\pr_{\tau}\psi}}\bigg| +\int_{\MM}|F|^2\\
&& +\left(\int_{\MM}|F|^2\right)^{\frac{1}{2}}\left(\int_{\MM_{r_+(1+\dhor'),10m}}|\Opw(x_1)\psi|^2+|\mu\pr_r\psi|^2+\M_{r\geq 10m}[\psi](\Reals)\right)^{\frac{1}{2}}\\
&&+\int_{\MM}r^{-4}|\psi|^2,
\eeaa
where we used in particular the fact that $E=\Opw(\widetilde{S}^{0,0}(\MM))$ and 
\beaa
X=\Opw(\widetilde{S}^{0,0}(\MM))\mu\pr_r+A\pr_\tau+\Opw(x_1)+\Opw(\widetilde{S}^{0,0}(\MM)), \qquad x_1\in\widetilde{S}^{1,0}(\MM),
\eeaa
in view of \eqref{eq:generalformofthePDOmultipliersXandE} and \eqref{eq:definitionofthesymbolx1patXinwidetildeS10}, we infer from \eqref{eq:globalmicro:degMF:current:robust} that}
\beaa
\nn&& \int_{\MM_{r_+(1+\dhor'),10m}}\frac{\mu^2|\pr_r\psi|^2}{r^2} +\int_{\MM_{r_+(1+\dhor'),10m}}\big(|\Opw(\sigma_{\trap})\psi|^2+|\Opw(e)\psi|^2\big)\\
\nn&&+\MF_{r\geq 10m}[\psi](\Reals)\\
\nn&\les& \int_{\MM_{r_+(1+\dhor'),10m}}|F||\pr_\tau\psi| +\int_{\Mntrap}|F|\big(|\pr_r\psi|+r^{-1}|\psi|\big) +\bigg|\int_{\Mntrap}{F\ov{\pr_{\tau}\psi}}\bigg| +\int_{\MM}|F|^2\\
&& +\frac{1}{\dhor^6}\int_{\MM}r^{-4}|\psi|^2 +\ep\EM[\psi](\Reals) +\dhor\M[\psi](\Reals)
\eeaa
as stated in \eqref{eq:globalmicro:degMF:current}. This concludes the proof of Proposition \ref{prop:conditionaldegenerateMorawetzflux:pertKerrr:MM}.
\end{proof}

%%%%%%%%%%%%%%%%%%%%%%%%%%%%%%%%%%%%%%

\subsubsection{A conditional nondegenerate Morawetz-flux estimate}
\lab{subsubsect:nondegenerateMFesti:Kerr}

%%%%%%%%%%%%%%%%%%%%%%%%%%%%%%%%%%%%%%
 
Next, we upgrade the {conditional degenerate Morawetz-flux estimate of Proposition \ref{prop:conditionaldegenerateMorawetzflux:pertKerrr:MM} to a conditional} nondegenerate Morawetz-flux estimate by {making use of the} redshift estimate. 
\begin{proposition}[Conditional nondegenerate Morawetz-flux estimate]
\label{prop:nondeg:Morawetz}
Assuming that $\psi$, $\g$ and $F$ satisfy the same assumptions as in Theorem \ref{th:main:intermediary},  we have the following conditional nondegenerate Morawetz-flux estimate
{\bea\lab{prop:eq:nondeg:Morawetz} 
\nn&& \sup_{\tau\in\mathbb{R}}\E_{r\leq r_+(1+\dred)}[\psi](\tau)+{\widetilde{\MF}[\psi]}\\
\nn&\les& \int_{\Mtrap}|F||\pr_\tau\psi| +\int_{\Mntrap}|F|\big(|\pr_r\psi|+r^{-1}|\psi|\big) +\bigg|\int_{\Mntrap}{F\ov{\pr_{\tau}\psi}}\bigg| +\int_{\MM}|F|^2\\
&& +\int_{\MM}r^{-4}|\psi|^2 +\ep\sup_{\tau\in\mathbb{R}}\E[\psi](\tau),
\eea
where the symbol $e\in\widetilde{S}^{1,0}(\MM)$ is introduced in  \eqref{eq:defintionofthesymboleasasquarerootofsigma2TXEmsumofsquares:1:defe}.}
\end{proposition} 
 
\begin{proof}
{We make use of the redshift estimate of Lemma \ref{lemma:redshiftestimates} with $s=0$, $\tau_1=-\infty$ and $\tau_2=+\infty$, and notice that $\E^{(s)}[\psi](-\infty)=0$ since $\psi=0$ for $\tau\leq 1$. We obtain 
\beaa
\EMF_{r\leq r_+(1+\dred)}[\psi](\Reals)\les \dred^{-1}\M_{r_+(1+\dred), r_+(1+2\dred)}[\psi](\Reals)+\int_{{\MM_{r\leq r_+(1+2\dred)}}}|F|^2.
\eeaa
Now, we choose $\dred$ such that  
\beaa
r_+(1+3\dred)\leq {\min_{\Xi\in\GG_5} r_{\text{trap}}}\quad \Longleftrightarrow \quad\dred\leq\frac{{\min_{\Xi\in\GG_5} r_{\text{trap}}} -r_+}{3r_+},
\eeaa
which implies, in view of the definition \eqref{eq:definitionofthesymbolsigmatrap} of $\sigma_{\trap}$, which depends on $r_{\trap}$ introduced in  \eqref{eq:definitionofrtrapinfunctionofrmaxandcutoffinmathcalG5} and satisfying \eqref{eq:UpperBoundforrmaxtrapOfthePotential:consequencertrap}, 
\beaa
&&\M_{r_+(1+\dred), r_+(1+2\dred)}[\psi](\Reals)\\
&\les&  \frac{1}{\dred^2}\left(\int_{\MM_{r_+(1+\dred), r_+(1+2\dred)}}\frac{\mu^2|\pr_r\psi|^2}{r^2}+\int_{\MM_{r_+(1+\dred), r_+(1+2\dred)}}|\Opw(\sigma_{\trap})\psi|^2\right).
\eeaa
In view of the above, we infer
\bea\lab{eq:applicationofredshiftestimatewiths=0tau1minftyandtau2pinfty}
\nn&&\EMF_{r\leq r_+(1+\dred)}[\psi](\Reals)\\
\nn&\les& \frac{1}{\dred^3}\left(\int_{\MM_{r_+(1+\dred), r_+(1+2\dred)}}\frac{\mu^2|\pr_r\psi|^2}{r^2}+\int_{\MM_{r_+(1+\dred), r_+(1+2\dred)}}|\Opw(\sigma_{\trap})\psi|^2\right)\\
&&+\int_{\MM_{r\leq r_+(1+2\dred)}(\Reals)}|F|^2.
\eea}

{Next, we multiply \eqref{eq:applicationofredshiftestimatewiths=0tau1minftyandtau2pinfty} by $\dred^4$ and sum it with \eqref{eq:globalmicro:degMF:current} which yields
\beaa
\nn&& \dred^4\EMF_{r\leq r_+(1+\dred)}[\psi](\Reals)+\int_{\MM_{r_+(1+\dhor'),10m}}\frac{\mu^2|\pr_r\psi|^2}{r^2} \\
&&+\int_{\MM_{r_+(1+\dhor'),10m}}\big(|\Opw(\sigma_{\trap})\psi|^2+|\Opw(e)\psi|^2\big)+\MF_{r\geq 10m}[\psi](\Reals)\\
\nn&\les& \int_{\Mtrap}|F||\pr_\tau\psi| +\int_{\Mntrap}|F|\big(|\pr_r\psi|+r^{-1}|\psi|\big) +\left|\int_{\Mntrap}{F\ov{\pr_{\tau}\psi}}\right| +\int_{\MM}|F|^2\\
&& +\frac{1}{\dhor^6}\int_{\MM}r^{-4}|\psi|^2 +\ep\EM[\psi](\Reals) +\dhor\M[\psi](\Reals)\\
&&  +\dred\left(\int_{\MM_{r_+(1+\dred), r_+(1+2\dred)}}\frac{\mu^2|\pr_r\psi|^2}{r^2}+\int_{\MM_{r_+(1+\dred), r_+(1+2\dred)}}|\Opw(\sigma_{\trap})\psi|^2\right).
\eeaa
We now choose $\dhor\ll \dred^4$ and $\dred$ sufficiently small which implies, using also $\ep$ small enough and the definition \eqref{eq:definitionofmicrolocalMorawetznormwidetildeM} of the microlocal Morawetz norm $\widetilde{\M}[\psi]$, 
\beaa
\nn&& \dred^4\sup_{\tau\in\mathbb{R}}\E_{r\leq r_+(1+\dred)}[\psi](\tau)+{\dred^4\widetilde{\MF}[\psi]}\\
\nn&\les& \int_{\Mtrap}|F||\pr_\tau\psi| +\int_{\Mntrap}|F|\big(|\pr_r\psi|+r^{-1}|\psi|\big) +\left|\int_{\Mntrap}{F\ov{\pr_{\tau}\psi}}\right| +\int_{\MM}|F|^2\\
&& +\frac{1}{\dhor^6}\int_{\MM}r^{-4}|\psi|^2 +\ep\sup_{\tau\in\mathbb{R}}\E[\psi](\tau).
\eeaa
As $\dhor$ and $\dred$ have been fixed small enough (only depending on $m-|a|$), we may now compress the dependence in $\dred$ and $\dhor$ in $\les$ and obtain 
\beaa
\nn&& \sup_{\tau\in\mathbb{R}}\E_{r\leq r_+(1+\dred)}[\psi](\tau)+{\widetilde{\MF}[\psi]}\\
\nn&\les& \int_{\Mtrap}|F||\pr_\tau\psi| +\int_{\Mntrap}|F|\big(|\pr_r\psi|+r^{-1}|\psi|\big) +\left|\int_{\Mntrap}{F\ov{\pr_{\tau}\psi}}\right| +\int_{\MM}|F|^2\\
&& +\int_{\MM}r^{-4}|\psi|^2 +\ep\sup_{\tau\in\mathbb{R}}\E[\psi](\tau),
\eeaa
as stated in \eqref{prop:eq:nondeg:Morawetz}. This concludes the proof of Proposition \ref{prop:nondeg:Morawetz}.}
\end{proof}

%%%%%%%%%%%%%%%%%%%%%%%%%%%%%%%%%%%%%%%%%%

\subsection{{End of the proof of Theorem \ref{th:mainenergymorawetzmicrolocal}}}
\label{subsect:proofoftheorem64}

%%%%%%%%%%%%%%%%%%%%%%%%%%%%%%%%%%%%%%%%%%

In this section, we derive an energy estimate for the wave equation in perturbations of Kerr, and based on this energy estimate and the {conditional nondegenerate Morawetz-flux estimate of Proposition \ref{prop:nondeg:Morawetz}}, we conclude the proof of {Theorem \ref{th:mainenergymorawetzmicrolocal}}.

%%%%%%%%%%%%%%%%%%%%%%%%%%%%%%%%%%%%%%%%

\subsubsection{{Conditional energy estimate}}
\label{subsect:CondEnerMoraFlux:Kerrpert}

%%%%%%%%%%%%%%%%%%%%%%%%%%%%%%%%%%%%%%%%

We now derive a conditional energy estimate. 
\begin{proposition}[Conditional energy estimate]
\lab{thm:nondeg:EnerandMora:Kerrandpert}
Assuming that $\psi$, $\g$ and $F$ satisfy the same assumptions as in Theorem \ref{th:main:intermediary},  we have the following conditional energy estimate 
\bea\lab{thm:eq:nondeg:EnerandMora:Kerrandpert}
\nn\sup_{\tau\in\Reals}\E[\psi](\tau) &\les&  {\widetilde{\M}[\psi]}+\int_{\MM}|F|^2+{\sum_{i=1}^{\iota}}\int_{\Mtrap}|F||\Opw(\Theta_i)V_i\Opw(\Theta_i)\psi|\\
&&+\sup_{\tau\in\Reals}\bigg|\int_{\Mntrap(-\infty, \tau)}{\Re\Big(F\ov{\pr_{\tau}\psi}\Big)}\bigg|,
\eea
where $\iota$ is a large enough integer{,} where $\Theta_i=\Theta_i(\Xi)\in\widetilde{S}^{0,0}(\MM)$, $i=1,\cdots,\iota$, are defined in \eqref{eq:definitionofpartitionofunityThetajj=1toiota}{, and where $V_i$, $i= 1,\cdots,\iota$, are timelike vectorfields in $\Mtrap$ introduced in \eqref{eq:defintionofthedifferentialoperatorsViformicrolocalNRGestimates}.} Moreover, we have the following alternative conditional energy estimate 
 \bea
\lab{eq:energyestimate:awayhorizon:Kerrandpert:lastsect}
\nn\sup_{\tau\in\Reals}\E[\psi](\tau) &\les&  {\widetilde{\M}[\psi]}+\int_{\MM}|F|^2+\sup_{\tau\in\Reals}\bigg|\int_{\Mntrap(-\infty, \tau)}{\Re\Big(F\ov{\pr_{\tau}\psi}\Big)}\bigg|\\
&& +\bigg(\min\bigg(\int_{\Mtrap}\tau^{1+\dec}|F|^2, \int_{\Mtrap}|\pr F|^2\bigg)\bigg)^{\frac{1}{2}}\Big(\EM[\psi](\Reals)\Big)^{\frac{1}{2}}.
\eea
\end{proposition}

\begin{proof}
The proof proceeds in the following steps. 

{\noindent{\bf Step 1.}} Recall that the symbol $r_{\trap}\in\widetilde{S}^{0,0}(\MM)$ defined in  \eqref{eq:definitionofrtrapinfunctionofrmaxandcutoffinmathcalG5} satisfies in view of \eqref{eq:UpperBoundforrmaxtrapOfthePotential:consequencertrap} $r_{\trap}=r_{\trap}(\Xi)\in (r_+, 8m]$, where $\Xi=(\xit, \xiphi, \Lambda)$. Let $\iota\in \mathbb{N}$ be a {constant to be fixed large enough below}, and define for $1\leq i\leq \iota$
\begin{align}
I_i:=&[r_{\text{min}, \trap}, r_{\text{max}, \trap}]\cap\nn\\
&\left(r_{\text{min}, \trap}+\frac{i-\frac{3}{2}}{\iota}(r_{\text{max}, \trap}-r_{\text{min}, \trap}),r_{\text{min}, \trap}+\frac{i+\frac{1}{2}}{\iota}(r_{\text{max}, \trap}-r_{\text{min}, \trap})\right). 
\end{align}
Then, recalling the frequency set $\GG_5$, introduced in Step 1 of the proof of Proposition \ref{prop:microlocalenergyMorawetzinKerronMMrp1dhorpR}, as well as the cut-off $\widetilde{\chi}_5=\widetilde{\chi}_5(\Xi)$ compactly supported in $\GG_5$ and involved in the definition  \eqref{eq:definitionofrtrapinfunctionofrmaxandcutoffinmathcalG5} of $r_{\trap}$,  we have 
\beaa
\textrm{supp}(\widetilde{\chi}_5)\subset\bigcup_{i=1}^{\iota}\NN_i, \qquad \NN_i:=r_{\trap}^{-1} (I_{i})\cap\GG_5, \quad i=1,\cdots, \iota.
\eeaa
Since $\textrm{supp}(\widetilde{\chi}_5)$ is compact, and hence closed, we obtain an open cover of the set of frequencies $\GG_\Xi$ as follows 
\beaa
\GG_\Xi=\bigcup_{i=0}^{\iota}\NN_i, \qquad \NN_0:=\GG_\Xi\setminus\textrm{supp}(\widetilde{\chi}_5). 
\eeaa
Thus, there exist {real valued symbols $\{\Theta_i\}_{i=-1}^{\iota}$,} {$\Theta_i=\Theta_i(\Xi)\in\widetilde{S}^{0,0}(\MM)$,} such that
\bea\lab{eq:definitionofpartitionofunityThetajj=0toiota:propsumto1}
\sum_{i=0}^\iota\Theta_i=1\quad\text{on}\quad\GG_\Xi\cap\{|\Xi|\geq 2\}, \qquad {\Theta_{-1}:=1-\sum_{i=0}^\iota\Theta_i,}
\eea
 and satisfying in addition
\bsub
\bea\lab{eq:definitionofpartitionofunityThetajj=1toiota}
\text{supp}(\Theta_i)\Subset\NN_i, \quad \sum_{i=1}^{\iota}\Theta_i(\Xi)=1 \quad\text{on}\quad \textrm{supp}(\widetilde{\chi}_5)\cap\{|\Xi|\geq 2\}, \quad \Theta_i=0\quad\textrm{on}\quad |\Xi|\leq 1,
\eea
\bea\lab{eq:definitionofpartitionofunityThetajj=1toiota:Theta0}
\text{supp}(\Theta_0)\Subset\NN_0=\GG_\Xi\setminus\textrm{supp}(\widetilde{\chi}_5)\,\,\Rightarrow \,\text{supp}(\Theta_0)\cap\textrm{supp}(\widetilde{\chi}_5)=\emptyset, \quad \Theta_0=0\,\,\,\textrm{on}\,\, |\Xi|\leq 1,
\eea
{and
\bea\lab{eq:definitionofpartitionofunityThetajj=1toiota:Theta-1}
\textrm{supp}(\Theta_{-1})\subset\{\Xi\leq 2\}\,\,\Rightarrow \,\Theta_{-1}\in\widetilde{S}^{-\infty,0}(\MM).
\eea}
\esub

{\noindent{\bf Step 2.} To derive energy estimates, we will consider two separate regions, $\MM_{r\leq R_1}$ and $\MM_{r\geq R_1}$, where $R_1\in[11m, 12m]$ is chosen such that
\bea
\lab{choiceofR1:energyestimate:kerrpert}
\int_{H_{R_1}}|\pr^{\leq 1}\psi|^2 \leq \frac{1}{m}\int_{\MM_{11m, 12m}}|\pr^{\leq 1}\psi|^2\les \M[\psi](\Reals).
\eea
We first focus on the region $\MM_{r\leq R_1}$ where we define 
\bea
\lab{def:psiiandFi:kerrpert}
\psi_i:= \Opw(\Theta_i)\psi, \qquad {F}_{i}:= \square_\g\psi_i, \qquad i={-1}, 0,1,\ldots, \iota. 
\eea
Then, since $\psi$ satisfies \eqref{eq:scalarwave}, we may rewrite $F_i$ as
\beaa
F_i &=& \square_\g\psi_i=\frac{1}{|q|^2}|q|^2\square_\g(\Opw(\Theta_i)\psi)\\
&=& \frac{1}{|q|^2}\Opw(\Theta_i)(|q|^2F)+\frac{1}{|q|^2}\big[|q|^2\square_\g, \Opw(\Theta_i)\big]\psi\\
&=& {\frac{1}{|q|^2}\Opw(\Theta_i)(|q|^2F)+\frac{1}{|q|^2}\big[|q|^2\square_{\gam}, \Opw(\Theta_i)\big]\psi+\frac{1}{|q|^2}\big[|q|^2(\square_\g-\square_{\gam}), \Opw(\Theta_i)\big]\psi}.
\eeaa
Now, using \eqref{eq:computationofsymbolmodqsquareboxgam}, Proposition \ref{prop:PDO:MM:Weylquan:mixedoperators} and Lemma \ref{lemma:simplepropertiesspecialcaseofsymbols:recoveringoperators}, we have
\beaa
\big[|q|^2\square_{\gam}, \Opw(\Theta_i)\big] &=& \big[\Opw(-\Delta\xi_r^2 -2\S_1\xi_r+\S_2 +\widetilde{S}^{0,0}(\MM)), \Opw(\Theta_i)\big]\\
&=& \Opw(\widetilde{S}^{-1,1}(\MM)),
\eeaa
and hence
\beaa
F_i = {\frac{1}{|q|^2}\Opw(\Theta_i)(|q|^2F)}+\Opw(\widetilde{S}^{-1,1}(\MM))\psi+\frac{1}{|q|^2}\big[|q|^2(\square_\g-\square_{\gam}), \Opw(\Theta_i)\big]\psi.
\eeaa
Also, recalling the definition $\Gac$ introduced in \eqref{eq:decaypropertiesofGac:microlocalregion}, which we now use in the region $\MM_{r\leq R_1}$, i.e., 
\bea\lab{eq:decaypropertiesofGac:microlocalregion:energy}
|\dk^{\leq 2}\Gac|\les \ep\tau^{-1-\dec}\qquad\textrm{on}\,\,\MM_{r\leq R_1},
\eea
we have the following analog of \eqref{eq:decompositionwaveoperatorpertrub:microlocalregion}
\beaa
\square_\g\psi &=& \square_{\gam}\psi+\Gac\pr^2\psi+\dk^{\leq 1}(\Gac)\pr\psi, \quad\textrm{on}\,\,\MM_{r\leq R_1}
\eeaa
which yields, on $\MM_{r\leq R_1}$,
\bea\lab{eq:structrureofFiRHSwaveeqOpThetaipsi:energy}
\nn F_i &=& {\frac{1}{|q|^2}\Opw(\Theta_i)(|q|^2F)}+\Opw(\widetilde{S}^{-1,1}(\MM))\psi\\
&&+\frac{1}{|q|^2}\Big[|q|^2\Gac\pr^2+|q|^2\dk^{\leq 1}(\Gac)\pr, \Opw(\Theta_i)\Big]\psi.
\eea
Finally, we consider a smooth cut-off function $\chi=\chi(r)$ such that $\chi=1$ on $r\geq r_+(1+\dred)$ and $\chi=0$ for $r\leq r_+(1+\dred/2)$, and we define
\bea\lab{def:psiiandFi:kerrpert:withcut-off}
\psi_{i,\chi}:=\chi(r)\psi_i, \qquad \square_\g(\psi_{i,\chi})=F_{i,\chi}, \qquad F_{i,\chi}:=\chi(r)F_i+[\square_\g, \chi]\psi_i.
\eea}

{\noindent{\bf Step 3.}} Let us first consider the case {$1\leq i\leq \iota$}. Given that $\pr_{\tt} + \frac{2amr}{(r^2+a^2)^2}\pr_{\tphi}$ is a globally timelike vectorfield in the region {$\MM_{r_+(1+\dred/2), R_1}$}, we {choose the integer $\iota$ sufficiently} large such that we can construct  globally timelike vectorfields 
\bea\lab{eq:defintionofthedifferentialoperatorsViformicrolocalNRGestimates}
V_i:=\pr_{\tt}+ d_i(r) \pr_{\tphi},\qquad 1\leq i\leq \iota,
\eea
each of which is {Killing} in the region {$\MM\cap\{r\in \widetilde{I}_i\}$} and equals ${\pr_\tau}$ in the region {$\MM_{10m, +\infty}$, where} $\{d_i(r)\}_{i=0,1,\ldots, \iota}$ are {smooth} scalar functions of $r${, and where the intervals $\widetilde{I}_i$, $1\leq i\leq \iota$, are defined by
\bea
\widetilde{I}_i:=\left(r_{\text{min}, \trap}+\frac{i-\frac{5}{2}}{\iota}(r_{\text{max}, \trap}-r_{\text{min}, \trap}),r_{\text{min}, \trap}+\frac{i+\frac{3}{2}}{\iota}(r_{\text{max}, \trap}-r_{\text{min}, \trap})\right),
\eea
so that we have
\bea\lab{eq:mainpropertyofintervalwidetildeIiwhichcontrainIi}
I_i\subset\widetilde{I}_i, \qquad \textrm{dist}(I_i, \Reals\setminus\widetilde{I}_i)=\frac{r_{\text{max}, \trap}-r_{\text{min}, \trap}}{\iota}, \quad 1\leq i\leq \iota.
\eea} 

{Next, we consider the divergence identity \eqref{eq:EnerIden:General:wave} with $\psi=\psi_{i,\chi}$, $X=V_i$ and $w=0$. Integrating it on $\MM_{r\leq R_1}(\tau_0, \tau)$ with $\tau_0\leq 0$ and $\tau\geq 1$, and taking the support of $\psi_{i,\chi}$ into account, we infer
\beaa
&&\int_{\Si_{r_+(1+\dred/2), R_1}(\tau)}\QQ_{\a\b}[\psi_{i,\chi}]V_i^{\b}N_{\Si(\tau)}^\a\\
&\les& \int_{\Si_{r_+(1+\dred/2), R_1}(\tau_0)}\QQ_{\a\b}[\psi_{i,\chi}]V_i^{\b}N_{\Si(\tau_0)}^\a +\int_{\MM_{r_+(1+\dred/2), R_1}(\tau_0, \tau)}\left|{}^{(V_i)} \pi \cdot \QQ[\psi_{i,\chi}]\right|\\
&&+ \left|\int_{\MM_{r_+(1+\dred/2), R_1}(\tau_0, \tau)}{\Re\Big(}F_{i,\chi}\ov{V_i(\psi_{i,\chi})}{\Big)}\right|+\M[\psi](\Reals),
\eeaa
where we have estimated the boundary term on $r=R_1$ by the last term on the RHS thanks to \eqref{choiceofR1:energyestimate:kerrpert}. Hence, using the properties of $\chi$ and the definition of $F_{i,\chi}$, we obtain 
{\beaa
&&\E_{r_+(1+\dred), R_1}[\psi_i](\tau)\\
&\les& \E_{r_+(1+\dred/2), R_1}[\psi_i](\tau_0)+\int_{\MM_{r_+(1+\dred/2), R_1}(\tau_0, \tau)}\left|{}^{(V_i)} \pi \cdot \QQ[\psi_{i,\chi}]\right|\\
&&+ \left|\int_{\MM_{r_+(1+\dred/2), R_1}(\tau_0, \tau)}{\Re\Big(}\chi(r)F_i\ov{V_i(\psi_{i,\chi})}{\Big)}\right|
+\M_{r\leq r_+(1+\dred)}[\psi_i](\Reals)+\M[\psi](\Reals).
\eeaa}
Also, we have, in view of the definition \eqref{eq:defintionofthedifferentialoperatorsViformicrolocalNRGestimates} of $V_i$,
\beaa
&&\int_{\MM_{r_+(1+\dred/2), R_1}(\tau_0, \tau)}\left|{}^{(V_i)} \pi \cdot \QQ[\psi_{i,\chi}]\right|\\
&\les& \int_{\MM_{r_+(1+\dred/2), R_1}(\tau_0, \tau)}\left(\left|{}^{(\pr_\tau)} \pi \cdot \QQ[\psi_{i,\chi}]\right|+\left|{}^{(\pr_{\tphi})} \pi\cdot\QQ[\psi_{i,\chi}]\right|+|d_i'(r)||\pr^{\leq 1}\psi_i|^2\right)\\
&\les& \ep\sup_{\tau'\in[\tau_0, \tau]}\E_{r_+(1+\dred),10m}[\psi_i](\tau')+\ep\M[\psi_i](\Reals)+\int_{\MM_{r\in[r_+(1+\dred/2), 10m]\setminus \widetilde{I}_i}}|\partial \psi_i|^2
\eeaa
where we used\footnote{{Notice that the energy part of the error term generated in Lemma \ref{lemma:basiclemmaforcontrolNLterms:ter} on a region of the type $\MM_{r_+(1+\dred/2), R_1}$ is only needed in $\Mtrap$ so that $\ep\sup_{\tau'\in[\tau_0, \tau]}\E_{r_+(1+\dred),10m}[\psi_i](\tau')$ indeed suffices.}}  Lemmas \ref{lemma:controlofdeformationtensorsforenergyMorawetz}
 and \ref{lemma:basiclemmaforcontrolNLterms:ter}, and the support properties of $d_i'(r)$. In view of the above, this implies, for $\ep$ small enough, 
 \beaa
\nn\E_{r_+(1+\dred), R_1}[\psi_i](\tau) &\les& \E_{r_+(1+\dred/2), R_1}[\psi_i](\tau_0)+\int_{\MM_{r\in[r_+(1+\dred/2), 10m]\setminus \widetilde{I}_i}}|\partial \psi_i|^2\\
&&+\ep\M[\psi_i](\Reals) {+\M[\psi](\Reals)}+ \left|\int_{\MM_{r_+(1+\dred/2), R_1}(\tau_0, \tau)}{\Re\Big(}\chi(r)F_i\ov{V_i(\psi_{i,\chi})}{\Big)}\right|.
\eeaa
Also, recalling that the symbol $e\in\widetilde{S}^{1,0}(\MM)$ defined in \eqref{eq:defintionofthesymboleasasquarerootofsigma2TXEmsumofsquares:1:defe} verifies \eqref{eq:defintionofthesymboleasasquarerootofsigma2TXEmsumofsquares:1}, and using the definition of $I_i$, $\widetilde{I}_i$ and $\Theta_i$, there exists a constant $c>0$ such that\footnote{{Indeed, if $\Xi$ is in the support of $\Theta_i(\Xi)$, then $\Xi\in\NN_i$ and hence $r_{\trap}\in I_i$, so that $|r-r_{\trap}|\geq\frac{r_{\text{max}, \trap}-r_{\text{min}, \trap}}{\iota}$ for $r\notin\widetilde{I}_i$ in view of \eqref{eq:mainpropertyofintervalwidetildeIiwhichcontrainIi}. The conclusion then follows from 
\beaa
e^2\gtrsim 1+(r-r_{\trap})^2|\Xi|^2\gtrsim 1+\left(\frac{r_{\text{max}, \trap}-r_{\text{min}, \trap}}{\iota}\right)^2|\Xi|^2\gtrsim_{\iota} 1+\Theta_i^2|\Xi|^2.
\eeaa}}
\beaa
e_2:=\sqrt{e^2 -c\Theta_i^2|\Xi|^2}, \quad  e_2\in\widetilde{S}^{1,0}(\MM\cap\{r\notin\widetilde{I}_j\}), 
\eeaa
which implies
\beaa
&&\int_{\MM_{r\in[r_+(1+\dred/2), 10m]\setminus \widetilde{I}_i}}|\partial \psi_i|^2\\
&\les&  \M[\psi](\Reals)+\int_{\MM_{r\in[r_+(1+\dred/2), 10m]\setminus \widetilde{I}_i}}|\Opw(\Theta_i\Xi)\psi|^2\\
&\les& \M[\psi](\Reals)+\int_{\MM_{r\in[r_+(1+\dred/2), 10m]\setminus \widetilde{I}_i}}\Re(\ov{\psi}\Opw(e^2-e^2_2)\psi)\\
&\les& \M[\psi](\Reals)+\int_{\MM_{r_+(1+\dred/2), 10m}}|\Opw(e)\psi|^2\nn\\
&{\les}&{\widetilde{\M}[\psi]},
\eeaa
and hence
\bea\lab{eq:intermediarycontrolofenergypsiionrp1pdredR1}
\nn\E_{r_+(1+\dred), R_1}[\psi_i](\tau) &\les& \E_{r_+(1+\dred/2), R_1}[\psi_i](\tau_0)+{\widetilde{\M}[\psi]}\\
&&+ \left|\int_{\MM_{r_+(1+\dred/2), R_1}(\tau_0, \tau)}{\Re\Big(}\chi(r)F_i\ov{V_i(\psi_{i,\chi})}{\Big)}\right|.
\eea}

\noindent{\bf Step 4.} Next, we focus on the last term on the RHS of \eqref{eq:intermediarycontrolofenergypsiionrp1pdredR1}. 
{To this end, we recall the formula \eqref{eq:structrureofFiRHSwaveeqOpThetaipsi:energy} of $F_i$. Then, we have 
\bea\lab{eq:intermediarycontrolofintRechirFiVipsiichionMMrplusoneplusderedon2toR1tau0tau}
\nn&&\left|\int_{\MM_{r_+(1+\dred/2), R_1}(\tau_0, \tau)}{\Re\Big(}\chi(r)F_i\ov{V_i(\psi_{i,\chi})}{\Big)}\right|\\
\nn&\les&\left|\int_{\MM_{r_+(1+\dred/2), R_1}(\tau_0, \tau)}{\Re\Big(}\chi^2(r){|q|^{-2}}\Opw(\Theta_i)(|q|^2F)\ov{V_i(\Opw(\Th_i)\psi)}{\Big)}\right|\\
\nn&&+\left|\int_{\MM_{r_+(1+\dred/2), R_1}(\tau_0, \tau)}{\Re\Big(}\chi^2(r)\Opw(\widetilde{S}^{-1,1}(\MM))\psi\ov{V_i(\Opw(\Th_i)\psi)}{\Big)}\right|\\
\nn&&+
\left|\int_{\MM_{r_+(1+\dred/2), R_1}(\tau_0, \tau)}{\Re\Big(}\chi^2(r)\frac{1}{|q|^2}\Big[|q|^2\Gac\pr^2+|q|^2\dk^{\leq 1}(\Gac)\pr, \Opw(\Theta_i)\Big]\psi\ov{V_i(\Opw(\Th_i)\psi)}{\Big)}\right|
\\
\nn&\les&\left|\int_{\Mtrap(\tau_0, \tau)}{\Re\Big(}{|q|^{-2}}\Opw(\Theta_i)(|q|^2F)\ov{V_i(\Opw(\Th_i)\psi)}{\Big)}\right|+\int_{\Mntrap_{r\leq R_1}}|F|^2+ \M[\psi](\Reals)\\
\nn&&+
\left|\int_{\MM_{r_+(1+\dred/2), R_1}(\tau_0, \tau)}{\Re\Big(}\chi^2(r)\frac{1}{|q|^2}\Big[|q|^2\Gac\pr^2+|q|^2\dk^{\leq 1}(\Gac)\pr, \Opw(\Theta_i)\Big]\psi\ov{V_i(\Opw(\Th_i)\psi)}{\Big)}\right|\\
&& +\left|\int_{\MM_{r_+(1+\dred/2), R_1}(\tau_0, \tau)}{\Re\Big(}\chi^2(r)\Opw(\widetilde{S}^{-1,1}(\MM))\psi\ov{V_i(\Opw(\Th_i)\psi)}{\Big)}\right|.
\eea}
{Now, in order to control the last two lines in \eqref{eq:intermediarycontrolofintRechirFiVipsiichionMMrplusoneplusderedon2toR1tau0tau}, we introduce the smooth cut-off functions $\chi_{\tau_0, \tau,j}=\chi_{\tau_0, \tau,j}(\tau)$, $j=0,1$, satisfying 
\begin{equation}
\lab{eq:definitionofchitau0taui:i=01}
\begin{aligned}
&\textrm{supp}(\chi_{\tau_0,\tau,0})\subset(\tau_0, \tau), \quad \chi_{\tau_0,\tau,0}=1\,\,\,\textrm{on}\,\,(\tau_0+1, \tau-1), \quad 0\leq\chi_{\tau_0,\tau,j}\leq 1, \,\,j=0,1,\\
&\textrm{supp}(\chi_{\tau_0,\tau,1})\subset(\tau_0-1, \tau_0+2)\cup(\tau-2, \tau+1), \quad \chi_{\tau_0,\tau,1}=1\,\,\,\textrm{on}\,\,(\tau_0, \tau_0+1)\cup(\tau-1, \tau).
\end{aligned}
\end{equation}
Using the properties of the cut-offs $\chi_{\tau_0, \tau,j}$, $j=0,1$, as well as Proposition \ref{prop:PDO:MM:Weylquan:mixedoperators} and Lemma \ref{lemma:actionmixedsymbolsSobolevspaces:MM}, we have
\bea\lab{eq:intermediarycontrolofintRechirFiVipsiichionMMrplusoneplusderedon2toR1tau0tau:aux1}
\nn&&\left|\int_{\MM_{r_+(1+\dred/2), R_1}(\tau_0, \tau)}{\Re\Big(}\chi^2(r)\Opw(\widetilde{S}^{-1,1}(\MM))\psi\ov{V_i(\Opw(\Th_i)\psi)}{\Big)}\right|\\
\nn&\les& \left|\int_{\MM_{r_+(1+\dred/2), R_1}}\chi_{\tau_0,\tau,0}\chi^2(r)\Opw(\widetilde{S}^{-1,1}(\MM))\psi\ov{V_i(\Opw(\Th_i)\psi)}\right|\\
\nn&&+\left|\int_{\MM_{r_+(1+\dred/2), R_1}(\tau_0, \tau)}(1-\chi_{\tau_0,\tau,0})\chi^2(r)\Opw(\widetilde{S}^{-1,1}(\MM))\psi\ov{V_i(\Opw(\Th_i)\psi)}\right|\\
\nn&\les& \left|\int_{\MM_{r_+(1+\dred/2), R_1}}\psi\,\ov{\Opw(\widetilde{S}^{-1,1}(\MM))\big(\chi_{\tau_0,\tau,0}\chi^2(r)V_i(\Opw(\Th_i)\psi)\big)}\right|\\
\nn&&+\int_{\MM_{r_+(1+\dred/2), R_1}}\chi_{\tau_0,\tau,1}\chi^2(r)\Big|\Opw(\widetilde{S}^{-1,1}(\MM))\psi\ov{V_i(\Opw(\Th_i)\psi)}\Big|\\
\nn&\les& \int_{\MM_{r_+(1+\dred/2), R_1}}|\psi||\Opw(\widetilde{S}^{0,1}(\MM))\psi|\\
\nn&&+\int_{\MM_{r_+(1+\dred/2), R_1}}\Big|\Opw(\widetilde{S}^{-1,1}(\MM))\psi\ov{V_i(\Opw(\Th_i)\chi_{\tau_0,\tau,1}\psi)}\Big|+\M[\psi](\Reals)\\
\nn&\les& \M[\psi](\Reals)+\Big(\M[\psi](\Reals)\Big)^{\frac{1}{2}}\left(\int_{\MM_{r_+(1+\dred/2), R_1}}|\pr^{\leq 1}(\chi_{\tau_0,\tau,1}\psi)|^2\right)^{\frac{1}{2}}\\
&\les&  \M[\psi](\Reals)+\Big(\M[\psi](\Reals)\Big)^{\frac{1}{2}}\Big(\EM[\psi](\Reals)\Big)^{\frac{1}{2}} 
\eea
which deals with the last line in \eqref{eq:intermediarycontrolofintRechirFiVipsiichionMMrplusoneplusderedon2toR1tau0tau}. To control the before to last line in \eqref{eq:intermediarycontrolofintRechirFiVipsiichionMMrplusoneplusderedon2toR1tau0tau}, we rely again on the properties of the cut-offs $\chi_{\tau_0, \tau,j}$, $j=0,1$, and obtain 
\beaa
&&\left|\int_{\MM_{r_+(1+\dred/2), R_1}(\tau_0, \tau)}{\Re\Big(}\chi^2(r)\frac{1}{|q|^2}\Big[|q|^2\Gac\pr^2+|q|^2\dk^{\leq 1}(\Gac)\pr, \Opw(\Theta_i)\Big]\psi\ov{V_i(\Opw(\Th_i)\psi)}{\Big)}\right|\\
&\les&\left|\int_{\MM_{r_+(1+\dred/2), R_1}}\chi_{\tau_0, \tau,0}\chi^2(r)\frac{1}{|q|^2}\Big[|q|^2\Gac\pr^2+|q|^2\dk^{\leq 1}(\Gac)\pr, \Opw(\Theta_i)\Big]\psi\ov{V_i(\Opw(\Th_i)\psi)}\right|\\
&&+\int_{\MM_{r_+(1+\dred/2), R_1}}\chi_{\tau_0, \tau,1}\left|\frac{1}{|q|^2}\Big[|q|^2\Gac\pr^2+|q|^2\dk^{\leq 1}(\Gac)\pr, \Opw(\Theta_i)\Big]\psi\right|\Big|V_i(\Opw(\Th_i)\psi)\Big|\\
&\les& \left|\int_{\MM_{r_+(1+\dred/2), R_1}}\chi_{\tau_0, \tau,0}\chi^2(r)\frac{1}{|q|^2}\Big[|q|^2\Gac\pr^2+|q|^2\dk^{\leq 1}(\Gac)\pr, \Opw(\Theta_i)\Big]\psi\ov{V_i(\Opw(\Th_i)\psi)}\right|\\
&&+\left(\int_{\MM_{r_+(1+\dred/2), R_1}}\left|\chi(r)\Big[|q|^2\Gac\pr^2+|q|^2\dk^{\leq 1}(\Gac)\pr, \Opw(\Theta_i)\Big]\psi\right|^2\right)^{\frac{1}{2}}\Big(\EM[\psi](\Reals)\Big)^{\frac{1}{2}}. 
\eeaa
Then}, we argue as in Lemma \ref{lemma:controloferrortermsinTXEmircolocalenergybulk}, using Lemma \ref{lem:gpert:MMtrap},  \eqref{eq:decaypropertiesofGac:microlocalregion} and Lemma \ref{lemma:actionmixedsymbolsSobolevspaces:MM}, to obtain 
\beaa
&&\left|\int_{\MM_{r_+(1+\dred/2), R_1}{(\tau_0, \tau)}}\Re\bigg(\chi^2(r)\frac{1}{|q|^2}\Big[|q|^2\Gac\pr^2+|q|^2\dk^{\leq 1}(\Gac)\pr, \Opw(\Theta_i)\Big]\psi\ov{V_i({\Opw(\Th_i)\psi})}\bigg)\right|\\
&\les& \left|\int_{\MM_{r_+(1+\dred/2), R_1}}{\chi_{\tau_0, \tau,0}}\chi^2(r)\Re\Big(\Big[\Gac, \Opw(\Theta_i)\pr\Big]\pr\psi\ov{V_i({\Opw(\Th_i)\psi})}\Big)\right| \\
&&{+\left(\int_{\MM_{r_+(1+\dred/2), R_1}}\left|\chi(r)\Big[\Gac, \Opw(\Theta_i)\pr\Big]\pr\psi\right|^2\right)^{\frac{1}{2}}\Big(\EM[\psi](\Reals)\Big)^{\frac{1}{2}}}\\
&&+ \ep\EM[\psi](\Reals)+\ep\int_{H_{R_1}}|\pr^{\leq 1}\psi|^2\\
&\les& {\left|\int_{\MM_{r_+(1+\dred/2), R_1}}\chi_{\tau_0, \tau,0}\chi^2(r)\Re\Big(\Big[\Gac, \Opw(\Theta_i)\pr\Big]\pr\psi\ov{V_i(\Opw(\Th_i)\psi)}\Big)\right|} \\
&&{+\left(\int_{\MM_{r_+(1+\dred/2), R_1}}\left|\chi(r)\Big[\Gac, \Opw(\Theta_i)\pr\Big]\pr\psi\right|^2\right)^{\frac{1}{2}}\Big(\EM[\psi](\Reals)\Big)^{\frac{1}{2}}+ \ep\EM[\psi](\Reals)},
\eeaa
where we used \eqref{choiceofR1:energyestimate:kerrpert} in the last inequality. {Now, as in Step 3 of the proof of Lemma \ref{lemma:controloferrortermsinTXEmircolocalenergybulk}, we decompose $\pr$ into $\pr_r$ and $\prtan$, and rely on \eqref{eq:decompositionofpr2psiinfunctionwaveandprprtan:microlocal} and the analog of \eqref{eq:lowerboundDeltaovermodqsquare} on $\MM_{r_+(1+\dred/2), R_1}$. We infer 
\begin{align*}
&\left|\int_{\MM_{r_+(1+\dred/2), R_1}(\tau_0, \tau)}\Re\bigg(\chi^2(r)\frac{1}{|q|^2}\Big[|q|^2\Gac\pr^2+|q|^2\dk^{\leq 1}(\Gac)\pr, \Opw(\Theta_i)\Big]\psi\ov{V_i(\Opw(\Th_i)\psi)}\bigg)\right|\\
\les& \left|\int_{\MM_{r_+(1+\dred/2), R_1}}\chi_{\tau_0, \tau,0}\chi^2(r)\Re\left(\prtan\left(\Big[\Gac, \Opw(\widetilde{S}^{1,1}(\MM))\Big]O(\dred^{-1})\psi\right)\ov{V_i(\Opw(\Th_i)\psi)}\right)\right| \\
&+\left(\int_{\MM_{r_+(1+\dred/2), R_1}}\left|\chi(r)\prtan\left(\Big[\Gac, \Opw(\widetilde{S}^{1,1}(\MM))\Big]O(\dred^{-1})\psi\right)\right|^2\right)^{\frac{1}{2}}\Big(\EM[\psi](\Reals)\Big)^{\frac{1}{2}}\\
& +\left|\int_{\MM_{r_+(1+\dred/2), R_1}}\chi_{\tau_0, \tau,0}\chi^2(r)\Re\Big(\Big[\Gac, \Opw(\Theta_i)\Big]\Big(O(\dred^{-1})\square_\g\psi+O(\dred^{-1})\pr\psi\Big)\ov{V_i(\Opw(\Th_i)\psi)}\Big)\right| \\
&+\left(\int_{\MM_{r_+(1+\dred/2), R_1}}\left|\chi(r)\Big[\Gac, \Opw(\Theta_i)\Big]\Big(O(\dred^{-1})\square_\g\psi+O(\dred^{-1})\pr\psi\Big)\right|^2\right)^{\frac{1}{2}}\Big(\EM[\psi](\Reals)\Big)^{\frac{1}{2}}\\
&+ \ep\EM[\psi](\Reals).
\end{align*}
Arguing like in Step 3 and Step 4 of Lemma \ref{lemma:controloferrortermsinTXEmircolocalenergybulk}, using in particular a dyadic decomposition, \eqref{eq:decaypropertiesofGac:microlocalregion}, and the commutator estimates of Lemma \ref{lemma:bascicommutatorlemmawithelementaryproof:mixedsymbols:MMcase}, we infer
\beaa
&&\left|\int_{\MM_{r_+(1+\dred/2), R_1}(\tau_0, \tau)}\chi^2(r)\frac{1}{|q|^2}\Re\bigg(\Big[|q|^2\Gac\pr^2+|q|^2\dk^{\leq 1}(\Gac)\pr, \Opw(\Theta_i)\Big]\psi\ov{V_i(\Opw(\Th_i)\psi)}\bigg)\right|\\
&\les& \ep\int_{\Mntrap_{r\leq R_1}}|F|^2+\ep\int_{\Mtrap}\tau^{-1-\dec}|F|^2 +\ep\int_{\Mtrap}\left|\Opw(\widetilde{S}^{-1,0}(\MM))F\right|^2+\ep\EM[\psi](\Reals).
\eeaa
Together with \eqref{eq:intermediarycontrolofintRechirFiVipsiichionMMrplusoneplusderedon2toR1tau0tau} and \eqref{eq:intermediarycontrolofintRechirFiVipsiichionMMrplusoneplusderedon2toR1tau0tau:aux1}}, this yields
{\bea\lab{eq:estimatemaincontributiontowidetildeNNMtrapforenergy:003}
&&\left|\int_{\MM_{r_+(1+\dred/2), R_1}(\tau_0, \tau)}{\Re\Big(}\chi(r)F_i\ov{V_i(\psi_{i,\chi})}{\Big)}\right|\nn\\
&\les&\left|\int_{\Mtrap(\tau_0, \tau)}{\Re\Big(}{|q|^{-2}}\Opw(\Theta_i)(|q|^2F)\ov{V_i(\Opw(\Th_i)\psi)}{\Big)}\right|+\int_{\Mntrap_{r\leq R_1}}|F|^2+\M[\psi](\Reals)\nn\\
&& {+\Big(\M[\psi](\Reals)\Big)^{\frac{1}{2}}\Big(\EM[\psi](\Reals)\Big)^{\frac{1}{2}} +\ep\int_{\Mtrap}\tau^{-1-\dec}|F|^2 +\ep\int_{\Mtrap}\left|\Opw(\widetilde{S}^{-1,0}(\MM))F\right|^2}\nn\\
&&+\ep\EM[\psi](\Reals)\nn\\
\nn&\les&\JJ_i(\tau_0,\tau)+\M[\psi](\Reals) {+\Big(\M[\psi](\Reals)\Big)^{\frac{1}{2}}\Big(\EM[\psi](\Reals)\Big)^{\frac{1}{2}}} +\ep\EM[\psi](\Reals)\\
&&+\int_{\Mntrap_{r\leq R_1}}|F|^2 { +\ep\int_{\Mtrap}\tau^{-1-\dec}|F|^2 +\ep\int_{\Mtrap}\left|\Opw(\widetilde{S}^{-1,0}(\MM))F\right|^2,}
\eea
where we have defined
\bea
\lab{eq:definitionofJJiterm}
\JJ_i(\tau_0,\tau):=\left|\int_{\Mtrap(\tau_0, \tau)}{\Re\Big(}{|q|^{-2}}\Opw(\Theta_i)(|q|^2F)\ov{V_i(\Opw(\Th_i)\psi)}{\Big)}\right|, \,\,\, {i=-1,0,1,\ldots, \iota}.
\eea
Plugging in \eqref{eq:intermediarycontrolofenergypsiionrp1pdredR1}, we deduce
\bea\lab{eq:intermediarycontrolofenergypsiionrp1pdredR1:1}
\nn\E_{r_+(1+\dred), R_1}[\psi_i](\tau) &\les& \E_{r_+(1+\dred/2), R_1}[\psi_i](\tau_0)+\widetilde{\M}[\psi]+\JJ_i(\tau_0,\tau) +\ep\EM[\psi](\Reals)\\
&& {+\Big(\M[\psi](\Reals)\Big)^{\frac{1}{2}}\Big(\EM[\psi](\Reals)\Big)^{\frac{1}{2}}} +\int_{\Mntrap_{r\leq R_1}}|F|^2\nn\\
&& { +\ep\int_{\Mtrap}\tau^{-1-\dec}|F|^2 +\ep\int_{\Mtrap}\left|\Opw(\widetilde{S}^{-1,0}(\MM))F\right|^2.}
\eea}

\noindent{\bf Step 5.} Next, we integrate \eqref{eq:intermediarycontrolofenergypsiionrp1pdredR1:1} in $\tau_0$ for $\tau_0$ in the interval $[-1, 0]$. This yields
\begin{align}\lab{eq:intermediarycontrolofenergypsiionrp1pdredR1:2}
\nn\E_{r_+(1+\dred), R_1}[\psi_i](\tau) \les{}&\int_{-1}^0\E_{r_+(1+\dred/2), R_1}[\psi_i](\tau_0)d\tau_0 +{\widetilde{\M}[\psi]}+\sup_{\tau_0\in[-1,0]}\JJ_i(\tau_0,\tau)\\
& {+\Big(\M[\psi](\Reals)\Big)^{\frac{1}{2}}\Big(\EM[\psi](\Reals)\Big)^{\frac{1}{2}}}  {+\ep\EM[\psi](\Reals)+\int_{\Mntrap_{r\leq R_1}}|F|^2}\nn\\
& { +\ep\int_{\Mtrap}\tau^{-1-\dec}|F|^2 +\ep\int_{\Mtrap}\left|\Opw(\widetilde{S}^{-1,0}(\MM))F\right|^2.}
\end{align}
Next, we estimate the first term on the RHS of \eqref{eq:intermediarycontrolofenergypsiionrp1pdredR1:2}. We have
\beaa
\int_{-1}^0\E_{r_+(1+\dred/2), R_1}[\psi_i](\tau_0)d\tau_0 &\les& \M[\psi](\Reals)+\int_{\MM_{r_+(1+\dred/2), R_1}(-1, 0)}|\prtan\psi_i|^2,
\eeaa
where the notation $\prtan$ has been introduced in \eqref{def:tangentialderivativeonHr:Kerrpert}. Also, since $\psi=0$ for $\tau\leq 1$, we have
\beaa
\int_{\MM_{r_+(1+\dred/2), R_1}(-1, 0)}|\prtan\psi_i|^2\les \int_{\MM_{r_+(1+\dred/2), R_1}}|\psi|^2
\eeaa
and hence
\beaa
\int_{-1}^0\E_{r_+(1+\dred/2), R_1}[\psi_i](\tau_0)d\tau_0 &\les& \M[\psi](\Reals),
\eeaa
which together with \eqref{eq:intermediarycontrolofenergypsiionrp1pdredR1:2} finally implies, for all $1\leq i\leq \iota$,  
\bea\lab{eq:intermediarycontrolofenergypsiionrp1pdredR1:3}
\nn\E_{r_+(1+\dred), R_1}[\psi_i](\tau) &\les& {\widetilde{\M}[\psi]}
+\sup_{\tau_0\in[-1,0]}\JJ_i(\tau_0,\tau){+\ep\EM[\psi](\Reals)}\\
&& {+\Big(\M[\psi](\Reals)\Big)^{\frac{1}{2}}\Big(\EM[\psi](\Reals)\Big)^{\frac{1}{2}}} {+\int_{\Mntrap_{r\leq R_1}}|F|^2}\nn\\
&&{ +\ep\int_{\Mtrap}\tau^{-1-\dec}|F|^2 +\ep\int_{\Mtrap}\left|\Opw(\widetilde{S}^{-1,0}(\MM))F\right|^2.}
\eea

\noindent{\bf Step 6.} Next, we estimate the case $i=0$. Recalling that the symbol $e\in\widetilde{S}^{1,0}(\MM)$ defined in \eqref{eq:defintionofthesymboleasasquarerootofsigma2TXEmsumofsquares:1:defe} verifies \eqref{eq:defintionofthesymboleasasquarerootofsigma2TXEmsumofsquares:1}, and using the {property} of $\Theta_0$ in {\eqref{eq:definitionofpartitionofunityThetajj=1toiota:Theta0}}, there exists a constant $c>0$ such that\footnote{Indeed, in view of \eqref{eq:definitionofpartitionofunityThetajj=1toiota:Theta0}, we have $\text{supp}(\Theta_0)\cap\textrm{supp}(\widetilde{\chi}_5)=\emptyset$ and hence $e^2\gtrsim 1+|\Xi|^2$ on $\text{supp}(\Theta_0)$.}
\beaa
e_3:=\sqrt{e^2 -c\Theta_0^2|\Xi|^2}, \quad  e_3\in\widetilde{S}^{1,0}(\MM), 
\eeaa
which implies, arguing as in Step 3, 
\bea\lab{eq:intermediarycontrolofenergypsiionrp1pdredR1:4}
\int_{\MM_{r_+(1+\dred/2), R_1}}|\partial \psi_0|^2 &\les& {\widetilde{\M}[\psi]}.
\eea
Then, we consider the divergence identity \eqref{eq:EnerIden:General:wave} with $\psi=\psi_{0,\chi}$, $X=V_0$ and $w=0$, where {$V_0=\pr_{\tt}+d_0(r)\pr_{\tphi}$ is a timelike vectorfield for $r>r_+(1+\dred/2)$ and equals $\pr_{\tt}$ for $r\geq 10m$}. Proceeding similarly to case $i\neq 0$, and noticing that the case $i=0$ is significantly simpler thanks to \eqref{eq:intermediarycontrolofenergypsiionrp1pdredR1:4}, we obtain the following  analog  of \eqref{eq:intermediarycontrolofenergypsiionrp1pdredR1:3} 
\bea\lab{eq:intermediarycontrolofenergyps0ionrp1pdredR1:0000}
\E_{r_+(1+\dred), R_1}[\psi_0](\tau) &\les& {\widetilde{\M}[\psi]} +{\sup_{\tau_0\in[-1,0]}\JJ_0(\tau_0,\tau)}  +\ep\EM[\psi](\Reals)\nn\\
&& {+\Big(\M[\psi](\Reals)\Big)^{\frac{1}{2}}\Big(\EM[\psi](\Reals)\Big)^{\frac{1}{2}}} {+\int_{\Mntrap_{r\leq R_1}}|F|^2}\nn\\
&&{ +\ep\int_{\Mtrap}\tau^{-1-\dec}|F|^2 +\ep\int_{\Mtrap}\left|\Opw(\widetilde{S}^{-1,0}(\MM))F\right|^2.}
\eea

{Next, we estimate the case $i=-1$.} 
{In view of $\Theta_{-1}\in\widetilde{S}^{-\infty,0}(\MM)$ as stated in \eqref{eq:definitionofpartitionofunityThetajj=1toiota:Theta-1}, we have}
\bea\lab{eq:intermediarycontrolofenergypsiionrp1pdredR1:4:4444}
\int_{\MM_{r_+(1+\dred/2), R_1}}|\partial\psi_{-1}|^2 &\les& \M[\psi](\Reals).
\eea
 Proceeding as for the case $i=0$, replacing \eqref{eq:intermediarycontrolofenergypsiionrp1pdredR1:4} by \eqref{eq:intermediarycontrolofenergypsiionrp1pdredR1:4:4444}, and noticing that the case $i=-1$ is even easier than the case $i=0$, we obtain the following analog of \eqref{eq:intermediarycontrolofenergyps0ionrp1pdredR1:0000}
\beaa
\E_{r_+(1+\dred), R_1}[\psi_{-1}](\tau) &\les& \M[\psi](\Reals) +\ep\EM[\psi](\Reals)+{\sup_{\tau_0\in[-1,0]}\JJ_{-1}(\tau_0,\tau)}\nn\\
&& {+\Big(\M[\psi](\Reals)\Big)^{\frac{1}{2}}\Big(\EM[\psi](\Reals)\Big)^{\frac{1}{2}}} {+\int_{\Mntrap_{r\leq R_1}}|F|^2}\nn\\
&&{ +\ep\int_{\Mtrap}\tau^{-1-\dec}|F|^2 +\ep\int_{\Mtrap}\left|\Opw(\widetilde{S}^{-1,0}(\MM))F\right|^2,}
\eeaa
{where, in defining $\JJ_{-1}(\tau_0,\tau)$ as in \eqref{eq:definitionofJJiterm}, we take $V_{-1}=V_0$ with $V_0$ as introduced above.}
Summing with \eqref{eq:intermediarycontrolofenergypsiionrp1pdredR1:3} for all $i=1,\cdots,\iota$, and with    \eqref{eq:intermediarycontrolofenergyps0ionrp1pdredR1:0000}, and taking the supremum in $\tau\in\Reals$, we deduce 
\bea\lab{eq:degenerateenergy:psii:sum:sup:middle:0101}
&& {\sup_{\tau\in\Reals}}\left(\sum_{i=-1}^\iota\E_{r_+(1+\dred), R_1}[\psi_i](\tau)\right)\nn\\
&\les&{\widetilde{\M}[\psi]}+\sup_{-1\leq \tau_0<\tau}\sum_{i={-1}}^{\iota}\JJ_i(\tau_0,\tau)+{\int_{\Mntrap_{r\leq R_1}}|F|^2} +\ep\EM[\psi](\Reals)\nn\\ 
&& {+\Big(\M[\psi](\Reals)\Big)^{\frac{1}{2}}\Big(\EM[\psi](\Reals)\Big)^{\frac{1}{2}}  +\ep\int_{\Mtrap}\tau^{-1-\dec}|F|^2}\nn\\ 
&& {+\ep\int_{\Mtrap}\left|\Opw(\widetilde{S}^{-1,0}(\MM))F\right|^2.}
\eea
Now, notice from \eqref{eq:definitionofpartitionofunityThetajj=0toiota:propsumto1} that 
\beaa
\sum_{i=-1}^\iota\Theta_i=1\quad\textrm{on}\,\,\GG_\Xi, 
\eeaa
which implies 
\beaa
\sum_{i=-1}^\iota\psi_i = \sum_{i=-1}^\iota\Opw(\Theta_i)\psi=\Opw\left(\sum_{i=-1}^\iota\Theta_i\right)\psi=\Opw(1)\psi=\psi
\eeaa
and hence
\beaa
\sup_{\tau\in\Reals}\E_{r_+(1+\dred), R_1}[\psi](\tau) &\les& \sup_{\tau\in\Reals}\left(\sum_{i=-1}^\iota\E_{r_+(1+\dred), R_1}[\psi_i](\tau)\right).
\eeaa
Together with \eqref{eq:degenerateenergy:psii:sum:sup:middle:0101}, we deduce  
\bea\lab{eq:degenerateenergy:psii:sum:sup:middle}
\nn&&\sup_{\tau\in\Reals}\E_{r_+(1+\dred), R_1}[\psi](\tau) 
{+{\sup_{\tau\in\Reals}}\left(\sum_{i=-1}^\iota\E_{r_+(1+\dred), R_1}[\psi_i](\tau)\right)}\\
&\les& {\widetilde{\M}[\psi]}+\sup_{-1\leq \tau_0<\tau}\sum_{i={-1}}^{\iota}\JJ_i(\tau_0,\tau)+{\int_{\Mntrap_{r\leq R_1}}|F|^2} +\ep\EM[\psi](\Reals)\nn\\ 
&& {+\Big(\M[\psi](\Reals)\Big)^{\frac{1}{2}}\Big(\EM[\psi](\Reals)\Big)^{\frac{1}{2}}  +\ep\int_{\Mtrap}\tau^{-1-\dec}|F|^2}\nn\\ 
&& {+\ep\int_{\Mtrap}\left|\Opw(\widetilde{S}^{-1,0}(\MM))F\right|^2.}
\eea

\noindent{\bf Step 7.} Next, we derive energy estimates for the wave equation \eqref{eq:scalarwave} of $\psi$ in $\MM_{R_1, +\infty}(-\infty, \tau)$. To this end, we consider the divergence identity \eqref{eq:EnerIden:General:wave} with $X=\pr_\tau$ and $w=0$. Integrating it on $\MM_{R_1, +\infty}(-\infty, \tau)$, and using the fact that $\psi=0$ for $\tau\leq 1$, we infer
\beaa
&&\int_{\Si_{R_1,+\infty}(\tau)}\QQ_{\a\b}[\psi]\pr_\tau^{\b}N_{\Si(\tau)}^\a\\
&\les& \int_{\MM_{R_1,+\infty}(-\infty, \tau)}\left|{}^{(\pr_\tau)} \pi \cdot \QQ[\psi]\right|+ \left|\int_{\MM_{R_1,+\infty}(-\infty, \tau)}{\Re\Big(}F\ov{\pr_{\tau}\psi}{\Big)}\right|+\M[\psi](\Reals),
\eeaa
where we have estimated the boundary term on $r=R_1$ by the last term on the RHS thanks to \eqref{choiceofR1:energyestimate:kerrpert}. We infer
\beaa
\E_{R_1,+\infty}[\psi](\tau) &\les& \int_{\MM_{R_1,+\infty}(-\infty, \tau)}\left|{}^{(\pr_\tau)} \pi \cdot \QQ[\psi]\right|+ \left|\int_{\MM_{R_1,+\infty}(-\infty, \tau)}{\Re\Big(}F\ov{\pr_{\tau}\psi}{\Big)}\right|+\M[\psi](\Reals).
\eeaa
Estimating the term involving $\pi \cdot \QQ[\psi]$ thanks to Lemmas \ref{lemma:controlofdeformationtensorsforenergyMorawetz}
 and \ref{lemma:basiclemmaforcontrolNLterms:ter}, we infer
 \beaa
\E_{R_1,+\infty}[\psi](\tau) &\les&  \left|\int_{\MM_{R_1,+\infty}(-\infty, \tau)}{\Re\Big(}F\ov{\pr_{\tau}\psi}{\Big)}\right|+\M[\psi](\Reals)+\ep\sup_{\tau\in\Reals}\E_{R_1,+\infty}[\psi](\tau).
\eeaa
Taking the supremum in $\tau\in\Reals$, and together with \eqref{eq:degenerateenergy:psii:sum:sup:middle}, we infer, 
\bea
\lab{eq:energyawayfromhorizon:kerrandpertu:energyestimatepart:robust}
&&\sup_{\tau\in\Reals}\E_{r_+(1+\dred), +\infty}[\psi](\tau) {+{\sup_{\tau\in\Reals}}\left(\sum_{i=-1}^\iota\E_{r_+(1+\dred), R_1}[\psi_i](\tau)\right)}\nn\\
&\les& {\widetilde{\M}[\psi]}+\sup_{-1\leq \tau_0<\tau}\sum_{i={-1}}^{\iota}\JJ_i(\tau_0,\tau)+{\int_{\Mntrap}|F|^2} +\ep\EM[\psi](\Reals)\nn\\
&&+\sup_{\tau\in\Reals}\left|\int_{\MM_{R_1,+\infty}(-\infty, \tau)}{\Re\Big(}F\ov{\pr_{\tau}\psi}{\Big)}\right| {+\Big(\M[\psi](\Reals)\Big)^{\frac{1}{2}}\Big(\EM[\psi](\Reals)\Big)^{\frac{1}{2}}} \nn\\ 
&& {+\ep\int_{\Mtrap}\tau^{-1-\dec}|F|^2 +\ep\int_{\Mtrap}\left|\Opw(\widetilde{S}^{-1,0}(\MM))F\right|^2,}
\eea
{with $\JJ_i(\tau_0,\tau)$, $i=-1,0,1,\ldots, \iota$, defined in \eqref{eq:definitionofJJiterm}.}

\noindent{\bf Step 8.} Next, we rely on the redshift estimates of Lemma \ref{lemma:redshiftestimates} with $s=0$, $\tau_1=-\infty$ and $\tau_2=+\infty$. Since $\psi=0$ for $\tau\leq 1$, we infer 
\bea
\EMF_{r\leq r_+(1+\dred)}[\psi](\Reals) \les \dred^{-1}\M_{r_+(1+\dred), r_+(1+2\dred)}[\psi](\Reals)+\int_{\MM_{r\leq r_+(1+2\dred)}}|F|^2.
\eea
Also, since $\dred$ has already been fixed, we may compress it in $\les$ which yields
\beaa
\EMF_{r\leq r_+(1+\dred)}[\psi](\Reals) &\les& \M[\psi](\Reals)+\int_{\MM_{r\leq r_+(1+2\dred)}}|F|^2.
\eeaa
Together with \eqref{eq:energyawayfromhorizon:kerrandpertu:energyestimatepart:robust}, and by taking $\ep$ small enough {and using the fact that $[r_+(1+2\dbl), 10m]\subset[r_+(1+\dred), R_1]$}, we infer  
\beaa
&&\sup_{\tau\in\Reals}\E[\psi](\tau) {+{\sup_{\tau\in\Reals}}\left(\sum_{i=-1}^\iota\E_{r_+(1+2\dbl), 10m}[\psi_i](\tau)\right)}\\
 &\les&  {\widetilde{\M}[\psi]}+\sup_{-1\leq \tau_0<\tau}\sum_{i={-1}}^{\iota}\JJ_i(\tau_0,\tau) +{\int_{\Mntrap}|F|^2}+\sup_{\tau\in\Reals}\left|\int_{\MM_{R_1,+\infty}(-\infty, \tau)}{\Re\Big(}F\ov{\pr_{\tau}\psi}{\Big)}\right|\\
 &&+\sup_{\tau\in\Reals}\left|\int_{\MM_{R_1,+\infty}(-\infty, \tau)}{\Re\Big(}F\ov{\pr_{\tau}\psi}{\Big)}\right|  {+\ep\int_{\Mtrap}\tau^{-1-\dec}|F|^2 +\ep\int_{\Mtrap}\left|\Opw(\widetilde{S}^{-1,0}(\MM))F\right|^2,}
\eeaa
{where we have used $\psi_i=\Opw(\Th_i)\psi$ from \eqref{def:psiiandFi:kerrpert},}
and hence\footnote{{In fact, in this paper, we only need the control of the energy of $\psi$ provided by \eqref{eq:energyawayfromhorizon:kerrandpertu:energyestimatepart}. The additional control for the energy of $\psi_i$, $i=-1,\cdots, \iota$, for $r\in [r_+(1+2\dbl), 10m]$ is used in the corresponding derivation of conditional energy estimates for Teukolsky in \cite{MaSz25}.}}
\bea\lab{eq:energyawayfromhorizon:kerrandpertu:energyestimatepart}
\nn&&\sup_{\tau\in\Reals}\E[\psi](\tau) {+{\sup_{\tau\in\Reals}}\left(\sum_{i=-1}^\iota\E_{r_+(1+2\dbl), 10m}[\psi_i](\tau)\right)}\\ 
&\les&  {\widetilde{\M}[\psi]}+\sup_{\tau\in\Reals}\left|\int_{\Mntrap(-\infty, \tau)}{\Re\Big(}F\ov{\pr_{\tau}\psi}{\Big)}\right|+{\int_{\Mntrap}|F|^2}+\sup_{-1\leq \tau_0<\tau}\sum_{i={-1}}^{\iota}\JJ_i(\tau_0,\tau)\nn\\
&& {+\ep\int_{\Mtrap}\tau^{-1-\dec}|F|^2 +\ep\int_{\Mtrap}\left|\Opw(\widetilde{S}^{-1,0}(\MM))F\right|^2,}
\eea
{with $\JJ_i(\tau_0,\tau)$, $i=-1,0,1,\ldots, \iota$, defined in \eqref{eq:definitionofJJiterm}.}

\noindent{\bf Step 9.} Next, we estimate the last term on the RHS of \eqref{eq:energyawayfromhorizon:kerrandpertu:energyestimatepart}. To this end, we introduce the following notation 
\bea\lab{eq:defmathcalJthatincludestheminon3termsandyieldsbothestimsametime}
\nn\mathcal{J} &:=&\min\Bigg[{\sum_{i=1}^{\iota}}\int_{\Mtrap}|F||\Opw(\Theta_i)V_i\Opw(\Theta_i)\psi|, \left(\int_{\Mtrap}\tau^{1+\dec}|F|^2\right)^{\frac{1}{2}}\left(\EM[\psi](\Reals)\right)^{\frac{1}{2}}, \\
&& \qquad\qquad\qquad\qquad\qquad\qquad\qquad\qquad\qquad\left(\int_{\Mtrap}|\pr F|^2\right)^{\frac{1}{2}}\Big({\EM[\psi](\Reals)}\Big)^{\frac{1}{2}}\Bigg].
\eea
Recall that
\beaa
\JJ_i(\tau_0,\tau)={\left|\int_{\Mtrap(\tau_0, \tau)}{\Re\Big(}{|q|^{-2}}\Opw(\Theta_i)(|q|^2F)\ov{V_i(\Opw(\Th_i)\psi)}{\Big)}\right|}.
\eeaa
{We rely on the properties \eqref{eq:definitionofchitau0taui:i=01} of the cut-offs $\chi_{\tau_0, \tau,j}$, $j=0,1$, and use Proposition \ref{prop:PDO:MM:Weylquan:mixedoperators} and Lemma \ref{lemma:actionmixedsymbolsSobolevspaces:MM} to obtain 
\bea
\lab{eq:controlofJJitau0tau:inproof}
\JJ_i(\tau_0,\tau)&\les&\int_{\Mtrap}\chi_{\tau_0, \tau,1}\left|{|q|^{-2}}\Opw(\Theta_i)(|q|^2F)\right|\left|V_i(\Opw(\Th_i)\psi)\right|\nn\\
&&+\left|\int_{\Mtrap}\Re\Big(\chi_{\tau_0, \tau,0}{|q|^{-2}}\Opw(\Theta_i)(|q|^2F)\ov{V_i(\Opw(\Th_i)\psi)}{\Big)}\right|\nn\\
&\les&\int_{\Mtrap}\left|{|q|^{-2}}\Opw(\Theta_i)(|q|^2F)\right|\left|V_i(\Opw(\Th_i)\chi_{\tau_0, \tau,1}\psi)\right|\nn\\
&&+\int_{\Mtrap}\left|\Opw(\widetilde{S}^{0,0}(\MM))F\right|\left|\Opw(\widetilde{S}^{0,0}(\MM))\psi)\right|\nn\\
&&+\left|\int_{\Mtrap}\Re\Big(\chi_{\tau_0, \tau,0}{|q|^{-2}}\Opw(\Theta_i)(|q|^2F)\ov{V_i(\Opw(\Th_i)\psi)}{\Big)}\right|\nn\\
&\les& \bigg(\int_{\Mtrap}|F|^2\bigg)^{\frac{1}{2}}\bigg(\int_{\Mtrap} |\psi|^2\bigg)^{\frac{1}{2}}+\bigg(\int_{\Mtrap}|F|^2\bigg)^{\frac{1}{2}}\bigg(\sup_{\tau\in\Reals}\E[\psi](\tau)\bigg)^{\frac{1}{2}}\nn\\
&& +\left|\int_{\Mtrap}\Re\Big(\chi_{\tau_0, \tau,0}{|q|^{-2}}\Opw(\Theta_i)(|q|^2F)\ov{V_i(\Opw(\Th_i)\psi)}{\Big)}\right|.
\eea
Taking the adjoint of $\Opw(\Theta_i)$ for the last term, and using Proposition \ref{prop:PDO:MM:Weylquan:mixedoperators} and Lemma \ref{lemma:actionmixedsymbolsSobolevspaces:MM}, we deduce, for $i=1,2,\ldots, \iota$,
\beaa
\JJ_i(\tau_0,\tau) &\les& \int_{\Mtrap}|F||\Opw(\Theta_i)V_i\Opw(\Theta_i)\psi|+\M[\psi](\Reals)+\int_{\Mtrap}|F|^2\nn\\
&&+\bigg(\int_{\Mtrap}|F|^2\bigg)^{\frac{1}{2}}\bigg(\sup_{\tau\in\Reals}\E[\psi](\tau)\bigg)^{\frac{1}{2}}.
\eeaa 
Applying} 
{Cauchy-Schwarz, we have, in view of \eqref{eq:intermediarycontrolofenergypsiionrp1pdredR1:4} \eqref{eq:intermediarycontrolofenergypsiionrp1pdredR1:4:4444},
\beaa
\sup_{-1\leq \tau_0<\tau}\JJ_{-1}(\tau_0,\tau) +\sup_{-1\leq \tau_0<\tau}\JJ_{0}(\tau_0,\tau) 
\les \widetilde{\M}[\psi]+\int_{{\Mtrap}}|F|^2 +\bigg(\int_{\Mtrap}|F|^2\bigg)^{\frac{1}{2}}\bigg(\sup_{\tau\in\Reals}\E[\psi](\tau)\bigg)^{\frac{1}{2}}.
\eeaa}
Also, integrating by parts in $V_i$ {in the last term of} \eqref{eq:controlofJJitau0tau:inproof}, we have
\bea
 \JJ_i(\tau_0,\tau)&\les& {\bigg(\int_{\Mtrap}|\pr^{\leq 1}F|^2\bigg)^{\frac{1}{2}}\bigg(\int_{\Mtrap} |\psi|^2\bigg)^{\frac{1}{2}}}\nn\\
 &&
 {+\bigg(\int_{{\Mtrap}}|F|^2\bigg)^{\frac{1}{2}}\bigg(\sup_{\tau\in\Reals}\E[\psi](\tau)\bigg)^{\frac{1}{2}}, \qquad \forall i=-1,0,1,\ldots,\iota.}
\eea
Finally, using Lemma \ref{lem:gpert:MMtrap} {to estimate the last term of \eqref{eq:controlofJJitau0tau:inproof}}, we have{, for any $i=-1,0,1,\ldots, \iota$.}
\beaa
\JJ_i(\tau_0,\tau)&\les& \left(\int_{\Mtrap}\tau^{1+\dec}|F|^2\right)^{\frac{1}{2}}\left(\EM[\psi](\Reals)\right)^{\frac{1}{2}}+\M[\psi](\Reals)+\int_{{\Mtrap}}|F|^2,
\eeaa
and we infer from the three estimates above and the definition of $\JJ$ that 
\beaa
\sup_{-1\leq \tau_0<\tau}\sum_{i={-1}}^{\iota}\JJ_i(\tau_0,\tau)&\les&\JJ + {\widetilde{\M}[\psi](\Reals)}+\int_{{\Mtrap}}|F|^2 {+\bigg(\int_{{\Mtrap}}|F|^2\bigg)^{\frac{1}{2}}\bigg(\sup_{\tau\in\Reals}\E[\psi](\tau)\bigg)^{\frac{1}{2}}}.
\eeaa
Plugging this estimate to \eqref{eq:energyawayfromhorizon:kerrandpertu:energyestimatepart}, we infer
\bea\lab{eq:energyawayfromhorizon:kerrandpertu:energyestimatepart:v2:JJ}
\nn\sup_{\tau\in\Reals}\E[\psi](\tau) &\les&  {{\widetilde{\M}[\psi]}+\int_{\MM}|F|^2+\sup_{\tau\in\Reals}\left|\int_{\Mntrap(-\infty, \tau)}{\Re\Big(}F\ov{\pr_{\tau}\psi}{\Big)}\right|+\JJ}.
\eea

Now, in view of the definition of $\JJ$ in  \eqref{eq:defmathcalJthatincludestheminon3termsandyieldsbothestimsametime}, we have
\beaa
\nn\mathcal{J} &\leq&{\sum_{i=1}^\iota} \int_{\Mtrap}|F||\Opw(\Theta_i)V_i\Opw(\Theta_i)\psi|,
\eeaa
and 
\beaa
\nn\mathcal{J} &\leq& \bigg(\min\bigg(\int_{\Mtrap}\tau^{1+\dec}|F|^2, \int_{\Mtrap}|\pr F|^2\bigg)\bigg)^{\frac{1}{2}}\Big(\EM[\psi](\Reals)\Big)^{\frac{1}{2}}.
\eeaa
Together with \eqref{eq:energyawayfromhorizon:kerrandpertu:energyestimatepart:v2:JJ}, we infer
\beaa
\nn\sup_{\tau\in\Reals}\E[\psi](\tau) &\les& {\widetilde{\M}[\psi]}+\int_{\MM}|F|^2+\sup_{\tau\in\Reals}\left|\int_{\Mntrap(-\infty, \tau)}{\Re\Big(}F\ov{\pr_{\tau}\psi}{\Big)}\right|\\
&&+{\sum_{i=1}^\iota}\int_{\Mtrap}|F||\Opw(\Theta_i)V_i\Opw(\Theta_i)\psi|,
\eeaa
as stated in \eqref{thm:eq:nondeg:EnerandMora:Kerrandpert}, and 
\beaa
\nn\sup_{\tau\in\Reals}\E[\psi](\tau) &\les&  {\widetilde{\M}[\psi]}+\int_{\MM}|F|^2+\sup_{\tau\in\Reals}\left|\int_{\Mntrap(-\infty, \tau)}{\Re\Big(}F\ov{\pr_{\tau}\psi}{\Big)}\right|\\
\nn&& +\bigg(\min\bigg(\int_{\Mtrap}\tau^{1+\dec}|F|^2, \int_{\Mtrap}|\pr F|^2\bigg)\bigg)^{\frac{1}{2}}\Big(\EM[\psi](\Reals)\Big)^{\frac{1}{2}},
\eeaa
as stated in \eqref{eq:energyestimate:awayhorizon:Kerrandpert:lastsect}. This concludes the proof of Proposition \ref{thm:nondeg:EnerandMora:Kerrandpert}.
\end{proof}

%%%%%%%%%%%%%%%%%%%%%%%%%%%%%%%%%%%%%%

\subsubsection{{End of the proof} of Theorem \ref{th:mainenergymorawetzmicrolocal}}
\label{subsubsect:proofoftheorem64:lastsubsubsection}

%%%%%%%%%%%%%%%%%%%%%%%%%%%%%%%%%%%%%%

{We are now ready to conclude the proof of Theorem \ref{th:mainenergymorawetzmicrolocal}. Recall the conditional nondegenerate Morawetz-flux estimate \eqref{prop:eq:nondeg:Morawetz}, i.e 
\beaa
\nn&& \sup_{\tau\in\mathbb{R}}\E_{r\leq r_+(1+\dred)}[\psi](\tau)+{\widetilde{\MF}[\psi]}\\
\nn&\les& \int_{\Mtrap}|F||\pr_\tau\psi| +\int_{\Mntrap}|F|\big(|\pr_r\psi|+r^{-1}|\psi|\big) +\left|\int_{\Mntrap}{F\ov{\pr_{\tau}\psi}}\right| +\int_{\MM}|F|^2\\
&& +\int_{\MM}r^{-4}|\psi|^2 +\ep\sup_{\tau\in\mathbb{R}}\E[\psi](\tau),
\eeaa
and the condition energy estimate \eqref{thm:eq:nondeg:EnerandMora:Kerrandpert}, i.e., 
\beaa
\nn\sup_{\tau\in\Reals}\E[\psi](\tau) &\les& {\widetilde{\M}[\psi]}+\int_{\MM}|F|^2+\sup_{\tau\in\Reals}\left|\int_{\Mntrap(-\infty, \tau)}{\Re\Big(}F\ov{\pr_{\tau}\psi}{\Big)}\right|\\
&&+{\sum_{i=1}^\iota}\int_{\Mtrap}|F||\Opw(\Theta_i)V_i\Opw(\Theta_i)\psi|.
\eeaa
Combining these two estimates immediately yields
\beaa
\nn&& \widetilde{\EMF}[\psi]\\
\nn&\les& \int_{\MM}r^{-4}|\psi|^2+\int_{\Mtrap}|F||\pr_\tau\psi| +\sum_{i=1}^\iota\int_{\Mtrap}|F||\Opw(\Theta_i)V_i\Opw(\Theta_i)\psi|\\
&&+\int_{\Mntrap}|F|\big(|\pr_r\psi|+r^{-1}|\psi|\big) + \sup_{\tau\in\Reals}\left|\int_{\Mntrap(-\infty, \tau)}{F\ov{\pr_{\tau}\psi}}\right|+\int_{\MM}|F|^2 +\ep\sup_{\tau\in\mathbb{R}}\E[\psi](\tau),
\eeaa
and hence, for $\ep$ small enough,
\beaa
\widetilde{\EMF}[\psi] &\les& \int_{\MM}r^{-4}|\psi|^2+\int_{\Mtrap}|F||\pr_\tau\psi| +\sum_{i=1}^\iota\int_{\Mtrap}|F||\Opw(\Theta_i)V_i\Opw(\Theta_i)\psi|\\
&&+\int_{\Mntrap}|F|\big(|\pr_r\psi|+r^{-1}|\psi|\big) + \sup_{\tau\in\Reals}\left|\int_{\Mntrap(-\infty, \tau)}{F\ov{\pr_{\tau}\psi}}\right|+\int_{\MM}|F|^2.
\eeaa
In view of the definition \eqref{eq:definitionwidetildemathcalNpsijnormRHS} of $\widetilde{\mathcal{N}}[\psi, F](\mathbb{R})$, we infer
\beaa
\widetilde{\EMF}[\psi] &\les& \widetilde{\mathcal{N}}[\psi, F](\mathbb{R}){+}\int_{\MM}r^{-4}|\psi|^2,
\eeaa
as stated in \eqref{th:eq:mainenergymorawetzmicrolocal}. This concludes the proof of Theorem \ref{th:mainenergymorawetzmicrolocal}.}

%%%%%%%%%%%%%%%%%%%%%%%%%%%%%%%%%%%%%%%%%%%%%%%%%%%%%%

\subsection{{Review of the symbols and operators appearing in the proof of Theorem \ref{th:mainenergymorawetzmicrolocal}}}

%%%%%%%%%%%%%%%%%%%%%%%%%%%%%%%%%%%%%%%%%%%%%%%%%%%%%%

We provide  in this section a summary of the properties of useful symbols and operators that are introduced in Sections \ref{sec:proofofth:main:intermediary}  and \ref{sect:CondEMF:Dynamic} for the proof of Theorem \ref{th:mainenergymorawetzmicrolocal}. This will allow us, in our companion paper \cite{MaSz25}, to easily refer to the relevant material in this paper.

\begin{proposition}[Properties of symbols and operators in the proof of Theorem \ref{th:mainenergymorawetzmicrolocal}]
\lab{prop:symbolsandoperators}
We have the following properties in the spacetime region $\MM_{r_+(1+\dhor', R)}$\footnote{This region is where the microlocal Morawetz estimate is proved, with $\dhor'$ and $R$ introduced in Remark \ref{rmk:choiceofconstantRbymeanvalue}.} for the symbols and operators introduced in proving the microlocal energy-Morawetz estimates\footnote{In the following, the symbols in point \eqref{item:symbolsPDOs:1} are used in  Definition \ref{def:microlocalenergyMorawetznorms} to define the microlocal Morawetz norm, the pseudodifferential operators in point \eqref{item:symbolsPDOs:3} are used in Sections \ref{subsect:GeneralEnerIden:PD}--\ref{subsect:CondMoraFlux:Kerrperturb}   to prove the microlocal Morawetz estimate of Proposition \ref{prop:conditionaldegenerateMorawetzflux:pertKerrr:MM}, and the symbols and vectorfields in point \eqref{item:symbolsPDOs:4} are used in Section \ref{subsect:proofoftheorem64} to derive the energy estimates of Proposition \ref{thm:nondeg:EnerandMora:Kerrandpert}.}.
\begin{enumerate}
\item\lab{item:symbolsPDOs:1} (Symbols $\upsilon, r_{\trap}, \sigma_{\trap}, e$). The real valued symbol $\upsilon\in \widetilde{S}^{1,0}(\MM)$ is given as in \eqref{def:upsilonsymbol}  by
\bea
\lab{def:upsilonsymbol:copysummary}
\upsilon=\sqrt{1+\xi_0^2+\mathring{\ga}^{bc} \langle \xi, \pr_{x^b}\rangle \langle \xi, \pr_{x^c}\rangle}
\eea
and satisfies $\pr_r (\upsilon)=0$, 
the real valued symbol  $r_{\trap}\in\widetilde{S}^{0,0}(\MM)$ is given as in \eqref{eq:definitionofrtrapinfunctionofrmaxandcutoffinmathcalG5} and satisfies $\pr_r(r_{\trap})=0$, the real valued symbol $\sigma_{\trap}\in\widetilde{S}^{1,0}(\MM)$ is given as in Definition \ref{def:microlocalenergyMorawetznorms} by 
\beaa
\sigma_{\trap}=(r-r_{\trap})\upsilon,
\eeaa
and the real valued symbol $e\in \widetilde{S}^{1,0}(\MM)$ is given as in \eqref{eq:defintionofthesymboleasasquarerootofsigma2TXEmsumofsquares:1:defe} and satisfies in $\Mtrap$
\bea
\lab{eq:defintionofthesymboleasasquarerootofsigma2TXEmsumofsquares:1:copy}
e\gtrsim 1+ \left(\sqrt{1-\chi_5^2}+|\chi_5||r-r_{\trap}|\right)\upsilon.
\eea

\item\lab{item:symbolsPDOs:3} (PDOs $X, E$ and their symbols). The pseudodifferential operators $X\in \Opw(\widetilde{S}^{1,1}(\MM))$ and $E\in \Opw(\widetilde{S}^{0,0}(\MM))$ are given by
\bea
X:=\Opw(is_0\mu\xi_r+ib_{\tphi}\xiphi+ib_{\tt}\xit)+A\pr_\tau, \qquad E:=\Opw(e_0),
\eea
where $A\geq 2$ is a large enough constant, where the real valued symbols $b_{\tphi}\in \widetilde{S}^{0,0}(\MM)$ and $b_{\tt}\in \widetilde{S}^{0,0}(\MM)$ satisfy in $\Mtrap$  
\begin{equation}\lab{eq:estimatesofbtphiandbtt:intermsofe}
\begin{split}
&\big(|b_{\tphi}|+|b_{\tt}|\big)\upsilon \lesssim e, \quad \big(|b_{\tphi}|+|b_{\tt}|\big)\upsilon^2 \lesssim e^2,\\
&\big(|\{b_{\tphi}, b\}|+|\{b_{\tt}, b\}|\big)\upsilon \lesssim_{b} e\quad \forall\,\, b\in\widetilde{S}^{1,0}(\MM),
\end{split}
\end{equation}
and where the real valued symbols $s_0\in\widetilde{S}^{0,0}(\MM)$  and 
$e_0\in\widetilde{S}^{0,0}(\MM)$  satisfy in $\Mtrap$
\bea
\lab{eq:boundofe0s0:intermsofe}
|s_0|\upsilon \les e,\qquad |e_0|\upsilon \les e.
\eea

\item\lab{item:symbolsPDOs:4} (Symbols $\Th_i$ and vectorfields $V_i$, $i=-1,0,1,\ldots, \iota$). For all $i=-1,0,1,\ldots, \iota$, the real valued symbols $\Th_i\in\widetilde{S}^{0,0}(\MM)$ satisfy $\pr_r(\Th_i)=0$, and the vectorfields $V_i$ are  given by
\bea
V_i:=\pr_\tau+d_i(r)\pr_{\tphi}, \qquad i=-1,0, 1,\cdots,\iota, 
\eea 
for smooth real valued functions $d_i(r)$ supported  in $r\leq 10m$.
\end{enumerate}
\end{proposition}

\begin{proof}
We start with proving point \eqref{item:symbolsPDOs:1}. The properties for $\upsilon$, $r_{\trap}$ and $\sigma_{\trap}$ are immediate from their definition, and the properties of $e$ follow from its definition \eqref{eq:defintionofthesymboleasasquarerootofsigma2TXEmsumofsquares:1:defe} together with the estimate \eqref{eq:defintionofthesymboleasasquarerootofsigma2TXEmsumofsquares:1} and the following trivial bound in $\Mtrap$
\beaa
\sqrt{r^{-2}\xit^2+r^{-4}\xiphi^2 +r^{-4}\Lambda^2}\gtrsim \upsilon.
\eeaa

Next, we show point \eqref{item:symbolsPDOs:3}.  
We have, in view of \eqref{eq:generalformofthePDOmultipliersXandE} and  \eqref{eq:definitionofQhQyQfQzintermsofTXEands0s1e0}, 
\bea\lab{eq:definitionofthesymbolx1andX:summary}
\bsplit
X=&\Opw(i\mu s_0\xi_r)+A\pr_\tau+\Opw(x_1),\\
x_1=&\frac{is_0 \S_1}{\R} + i s_1 -iA\xit, \quad x_1\in\widetilde{S}^{1,0}(\MM).
\end{split}
\eea
Since $z_j=A(\xit+\chi_{z_j}\om_\HH\xiphi)$ for $j=1,2$, and $z_j=A\xit$ for $j=3,4,5$, we have, in view of \eqref{eq:finalchoicesofhyfzsymbols},
\bea\lab{eq:definitionofthesymbolx1andX:summary:bisx1}
x_1 &=& i\sum_{j=1}^2\chi_j^2\left(\frac{2(y_j+f_j)}{r^2+a^2}\S_1+A\chi_{z_j}\om_\HH\xiphi\right)+ i\sum_{j=3}^5\chi_j^2\frac{2(y_j+f_j)}{r^2+a^2}\S_1\nn\\
&=& i\sum_{j=1}^2\chi_j^2A\chi_{z_j}\om_\HH\xiphi+ i\sum_{j=1}^5\chi_j^2\frac{2(y_j+f_j)}{r^2+a^2}\S_1.
\eea
By defining $x_1=b_{\tt}\xit + b_{\tphi}\xiphi$, and 
recalling from \eqref{eq:computationofsymbolmodqsquareboxgam:1} that
\beaa
\S_1= (\R)(1-\mu \tmod')\xit+(a-\Delta\phimod')\xiphi, \qquad 
\S_1\in\widetilde{S}^{1,0}(\MM),
\eeaa
we infer
\bea\lab{eq:definitionofXandbttandbtphi:summary}
\bsplit
X=&\Opw(is_0\mu\xi_r+ib_{\tphi}\xiphi+ib_{\tt}\xit)+A\pr_\tau,\\
b_{\tt}=& 2(1-\mu \tmod')\sum_{j=1}^5\chi_j^2 (y_j+f_j), \quad b_{\tphi}=\sum_{j=1}^2\chi_j^2A\om_\HH\chi_{z_j}
+\frac{2(a-\Delta\phimod')}{r^2+a^2}\sum_{j=1}^5\chi_j^2(y_j+f_j),
\end{split}
\eea
with $b_{\tt}$ and $b_{\tphi}$ real valued symbols in $\widetilde{S}^{0,0}(\MM)$. 
In view of Lemma \ref{lem:specificchoice:normalizedcoord}, both $1-\mu \tmod'$ and $a-\Delta\phimod'$ vanish identically in $\MM_{r_+(1+2\dbl), 12m}$, and hence, the desired estimates \eqref{eq:estimatesofbtphiandbtt:intermsofe} then follow from the lower bound \eqref{eq:defintionofthesymboleasasquarerootofsigma2TXEmsumofsquares:1:copy} for $e$.

Next, in view of \eqref{eq:generalformofthePDOmultipliersXandE},   \eqref{eq:definitionofQhQyQfQzintermsofTXEands0s1e0} and \eqref{eq:finalchoicesofhyfzsymbols}, we have
\beaa
s_0&=&2\sum_{j=1}^5 \chi_j^2(y_j + f_j), \\
e_0&=&\sum_{j=1}^5 \chi_j^2\bigg(\mu h_j + \frac{2\mu r}{\R}y_j - \pr_r (\mu y_j) + \frac{2\mu r}{\R}f_j - \pr_r (\mu f_j) +\mu \pr_r f_j\bigg)\nn\\
&=&\sum_{j=1}^5 \chi_j^2\bigg(\mu h_j + \frac{2\mu r}{\R}y_j - \pr_r (\mu y_j) + \bigg(\frac{2\mu r}{\R} -\pr_r (\mu) \bigg)f_j \bigg),
\eeaa
both of which are real valued symbols in $\widetilde{S}^{0,0}(\MM)$.
By the choices of $h_j$, $y_j$ and $f_j$ made in Section \ref{subsect:scalarwave:Kerr:condiMora}, both $h_5$ and $y_5$ vanish identically in a neighborhood of $r_{max}$ and $f_5$ vanishes linearly at $r_{max}$, hence it follows from the definition \eqref{eq:definitionofrtrapinfunctionofrmaxandcutoffinmathcalG5} of $r_{\trap}$ that
\beaa
|s_0|+|e_0| \les \sum_{j=1}^4\chi_j^2 +\chi_5^2|r-r_{max}|=1-\chi_5^2+\chi_5^2|r-r_{\trap}| \les \sqrt{1-\chi_5^2} +|\chi_5| |r-r_{\trap}|
\eeaa
as desired. This concludes the proof of point \eqref{item:symbolsPDOs:3}.

In the end, we show point \eqref{item:symbolsPDOs:4}. By construction, see  \eqref{eq:definitionofpartitionofunityThetajj=0toiota:propsumto1} and the line above, $\Th_i=\Th_i(\Xi)$ are real valued symbols in $\widetilde{S}^{0,0}(\MM)$ and satisfy $\pr_r(\Th_i)=0$.  Also, in view of the definition \eqref{eq:defintionofthedifferentialoperatorsViformicrolocalNRGestimates} for $\{V_i\}_{i=1,2,\ldots, \iota}$ and the choice of $V_0$ and $V_{-1}$ in Step 6 of the proof for Proposition \ref{thm:nondeg:EnerandMora:Kerrandpert}, it follows that $V_i$ are vectorfields satisfying
\beaa
V_i=\pr_\tau+d_i(r)\pr_{\tphi}, \qquad i=-1,0, 1,\cdots,\iota, 
\eeaa
for smooth real valued functions $d_i(r)$ supported  in $r\leq 10m$.
\end{proof}

%%%%%%%%%%%%%%%%%%%%%%%%%%%%%%%%%%%%%%%%%%

\section{Proof of Proposition \ref{prop:energymorawetzmicrolocalwithblackbox}}
\lab{sec:proofofprop:energymorawetzmicrolocalwithblackbox}

%%%%%%%%%%%%%%%%%%%%%%%%%%%%%%%%%%%%%%%%%%

In order to prove Proposition \ref{prop:energymorawetzmicrolocalwithblackbox}, we make use of the following lemma, which proves an energy-Morawetz estimate for general inhomogeneous wave equations in a subextremal Kerr spacetime. 

\begin{lemma}[Energy-Morawetz estimate for inhomogeneous wave equations on Kerr]
\lab{lem:EMFestimates:inhomowave:Kerr:general}
Let $\psi$ be a solution to the wave equation
\bea
\square_{\gam}\psi=F.
\eea
Assume $\psi$ vanishes for $\tau\leq 1$ and assume $F$ is compactly supported in $[1, +\infty)$. Then, we have
the following energy-Morawetz estimate
\bea
\lab{eq:EMFestimates:inhomowave:general}
\EMF[\psi](\Reals) &\les& \widetilde{\mathcal{N}}'[\psi, F],
\eea
where $\widetilde{\mathcal{N}}'[\psi, F]$ is given by 
\beaa
\widetilde{\mathcal{N}}'[\psi, F] &:=&\widetilde{\mathcal{N}}'_{trap}[F] +\sup_{\tau\in\Reals}\bigg|\int_{\Mntrap(-\infty, \tau)}F{\ov{\pr_{\tau}\psi}}\bigg|+\int_{\Mntrap}\big(|\pr_r\psi|+r^{-1}|\psi|\big)|F|+\int_{\MM}|F|^2,
\eeaa
with 
\beaa
\widetilde{\mathcal{N}}'_{trap}[F] :=\min\left(\int_{\Mtrap}|\pr F|^2, \int_{\Mtrap}\tau^{1+\dec}|F|^2\right).
\eeaa
\end{lemma}

\begin{proof}
This lemma is an adaption of \cite[Proposition 9.8.1 and Proposition  13.1]{DRSR}. In view of our assumptions on $\psi$ and $F$, the solution $\psi$ is  past integrable and, in view of \cite[Theorem 3.2]{DRSR} for homogeneous scalar wave equation in a subextremal Kerr spacetime, is also future integrable. Hence, compared to \cite[Proposition 9.8.1]{DRSR}, we do not need to apply a cutoff in the past time to the solution and thus do not have $\int_{\Sigma_0}|\psi|^2$ present at the RHS. Repeating the proof of {\cite[Proposition  9.1]{DRSR}} in the same manner, we deduce a Morawetz-flux estimate
\bea
{\widetilde{\MF}[\psi]}&\les& \bigg|\int_{\MM} {F\ov{(S^1 +r^{-1} S^0)\psi}}\bigg|,
\eea
where {$S^1\in\Opw(\widetilde{S}^{1,1}(\MM))$} and {$S^0\in\Opw(\widetilde{S}^{0,0}(\MM))$} denote all the first order and zeroth order pseudodifferential and differential operators {used in the proof}  and where {$e\in\widetilde{S}^{1,0}(\MM)$ satisfies \eqref{eq:defintionofthesymboleasasquarerootofsigma2TXEmsumofsquares:1}}. In particular, the operators $S^1$ and $S^0$ are purely differential operators for $r$ large. 

On $\Mtrap$, we decompose $S^1=S^1_0\pr_r+S^1_1$ where $S^1_0\in\Opw(\widetilde{S}^{0,0}(\MM))$ and $S^1_1\in\Opw(\widetilde{S}^{1,0}(\MM))$. We apply Cauchy--Schwarz to deduce 
\beaa
\bigg|\int_{\Mtrap}{F\ov{(S^1_0\pr_r+r^{-1} S^0)}\psi}\bigg| \les {\bigg(\int_{\Mtrap} |F|^2\bigg)^{\frac{1}{2}} \big(\M[\psi](\Reals)\big)^{\frac{1}{2}}},
\eeaa
and we can control the integral $\big|\int_{\Mtrap} {FS^1_1}\psi\big|$ either by {taking the adjoint of $S^1_1$} to find it bounded by $
(\int_{\Mtrap}|{\pr^{\leq 1} F}|^2)^{\frac{1}{2}}{(\M[\psi](\Reals))^{\frac{1}{2}}}$ or by applying Cauchy--Schwarz and Lemma \ref{lem:gpert:MMtrap} to bound it by
$(\int_{\Mtrap}\tau^{1+\de}| F|^2)^{\frac{1}{2}} {(\EM[\psi](\Reals))^{\frac{1}{2}}}$. To conclude, we have
\bea
 \bigg|\int_{\Mtrap} {F\ov{(S^1 +r^{-1} S^0)\psi}}\bigg|\les \big(\widetilde{\mathcal{N}}'_{trap}[F] \big)^{\frac{1}{2}}{( \EM[\psi](\Reals))^{\frac{1}{2}}},
\eea
which then yields
{\bea
\lab{eq:nondegenerateEMF:MFpart:exactkerr}
{\widetilde{\MF}[\psi]}
&\les& \big(\widetilde{\mathcal{N}}'_{trap}[F] \big)^{\frac{1}{2}} {(\EM[\psi](\Reals))^{\frac{1}{2}}}+\bigg|\int_{\Mntrap}{F\ov{(S^1 +r^{-1} S^0)\psi}}\bigg|.
\eea}

We next consider the integral $\big|\int_{\Mntrap}{F\ov{(S^1 +r^{-1} S^0)\psi}}\big|$. The integral over a finite radius region is manifestly bounded by $\widetilde{\mathcal{N}}'[\psi, F]$ by applying Cauchy--Schwarz. Notice that when the integral is integrated over the large radius region where $S^1$ and $S^0$ are differential operators, $S^1$ takes {the form} $A\pr_{\tau} +O(r^{-1})\pr_{\tau}+O(1)\pr_{r}+O(r^{-2})\pr_{\tphi}$ and $S^0=O(1)$. Hence, {applying Cauchy--Schwarz}, we infer
\bea
\nn\bigg|\int_{\Mntrap}{F\ov{(S^1 +r^{-1} S^0)\psi}}\bigg| &\les& {\bigg|\int_{\Mntrap}F\ov{\pr_\tau\psi}\bigg|+\int_{\Mntrap}|F|\big(|\pr_r\psi|+r^{-1}|\psi|+r^{-1}|\nab\psi|\big)}\\
&\les&\widetilde{\mathcal{N}}'[\psi, F] {+{\bigg(\int_{\Mntrap} |F|^2\bigg)^{\frac{1}{2}} \big(\M[\psi](\Reals)\big)^{\frac{1}{2}}}}.
\eea
Putting together the above estimates yields 
\bea
\label{eq:morawetzfluxestimate:EMFestimate:exactKerr:lastsectproof}
{\widetilde{\MF}[\psi]}&\les& \widetilde{\mathcal{N}}'[\psi, F]+ {\big(\widetilde{\mathcal{N}}'[F] \big)^{\frac{1}{2}}(\EM[\psi](\Reals))^{\frac{1}{2}}}.
\eea

Next, we apply the energy estimate \eqref{eq:energyestimate:awayhorizon:Kerrandpert:lastsect}, which holds in perturbations of Kerr, in the particular case $\g=\gam$. Then,  \eqref{eq:energyestimate:awayhorizon:Kerrandpert:lastsect}, which is applied here in the case $\ep=0$, implies
\bea
\lab{eq:energyestimate:globally:exactKerr:lastsect}
{\sup_{\tau\in\Reals}\E[\psi](\tau)}
&\les&{\widetilde{\MF}[\psi]}
+\int_{\MM}|F|^2 +\sup_{\tau\in\Reals}\bigg|\int_{\Mntrap(-\infty, \tau)}{F\ov{\pr_{\tau}\psi}}\bigg|\nn\\
&&
+\bigg(\min\bigg(\int_{\Mtrap}\tau^{1+\dec}|F|^2, \int_{\Mtrap}|\pr F|^2\bigg)\bigg)^{\frac{1}{2}}\Big(\EM[\psi](\Reals)\Big)^{\frac{1}{2}}.
\eea
Together with the Morawetz-flux estimate \eqref{eq:morawetzfluxestimate:EMFestimate:exactKerr:lastsectproof}, we deduce
\beaa
{\widetilde{\EMF}[\psi]}
&\les& \widetilde{\mathcal{N}}'[\psi, F]+ {\big(\widetilde{\mathcal{N}}'[F] \big)^{\frac{1}{2}}( \EM[\psi](\Reals))^{\frac{1}{2}}}+\int_{\MM}|F|^2 +\sup_{\tau\in\Reals}\bigg|\int_{\Mntrap(-\infty, \tau)}{F\ov{\pr_{\tau}\psi}}\bigg|\\
&&
+\bigg(\min\bigg(\int_{\Mtrap}\tau^{1+\dec}|F|^2, \int_{\Mtrap}|\pr F|^2\bigg)\bigg)^{\frac{1}{2}}\Big(\EM[\psi](\Reals)\Big)^{\frac{1}{2}},
\eeaa
and hence
\beaa
\EMF[\psi](\Reals) &\les& \widetilde{\mathcal{N}}'[\psi, F],
\eeaa
as stated in \eqref{eq:EMFestimates:inhomowave:general}. This concludes the proof of Lemma \ref{lem:EMFestimates:inhomowave:Kerr:general}.
\end{proof}

We are now ready to prove Proposition \ref{prop:energymorawetzmicrolocalwithblackbox}. We have introduced in Lemma \ref{lemma:computationofthederiveativeofsrqtg} the 1-form $N_{det}$ given by  
\beaa
(N_{det})_\mu=\frac{1}{\sqrt{|\g|}}\pr_\mu\sqrt{|\g|} - \frac{1}{\sqrt{|\g_{a,m}|}}\pr_\mu\sqrt{|\g_{a,m}|}. 
\eeaa
This allows us to rewrite the scalar wave equation $\square_\g\psi = F$ as 
\bea\lab{eq:waveequationforpsiafterthrowingthequasilineartermonRHS}
\square_{\g_{a,m}}\psi = F-\err[\psi]
\eea
where 
\bea\lab{eq:waveequationforpsiafterthrowingthequasilineartermonRHS:errorterms}
\bsplit
\err[\psi] =& \err_1[\psi]+\err_2[\psi]+\err_3[\psi],\\
\err_1[\psi]:=&\widecheck{\g}^{\a\b}\pr_{\a}\pr_{\b}\psi, \qquad \err_2[\psi]:=\frac{1}{\sqrt{|\g_{a,m}|}}\pr_{\a}\left(\sqrt{|\g_{a,m}|}\widecheck{\g}^{\a\b}\right)\pr_{\b}\psi,\\
\err_3[\psi]:=&(N_{det})^\a\pr_{\a}F.
\end{split}
\eea

We apply Lemma \ref{lem:EMFestimates:inhomowave:Kerr:general}  to \eqref{eq:waveequationforpsiafterthrowingthequasilineartermonRHS} {which yields}
\bea\lab{eq:waveequationforpsiafterthrowingthequasilineartermonRHS:basicNRJMor}
\nn\EMF[\psi] (\Reals)&\les& {\widetilde{\mathcal{N}}'}[\psi, F{-}\err[\psi]]\\
&\les& {\widetilde{\mathcal{N}}'}[\psi, F]+{\widetilde{\mathcal{N}}'}[\psi, \err_1[\psi]]+{\widetilde{\mathcal{N}}'}[\psi, \err_2[\psi]]+{\widetilde{\mathcal{N}}'}[\psi, \err_3[\psi]].
\eea
Now, the assumptions \eqref{eq:controloflinearizedinversemetriccoefficients} for the perturbed inverse metric coefficients, as well as the control of $N_{det}$ provided by Lemma \ref{lemma:computationofthederiveativeofsrqtg}, immediately yields
\beaa
\sum_{j=1}^3\widetilde{\mathcal{N}}'_{{trap}}[\psi, \err_j[\psi]](\mathbb{R})\les \ep\left({\int_{1}^{\infty}\frac{d\tau}{\tau^{1+\dec}}}\right){\sup_{\tau\in\mathbb{R}}\E^{(1)}[\psi](\tau)}\les\ep\,{\sup_{\tau\in\mathbb{R}}\E^{(1)}[\psi](\tau)}.
\eeaa
Together with \eqref{eq:waveequationforpsiafterthrowingthequasilineartermonRHS:basicNRJMor}, we infer
\bea\lab{eq:waveequationforpsiafterthrowingthequasilineartermonRHS:basicNRJM:1}
\nn\EMF[\psi] (\Reals) &\les& \int_{\Mtrap}|\pr F|^2{+\sup_{\tau\in\Reals}\bigg|\int_{\Mntrap(-\infty, \tau)}{F\ov{\pr_{\tau}\psi}}\bigg|}
+\int_{\Mntrap}\big(|\pr_r\psi|+r^{-1}|\psi|\big)|F|\\
&&+\int_{\MM}|F|^2+\sum_{j=1}^3K_j+\ep\sup_{\tau\in\mathbb{R}}\E^{(1)}[\tau],
\eea
where 
\beaa
K_j&:=&
{\sup_{\tau\in\Reals}\bigg|\int_{\Mntrap(-\infty, \tau)}{\err_j[\psi]\ov{\pr_{\tau}\psi}} \bigg|}\nn\\
&&+\int_{\Mntrap}\big(|\pr_r\psi|+r^{-1}|\psi|\big)|\err_j[\psi]|+\int_{\MM}|\err_j[\psi]|^2, \quad j=1,2,3.
\eeaa

Next, we estimate the terms $K_j$ appearing in the RHS of \eqref{eq:waveequationforpsiafterthrowingthequasilineartermonRHS:basicNRJM:1}. First, in view of the assumptions \eqref{eq:controloflinearizedinversemetriccoefficients} for the perturbed inverse metric coefficients and in view of the control of $N_{det}$ provided by Lemma \ref{lemma:computationofthederiveativeofsrqtg}, we infer from Lemma \ref{lemma:basiclemmaforcontrolNLterms:bis} that 
\beaa
K_1+K_3\les\ep\EM^{(1)}[\psi](\mathbb{R})
\eeaa
which together with \eqref{eq:waveequationforpsiafterthrowingthequasilineartermonRHS:basicNRJM:1} implies
\bea\lab{eq:waveequationforpsiafterthrowingthequasilineartermonRHS:basicNRJM:2}
\nn\EMF[\psi] (\Reals)&\les& \int_{\Mtrap}|\pr F|^2{+\sup_{\tau\in\Reals}\bigg|\int_{\Mntrap(-\infty, \tau)}{F\ov{\pr_{\tau}\psi}}\bigg|}\\
\nn&&+\int_{\Mntrap}\big(|\pr_r\psi|+r^{-1}|\psi|\big)|F|+\int_{\MM}|F|^2+K_2+\ep\EM^{(1)}[\psi](\mathbb{R}).
\eea

It remains to control $K_2$. In view of the definition of $\err_2[\psi]$ in \eqref{eq:waveequationforpsiafterthrowingthequasilineartermonRHS:errorterms}, we  have
\beaa
\err_2[\psi]=N^\a\pr_\a\psi, \qquad N^\a:=\frac{1}{\sqrt{|\g_{a,m}|}}\pr_{x^\b}\left(\sqrt{|\g_{a,m}|}\widecheck{\g}^{\a\b}\right).
\eeaa
Using \eqref{eq:assymptiticpropmetricKerrintaurxacoord:volumeform}, we infer
 \beaa
 N^\a &=& \pr_\tau(\widecheck{\g}^{\a\tau})+\pr_r(\widecheck{\g}^{\a r})+\pr_{x^a}(\widecheck{\g}^{\a a})+O(r^{-1})\widecheck{\g}^{\a r}+{O(1)\widecheck{\g}^{\a a}}\\
 &=& \dk(\widecheck{\g}^{\a\tau})+O(r^{-1})\dk^{\leq 1}(\widecheck{\g}^{\a r})+\dk^{\leq  1}(\widecheck{\g}^{\a a}),
 \eeaa
which together with \eqref{eq:controloflinearizedinversemetriccoefficients} implies
\beaa
N^\tau &=& \dk(\widecheck{\g}^{\tau\tau})+O(r^{-1})\dk^{\leq 1}(\widecheck{\g}^{\tau r})+\dk^{\leq  1}(\widecheck{\g}^{\tau a})=\dk^{\leq 1}\Ga_g,\\
N^r &=& \dk(\widecheck{\g}^{r\tau})+O(r^{-1})\dk^{\leq 1}(\widecheck{\g}^{rr})+\dk^{\leq  1}(\widecheck{\g}^{ra})=r\dk^{\leq 1}\Ga_g,\\
N^a &=& \dk(\widecheck{\g}^{a\tau})+O(r^{-1})\dk^{\leq 1}(\widecheck{\g}^{ar})+\dk^{\leq  1}(\widecheck{\g}^{ab})=\dk^{\leq 1}\Ga_g.
\eeaa
In view of Lemma \ref{lemma:basiclemmaforcontrolNLterms:bis}, we infer
\beaa
K_2\les\ep\EM^{(1)}[\psi](\mathbb{R})
\eeaa
which together with \eqref{eq:waveequationforpsiafterthrowingthequasilineartermonRHS:basicNRJM:2} implies
\beaa
\nn\EMF[\psi] (\Reals)&\les& \int_{\Mtrap}|\pr F|^2{+\sup_{\tau\in\Reals}\bigg|\int_{\Mntrap(-\infty, \tau)}{F\ov{\pr_{\tau}\psi}}\bigg|}\\
\nn&&+\int_{\Mntrap}\big(|\pr_r\psi|+r^{-1}|\psi|\big)|F|+\int_{\MM}|F|^2+\ep\EM^{(1)}[\psi](\mathbb{R})
\eeaa
as stated. This concludes the proof of Proposition \ref{prop:energymorawetzmicrolocalwithblackbox}.

%%%%%%%%%%%%%%%%%%%%%%%%%%%%%%%%%%%%%%%%%%
%%%%%%%%%%%%%%%%%%%%%%%%%%%%%%%%%%%%%%%%%%

%%%%%%%%%%%%%%%%%%%%%%%%%%%%%%%%%%%%%%%%%%
%%%%%%%%%%%%%%%%%%%%%%%%%%%%%%%%%%%%%%%%%%

\bigskip
\footnotesize

Siyuan Ma, \par\nopagebreak
\textsc{State Key Laboratory of Mathematical Sciences, Academy of Mathematics and Systems Science, Chinese Academy of Sciences, Beijing 100190, China}\par\nopagebreak
\textit{E-mail address:} \href{mailto:siyuan.ma@amss.ac.cn}{siyuan.ma@amss.ac.cn}

\bigskip

J\'{e}r\'{e}mie Szeftel, \par\nopagebreak
\textsc{CNRS \& Laboratoire Jacques-Louis Lions, Sorbonne Universit\'{e},
4 place Jussieu 75005 Paris, France}\par\nopagebreak
\textit{E-mail address}: \href{mailto:jeremie.szeftel@sorbonne-universite.fr}{jeremie.szeftel@sorbonne-universite.fr}

%%%%%%%%%%%%%%%%%%%%%%%%%%%%%%%%%%%%%%%%%%
%%%%%%%%%%%%%%%%%%%%%%%%%%%%%%%%%%%%%%%%%%

\end{document}